\documentclass[11pt,reqno]{amsart}
\usepackage{geometry} 

\usepackage{amsmath,amssymb,amsthm,mathrsfs,appendix}
\usepackage{xcolor}
\usepackage{graphicx}
\usepackage{tikz}
\usepackage{tikz-3dplot}
\usetikzlibrary{positioning}
\usepackage{pgfplots}
\usepackage{subcaption}
\usepackage{adjustbox}

\usetikzlibrary{arrows.meta}

\usepackage{marvosym}
\usepackage{cancel}

\usepackage{latexsym}
\usepackage[nobysame]{amsrefs}[11pts]
\usepackage{bm}
\usepackage{enumerate}
\usepackage{indentfirst}
\usepackage[colorlinks,
            linkcolor=black,
            anchorcolor=black,
            citecolor=black
            ]{hyperref}

\newcommand{\into}{\hookrightarrow}

\newcommand{\ds}{\mathrm{d}s}
\newcommand{\dt}{\mathrm{d}t}

\newcommand\absB[1]{\Big|#1\Big|}
\newcommand\norm[1]{\left\|#1\right\|}
\newcommand\normf[1]{\|#1\|}

\DeclareMathOperator{\support}{supp}
\DeclareMathOperator{\dive}{div}

\newcommand{\EE}{\mathbb{E}}

\newcommand{\II}{\mathbb{I}}

\newcommand{\NN}{\mathbb{N}}

\newcommand{\RR}{\mathbb{R}}
\renewcommand{\SS}{\mathbb{S}}

\newcommand{\ZZ}{\mathbb{Z}}

\newcommand{\cA}{\mathcal{A}}
\newcommand{\cB}{\mathcal{B}}

\newcommand{\cD}{\mathcal{D}}
\newcommand{\cE}{\mathcal{E}}
\newcommand{\cF}{\mathcal{F}}
\newcommand{\cG}{\mathcal{G}}
\newcommand{\cH}{\mathcal{H}}

\newcommand{\cJ}{\mathcal{J}}
\newcommand{\cK}{\mathcal{K}}

\newcommand{\cO}{\mathcal{O}}

\newcommand{\cT}{\mathcal{T}}

\newcommand{\cV}{\mathcal{V}}

\newcommand{\cX}{\mathcal{X}}
\newcommand{\cY}{\mathcal{Y}}

\newcommand{\wU}{\widehat{U}}

\newcommand{\weta}{\widehat{\eta}}

\newtheorem{Definition}{Definition}
\newtheorem{Theorem}{Theorem}[section]
\newtheorem{Lemma}{Lemma}[section]
\newtheorem{Proposition}{Proposition}
\newtheorem{Remark}{Remark}

\numberwithin{Remark}{section}
\numberwithin{Proposition}{section}
\numberwithin{Definition}{section}
\numberwithin{Lemma}{section}
\numberwithin{equation}{section}
\numberwithin{Theorem}{section}

\allowdisplaybreaks[4]
\pgfplotsset{compat=1.18}

\begin{document}
\title[degenerate compressible Navier-Stokes equation]{Global Well-Posedness\\ of the Vacuum Free Boundary Problem for \\ the Degenerate Compressible Navier-Stokes Equations\\ with Large Initial Data of Spherical Symmetry}
\date{\today}

\author{Gui-Qiang G. Chen}
\address[Gui-Qiang G. Chen]{Mathematical Institute, University of Oxford, Oxford, OX2 6GG, UK.} \email{\tt gui-qiang.chen@maths.ox.ac.uk}

\author{Jiawen Zhang}
\address[Jiawen Zhang]{School of Mathematical Sciences, Shanghai Jiao Tong University, Shanghai 200240, P. R. China} \email{\tt zhangjiawen317@sjtu.edu.cn}

\author{Shengguo Zhu }
\address[Shengguo Zhu]{School of Mathematical Sciences, CMA-Shanghai, and MOE-LSC, Shanghai Jiao Tong University, Shanghai 200240, P. R. China.} \email{\tt  zhushengguo@sjtu.edu.cn}

\begin{abstract}
The study of global-in-time dynamics of vacuum is crucial for understanding viscous flows. In particular, physical vacuum,  which ensures a finite, nontrivial  pressure-induced acceleration in the  normal direction on the moving  boundary, naturally arises in the motion of gaseous stars and shallow water. The corresponding large-data problems for spherically symmetric flows remain open, due to the combined difficulties of coordinate singularity at the origin and degeneracy on the moving boundary. In this paper, we analyze the free boundary problem for the barotropic compressible Navier-Stokes equations with density-dependent viscosity coefficients (as in the shallow water equations) in two and three spatial dimensions. For a general class of spherically symmetric initial densities:  $\rho_0^{\beta}\in H^3$ with $\beta\in (\frac{1}{3},\gamma-1]$ ($\gamma$: adiabatic exponent), vanishing on the moving  boundary in the form of a distance function, 
we establish the global existence  of the unique classical solution with large  data. Moreover, our solution  is  classical up to the moving boundary, and the physical vacuum is allowed for the data under consideration here  when   $\beta=\gamma-1$. The key ingredient of our analysis is to obtain the uniform control over the first-order spatial derivatives of the flow map $\eta(t,r)$ in spherically symmetric Lagrangian coordinates, which is  achieved by some novel density-weighted estimates  and flow-map-weighted estimates for the solutions.
\end{abstract}

\date{\today}
\keywords{Degenerate compressible Navier-Stokes equations, Vacuum free boundary problem, Classical solutions, Large data, Global well-posedness, Spherical symmetry\\[1mm]
 * Corresponding author: Gui-Qiang G. Chen}

\maketitle

\tableofcontents

\section{Introduction}\label{section1}
 In this paper, we establish the global well-posedness of classical solutions with large initial data of spherical symmetry to the vacuum free boundary problem (\textbf{VFBP}) of the following compressible Navier-Stokes equations (\textbf{CNS}) with degenerate density-dependent viscosity coefficients (as in the shallow water equations, \textit{i.e.}, the viscous Saint-Venant system) in two and three spatial dimensions:
\begin{equation}\label{eq:1.1-vfbp}
\begin{cases}
\rho_t+\dive(\rho \boldsymbol{u})=0 &\text{in }\Omega(t),\\[4pt]
(\rho \boldsymbol{u})_t+\dive(\rho \boldsymbol{u}\otimes \boldsymbol{u})+\nabla P =2\mu\dive(\rho D(\boldsymbol{u}))&\text{in }\Omega(t),\\[4pt]
\rho>0&\text{in }\Omega(t),\\[4pt]
\rho=0&\text{on }\partial\Omega(t),\\[4pt]
\cV(\partial\Omega(t))=\boldsymbol{u}\cdot\boldsymbol{n}(t)&\text{on }\partial\Omega(t),\\[4pt]
(\rho,\boldsymbol{u})|_{t=0}=(\rho_0,\boldsymbol{u}_0) &\text{in }\Omega:=\Omega(0).
\end{cases}
\end{equation}
Here, $t\geq 0$ is the time, $\boldsymbol{x}=(x_1,\cdots\!,x_n)^{\top}\in \mathbb{R}^n$  is the Eulerian spatial coordinate, the open and bounded subset  $\Omega(t)\subset \mathbb{R}^n$ denotes the changing volume occupied by the fluid and 
\begin{equation*}
\Omega(0)=B_1=\{\boldsymbol{x}: \,|\boldsymbol{x}| <1\},
\end{equation*}
$\partial \Omega(t)$ denotes the moving vacuum boundary,  $\cV(\partial \Omega(t))$ denotes the normal velocity of $\partial \Omega(t)$, and $ \boldsymbol{n}(t) $ denotes the exterior unit normal vector to $\partial \Omega(t)$. 
Moreover, $\rho\geq 0$ denotes the mass density of the fluid, $\boldsymbol{u}=(u_1,\cdots\!, u_n)^\top$ $\in \mathbb{R}^n$  denotes the Eulerian velocity field,
\begin{equation*}
D(\boldsymbol{u})=\frac{1}{2}(\nabla \boldsymbol{u}+(\nabla \boldsymbol{u})^\top)
\end{equation*}
is the strain tensor, and $P$ denotes the pressure function. 
For polytropic fluids,  the constitutive relation is given by 
\begin{equation*}
P=A\rho^{\gamma},
\end{equation*}
where $A>0$ is  the entropy  constant and  $\gamma>1$ is the adiabatic exponent. Equation $\eqref{eq:1.1-vfbp}_3$ asserts that there is no vacuum inside the fluid, $\eqref{eq:1.1-vfbp}_4$ is the vacuum boundary condition stating that the density vanishes along the moving vacuum boundary $\partial \Omega(t)$, $\eqref{eq:1.1-vfbp}_5$ is the kinematic boundary condition that requires that the boundary movement is tangential to the fluid particles, and $\eqref{eq:1.1-vfbp}_6$ provides the initial conditions for the density, velocity, and domain.  It is worth emphasizing that no restriction is imposed on the size of the initial data in our results, and  the physical vacuum is allowed for the data we considered here.

When $\gamma=n=2$, system $\eqref{eq:1.1-vfbp}_1$--$\eqref{eq:1.1-vfbp}_2$ corresponds to the well-known two-dimensional (2-D) viscous shallow water equations ({\it i.e.}, the viscous Saint-Venant system): 
\begin{equation}\label{shallow1}
\begin{cases}
h_t+\dive(h \boldsymbol{u})=0,\\[3pt]
(h\boldsymbol{u})_t+\dive(h \boldsymbol{u}\otimes \boldsymbol{u})+A\nabla (h^2)={V}(h, \boldsymbol{u}),
\end{cases}
\end{equation}
where $h$ is  the height of the fluid surface,  $\boldsymbol{u}\in \mathbb{R}^2$ is the  horizontal velocity, and $V(h, \boldsymbol{u})$ is the viscosity term. For the spherically symmetric flow,  since $D(\boldsymbol{u})=\nabla \boldsymbol{u}$, then, in \eqref{shallow1},
\begin{equation}\label{shallownian}
{V}(h, \boldsymbol{u})=2\mu\dive(h D(\boldsymbol{u}))
= 2\mu\dive(h\nabla \boldsymbol{u}).
\end{equation}
We refer to \cites{BP,BN2,Gent,gpm,lions,Mar,oran} and the references therein for more details on the viscous shallow water system. Thus, the results established in this paper not only are of fundamental importance for the mathematical analysis of \textbf{CNS}, but also are directly related to an important physical situation, the 2-D shallow water equations with vacuum, as arise, for example, in dam-break problems. Since the vacuum region moves with the flow, the underlying problem is naturally formulated as a \textbf{VFBP}. 

\subsection{Background of the Problem}
The study of vacuum is crucial for understanding compressible viscous flows. In fact, a vacuum inevitably appears in the far field under natural physical requirements, such as finite total mass and total energy in the whole space $\mathbb{R}^n$. The presence of vacuum leads to a degeneracy in the time evolution of \textbf{CNS}, which poses substantial challenges to the analysis of global well-posedness for large solutions; see, for example, \cites{BJ,chz,J,JZ,MPI,fu3,lions,  fu1, hisen, jensen1}. When the viscosity coefficients are taken to be constants, several singular and counterintuitive behaviors of solutions with vacuum have been observed for \textbf{CNS}, which include  the non-conservation of momentum \cites{zz2}, the failure of continuous dependence on the initial data \cite{hoffserre}, among others. Such pathological phenomena may be traced back to the unphysical assumption of constant viscosity coefficients when modeling viscous fluids in the presence of vacuum, under which the vacuum exerts an artificial force on the fluid across the fluid--vacuum interface. From a physical standpoint, compressible viscous flows near vacuum are therefore more appropriately modeled by the degenerate \textbf{CNS}, which can be derived from the Boltzmann equation via the Chapman-Enskog expansion; see Chapman-Cowling \cite{chap}. In this framework, the viscosity coefficients depend on the temperature and, for isentropic flows, this dependence translates into a density dependence through the laws of Boyle and Gay-Lussac. Furthermore, beyond the degenerate \textbf{CNS} and the shallow water equations \eqref{shallow1}, several other physically relevant fluid models incorporate density-dependent viscosities, including the Korteweg system, the lake equations, and the quantum Navier-Stokes equations; see \cites{bd6, bd2, BN2,con s, Gent, jun, Mar} and the references therein.

Recently, the vacuum problem for the degenerate \textbf{CNS} has attracted significant attention. Owing to the strong degeneracy in both the time evolution and spatial dissipation near vacuum, establishing the global existence of the multi-dimensional (M-D) solutions with large initial data remains highly challenging. To date, based on the derivation of the system, there exist several approaches to studying this class of problems. One approach consists of solving the degenerate \textbf{CNS} in the whole space and requiring that the equations hold in the sense of distributions on $[0,T]\times \mathbb{R}^n$ for arbitrarily large time $T$. Along this way, a remarkable analytical framework was initiated by Bresch-Desjardins in a series of papers \cites{bd6, bd8} for barotropic flows (started in 2003 with Lin \cite{bd2} in the context of the Navier-Stokes-Korteweg system with linear shear viscosity). This framework yields key information on the gradient of a function of the density when the viscosity coefficients satisfy the so-called Bresch-Desjardins (BD) constraint. The resulting estimate is now referred to as the BD entropy estimate which, for  degenerate \textbf{CNS} $\eqref{eq:1.1-vfbp}_1$--$\eqref{eq:1.1-vfbp}_2$ in $\mathbb{R}^n$, can be given as 
\begin{equation}\label{bdshallow}
\|\nabla\sqrt{\rho}(t)\|_{L^2(\mathbb{R}^n)}<\infty \qquad \text{for any $t>0$},
\end{equation}
provided that $\nabla\sqrt{\rho_0}\in L^2(\mathbb{R}^n)$ for any $n\geq 1$. This observation has played a fundamental role in the development of global existence theories of M-D weak solutions with finite energy of the Cauchy problem for the degenerate \textbf{CNS}; see  Bresch-Vasseur-Yu \cite{bvy}, Li-Xin \cite{lz}, Vasseur-Yu \cite{vayu}, and the references therein. However, due to the low regularity of these solutions, the uniqueness problem for M-D weak solutions to the degenerate \textbf{CNS} remains widely open.

A second approach is to consider solving the degenerate \textbf{CNS} in the whole space, with the aim of establishing that the system holds pointwise. However, within this framework for the vacuum problem, one encounters a fundamental and intricate difficulty. Mathematically, the degeneracy induced by vacuum creates a severe obstruction to defining the velocity field, since it is highly nontrivial to extend the notion of velocity into regions where the density vanishes. From a physical standpoint, the concept of fluid velocity itself loses meaning in the absence of fluid. Indeed, in the derivation of hydrodynamic equations from first principles, a key underlying assumption is that the fluid is non-dilute and can be described as a continuum. This assumption breaks down in vacuum regions, rendering hydrodynamic equations inapplicable for describing the time evolution of thermodynamic states there. As a result, one cannot meaningfully study the dynamics of vacuum regions by using classical hydrodynamic models. These considerations have naturally led to investigations of classical solutions of the Cauchy problem with far-field vacuum, in which the initial density remains positive for all $\boldsymbol{x}\in \mathbb{R}^n$ but decays to zero in the far field. For such initial density profiles, the degeneracy arising from the far-field vacuum still poses significant analytical challenges. To address this difficulty, an enlarged reformulation of the degenerate \textbf{CNS} was introduced by Li-Pan-Zhu \cite{sz3} for  the  following variables:
\begin{equation}\label{bianhuan}
\phi=\frac{A\gamma}{\gamma-1}\rho^{\gamma-1},\qquad \boldsymbol{\psi} =\frac{1}{\gamma-1}\nabla\log\phi=\nabla\log\rho=(\psi_{1},\cdots\!,\psi_{n})^\top.
\end{equation}
Then  $\eqref{eq:1.1-vfbp}_1$--$\eqref{eq:1.1-vfbp}_2$  can be rewritten as the following enlarged system: 
\begin{equation}\label{eqn1}
\begin{cases}
\phi_t+\boldsymbol{u}\cdot \nabla\phi+(\gamma-1)\phi \dive\boldsymbol{u}=0,\\[4pt]
\boldsymbol{u}_t+\boldsymbol{u}\cdot \nabla \boldsymbol{u}+\nabla \phi+L\boldsymbol{u}=\boldsymbol{\psi}\cdot Q(\boldsymbol{u}),\\[2pt]
\displaystyle
\boldsymbol{\psi}_t+\sum\limits_{l=1}^nA_l(\boldsymbol{u})\partial_l\boldsymbol{\psi}+B(\boldsymbol{u}) \boldsymbol{\psi}+\nabla\dive\boldsymbol{u}=\boldsymbol{0},
\end{cases}
\end{equation}
where matrices $A_l(\boldsymbol{u})=(a_{ij}^{(l)})_{n\times n}(i,j,l=1,\cdots\!,n)$ are symmetric with $a_{ij}^{(l)}=u_l$ for $i=j$ and otherwise  $a_{ij}^{(l)}=0$,  
\begin{equation*}
B(\boldsymbol{u})=(\nabla \boldsymbol{u})^{\top}, \qquad  L\boldsymbol{u}=-\mu\Delta \boldsymbol{u}-\mu\nabla \dive\boldsymbol{u}, \qquad Q(\boldsymbol{u})=2\mu D(\boldsymbol{u}).
\end{equation*}
System \eqref{eqn1} transfers the degeneracies both in the time evolution and spatial dissipation to the possible singularity of $\boldsymbol{\psi}$, which enables us to establish the local well-posedness of regular solutions with far field vacuum to $\eqref{eq:1.1-vfbp}_1$--$\eqref{eq:1.1-vfbp}_2$. We also refer the reader to \cites{sz333,zhuthesis} for other  related progress on the local well-posedness of the degenerate \textbf{CNS}. More recently, through an elaborate analysis of the intrinsic degenerate--singular structures of the degenerate \textbf{CNS}, a series of advances has been made on the 
global well-posedness of classical solutions with far-field vacuum for spherically symmetric flow; see \cites{clz1, GJG} and the references therein.

A third approach is designed to address the case in which vacuum appears in an open set, 
a situation arising in many important physical contexts, such as astrophysics and shallow water waves. As formulated in problem \eqref{eq:1.1-vfbp}, this requires that the degenerate \textbf{CNS} holds only on the set $\{(t,\boldsymbol{x}):\rho(t,\boldsymbol{x})>0\}$, together with an evolution equation for the boundary $\partial \Omega(t)$ of the time-dependent domain $\Omega(t)$ occupied by the fluid. Here, $\partial \Omega(t)$ is itself part of the unknown. This formulation of the vacuum problem is referred to as \textbf{VFBP}, and in this setting, an appropriate boundary condition on the vacuum interface is required. A key feature of \textbf{VFBP} is that the boundary $\partial \Omega(t)$ propagates with finite speed when the initial density is of compact support.
For the physical relevance of this phenomenon, we refer the reader to the survey by Nishida \cite{nishida}. Along this way, physical vacuum has emerged as a particularly important class of vacuum states and has been extensively studied in the context of \textbf{VFBP} for compressible flows. 
We denote by $c=\sqrt{P'(\rho)}$ the speed of sound, and then the physical vacuum is characterized by the following  property: 
\begin{equation}\label{PVcondition}
-\infty<\frac{\partial c^2}{\partial\boldsymbol{n}}<0 \qquad \text{on $\partial \Omega(t)$},
\end{equation}
which ensures that the pressure will induce a finite, non-zero acceleration in the outward normal direction on the moving boundary. Condition \eqref{PVcondition} was first proposed by Liu \cite{LiuTP} in the study of the self-similar solutions to the compressible Euler equations with damping. Moreover, this notion of physical vacuum can also be realized by some self-similar solutions and stationary solutions for other physical systems, such as the Euler-Poisson systems or the Navier-Stokes-Poisson systems for gaseous stars. For {\rm\textbf{VFBP}} of the isentropic compressible Euler equations $\eqref{eq:1.1-vfbp}_1$--$\eqref{eq:1.1-vfbp}_2$  with $\mu=0$, substantial progress has been made on the well-posedness of smooth solutions satisfying condition \eqref{PVcondition}; see \cites{coutand3,coutand2,Jang-M2,Jang-Hadzic,sideris2} and the references therein.

For the corresponding  {\rm\textbf{VFBP}} of the  degenerate \textbf{CNS}, the momentum equations become degenerate  in both the time evolution and spatial dissipation near the moving boundary, which  make it intricate to provide an effective propagation  mechanism for the  regularity of the fluid velocity,
especially in smooth function spaces.
To date, only very limited works have addressed  the global well-posedness of strong/classical solutions to \textbf{VFBP} for the degenerate barotropic \textbf{CNS}.  When gravitational effects are taken into account and the physical vacuum boundary condition \eqref{PVcondition} holds, Ou-Zeng \cite{OZ} established the global existence of one-dimensional (\text{1-D}) strong solutions with small initial data.  Recently,  Xin-Zhang-Zhu \cites{ZJW} established the global 
well-posedness with large data in smooth function spaces for the \text{1-D} viscous Saint-Venant system with  physical vacuum.
Some other  related developments can be found in \cites{Yang-Zhu,LWX2,LXZ3} and the references therein.

Despite these significant advances on \textbf{VFBP} for degenerate \textbf{CNS}, the understanding of the global well-posedness of classical solutions with large initial data 
in the M-D spatial settings remains very limited. Nevertheless, in order to gain deeper insight into the behavior of solutions near the moving vacuum boundary, it is highly desirable to establish global regularity of the corresponding solutions and, in particular, to prove the smoothness of solutions all the way up to the boundary.

\subsection{Lagrangian reformulation of \textbf{VFBP} in spherical coordinates}\label{sec-1.2}
In this paper, we establish the global well-posedness of both 2-D and three-dimensional (\text{3-D}) classical solutions, taking the form 
\begin{equation}\label{ss-ass}
(\rho,\boldsymbol{u})(t,\boldsymbol{x}) = (\rho(t,|\boldsymbol{x}|), u(t,|\boldsymbol{x}|)\frac{\boldsymbol{x}}{|\boldsymbol{x}|}),
\end{equation}
of \textbf{VFBP} \eqref{eq:1.1-vfbp} with the initial data:
\begin{equation}\label{eq:IC}
(\rho,\boldsymbol{u})(0,\boldsymbol{x}) =(\rho_0,\boldsymbol{u}_0)(\boldsymbol{x})= (\rho_0(|\boldsymbol{x}|), u_0(|\boldsymbol{x}|)\frac{\boldsymbol{x}}{|\boldsymbol{x}|}).
\end{equation}
Our results hold for physical adiabatic exponents $\gamma\in (\frac{4}{3},\infty)$ in two dimensions, and for $\gamma\in (\frac{4}{3},3)$ in three dimensions, without any restriction on the size of the initial data. The  initial density $\rho_0$ under consideration 
satisfies the following condition:
\begin{equation}\label{distanceeuler}
\rho_0^\beta(\boldsymbol{x})\in H^3(\Omega), \quad\,\, \cK_1(1-|\boldsymbol{x}|)^\frac{1}{\beta}\leq \rho_0(\boldsymbol{x})\leq \cK_2(1-|\boldsymbol{x}|)^\frac{1}{\beta} \qquad\,\, \text{for all $\boldsymbol{x}\in \overline\Omega$},
\end{equation}
for some  constants $\cK_2>\cK_1>0$ and $\beta\in (\frac{1}{3},\gamma-1]$. It is worth noting that the set of $\rho_0$ defined by \eqref{distanceeuler} contains three typical types of initial density profiles:
\begin{itemize}
\item [\rm (i)] When $\beta\in (0,1)$, \eqref{distanceeuler} implies that  $\rho_0$  satisfies the initial condition of the BD entropy estimate in $\Omega$ (see \eqref{bdshallow} at $t=0$ shown in the whole space $\mathbb{R}^n$, or \cites{bd6,bd8,bd2}):
\begin{equation}\label{BDconditionr}
\|\nabla\sqrt{\rho_0}\|_{L^2(\Omega)}<\infty.
\end{equation}
\smallskip
\item [\rm (ii)] When $\beta=\gamma-1$, \eqref{distanceeuler} implies that $\rho_0$  satisfies the well-known physical vacuum boundary condition for spherically symmetric flow, {\it i.e.},
\begin{equation}\label{PVconditionr}
\rho^{\gamma-1}_0\sim 1-|\boldsymbol{x}| \qquad\,\, 
\text{as $|\boldsymbol{x}|$ approaches the vacuum boundary 
$|\boldsymbol{x}|=1$}.
\end{equation}
Furthermore, it is direct to check that these two types of initial conditions on the density shown in \eqref{BDconditionr}--\eqref{PVconditionr} are not compatible in \eqref{distanceeuler} when $\gamma\geq 2$.
\smallskip
\item [\rm (iii)] When $\gamma> 2$ and $\beta\in (1,\gamma-1)$,  both \eqref{BDconditionr} and \eqref{PVconditionr} fail here. 
\end{itemize}
The well-posedness  established here represents a substantial step toward global regularity of classical solutions with general  data for \textbf{VFBP} of the M-D degenerate \textbf{CNS}, a problem that is closely related to the dam-break phenomenon in the shallow water equations when $\gamma=2$.

Since we focus on the  spherically symmetric flow, we first reformulate problem \eqref{eq:1.1-vfbp} into the following form in $I(t)=[0,R(t))$ as the radial projection of the moving domain $\Omega(t)$ 
with $R(0)=1$ and $I=[0,1)$:
\begin{equation}\label{shallow-SSR-euler}
\begin{cases}
\displaystyle 
\rho_t+u\rho_x+ \rho \big(u_x+\frac{m u}{x}\big)=0&\text{in $I(t)$},\\[3pt]
\displaystyle
\rho u_t+\rho u u_x+P_x=2\mu\Big(\rho \big(u_x+\frac{m u}{x}\big)\Big)_x-2\mu m \frac{\rho_x u}{x} &\text{in $I(t)$},\\[3pt]
\rho>0&\text{in $I(t)$},\\[3pt]
\rho=0&\text{on $\{x=R(t)\}$},\\[3pt]
R'(t)=u &\text{on $\{x=R(t)\}$},\\[3pt]
(\rho,u)|_{t=0}=(\rho_0,u_0)&\text{in $I(0):=I$}, 
\end{cases}
\end{equation}
where $x=|\boldsymbol{x}|$ and $m=n-1$.

Second, problem \eqref{shallow-SSR-euler}, formulated in Eulerian coordinates on the moving interval $I(t)$, can be transformed to a problem on the fixed interval $I$ by introducing the Lagrangian coordinates. To this end, denote by $x=\eta(t,r)$ the position of the fluid particle $x\in I(t)$ at $t\geq 0$ so that
\begin{equation}\label{flowmap-r-la}
\eta_t(t,r)=u(t,\eta(t,r)),\qquad \eta(0,r)=r,
\end{equation}
and $(t,r)$ are the Lagrangian coordinates. 
Then, by introducing the Lagrangian density and velocity 
\begin{equation}\label{varrho-U}
\varrho(t,r)= \rho(t, \eta(t,r)),\qquad U(t,r)= u(t, \eta(t,r)),
\end{equation}
problem \eqref{shallow-SSR-euler} can be written as the following initial-boundary value problem ({\rm\bf IBVP}) in the fixed domain $I$ in Lagrangian coordinates $(t,r)$:
\begin{equation}\label{eq:VFBP-La}
\begin{cases}
\displaystyle \varrho_t + \varrho\big(\frac{U_r}{\eta_r}+\frac{mU}{\eta}\big)=0 & \text{in } (0, T] \times I,\\[3pt]
\displaystyle \eta_r \varrho U_t +A  (\varrho^{\gamma})_r=2\mu  \Big(\varrho\big(\frac{U_r}{\eta_r}+ \frac{mU}{\eta}\big)\Big)_r - 2\mu m\frac{\varrho_r U}{\eta}& \text{in }(0, T] \times I,\\[3pt]
\eta_t = U & \text{in } (0, T] \times I,\\[3pt]
\varrho>0 & \text{in } (0, T] \times I,\\[3pt]
\varrho|_{r=1}=0 & \text{on } (0, T],\\[3pt] 
(\varrho, U, \eta)(0,r)= (\rho_0(r), u_0(r),r) & \text{for $r\in I$}.
\end{cases}
\end{equation}
Moreover, it follows from Lemma \ref{lemma-initial} in Appendix \ref{appb} that condition \eqref{distanceeuler}, which is initially satisfied by $\rho_0(\boldsymbol{x})$ in M-D Eulerian coordinates, can be rewritten in spherical coordinates as a condition satisfied by $\rho_0(r)$: for some  constants $\cK_2>\cK_1>0$ and $\beta\in (\frac{1}{3},\gamma-1]$,
\begin{equation}\label{distance-la}
\begin{aligned}
&r^\frac{m}{2}\big(\rho_0^\beta,(\rho_0^\beta)_r,(\rho_0^\beta)_{rr},\frac{(\rho_0^\beta)_r}{r},(\rho_0^\beta)_{rrr},(\frac{(\rho_0^\beta)_r}{r})_r\big)\in L^2(I),\\
&\cK_1(1-r)^\frac{1}{\beta}\leq \rho_0(r)\leq \cK_2(1-r)^\frac{1}{\beta} \qquad \text{for all $r\in I$}.
\end{aligned}
\end{equation}

In fact, the corresponding study of \eqref{eq:VFBP-La} is extremely difficult, since the structure of the momentum equation $\eqref{eq:VFBP-La}_2$ is full of degeneracy and singularity. This equation can be written in the following form:
\begin{equation}\label{cosingu}
\underbrace{\eta_r\varrho U_t}_{\circledast} +A (\varrho^{\gamma})_r=\underbrace{2\mu \big(\varrho \frac{U_r}{\eta_r}\big)_r }_{\Diamond} + \underbrace{2\mu m\varrho\big(\frac{ U}{\eta}\big)_r}_{\blacktriangle},
\end{equation}
where $\blacktriangle$ denotes the coordinate singularity, $\circledast$ denotes the degenerate time evolution, and $\Diamond$ denotes the degenerate spatial dissipation. 
Owing to the compressibility of fluids and the introduction of the Lagrangian coordinates to handle the moving boundary, our analysis encounters two  major obstacles:
\begin{itemize}
\item \textit{Possible cavitation or implosion within the fluids {\rm (see Figure \ref{fig:1})}};
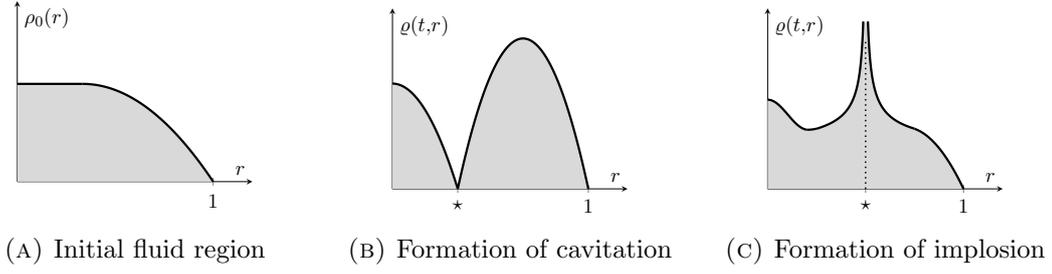
\begin{figure}[t]
\centering
\begin{subfigure}[t]{0.3\textwidth}
\centering
\begin{tikzpicture}[scale=0.7]
\begin{axis}[
    width=6cm,
    height=5cm,
    xmin=0, xmax=1.2,
    ymin=0, ymax=1.2,
    axis lines = middle,
    xlabel = $r$,
    ylabel = $\rho_0(r)$,
    xtick = {0, 1},
    xticklabels = {0, $\textcolor{white}{(}1\textcolor{white}{)}$},
    ytick = {0},
    samples = 200,
    domain = 0:1]
\fill[gray!30] (axis cs:0,0) 
    -- plot[domain=0:1/3] (axis cs:\x,0.65) 
    -- plot[domain=1/3:1] (axis cs:\x,{0.65*(1-2.25*(\x-1/3)^2)}) 
    -- (axis cs:1,0) 
    -- cycle;
\addplot[black, very thick, domain=0:1/3]{0.65};
\addplot[black, very thick, domain=1/3:1, smooth]{0.65*(1-2.25*(x-1/3)^2)};
\end{axis}
\end{tikzpicture}
\caption{Initial fluid region}
\end{subfigure}
\hspace{0cm}
\begin{subfigure}[t]{0.3\textwidth}
\centering
\begin{tikzpicture}[scale=0.7]
\begin{axis}[
    width=6cm,
    height=5cm,
    xmin=0, xmax=1.2,
    ymin=0, ymax=1.2,
    axis lines = middle,
    xlabel = $r$,
    ylabel = $\varrho(t\text{,}r)$,
    xtick = {0,1/3, 1},
    xticklabels = {0,$\star$,$1$},
    ytick = {0},
    samples = 200,
    domain = 0:1]
\fill[gray!30] (axis cs:0,0) 
    -- plot[domain=0:1/3] (axis cs:\x,{0.7-6.3*\x^2}) 
    -- plot[domain=1/3:1] (axis cs:\x,{9*(1-\x)*(\x-1/3)}) 
    -- (axis cs:1,0) 
    -- cycle;
\addplot[black, very thick, domain=0:1/3, smooth]{0.7-6.3*x^2};
\addplot[black, very thick, domain=1/3:1, smooth]{9*(1-x)*(x-1/3)};
\end{axis}
\end{tikzpicture}
\caption{Formation of cavitation}
\end{subfigure}
\hspace{0cm}
\begin{subfigure}[t]{0.3\textwidth}
\centering
\begin{tikzpicture}[scale=0.7]
\begin{axis}[
    width=6cm,
    height=5cm,
    xmin=0, xmax=1.2,
    ymin=0, ymax=1.2,
    axis lines = middle,
    xlabel = $r$,
    ylabel = $\varrho(t\text{,}r)$,
    xtick = {0, 1/2, 1},
    xticklabels = {$0$, $\star$, 1},
    ytick = {0},
    samples = 200,
    domain = 0:1]
\fill[gray!30] (axis cs:0,0) 
    -- plot[domain=0:0.2] (axis cs:\x,{0.1*cos(900*\x)+0.493}) 
    -- plot[domain=0.25:0.488] (axis cs:\x,{0.2+0.1/(0.5-\x)^0.5}) 
    -- plot[domain=0.512:0.75] (axis cs:\x,{0.2+0.1/(\x-0.5)^0.5})
    -- plot[domain=0.75:1] (axis cs:\x,{4.55*(\x-0.4)*(1-\x)})
    -- (axis cs:1,0) 
    -- cycle;
\addplot[black, very thick, domain=0.2:0.25, smooth]{3*x*(x-0.4)+0.513};
\addplot[black, very thick, domain=0:0.2, smooth]{0.1*cos(870*x)+0.493};
\addplot[black, very thick, domain=0.25:0.488, smooth]{0.2+0.1/(0.5-x)^0.5};
\addplot[black, very thick, domain=0.512:0.75, smooth]{0.2+0.1/(x-0.5)^0.5};
\addplot[black, very thick, domain=0.75:1, smooth]{4.565*(x-0.4)*(1-x)};
\addplot[black, dotted, thick] coordinates {(1/2,0) (1/2,1)};
\end{axis}
\end{tikzpicture}
\caption{Formation of implosion}
\end{subfigure}
\caption{Possible cavitation or implosion inside the fluids.}
\label{fig:1}
\end{figure}
\vspace{4pt}
\item 
\textit{Possible degeneracy of the coordinate transformation between the Eulerian and Lagrangian coordinates.} 
\end{itemize}
Dealing with these two obstacles is particularly challenging due to several inherent and intertwined issues:
\begin{itemize}
\vspace{2pt}\item [\rm (i)] Coordinate singularity $\blacktriangle$ at the origin, manifested by the singular factor $\tfrac{1}{\eta}$ in \eqref{cosingu};
\vspace{5pt}
\item [\rm (ii)] Strong degeneracy in both time  evolution $\circledast$ and spatial dissipation $\Diamond$ on the vacuum boundary, which makes it formidable to identify suitable propagation and mollification mechanisms for the regularity of $U$, particularly in smooth function spaces;
\vspace{5pt}
\item [\rm (iii)]Lack of uniform positive lower and upper bounds of $(\eta_r,\frac{\eta}{r})$, which is closely tied to the high-order regularity of $(U_r,\frac{U}{r})$.
\end{itemize}

Thus, new ideas and techniques are required to establish the global well-posedness of classical solutions 
of \textbf{VFBP} \eqref{eq:1.1-vfbp}, particularly in the important case of the 2-D shallow water equations.
Fortunately, by exploiting the intrinsic degenerate-singular structure of system \eqref{eq:1.1-vfbp} 
in the Lagrangian coordinates $(t,r)$, {\it i.e.}, $\eqref{eq:VFBP-La}$, and assuming that
\eqref{distance-la} holds, we are able to prove the global well-posedness of classical solutions 
of \textbf{VFBP} \eqref{eq:1.1-vfbp} with large initial data of spherical symmetry in two and three spatial dimensions. The most crucial step in our analysis is to obtain uniform upper and lower bounds for $(\eta_r,\frac{\eta}{r})$ on $I=[0,1)$, a task that is highly nontrivial due to the aforementioned difficulties. To achieve this, we decompose the interval $I$ into the interior interval and the exterior interval. Combining the interior BD entropy, flow-map-weighted density estimates, and weighted $L^p$-estimates for the velocity, we first obtain an upper bound for $\varrho$. The mass identity and integrated momentum identities then yield the lower bounds for $(\eta_r,\frac{\eta}{r})$. To obtain the upper bounds for $(\eta_r,\frac{\eta}{r})$, we use the momentum equation to obtain the estimates near the vacuum boundary; in the interior, we combine effective velocity estimates with $\eta_r$-weighted estimates to control $\log\varrho$. These estimates then yield the uniform upper bounds for $(\eta_r,\frac{\eta}{r})$ on $I$.

\subsection{Outline of the paper}
The rest of the paper is organized as follows: 
In \S\ref{Section-maintheorem}, 
we state the main theorems. 
In \S \ref{Section-notation}, we outline the main strategy for establishing the global 
existence  theory stated in \S \ref{Section-maintheorem}. \S\ref{Section-densityupper}--\S\ref{Section-global} contain the detailed proof for
the global existence  theory of the classical solutions 
for general smooth, spherically symmetric initial data.
In \S\ref{Section-densityupper}--\S\ref{Section-globalestimates},
we establish the global {\it a priori} estimates needed for the continuation argument; this is achieved in five steps:
\begin{enumerate}
\item[{\rm (i)}] derive the global-in-time {\it a priori} upper bound of the density $\varrho$ and the global uniform lower bounds of $(\eta_r,\frac{\eta}{r})$ (\S \ref{Section-densityupper});
\vspace{3pt}
\item[{\rm (ii)}] establish the global estimates for the effective velocity, including its  $L^\infty$-estimate in the interior interval and $\rho_0$-weighted estimates in the exterior interval (\S \ref{Section-effectivevelocity});
\vspace{3pt}
\item[{\rm (iii)}] derive the global-in-time {\it a priori} upper bounds of $(\eta_r,\frac{\eta}{r})$ (\S \ref{Section-etarupper});
\vspace{3pt}
\item[{\rm (iv)}] show that vacuum does not form inside the fluids in finite time (\S \ref{Section-densitylower});
\vspace{3pt}
\item[{\rm (v)}] establish the global uniform estimates for the velocity (\S \ref{Section-globalestimates}).
\end{enumerate}  
With these estimates in hand, we obtain the desired global existence  of the unique classical solution in \S\ref{Section-global} by continuation arguments. 
Moreover, owing to the double degeneracy exhibited in \eqref{cosingu} in the presence of vacuum, the corresponding local existence and uniqueness  for classical solutions of \textbf{IBVP} \eqref{eq:VFBP-La} or \eqref{eq:VFBP-La-eta} is highly nontrivial; this issue is addressed in detail in \S\ref{Section-local}.
Finally, several auxiliary lemmas and useful coordinate transformations for spherically symmetric functions are collected in Appendices \ref{appendix A}--\ref{subsection2.2}.

\section{Main Theorems}\label{Section-maintheorem}

This section is devoted to stating our main theorems on the global existence and uniqueness  of classical solutions of \textbf{VFBP} \eqref{eq:1.1-vfbp} 
with large spherically symmetric initial data in two and three spatial dimensions. For simplicity, throughout this paper, for any function space 
$X$, positive integer $k$ and functions $(\varphi,g_1,\cdots\!,g_k)$,
the following convention is used: 
\begin{equation*}
\|\varphi(g_1,\cdots\!,g_k)\|_{X}:=\sum_{i=1}^k\|\varphi g_i\|_X.
\end{equation*}

\subsection{Main results in Lagrangian coordinates}

First, $\eqref{eq:VFBP-La}_1$ and $\eqref{eq:VFBP-La}_3$ imply that
\begin{equation}\label{eq:eta}
\varrho(t,r) = \frac{r^m\rho_0(r)}{\eta^m\eta_r}.
\end{equation}
Then, based on the definition of the Eulerian derivative $D_\eta$:
\begin{equation*}
D_\eta f=\frac{f_r}{\eta_r} \qquad\text{for some function $f=f(r)$},
\end{equation*}
problem \eqref{eq:VFBP-La}, combined with \eqref{eq:eta}, can be written as the following  \textbf{IBVP} for $(U, \eta)$:
\begin{equation}\label{eq:VFBP-La-eta}
\begin{cases}
\displaystyle \varrho U_t +A D_\eta(\varrho^{\gamma})=2\mu D_\eta\Big(\varrho\big(D_\eta U+ \frac{mU}{\eta}\big)\Big) - 2\mu m \frac{ U D_\eta\varrho }{\eta}& \text{in $(0, T]\times I$},\\[3pt]
\eta_t = U & \text{in $(0, T] \times I$},\\[3pt]
(U, \eta)(0,r)= (u_0(r), r) & \text{for $r\in I$},
\end{cases}
\end{equation}
where density $\varrho$ is given  by \eqref{eq:eta}. In the rest of this paper,  we denote $D^k_\eta f=D_\eta ( D^{k-1}_\eta f)$ for any positive integer $k$ and  spherically symmetric function $f=f(r)$.

Second, we define classical solutions of problem \eqref{eq:VFBP-La-eta} as follows:
\begin{Definition}\label{definition-lag}
Let $T>0$. A vector function $(U,\eta)(t,r)$ is called 
a classical solution of {\rm\bf IBVP} \eqref{eq:VFBP-La-eta} 
in $[0,T]\times \bar I$ if the following properties hold{\rm:}
\begin{enumerate}
\item[{\rm (i)}] $(U,\eta)(t,r)$ satisfies equations 
$\eqref{eq:VFBP-La-eta}_1${\rm--}$\eqref{eq:VFBP-La-eta}_2$ pointwise in $(0,T]\times \bar I$ {\rm(}at $r=1$, the right-hand side of $\eqref{eq:VFBP-La-eta}_1$ is understood as a whole by continuous extension from $r<1${\rm)} and takes the initial data $\eqref{eq:VFBP-La-eta}_3$ continuously{\rm ;}
\item[{\rm (ii)}] $\eta_r(t,r)$ and $\frac{\eta}{r}(t,r)$ are strictly positive in $[0,T]\times \bar I${\rm:}
\begin{equation*}
\inf_{[0,T]\times \bar I} \ \eta_r >0, \qquad \inf_{[0,T]\times \bar I} \ \frac{\eta}{r}>0;
\end{equation*}
\item[{\rm (iii)}] $(U,\eta)(t,r)$ satisfies the following regularity properties{\rm:}
\begin{equation*}
\begin{aligned}
&\big(U,U_r,\frac{U}{r}\big)\in C([0,T];C(\bar I)),\qquad \big(U_{rr},(\frac{U}{r})_r,U_t\big)\in C((0,T];C(\bar I)),\\
&\big(\eta,\eta_r,\frac{\eta}{r}\big)\in C^1([0,T];C(\bar I)),\qquad \,\big(\eta_{rr},(\frac{\eta}{r})_r\big)\in
C^1((0,T];C(\bar I)).
\end{aligned}
\end{equation*}
\end{enumerate}
\end{Definition}

Next, in order to clearly state our main results, we need to define the following nonlinear weighted energy functionals and related parameters:
\begin{itemize}
\item A universal parameter $\varepsilon_0>0$ throughout the paper, which satisfies
\begin{equation}\label{varepsilon0}
\begin{cases}
\displaystyle 0<\varepsilon_0 < \min\big\{\frac{3}{2}-\frac{1}{2\beta},\frac{\gamma-1}{\beta}-1,\frac{1}{2}\big\}& \displaystyle\quad\text{for }\beta\in \big(\frac{1}{3},\gamma-1\big),\\[6pt]
\displaystyle 0<\varepsilon_0< \min\big\{\frac{3}{2}-\frac{1}{2\beta},\frac{1}{2}\big\}& \quad\text{for }\beta=\gamma-1.
\end{cases}
\end{equation}

\smallskip
\item The total energy:
\begin{equation}\label{E-1}
\cE(t,f)=\cE_{\mathrm{in}}(t,f)+\cE_{\mathrm{ex}}(t,f),
\end{equation}
where 
\begin{equation}\label{E-2}
\begin{aligned}
\qquad \, \mathcal{E}_{\mathrm{in}}(t,f)&:=\Big\|\zeta r^\frac{m}{2}\big(f,D_\eta f,\frac{f}{\eta},f_t,D_\eta f_t,\frac{f_t}{\eta}\big)(t)\Big\|_{L^2(I)}^2\\
&\quad +\Big\|\zeta r^\frac{m}{2}\Big(D_\eta^2 f, D_\eta\big(\frac{f}{\eta}\big),D_\eta^3 f, D_\eta^2\big(\frac{f}{\eta}\big),\frac{1}{\eta}D_\eta\big(\frac{f}{\eta}\big) \Big)(t)\Big\|_{L^2(I)}^2,\\
\qquad  \, \mathcal{E}_{\mathrm{ex}}(t,f)&:=\big\|\rho_0^\frac{1}{2}(f,D_\eta f,f_t,D_\eta f_t)(t)\big\|_{L^2(\frac{1}{2},1)}^2+\big\|\rho_0^{(\frac{3}{2}-\varepsilon_0)\beta}(D_\eta^2 f,D_\eta^3 f)(t)\big\|_{L^2(\frac{1}{2},1)}^2,
\end{aligned}
\end{equation}
and $\zeta=\zeta(r)\in C^\infty[0,1]$ denotes a decreasing cut-off function satisfying
\begin{equation}\label{zeta}
\zeta\in [0,1],\qquad \zeta(r)=1 \ \ \text{for  $r\in \big[0,\frac{1}{2}\big]$},\qquad \zeta(r)=0 \ \ \text{for $r\in \big[\frac{5}{8},1\big]$}.
\end{equation}

\smallskip
\item The total dissipation:
\begin{equation}\label{D-1}
\cD(t,f)=\cD_{\mathrm{in}}(t,f)+\cD_{\mathrm{ex}}(t,f),
\end{equation}
where 
\begin{equation}\label{D-2}
\begin{aligned}
\qquad\cD_{\mathrm{in}}(t,f)&:=\Big\|\zeta r^\frac{m}{2}\Big(f_{tt},D_\eta^2 f_{t},D_\eta\big(\frac{f_{t}}{\eta}\big),D_\eta^4 f, D_\eta^3\big(\frac{f}{\eta}\big),D_\eta\big(\frac{1}{\eta}D_\eta(\frac{f}{\eta})\big) \Big)(t)\Big\|_{L^2(I)}^2,\\
\qquad\cD_{\mathrm{ex}}(t,f)&:=\big\|\rho_0^\frac{1}{2}f_{tt}(t)\big\|_{L^2(\frac{1}{2},1)}^2+\big\|\rho_0^{(\frac{3}{2}-\varepsilon_0)\beta}(D_\eta^2 f_t,D_\eta^4 f)(t)\big\|_{L^2(\frac{1}{2},1)}^2.
\end{aligned}
\end{equation}
\end{itemize}

We are now ready to state the main results in Lagrangian coordinates.  
\begin{Theorem}\label{Theorem1.1} 
Let $n=2$ or $3$ and 
\begin{equation}\label{gamma-hold}
\gamma\in (\frac{4}{3},\infty) \ \ \text{if $n=2$}, \qquad\quad \gamma\in (\frac{4}{3},3) \ \ \text{if $n=3$}.
\end{equation}
If $\rho_0(r)$ satisfies \eqref{distance-la} for some
\begin{equation}\label{beta}
\beta\in \big(\frac{1}{3},\gamma-1\big],
\end{equation}
and $u_0(r)$ satisfies
\begin{equation}\label{a2}
\cE(0,U)<\infty,
\end{equation}
then, for any $T>0$, {\rm\bf IBVP} \eqref{eq:VFBP-La-eta} admits a unique classical solution $(U,\eta)(t,r)$ in $[0,T]\times \bar I$ such that
\begin{equation}\label{b1}
\begin{aligned}
&\sup_{t\in[0,T]}\big(\cE(t,U)+t\cD(t,U)\big)+\int_0^T\cD (s,U)\,\mathrm{d}s\leq C(T),\\
&(\eta_r,\frac{\eta}{r})(t,r)\in [C^{-1}(T),C(T)] \qquad \text{for all $(t,r)\in [0,T]\times \bar I$},
\end{aligned}
\end{equation}
where $C(T)>1$ is a constant  depending  only on  $(n,\mu,A,\gamma,\beta,\varepsilon_0,\rho_0,u_0,\cK_1, \cK_2,T)$. Moreover, such a classical solution admits the following boundary conditions{\rm:}
\begin{equation}\label{N111}
U|_{r=0}=U_r|_{r=1}=0\qquad \text{on $(0,T]$}.
\end{equation}
\end{Theorem}

We make some remarks on the results in Theorem \ref{Theorem1.1}.

\begin{Remark}\label{remark-energy function}
We first give some explanations for the constraint of $(\gamma,\beta)$ shown in \eqref{gamma-hold}--\eqref{beta} and the form of energy functionals given in \eqref{E-1}{\rm--}\eqref{D-2}.
Due to the coordinate singularity at the origin and the strong degeneracy on the boundary, we need to define the energy functionals separately in their respective neighborhoods and then combine them appropriately.

First, in the region including the origin, if we consider the corresponding {\rm\bf{IBVP}} in M-D Lagrangian coordinates, the analysis will be very clear. To this end, denote by $\boldsymbol{x}=\boldsymbol{\eta}(t,\boldsymbol{y})$ the position of the fluid particle $\boldsymbol{x} \in \Omega(t)$  so that
\begin{equation*}
\boldsymbol{\eta}_t(t,\boldsymbol{y})=\boldsymbol{u}(t,\boldsymbol{\eta}(t,\boldsymbol{y})) \ \ \text{for $t>0$},\qquad \text{with $\boldsymbol{\eta}(0,\boldsymbol{y})=\boldsymbol{y}$},
\end{equation*}
where  $(t,\boldsymbol{y})$ are the M-D Lagrangian coordinates. Then, by introducing the M-D Lagrangian density and velocity 
\begin{equation*}
\varrho(t,\boldsymbol{y})=\rho(t,\boldsymbol{\eta}(t,\boldsymbol{y})),\qquad \boldsymbol{U}(t,\boldsymbol{y})=\boldsymbol{u}(t,\boldsymbol{\eta}(t,\boldsymbol{y})),
\end{equation*}
one can obtain {\rm\bf{IBVP}} \eqref{eq:1.1-la-vfbp}  in  M-D Lagrangian coordinates from {\rm\bf{VFBP}} \eqref{eq:1.1-vfbp} in M-D Eulerian coordinates {\rm(}see {\rm Appendix \ref{appb})}.  Starting from {\rm\bf{IBVP}} \eqref{eq:1.1-la-vfbp}, we consider the region $B:=\{\boldsymbol{y}:\,|\boldsymbol{y}|<\frac{1}{2}\}$, which includes the origin. To ensure that $\boldsymbol{U}\in C^2(\overline B)$ and $\boldsymbol{U}_t\in C(\overline B)$ in positive time, by the classical Sobolev embedding theorem, it suffices to require $\boldsymbol{U}$ to have the following regularity in $H^k(B)${\rm :}  
\begin{equation}\label{regularity-La-M}
\begin{aligned}
\|\boldsymbol{U}(t)\|_{H^3(B)}^2+t\|\boldsymbol{U}(t)\|_{H^4(B)}^2+\|\boldsymbol{U}_t(t)\|_{H^1(B)}^2+t\|\boldsymbol{U}_t(t)\|_{H^2(B)}^2 \in L^\infty(0,T).
\end{aligned}
\end{equation}
The detailed proof can be found in  {\rm \S \ref{Section-global}}. According to {\rm Lemma \ref{lemma-initial}} in {\rm Appendix \ref{appb}} and {\rm Lemma \ref{lemma-gaowei}} in {\rm Appendix \ref{AppB}}, for spherically symmetric solutions, the regularity requirement \eqref{regularity-La-M} in the radial coordinate can be read as 
\begin{equation*}
\cE_{\mathrm{in}}(t,U)+t\cD_{\mathrm{in}}(t,U)\in L^\infty(0,T),
\end{equation*}
except for some additional information on $U_{tt}$.  The introduction of the decreasing cut-off function $\zeta$ in $(\cE_{\mathrm{in}},\cD_{\mathrm{in}})(t,U)$ is to ensure that the interior and exterior energy functionals can be effectively spliced into an energy functional on $I$.

Second, to ensure that the solution is classical in the region  $B^\sharp:=\{\boldsymbol{y}:\,\frac{1}{2}\leq |\boldsymbol{y}|<1\}$ and smooth up to the vacuum boundary at least when $t>0$, we establish some weighted $H^4(\frac{1}{2},1)$-estimates on $U$ inspired by the following embedding relation, due to the Hardy and Sobolev inequalities {\rm(}see {\rm Lemmas \ref{sobolev-embedding}--\ref{hardy-inequality}} in {\rm Appendix \ref{appendix A}}{\rm)}{\rm:}
\begin{equation*}
H^4_{\rho_0^{2\alpha}}\big(\frac{1}{2},1\big)\hookrightarrow W^{3,1}\big(\frac{1}{2},1\big)\hookrightarrow C^2\big[\frac{1}{2},1\big] \qquad\text{for some $\alpha<\frac{3}{2}\beta$ and $\big|\alpha-\frac{3}{2}\beta\big|\ll 1$},
\end{equation*}
where $|a_1-a_2|\ll 1$ denotes that $a_1$ is sufficiently close to $a_2$, and the weighted Sobolev space $H^4_{\rho_0^{2\alpha}}(\frac{1}{2},1)$ is defined by 
\begin{equation*}
H^4_{\rho_0^{2\alpha}}\big(\frac{1}{2},1\big)
:=\big\{f\in L^1_{\mathrm{loc}}\big(\frac{1}{2},1\big): \,\rho_0^{\alpha}\partial_r^j f\in L^2\big(\frac{1}{2},1\big) 
\ \text{for $0\leq j\leq 4$}\}.
\end{equation*}
Then the natural exterior energy and dissipation take the form{\rm:}
\begin{equation*}
\begin{aligned}
\cE^*_{\mathrm{ex}}(t,U)&=\big\|\rho_0^\frac{1}{2}(U,D_\eta U,U_t,D_\eta U_t)(t)\big\|_{L^2(\frac{1}{2},1)}^2+
\big\|\rho_0^{\alpha}(D_\eta^2 U,D_\eta^3 U)(t)\big\|_{L^2(\frac{1}{2},1)}^2,\\
\cD^*_{\mathrm{ex}}(t,U)&=\big\|\rho_0^\frac{1}{2}U_{tt}(t)\big\|_{L^2(\frac{1}{2},1)}^2+\big\|\rho_0^{\alpha}(D_\eta^2 U_t,D_\eta^4 U)(t)\big\|_{L^2(\frac{1}{2},1)}^2,
\end{aligned}
\end{equation*}
where $\alpha$ can be determined via the balance between the pressure and spatial dissipation. In fact, to derive the highest-order elliptic estimates, multiply $\eqref{eq:VFBP-La-eta}_1$ by $\varrho^{\beta-1}${\rm:}
\begin{equation}\label{reform}
2\mu \varrho^\beta D_\eta^2 U =\varrho^{\beta}U_t+\frac{A\gamma}{\beta}\varrho^{\gamma-1}D_\eta(\varrho^\beta)-\frac{2\mu}{\beta} D_\eta(\varrho^\beta) D_\eta U-2m\mu \varrho^\beta D_\eta\big(\frac{U}{\eta}\big).
\end{equation}
Then we formally obtain from the above that
\begin{equation*}
\begin{aligned}
\varrho^\alpha D_\eta^4 U &= \frac{A\gamma(\gamma-1)(\gamma-1-\beta)}{2\mu\beta^3}\varrho^{\alpha+\gamma-1-3\beta}(D_\eta(\varrho^\beta))^3+\frac{1}{2\mu}\varrho^\alpha D_\eta^2 U_{t}+\mathrm{R}_{(\alpha)}\\
&= \underline{\frac{A\gamma(\gamma-1)(\gamma-1-\beta)}{2\mu\beta^3}\varrho^{\alpha+\gamma-1-3\beta}(D_\eta(\varrho^\beta))^3}_{:=\clubsuit_1} +\underline{\frac{1}{4\mu^2}\varrho^\alpha U_{tt}}_{:=\clubsuit_2}+\widetilde{\mathrm{R}}_{(\alpha)},
\end{aligned}
\end{equation*}
where $(\mathrm{R}_{(\alpha)},\widetilde{\mathrm{R}}_{(\alpha)})$ denote some harmless terms which possess either higher-order $\varrho$-weights or lower-order derivatives of $(\varrho,U)$.  As can be checked, $\clubsuit_1$ is the most singular part of the pressure's derivatives, and $\clubsuit_2$ is the highest-order tangential derivative. These two terms present the main obstacles in controlling the $L^2(\frac{1}{2},1)$-norm of $\varrho^\alpha D_\eta^4 U$. 
Thus, \eqref{distance-la} leads to 
\begin{equation*}
\begin{aligned}
\clubsuit_1\in L^2\big(\frac{1}{2},1\big) &\impliedby \alpha> \frac{5\beta}{2}-\gamma+1 \ \text{ or } \ \beta=\gamma-1;\\
\clubsuit_2\in L^2\big(\frac{1}{2},1\big) &\impliedby \alpha\geq \frac{1}{2} .
\end{aligned}
\end{equation*}
Thus, in order to choose such an admissive $\alpha$, we need to restrict $\beta\in (\frac{1}{3},\gamma-1]$. Finally, by setting $\alpha:=(\frac{3}{2}-\varepsilon_0)\beta<\frac{3}{2}\beta$ with some suitable small $\varepsilon_0$ defined in \eqref{varepsilon0}, we recover the desired exterior energy and dissipation $(\cE_{\mathrm{ex}},\cD_{\mathrm{ex}})(t,U)$ in \eqref{E-2} and \eqref{D-2}.
\end{Remark}

\begin{Remark}\label{moregamma}
The constraint $\beta\in (\frac{1}{3},\gamma-1]$ in \eqref{beta} not only plays a key role in the choice of the energy functionals given in \eqref{E-1}{\rm--}\eqref{D-2} {\rm (}see {\rm Remark \ref{remark-energy function})}, but is also crucial for obtaining the desired global-in-time uniform energy estimates {\rm (}see {\rm \S \ref{Section-global})}. Whether the methodology developed in this paper can be adapted to the case $0<\beta\leq \min\{\frac{1}{3},\gamma-1\}$ and $\gamma>1$ remains unclear.
We leave this as an open problem for future investigation.
\end{Remark}

\begin{Remark}\label{initialexample}
We give some examples of the initial data required in {\rm Theorem \ref{Theorem1.1}}.
In {\rm Appendix \ref{AppB}}, we present several equivalent formulations of the energy functionals defined 
in \eqref{E-1}{\rm--}\eqref{D-2}, which will be used frequently in the subsequent analysis. 
Moreover, by {\rm Lemmas \ref{reduce-ex}--\ref{reduce-in}} in {\rm Appendix \ref{AppB}},
we find that {\rm Theorem \ref{Theorem1.1}}
can be established, if  
the initial data
belong to 
the following class{\rm :}
\begin{equation}\label{initial-ex1}
\rho_0(\boldsymbol{y})=\rho_0(r)\qquad \text{with} \ \ \rho_0(r)=(1-r^{2k})^\frac{1}{\beta} \ \ \text{for $k\in \NN^*$},
\end{equation}
and $\boldsymbol{u}_0(\boldsymbol{y})=u_0(r)\frac{\boldsymbol{y}}{r}$ with 
\begin{equation}\label{initial-ex2}
\begin{aligned}
u_0(r)&= \tilde{u}_0(r) &&\quad \text{if $\beta\in \big(\frac{1}{3},\frac{2\gamma-1}{5}\big)$ or $\beta=\gamma-1$},\\
u_0(r)&= -\zeta^\sharp_\frac{1}{3} \frac{A}{2\mu}\int_r^1 \rho_0^{\gamma-1}\,\mathrm{d}\tilde{r}+\tilde{u}_0(r) &&\quad \text{if $\beta\in \big[\frac{2\gamma-1}{5},\gamma-1\big)$},    
\end{aligned}
\end{equation}
where $\zeta^\sharp_{\frac{1}{3}}$ is a smooth cut-off function such that $\zeta^\sharp_{\frac{1}{3}}=0$ on $[0,\frac{1}{3}]$ and $\zeta^\sharp_{\frac{1}{3}}=1$ on $[\frac{1}{2},1]$, and  $\tilde{u}_0(r)$ is a function such that $\tilde{\boldsymbol{u}}_0(\boldsymbol{y})=\tilde{u}_0(r)\frac{\boldsymbol{y}}{r}\in C_\mathrm{c}^\infty(\Omega)$; see {\rm Remark \ref{initalexample3}} for the detailed proof.
\end{Remark}

\begin{Remark}\label{rmk1.4}
We briefly explain how the boundary condition $U_r|_{r=1}=0$ in \eqref{N111} can be derived. First, it follows from \eqref{flowmap-r-la}, \eqref{distance-la}, $\eqref{b1}_2$, and {\rm Definition \ref{definition-lag}} that 
\begin{equation}\label{eq117}
\begin{gathered}
(\varrho^\beta, \, D_\eta(\varrho^\beta),\, U,\ D_{\eta}U, \, D_{\eta}^2U, \, D_{\eta}\big(\frac{U}{\eta}\big), \, U_t)\in  C((0,T]\times [\frac{1}{2},1]).
\end{gathered}
\end{equation}
Next, taking the limit $r\to 1$ in \eqref{reform} 
and using \eqref{eq117} together with the established lower bounds of  $(\eta,\eta_r)$ near the boundary, we obtain 
\begin{equation*}
D_\eta(\varrho^\beta) D_\eta U|_{r=1}=0.
\end{equation*}
Since $\rho_0^\beta \sim 1-r$ and, by \eqref{eq:eta}, $D_\eta(\varrho^\beta)|_{r=1}\neq 0$, 
it follows immediately that $U_r|_{r=1}=0$. 
We emphasize that the boundary condition \eqref{N111} plays a crucial role in establishing the uniform lower bounds of $(\eta_r,\frac{\eta}{r})$ in {\rm \S \ref{Section-densityupper}} and their uniform upper bounds in {\rm \S \ref{Section-etarupper}}.
\end{Remark}

\medskip
\subsection{Applications to the 2-D shallow water equations (the viscous Saint-Venant system)}
Another aim of this paper is to establish
the global existence and uniqueness  of spherically symmetric classical solutions of \textbf{VFBP} with large initial data for the shallow water system \eqref{shallow1}:
\begin{equation}\label{eq:1.1-vfbp-shallow}
\begin{cases}
h_t+\dive(h \boldsymbol{u})=0 &\text{in }\Omega(t),\\[4pt]
(h \boldsymbol{u})_t+\dive(h \boldsymbol{u}\otimes \boldsymbol{u})+A\nabla h^2 =V(h,\boldsymbol{u})&\text{in }\Omega(t),\\[4pt]
h>0&\text{in }\Omega(t),\\[4pt]
h=0&\text{on }\partial\Omega(t),\\[4pt]
\cV(\partial\Omega(t))=\boldsymbol{u}\cdot\boldsymbol{n}(t)&\text{on }\partial\Omega(t),\\[4pt]
(h,\boldsymbol{u})|_{t=0}=(h_0,\boldsymbol{u}_0) &\text{in }\Omega:=\Omega(0).
\end{cases}
\end{equation}
Problem \eqref{eq:1.1-vfbp-shallow} is a special case of \textbf{VFBP} \eqref{eq:1.1-vfbp} with $\gamma=n=2$. 
For the spherically symmetric flow, since $D(\boldsymbol{u})=\nabla \boldsymbol{u}$, the viscosity term $V(h,\boldsymbol{u})$ in $\eqref{eq:1.1-vfbp-shallow}_2$ satisfies
\begin{equation*}
{V}(h, \boldsymbol{u})=2\mu\dive(h D(\boldsymbol{u}))
= 2\mu\dive(h\nabla \boldsymbol{u}).
\end{equation*}

We establish the global existence of the unique classical solution, taking the form: 
\begin{equation}\label{ss-ass-shallow}
(h,\boldsymbol{u})(t,\boldsymbol{x}) = (h(t,|\boldsymbol{x}|), u(t,|\boldsymbol{x}|)\frac{\boldsymbol{x}}{|\boldsymbol{x}|}),
\end{equation}
of \textbf{VFBP} \eqref{eq:1.1-vfbp-shallow} with the initial data:
\begin{equation}\label{eq:IC-shallow}
(h,\boldsymbol{u})(0,\boldsymbol{x}) =(h_0,\boldsymbol{u}_0)(\boldsymbol{x})= (h_0(|\boldsymbol{x}|), u_0(|\boldsymbol{x}|)\frac{\boldsymbol{x}}{|\boldsymbol{x}|}).
\end{equation}
The  initial depth  $h_0$ we consider satisfies the following condition:
\begin{equation}\label{distanceeuler-shallow}
h_0^\beta(\boldsymbol{x})\in H^3(\Omega), \quad\, \cK_1(1-|\boldsymbol{x}|)^\frac{1}{\beta}\leq h_0(\boldsymbol{x})\leq \cK_2(1-|\boldsymbol{x}|)^\frac{1}{\beta} \qquad\, 
\text{for all $\boldsymbol{x}\in \overline\Omega$},
\end{equation}
for some  constants $\cK_2>\cK_1>0$ and $\beta\in (\frac{1}{3},1]$. 
It is worth emphasizing that no restriction is imposed on the size of the initial data in our result, and the solutions obtained for \eqref{eq:1.1-vfbp-shallow} remain smooth all the way up to the moving boundary.

Next, following the reformulation in \S\ref{sec-1.2}, we rewrite  \eqref{eq:1.1-vfbp-shallow} into a problem on $I=[0,1)$, that is, problem \eqref{eq:VFBP-La} with $(m,\gamma)=(1,2)$ and $\varrho$ replaced by $\hat h$. In this case, $(\hat h,U)$ denote the Lagrangian depth and horizontal velocity, respectively, which are defined by
\begin{equation*}
\hat h (t,r)=h(t,\eta(t,r)),\qquad U(t,r)=u(t,\eta(t,r)). 
\end{equation*}
Then, following a discussion similar to that in 
\eqref{eq:eta}--\eqref{eq:VFBP-La-eta}, we arrive at the following \textbf{IBVP} for $(U, \eta)$:
\begin{equation}\label{eq:shallow}
\begin{cases}
\displaystyle \hat h U_t +A D_\eta(\hat h^{2})=2\mu D_\eta\Big(\hat h\big(D_\eta U+ \frac{U}{\eta}\big)\Big) - 2\mu \frac{ D_\eta\hat h U}{\eta}& \text{in $(0, T]\times I$},\\[5pt]
\eta_t = U & \text{in $(0, T] \times I$},\\[5pt]
(U, \eta)(0,r)= (u_0(r), r) & \text{for $r\in I$}, 
\end{cases}
\end{equation}
where $\hat h$ is given  by
\begin{equation*}
\hat h(t,r)=\frac{rh_0(r)}{\eta\eta_r}.
\end{equation*}
Moreover, according to  Lemma \ref{lemma-initial} and \eqref{distanceeuler-shallow},  $h_0(r)$ satisfies 
\begin{equation}\label{distance-las}
\begin{aligned}
&r^\frac{1}{2}\big(h_0^\beta,(h_0^\beta)_r,(h_0^\beta)_{rr},\frac{(h_0^\beta)_r}{r},(h_0^\beta)_{rrr},(\frac{(h_0^\beta)_r}{r})_r\big)\in L^2(I),\\
&\cK_1(1-r)^\frac{1}{\beta}\leq h_0(r)\leq \cK_2(1-r)^\frac{1}{\beta} \qquad \text{for all $r\in I$}.
\end{aligned}
\end{equation}

In order to construct smooth solutions to \textbf{IBVP} \eqref{eq:shallow}, we similarly define:
\begin{itemize}
\item The total energy:
\begin{equation*} 
\begin{aligned}
\qquad \cE_{\mathrm{sw}}(t,f)&:=\Big\|\zeta r^\frac{1}{2}\big(f,D_\eta f,\frac{f}{\eta},f_t,D_\eta f_t,\frac{f_t}{\eta}\big)(t)\Big\|_{L^2(I)}^2\\
&\quad +\Big\|\zeta r^\frac{1}{2}\Big(D_\eta^2 f, D_\eta\big(\frac{f}{\eta}\big),D_\eta^3 f, D_\eta^2\big(\frac{f}{\eta}\big),\frac{1}{\eta}D_\eta\big(\frac{f}{\eta}\big) \Big)(t)\Big\|_{L^2(I)}^2\\
&\quad +\big\|h_0^\frac{1}{2}(f,D_\eta f,f_t,D_\eta f_t)(t)\big\|_{L^2(\frac{1}{2},1)}^2+\big\|h_0^{(\frac{3}{2}-\varepsilon_0)\beta}(D_\eta^2 f,D_\eta^3 f)(t)\big\|_{L^2(\frac{1}{2},1)}^2,
\end{aligned}
\end{equation*}
where $\varepsilon_0$ and $\zeta$ are defined in \eqref{varepsilon0} and \eqref{zeta}, respectively.

\smallskip
\item The total dissipation:
\begin{equation*} 
\begin{aligned}
\cD_{\mathrm{sw}}(t,f)&:=\Big\|\zeta r^\frac{1}{2}\Big(f_{tt},D_\eta^2 f_{t},D_\eta\big(\frac{f_{t}}{\eta}\big),D_\eta^4 f, D_\eta^3\big(\frac{f}{\eta}\big),D_\eta\big(\frac{1}{\eta}D_\eta(\frac{f}{\eta})\big) \Big)(t)\Big\|_{L^2(I)}^2\\
&\quad +\big\|h_0^\frac{1}{2}f_{tt}(t)\big\|_{L^2(\frac{1}{2},1)}^2+\big\|h_0^{(\frac{3}{2}-\varepsilon_0)\beta}(D_\eta^2 f_t,D_\eta^4 f)(t)\big\|_{L^2(\frac{1}{2},1)}^2.
\end{aligned}
\end{equation*}
\end{itemize}

Besides, the classical solutions of \textbf{IBVP} \eqref{eq:shallow} can
be defined analogously to  Definition \ref{definition-lag} with  $\gamma=n=2$, $(\varrho,U,\eta)$ replaced by $(\hat h,U,\eta)$, and \eqref{eq:VFBP-La-eta} replaced by \eqref{eq:shallow}. Then, from Theorem \ref{Theorem1.1}, the following conclusion  holds:
\begin{Theorem}\label{Theorem1.2} 
If $h_0(r)$ satisfies \eqref{distance-las} for some
\begin{equation}
\beta\in \big(\frac{1}{3},1\big],
\end{equation}
and $u_0(r)$ satisfies
\begin{equation}\label{a2-shallow}
\cE_{\mathrm{sw}}(0,U)<\infty,
\end{equation}
then, for any $T>0$, {\rm\bf IBVP} \eqref{eq:shallow} admits a unique classical solution $(U,\eta)(t,r)$ in $[0,T]\times \bar I$ satisfying \eqref{N111} and 
\begin{equation}\label{b1s}
\begin{gathered}
\sup_{t\in[0,T]}\big(\cE_{\mathrm{sw}}(t,U)+t\cD_{\mathrm{sw}}(t,U)\big)+\int_0^T\cD_{\mathrm{sw}} (s,U)\,\mathrm{d}s\leq C(T),\\
\quad (\eta_r,\frac{\eta}{r})(t,r)\in [C^{-1}(T),C(T)] \qquad\, \text{for all $(t,r)\in [0,T]\times \bar I$},
\end{gathered}
\end{equation}
where $C(T)>1$ is a constant  depending  only on  $(\mu,A,\beta,\varepsilon_0,h_0,u_0,\cK_1,\cK_2,T)$. 
\end{Theorem}

\medskip
\subsection{Main results in Eulerian coordinates}

Denote 
\begin{equation*}
\mathbb{E}(T)=\{(t,\boldsymbol{x}) :\,  t\in (0,T], \, \boldsymbol{x}\in \overline\Omega(t)\},\quad \mathbb{E}^\circ (T)=\{(t,\boldsymbol{x}) :\,  t\in (0,T], \, \boldsymbol{x}\in \Omega(t)\}.
\end{equation*}
The classical solutions of {\rm\textbf{VFBP}} \eqref{eq:1.1-vfbp} in $\mathbb{E}(T)$ can be defined as follows:

\begin{Definition}\label{definition-M-euler}
Let $T>0$. A triple $(\rho(t,\boldsymbol{x}),\boldsymbol{u}(t,\boldsymbol{x}),\partial \Omega(t))$ is said to be a classical solution of {\rm\textbf{VFBP}} \eqref{eq:1.1-vfbp} in $\overline{\mathbb{E}(T)}$ if  
\begin{enumerate}
\item[{\rm (i)}] the equations   $\eqref{eq:1.1-vfbp}_1${\rm--}$\eqref{eq:1.1-vfbp}_2$ hold pointwise in $\mathbb{E}^\circ(T)$, where $\rho>0$.  The initial data $\eqref{eq:1.1-vfbp}_6$ are attained continuously at $t=0$, and the boundary conditions $\eqref{eq:1.1-vfbp}_4${\rm--}$\eqref{eq:1.1-vfbp}_5$ hold continuously on $\Gamma(T)=\{(t,\boldsymbol{x}) :\,  t\in (0,T], \, \boldsymbol{x}\in \partial \Omega(t)\}${\rm;}
\vspace{3pt}
\item[{\rm (ii)}]  the moving boundary $\partial \Omega(t)\in C^2((0,T])${\rm;}
\vspace{3pt}
\item[{\rm (iii)}] all the terms in equations $\eqref{eq:1.1-vfbp}_1${\rm--}$\eqref{eq:1.1-vfbp}_2$  are continuous in $\mathbb{E}(T)${\rm :}
\begin{equation*}
\begin{split}
(\rho, \,\rho_t, \,\nabla \rho,\,\boldsymbol{u}, \,\nabla \boldsymbol{u}, \,\nabla^2 \boldsymbol{u}, \,\boldsymbol{u}_t)
\in C(\mathbb{E}(T)).
\end{split}
\end{equation*}
\end{enumerate}
\end{Definition}

Now, our main result on the global existence of the unique classical solution to problem   \eqref{eq:1.1-vfbp} with large initial data of spherical symmetry can be stated in M-D Eulerian coordinates, which can be derived from Theorem \ref{Theorem1.1} and Lemma \ref{lemma-B3}, as follows:
\begin{Theorem}\label{theorem1.3}
Let $n=2$ or $3$ and \eqref{gamma-hold} hold. Let $(\rho_0,\boldsymbol{u}_0)(\boldsymbol{x})$ be spherically symmetric, take the form \eqref{eq:IC}, and satisfy \eqref{distanceeuler} for some $\beta\in (\frac{1}{3},\gamma-1]$ and \eqref{a2}. Then
\begin{enumerate}
\item[{\rm(i)}] If $\beta\leq 1$, for any $T>0$, {\rm\textbf{VFBP}} \eqref{eq:1.1-vfbp} admits a unique spherically symmetric classical solution $(\rho(t,\boldsymbol{x}),\boldsymbol{u}(t,\boldsymbol{x}),\partial \Omega(t))$ in $\overline{\mathbb{E}(T)}$ of the form \eqref{ss-ass}{\rm;}  
\smallskip
\item[{\rm(ii)}] If $\beta>1$, for any $T>0$, {\rm\textbf{VFBP}} \eqref{eq:1.1-vfbp} admits a unique spherically symmetric solution $(\rho(t,\boldsymbol{x}),\boldsymbol{u}(t,\boldsymbol{x}),\partial \Omega(t))$ in $\overline{\mathbb{E}(T)}$ of the form \eqref{ss-ass}, which satisfies {\rm(i)--(ii)} of {\rm Definition \ref{definition-M-euler}} and
\begin{equation*}
(\rho, \,\rho_t+\boldsymbol{u}\cdot\nabla \rho, 
\,\boldsymbol{u}, \,\nabla \boldsymbol{u}, 
\,\nabla^2 \boldsymbol{u}, \,\boldsymbol{u}_t) \in C(\mathbb{E}(T)).
\end{equation*} 
\end{enumerate}
Moreover,  $\boldsymbol{u}$ satisfies the boundary conditions
\begin{equation}\label{N111euler}
\boldsymbol{u}|_{\boldsymbol{x}=\boldsymbol{0}}=\nabla \boldsymbol{u}\cdot \boldsymbol{n}|_{\boldsymbol{x}\in \partial\Omega(t)}=\boldsymbol{0} \qquad \text{for all }t\in (0,T].
\end{equation}
\end{Theorem}

Second,  the  classical solutions of \textbf{VFBP} \eqref{eq:1.1-vfbp-shallow} can
be defined analogously to  Definition \ref{definition-M-euler} with  $\gamma=n=2$, $(\rho,\boldsymbol{u},\partial \Omega(t))$ replaced by $(h,\boldsymbol{u},\partial \Omega(t))$, and \eqref{eq:1.1-vfbp} replaced by \eqref{eq:1.1-vfbp-shallow}. Then the global existence  of the unique classical solution to problem  \eqref{eq:1.1-vfbp-shallow} with large data of spherical symmetry can be given in  M-D Eulerian coordinates as follows:

\begin{Theorem}\label{theorem1.3-shallow}
Let $(h_0,\boldsymbol{u}_0)(\boldsymbol{x})$ be spherically symmetric, take the form \eqref{eq:IC-shallow}, and satisfy \eqref{distanceeuler-shallow} for some  $\beta\in (\frac{1}{3},1]$ and \eqref{a2-shallow}. Then, for any $T>0$, there exists a unique spherically symmetric classical solution $(h(t,\boldsymbol{x}),\boldsymbol{u}(t,\boldsymbol{x}),\partial \Omega(t))$ of the form \eqref{ss-ass-shallow}, in $\overline{\mathbb{E}(T)}$, of {\rm\textbf{VFBP}} \eqref{eq:1.1-vfbp-shallow}. Moreover, $\boldsymbol{u}$ satisfies  \eqref{N111euler}.
\end{Theorem}

\smallskip
\begin{Remark}
For  {\rm\textbf{VFBP}} \eqref{eq:1.1-vfbp}, it follows from \eqref{distanceeuler} and {\rm Theorem \ref{theorem1.3}}  that  the usual
stress-free boundary condition holds automatically{\rm:}
\begin{equation*}
(\mathbb{T}-P\II_n)\cdot \boldsymbol{n} = (2\mu\rho D(\boldsymbol{u})-A\rho^\gamma\II_n)\cdot \boldsymbol{n}=\boldsymbol{0} \qquad \text{for $t\in (0,T]$ and $\boldsymbol{x}\in \partial\Omega(t)$},
\end{equation*}
where $\mathbb{T}$ is the viscous stress tensor and  $\II_n$ denotes the $n\times n$ unit matrix.
\end{Remark}

\begin{Remark}\label{generalcase}
Under proper modifications, the methodology developed in this paper can be applied to establishing the global existence and uniqueness   of classical solutions with general smooth, spherically symmetric initial data  of the corresponding {\rm \textbf{VFBP}} of the barotropic {\rm\textbf{CNS}} 
with nonlinear density-dependent viscosity coefficients in two and three spatial dimensions, which is addressed in \cite{CZZ1}.
\end{Remark}

\section{Notations and Main Strategies}\label{Section-notation}

In this section, we first present some notations in \S\ref{section-notaions}, which will be frequently used throughout this paper.  In \S\ref{subsection-strategy}, we show the main strategies and new ideas in our analysis.

\subsection{Notations}\label{section-notaions}
The following notations will be frequently used in this paper.
\subsubsection{Notations on coordinates and operators}
\begin{itemize}
\item We always let $n=2$ or $3$ be the dimension number, and denote $m:=n-1$.
\smallskip
\item $\boldsymbol{x}\in \RR^n$ denotes the M-D Eulerian spatial coordinates. $\boldsymbol{y}\in \Omega:=\{\boldsymbol{y}:\,|\boldsymbol{y}|<1\}$ denotes the M-D Lagrangian spatial coordinates.
\smallskip
\item $I:=[0,1)$, $r=|\boldsymbol{y}|\in I$ denotes the radial coordinate.
\smallskip
\item For any function $f$ defined on a measurable subset of $\mathbb{R}^l$ ($l\geq 1$), if the independent variables of $f$ are $\boldsymbol{z}=(z_1,\cdots\!,z_l)^\top$, then 
\begin{equation*}
\begin{aligned}
&\qquad \partial_{\boldsymbol{z}}^{\boldsymbol{\varsigma}} f=\partial_{z_1}^{\varsigma_1}\cdots\partial_{z_l}^{\varsigma_l} f=f_{\underbrace{\text{\tiny$z_1\cdots z_1$}}_{\text{$\varsigma_1$-times}}\cdots\underbrace{\text{\tiny$z_l\cdots z_l$}}_{\text{$\varsigma_l$-times}}}= \frac{\partial^{\varsigma_1+\cdots+ \varsigma_l}}{\partial z_1^{\varsigma_1}\cdots \partial z_l^{\varsigma_l}}f \qquad \text{for }{\boldsymbol{\varsigma}}=(\varsigma_1,\cdots\!,\varsigma_l)\in \mathbb{N}^l,\\
&\qquad \nabla_{\boldsymbol{z}} f=(\partial_{z_1} f,\cdots\!,\partial_{z_l} f)^\top,\qquad \Delta_{\boldsymbol{z}} f=\sum_{i=1}^l \partial_{z_i}^2 f,\\
&\qquad \nabla_{\boldsymbol{z}}^k f \text{ denotes one generic } \partial_{\boldsymbol{z}}^{\boldsymbol{\varsigma}} f \text{ with }|\boldsymbol{\varsigma}|=\sum_{i=1}^l \varsigma_i=k \text{ for integer }k\geq 2,\\
&\qquad |\nabla_{\boldsymbol{z}}^k f|=\Big(\sum_{|\boldsymbol{\varsigma}|=k}|\partial_{z_1}^{\varsigma_1} \cdots\partial^{\varsigma_l}_{z_l}f|^2\Big)^\frac{1}{2} \qquad\text{ for } k\in \mathbb{N}^*.
\end{aligned}
\end{equation*}
In particular, for the derivatives with respect to the
variable $\boldsymbol{x}=(x_1,\cdots\!,x_n)^\top\in \mathbb{R}^n$, 
we use the notation: 
$(\partial_i^{\varsigma_i},\partial^{\boldsymbol{\varsigma}},\nabla,\Delta,\nabla^k)=(\partial_{x_i}^{\varsigma_i}, \partial_{\boldsymbol{x}}^{\boldsymbol{\varsigma}},\nabla_{\boldsymbol{x}},\Delta_{\boldsymbol{x}},\nabla_{\boldsymbol{x}}^k)$.

\smallskip
\item If $\boldsymbol{f}: E\subset \mathbb{R}^l \to \mathbb{R}^q$ ($l,q\geq 2$, $E$ is a measurable set) is a vector function with the independent variables $\boldsymbol{z}=(z_1,\cdots\!,z_l)^\top$ and $X \in \{\partial_{z_i}^{\varsigma_i},\partial_{\boldsymbol{z}}^{\boldsymbol{\varsigma}},\Delta_{\boldsymbol{z}},\nabla_{\boldsymbol{z}}^k\}$, then 
\begin{equation*}
\begin{aligned}
&X\boldsymbol{f}=\big(Xf_1,\cdots\!,Xf_q\big)^\top, \ \quad \nabla_{\boldsymbol{z}}\boldsymbol{f}=\begin{pmatrix} 
\partial_{z_1} f_1 & \partial_{z_2} f_1 & \cdots & \partial_{z_l} f_1\\[2mm]
\partial_{z_1} f_2 & \partial_{z_2} f_2 & \cdots & \partial_{z_l} f_2 \\[2mm]
\vdots & \vdots & \ddots & \vdots \\[2mm]
\partial_{z_1} f_q & \partial_{z_2} f_q & \cdots & \partial_{z_l} f_q
\end{pmatrix}_{q\times l},\\
&|\nabla_{\boldsymbol{z}}^k \boldsymbol{f}|=\Big(\sum_{i=1}^q\sum_{|\boldsymbol{\varsigma}|=k}\big|\partial_{z_1}^{\varsigma_1} \cdots\partial^{\varsigma_l}_{z_l}f_{i}\big|^2\Big)^\frac{1}{2} \qquad \text{ for } k\in \mathbb{N}^*.
\end{aligned}
\end{equation*}
Moreover, if $l=j+k$ with $j\geq 0$ and the independent variables $\boldsymbol{z}$ take the form $\boldsymbol{z}=(\boldsymbol{s},\boldsymbol{\hat z})^\top$ with $\boldsymbol{s}=(s_1,\cdots\!,s_j)^\top$ and $\boldsymbol{\hat z}=(\hat z_1,\cdots\!,\hat z_k)^\top$, then
\begin{equation*}
\mathrm{div}_{\boldsymbol{\hat z}} \boldsymbol{f}=\sum_{i=1}^k \partial_{\hat z_i}f_i.  
\end{equation*}
In particular, if $j=0$ or $1$, $\,k=n$, and $\boldsymbol{\hat z}=\boldsymbol{x}=(x_1,\cdots\!,x_n)^\top\in \mathbb{R}^n$, then $\mathrm{div}\,=\mathrm{div}_{\boldsymbol{x}}$.
\smallskip
\item For any function $f=f(r)$ defined on $I$,
\begin{equation*}
D_\eta f=\frac{f_r}{\eta_r},\qquad\,\, D_\eta^k f=D_\eta(D_\eta^{k-1} f) \quad\text{for $k\in \NN^*$ and $k\geq 2$.}
\end{equation*}
\end{itemize}

\subsubsection{Notations on function spaces}
\begin{itemize}
\item For any function space $X(I)$ appearing  in this paper, unless otherwise specified, the following conventions are used:
\begin{equation*}
\begin{aligned}
&X=X(I),\qquad X^*\text{ --- the dual space of $X$},\\[4pt]
&X^*([0,T];Y^*)\text{ --- the dual space of $X([0,T];Y)$},\\[4pt]
&|f|_p=\|f\|_{L^p},\quad \|f\|_{k,p}=\|f\|_{W^{k,p}},\quad \|f\|_k=\|f\|_{H^k},\\[4pt]
&L^p_{\mathrm{loc}}:=\big\{f:\, f\in L^p(K) \,\, \text{for any open interval $K$ such that $\bar K\subset I\backslash \{0\}$}\big\},\\[4pt]
&H^k_{\mathrm{loc}}:=\big\{f: \, \partial_r^jf\in L^2_{\mathrm{loc}} \,\, \text{for any $0\leq j\leq k$}\big\},\qquad \|f\|_{X_t(Y)}=\|f\|_{X([0,T];Y(I))}.
\end{aligned}
\end{equation*}
\item Unless otherwise specified, the following definitions of weighted function spaces are used: let $J\subset I$ and let $0\leq \mathrm{w}=\mathrm{w}(r)$ be some function on $J$,
\begin{equation*}
\begin{aligned}
&H^{k}_\mathrm{w}(J):=\big\{f:\, \sqrt{\mathrm{w}}\partial_r^j f\in L^2(J) \,\,\text{for $0\leq j\leq k$}\big\}, \qquad H^{-k}_{\mathrm{w}}(J):=(H^{k}_{\mathrm{w}}(J))^*,\\[-2pt]
&L^2_\mathrm{w}(J)=H^{0}_\mathrm{w}(J),\quad \|f\|_{L^2_\mathrm{w}(J)}=\|\sqrt{\mathrm{w}} f\|_{L^2(J)},\quad \|f\|_{H^k_\mathrm{w}(J)}=\sum_{j=0}^k \|\partial_r^j f\|_{L^2_\mathrm{w}(J)}.
\end{aligned}
\end{equation*}
In particular, if $J=I$, then 
\begin{equation*}
\begin{aligned}
&H^{k}_\mathrm{w}=H^{k}_\mathrm{w}(I), \qquad H^{-k}_\mathrm{w}=H^{-k}_\mathrm{w}(I), \qquad L^2_\mathrm{w}=L^2_\mathrm{w}(I),\\[4pt]
&|f|_{2,\mathrm{w}}=\|f\|_{L^2_\mathrm{w}},\qquad \|f\|_{k,\mathrm{w}}=\|f\|_{H^k_\mathrm{w}}.
\end{aligned}
\end{equation*}
\smallskip
\item For any open set $\mathrm{Q}\subset \mathbb{R}^q$ ($q\in \NN^*$), $C^\ell(\overline{\mathrm{Q}})$ $(C(\overline{\mathrm{Q}})=C^0(\overline{\mathrm{Q}}))$ denotes the space of all functions $f(\boldsymbol{z})\in C^\ell(\mathrm{Q})$ such that $\nabla_{\boldsymbol{z}}^j f$ $(0\leq j\leq \ell)$ admits a unique continuous extension to $\overline{\mathrm{Q}}$, which is equipped with the norm:
\begin{equation*}
\|f\|_{C^\ell(\overline{\mathrm{Q}})}:= \max_{0\leq j\leq \ell}\|\nabla_{\boldsymbol{z}}^j f\|_{L^\infty(\mathrm{Q})}.
\end{equation*} 
Denote $C^\infty(\overline{\mathrm{Q}}):=\bigcap_{\ell\geq 0}C^\ell(\overline{\mathrm{Q}})$. 
In particular, if $\mathrm{Q}$ is an open interval $(a,b)$, then we simply write $C^\ell[a,b]=C^\ell ([a,b])$.
\smallskip
\item $X_\mathrm{c}(\mathrm{Q})=\big\{f \in X(\mathrm{Q}): \text{$f$ has compact support in $\mathrm{Q}$}\big\}$, for any set $\mathrm{Q}\subset \mathbb{R}^q$ ($q\in \NN^*$).
\smallskip
\item For any function space $X$ and functions $(\varphi,g_1,\cdots\!,g_k)$,
\begin{equation*}
\|\varphi(g_1,\cdots\!,g_k)\|_{X}:=\sum_{i=1}^k\|\varphi g_i\|_X,\qquad|\varphi(g_1,\cdots\!,g_k)|:=\sum_{i=1}^k|\varphi g_i|.
\end{equation*}
\item Denote by $\langle\cdot,\cdot\rangle_{X^*\times X}$ the  pairing between the space  $X$ and its dual space $X^*$, and $\langle\cdot,\cdot\rangle$ the inner product of $L^2$, \textit{i.e.},
\begin{equation*}
\qquad \left<F,f\right>_{X^*\times X}:=F(f) \,\,\, \text{for } F\in X^*,f\in X,
\qquad\,\,    \langle f,g\rangle:=\int_0^1 fg\,\mathrm{d}r \quad \text{for }f,g\in L^2.    
\end{equation*}
$\left< F,f \right>_{X_t^*(Y^*)\times X_t(Y)}$ denotes the pairing between $X([0,T];Y)$ and $X^*([0,T];Y^*)$. 
\end{itemize}

\subsubsection{Other notations}\label{othernotation}
\begin{itemize}
\item $\delta_{ij}$ denotes the Kronecker symbol with indices $(i,j)$: $\delta_{ij}= 1$ if $i=j$, $\delta_{ij}=0$ if $i\neq j$.
\smallskip
\item For any $n\times n$ real matrix $\mathcal{M}$, $\mathcal{M}_{ij}$ denotes its $(i,j)$-th entry. Moreover, $\mathrm{O}(n)$  denotes the set of all $n\times n$ real orthogonal matrices $\mathcal{O}$.
\smallskip
\item $E\sim F$ denotes $C^{-1}_*E\leq F\leq C_*E$ for some constant $C_*\geq 1$, where the form of $C_*$ may be different at each occurrence.

\smallskip
\item $\zeta_{a}=\zeta_{a}(r)\in C^\infty[0,1]$ ($a\in (0,1)$) denotes a cut-off function satisfying
\begin{equation*}
\zeta_{a}\in [0,1],\qquad (\zeta_{a})_r\leq 0, \qquad  \zeta_{a}=1 \ \ \text{on $[0,a]$},\qquad \zeta_{a}=0 \ \ \text{on $\big[\frac{1+3a}{4},1\big]$},
\end{equation*}
and $\zeta_{a}^\sharp=\zeta_{a}^\sharp(r):=1-\zeta_{a}(r)$. Certainly, it follows that 
\begin{equation*}
\begin{aligned}
&\,\zeta_{a}\leq \zeta_{\tilde a}, \quad  \zeta_{a}^\sharp\geq \zeta^\sharp_{\tilde a} \qquad\,\, \text{for $0<a<\tilde a<1$},\\
&\support\, (\zeta_{a})_r\cup \support\, (\zeta_{a}^\sharp)_r\subset \big[a,\frac{1+3a}{4}\big],\qquad |(\zeta_{a})_r|+|(\zeta_{a}^\sharp)_r|\leq C(a),
\end{aligned}
\end{equation*}
where $C(a)>0$ is a constant depending only on $a$. In particular, if $a=\frac{1}{2}$, define \begin{equation*}
\zeta=\zeta(r):=\zeta_{\frac{1}{2}}(r),\qquad \zeta^\sharp=\zeta^\sharp(r):=1-\zeta(r).
\end{equation*}
\item $\chi_{a}=\chi_{a}(r)$ denotes the characteristic function on $[0,a]$ $(a\in (0,1))$, {\it i.e.}, $\chi_{a}=1$ on $[0,a]$ and $\chi_{a}=0$ on $(a,1]$, and $\chi_{a}^\sharp=1-\chi_{a}$. 
Then
\begin{equation*}
\begin{aligned}
&\chi_{a}\leq \zeta_{a} \quad \text{for $a\in (0,1)$},\qquad \chi_{a} \geq  \zeta_{\frac{4a-1}{3}} \quad \text{for $a\in \big(\frac{1}{4},1\big)$};\\
&\chi_{a}^\sharp\geq  \zeta_{a}^\sharp \quad \text{for $a\in (0,1)$},\qquad \chi_{a}^\sharp\leq \zeta_{\frac{4a-1}{3}}^\sharp \quad \text{for $a\in \big(\frac{1}{4},1\big)$}.
\end{aligned}
\end{equation*}
In particular, if $a=\frac{1}{2}$, define 
\begin{equation*}
\chi=\chi(r):=\chi_{\frac{1}{2}}(r),\qquad \chi^\sharp=\chi^\sharp(r):=1-\chi(r).
\end{equation*}
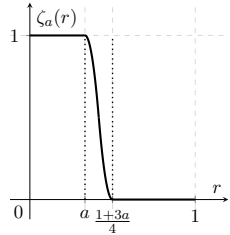
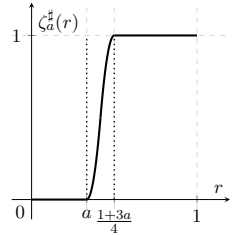
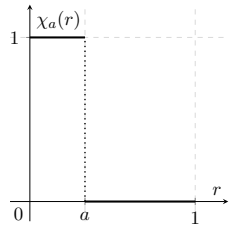
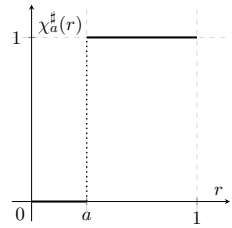
\begin{figure}[ht]
\centering
\begin{subfigure}[t]{0.4\textwidth}
\centering
\begin{tikzpicture}[scale=0.65]
\begin{axis}[
    width=6cm,
    height=6cm,
    xmin=-0.12, xmax=1.2,
    ymin=-0.12, ymax=1.2,
    axis lines = middle,
    xlabel = $r$,
    ylabel = $\zeta_{a}(r)$,
    xtick = {0, 1/3,1/2, 1},
    xticklabels = {$0$, $a$, $\frac{1+3a}{4}$, 1},
    ytick = {0, 1},
    grid = major,
    grid style = {dashed, gray!30},
    samples = 200,
    domain = 0:1]
\addplot[black, very thick, domain=0:1/3]{1};
\addplot[black, very thick, domain=1/2:1]{0};
\addplot[black, very thick, domain=1/3:5/12, smooth]{1-72*(x-1/3)^2};
\addplot[black, very thick, domain=5/12:1/2, smooth]{72*(x-1/2)^2};
\addplot[black, dotted, thick] coordinates {(1/3,0) (1/3,1)};
\addplot[black, dotted, thick] coordinates {(1/2,0) (1/2,1)};
\node[anchor=north east] at (axis cs:0,0) {$0$};
\end{axis}
\end{tikzpicture}
\caption{Function $\zeta_{a}$}
\end{subfigure}
\hspace{0cm}
\begin{subfigure}[t]{0.4\textwidth}
\centering
\begin{tikzpicture}[scale=0.65]
\begin{axis}[
    width=6cm,
    height=6cm,
    xmin=-0.12, xmax=1.2,
    ymin=-0.12, ymax=1.2,
    axis lines = middle,
    xlabel = $r$,
    ylabel = $\zeta_{a}^\sharp(r)$,
    xtick = {0, 1/3,1/2, 1},
    xticklabels = {0,$a$, $\frac{1+3a}{4}$, 1},
    ytick = {0, 1},
    grid = major,
    grid style = {dashed, gray!30},
    samples = 200,
    domain = 0:1]
\addplot[black, very thick, domain=1/2:1]{1};
\addplot[black, very thick, domain=0:1/3]{0};
\addplot[black, very thick, domain=1/3:5/12, smooth]{72*(x-1/3)^2};
\addplot[black, very thick, domain=5/12:1/2, smooth]{1-72*(1-x-1/2)^2};
\addplot[black, dotted, thick] coordinates {(1/3,0) (1/3,1)};
\addplot[black, dotted, thick] coordinates {(1/2,0) (1/2,1)};
\node[anchor=north east] at (axis cs:0,0) {$0$};
\end{axis}
\end{tikzpicture}
\caption{Function $\zeta^\sharp_{a}$}
\end{subfigure}

\vspace{0.5cm}

\begin{subfigure}[t]{0.4\textwidth}
\centering
\begin{tikzpicture}[scale=0.65]
\begin{axis}[
    width=6cm,
    height=6cm,
    xmin=-0.12, xmax=1.2,
    ymin=-0.12, ymax=1.2,
    axis lines = middle,
    xlabel = $r$,
    ylabel = $\chi_{a}(r)$,
    xtick = {0,1/3, 1},
    xticklabels = {0, $a$, 1},
    ytick = {0, 1},
    grid = major,
    grid style = {dashed, gray!30},
    samples = 200,
    domain = 0:1]
\addplot[black, very thick, domain=0:1/3] {1};
\addplot[black, very thick, domain=1/3:1] {0};
\addplot[black, dotted, thick] coordinates {(1/3,0) (1/3,1)};
\node[anchor=north east] at (axis cs:0,0) {$0$};
\end{axis}
\end{tikzpicture}
\caption{Function $\chi_{a}$}
\end{subfigure}
\hspace{0cm}
\begin{subfigure}[t]{0.4\textwidth}
\centering
\begin{tikzpicture}[scale=0.65]
\begin{axis}[
    width=6cm,
    height=6cm,
    xmin=-0.12, xmax=1.2,
    ymin=-0.12, ymax=1.2,
    axis lines = middle,
    xlabel = $r$,
    ylabel = $\chi^\sharp_{a}(r)$,
    xtick = {0, 1/3, 1},
    xticklabels = {0, $a$, 1},
    ytick = {0, 1},
    grid = major,
    grid style = {dashed, gray!30},
    samples = 200,
    domain = 0:1]
\addplot[black, very thick, domain=0:1/3] {0};
\addplot[black, very thick, domain=1/3:1] {1};
\addplot[black, dotted, thick] coordinates {(1/3,0) (1/3,1)};
\node[anchor=north east] at (axis cs:0,0) {$0$};
\end{axis}
\end{tikzpicture}
\caption{Function $\chi_{a}^\sharp$}
\end{subfigure}
\caption{Four types of the cut-off functions $(\zeta_{a},\zeta_{a}^\sharp,\chi_{a},\chi_{a}^\sharp)$.}
\end{figure}
\end{itemize}

\subsection{Main strategies}\label{subsection-strategy}
In this subsection, we present our main strategies and new ideas
to establish the main theorems. 
To overcome the difficulties arising from the coordinate singularity at the origin and the strong degeneracy on the moving vacuum boundary for large-data problems, our analysis relies on the following key ingredients:
\begin{itemize}
\item[\S \ref{subsec-2.0}] introduction of new weighted energy functionals (see \S \ref{Section-maintheorem});
\smallskip
\item[\S \ref{subsub322}] establishment of the interior BD entropy estimates and new $(\eta,\eta_r)$-weighted estimates for density $\varrho$ near the origin, leading to the global lower bounds for $(\eta_r,\frac{\eta}{r})$ in $[0,T]\times \bar I$ (see \S\ref{Section-densityupper}); 
\smallskip
\item[\S \ref{subsub323}] derivation of novel global estimates for the effective velocity, including its interior boundedness and $\rho_0$-weighted $L^p$-estimates for $p\in [2,\infty]$ near the boundary (see \S\ref{Section-effectivevelocity});
\smallskip
\item[\S \ref{subsub324}] establishment of the global uniform upper bounds for $(\eta_r,\frac{\eta}{r})$ by combining the quantitative estimate near the vacuum boundary with new $\eta_r$-weighted estimates for $(U,V)$ and $\log\varrho $ near the symmetric center; the resulting uniform upper bound of $(\eta_r,\frac{\eta}{r})$ yields the lower bound for $\varrho$ inside the fluids (see \S\ref{Section-etarupper}--\S\ref{Section-densitylower});
\smallskip
\item[\S \ref{subsub325}] establishment of the global uniform estimates for $U$ in $[0,T]\times \bar I$ under the well-designed energy functionals; together with the continuation argument, this extends the local classical solution to the global one (see \S\ref{Section-globalestimates}--\S\ref{Section-global}).
\end{itemize}

Throughout the rest of \S \ref{subsection-strategy}, $ C_0\in (1,\infty)$ denotes a generic constant depending only on $(n,\mu,A,\gamma,\beta,\varepsilon_0,\rho_0,u_0,\cK_1,\cK_2)$,
and $C(l_1,\cdots\!,l_k)\in (1,\infty)$ a generic  constant depending on $C_0$
and parameters $(l_1,\cdots\!,l_k)$, which may be different at each occurrence.

\subsubsection{Some new weighted energy functionals} \label{subsec-2.0}

The strong degeneracy of the momentum equation $\eqref{eq:VFBP-La-eta}_1$ makes it intricate to provide an effective propagation mechanism for the regularity of $U$ near the vacuum in general Sobolev spaces. Then the first key point on the existence  is to introduce some proper weighted energy functionals. By considering the balance of the pressure and the spatial dissipation near the vacuum, we introduce $(\cE,\cD)(t,U)$ for $\beta\in (\frac{1}{3},\gamma-1]$, which turns out to be reasonable later. The details on how to construct these energy functionals can be found in Remark \ref{remark-energy function}.

Based on the choice of $(\cE,\cD)(t,U)$, the desired local-in-time existence and uniqueness  of classical solutions of {\rm\textbf{IBVP}} \eqref{eq:VFBP-La-eta}
can be stated as follows:
\begin{Theorem}\label{local-Theorem1.1} 
Let $n=2$ or $3$ and $\gamma\in (\frac{4}{3},\infty)$. Assume that $\rho_0(r)$ satisfies \eqref{distance-la} for some $\beta\in (\frac{1}{3},\gamma-1]$, and $u_0(r)$ satisfies
\begin{equation}\label{a2-lo}
\cE(0,U)<\infty.
\end{equation}
Then there exists $T_*>0$, which depends only on $(n,\mu,A,\gamma,\beta,\varepsilon_0,\rho_0,u_0,\cK_1,\cK_2)$, such that {\rm\textbf{IBVP}} \eqref{eq:VFBP-La-eta} admits a unique classical solution $(U,\eta)(t,r)$ in $[0,T_*]\times \bar I $ satisfying  
\begin{equation}\label{b1-lo}
\begin{aligned}
&\cE(t,U)+t \cD(t,U)\in L^\infty(0,T_*),\qquad  \cD (t,U)\in L^1(0,T_*),\\
& (\eta_r,\frac{\eta}{r})(t,r)\in \big[\frac{1}{2},\frac{3}{2}\big] \qquad \quad \ \ \,\text{for $(t,r)\in [0,T_*]\times \bar I$},\\[6pt]
& U|_{r=0}=U_r|_{r=1}=0\qquad\qquad \,\text{on $(0,T_*]$}.
\end{aligned}
\end{equation}
\end{Theorem}

The proof for Theorem \ref{local-Theorem1.1}  will be given in \S \ref{Section-local}.

\begin{Remark}\label{remk31}
In fact, to carry out the continuation argument in {\rm\S\ref{Section-global}}, {\rm Theorem \ref{local-Theorem1.1}} can be extended to a more general case. More precisely, let $\tau> 0$ be fixed, and let $(U,\eta)|_{t=\tau}=(u_\star,\eta_\star)(r)$ be given radial functions, satisfying the following regularity properties: 
\begin{equation}\label{general-initial-assumptions}
\delta_*\leq\big((\eta_\star)_r,\frac{\eta_\star}{r}\big)(r)\leq\delta^*\quad\text{on }\bar I,\qquad \cE(\tau,U) +\mathring\cE(\eta_\star)\leq \mathfrak{M}_0,
\end{equation}
for some $0<\delta_*<\delta^*$ and $\mathfrak{M}_0>1$, where $\mathring\cE(f)$ is defined by
\begin{equation}\label{chongqi-def}
\begin{aligned}
\mathring\cE(f)&:=\Big\|\zeta r^\frac{m}{2}\Big(f,f_r,\frac{f}{r},f_{rr}, \big(\frac{f}{r}\big)_r,f_{rrr}, \big(\frac{f}{r}\big)_{rr},\frac{1}{r}\big(\frac{f}{r}\big)_r,f_{rrrr}, \big(\frac{f}{r}\big)_{rrr},\big(\frac{1}{r}(\frac{f}{r})_r\big)_r \Big)\Big\|_{L^2(I)}^2\\
&\quad +\big\|\rho_0^\frac{1}{2}(f,f_r)\big\|_{L^2(\frac{1}{2},1)}^2+\big\|\rho_0^{(\frac{3}{2}-\varepsilon_0)\beta}(f_{rr},f_{rrr},f_{rrrr})\big\|_{L^2(\frac{1}{2},1)}^2.
\end{aligned}
\end{equation}
 
With the original Lagrangian labels retained, the mass identity \eqref{eq:eta} remains valid, \textit{i.e.}, $\eta^m\eta_r\varrho=r^m\rho_0$. For $T>0$, we thus consider the following version of \eqref{eq:VFBP-La-eta} on $(\tau,\tau+T]\times I$:
\begin{equation}\label{general-initial-problem}
\begin{cases}
\displaystyle
\varrho U_t+A D_\eta(\varrho^\gamma) =2\mu D_\eta\Big(\varrho\big(D_\eta U+\frac{mU}{\eta}\big)\Big)
-2\mu m \frac{UD_{\eta}\varrho}{\eta}&\text{in $(\tau,\tau+T]\times I$},\\
\eta_t=U&\text{in $(\tau,\tau+T]\times I$},\\[4pt]
(U,\eta)(\tau,r)=(u_\star(r),\eta_\star(r))&\text{for $r\in I$}.
\end{cases}
\end{equation}
Then we can show that {\rm Theorem \ref{local-Theorem1.1}} still holds, that is, problem \eqref{general-initial-problem} with $T=T_*$ admits a unique classical solution $(U,\eta)(t, r)$ in $[\tau, \tau+T_*]\times \bar I$, satisfying
\begin{equation}\label{b1-lo3}
\begin{aligned}
& (\eta_r,\frac{\eta}{r})(t,r)\in \big[\frac{\delta_*}{2},\frac{3\delta^*}{2}\big] \qquad   \ \,\text{for $(t,r)\in [\tau,\tau+T_*]\times \bar I$},\\
&\cE(t,U)+ (t-\tau)\cD(t,U)\in L^\infty(\tau,\tau+T_*),\quad \cD\in L^1(\tau,\tau+T_*),  \\
&U|_{r=0}=U_r|_{r=1}=0 \qquad \text{on $[\tau,\tau+T_*]$}.
\end{aligned}
\end{equation}
In this case, $T_*>0$ depends only on $(\delta_*,\delta^*,\mathfrak{M}_0,n,\mu,A,\gamma,\beta,\varepsilon_0,\rho_0,\cK_1,\cK_2)$ and is
independent of $\tau$. We sketch the proof of this extension in {\rm\S\ref{subsection9.4}}. 
\end{Remark}

\subsubsection{Global uniform upper bound of the density and lower bounds of $(\eta_r,\frac{\eta}{r})$}\label{subsub322}
The derivation begins with estimates that retain the degeneracy of the initial density, rather than with a compactness argument at the vacuum boundary. For simplicity, we focus on the 3-D case below, since the 2-D case is analogous.

The fundamental energy estimates first give an $L^\infty$ bound for $\eta^m\varrho$ (see Lemmas \ref{lemma-basic energy}--\ref{lemma-far depth}). By the mass identity, this yields the weak geometric bounds
\begin{equation}\label{33}
\eta_r(t,r)\geq C_0^{-1}r^m\rho_0(r),\qquad
\eta(t,r)\geq C_0^{-1}\int_0^r \hat r^m\rho_0(\hat r)\,\mathrm{d}\hat r.
\end{equation}
They give uniform lower bounds for $\eta$ in the exterior domain and for $\eta_r$ on any compact subintervals of $(0,1)$. These weak bounds suffice for the interior BD entropy and density estimates and, after retaining the mass weight in all exterior pressure terms, for the weighted $L^p$-estimates of $(U,V)$. The resulting global density upper bound is
\begin{equation}\label{near-d}
|\varrho(t)|_\infty\leq C(T) \qquad\text{for all $t\in [0,T]$}.
\end{equation}
Then, by \eqref{eq:eta},
\begin{equation*}
\eta^{m+1}(t,r)
=(m+1)\int_0^r\frac{\hat r^m\rho_0}{\varrho(t,\hat r)}\,\mathrm{d}\hat r
\geq C(T)^{-1}r^{m+1},
\end{equation*}
and hence $\frac{\eta(t,r)}{r}\geq C(T)^{-1}$. To derive the lower bounds of $\eta_r$ near the origin and the boundary, we integrate the momentum equation over $[0,r]$ and $[r,1]$, respectively. Then,  the $L^p$-estimate of $U$ leads us to the following two types of estimates of $\eta^m\varrho$:
\begin{equation*}
(\eta^m\varrho)(t,r)\leq C(T)r^m,\qquad (\eta^m\varrho)(t,r)\leq C(T)\rho_0,
\end{equation*}
so that $\eta_r(t,r)\geq C(T)^{-1}\max\{r^m,\rho_0\}$. Hence, these two estimates, combined with \eqref{33}, give the desired lower bounds of $(\eta_r,\frac{\eta}{r})$: 
\begin{equation}\label{2,,5}
\frac{\eta(t,r)}{r}\geq C(T)^{-1},\qquad \eta_r(t,r)\geq C(T)^{-1}
\qquad\text{in $[0,T]\times\bar I$}.
\end{equation}

Now, we explain how to prove the interior part of \eqref{near-d}. Our approach is based on the effective velocity $\boldsymbol{v}=\boldsymbol{u}+2\mu \nabla\log\rho$ and the classical Sobolev embedding $W^{1,p}(\Omega(t))\into L^\infty (\Omega(t))$ ($p>3$) for each $t>0$: 
\begin{equation*}
\|\varphi\rho^{\frac{1}{p}}\|_{L^{\infty}(\Omega(t))} \leq C(p)\|\varphi\rho^{\frac{1}{p}}\|_{W^{1,p}(\Omega(t))} \leq C(p,\varphi)\big(\|\rho\|_{L^1(\Omega(t))}^{\frac{1}{p}}+\|\varphi\rho^{\frac{1}{p}}(\boldsymbol{u}, \boldsymbol{v})\|_{L^p(\Omega(t))}\big),
\end{equation*}
where $\varphi\in C_{\rm c}^\infty(\Omega(t))$ is a smooth cut-off function such that $\varphi\in[0,1]$. Returning to the Lagrangian coordinates $(t,r)$, 
we see that it suffices to establish 
\begin{equation}\label{3.6}
E_{\star}(t):=\big|\big((r^2\rho_0)^\frac{1}{p}U,\,(\zeta r^2\rho_0)^\frac{1}{p} V\big)(t)\big|_p^p\leq C(p,T) \qquad\text{for all $t\in [0,T]$ and some $p>3$}.
\end{equation}
Here, $V:=U+2\mu D_\eta \log\varrho$ is the Lagrangian radial projection of $\boldsymbol{v}$. Due to the behavior of $D_\eta\log\varrho$ near the vacuum boundary, the cut-off function $\zeta$ in \eqref{3.6} is indispensable.

The proof of \eqref{3.6} relies on two key observations. First, to circumvent the physical vacuum singularity, we develop an ``interior BD entropy estimate'' near the origin. Recall that the initial condition of the BD entropy estimate \eqref{BDconditionr} fails to hold when $\rho_0$ satisfies \eqref{distance-la} with $\beta>1$. Hence, while deriving the BD entropy estimate, we introduce a smooth cut-off function $\zeta_a$, thereby obtaining the following estimates (see Lemma \ref{near-BD}):
\begin{equation}\label{inter-bd}
\big|(\zeta_a \eta^2\eta_r)^\frac{1}{2}D_\eta\sqrt{\varrho}(t)\big|_2\leq C(a,T)\qquad\text{for any $a\in (0,1)$ and $t\in [0,T]$}.
\end{equation}

Building on this, we are led to the second key ingredient. Heuristically, the flow map $\eta$ acts as a radial weighting near the origin, so that one can expect to obtain some favorable flow-map-weighted estimates of $\varrho$ from \eqref{inter-bd} and the Hardy inequality. However, the Hardy inequality is not directly applicable here, due to the lack of \textit{a priori} upper bounds for $(\eta,\eta_r)$ near the origin. Fortunately, by carefully utilizing \eqref{inter-bd} alongside the fundamental theorem of calculus, we establish the following $L^1$- and $L^\infty$-estimates for $\varrho$ weighted by $(\eta,\eta_r)$ in Lemma \ref{lemma-near depth}: For $q_1\in [1,3]$ and $q_2\in [1,2]$,
\begin{equation}\label{3.8}
|\zeta_a\eta^{q_1-1}\eta_r\varrho(t)|_1\leq C(a,q_1,T), \quad |\zeta_a \eta^{q_2}\varrho(t)|_\infty\leq C(a,q_2,T) \qquad\,\,\text{for any $t\in[0,T]$}.
\end{equation}

Hence, multiplying $\eqref{eq:VFBP-La-eta}_1$ and \eqref{eq:v} by $\eta^2\eta_r|U|^{p-2}U$ and $\zeta r^2\rho_0|V|^{p-2}V$ with $p>3$, respectively, and integrating over $I$, we use the compact-subinterval consequences of the weak geometry bound, exact cancellation of the exterior mass weights, and \eqref{inter-bd}--\eqref{3.8} to arrive at the following inequalities:
\begin{align}
&\frac{\mathrm{d}}{\mathrm{d}t}\big|(r^2\rho_0)^{\frac{1}{p}}U\big|_{p}^{p} + \cD_U(t) \leq C(p) \Big(1+E_\star(t)+ \underline{\int_0^\frac{1}{2} \eta^{2+\vartheta_1(p-2)}\eta_r\varrho^{p\gamma-p+1}\,\mathrm{d}r}_{:=\mathrm{A}_{(\vartheta_1)}} \Big),\label{04}\\
&\frac{\mathrm{d}}{\mathrm{d}t}\big|(\zeta r^2\rho_0)^\frac{1}{p}V\big|_p^p+ \cD_V(t) \leq C(p)\Big(E_\star(t)+ \cD_U(t)\underline{\Big|\zeta_{\frac{5}{8}}\eta^\frac{2}{(\gamma-1)(\vartheta_2 p+1-\vartheta_2)}\varrho\Big|_\infty}^{\!\!\!\!\!\!\!(\gamma-1)(\vartheta_2 p+1-\vartheta_2)}_{:=\mathrm{B}_{(\vartheta_2)}}\Big),\label{04*}
\end{align}
where $\vartheta_1,\vartheta_2\in [0,1]$ are any fixed constants and $(\cD_U,\cD_V)(t)$ are the dissipation terms:
\begin{equation*}
\cD_U(t):=\Big|\big(\frac{r^2\rho_0}{\eta^2}\big)^\frac{1}{p}U\Big|_p^p,\qquad \cD_V(t):=\big|(\zeta \eta^2\eta_r\varrho^{\gamma})^\frac{1}{p}  V\big|_p^p.
\end{equation*}

For $(\mathrm{A}_{(\vartheta_1)},\mathrm{B}_{(\vartheta_2)})$  when $\gamma\in (1,2]$, we can choose suitable $(\vartheta_1,\vartheta_2)\in [0,1]$ such that \eqref{3.8} can be applied to $(\mathrm{A}_{(\vartheta_1)},\mathrm{B}_{(\vartheta_2)})$ directly, leading to $\mathrm{A}_{(\vartheta_1)}+\mathrm{B}_{(\vartheta_2)}\leq C(p,T)$. 

To handle $(\mathrm{A}_{(\vartheta_1)},\mathrm{B}_{(\vartheta_2)})$ when $\gamma  \in (2,3)$, we can first choose $\vartheta_2=0$ and derive from \eqref{3.8} that  $\mathrm{B}_{(0)}\leq C(p,T)$. For $\mathrm{A}_{(\vartheta_1)}$, since the power of $\eta$ is too low for \eqref{3.8} to apply, we set $\vartheta_1=1$ and develop an iterative scheme to increase the power of $\eta$ in $\mathrm{A}_{(1)}$. Employing integration by parts multiple times and $2\mu\varrho_r=\eta_r\varrho(V-U)$, we see that, for $\varepsilon\in (0,1)$ and $j\in \mathbb{N}^*$,
\begin{equation}\label{055}
\begin{aligned}
\mathrm{A}_{(1)} &\leq  C(\varepsilon,p,j,T)\Big(1+\underline{\int_0^1 \zeta \eta^{a_j} \eta_r \varrho^{b_j}\,\mathrm{d}r}_{:=\mathrm{I}_{(a_j,b_j)}}\Big)  +\varepsilon\big(\mathrm{B}_{(0)}^{\gamma-1} \cD_U(t)+ \cD_V(t)\big),
\end{aligned}   
\end{equation}
where $(a_j,b_j)$ are two strictly increasing sequences, and we find that \eqref{3.8} is applicable to $\mathrm{I}_{(a_{j_0},b_{j_0})}$ for a finite $j=j_0$. A wider cut-off, equal to one on the support of $\zeta$, is used in the terminal factor and in the $U$ forcing term. Distributing separate internal absorption parameters along the finite iteration and then combining the $U$ and $V$ inequalities gives a Gr\"onwall-type inequality for $E_\star(t)$, hence \eqref{3.6}. For the overall details on establishing \eqref{3.6} in the 2-D and 3-D cases,
see Lemmas \ref{lp-uv}--\ref{lemma-v-lp}.

\subsubsection{Global estimates for the effective velocity}\label{subsub323}
To establish the global upper bounds for $(\eta_r,\frac{\eta}{r})$, we develop some crucial estimates for $V$ both near and away from the origin. 
 
On one hand, we can solve for $V$ from its damped transport equation \eqref{eq:v}:
\begin{equation}\label{odesudu}
\begin{aligned}
V(t,r)&=v_0(r)\exp\Big(-\frac{A\gamma}{2\mu}\int_0^t \varrho^{\gamma-1}(s,r)\,\mathrm{d}s\Big)\\
&\quad+\frac{A\gamma}{2\mu}\int_0^t (\varrho^{\gamma-1} U)(\tau,r)\exp\Big(-\frac{A\gamma}{2\mu}\int_\tau^t \varrho^{\gamma-1}(s,r)\,\mathrm{d}s\Big) \,\mathrm{d}\tau,
\end{aligned}
\end{equation}
where $v_0=V|_{t=0}$. Due to $\eqref{distance-la}_2$ and $u_0\in L^\infty$, $\rho_0^Kv_0\in L^p$ for $p\in [2,\infty)$ whenever $K>\frac{p-1}{p}\beta$. Hence, \eqref{odesudu}, combined with the density upper bound \eqref{near-d}, \eqref{3.6}, $\varrho\geq 0$, and the Minkowski integral inequality, gives the following weighted $L^p$-estimates for $V$ away from the origin (see Lemma \ref{lemma-v Lp ex}):
\begin{equation}\label{38'}
|\chi^\sharp\rho_0^{\iota\beta}V(t)|_p \leq C(p,\iota,T) \qquad\text{for any $p\in[2,\infty)$, $\iota>\frac{p-1}{p}$, and $t\in[0,T]$}.
\end{equation}
Moreover, since $\rho_0^\beta v_0\in L^\infty$, by employing an approach based on the Sobolev embedding $W^{1,1}\into L^\infty$ and the formula $2\mu\varrho_r=\eta_r\varrho(V-U)$, we obtain the $L^1([0,T];L^\infty)$-estimate for $\chi^\sharp\rho_0^\beta\varrho^{\gamma-1} U$ from \eqref{3.6}, \eqref{38'}, the density upper bound, and the global lower bound of $\frac{\eta}{r}$. This, along with \eqref{odesudu}, leads to the weighted $L^\infty$-estimate for $V$ in Lemma \ref{lemma-v Linfty ex}:
\begin{equation}\label{VVVddd}
|\chi^\sharp\rho_0^{\beta}V(t)|_\infty \leq C(T) \qquad\text{for any $t\in[0,T]$}.
\end{equation}

Near the origin, we first establish some special weighted $L^p$-estimates for $U$. In contrast to the calculation in \eqref{04} in the proof of \eqref{3.6}, we multiply $\eqref{eq:VFBP-La-eta}_1$ by $\zeta^2\eta_r|U|^{p-2}U$ for $p\geq 2$. Then, utilizing $2\mu\varrho_r=\eta_r\varrho(V-U)$, the $L^\infty$ norm of $\zeta V$, the density upper bound, and only the compact-annulus geometry on the support of $\zeta_r$, we obtain in Lemma \ref{new-u-lp} that, for any $p\in [2,\infty)$ and $t\in [0,T]$,
\begin{equation*}
\big|(\zeta^2\eta_r\varrho)^\frac{1}{p}U(t)\big|_p^p+\int_0^t \Big|(\zeta^2\eta_r\varrho)^\frac{1}{2}|U|^\frac{p-2}{2}\big(D_\eta U,\frac{U}{\eta}\big)\Big|_2^2\,\mathrm{d}s \leq C(p,T)\big(\sup_{s\in[0,t]}|\zeta V|_\infty^2+1\big).
\end{equation*}
Subsequently, using a similar argument in deriving the $L^1([0,T];L^\infty)$-estimate for $\chi^\sharp\rho_0^\beta\varrho^{\gamma-1} U$, we can obtain from the above the following type of estimate: For any $\varepsilon\in (0,1)$,
\begin{equation*}
\int_0^t |\zeta \varrho^{\gamma-1}U|_\infty\mathrm{d}s\leq C(\varepsilon,T) +\varepsilon\!\sup_{s\in[0,t]}|\zeta V|_\infty,
\end{equation*}
which, along with \eqref{odesudu}, leads to the $L^\infty$-estimate for $\zeta V$ in Lemma \ref{lemma-v Linfty in}:
\begin{equation}\label{38'''}
|\zeta V(t)|_\infty\leq C(T) \qquad\text{for all $t\in [0,T]$}.
\end{equation}

\subsubsection{Global uniform upper bounds of $(\eta_r,\frac{\eta}{r})$ and lower bound of the density inside the fluids}\label{subsub324}
To obtain the global uniform upper bounds for $(\eta_r,\frac{\eta}{r})$, we treat the exterior domain and the region near the symmetric center by two different mechanisms.  

We first consider the exterior domain. In Lemma \ref{lemma-upper jacobi-ex}, we obtain an ODE inequality associated with $\eta^m\varrho$ from the upper bound of the density and momentum equation $\eqref{eq:VFBP-La-eta}_1$. Then, solving this resulting inequality gives the following lower-bound control of $\eta^m\varrho$:
\begin{equation}\label{39'}
(\eta^m\varrho)(t,r)\geq \rho_0(r)\Big(C(T)^{-1}-C(T)\big(\rho_0(r)^{\frac{\beta^2}{1+2\beta}}+\rho_0(r)^\beta\big)\Big).
\end{equation}
Since $\rho_0^\beta\sim 1-r$, estimate \eqref{39'} allows us to choose $a_T\in(0,1)$ sufficiently close to $1$ such that $\eta_r\leq C(T)$ on $[0,T]\times[a_T,1]$. The bound for $\frac{\eta}{r}$ in this region follows from $\eta_t=U$, the $\rho_0$-weighted estimates for $U$, and the Sobolev embedding $W^{1,1}\into L^\infty$.

Near the symmetric center, we first establish estimates that retain the factor $\eta_r$, without assuming its upper bound. More precisely, the equations for $U$ and $V$, together with the effective-velocity bounds established in \S \ref{Section-effectivevelocity}, imply, for any $a\in(0,1)$,
\begin{align}
&\begin{aligned}
|\zeta_a\sqrt{\eta_r}U(t)|_2^2
&+\int_0^t\Big(\Big|\zeta_a\sqrt{\eta_r}\big(D_\eta U,\frac{U}{\eta}\big)\Big|_2^2+|\zeta_aU|_\infty^2\Big)\,\mathrm{d}s\leq C(a,T),
\end{aligned}\label{edt}\\
&|\chi_a\sqrt{\eta_r}V(t)|_2\leq C(a,T).\label{edt-v}
\end{align}
The continuity equation for $\log\varrho$ and estimates \eqref{edt}--\eqref{edt-v} yield a Gr\"onwall inequality for the weighted logarithmic energy. Consequently,
\begin{equation*}
|\chi_a\sqrt{\eta_r}\log\varrho(t)|_2\leq C(a,T).
\end{equation*}
This bound, together with the identity $2\mu D_\eta\log\varrho=V-U$, the estimate for $\chi_a\sqrt{\eta_r}V$, and the lower bound of $\eta_r$, gives
\begin{equation}\label{Vrq}
|\chi_a\log\varrho(t)|_\infty\leq C(a,T) \qquad\text{for all $t\in[0,T]$}.
\end{equation}
It follows from \eqref{Vrq} that $\varrho\geq C(a,T)^{-1}$ on $[0,T]\times[0,a]$. Since $\rho_0$ is bounded above, the mass identity implies
\begin{equation*}
\eta^m(t,r)\eta_r(t,r)=\frac{r^m\rho_0(r)}{\varrho(t,r)}\leq C(a,T)r^m \qquad\text{for $(t,r)\in[0,T]\times[0,a]$}.
\end{equation*}
Combining this estimate with the lower bounds in \eqref{2,,5}, we obtain the upper bounds of $(\eta_r,\frac{\eta}{r})$ near the symmetric center. The interior and exterior estimates therefore establish the global two-sided geometry. Finally, \eqref{eq:eta} implies
\begin{equation}\label{322}
C(T)^{-1}\leq\frac{\varrho(t,r)}{\rho_0(r)}\leq C(T)
\qquad\text{for all $(t,r)\in[0,T]\times\bar I$}.
\end{equation}
In particular, \eqref{322} shows that no vacuum state will occur inside the fluid in any finite time.

\subsubsection{Global weighted energy estimates for the velocity}\label{subsub325}
For establishing the global weighted energy estimates, we first establish in \S \ref{subsub911}--\S\ref{subsub914} all tangential estimates: 
\begin{equation}\label{2211}
\sup_{t\in[0,T]}\Big|(r^m \rho_0)^\frac{1}{2}\big(\partial_t^kU,D_\eta \partial_t^kU, \frac{\partial_t^kU}{\eta}\big) \Big|_2+\int_0^T \big|(r^m\rho_0)^\frac{1}{2}U_{tt}\big|_2^2\mathrm{d}s\leq C(T) \qquad \text{for $k=0,1$},
\end{equation}
alongside the second- and third-order interior elliptic estimates and $W^{1,\infty}$-estimates for $U$:
\begin{equation}\label{2211*}
\sup_{t\in[0,T]}\Big(\Big|\zeta r^\frac{m}{2}\big(D_\eta^2 U,D_\eta(\frac{U}{\eta}),D_\eta^3 U,D_\eta^2(\frac{U}{\eta}),\frac{1}{\eta}D_\eta (\frac{U}{\eta})\big)\Big|_2+\Big|\big(U,D_\eta U,\frac{U}{\eta}\big)\Big|_\infty\Big)\leq C(T).
\end{equation}
Note that the calculations in \eqref{2211}--\eqref{2211*} rely on the global uniform upper and lower bounds of $(\eta_r,\frac{\eta}{r})$ and the estimates on $V$ given in \eqref{38'}--\eqref{38'''}.

Afterward, in \S \ref{sub92}--\S\ref{sub94}, we derive the second- and third-order exterior elliptic estimates for $U$ and the $L^1(0,T)$-estimate and time-weighted estimate for $\cD(t,U)$. Here, we take the $L^2([0,T];L^2)$-estimates for $\chi^\sharp\rho_0^{(\frac{3}{2}-\varepsilon_0)\beta} D_\eta^4 U$ as an example, and let $\cE_{\mathrm{ex}}(t,U)\in L^\infty(0,T)$, $\cD_{\mathrm{in}}(t,U)\in L^1(0,T)$, and $\chi^\sharp\rho_0^{(\frac{3}{2}-\varepsilon_0)\beta} D_\eta^2 U_t\in L^2([0,T];L^2)$. From \eqref{2211}--\eqref{2211*} and these assumptions, we can show that the crossing term $\mathcal{T}_{\mathrm{cross}}:=(D_\eta^3 U)_r+(\beta^{-1}+2) \rho_0^{-\beta}(\rho_0^\beta)_r D_\eta^3 U$ satisfies a Gr\"onwall-type inequality:
\begin{equation*}
\big|\zeta^\sharp \rho_0^{(\frac{3}{2}-\varepsilon_0)\beta}\mathcal{T}_{\mathrm{cross}}\big|_2\leq C(T)\Big(\int_0^t \big|\chi^\sharp\rho_0^{(\frac{3}{2}-\varepsilon_0)\beta} D_\eta^4 U\big|_2\,\mathrm{d}s+\big|\chi^\sharp\rho_0^{(\frac{3}{2}-\varepsilon_0)\beta} D_\eta^2 U_t\big|_2+1\Big).
\end{equation*}
Then Proposition \ref{prop2.1} in Appendix \ref{subsection2.2} and $\cE_{\mathrm{ex}}(t,U)\in L^\infty(0,T)$ give 
\begin{equation*}
\begin{aligned}
\big|\chi^\sharp \rho_0^{(\frac{3}{2}-\varepsilon_0)\beta} D_\eta^4 U\big|_2&\leq C(T)\big(\cD_{\mathrm{in}}(t,U)^\frac{1}{2}+ \big|\chi^\sharp \rho_0^{(\frac{3}{2}-\varepsilon_0)\beta} D_\eta^3 U\big|_2 + \big|\zeta^\sharp \rho_0^{(\frac{3}{2}-\varepsilon_0)\beta}\mathcal{T}_{\mathrm{cross}}\big|_2\big)\\
&\leq C(T)\Big(\int_0^t\big|\chi^\sharp \rho_0^{(\frac{3}{2}-\varepsilon_0)\beta} D_\eta^4 U\big|_2\,\mathrm{d}s+\big|\chi^\sharp \rho_0^{(\frac{3}{2}-\varepsilon_0)\beta} D_\eta^2 U_t\big|_2+\cD_{\mathrm{in}}(t,U)^\frac{1}{2}\Big),
\end{aligned}
\end{equation*}
which, along with the Gr\"onwall inequality, leads to the desired estimate.

Finally, based on the preceding estimates and the continuation argument in \S \ref{Section-global}, we can extend the local classical solution $(U,\eta)$ in Theorem \ref{local-Theorem1.1} to the global one.

\section{Global-In-Time Uniform   Lower Bounds of \texorpdfstring{$(\eta_r,\frac{\eta}{r})$}{}}\label{Section-densityupper}

The purpose of this section is to establish the global-in-time upper bound of the density $\varrho$ and the global uniform positive lower bounds of $(\eta_r,\frac{\eta}{r})$.
For simplicity,  we first  define  a solution class $ D(T)$  as follows:
\begin{Definition}\label{d2}
Let $T>0$, $\gamma\in (1,\infty)$ if $n=2$, and $\gamma\in (1,3)$ if  $n=3$. For  {\rm \textbf{IBVP}} \eqref{eq:VFBP-La-eta},  a  solution $(U,\eta)(t,r)$ is said to be in the class $ D(T)$ if  the following conditions hold{\rm:}
\begin{itemize}
\item $(U,\eta)(t,r)$ is a classical solution of  {\rm\bf  IBVP} \eqref{eq:VFBP-La-eta} in $[0,T]\times \bar I$ defined by {\rm Definition \ref{definition-lag}}{\rm;}
\item the boundary conditions in  \eqref{N111} hold{\rm:}
\begin{equation*}
U|_{r=0}=U_r|_{r=1}=0 \qquad\text{on $(0,T]$};
\end{equation*}
\item the initial data $(\rho_0,u_0)$ satisfy
\begin{equation*}
u_0\in C^1(\bar I), \qquad (1-r)\sim \rho_0^\beta\in C^1(\bar I)\cap C^2((0,1]) \quad \text{for some $\beta\in (0,\gamma-1]$}.
\end{equation*}
\end{itemize}
\end{Definition}
Such local existence theory  in the class $D(T_*)$  has been established in Theorem \ref{local-Theorem1.1} when $\beta\in (\frac{1}{3},\gamma-1]$ and $\gamma \in (\frac{4}{3},\infty)$ in two and three spatial dimensions.
Nevertheless, it is important to observe that the key \emph{a priori} estimates developed in
{\rm\S \ref{Section-densityupper}--\S \ref{Section-densitylower}}\,---\,namely, the uniform upper bound of the density, the uniform lower and upper bounds of  $(\eta_r,\frac{\eta}{r})$, 
the uniform estimates for the effective velocity, 
and the lower bound of the density near the origin\,---\,in fact hold for general classical solutions in the class 
$D(T_*)$ as defined in {\rm Definition \ref{d2}}. 
In a word, the regularity requirement $\cE(0,U)<\infty$ imposed in \eqref{a2} is not required for the analysis in {\rm\S \ref{Section-densityupper}--\S \ref{Section-densitylower}}.

Therefore, throughout \S\ref{Section-densityupper}--\S\ref{Section-densitylower}, we always assume the following:
\begin{itemize}
\item $(\beta,\gamma)$ satisfy
\begin{equation*}
\beta\in (0,\gamma-1],\qquad \gamma\in (1,\infty) \ \ \text{if }n=2, \qquad  \gamma\in (1,3) \ \ \text{if } n=3;
\end{equation*}
\item $(U,\eta)(t,r)$ is a classical solution in the class $D(T_*)$ of {\rm\textbf{IBVP}} \eqref{eq:VFBP-La-eta} in $[0,T_*]\times \bar I $ for some $T_*>0$, as defined in {\rm Definition \ref{d2}}; let $\overline T_*$ denote the maximal existence time of such a
classical solution, and let $0<T<\overline T_*$ be arbitrary;
\smallskip
\item $C_0\in (1,\infty)$ is a generic constant depending only on $(n,\mu,A,\gamma,\beta,\varepsilon_0,\rho_0,u_0,\cK_1,\cK_2)$, and $C(l_1,\cdots\!,l_k)\in (1,\infty)$ is a generic constant depending on $C_0$ and parameters $(l_1,\cdots\!,l_k)$, 
which may be different at each occurrence, \textit{i.e.},
\begin{equation*}
C_0=C_0(n,\mu,A,\gamma,\beta,\varepsilon_0,\rho_0,u_0,\cK_1,\cK_2),\qquad C(l_1,\cdots\!,l_k)=C(l_1,\cdots\!,l_k,C_0);
\end{equation*}
in particular, the constants $C(\cdots\!,T)$ are chosen increasing in $T$ and finite for any $T$. 
\end{itemize}
Notice that, for this solution $(U,\eta)$,   density $\varrho$ can be given by \eqref{eq:eta}, and $(\varrho,U,\eta)(t,r)$ still solves  problem \eqref{eq:VFBP-La} in $[0,T]\times \bar I $.

Now we are ready to introduce the so-called effective velocity.
\begin{Definition}\label{def-v}
We say that $V$ is the effective velocity if 
\begin{equation}\label{v-expression}
V=U+2\mu\frac{\eta^m \varrho_r}{r^m\rho_0}=U+2\mu D_\eta\log \varrho.
\end{equation}
Besides, we define the initial value of $V$ as $v_0=V|_{t=0}=u_0+2\mu(\log\rho_0)_r$.
\end{Definition}

Then we have
\begin{Proposition}\label{lemma-V expression}
The effective velocity $V$ satisfies the equation{\rm :}
\begin{equation}\label{eq:v}
V_t+ A\gamma \varrho^{\gamma-2} D_\eta\varrho=V_t+\frac{A\gamma}{2\mu}\varrho^{\gamma-1}(V-U)=0,
\end{equation}
and takes the form{\rm:}
\begin{equation}\label{V-solution}
\begin{aligned}
V(t,r)&=v_0(r)\exp\Big(-\frac{A\gamma}{2\mu}\int_0^t \varrho^{\gamma-1}(s,r)\,\mathrm{d}s\Big)\\
&\quad+\frac{A\gamma}{2\mu}\int_0^t (\varrho^{\gamma-1} U)(\tau,r)\exp\Big(-\frac{A\gamma}{2\mu}\int_\tau^t \varrho^{\gamma-1}(s,r)\,\mathrm{d}s\Big) \,\mathrm{d}\tau.
\end{aligned}
\end{equation}
\end{Proposition}
\begin{proof}
To derive \eqref{eq:v}, we use \eqref{eq:VFBP-La} and Definition \ref{def-v} to obtain
\begin{equation*}
\begin{aligned}
V_t&=U_t+2\mu(D_\eta \log\varrho)_t=U_t+2\mu D_\eta (\log\varrho)_t-2\mu(D_\eta U) (D_\eta\log\varrho) \\
&=U_t-2\mu D_\eta\big(D_\eta U+\frac{mU}{\eta}\big)-2\mu D_\eta U\frac{D_\eta \varrho}{\varrho} \\
&=U_t-\frac{2\mu}{\varrho} D_\eta\Big(\varrho\big(D_\eta U+\frac{mU}{\eta}\big)\Big)+\frac{2\mu}{\varrho}\frac{m D_\eta\varrho U }{\eta}\\
&=-A\gamma \varrho^{\gamma-2}D_\eta\varrho=-\frac{A\gamma}{2\mu}\varrho^{\gamma-1}(V-U).
\end{aligned}
\end{equation*}
Then \eqref{V-solution} can be directly obtained by solving ODE \eqref{eq:v}. 
\end{proof}

\subsection{Some basic estimates}
First, we have the fundamental energy estimate:
\begin{Lemma}\label{lemma-basic energy}
There exists a constant $C_0>0$ such that, for any $t\in [0,T]$, 
\begin{equation*}
\big|(r^m\rho_0)^\frac{1}{2} U(t)\big|_2^2+\big|\eta^m\eta_r\varrho^\gamma(t)\big|_1+\int_0^t \Big|(r^m\rho_0)^\frac{1}{2}\big(D_\eta U,\frac{U}{\eta}\big)\Big|_2^2\,\ds\leq C_0.
\end{equation*}
\end{Lemma}
\begin{proof}
Multiplying $\eqref{eq:VFBP-La-eta}_1$ by $2\eta^m\eta_rU$, together with $\eqref{eq:VFBP-La}_1$ and \eqref{eq:eta}, gives
\begin{equation*}
\big(r^m\rho_0U^2+\frac{2A}{\gamma-1}\eta^m\eta_r\varrho^\gamma\big)_t+4\mu r^m\rho_0\big(|D_\eta U|^2+ \frac{mU^2}{\eta^2}\big)=\big(4\mu r^m\rho_0\frac{U D_\eta U}{\eta_r}-2A\varrho^{\gamma}\eta^m U\big)_r.
\end{equation*}
Then integrating the above over $[0,t]\times I$ with $t\in [0,T]$ leads to the desired result. 
\end{proof}

Next, the following $\eta$-weighted estimates hold for $\varrho$\,:
\begin{Lemma}\label{lemma-far depth}
There exists a constant $C_0>0$ such that, for any $t\in [0,T]$,
\begin{equation}
|\eta^{m-1}\eta_r\varrho(t)|_1+|(\eta^m \varrho) (t)|_\infty\leq C_0.
\end{equation}
\end{Lemma}
\begin{proof}
Multiplying $\eqref{eq:VFBP-La-eta}_1$ by $\eta^m\eta_r$, along with  $\eqref{eq:VFBP-La}_1$, yields
\begin{equation}\label{lem42-iden}
(2\mu(\eta^m \varrho)_r+r^m\rho_0U-2\mu m\eta^{m-1}\eta_r\varrho)_t= Am\eta^{m-1}\eta_r\varrho^{\gamma}-A(\eta^m\varrho^{\gamma})_r.
\end{equation}
Then, for arbitrary $\tilde{r}\in [0,1)$, integrating the above with respect to $r$ over $[\tilde{r},1]$, together with $\varrho|_{r=1}=0$ and \eqref{eq:eta}, gives
\begin{equation}\label{lem42-den3}
\begin{aligned}
&\frac{\mathrm{d}}{\mathrm{d}t}\Big(2\mu(\eta^m \varrho)(t,\tilde{r})-\int_{\tilde{r}}^1 r^m\rho_0 U \,\mathrm{d}r+2\mu m\int_{\tilde{r}}^1\eta^{m-1}\eta_r\varrho\,\mathrm{d}r\Big)\\
&=-A m\int_{\tilde{r}}^1 \eta^{m-1}\eta_r \varrho^\gamma \,\mathrm{d}r-A(\eta^m \varrho^\gamma)(t,\tilde{r}),
\end{aligned}    
\end{equation}
which, along with the fact that $\eta_r\geq 0$, implies that 
\begin{equation}\label{lem42-den}
\frac{\mathrm{d}}{\mathrm{d}t}\Big(2\mu(\eta^m \varrho)(t,\tilde{r})+2\mu m\int_{\tilde{r}}^1\eta^{m-1}\eta_r\varrho\,\mathrm{d}r\Big)\leq \frac{\mathrm{d}}{\dt}\int_{\tilde{r}}^1 r^m\rho_0 U \,\mathrm{d}r.
\end{equation}

Integrating the above over $[0,t]$, we obtain from Lemma \ref{lemma-basic energy} and the H\"older inequality that
\begin{equation*}
\begin{aligned}
\sup_{t\in[0,T]}\Big(\eta^m \varrho (t,\tilde{r})+\int_{\tilde{r}}^1\eta^{m-1}\eta_r\varrho\,\mathrm{d}r\Big)&\leq C_0\big(\sup_{t\in[0,T]}|r^m\rho_0 U|_1+ |r^{m-1}\rho_0|_1+ |r^m \rho_0|_\infty\big)\\
&\leq C_0\big(\sup_{t\in[0,T]}\big|(r^m\rho_0)^\frac{1}{2}U\big|_2|r^m\rho_0|_1^\frac{1}{2}+1\big)\leq C_0.
\end{aligned}
\end{equation*}
Since $C_0$ is independent of the choice of $\tilde{r}$, we derive the desired conclusion. 
\end{proof}

We first record the lower bounds that follow directly from the weighted density estimate. These bounds suffice for the interior estimates and the weighted velocity estimates below.
\begin{Lemma}\label{lemma-low jacobi 1}
For any $t\in [0,T]$, we have
\begin{equation*}
\eta_r(t,r)\geq C_0^{-1}r^m\rho_0(r),\qquad
\eta(t,r)\geq C_0^{-1}\int_0^r \hat r^m\rho_0(\hat r)\,\mathrm{d}\hat r.
\end{equation*}
In particular, for any $0<a<b<1$, there exist constants $C(a)>1$ and $C(a,b)>1$ such that
\begin{equation*}
\begin{aligned}
\eta(t,r)&\geq C(a)^{-1},\quad \varrho(t,r)\leq C(a)\quad &\text{for $a\leq r\leq1$},\\
\eta_r(t,r)&\geq C(a,b)^{-1}\quad &\text{for $a\leq r\leq b$}.
\end{aligned}
\end{equation*}
\end{Lemma}
\begin{proof}
It follows from Lemma \ref{lemma-far depth} and the mass identity that
\begin{equation*}
\frac{r^m\rho_0(r)}{\eta_r(t,r)}=\eta^m(t,r)\varrho(t,r)\leq C_0.
\end{equation*}
This gives the first estimate. Integrating it from $0$ to $r$ and using $\eta(t,0)=0$ gives the second estimate. The positivity of $\rho_0$ on every compact subinterval of $[0,1)$ then gives the stated lower bounds. Finally,
\begin{equation*}
\varrho(t,r)=\eta^{-m}(t,r)(\eta^m\varrho)(t,r)\leq C(a)\qquad \text{for $a\leq r\leq1$}.
\end{equation*}

This completes the proof.
\end{proof}

\subsection{Interior BD entropy estimates}
\begin{Lemma}\label{near-BD}
For any $a\in (0,1)$, there exists a constant $C(a,T)>0$ such that
\begin{equation*}
\big|(\zeta_{a}r^m\rho_0)^\frac{1}{2}V(t)\big|_2+\big|(\zeta_{a}\eta^m\eta_r)^\frac{1}{2}D_\eta \sqrt{\varrho}(t)\big|_2\leq C(a,T) \qquad \text{for all $t\in [0,T]$},
\end{equation*}
where the cut-off function $\zeta_a$ is defined in {\rm\S \ref{othernotation}}.
\end{Lemma}
\begin{proof}
Multiplying \eqref{eq:v} by $\zeta_{a}r^m \rho_0V$ and integrating over $I$ give that
\begin{equation}\label{dt-G1}
\begin{aligned}
&\,\frac{1}{2}\frac{\mathrm{d}}{\dt}\int_0^1 \zeta_{a} r^m\rho_0V^2\,\mathrm{d}r+2\mu A\gamma \int_0^1 \zeta_{a}\eta^{m}\eta_r\varrho^{\gamma-2}|D_\eta\varrho|^2 \,\mathrm{d}r\\
&=-A \int_0^1 \zeta_{a}\eta^m\eta_r D_\eta(\varrho^\gamma) U\,\mathrm{d}r=
-A \int_0^1 \zeta_{a}\eta^m(\varrho^\gamma)_r U\,\mathrm{d}r
\\
&=A \int_0^1 \zeta_{a} \varrho^\gamma \eta_r D_\eta(\eta^m U)\,\mathrm{d}r+A\int_0^1 (D_\eta \zeta_{a})  \eta^m \eta_r \varrho^\gamma U \,\mathrm{d}r:=\sum_{i=1}^2\mathrm{G}_{i}. 
\end{aligned}
\end{equation}

Then, for $\mathrm{G}_{1}$, it follows from \eqref{eq:eta} that
\begin{equation}\label{G1,1}
\begin{aligned}
\mathrm{G}_{1}&=A \int_0^1 \zeta_{a}\frac{(r^{m}\rho_0)^{\gamma}}{(\eta^{m}\eta_r)^\gamma} (\eta^{m} \eta_r)_t\,\mathrm{d}r=-\frac{A}{\gamma-1} \frac{\mathrm{d}}{\mathrm{d}t}\int_0^1 \zeta_{a}\frac{(r^{ m}\rho_0)^{\gamma}}{(\eta^{m}\eta_r)^{\gamma-1}} \,\mathrm{d}r\\
&=-\frac{A}{\gamma-1} \frac{\mathrm{d}}{\mathrm{d}t}\int_0^1\zeta_{a}\eta^m \eta_r \varrho^\gamma \,\mathrm{d}r.
\end{aligned}
\end{equation}

For $\mathrm{G}_{2}$, recall that $\operatorname{supp}(\zeta_a)_r\subset[a,\frac{1+3a}{4}]$. We therefore obtain from \eqref{eq:eta}, $\rho_0^\beta\sim1-r$, Lemmas \ref{lemma-basic energy} and \ref{lemma-low jacobi 1}, and the H\"older inequality that
\begin{equation}\label{G1,2}
\begin{aligned}
\mathrm{G}_{2}&= A\int_{0}^{1}  (\zeta_{a})_r\frac{(r^{m}\rho_0)^{\gamma}U}{\eta^{(\gamma-1)m}\eta_r^\gamma} \,\mathrm{d}r \leq C(a,T) \int_{a}^{\frac{1+3a}{4}} (r^{m}\rho_0)^{\gamma}|U|\,\mathrm{d}r\\
& \leq C(a,T) \int_{a}^{\frac{1+3a}{4}} (r^m\rho_0)^\frac{1}{2} |U| \,\mathrm{d}r\leq C(a,T) \Big(\int_0^1 r^{m} \rho_0 U^2 \,\mathrm{d}r\Big)^\frac{1}{2}\leq C(a,T).  
\end{aligned}
\end{equation}

Consequently, substituting \eqref{G1,1}--\eqref{G1,2} into \eqref{dt-G1} and integrating the resulting inequality over $[0,t]$ yield
\begin{equation}\label{L2-V}
\begin{aligned}
&\,\big|(\zeta_{a} r^m\rho_0)^\frac{1}{2}V(t)\big|_2^2+\big|\zeta_{a}\eta^m\eta_r \varrho^\gamma(t)\big|_1+\int_0^t \big|(\zeta_{a}\eta^m\eta_r)^\frac{1}{2}  D_\eta (\varrho^\frac{\gamma}{2})\big|_2^2\,\mathrm{d}s\\
&\leq C_0(\big|(\zeta_{a} r^m\rho_0)^\frac{1}{2}v_0\big|_2^2+\big|\zeta_{a}r^m \rho_0^\gamma\big|_1)+C(a,T)\leq C(a,T).
\end{aligned}
\end{equation}
Here, to obtain the boundedness of the initial data, it suffices to note that
$\rho_0^\beta\sim 1-r$ and
\begin{equation*}
|(\zeta_{a} r^m\rho_0)^\frac{1}{2}(\log\rho_0)_r|_2\leq C(a)|r^\frac{m}{2}(\rho_0^\beta)_r|_2\leq C(a).
\end{equation*}

Thanks to \eqref{v-expression}, we thus obtain from Lemma \ref{lemma-basic energy} and \eqref{L2-V} that
\begin{equation}
\big|(\zeta_{a}\eta^m\eta_r)^\frac{1}{2}  D_\eta\sqrt{\varrho}(t)\big|_2 \leq C(a,T),
\end{equation}
which leads to the desired estimates.
\end{proof}

\subsection{Some new global $\eta$-weighted estimates for the density}
\begin{Lemma}\label{lemma-near depth}
Let parameters $(q_1,q_2)$ satisfy
\begin{equation*}
\begin{aligned}
q_1&\in (0,2], \quad  q_2\in (0,1] && \ \qquad \text{if $n=2$};\\
q_1&\in [1,3], \quad \,q_2\in [1,2] && \ \qquad \text{if $n=3$}.
\end{aligned}
\end{equation*}
Then, for any $a\in (0,1)$ and $(q_1,q_2)$ defined above, there exist positive constants $C(a,q_1,T)$ and $C(a,q_2,T)$ such that, for all $t\in [0,T]$,
\begin{equation}\label{fff}
|\zeta_{a} \eta^{q_1-1}\eta_r \varrho(t)|_1\leq C(a,q_1,T),\qquad 
|\zeta_{a}\eta^{q_2} \varrho(t)|_\infty\leq C(a,q_2,T).
\end{equation}
\end{Lemma}
\begin{proof}
We divide the proof into two steps.

\smallskip
\textbf{1. Case $n=2$ ($m=1$).}
First, for the first estimate in \eqref{fff}, if $q_1\in [1,2]$, it follows from Lemma \ref{lemma-far depth}, \eqref{eq:eta}, and the H\"older inequality that, for all $a\in (0,1)$ and $t\in [0,T]$,
\begin{equation}\label{claim0}
|\zeta_{a} \eta^{q_1-1}\eta_r \varrho|_1\leq |\eta^{q_1-1}\eta_r\varrho|_1\leq |\eta_r\varrho|_1^{2-q_1}|\eta\eta_r\varrho|_1^{q_1-1}=|\eta_r\varrho|_1^{2-q_1}|r\rho_0|_1^{q_1-1}\leq  C_0.
\end{equation}

If $q_1\in (0,1)$,  we claim that, for all $a\in (0,1)$ and $t\in [0,T]$,
\begin{equation}\label{413}
|\zeta_{a}\eta^{q_1-1}\eta_r\varrho|_1\leq |\zeta_{a}\eta^{2q_1-1}\eta_r\varrho|_1+C(a,q_1,T).
\end{equation}
Indeed, since $\eta|_{r=0}=\varrho|_{r=1}=0$ and $\eta,\eta_r,\varrho\geq 0$, the cut-off derivative is supported in $[a,\frac{1+3a}{4}]$. Thus integration by parts, \eqref{eq:eta}, $\rho_0^\beta\sim1-r$, the compact-subinterval bounds in Lemma \ref{lemma-low jacobi 1}, Lemma \ref{near-BD}, and the H\"older and Young inequalities give
\begin{equation*}
\begin{aligned}
|\zeta_{a}\eta^{q_1-1}\eta_r\varrho|_1 &=\frac{1}{q_1}\int_0^1 \zeta_{a} \varrho\,\mathrm{d}(\eta^{q_1})=-\frac{2}{q_1}\int_0^1 \zeta_{a}\eta^{q_1}\sqrt{\varrho}(\sqrt{\varrho})_r\,\mathrm{d}r-\frac{1}{q_1}\int_0^1 (\zeta_{a})_r\eta^{q_1} \varrho\,\mathrm{d}r\\
&\leq \frac{2}{q_1}\big|(\zeta_{a}\eta\eta_r)^\frac{1}{2} D_\eta\sqrt{\varrho} \big|_2|\zeta_{a}\eta^{2q_1-1}\eta_r\varrho|_1^\frac{1}{2}+C(a,q_1)\int_{a}^{\frac{1+3a}{4}} \frac{r\rho_0}{\eta^{1-q_1}\eta_r}\,\mathrm{d}r\\
&\leq C(a,q_1,T) \Big(|\zeta_{a}\eta^{2q_1-1}\eta_r\varrho|_1^\frac{1}{2}+ \int_{a}^{\frac{1+3a}{4}} r(1-r)^\frac{1}{\beta}\,\mathrm{d}r\Big)\\
&\leq |\zeta_{a}\eta^{2q_1-1}\eta_r\varrho|_1+C(a,q_1,T),
\end{aligned}
\end{equation*}
which implies claim \eqref{413}.

Now, for each $q_1\in (0,1)$, there exists a fixed $\ell\in \NN^*$ depending only on $q_1$ such that 
\begin{equation*}
0<2^{\ell-1}q_1<1\leq 2^\ell q_1<2.
\end{equation*}
We can iteratively use \eqref{413} and then obtain from \eqref{claim0} that, for any $t\in [0,T]$,
\begin{equation}\label{417}
\begin{aligned}
|\zeta_{a}\eta^{q_1-1}\eta_r\varrho|_1&\leq \big|\zeta_{a}\eta^{2^{\ell-1}q_1-1}\eta_r\varrho\big|_1+C(a,q_1,T)\\
&\leq\big|\zeta_{a}\eta^{2^\ell q_1-1}\eta_r \varrho\big|_1+C(a,q_1,T)\leq C(a,q_1,T).
\end{aligned}    
\end{equation}
This completes the proof of $\eqref{fff}_1$.

Finally, we show $\eqref{fff}_2$. For any $q_2\in (0,1]$ and $a\in (0,1)$, since $\mathrm{supp}(\zeta_a)_r\subset [a,\frac{1+3a}{4}]$, by $\eqref{fff}_1$,  Lemmas \ref{lemma-low jacobi 1}--\ref{near-BD} and \ref{sobolev-embedding}, and the H\"older inequality, we have
\begin{equation*}
\begin{aligned}
|\zeta_{a}\eta^{q_2} \varrho|_\infty &\leq C_0|(\zeta_{a} \eta^{q_2}\varrho)_r|_1\leq C_0\big|\big((\zeta_{a})_r \eta^{q_2}\varrho,q_2 \zeta_{a}\eta^{q_2-1}\eta_r\varrho,2 \zeta_{a}\eta^{q_2}\sqrt{\varrho}(\sqrt{\varrho})_r\big)\big|_1 \\
&\leq C(a)\int_a^\frac{1+3a}{4} \!\!\eta^{q_2}\varrho\,\mathrm{d}r+C(q_2)|\zeta_{a}\eta^{q_2-1}\eta_r\varrho|_1+ C_0\big|(\zeta_{a}\eta \eta_r)^\frac{1}{2} D_\eta\sqrt{\varrho}\big|_2\big|\zeta_{a} \eta^{2q_2-1}\eta_r \varrho\big|_1^\frac{1}{2}\\
&\leq C(a,T)\big|\zeta_{\frac{1+3a}{4}} \eta^{q_2}\eta_r\varrho\big|_1+C(a,q_2,T)\leq C(a,q_2,T).
\end{aligned}    
\end{equation*}

This completes the proof of Lemma \ref{lemma-near depth} when $n=2$ $(m=1)$.

\smallskip
\textbf{2. Case $n=3$ ($m=2$).} 
First, it follows from integration by parts, $\mathrm{supp}(\zeta_a)_r\subset [a,\frac{1+3a}{4}]$, Lemmas \ref{lemma-far depth}--\ref{near-BD}, and the H\"older inequality that, for all $t\in [0,T]$ and $a\in (0,1)$,
\begin{equation*}
\begin{aligned}
|\zeta_{a}\eta_r \varrho|_1 &=\int_0^1 \zeta_{a} \varrho\,\mathrm{d}\eta=-\int_0^1 (\zeta_{a})_r\eta \varrho\,\mathrm{d}r-2\int_0^1\zeta_{a}\eta \sqrt{\varrho}(\sqrt{\varrho})_r\,\mathrm{d}r\\
&\leq C(a)\int_a^\frac{1+3a}{4} \eta \varrho\,\mathrm{d}r+2\big|(\zeta_{a}\eta^2\eta_r)^\frac{1}{2} D_\eta\sqrt{\varrho}\big|_2|\zeta_{a}\eta_r \varrho|_1^\frac{1}{2}\\
&\leq C(a,T)\big(|\eta\eta_r\varrho|_1+ |\zeta_{a}\eta_r \varrho|_1^\frac{1}{2}\big) 
\leq C(a,T)\big(1+ |\zeta_{a}\eta_r \varrho|_1^\frac{1}{2}\big),
\end{aligned}
\end{equation*}
which, along with the Young inequality,  yields
\begin{equation*} 
|\zeta_{a}\eta_r \varrho(t)|_1\leq C(a,T)\qquad \text{for all $t\in[0,T]$}.
\end{equation*}
This, together with \eqref{eq:eta} and the H\"older inequality, yields that, for all $q_1\in [1,3]$, 
\begin{equation*}
|\zeta_{a}\eta^{q_1-1}\eta_r \varrho|_1\leq|\zeta_{a}\eta_r \varrho|_1^\frac{3-q_1}{2} |\zeta_{a}\eta^2\eta_r \varrho|_1^\frac{q_1-1}{2}\leq |\zeta_{a}\eta_r \varrho|_1^\frac{3-q_1}{2} |r^2\rho_0|_1^\frac{q_1-1}{2}\leq C(a,T), 
\end{equation*}
which implies $\eqref{fff}_1$.

Next, using $\eqref{fff}_1$, Lemmas \ref{lemma-far depth}--\ref{near-BD} and \ref{sobolev-embedding}, and $\mathrm{supp}(\zeta_a)_r\subset [a,\frac{1+3a}{4}]$, we obtain
\begin{equation*}
\begin{aligned}
|\zeta_{a}\eta \varrho|_\infty &\leq C_0|(\zeta_{a}\eta \varrho)_r|_1 \leq C_0\big|((\zeta_{a})_r\eta \varrho,\,\zeta_{a} \eta_r\varrho,\,\zeta_{a}\eta \sqrt{\varrho}(\sqrt{\varrho})_r)\big|_1\\
&\leq  C(a)\int_a^\frac{1+3a}{4}  \eta \varrho\,\mathrm{d}r+C_0|\zeta_{a} \eta_r \varrho|_1  + C_0\big|(\zeta_{a}\eta^2\eta_r)^\frac{1}{2} D_\eta\sqrt{\varrho}\big|_2 | \zeta_{a}\eta_r \varrho|_1^\frac{1}{2} \\
&\leq C(a,T) \big(|\eta\eta_r \varrho|_1+| \zeta_{a}\eta_r \varrho|_1+| \zeta_{a}\eta_r \varrho|_1^\frac{1}{2}\big)\leq C(a,T),    
\end{aligned}
\end{equation*}
which, along with $|\eta^2 \varrho|_\infty\leq C_0$ in Lemma \ref{lemma-far depth}, leads to $\eqref{fff}_2$. 

This completes the proof of Lemma \ref{lemma-near depth} when $n=3$ ($m=2$).
\end{proof}

\subsection{\texorpdfstring{$L^p$}{}-estimates for \texorpdfstring{$(r^m\rho_0)^\frac{1}{p}(U,V)$}{}} 
Our goal of this subsection is to establish the $L^p$-energy estimates ($p\in[2,\infty)$) for $(U,V)$, which can be stated as follows:
\begin{Lemma}\label{lp-uv}  
Let $\gamma\in (1,\infty)$ if $n=2$, and $\gamma\in (1,3)$ if $n=3$. Then, for any $p\in [2,\infty)$, there exists a constant $C(p,T)>0$ such that, for all $t\in[0,T]$,
\begin{equation}\label{zong-jiehe}
\big|(r^m\rho_0)^\frac{1}{p}U(t)\big|_p^p+\big|(\zeta r^m\rho_0)^\frac{1}{p}V(t)\big|_p^p+\int_0^t \Big|(r^m\rho_0)^\frac{1}{2}|U|^\frac{p-2}{2}\big(D_\eta U,\frac{U}{\eta}\big)\Big|_2^2\,\mathrm{d}s\leq C(p,T),
\end{equation}
where the cut-off function $\zeta$ is defined in {\rm\S \ref{othernotation}}.
\end{Lemma}

The proof of Lemma \ref{lp-uv} will be divided into the following three cases:
\begin{itemize}
\item[\S\ref{subsub-1}:] Case when $n=2$ and $\gamma\in (1,\infty)$.
\item[\S\ref{subsub-2}:] Case when $n=3$ and $\gamma\in (1,2]$.
\item[\S\ref{subsub-3}:] Case when $n=3$ and $\gamma\in (2,3)$.  
\end{itemize}

\subsubsection{Case when $n=2$ and $\gamma\in (1,\infty)$}\label{subsub-1}
We first establish the $L^p$-energy estimates for $U$.
\begin{Lemma}\label{lemma-u Lp-n2}
Let $n=2$ and $\gamma\in (1,\infty)$. Then, for any $p\in[2,\infty)$, there exists a constant $C(p,T)>0$ such that
\begin{equation*}
\big|(r\rho_0)^\frac{1}{p}U(t)\big|_p^p+\int_0^t\Big|(r\rho_0)^\frac{1}{2}|U|^\frac{p-2}{2}\big(D_\eta U,\frac{U}{\eta}\big)\Big|_2^2\,\ds\leq C(p,T)
\qquad\text{for all $t\in [0,T]$}.
\end{equation*}
\end{Lemma}
\begin{proof}
Let $n=2$, $\gamma\in(1,\infty)$, and $p\in [2,\infty)$. Multiplying $\eqref{eq:VFBP-La-eta}_1$ by $\eta\eta_r |U|^{p-2}U$ gives
\begin{equation*}
\begin{aligned}
&\,\frac{1}{p}(r\rho_0|U|^p)_t+2\mu(p-1) r\rho_0  |U|^{p-2}|D_\eta U|^2 +2\mu  r\rho_0 \frac{|U|^p}{\eta^2} \\
&= \Big(\frac{2\mu r\rho_0}{\eta_r} |U|^{p-2}U D_\eta U-\frac{Ar^{\gamma }\rho_0^\gamma}{\eta^{\gamma-1}\eta_r^\gamma} |U|^{p-2}U\Big)_r+\frac{A(r\rho_0)^\gamma}{(\eta\eta_r)^{\gamma-1}}|U|^{p-2}\big((p-1)D_\eta U+\frac{U}{\eta}\big).
\end{aligned}
\end{equation*}
Then integrating the above over $I$, along with the H\"older and Young inequalities, yields
\begin{equation}\label{dt-G2n=2}
\begin{aligned}
&\,\frac{1}{p}\frac{\mathrm{d}}{\mathrm{d}t}\big|(r\rho_0)^\frac{1}{p}U\big|_p^p+\mu\Big|(r\rho_0)^\frac{1}{2}|U|^\frac{p-2}{2}\big(D_\eta U,\frac{U}{\eta}\big)\Big|_2^2\\
&\leq  \int_0^1 \frac{A(r\rho_0)^\gamma}{(\eta\eta_r)^{\gamma-1}}|U|^{p-2}\big((p-1)D_\eta U+\frac{U}{\eta}\big)\,\mathrm{d}r\\
&\leq C(p)\Big|\frac{(r\rho_0)^{\gamma-\frac{1}{2}}}{(\eta\eta_r)^{\gamma-1}}|U|^\frac{p-2}{2}\Big|_2\Big|(r\rho_0)^\frac{1}{2}|U|^\frac{p-2}{2}\big(D_\eta U,\frac{U}{\eta}\big)\Big|_2\\
&\leq C(p)\underline{\Big|\frac{(r\rho_0)^{\gamma-\frac{1}{2}}}{(\eta\eta_r)^{\gamma-1}}|U|^\frac{p-2}{2}\Big|_2^2}_{:=\mathrm{G}_{3}}+\frac{\mu}{10}\Big|(r\rho_0)^\frac{1}{2}|U|^\frac{p-2}{2}\big(D_\eta U,\frac{U}{\eta}\big)\Big|_2^2.
\end{aligned}    
\end{equation}

Now, we estimate $\mathrm{G}_{3}$. Let $\iota_0$ be a fixed constant such that 
\begin{equation*}
0<\iota_0<\min\{2,p\gamma-p\}.
\end{equation*}
Then it follows from the mass identity, Lemmas \ref{lemma-low jacobi 1} and \ref{lemma-near depth}, and the H\"older and Young inequalities that
\begin{equation}\label{G2:n=2}
\begin{aligned}
\mathrm{G}_{3}&=\int_0^1 \frac{(r\rho_0)^{2\gamma-1}}{(\eta\eta_r)^{2\gamma-2}}|U|^{p-2}\,\mathrm{d}r\leq \big|(r\rho_0)^\frac{1}{p}U\big|_p^{p-2}\Big(\int_0^1 \frac{(r\rho_0)^{p\gamma-p+1}}{(\eta\eta_r)^{p\gamma-p}}\,\mathrm{d}r\Big)^\frac{2}{p}\\
&=\big|(r\rho_0)^\frac{1}{p}U\big|_p^{p-2}\Big(\int_0^\frac{1}{2}\eta\eta_r\varrho^{p\gamma-p+1}\,\mathrm{d}r+\int_\frac{1}{2}^1 r\rho_0\varrho^{p\gamma-p}\,\mathrm{d}r\Big)^\frac{2}{p}\\
&\leq \big|(r\rho_0)^\frac{1}{p}U\big|_p^p+C(p)\big|\zeta \eta^\frac{\iota_0}{p\gamma-p}\varrho\big|_\infty^{p\gamma-p}|\zeta \eta^{1-\iota_0}\eta_r\varrho|_1+C(p,T)\\
&\leq \big|(r\rho_0)^\frac{1}{p}U\big|_p^p+C(p,T).
\end{aligned}
\end{equation}
Here the exterior integral is handled by Lemma \ref{lemma-low jacobi 1}. The cut-off function $\zeta$ is defined in \S \ref{othernotation}.

Substituting \eqref{G2:n=2} into \eqref{dt-G2n=2} and then applying the Gr\"onwall inequality to the resulting inequality yield the desired result.
\end{proof}

Based on Lemma \ref{lemma-u Lp-n2}, we obtain the following interior $L^p$-energy estimates for $V$:
\begin{Lemma}\label{lemma-v-n=2}
Let $n=2$ and $\gamma\in (1,\infty)$. Then, for any $p\in [2,\infty)$, there exists a constant $C(p,T)>0$ such that, for all $t\in[0,T]$,
\begin{equation*}
\big|(\zeta r\rho_0)^\frac{1}{p}V(t)\big|_p^p\leq C(p,T).
\end{equation*}
\end{Lemma}
\begin{proof}
Let $n=2$ and $\gamma\in (1,\infty)$. We first restrict parameter $p$ such that
\begin{equation}\label{p-par0}
\max\big\{2,\frac{2}{\gamma-1}\big\}\leq p<\infty.
\end{equation}
Multiplying \eqref{eq:v} by $\zeta r\rho_0 |V|^{p-2}V$, together with \eqref{eq:eta}, gives
\begin{equation*}
\frac{1}{p}(\zeta r\rho_0|V|^p)_t+\frac{A\gamma}{2\mu}\zeta \eta\eta_r \varrho^{\gamma}|V|^p=\frac{A\gamma}{2\mu}\zeta \eta \eta_r\varrho^{\gamma}|V|^{p-2}VU.
\end{equation*}
Then, integrating the above over $I$, we obtain from the fact that
\begin{equation*}
\zeta^k(r)\leq \zeta_\frac{5}{8}(r)\qquad\text{for all $r\in \bar I$ and $k>0$},
\end{equation*}
Lemma \ref{lemma-near depth}, and the Young inequality that
\begin{equation*}
\begin{aligned}
&\,\frac{1}{p}\frac{\mathrm{d}}{\mathrm{d}t}\big|(\zeta r\rho_0)^\frac{1}{p}V\big|_p^p+\frac{A\gamma}{2\mu}\big|(\zeta \eta\eta_r\varrho^{\gamma})^\frac{1}{p}  V\big|_p^p=\frac{A\gamma}{2\mu} \int_0^1 \zeta \eta \eta_r\varrho^{\gamma}|V|^{p-2}VU\,\mathrm{d}r\\
&\leq C(p) \big|(\zeta \eta\eta_r\varrho^{p\gamma-p+1})^\frac{1}{p}  U\big|_p^p+C(p)\big|(\zeta r\rho_0)^\frac{1}{p}V\big|_p^p \\
&\leq C(p)\big|\zeta_{\frac{5}{8}}\eta^\frac{2}{p\gamma-p}\varrho\big|_\infty^{p\gamma-p} \Big|\big(\frac{r\rho_0}{\eta^2}\big)^\frac{1}{p}U\Big|_p^p+C(p)\big|(\zeta r\rho_0)^\frac{1}{p}  V\big|_p^p\\
&\leq C(p,T)\Big|\big(\frac{r\rho_0}{\eta^2}\big)^\frac{1}{p}U\Big|_p^p+C(p)\big|(\zeta r\rho_0)^\frac{1}{p}  V\big|_p^p.
\end{aligned}
\end{equation*}
Notice that the forcing term is time integrable, since
\begin{equation*}
\int_0^T\Big|\Big(\frac{r\rho_0}{\eta^2}\Big)^\frac{1}{p}U\Big|_p^p\,\mathrm{d}t
=\int_0^T\int_0^1r\rho_0|U|^{p-2}\Big|\frac{U}{\eta}\Big|^2\,\mathrm{d}r\,\mathrm{d}t\leq C(p,T)
\end{equation*}
by Lemma \ref{lemma-u Lp-n2}. Hence, after integrating the preceding differential inequality and applying the Gr\"onwall inequality, we obtain, for every $p$ satisfying \eqref{p-par0},
\begin{equation}\label{vabove}
\sup_{0\leq t\leq T}\big|(\zeta r\rho_0)^\frac{1}{p}V(t)\big|_p^p
\leq C(p,T)\Big(1+\big|(\zeta r\rho_0)^\frac{1}{p}v_0\big|_p^p\Big)\leq C(p,T).
\end{equation}
Finally, using Lemma \ref{near-BD}  and interpolation on the finite measure space with measure $\zeta r\rho_0\,\mathrm{d}r$, we obtain \eqref{vabove} for every $p\in[2,\infty)$.
\end{proof}

Now we can prove Lemma \ref{lp-uv} for the case when $n=2$ and $\gamma\in(1,\infty)$.
\begin{proof}[Proof of Lemma \ref{lp-uv}]
Combining Lemmas \ref{lemma-u Lp-n2}--\ref{lemma-v-n=2}, we derive the desired estimates of Lemma \ref{lp-uv} when $n=2$ and $\gamma\in (1,\infty)$.
\end{proof}

\subsubsection{Case when $n=3$ and $\gamma\in(1,2]$}\label{subsub-2}
First, we establish the $L^p$-energy estimates for $U$.
\begin{Lemma}\label{lemma-u Lp-n3}
Let $n=3$ and $\gamma\in (1,2]$. Then, for any $p\in[2,\infty)$, there exists a constant $C(p,T)>0$ such that, for all $t\in [0,T]$, 
\begin{equation*}
\big|(r^2\rho_0)^\frac{1}{p}U(t)\big|_p^p+\int_0^t\Big|(r^2\rho_0)^\frac{1}{2}|U|^\frac{p-2}{2}\big(D_\eta U,\frac{U}{\eta}\big)\Big|_2^2\,\ds\leq C(p,T).
\end{equation*}
\end{Lemma}
\begin{proof}
We divide the proof into three steps.

\smallskip
\textbf{1.} Let $n=3$, $\gamma\in (1,2]$, and $p\in [2,\infty)$. Multiplying $\eqref{eq:VFBP-La-eta}_1$ by $\eta^2\eta_r |U|^{p-2}U$ gives
\begin{equation*}
\begin{aligned}
&\,\frac{1}{p}(r^2\rho_0|U|^p)_t+2\mu(p-1) r^2\rho_0|U|^{p-2}|D_\eta U|^2 +4\mu  r^2\rho_0 \frac{|U|^p}{\eta^2} \\
&= \Big(2\mu\frac{r^2\rho_0}{\eta_r} |U|^{p-2}U D_\eta U-\frac{Ar^{2\gamma}\rho_0^\gamma}{\eta^{2(\gamma-1)}\eta_r^\gamma} |U|^{p-2}U\Big)_r + \frac{A(r^{2}\rho_0)^\gamma}{(\eta^{2}\eta_r)^{\gamma-1}}|U|^{p-2}\big((p-1)D_\eta U+\frac{2U}{\eta}\big).
\end{aligned}
\end{equation*}
Then integrating the above over $I$, along with the H\"older and Young inequalities, yields
\begin{equation}\label{dt-G2,n3}
\begin{aligned}
&\,\frac{1}{p}\frac{\mathrm{d}}{\mathrm{d}t}\big|(r^2\rho_0)^\frac{1}{p}U\big|_p^p+\mu\Big|(r^2\rho_0)^\frac{1}{2}|U|^\frac{p-2}{2}\big(D_\eta U,\frac{U}{\eta}\big)\Big|_2^2\\
&\leq \int_0^1 \frac{A(r^{2}\rho_0)^\gamma}{(\eta^{2}\eta_r)^{\gamma-1}}|U|^{p-2}\big((p-1)D_\eta U+\frac{2U}{\eta}\big)\,\mathrm{d}r\\
&\leq C(p)\Big|\frac{(r^2\rho_0)^{\gamma-\frac{1}{2}}}{(\eta^2\eta_r)^{\gamma-1}}|U|^\frac{p-2}{2}\Big|_2\Big|(r^2\rho_0)^\frac{1}{2}|U|^\frac{p-2}{2}\big(D_\eta U,\frac{U}{\eta}\big)\Big|_2\\
&\leq C(p)\underline{\Big|\frac{(r^2\rho_0)^{\gamma-\frac{1}{2}}}{(\eta^2\eta_r)^{\gamma-1}}|U|^\frac{p-2}{2}\Big|_2^2}_{:=\mathrm{G}_{4}}+\frac{\mu}{10}\Big|(r^2\rho_0)^\frac{1}{2}|U|^\frac{p-2}{2}\big(D_\eta U,\frac{U}{\eta}\big)\Big|_2^2.
\end{aligned}   
\end{equation}

\smallskip
\textbf{2.} For $\mathrm{G}_4$, we first use \eqref{eq:eta}, Lemma \ref{lemma-low jacobi 1}, and the H\"older and Young inequalities to obtain
\begin{equation}\label{G2:n=3}
\begin{aligned}
\mathrm{G}_{4} &\leq \underline{\int_0^\frac{1}{2} \frac{(r^2\rho_0)^{2\gamma-1}}{(\eta^2\eta_r)^{2\gamma-2}}|U|^{p-2}\,\mathrm{d}r}_{:=\mathrm{G}_{4,1}} +\Big(\int_\frac{1}{2}^1\frac{(r^2\rho_0)^{p\gamma-p+1}}{(\eta^2\eta_r)^{p\gamma-p}}\,\mathrm{d}r\Big)^\frac{2}{p}\big|(r^2\rho_0)^\frac{1}{p}U\big|_p^{p-2}\\
&=\mathrm{G}_{4,1}+\Big(\int_\frac{1}{2}^1r^2\rho_0\varrho^{p\gamma-p}\,\mathrm{d}r\Big)^\frac{2}{p}\big|(r^2\rho_0)^\frac{1}{p}U\big|_p^{p-2} \leq \mathrm{G}_{4,1}+\big|(r^2\rho_0)^\frac{1}{p}U\big|_p^{p} +C(p,T).
\end{aligned}
\end{equation}

For $\mathrm{G}_{4,1}$, it follows from \eqref{eq:eta}, Lemma \ref{lemma-near depth}, and the H\"older and Young inequalities that, for any $\vartheta\in [0,1]$ and $\varepsilon\in (0,1)$,
\begin{equation}\label{g2,1-xiao}
\begin{aligned}
\mathrm{G}_{4,1}&=\int_0^\frac{1}{2} \frac{(r^2\rho_0)^{2\gamma-1-\frac{p-2}{p}}}{\eta^{4\gamma-4-\frac{\vartheta(2p-4)}{p}}\eta_r^{2\gamma-2}}\Big(\big(\frac{r^2\rho_0}{\eta^2}\big)^{\frac{p-2}{p}}|U|^{p-2}\Big)^{\vartheta}\Big((r^2\rho_0)^{\frac{p-2}{p}}|U|^{p-2}\Big)^{1-\vartheta}\,\mathrm{d}r\\
&\leq \Big(\int_0^\frac{1}{2} \eta^{2+\vartheta(p-2)}\eta_r\varrho^{p\gamma-p+1}\,\mathrm{d}r\Big)^\frac{2}{p}\Big|\big(\frac{r^2\rho_0}{\eta^2}\big)^\frac{1}{p}U\Big|_p^{\vartheta(p-2)}\big|(r^2\rho_0)^\frac{1}{p}U\big|_p^{(1-\vartheta)(p-2)}\\
&\leq \big|\zeta \eta^\frac{2+\vartheta(p-2)}{p\gamma-p} \varrho\big|_\infty^{2\gamma-2} \big|\zeta \eta_r\varrho\big|_1^\frac{2}{p}\Big|\big(\frac{r^2\rho_0}{\eta^2}\big)^\frac{1}{p}U\Big|_p^{\vartheta(p-2)}\big|(r^2\rho_0)^\frac{1}{p}U\big|_p^{(1-\vartheta)(p-2)}\\
&\leq C(\varepsilon,p,T)\big|\zeta \eta^\frac{2+\vartheta(p-2)}{p\gamma-p} \varrho\big|_\infty^{p\gamma-p}+ \varepsilon\Big|\big(\frac{r^2\rho_0}{\eta^2}\big)^\frac{1}{p}U\Big|_p^{\vartheta p}\big|(r^2\rho_0)^\frac{1}{p}U\big|_p^{(1-\vartheta)p}. 
\end{aligned}
\end{equation}

In order to apply Lemma \ref{lemma-near depth} to the first term on the right-hand side of \eqref{g2,1-xiao}, for each $\gamma\in (1,2]$, we need to choose suitable $(p,\vartheta)$, which depend only on $\gamma$, such that 
\begin{equation}\label{dengjia0}
1\leq \frac{2+\vartheta(p-2)}{p\gamma-p}\leq 2.
\end{equation}
Set $\vartheta=\gamma-1\in (0,1]$ above. Then 
\begin{equation*}
\eqref{dengjia0}\iff 0\leq \frac{4-2\gamma}{p\gamma-p}\leq 1,
\end{equation*}
which holds for all 
\begin{equation}\label{p-fanwei}
p\in [p_0(\gamma),\infty) \qquad \text{with $p_0(\gamma):=\max\big\{2,\frac{4-2\gamma}{\gamma-1}\big\}$}.
\end{equation}
This implies that \eqref{dengjia0} holds for $\vartheta=\gamma-1$, any $\gamma\in (1,2]$,
and any $p\in [p_0(\gamma),\infty)$. Consequently, based on \eqref{g2,1-xiao}, the above discussion, Lemma \ref{lemma-near depth}, and the Young inequality, for any $\varepsilon\in (0,1)$, $\gamma\in (1,2]$, and $p\in [p_0(\gamma),\infty)$ with $p_0(\gamma)$ defined in \eqref{p-fanwei}, we have
\begin{equation}\label{g2,1-xiao-de}
\mathrm{G}_{4,1}
\leq C_0\big|(r^2\rho_0)^\frac{1}{p}U\big|_p^{p}+\varepsilon\Big|\big(\frac{r^2\rho_0}{\eta^2}\big)^\frac{1}{p}U\Big|_p^{p}+C(\varepsilon,p,T).
\end{equation}

Substituting \eqref{g2,1-xiao-de} into \eqref{G2:n=3} gives that, for all $\varepsilon\in (0,1)$,
\begin{equation}\label{g3above}
\mathrm{G}_{4} \leq C_0\big|(r^2\rho_0)^\frac{1}{p}U\big|_p^{p}+\varepsilon\Big|\big(\frac{r^2\rho_0}{\eta^2}\big)^\frac{1}{p}U\Big|_p^{p}+C(\varepsilon,p,T). 
\end{equation}

\smallskip
\textbf{3.} Combining \eqref{dt-G2,n3} and \eqref{g3above}, then choosing a suitable small $\varepsilon\in(0,1)$ and applying the Gr\"onwall inequality to the resulting inequality, we have  
\begin{equation*}
\big|(r^2\rho_0)^\frac{1}{p}U(t)\big|_p^p+\int_0^t\Big|(r^2\rho_0)^\frac{1}{2}|U|^\frac{p-2}{2}\big(D_\eta U,\frac{U}{\eta}\big)\Big|_2^2\,\mathrm{d}s\leq C(p,T)
\end{equation*}
for all $p\in [p_0(\gamma),\infty)$ with $p_0(\gamma)$ defined in \eqref{p-fanwei}. Using Lemma \ref{lemma-basic energy} and the H\"older inequality, we obtain the full estimate for every $p\in[2,\infty)$.
\end{proof}

Based on Lemma \ref{lemma-u Lp-n3}, we obtain the following interior $L^p$-energy estimates for $V$.
\begin{Lemma}\label{lemma-v-n=3}
Let $n=3$ and $\gamma\in (1,2]$. Then, for any $p\in [2,\infty)$, there exists a constant $C(p,T)>0$ such that, for all $t\in[0,T]$,
\begin{equation*}
\big|(\zeta r^2\rho_0)^\frac{1}{p}V(t)\big|_p^p\leq C(p,T).
\end{equation*}
\end{Lemma}
\begin{proof}
We divide the proof into three steps.

\smallskip
\textbf{1.} Let $n=3$ and $\gamma\in (1,2]$. We first restrict the parameter $p$ such that
\begin{equation}\label{p-par}
\max\big\{2,\frac{1}{\gamma-1}\big\}<p<\infty.
\end{equation}
Multiplying \eqref{eq:v} by $\zeta r^2\rho_0 |V|^{p-2}V$ gives
\begin{equation*}
\frac{1}{p}(\zeta r^2\rho_0|V|^p)_t+\frac{A\gamma}{2\mu}\zeta \eta^2\eta_r \varrho^{\gamma}|V|^p=\frac{A\gamma}{2\mu}\zeta \eta^2 \eta_r\varrho^{\gamma}|V|^{p-2}VU.
\end{equation*}
Then integrating over $I$, together with the Young inequality, gives that, for all $\vartheta\in (0,1)$,
\begin{equation}\label{dt-v}
\begin{aligned}
&\,\frac{1}{p}\frac{\mathrm{d}}{\mathrm{d}t}\big|(\zeta r^2\rho_0)^\frac{1}{p}V\big|_p^p+\frac{A\gamma}{2\mu}\big|(\zeta \eta^2\eta_r\varrho^{\gamma})^\frac{1}{p}  V\big|_p^p\\
&\leq\frac{A\gamma}{2\mu} \int_0^1  \big(\zeta \eta^2\eta_r\varrho^{(\gamma-1)(\vartheta p+1-\vartheta)+1}\big)^\frac{1}{p} |U| \\
&\qquad\quad \qquad \times \big((\zeta \eta^2 \eta_r\varrho)^\frac{1}{p}|V|\big)^{\vartheta(p-1)}\big((\zeta \eta^2 \eta_r\varrho^\gamma)^\frac{1}{p}|V|\big)^{(1-\vartheta)(p-1)}\,\mathrm{d}r\\[4pt]
&\leq C(p)\underline{\Big|\zeta_{\frac{5}{8}}\eta^\frac{2}{(\gamma-1)(\vartheta p+1-\vartheta)}\varrho\Big|_\infty}^{\!\!\!\!\!\!(\gamma-1)(\vartheta p+1-\vartheta)}_{:=\mathrm{G}_5} \Big|\big(\frac{r^2\rho_0}{\eta^2}\big)^\frac{1}{p}U\Big|_p^p\\
&\quad +C(p)\big|(\zeta r^2\rho_0)^\frac{1}{p}V\big|_p^{p}+\frac{A\gamma}{4\mu}\big|(\zeta \eta^2\eta_r\varrho^{\gamma})^\frac{1}{p}  V\big|_p^{p}.
\end{aligned}
\end{equation}

\smallskip
\textbf{2.} Now we estimate $\mathrm{G}_5$. In order to apply Lemma \ref{lemma-near depth} to $\mathrm{G}_5$, for each $\gamma\in (1,2]$ and $p$ satisfying \eqref{p-par}, we need to choose suitable $\vartheta$ depending only on $(\gamma,p)$ such that
\begin{equation*}
1\leq \frac{2}{(\gamma-1)(\vartheta p+1-\vartheta)}\leq 2,
\end{equation*}
which is equivalent to showing that there exists $\vartheta=\vartheta(\gamma,p)$ satisfying
\begin{equation}\label{equiv1}
p\in [p_*(\vartheta;\gamma),p^*(\vartheta;\gamma)] \quad\, \text{with } p_*(\vartheta;\gamma):=1+\frac{2-\gamma}{\vartheta(\gamma-1)} \text{ and } p^*(\vartheta;\gamma):=1+\frac{3-\gamma}{\vartheta(\gamma-1)}.
\end{equation}
A direct calculation shows that, if $\gamma=2$, $p_*(\vartheta;2)\equiv 1$ and $p^*(\vartheta;2)=1+\frac{1}{\vartheta}$; for each $\gamma\in (1,2)$, $p_*(\vartheta;\gamma)$ and $p^*(\vartheta;\gamma)$ are both strictly decreasing with respect to $\vartheta$, and
\begin{equation*}
\lim_{\vartheta\to 0}p^*(\vartheta;\gamma)=\lim_{\vartheta\to 0}p_*(\vartheta;\gamma) =\infty,\quad  \lim_{\vartheta\to 1}p_*(\vartheta;\gamma)=\frac{1}{\gamma-1},\quad \lim_{\vartheta\to 1}p^*(\vartheta;\gamma)=\frac{2}{\gamma-1}.
\end{equation*}
Since $\gamma\in (1,2]$ and $p$ satisfies \eqref{p-par}, for each $(\gamma,p)$, we can fix 
\begin{equation*}
\vartheta=\vartheta_0=\vartheta_0(\gamma,p)\in (0,1)
\end{equation*}
such that \eqref{equiv1} holds, and then we can apply Lemma \ref{lemma-near depth} to $\mathrm{G}_5$ to obtain
\begin{equation}\label{g4}
\mathrm{G}_5\leq C(p,T) \qquad \text{for all $t\in[0,T]$}.
\end{equation}

\smallskip
\textbf{3.} Substituting \eqref{g4} into \eqref{dt-v}, integrating in time, and applying the Gr\"onwall inequality, together with Lemma \ref{lemma-u Lp-n3}, yield that, for every $p$ satisfying \eqref{p-par},
\begin{equation}\label{vabove3}
\sup_{0\leq t\leq T}\big|(\zeta r^2\rho_0)^\frac{1}{p}V(t)\big|_p^p
\leq C(p,T)\Big(1+\big|(\zeta r^2\rho_0)^\frac{1}{p}v_0\big|_p^p\Big)\leq C(p,T).
\end{equation}
Finally, it follows from Lemma \ref{near-BD} and the interpolation on the finite measure space with measure $\zeta r^2\rho_0\,\mathrm{d}r$ that \eqref{vabove3} holds for all $p\in[2,\infty)$.
\end{proof}

Now we can prove Lemma \ref{lp-uv} for the case when $n=3$ and $\gamma\in(1,2]$.
\begin{proof}[Proof of Lemma \ref{lp-uv}]
Combining Lemmas \ref{lemma-u Lp-n3}--\ref{lemma-v-n=3}, we obtain the desired estimates of Lemma \ref{lp-uv} when $n=3$ and $\gamma\in (1,2]$.
\end{proof}

\subsubsection{Case when $n=3$ and $\gamma\in(2,3)$}\label{subsub-3}
We first consider the $L^p$-energy estimates for $U$.
\begin{Lemma}\label{lemma-u Lp}
Let $n=3$ and $\gamma\in (2,3)$. Then, for any $p\in(2,\infty)$ and $\varepsilon\in (0,1)$, there exists a constant $C(\varepsilon,p,T)>0$ such that, for all $t\in [0,T]$,
\begin{equation}\label{eq,u-lp}
\begin{aligned}
&\frac{\mathrm{d}}{\mathrm{d}t}\big|(r^2\rho_0)^\frac{1}{p}U\big|_p^p+\mu \Big|(r^2\rho_0)^\frac{1}{2}|U|^\frac{p-2}{2}\big(D_\eta U,\frac{U}{\eta}\big)\Big|_2^2 \\
&\leq C(p)\big|(r^2\rho_0)^\frac{1}{p}U\big|_p^p+\varepsilon\big|(\zeta\eta^2\eta_r\varrho^\gamma)^\frac{1}{p}V\big|_p^p+C(\varepsilon,p,T).
\end{aligned}
\end{equation}
\end{Lemma}
\begin{proof}
We divide the proof into three steps. The endpoint $p=2$ needed in Lemma \ref{lp-uv} will be supplied directly by Lemmas \ref{lemma-basic energy} and \ref{near-BD}. Recall that $\zeta_{\frac{5}{8}}\equiv1$ on $\operatorname{supp}\zeta$.

\smallskip
\textbf{1.} Let $n=3$, $\gamma\in (2,3)$, and $p\in (2,\infty)$. First, repeating the same calculations as in Step 1 of the proof for Lemma \ref{lemma-u Lp-n3} gives
\begin{equation}\label{dt-G2}
\frac{1}{p}\frac{\mathrm{d}}{\mathrm{d}t}\big|(r^2\rho_0)^\frac{1}{p}U\big|_p^p+\mu\Big|(r^2\rho_0)^\frac{1}{2}|U|^\frac{p-2}{2}\big(D_\eta U,\frac{U}{\eta}\big)\Big|_2^2\leq C(p)\mathrm{G}_{4},  
\end{equation}
where $\mathrm{G}_4$ is defined in \eqref{dt-G2,n3} of the proof for Lemma \ref{lemma-u Lp-n3}.

\smallskip
\textbf{2.} Now, we estimate $\mathrm{G}_{4}$. Repeating  calculations \eqref{G2:n=3} and $\eqref{g2,1-xiao}_1$--$\eqref{g2,1-xiao}_2$ in Step 2 of the proof for Lemma \ref{lemma-u Lp-n3} and taking $\vartheta=1$ in $\eqref{g2,1-xiao}_2$, we have
\begin{equation}\label{g2,1-da0}
\mathrm{G}_{4}\leq \mathrm{G}_{4,1}+\big|(r^2\rho_0)^\frac{1}{p}U\big|_p^p+C(p,T),
\end{equation}
and, for any $\varepsilon\in(0,1)$,
\begin{equation}\label{g2,1-da}
\begin{aligned}
\mathrm{G}_{4,1}&\leq \Big(\int_0^\frac{1}{2} \eta^{p}\eta_r\varrho^{p\gamma-p+1}\,\mathrm{d}r\Big)^\frac{2}{p}\Big|\big(\frac{r^2\rho_0}{\eta^2}\big)^\frac{1}{p}U\Big|_p^{p-2}\\
&\leq C(\varepsilon,p)\int_0^\frac{1}{2} \zeta\eta^{p}\eta_r\varrho^{p\gamma-p+1}\,\mathrm{d}r+  \varepsilon\Big|\big(\frac{r^2\rho_0}{\eta^2}\big)^\frac{1}{p}U\Big|_p^{p}.
\end{aligned}
\end{equation}
Then we choose 
\begin{equation*}
(a_0,b_0)=(p,p\gamma-p+1),\qquad
(a_1,b_1)=(\frac{p}{p-1}a_0+\frac{p-2}{p-1},\frac{p}{p-1}b_0-\frac{\gamma}{p-1}), 
\end{equation*}
and denote
\begin{equation*}
\mathrm{I}_{(k,\ell)}:=\int_0^1\zeta \eta^{k}\eta_r\varrho^{\ell}\,\mathrm{d}r.
\end{equation*}
Since $2b_0\geq a_0+1$, it follows from \eqref{eq:eta}, the fact that $\rho_0^\beta\sim 1-r$, Lemma \ref{lemma-low jacobi 1}, integration by parts, and the Young inequality that, for all $\sigma_0\in(0,1)$,
\begin{equation}\label{a0}
\begin{aligned}
\mathrm{I}_{(a_0,b_0)}&= -\frac{1}{a_0+1}\int_0^1 \zeta_r \eta^{a_0+1} \varrho^{b_0}\,\mathrm{d}r -\frac{b_0p}{\gamma(a_0+1)}\int_0^1 \zeta\eta^{a_0+1} \varrho^{b_0-\frac{\gamma}{p}} (\varrho^\frac{\gamma}{p})_r\,\mathrm{d}r\\
&\leq C(a_0)\int_\frac{1}{2}^\frac{5}{8} \frac{r^{2b_0}\rho_0^{b_0}}{\eta^{2b_0-a_0-1}\eta_r^{b_0}}\,\mathrm{d}r +\sigma_0 \big|(\zeta\eta^2\eta_r)^\frac{1}{p}D_\eta(\varrho^\frac{\gamma}{p})\big|_p^p+C(p,a_0,b_0) \sigma_0^{-\frac{1}{p-1}}\mathrm{I}_{(a_1,b_1)} \\
&\leq  C(a_0,b_0,T)+\sigma_0\big|(\zeta\eta^2\eta_r)^\frac{1}{p}D_\eta(\varrho^\frac{\gamma}{p})\big|_p^p+ C(p,a_0,b_0) \sigma_0^{-\frac{1}{p-1}}\mathrm{I}_{(a_1,b_1)} .
\end{aligned}    
\end{equation}

Next, define two sequences $(\{a_j\}_{j\in\NN},\{b_j\}_{j\in\NN})$ as follows:
\begin{equation}\label{ditui}
\begin{aligned}
a_{j+1}&=\frac{p}{p-1}a_j+\frac{p-2}{p-1} \qquad \text{with } a_0=p,\\
b_{j+1}&=\frac{p}{p-1}b_j-\frac{\gamma}{p-1} \qquad  \text{with } b_0=p\gamma-p+1.
\end{aligned}
\end{equation}
Clearly, we can solve for $(a_j,b_j)$ from \eqref{ditui} and obtain, for $j\in\mathbb{N}$,
\begin{equation}
a_{j}=2(p-1)\big(\frac{p}{p-1}\big)^j-(p-2),\quad\,\, b_{j}=(\gamma-1)(p-1)\big(\frac{p}{p-1}\big)^j+\gamma,
\end{equation}
and check that $2b_j\geq a_j+1$ for $j\in \NN$, $\gamma\in (2,3)$, and $p\in(2,\infty)$. Following the same argument as in \eqref{a0} thus implies that, for all $\sigma_j\in (0,1)$ and $j\in\NN$,
\begin{equation*}
\mathrm{I}_{(a_j,b_j)} \leq  C(a_j,b_j,T)+\sigma_j\big|(\zeta\eta^2\eta_r)^\frac{1}{p}D_\eta(\varrho^\frac{\gamma}{p})\big|_p^p+ C(p,a_j,b_j) \sigma_j^{-\frac{1}{p-1}}\mathrm{I}_{(a_{j+1},b_{j+1})},
\end{equation*}
which, along with \eqref{a0}, yields   
\begin{equation} \label{aj}
\begin{aligned}
\mathrm{I}_{(a_0,b_0)}&\leq C(a_0,b_0,T)+\sum_{k=1}^j C(a_k,b_k,T)\prod_{\ell=0}^{k-1}C(p,a_\ell,b_\ell) \sigma_\ell^{-\frac{1}{p-1}} \\
&\quad +\Big(\sigma_0+\sum_{k=1}^j \sigma_k\prod_{\ell=0}^{k-1}C(p,a_\ell,b_\ell) \sigma_\ell^{-\frac{1}{p-1}}\Big)\big|(\zeta\eta^2\eta_r)^\frac{1}{p}D_\eta(\varrho^\frac{\gamma}{p})\big|_p^p  \\
&\quad+\Big(\prod_{\ell=0}^{j}C(p,a_\ell,b_\ell) \sigma_\ell^{-\frac{1}{p-1}}\Big)\mathrm{I}_{(a_{j+1},b_{j+1})} \\
&\leq  C(a_0,b_0,T)+\sum_{k=1}^j C(a_k,b_k,T)\prod_{\ell=0}^{k-1}C(p,a_\ell,b_\ell) \sigma_\ell^{-\frac{1}{p-1}}\\
&\quad +\Big(\sigma_0+\sum_{k=1}^j \sigma_k\prod_{\ell=0}^{k-1}C(p,a_\ell,b_\ell) \sigma_\ell^{-\frac{1}{p-1}}\Big)\big|(\zeta\eta^2\eta_r)^\frac{1}{p}D_\eta(\varrho^\frac{\gamma}{p})\big|_p^p\\
&\quad +\Big(\prod_{\ell=0}^{j}C(p,a_\ell,b_\ell) \sigma_\ell^{-\frac{1}{p-1}}\Big)\big|\zeta_{\frac{5}{8}}\eta\varrho\big|_\infty^{b_{j+1}-1}\big|\zeta\eta^{a_{j+1}-b_{j+1}+1} \eta_r \varrho\big|_1.
\end{aligned}
\end{equation}
Here, we need to check that Lemma \ref{lemma-near depth} is applicable to the last term of the right-hand side of the above. To this end, for each $\gamma\in (2,3)$ and $p\in(2,\infty)$, we need to show that
\begin{equation*}
\exists\, j\in \mathbb{N}, \text{ which depends only on $(p,\gamma)$, such that $a_{j+1}-b_{j+1}+1\in [0,2]$}.
\end{equation*}
After a direct calculation, this is equivalent to showing that, for each $\gamma\in (2,3)$ and $p\in(2,\infty)$,
\begin{equation}\label{baohan}
\gamma\in (2,3)\subset \bigcup_{j\in \mathbb{N}}\big[f_*(j;p),f^*(j;p)\big],
\end{equation}
where
\begin{equation*}
f_*(j;p):=3-\frac{p+2}{(p-1)(\frac{p}{p-1})^{j+1}+1},\qquad
f^*(j;p):=3-\frac{p}{(p-1)(\frac{p}{p-1})^{j+1}+1}. 
\end{equation*}
Clearly, both $f_*(j;p)$ and $f^*(j;p)$ are increasing with respect to $j$ as $j\to\infty$, and 
\begin{equation*}
f_*(0;p)<2<f^*(0;p),\qquad  \lim_{j\to\infty}f_*(j;p)=\lim_{j\to\infty}f^*(j;p)=3.
\end{equation*}
Consequently, to obtain \eqref{baohan}, it suffices to show that, for any $j\in \mathbb{N}$,
\begin{equation*}
[f_*(j;p),f^*(j;p)]\cap [f_*(j+1;p),f^*(j+1;p)]\neq\varnothing,
\end{equation*}
or, equivalently, to show that $f_*(j+1;p)\leq f^*(j;p)$ for any $j\in\mathbb{N}$, {\it i.e.},
\begin{equation}\label{baohan2}
3-\frac{p+2}{(p-1)(\frac{p}{p-1})^{j+2}+1}\leq 3-\frac{p}{(p-1)(\frac{p}{p-1})^{j+1}+1} \qquad \text{for any $j\in\mathbb{N}$}.   
\end{equation}
Indeed, for $\gamma\in (2,3)$ and $p\in (2,\infty)$, a direct calculation gives
\begin{equation*}
\begin{aligned}
\eqref{baohan2}&\iff (p+2)(p-1)\big(\frac{p}{p-1}\big)^{j+1}+p+2 \geq p^2\big(\frac{p}{p-1}\big)^{j+1} +p\\
&\iff  (p-2) \big(\frac{p}{p-1}\big)^{j+1}+2\geq 0,  
\end{aligned}
\end{equation*}
which thus implies claim \eqref{baohan2}. Therefore, for each $\gamma\in (2,3)$ and $p\in(2,\infty)$, we can set $j=j_0$ depending only on $(p,\gamma)$ in \eqref{aj} such that $a_{j_0+1}-b_{j_0+1}+1\in [0,2]$, and then obtain from Lemma \ref{lemma-near depth} that
\begin{equation}\label{aj'}
\begin{aligned}
\mathrm{I}_{(a_0,b_0)}&\leq \Big(C(a_0,b_0,T)+\sum_{k=1}^{j_0} C(a_k,b_k,T)\prod_{\ell=0}^{k-1}C(p,a_\ell,b_\ell) \sigma_\ell^{-\frac{1}{p-1}} \Big)\\
&\quad +\Big(\sigma_0+\sum_{k=1}^{j_0} \sigma_k\prod_{\ell=0}^{k-1}C(p,a_\ell,b_\ell) \sigma_\ell^{-\frac{1}{p-1}}\Big)\big|(\zeta\eta^2\eta_r)^\frac{1}{p}D_\eta(\varrho^\frac{\gamma}{p})\big|_p^p\\
&\quad + \Big(\prod_{\ell=0}^{j_0} C(p,a_\ell,b_\ell) \sigma_\ell^{-\frac{1}{p-1}}\Big)C(a_{j_0+1},b_{j_0+1},T).  
\end{aligned}    
\end{equation}

Now, let $\tilde\varepsilon\in (0,1)$, and set 
\begin{equation*}
\sigma_0=\tilde\varepsilon,\qquad\,\,\,\, \sigma_k=\frac{\tilde\varepsilon}{j_0}\prod_{\ell=0}^{k-1}\frac{\sigma_\ell^\frac{1}{p-1}}{C(p,a_\ell,b_\ell)} \quad\, \text{for } 1\leq k\leq j_0 \quad \text{when $j_0\geq 1$}.
\end{equation*}
Then we obtain from \eqref{eq:eta}, \eqref{v-expression}, \eqref{aj'}, $\gamma\in(2,3)$, and Lemma \ref{lemma-near depth} that
\begin{equation}\label{aj''}
\begin{aligned}
\mathrm{I}_{(a_0,b_0)}&\leq C(\tilde\varepsilon,p,T)  +\Big(\tilde\varepsilon+\sum_{k=1}^{j_0}\frac{\tilde\varepsilon}{j_0}\Big)\big|(\zeta\eta^2\eta_r)^\frac{1}{p}D_\eta(\varrho^\frac{\gamma}{p})\big|_p^p\\[-2pt]
&\leq  C(\tilde\varepsilon,p,T)  +2\tilde\varepsilon \big|(\zeta\eta^2\eta_r)^\frac{1}{p}D_\eta(\varrho^\frac{\gamma}{p})\big|_p^p\\
&\leq C(\tilde\varepsilon,p,T)  +C(p)\tilde\varepsilon \big|(\zeta\eta^2\eta_r)^\frac{1}{p}\varrho^\frac{\gamma}{p}(U,V)\big|_p^p\\
&\leq C(\tilde\varepsilon,p,T)  +C(p)\tilde\varepsilon  \big|\zeta_\frac{5}{8}\eta^\frac{2}{\gamma-1}\varrho\big|_\infty^{\gamma-1} \big|(\eta_r\varrho)^\frac{1}{p}U\big|_p^p  +C(p)\tilde\varepsilon \big|(\zeta\eta^2\eta_r\varrho^\gamma)^\frac{1}{p}V\big|_p^p\\
&\leq C(\tilde\varepsilon,p,T)  +C(p,T)\tilde\varepsilon \Big|\big(\frac{r^2\rho_0}{\eta^2}\big)^\frac{1}{p}U\Big|_p^p+C(p)\tilde\varepsilon\big|(\zeta\eta^2\eta_r\varrho^\gamma)^\frac{1}{p}V\big|_p^p.    
\end{aligned}
\end{equation}
When $j_0=0$, the sum is empty and only $\sigma_0=\tilde{\varepsilon}$ is needed. Substituting \eqref{aj''} into \eqref{g2,1-da} yields that, for all $\varepsilon,\tilde\varepsilon \in (0,1)$,
\begin{equation}\label{g2,1-da'''}
\mathrm{G}_{4,1}\leq \big(C(\varepsilon,p,T)\tilde\varepsilon +\varepsilon\big)\Big|\big(\frac{r^2\rho_0}{\eta^2}\big)^\frac{1}{p}U\Big|_p^p + C(\varepsilon,p)\tilde\varepsilon \big|(\zeta\eta^2\eta_r\varrho^\gamma)^\frac{1}{p}V\big|_p^p+C(\tilde\varepsilon,\varepsilon,p,T).           
\end{equation}
Then, for any $\varepsilon\in (0,1)$, setting $\tilde\varepsilon$ such that
\begin{equation*}
0<\tilde\varepsilon<\min\big\{\varepsilon,\frac{\varepsilon}{C(\varepsilon,p)},\frac{\varepsilon}{C(\varepsilon,p,T)}\big\}<1,
\end{equation*}
we thus obtain from the above and \eqref{g2,1-da'''} that, for all $\varepsilon\in (0,1)$,
\begin{equation}\label{g2,1-da'}
\mathrm{G}_{4,1} \leq 2\varepsilon\Big|\big(\frac{r^2\rho_0}{\eta^2}\big)^\frac{1}{p}U\Big|_p^{p}+\varepsilon \big|(\zeta\eta^2\eta_r\varrho^\gamma)^\frac{1}{p}V\big|_p^p+C(\varepsilon,p,T).
\end{equation}

Collecting \eqref{g2,1-da0} and \eqref{g2,1-da'} leads to the estimate of $\mathrm{G}_4$:
\begin{equation}\label{G2:n=3-2}
\mathrm{G}_4\leq \big|(r^2\rho_0)^\frac{1}{p}U\big|_p^p+2\varepsilon\Big|\big(\frac{r^2\rho_0}{\eta^2}\big)^\frac{1}{p}U\Big|_p^{p}+\varepsilon \big|(\zeta\eta^2\eta_r\varrho^\gamma)^\frac{1}{p}V\big|_p^p+C(\varepsilon,p,T).
\end{equation}

\smallskip
\textbf{3.} Collecting \eqref{dt-G2} and \eqref{G2:n=3-2}, and then setting $\varepsilon$ sufficiently small, we eventually obtain the desired estimates.
\end{proof}

Based on Lemma \ref{lemma-u Lp}, we obtain the following interior $L^p$-energy estimates for $V$.

\begin{Lemma}\label{lemma-v-lp}
Let $n=3$ and $\gamma\in (2,3)$. Then, for any $p\in[2,\infty)$, there exists a constant $C(p,T)>0$ such that, for all $t\in[0,T]$,
\begin{equation}\label{eq,v-lp}
\frac{\mathrm{d}}{\mathrm{d}t}\big|(\zeta r^2\rho_0)^\frac{1}{p}V\big|_p^p+\frac{pA\gamma}{4\mu}\big|(\zeta \eta^2\eta_r\varrho^{\gamma})^\frac{1}{p}  V\big|_p^p\leq C(p,T)\Big|\big(\frac{r^2\rho_0}{\eta^2}\big)^\frac{1}{p}U\Big|_p^p.
\end{equation}
\end{Lemma}
\begin{proof}
Let $n=3$, $\gamma\in (2,3)$, and $p\in[2,\infty)$. Multiplying \eqref{eq:v} by $\zeta r^2\rho_0 |V|^{p-2}V$ leads to
\begin{equation*}
\frac{1}{p}(\zeta r^2\rho_0|V|^p)_t+\frac{A\gamma}{2\mu}\zeta \eta^2\eta_r \varrho^{\gamma}|V|^p=\frac{A\gamma}{2\mu}\zeta \eta^2 \eta_r\varrho^{\gamma}|V|^{p-2}VU.
\end{equation*}
Then, integrating the above over $I$, we obtain from Lemma \ref{lemma-near depth} and the Young inequality that
\begin{equation*}
\begin{aligned}
&\,\frac{1}{p}\frac{\mathrm{d}}{\mathrm{d}t}\big|(\zeta r^2\rho_0)^\frac{1}{p}V\big|_p^p+\frac{A\gamma}{2\mu}\big|(\zeta \eta^2\eta_r\varrho^{\gamma})^\frac{1}{p}  V\big|_p^p\leq \frac{A\gamma}{2\mu} \int_0^1 \zeta \eta^2 \eta_r\varrho^{\gamma}|V|^{p-1}|U|\,\mathrm{d}r\\
&\leq C(p)\big|(\zeta \eta^2\eta_r\varrho^{\gamma})^\frac{1}{p}  U\big|_p^p+\frac{A\gamma}{4\mu}\big|(\zeta \eta^2\eta_r\varrho^{\gamma})^\frac{1}{p}  V\big|_p^p\\
&\leq C(p)\big|\zeta_{\frac{5}{8}}\eta^\frac{2}{\gamma-1}\varrho\big|_\infty^{\gamma-1} \Big|\big(\frac{r^2\rho_0}{\eta^2}\big)^\frac{1}{p}U\Big|_p^p+\frac{A\gamma}{4\mu}\big|(\zeta \eta^2\eta_r\varrho^{\gamma})^\frac{1}{p}  V\big|_p^p\\
&\leq C(p,T)\Big|\big(\frac{r^2\rho_0}{\eta^2}\big)^\frac{1}{p}U\Big|_p^p+\frac{A\gamma}{4\mu}\big|(\zeta \eta^2\eta_r\varrho^{\gamma})^\frac{1}{p}  V\big|_p^p,
\end{aligned}
\end{equation*}
which yields the desired result.
\end{proof}

Now, we prove Lemma \ref{lp-uv} for the case when $n=3$ and $\gamma\in (2,3)$.
\begin{proof}[Proof of Lemma \ref{lp-uv}]
Let $n=3$, $\gamma\in (2,3)$, and $p\in[2,\infty)$. Multiplying \eqref{eq,v-lp} in Lemma \ref{lemma-v-lp} by $\frac{8\mu \varepsilon}{pA\gamma}$, combined with \eqref{eq,u-lp} in Lemma \ref{lemma-u Lp}, gives that, for all $\varepsilon\in (0,1)$,
\begin{equation}\label{jiehe}
\begin{aligned}
&\frac{\mathrm{d}}{\mathrm{d}t}\Big(\big|(r^2\rho_0)^\frac{1}{p}U\big|_p^p+\frac{8\mu \varepsilon}{pA\gamma}\big|(\zeta r^2\rho_0)^\frac{1}{p}V\big|_p^p\Big)+\mu \Big|(r^2\rho_0)^\frac{1}{2}|U|^\frac{p-2}{2}\big(D_\eta U,\frac{U}{\eta}\big)\Big|_2^2 \\
&\leq C(p)\big|(r^2\rho_0)^\frac{1}{p}U\big|_p^p+C(p,T)\varepsilon\Big|\big(\frac{r^2\rho_0}{\eta^2}\big)^\frac{1}{p}U\Big|_p^p+C(\varepsilon,p,T).
\end{aligned}
\end{equation}

Thus, we can choose $\varepsilon$ in \eqref{jiehe} sufficiently small such that 
\begin{equation*}
0<\varepsilon<\min\big\{1,\frac{\mu}{2C(p,T)}\big\},
\end{equation*}
and obtain from the Gr\"onwall inequality that \eqref{zong-jiehe} in Lemma \ref{lp-uv} holds.

This completes the proof of Lemma \ref{lp-uv} for the case when $n=3$ and $\gamma\in(2,3)$.
\end{proof}

\subsection{Global uniform upper bound of the density}

With the help of Lemmas \ref{lemma-near depth}--\ref{lp-uv}, we are now ready to establish the global uniform upper bound for $\varrho$ in $[0,T]\times \bar I$.

\begin{Lemma}\label{lemma-bound depth}
Let $\gamma\in(1,\infty)$ if $n=2$, and $\gamma\in (1,3)$ if $n=3$. Then there exists a constant $C(T)>0$ such that
\begin{equation*}
|\varrho(t)|_\infty \leq C(T) \qquad  \text{for all } t\in [0,T].
\end{equation*}
\end{Lemma}
\begin{proof}
Recall that $\operatorname{supp}\zeta_r\subset[\frac{1}{2},\frac{5}{8}]$. It follows from \eqref{v-expression}, the compact-subinterval bounds in Lemma \ref{lemma-low jacobi 1}, Lemmas \ref{lemma-near depth}--\ref{lp-uv}, Lemma \ref{sobolev-embedding}, and the H\"older and Young inequalities that
\begin{equation}\label{shangjie}
\begin{aligned}
|\zeta \varrho|_\infty&= |\zeta (\varrho^\frac{1}{5})^5|_\infty\leq C_0\big|(\zeta_r \varrho,\zeta \varrho^\frac{4}{5}(\varrho^\frac{1}{5})_r)\big|_1\\
&\leq C_0\Big(|\zeta_r|_\infty\|\eta_r^{-1}\|_{L^\infty(\frac{1}{2},\frac{5}{8})}\big|\zeta_{\frac{5}{8}}\eta_r\varrho\big|_1+\big|\zeta\eta^{-\frac{m}{4}} \eta_r \varrho \big|_1^\frac{4}{5} \big|(\zeta r^m\rho_0)^\frac{1}{5}(V,U)\big|_5\Big)\\
&\leq C(T)+\underline{\big|\zeta\eta^{-\frac{m}{4}} \eta_r \varrho \big|_1}_{:=\mathrm{G}_*}.
\end{aligned}    
\end{equation}

To estimate $\mathrm{G}_*$, first consider $n=2$ ($m=1$). Since
\begin{equation*}
-\frac{m}{4}=-\frac{1}{4}\in (-1,0),
\end{equation*}
we can apply Lemma \ref{lemma-near depth} to obtain 
\begin{equation}\label{456}
\mathrm{G}_*\leq C(T) \qquad \text{for all $t\in[0,T]$}.
\end{equation}
If $n=3$ ($m=2$), since $\zeta_r$ is supported in $[\frac{1}{2},\frac{5}{8}]$, we obtain from integration by parts, \eqref{eq:eta}, \eqref{v-expression}, Lemmas \ref{lemma-low jacobi 1} and \ref{lemma-near depth}--\ref{lp-uv}, and the H\"older inequality that, for all $t\in[0,T]$,
\begin{equation}\label{458}
\begin{aligned}
\mathrm{G}_{*}&=\int_0^1 \zeta \eta^{-\frac{1}{2}}\eta_r\varrho\,\mathrm{d}r=-2\Big(\int_0^1 \zeta_r \eta^{\frac{1}{2}}\varrho\,\mathrm{d}r+\frac{1}{2\mu}\int_0^1 \zeta \eta^{\frac{1}{2}}\eta_r\varrho (V-U)\,\mathrm{d}r\Big)\\
&\leq C_0\int_\frac{1}{2}^\frac{5}{8} \frac{r^2\rho_0}{\eta^{\frac{3}{2}}\eta_r} \mathrm{d}r  + C_0\Big( \int_0^1 \zeta  \eta^{\frac{1}{8}}\eta_r\varrho\, \mathrm{d}r\Big)^\frac{4}{5}\big|(\zeta r^2\rho_0)^\frac{1}{5}(V,U)\big|_5\leq C(T).
\end{aligned}    
\end{equation}

Thus, collecting \eqref{shangjie}--\eqref{458} gives that, for all $t\in[0,T]$,
\begin{equation*} 
|\zeta\varrho|_\infty\leq C(T)+\mathrm{G}_*\leq C(T),
\end{equation*}
which, together with the bound of $\eta^m\varrho$ in Lemma \ref{lemma-far depth} and the fixed-exterior lower bound of $\eta$ in Lemma \ref{lemma-low jacobi 1}, yields, for all $t\in[0,T]$,
\begin{equation*}
|\varrho|_\infty\leq |\zeta\varrho|_\infty+|\zeta^\sharp\varrho|_\infty\leq C(T)+|\zeta^\sharp\eta^{-m}|_\infty|\eta^m\varrho|_\infty\leq C(T),
\end{equation*}
where the cut-off function $\zeta^\sharp$ is defined in \S \ref{othernotation}.
\end{proof}

\subsection{Global uniform lower bounds of $(\eta_r,\frac{\eta}{r})$}

Now we can establish the global uniform lower bounds of $(\eta_r,\frac{\eta}{r})$ in $[0,T]\times \bar I$.

\begin{Lemma}\label{lemma-lower bound jacobi}
There exists a constant $C(T)>1$ such that
\begin{equation*}
\eta_r(t,r)\geq C(T)^{-1},\qquad
\frac{\eta(t,r)}{r}\geq C(T)^{-1}
\qquad \text{for $(t,r)\in[0,T]\times\bar I$}.
\end{equation*}
\end{Lemma}
\begin{proof}
We divide the proof into three steps.

\smallskip
\textbf{1. Lower bound of $\frac{\eta}{r}$} For $r\in [0,1]$, it follows from \eqref{eq:eta} and Lemma \ref{lemma-bound depth} that
\begin{equation*}
\eta^{m+1}(t,r)
=(m+1)\int_0^r\frac{\hat r^m\rho_0(\hat r )}{\varrho(t,\hat r)}\,\mathrm{d}\hat r
\geq C(T)^{-1}\int_0^r \hat r^m\rho_0(\hat r)\,\mathrm{d}\hat r
\geq C(T)^{-1}r^{m+1},
\end{equation*}
which implies that, for $(t,r)\in [0,T]\times \bar I$,
\begin{equation}\label{eq:eta-over-r-lower}
\eta(t,r)\geq C(T)^{-1}r.
\end{equation}
 
\textbf{2. The interior lower bound of $\eta_r$.} Recall \eqref{lem42-iden} from Lemma \ref{lemma-far depth}:
\begin{equation*}
\big(2\mu(\eta^m\varrho)_r+r^m\rho_0U-2\mu m\eta^{m-1}\eta_r\varrho\big)_t
=Am\eta^{m-1}\eta_r\varrho^\gamma-A(\eta^m\varrho^\gamma)_r.
\end{equation*}
Integrate over $[\varepsilon,r]$ and let $\varepsilon\to0$. Using \eqref{eq:eta-over-r-lower} and the fact that $\eta^m\varrho|_{r=0}=\eta^m\varrho^\gamma|_{r=0}=0$, we obtain
\begin{equation}\label{eq:center-integral-identity}
\begin{aligned}
&\frac{\mathrm{d}}{\mathrm{d}t}\Big(2\mu(\eta^m\varrho)(t,r)
+\int_0^r \hat r^m\rho_0 U \,\mathrm{d}\hat r
-2\mu m\int_0^r\eta^{m-1} \eta_r \varrho \,\mathrm{d}\hat r\Big)\\
&=Am\int_0^r\eta^{m-1} \eta_r \varrho^\gamma \,\mathrm{d}\hat r
-\underline{A(\eta^m \varrho^\gamma)(t,r)}_{\geq 0}.
\end{aligned}
\end{equation}
Thus, integrating \eqref{eq:center-integral-identity} over $[0,t]$, along with \eqref{eq:eta}, \eqref{eq:eta-over-r-lower}, and Lemmas \ref{lp-uv} and \ref{lemma-bound depth}, implies
\begin{equation*} 
\begin{aligned}
2\mu(\eta^m \varrho)(t,r)
&\leq2\mu r^m\rho_0(r)
+2\sup_{t\in[0,T]} \int_0^r \hat r^m\rho_0|U|\,\mathrm{d}\hat r +4\mu m\sup_{t\in[0,T]}\int_0^r\eta^{m-1} \eta_r\varrho \,\mathrm{d}\hat r\\
&\quad +Am\int_0^t\int_0^r\eta^{m-1} \eta_r\varrho^\gamma \,\mathrm{d}\hat r\,\mathrm{d}\tau,\\
&\leq C_0r^m+2\big|(r^m\rho_0)^\frac{1}{n}U\big|_n\Big(\int_0^r \hat r^m\rho_0 \,\mathrm{d}\hat r\Big)^{\frac{m}{n}}\\
&\quad +C(T)\sup_{t\in[0,T]}\Big(\int_0^r \frac{\hat r^m\rho_0(1+\varrho^{\gamma-1})}{\eta}\,\mathrm{d}\hat r \Big)\leq C(T)r^m.
\end{aligned}
\end{equation*}
Consequently, for $(t,r)\in [0,T]\times[0,\frac{3}{4}]$,
\begin{equation*}
\eta_r(t,r)=\frac{r^m\rho_0}{\eta^m\varrho}
\geq C(T)^{-1}\rho_0(r)\geq C(T)^{-1}.
\end{equation*}

\textbf{3. The exterior lower bound of $\eta_r$.}
For $r\in [\frac{1}{2},1]$, recall from \eqref{lem42-den} that
\begin{equation}\label{eq:repair-tail-identity}
\frac{\mathrm{d}}{\mathrm{d}t}\Big(2\mu\eta^m(t,r)\varrho(t,r)
+2\mu m\int_r^1\frac{\hat r^m\rho_0 }{\eta}\,\mathrm{d}\hat r\Big)\leq \frac{\mathrm{d}}{\mathrm{d}t}\int_r^1\hat r^m\rho_0 U \,\mathrm{d}\hat r.
\end{equation}
Integrating the above over $[0,t]$, we obtain
\begin{equation*} 
\begin{aligned}
2\mu(\eta^m \varrho)(t,r)
&\leq2\mu r^m\rho_0(r)
+2\sup_{t\in[0,T]} \int_r^1 \hat r^m\rho_0|U|\,\mathrm{d}\hat r +4\mu m\sup_{t\in[0,T]}\int_r^1\eta^{m-1} \eta_r\varrho \,\mathrm{d}\hat r\\
&\leq C_0\rho_0+\widetilde{G}_1+\widetilde{G}_2,
\end{aligned}
\end{equation*}
where the estimates of $\widetilde{G}_1$ and $\widetilde{G}_2$ can be handled by using \eqref{eq:eta}, Lemmas \ref{lemma-low jacobi 1}--\ref{lp-uv}, $\rho_0^\beta\sim1-r$, and the H\"older inequality: 
\begin{equation}\label{eq:repair-tail-errors}
\widetilde{G}_1 \leq 2\big|(r^m\rho_0)^\frac{\beta}{1+2\beta} U\big|_\frac{1+2\beta}{\beta}
\Big(\int_r^1\hat r^m\rho_0 \,\mathrm{d}\hat r\Big)^{\frac{1+\beta}{1+2\beta}} \leq C(T)\rho_0^{1+\frac{\beta^2}{1+2\beta}} ,\qquad \widetilde{G}_2 \leq  C_0\rho_0^{1+\beta}.
\end{equation}
Thus, it follows from \eqref{eq:repair-tail-errors} that
\begin{equation*}
(\eta^m \varrho)(t,r)
\leq C(T)\rho_0(r)\implies \eta_r(t,r)=\frac{r^m\rho_0}{\eta^m\varrho}
\geq C(T)^{-1}r^m\geq C(T)^{-1}
\end{equation*}
for $\frac{1}{2}\leq r<1$.  At $r=1$, equation $\eta_t=U$, together with $U_r(t,1)=0$ in \eqref{N111}, shows that $\eta_r(t,1)=1$.  This completes the proof.
\end{proof}

\section{Global-In-Time Uniform Estimates of the Effective Velocity}\label{Section-effectivevelocity}

The purpose of this section is to establish the global-in-time uniform  estimates of the effective velocity.  

\subsection{Boundedness of the effective velocity away from the origin}
\begin{Lemma}\label{lemma-refine u Lp}
For any $p\in (0,\infty)$, $\iota\in (-\frac{1}{p},\infty)$, and $a\in (0,1)$, there exists a positive constant $C(p,\iota,a,T)$ such that
\begin{equation*}
\big|\chi_{a}^\sharp \rho_0^{\iota\beta} U(t)\big|_p\leq C(p,\iota,a,T) \qquad \text{for all $t\in [0,T]$},
\end{equation*}
where  $\chi_{a}^\sharp$ denotes the characteristic function on $(a,1]$ for $a\in (0,1)$
{\rm  (}see {\rm \S \ref{othernotation})}.

\end{Lemma}
\begin{proof}
Let $p\in (0,\infty)$, $\iota\in (-\frac{1}{p},\infty)$, and $a\in (0,1)$. Let $\varepsilon>0$ be a fixed constant such that
\begin{equation*}
\max\Big\{0,(\iota p+1)\beta-1,(\iota p+1)\beta-\frac{p}{2},\frac{(\iota p+1)\beta^2}{\beta+1}\Big\}<\varepsilon<(\iota p+1)\beta.
\end{equation*}
Note that $\varepsilon$ is well-defined under the above constraint and satisfies
\begin{equation*}
0<(\iota p+1)\beta-\varepsilon<1,\qquad \frac{p}{(\iota p+1)\beta-\varepsilon}>2, \qquad \text{and} \ \ \frac{-\beta+\varepsilon}{1+\varepsilon-(\iota p+1)\beta}>-\beta.
\end{equation*}
Hence, it follows from Lemma \ref{lp-uv}, the fact that $\rho_0^\beta\sim 1-r$, and the H\"older inequality that
\begin{equation*}
\begin{aligned}
|\chi_{a}^\sharp \rho_0^{\iota\beta} U|_p^p&=\int_a^1 \rho_0^{\iota\beta p} |U|^p\,\mathrm{d}r=\int_a^1 \rho_0^{-\beta+\varepsilon} (\rho_0^{(\iota p+1)\beta-\varepsilon} |U|^p)\,\mathrm{d}r\\
&\leq \Big(\int_a^1 \rho_0^\frac{-\beta+\varepsilon}{1+\varepsilon-(\iota p+1)\beta}\,\mathrm{d}r\Big)^{1+\varepsilon-(\iota p+1)\beta}\Big(\int_a^1 \rho_0 |U|^\frac{p}{(\iota p+1)\beta-\varepsilon}\,\mathrm{d}r\Big)^{(\iota p+1)\beta-\varepsilon}\\
&\leq C(p,\iota,a)\Big(\int_0^1 r^m\rho_0 |U|^\frac{p}{(\iota p+1)\beta-\varepsilon}\,\mathrm{d}r\Big)^{(\iota p+1)\beta-\varepsilon}\leq C(p,\iota,a,T).
\end{aligned}
\end{equation*}

This completes the proof.
\end{proof}

\begin{Lemma}\label{lemma-v Lp ex}
For any $p\in [2,\infty)$, $\iota\in (\frac{p-1}{p},\infty)$, and $a\in (0,1)$, there exists a positive constant $C(p,\iota,a,T)$ such that 
\begin{equation*}
\big|\chi_{a}^\sharp\rho_0^{\iota\beta}V(t)\big|_p\leq C(p,\iota,a,T) \qquad \text{for all }t\in[0,T].
\end{equation*}
\end{Lemma}
\begin{proof}
Let $p\in [2,\infty)$, $\iota\in (\frac{p-1}{p},\infty)$ and $a\in (0,1)$. First, multiplying \eqref{V-solution} by $\chi_{a}^\sharp\rho_0^{\iota\beta}$ and taking the $L^p$-norm of the resulting equality, we obtain from \eqref{eq:eta}, Lemmas \ref{lemma-bound depth} and \ref{lemma-refine u Lp}, and the Minkowski integral inequality that, for all $t\in[0,T]$,
\begin{equation*}
\begin{aligned}
|\chi_{a}^\sharp\rho_0^{\iota\beta}V|_p
&\leq |\chi_{a}^\sharp\rho_0^{\iota\beta}v_0|_p
+C_0\int_0^t\big|\chi_a^\sharp\rho_0^{\iota\beta}\varrho^{\gamma-1}U\big|_p\,\mathrm{d}s\\
&\leq |\chi_{a}^\sharp\rho_0^{\iota\beta}v_0|_p
+C(a,p,T)\int_0^t\big|\chi_a^\sharp\rho_0^{\iota\beta}U\big|_p\,\mathrm{d}s
\leq C(p,\iota,a,T).
\end{aligned}
\end{equation*}
Here, for the initial data, since $\rho_0^\beta\sim 1-r$ and $\iota-1>-\frac{1}{p}$, we have
\begin{equation*}
\begin{aligned}
\big|\chi_{a}^\sharp\rho_0^{\iota\beta}v_0\big|_p&\leq C(p)\big(\big|\chi_{a}^\sharp\rho_0^{\iota\beta}u_0\big|_p+\big|\chi_{a}^\sharp\rho_0^{(\iota-1)\beta}(\rho_0^\beta)_r\big|_p\big)\\
&\leq C(p)\big(|\rho_0^{\beta}|_\infty^\iota |u_0|_\infty+|(1-r)^{\iota-1}|_p|(\rho_0^\beta)_r|_\infty\big)\leq C(\iota,p).
\end{aligned}    
\end{equation*}

This completes the proof.
\end{proof}

\begin{Lemma}\label{lemma-v Linfty ex}
For any $a\in (0,1)$, there exists a constant $C(a,T)>0$ such that 
\begin{equation*}
\big|\chi_{a}^\sharp\rho_0^{\beta}V(t)\big|_\infty\leq C(a,T) \qquad \text{for all $t\in [0,T]$}.
\end{equation*}
\end{Lemma}
\begin{proof}
We divide the proof into the following two steps.

\smallskip
\textbf{1.} Let $k\in \NN^*$ be a fixed constant such that 
\begin{equation*}
k\geq \max\big\{2,\frac{1}{\gamma-1}\big\}.
\end{equation*}
It follows from  \eqref{v-expression} and Lemma \ref{sobolev-embedding} that
\begin{equation}\label{pre-g4-g7}
\begin{aligned}
\big|\chi_{a}^\sharp \rho_0^{\beta}\varrho^{\gamma-1}U\big|_\infty^k&\leq C(a)\big|\chi_a^\sharp(\rho_0^{k\beta}\varrho^{(\gamma-1)k}|U|^k)_r\big|_1\\
&\leq C(a)\big|\chi_a^\sharp\rho_0^{(k-1)\beta}(\rho_0^\beta)_r\varrho^{(\gamma-1)k}U^k\big|_1+C(a)\big|\chi_a^\sharp\rho_0^{k\beta}\varrho^{(\gamma-1)k}\eta_rU^k(V,U)\big|_1\\
&\quad +C(a)\big|\chi_a^\sharp\rho_0^{k\beta}\varrho^{(\gamma-1)k}\eta_rU^{k-1}D_\eta U\big|_1:=\sum_{i=6}^8 \mathrm{G}_i.
\end{aligned}    
\end{equation}

For $\mathrm{G}_6$--$\mathrm{G}_8$, by \eqref{eq:eta}, Lemmas \ref{lemma-bound depth}--\ref{lemma-lower bound jacobi} and \ref{lemma-refine u Lp}--\ref{lemma-v Lp ex}, and the H\"older inequality, we have
\begin{equation}\label{G4-G7}
\begin{aligned}
\mathrm{G}_6&\leq C_0|\varrho|_\infty^{(\gamma-1)k}|(\rho_0^\beta)_r|_\infty\big|\chi_{a}^\sharp\rho_0^{(1-\frac{1}{k})\beta}U\big|_k^k\leq C(a,T),\\
\mathrm{G}_7&\leq C_0\Big|\frac{r^m}{\eta^m}\Big|_\infty|\varrho|_\infty^{(\gamma-1)k-1} \big(|\chi_{a}^\sharp \rho_0^{\beta}V|_2\big|\chi_{a}^\sharp\rho_0^\frac{(k-1)\beta+1}{k}U\big|_{2k}^k+\big|\chi_{a}^\sharp \rho_0^\frac{k\beta+1}{k+1}U\big|_{k+1}^{k+1}\big)\leq C(a,T),\\
\mathrm{G}_8&\leq C(a)\Big|\frac{r^m}{\eta^m}\Big|_\infty|\varrho|_\infty^{(\gamma-1)k-1}\big|\chi_{a}^\sharp \rho_0^{k\beta+\frac{1}{2}} U^{k-1} \big|_2\big|(r^m\rho_0)^\frac{1}{2}D_\eta U\big|_2\\
&\leq C(a,T)\big|\chi_{a}^\sharp\rho_0^{\frac{2k\beta+1}{2k-2}}U \big|_{2k-2}^{k-1}\big|(r^m\rho_0)^\frac{1}{2}D_\eta U\big|_2 \leq C(a,T)\big|(r^m\rho_0)^\frac{1}{2}D_\eta U\big|_2.
\end{aligned}    
\end{equation}
Substituting \eqref{G4-G7} into \eqref{pre-g4-g7} leads to
\begin{equation}\label{abb}
\big|\chi_{a}^\sharp \rho_0^{\beta}\varrho^{\gamma-1}U\big|_\infty^k\leq C(a,T)\big(\big|(r^m\rho_0)^\frac{1}{2}D_\eta U\big|_2+1\big). 
\end{equation}

\textbf{2.} Next, multiplying \eqref{V-solution} by $\chi_{a}^\sharp\rho_0^{\beta}$ and taking the $L^\infty$-norm of the resulting equality, we obtain from \eqref{abb}, Lemma \ref{lemma-basic energy}, and the H\"older inequality that
\begin{equation*} 
\begin{aligned}
|\chi_{a}^\sharp \rho_0^{\beta}V|_\infty&\leq |\chi_{a}^\sharp \rho_0^{\beta}v_0|_\infty+\frac{A\gamma}{2\mu}\int_0^t \big|\chi_{a}^\sharp \rho_0^{\beta}\varrho^{\gamma-1}U\big|_\infty\,\mathrm{d}s\\
&\leq |\chi_{a}^\sharp \rho_0^{\beta}v_0|_\infty+C(a,T)\int_0^t \big(\big|(r^m\rho_0)^\frac{1}{2}D_\eta U\big|_2^\frac{1}{k}+1\big)\,\mathrm{d}s \leq C(a,T),
\end{aligned}
\end{equation*}
where the initial data can be controlled by
\begin{equation*}
|\chi_{a}^\sharp \rho_0^{\beta}v_0|_\infty\leq C_0\big|(\rho_0^{\beta}u_0,(\rho_0^\beta)_r)\big|_\infty \leq C_0(|\rho_0^{\beta}|_\infty |u_0|_\infty+ |(\rho_0^\beta)_r|_\infty)\leq C_0.    
\end{equation*}
\end{proof}

\subsection{Boundedness of the effective velocity near the symmetric center}

\begin{Lemma}\label{new-u-lp}
For all $p\in [2,\infty)$, there exists a constant $C(p,T)>0$ such that, for all $t\in [0,T]$,
\begin{equation*}
\big|(\zeta^2\eta_r\varrho)^\frac{1}{p}U(t)\big|_p^p+\int_0^t \Big|(\zeta^2\eta_r\varrho)^\frac{1}{2}|U|^\frac{p-2}{2}\big(D_\eta U,\frac{U}{\eta}\big)\Big|_2^2\,\mathrm{d}s \leq C(p,T)\big(\sup_{s\in[0,t]}|\zeta V|_\infty^2+1\big).
\end{equation*}
\end{Lemma}
\begin{proof}
Multiplying  $\eqref{eq:VFBP-La-eta}_1$ by $\zeta^2\eta_r|U|^{p-2}U$ gives
\begin{equation*}
\begin{aligned}
&\,\frac{1}{p}(\zeta^2\eta_r \varrho|U|^p)_t  +2\mu(p-1) \zeta^2\eta_r\varrho|U|^{p-2}\Big(|D_\eta U|^2+\frac{m}{p}\frac{U^2}{\eta^2}\Big)\\
&=A(p-1)\zeta^2\eta_r\varrho^\gamma|U|^{p-2}D_\eta U- \frac{m}{p}\frac{\zeta^2\eta_r \varrho V|U|^p}{\eta}\\
&\quad -2\zeta \zeta_r\Big( 2\mu \varrho|U|^{p-2}U D_\eta U+\frac{2\mu m}{p} \frac{\varrho|U|^p}{\eta}-A\varrho^\gamma|U|^{p-2}U\Big)\\
&\quad+\Big(2\mu \zeta^2\varrho|U|^{p-2}U D_\eta U +\frac{2\mu m}{p} \frac{\zeta^2\varrho|U|^p}{\eta}- A\zeta^2\varrho^\gamma|U|^{p-2}U \Big)_r.
\end{aligned}
\end{equation*}

Then integrating the above over $I$ leads to
\begin{equation}\label{G8-G10}
\begin{aligned}
&\,\frac{1}{p}\frac{\mathrm{d}}{\mathrm{d}t}\big|(\zeta^2\eta_r \varrho)^\frac{1}{p} U\big|_p^p  +\frac{\mu}{2} \Big|(\zeta^2\eta_r\varrho)^\frac{1}{2}|U|^\frac{p-2}{2}\big(D_\eta U,\frac{U}{\eta}\big)\Big|_2^2\\
&\leq A(p-1) \int_0^1\zeta^2\varrho^\gamma|U|^{p-2}U_r\,\mathrm{d}r- \frac{m}{p} \int_0^1 \frac{\zeta^2\eta_r \varrho V|U|^p}{\eta}\,\mathrm{d}r\\
&\quad- 2 \int_0^1 \zeta\zeta_r \Big( 2\mu \varrho|U|^{p-2}U D_\eta U+\frac{2\mu m}{p} \frac{\varrho|U|^p}{\eta}-A\varrho^\gamma|U|^{p-2}U\Big)\,\mathrm{d}r:=\sum_{i=9}^{11} \mathrm{G}_i.
\end{aligned}
\end{equation}

For $\mathrm{G}_9$, it follows from Lemmas \ref{lemma-near depth} and \ref{lemma-bound depth}, and the H\"older and Young inequalities that
\begin{equation}\label{G8}
\begin{aligned}
\mathrm{G}_9&\leq C(p)|\varrho|_\infty^{\gamma-1}|\zeta\eta_r\varrho|_1^\frac{1}{p}\big|(\zeta^2\eta_r\varrho)^\frac{1}{p} U\big|_p^\frac{p-2}{2}\big|(\zeta^2\eta_r\varrho)^\frac{1}{2}|U|^\frac{p-2}{2}D_\eta U\big|_2\\
&\leq C(p,T)+\big|(\zeta^2\eta_r\varrho)^\frac{1}{p} U\big|_{p}^{p}+\frac{\mu}{20}\big|(\zeta^2\eta_r\varrho)^\frac{1}{2}|U|^\frac{p-2}{2}D_\eta U\big|_2^2.
\end{aligned}
\end{equation}

For $\mathrm{G}_{10}$, it follows from \eqref{eq:eta}, and the H\"older and Young inequalities that
\begin{align*}
&\begin{aligned}
\text{if $n=2$: } \quad \mathrm{G}_{10}&=-\frac{1}{p}\int_0^1 \frac{\zeta^2 \eta_r \varrho V|U|^p}{\eta}\,\mathrm{d}r\leq C(p)|\zeta V|_\infty \Big|\big(\frac{r\rho_0}{\eta^2}\big)^\frac{1}{p}U\Big|_p^p \\
&\leq C(p)(|\zeta V|_\infty^2+1) \Big|\big(\frac{r^m\rho_0}{\eta^2}\big)^\frac{1}{p}U\Big|_p^p;
\end{aligned}\\
&\begin{aligned}
\text{if $n=3$: } \quad \mathrm{G}_{10}&=-\frac{2}{p}\int_0^1 \frac{\zeta^2 \eta_r \varrho V|U|^p}{\eta}\,\mathrm{d}r \leq C(p)|\zeta V|_\infty \Big|\big(\frac{r^2\rho_0}{\eta^2}\big)^\frac{1}{p}U\Big|_p^\frac{p}{2} \Big|\big(\frac{\zeta^2\eta_r\varrho}{\eta^2}\big)^\frac{1}{p}U\Big|_p^\frac{p}{2}\\
&\leq C(p)|\zeta V|_\infty^2 \Big|\big(\frac{r^m\rho_0}{\eta^2}\big)^\frac{1}{p}U\Big|_p^p+\frac{\mu}{20} \Big|\big(\frac{\zeta^2\eta_r\varrho}{\eta^2}\big)^\frac{1}{p}U\Big|_p^p.    
\end{aligned}
\end{align*}
To sum up, in both cases, we have
\begin{equation}
\mathrm{G}_{10}\leq C(p)(|\zeta V|_\infty^2+1) \Big|\big(\frac{r^m\rho_0}{\eta^2}\big)^\frac{1}{p}U\Big|_p^p+\frac{\mu}{20} \Big|\big(\frac{\zeta^2\eta_r\varrho}{\eta^2}\big)^\frac{1}{p}U\Big|_p^p.
\end{equation}

For $\mathrm{G}_{11}$,  Lemma \ref{lemma-low jacobi 1}, together with \eqref{eq:eta},
Lemmas \ref{lp-uv} and \ref{lemma-bound depth}--\ref{lemma-lower bound jacobi}, and the H\"older and
Young inequalities, gives
\begin{equation}\label{G10}
\begin{aligned}
\mathrm{G}_{11}&\leq C(p)\int_\frac{1}{2}^\frac{5}{8} \Big(\frac{r^m\rho_0|U|^{p-1} |D_\eta U|}{\eta^m\eta_r}+ \frac{r^m\rho_0|U|^p}{\eta^{m+1}\eta_r}+\varrho^{\gamma-1}\big(\frac{r^m\rho_0}{\eta^m\eta_r}\big)|U|^{p-1}\Big)\,\mathrm{d}r\\
&\leq C(p,T)\big(\big|(r^m\rho_0)^\frac{1}{2}|U|^{p-1}D_\eta U\big|_2+ \big|(r^m\rho_0)^\frac{1}{p}U\big|_p^p +|\varrho|_\infty^{\gamma-1}\big|(r^m\rho_0)^\frac{1}{p}U\big|_{p}^{p-1}\big)\\
&\leq C(p,T)+\big|(r^m\rho_0)^\frac{1}{2}|U|^{p-1}D_\eta U\big|_2^2.
\end{aligned}    
\end{equation}

Consequently, collecting \eqref{G8-G10}--\eqref{G10}, we have
\begin{equation*}
\begin{aligned}
&\,\frac{1}{p}\frac{\mathrm{d}}{\mathrm{d}t}\big|(\zeta^2\eta_r \varrho)^\frac{1}{p} U\big|_p^p  +\frac{\mu}{4}\Big|(\zeta^2\eta_r\varrho)^\frac{1}{2}|U|^\frac{p-2}{2}\big(D_\eta U,\frac{U}{\eta}\big)\Big|_2^2\notag\\
&\leq \big|(\zeta^2\eta_r\varrho)^\frac{1}{p} U\big|_{p}^{p}+(|\zeta V|_\infty^2+1) \Big|\big(\frac{r^m\rho_0}{\eta^2}\big)^\frac{1}{p}U\Big|_p^p+\big|(r^m\rho_0)^\frac{1}{2}|U|^{p-1}D_\eta U\big|_2^2+C(p,T),
\end{aligned}
\end{equation*}
which, along with Lemma \ref{lp-uv} and the Gr\"onwall inequality, leads to the desired estimates.
\end{proof}

Next, we can derive the $L^1([0,T];L^\infty)$-estimate of $\zeta\varrho^{\gamma-1} U$.
\begin{Lemma}\label{cru3}
For any $\varepsilon\in(0,1)$, there exists a constant $C(\varepsilon,T)>0$ such that 
\begin{equation*}
\int_0^t |\zeta\varrho^{\gamma-1}U|_{\infty}\, \mathrm{d}s\leq C(\varepsilon,T) +\varepsilon\sup_{s\in[0,t]}|\zeta V|_\infty \qquad \text{for all $t\in [0,T]$}.
\end{equation*}
\end{Lemma}
\begin{proof}
Let $k\in \NN^*$ be a fixed number such that 
\begin{equation*}
k>\max\big\{3,\frac{1}{\gamma-1}\big\}.
\end{equation*}

First, it follows from Lemma \ref{sobolev-embedding} that
\begin{equation}\label{6.10}
\begin{aligned}
&\,|\zeta \varrho^{\gamma-1}U|_\infty^{k}\leq C_0\big|(\zeta^k \varrho^{k\gamma-k} |U|^{k})_r\big|_1\\
&\leq C_0\big(\big|\zeta^{k-1}\zeta_r\varrho^{k\gamma-k} U^{k}\big|_1+ \big|\zeta^k\varrho^{k\gamma-k}\eta_rU^k(V,U)\big|_1+ \big|\zeta^k\varrho^{k\gamma-k}\eta_r U^{k-1}D_\eta U\big|_1\big):=\sum_{i=12}^{14}\mathrm{G}_i.
\end{aligned}    
\end{equation}

Then, using \eqref{v-expression}, Lemmas \ref{lemma-bound depth}--\ref{lemma-lower bound jacobi}, and the H\"older inequality, we have
\begin{equation}\label{wuqiong-G12-G14}
\begin{aligned}
\mathrm{G}_{12}&\leq C_0|\eta_r^{-1}|_{\infty}|\zeta|_\infty^{k-3}|\varrho|_\infty^{k(\gamma-1)-1}\big|(\zeta^2\varrho \eta_r)^\frac{1}{k}U\big|_k^{k}\leq C(T)\big|(\zeta^2\varrho \eta_r)^\frac{1}{k}U\big|_k^{k},\\
\mathrm{G}_{13}&\leq C_0|\varrho|_\infty^{k(\gamma-1)-1}|\zeta|_\infty^{k-3}\big(|\zeta V|_\infty\big|(\zeta^2\eta_r\varrho)^\frac{1}{k}U\big|_k^k+ |\zeta|_\infty\big|(\zeta^2\eta_r\varrho)^\frac{1}{k+1}U\big|_{k+1}^{k+1}\big)\\
&\leq C(T)\big(|\zeta V|_\infty\big|(\zeta^2\eta_r\varrho)^\frac{1}{k}U\big|_k^k+ \big|(\zeta^2\eta_r\varrho)^\frac{1}{k+1}U\big|_{k+1}^{k+1}\big),\\
\mathrm{G}_{14}&\leq C_0|\varrho|_\infty^{k(\gamma-1)-1}|\zeta|_\infty^{k-2}\big|(\zeta^2\eta_r\varrho)^{\frac{1}{k}}U\big|_{k}^\frac{k}{2}\big|(\zeta^2\eta_r\varrho)^\frac{1}{2}|U|^\frac{k-2}{2}D_\eta U\big|_2 \\
&\leq C(T)\big|(\zeta^2\eta_r\varrho)^{\frac{1}{k}}U\big|_{k}^\frac{k}{2}\big|(\zeta^2\eta_r\varrho)^\frac{1}{2}|U|^\frac{k-2}{2}D_\eta U\big|_2.  
\end{aligned}
\end{equation}

Combining \eqref{6.10}--\eqref{wuqiong-G12-G14} gives
\begin{equation*}
\begin{aligned}
|\zeta \varrho^{\gamma-1}U|_\infty&\leq C(T)(1+|\zeta V|_\infty^\frac{1}{k})\big|(\zeta^2\eta_r\varrho)^{\frac{1}{k}}U\big|_{k}+C(T)\big|(\zeta^2\eta_r\varrho)^{\frac{1}{k+1}}U\big|_{k+1}^\frac{k+1}{k}\\
&\quad+C(T)\big|(\zeta^2\eta_r\varrho)^{\frac{1}{k}}U\big|_{k}^\frac{1}{2}\big|(\zeta^2\eta_r\varrho)^\frac{1}{2}|U|^\frac{k-2}{2}D_\eta U\big|_2^\frac{1}{k},
\end{aligned}
\end{equation*}
which, along with Lemma \ref{new-u-lp} and the Young inequality, leads to
\begin{equation*}
|\zeta \varrho^{\gamma-1}U|_\infty\leq C(T)\Big(1+\sup_{s\in[0,t]}|\zeta V|_\infty^\frac{3}{k} +\big(1+\sup_{s\in[0,t]}|\zeta V|_\infty^\frac{1}{k}\big)\big|(\zeta^2\eta_r\varrho)^\frac{1}{2}|U|^\frac{k-2}{2}D_\eta U\big|_2^\frac{1}{k}\Big).  
\end{equation*}

Finally, integrating the above over $[0,t]$, then we obtain from the fact that $k>3$, Lemma \ref{new-u-lp}, and the H\"older and Young inequalities that, for all $\varepsilon\in (0,1)$,
\begin{equation*}
\begin{aligned}
&\int_0^t |\zeta \varrho^{\gamma-1}U|_\infty\,\mathrm{d}s\\
&\leq C(T)\Big(1+\sup_{s\in[0,t]}|\zeta V|_\infty^\frac{3}{k}+\big(1+\sup_{s\in[0,t]}|\zeta V|_\infty^\frac{1}{k}\big)\int_0^t\big|(\zeta^2\eta_r\varrho)^\frac{1}{2}|U|^\frac{k-2}{2}D_\eta U\big|_2^\frac{1}{k}\,\mathrm{d}s\Big)\\
&\leq C(T)\Big(1+\sup_{s\in[0,t]}|\zeta V|_\infty^\frac{3}{k}+\sup_{s\in[0,t]}|\zeta V|_\infty^\frac{1}{k}\Big) \leq C(\varepsilon,T) +\varepsilon\!\sup_{s\in[0,t]}|\zeta V|_\infty.
\end{aligned}
\end{equation*}

This completes the proof.
\end{proof}

\begin{Lemma}\label{lemma-v Linfty in}
For any $a\in(0,1)$, there exists a constant $C(a,T)>0$ such that
\begin{equation*}
|\zeta_{a} V(t)|_\infty\leq C(a,T) \qquad \text{for all $t\in [0,T]$}.
\end{equation*}
\end{Lemma}
\begin{proof}
First, it follows from \eqref{V-solution} and Lemma \ref{cru3} that, for any $\varepsilon\in(0,1)$,
\begin{equation}\label{daiding}
\sup_{s\in[0,t]}|\zeta V|_\infty \leq |\zeta v_0|_\infty +\frac{A\gamma}{2\mu}\int_0^t \big|\zeta \varrho^{\gamma-1}U\big|_\infty\,\mathrm{d}s \leq C(\varepsilon,T) +\frac{A\gamma \varepsilon}{2\mu}\sup_{s\in[0,t]}|\zeta V|_\infty.
\end{equation}
Here, the bound of $|\zeta v_0|_\infty$ follows from \eqref{v-expression} and the fact that $\rho_0^\beta\sim 1-r$, {\it i.e.},
\begin{equation}
|\zeta v_0|_\infty\leq |\zeta u_0|_\infty+2\mu |\zeta(\log\rho_0)_r|_\infty\leq |\zeta u_0|_\infty+C_0|\zeta\rho_0^{-\beta}|_\infty|(\rho_0^\beta)_r|_\infty\leq C_0.
\end{equation}

Then setting $\varepsilon$ in \eqref{daiding} such that
\begin{equation*}
\varepsilon=\min\big\{\frac{1}{2},\frac{\mu}{A\gamma}\big\},
\end{equation*}
we obtain that, for all $t\in [0,T]$,
\begin{equation*}
|\zeta V(t)|_\infty\leq C(T)\implies |\zeta_{a} V(t) |_\infty\leq C(T) \qquad \text{for any } a\in\big(0,\frac{1}{2}\big].
\end{equation*}

Finally, for $a\in (\frac{1}{2},1)$, it follows from the above,  $\rho_0^\beta\sim 1-r$, and Lemma \ref{lemma-v Linfty ex} that
\begin{equation*}
|\zeta_{a} V(t)|_\infty\leq |\zeta V(t)|_\infty+\|V(t)\|_{L^\infty(\frac{1}{2},\frac{1+3a}{4})}\leq C(T)+C(a) |\chi^\sharp\rho_0^\beta V(t)|_\infty\leq C(a,T).
\end{equation*}

This completes the proof.
\end{proof}

\section{Global-In-Time Uniform Upper Bounds of \texorpdfstring{$(\eta_r,\frac{\eta}{r})$}{}}\label{Section-etarupper}

The purpose of this section is to establish the global uniform upper bounds of $(\eta_r,\frac{\eta}{r})$. The corresponding global uniform lower bounds have already been obtained in Lemma \ref{lemma-lower bound jacobi}.
 
\subsection{Uniform upper bounds of \texorpdfstring{$(\eta_r,\frac{\eta}{r})$}{} in the exterior domain}

We first derive the uniform upper bounds of $(\eta_r,\frac{\eta}{r})$ in the exterior domain. 
\begin{Lemma}\label{lemma-upper jacobi-ex}
There exist a constant $C(T)>0$ and a constant $a_T\in (0,1)$ which depends only on $C(T)$ such that, for any $(t,r)\in [0,T]\times [a_T,1]$,
\begin{equation*}
\frac{\eta(t,r)}{r}+\eta_r(t,r) \leq C(T).
\end{equation*}
\end{Lemma}
\begin{proof}
First, rewrite \eqref{lem42-den3} as
\begin{equation}\label{6624}
\mathfrak{H}_t(t,r) =-\mathfrak{L}(t,r)+\mathfrak{F}_t(t,r) 
\end{equation}
for $t\in [0,T]$ and $r\in (0,1]$, where
\begin{equation*}
\begin{aligned}
\mathfrak{H}(t,r)&:=2\mu(\eta^m \varrho)(t,r)+2\mu m\int_{r}^1\eta^{m-1}\eta_r\varrho\,\mathrm{d}\hat r, \qquad \mathfrak{F}(t,r):= \int_{r}^1 \hat r^m\rho_0 U \,\mathrm{d}\hat r,\\
\mathfrak{L}(t,r)&:=A(\eta^m \varrho^\gamma)(t,r)+A m\int_{r}^1 \eta^{m-1}\eta_r \varrho^\gamma \,\mathrm{d}\hat r.    
\end{aligned}
\end{equation*}
Notice that, since $\gamma>1$, we can deduce from Lemma \ref{lemma-bound depth} that $\mathfrak{L}(t,r)\leq C(T)\mathfrak{H}(t,r)$. Hence, it follows from \eqref{6624} that 
\begin{equation*}
\mathfrak{H}_t(t,r)+C(T)\mathfrak{H}(t,r)\geq \mathfrak{F}_t(t,r).
\end{equation*}
Then, we can solve the above differential inequality for $\mathfrak{H}$ to obtain
\begin{equation*}
\begin{aligned}
\mathfrak{H}(t,r) &\geq e^{-C(T)t}
\mathfrak{H}(0,r) +\int_r^1\hat r^m\rho_0 U(t,\hat r)\,\mathrm{d}\hat r
-e^{-C(T)t}\int_r^1\hat r^m\rho_0 u_0 \,\mathrm{d}\hat r\\
&\quad-C(T)\int_0^t e^{-C(T)(t-\tau)}
\Big(\int_r^1\hat r^m\rho_0 U(\tau,\hat r)\,\mathrm{d}\hat r\Big)\,\mathrm{d}\tau\\
&\geq C(T)^{-1}\rho_0(r)
-C(T)\rho_0 ^{1+\frac{\beta^2}{1+2\beta}},
\end{aligned}
\end{equation*}
where we have used the estimate of $\widetilde{G}_1$ in \eqref{eq:repair-tail-errors} and the fact that $\mathfrak{H}(0,r)\geq 2\mu r^m\rho_0$ in the last inequality. Hence, using the estimate for $\widetilde{G}_2$ in \eqref{eq:repair-tail-errors}, we obtain
\begin{equation}\label{llll}
(\eta^m\varrho)(t,r)\geq\rho_0(r)\Big(C(T)^{-1}-C(T)\big(\rho_0(r)^{\frac{\beta^2}{1+2\beta}}+\rho_0(r)^{\beta}\big)\Big).
\end{equation}

For any $T>0$, we can choose a constant $a_T>0$ sufficiently close to $1$, which depends only on $C(T)$,  so that
\begin{equation*}
C(T)\big(\rho_0(r)^{\frac{\beta^2}{1+2\beta}}+\rho_0(r)^{\beta}\big)\leq\frac{1}{2}C(T)^{-1}\qquad \text{for any $r\in [a_T,1]$}.
\end{equation*}
Therefore, by \eqref{llll}, for all $(t,r)\in[0,T]\times[a_T,1]$,
\begin{equation}\label{eq:boundary-collar-geometry}
\eta^m\varrho=\frac{r^m\rho_0}{\eta_r}\geq C(T)^{-1}\rho_0 \implies \eta_r\leq C(T)r^m\leq C(T).
\end{equation}

Finally, it follows from \eqref{eq:boundary-collar-geometry} and Lemmas \ref{lemma-refine u Lp} and \ref{sobolev-embedding} that, for any $(t,r)\in[0,T]\times[a_T,1]$,
\begin{equation*}
\begin{aligned}
\frac{\eta(t,r)}{r}&\leq a_T^{-1}\|\eta(t)\|_{L^\infty(a_T,1)}\leq C(T)\Big(\int_{a_T}^1 \eta(t,r)\,\mathrm{d}r+ \int_{a_T}^1 \eta_r(t,r)\,\mathrm{d}r\Big)\\
&\leq C(T)\int_{a_T}^1\Big(r+\int_0^t U(s,r)\,\mathrm{d}s\Big)\,\mathrm{d}r+ C(T) \leq C(T).
\end{aligned}
\end{equation*} 

This completes the proof of Lemma \ref{lemma-upper jacobi-ex}.
\end{proof}

\subsection{Uniform upper bounds of \texorpdfstring{$(\eta_r,\frac{\eta}{r})$}{} near the symmetric center}

\begin{Lemma}\label{lemma-u infty}
For any $a\in(0,1)$, there exists a constant $C(a,T)$ such that, for all $t\in [0,T]$,
\begin{equation*}
|\zeta_{a}\sqrt{\eta_r}U(t)|_2^2+\int_0^t\Big(\Big|\zeta_{a}\sqrt{\eta_r}\big(D_\eta U,\frac{U}{\eta}\big)\Big|_2^2+|\zeta_{a} U|_\infty^2\Big)\,\mathrm{d}s\leq C(a,T).
 \end{equation*}
\end{Lemma}
\begin{proof}
We divide the proof into three steps.

\smallskip
\textbf{1.} First, multiplying $\eqref{eq:VFBP-La-eta}_1$ by $2\zeta_{a}^2\eta_r\varrho^{-1} U$ with $a\in(0,1)$, along with \eqref{v-expression}, gives
\begin{equation*}
\begin{aligned}
&\,(\zeta_{a}^2\eta_rU^2)_t+2\mu \zeta_{a}^2\eta_r\Big(2 |D_\eta U|^2 + m \frac{U^2}{\eta^2}\Big)\\
&=\Big(4\mu \zeta_{a}^2 UD_\eta U+ 2\mu m \zeta_{a}^2\frac{U^2}{\eta} -\frac{1}{3}\zeta_{a}^2U^3\Big)_r-\frac{A\gamma}{\mu}\zeta_{a}^2\eta_r\varrho^{\gamma-1}(V-U)U+2\zeta_{a}^2\eta_r VUD_\eta U\\
&\quad- 2\zeta_{a}(\zeta_{a})_r\Big(4\mu UD_\eta U+2\mu m \frac{U^2}{\eta}-\frac{1}{3}U^3\Big).
\end{aligned}
\end{equation*}
Integrating the above over $I$ leads to
\begin{equation}\label{dt-g17-g19}
\begin{aligned}
&\,\frac{\mathrm{d}}{\dt}|\zeta_{a}\sqrt{\eta_r} U|_2^2 + \mu\Big|\zeta_{a} \sqrt{\eta_r} \big(D_\eta U,\frac{U}{\eta}\big)\Big|_2^2 \\
&\leq -\frac{A\gamma}{\mu}\int_0^1\zeta_{a}^2\eta_r\varrho^{\gamma-1}(V-U)U\,\mathrm{d}r+2\int_0^1\zeta_{a}^2\eta_rVU D_\eta U\,\mathrm{d}r\\
&\quad- 2\int_0^1\zeta_{a}(\zeta_{a})_r\Big(4\mu UD_\eta U+2\mu m \frac{U^2}{\eta}-\frac{1}{3}U^3\Big)\,\mathrm{d}r:=\sum_{i=18}^{20}\mathrm{G}_i. 
\end{aligned}
\end{equation}

For $\mathrm{G}_{18}$ and $\mathrm{G}_{19}$, we use Lemmas \ref{lemma-bound depth} and \ref{lemma-v Linfty in}. For $\mathrm{G}_{20}$, Lemma \ref{lemma-lower bound jacobi}, together with Lemma \ref{lp-uv} and the H\"older and Young inequalities, gives
\begin{equation}\label{g17-g19}
\begin{aligned}
\mathrm{G}_{18}&\leq C_0|\zeta_{a}\sqrt{\eta_r} U|_2\big(\big|\zeta_{a}\sqrt{\eta_r}\varrho^{\gamma-1} V\big|_2+|\varrho|_\infty^{\gamma-1} |\zeta_{a}\sqrt{\eta_r} U|_2\big) \leq C(T)\big|\zeta_{a}\sqrt{\eta_r} (U,\varrho^{\gamma-1} V)\big|_2^2,\\
\mathrm{G}_{19}&\leq C_0\big|\zeta_{\frac{1+3a}{4}} V\big|_\infty |\zeta_{a}\sqrt{\eta_r} U|_2|\zeta_{a} \sqrt{\eta_r}D_\eta U|_2\leq C(a,T)|\zeta_{a}\sqrt{\eta_r} U|_2^2+\frac{\mu}{8}|\zeta_{a} \sqrt{\eta_r}D_\eta U|_2^2,\\
\mathrm{G}_{20}&\leq C(a,T)\big|(r^m\rho_0)^\frac{1}{2}U\big|_2\Big(\Big|\zeta_{a}\sqrt{\eta_r}\big(D_\eta U,\frac{U}{\eta}\big)\Big|_2+\big|(r^m\rho_0)^\frac{1}{4}U\big|_4^2\Big)\\
&\leq C(a,T)+\frac{\mu}{8}\Big|\zeta_a\sqrt{\eta_r}\big(D_\eta U,\frac{U}{\eta}\big)\Big|_2^2.
\end{aligned}    
\end{equation}

Therefore, we obtain from \eqref{dt-g17-g19}--\eqref{g17-g19} that
\begin{equation}\label{dtu}
\frac{\mathrm{d}}{\dt}\big|\zeta_{a}\sqrt{\eta_r} U\big|_2^2 + \frac{\mu}{2}\Big|\zeta_{a}\sqrt{\eta_r}\big(D_\eta U,\frac{U}{\eta}\big)\Big|_2^2 \leq C(a,T)\big(\big|\zeta_{a}\sqrt{\eta_r} (U,\varrho^{\gamma-1} V)\big|_2^2+1\big).  
\end{equation}

\smallskip
\textbf{2.} Multiplying \eqref{eq:v} by $2\zeta_{a}^2\eta_r\varrho^{2\gamma-2} V$, together with $\eqref{eq:VFBP-La}_1$, we have 
\begin{equation*}
(\zeta_{a}^2\eta_r\varrho^{2\gamma-2}V^2)_t=\zeta_{a}^2\varrho^{2\gamma-2}V^2\eta_r\big((3-2\gamma)D_\eta U-(2\gamma-2)m \frac{U}{\eta}\big)-\frac{A\gamma}{\mu}\zeta_{a}^2\eta_r\varrho^{3\gamma-3}V(V-U).
\end{equation*}
Then, integrating the resulting equality over $I$, we can obtain from Lemmas \ref{lemma-bound depth} and \ref{lemma-v Linfty in}, and the H\"older and Young inequalities that
\begin{equation}\label{dtv}
\begin{aligned}
\frac{\mathrm{d}}{\dt}\big|\zeta_{a}\sqrt{\eta_r}\varrho^{\gamma-1} V\big|_2^2&\leq C_0|\varrho|_\infty^{\gamma-1} \big|\zeta_{\frac{1+3a}{4}}V\big|_\infty \big|\zeta_{a}\sqrt{\eta_r}\varrho^{\gamma-1} V\big|_2\Big|\zeta_{a}\sqrt{\eta_r}\big(D_\eta U,\frac{U}{\eta}\big)\Big|_2\\
&\quad + C_0\big(\big|\zeta_{a}\sqrt{\eta_r}\varrho^{\gamma-1} V\big|_2+ |\varrho|_\infty^{\gamma-1} \big|\zeta_{a}\sqrt{\eta_r}U\big|_2\big)|\varrho|_\infty^{\gamma-1}\big|\zeta_{a}\sqrt{\eta_r}\varrho^{\gamma-1} V\big|_2\\
&\leq  C(a,T)\big|\zeta_{a}\sqrt{\eta_r} (U,\varrho^{\gamma-1}V)\big|_2^2+\frac{\mu}{8}\Big|\zeta_{a}\sqrt{\eta_r}\big(D_\eta U,\frac{U}{\eta}\big)\Big|_2^2.
\end{aligned}    
\end{equation}

\smallskip
\textbf{3.} Therefore, combining \eqref{dtu}--\eqref{dtv} leads to
\begin{align*}
&\frac{\mathrm{d}}{\mathrm{d}t}\big(\big|\zeta_{a}\sqrt{\eta_r}U\big|_2^2+\big|\zeta_{a}\sqrt{\eta_r} \varrho^{\gamma-1} V \big|_2^2\big)+ \frac{\mu}{8}\Big|\zeta_{a}\sqrt{\eta_r}\big(D_\eta U,\frac{U}{\eta}\big)\Big|_2^2\\
&\leq C(a,T)\big(\big|\zeta_{a}\sqrt{\eta_r}U\big|_2^2+\big|\zeta_{a}\sqrt{\eta_r} \varrho^{\gamma-1} V \big|_2^2+1\big),
\end{align*}
which, along with the Gr\"onwall inequality, yields that, for all $a\in(0,1)$ and $t\in [0,T]$,
\begin{equation*}
\begin{aligned}
&\big|\zeta_{a}\sqrt{\eta_r}U\big|_2^2+\big|\zeta_{a}\sqrt{\eta_r} \varrho^{\gamma-1} V \big|_2^2+ \int_0^t\Big|\zeta_{a}\sqrt{\eta_r}\big(D_\eta U,\frac{U}{\eta}\big)\Big|_2^2\,\mathrm{d}s\\
&\leq C(a,T)\big(\big|\zeta_{a}u_0\big|_2^2+\big|\zeta_{a}\rho_0^{\gamma-1}v_0\big|_2^2+1\big) \leq C(a,T).
\end{aligned}    
\end{equation*}
Here, to derive the $L^2$-bound of $\zeta_{a}(u_0,\rho_0^{\gamma-1}v_0)$, it suffices to check that
\begin{equation*}
|\zeta_{a}\rho_0^{\gamma-1}(\log\rho_0)_r|_2\leq |\zeta_{a}\rho_0^{\gamma-1-\beta}(\rho_0^\beta)_r|_2\leq C(a)|(\rho_0^\beta)_r|_\infty\leq C(a).
\end{equation*}

Finally, it follows from the above, Lemmas \ref{lemma-lower bound jacobi} and \ref{sobolev-embedding}, and the H\"older and Young inequalities that, for all $a\in(0,1)$ and  $t\in[0,T]$,
\begin{equation*}
\begin{aligned}
\int_0^t |\zeta_{a} U|_\infty^2\,\mathrm{d}s &\leq C_0\int_0^t\big|\zeta_{a}U((\zeta_{a})_rU,\zeta_{a} U_r)\big|_1 \,\mathrm{d}s\\
&\leq C(a)\int_0^t \big(\big|\zeta_{\frac{1+3a}{4}}\sqrt{\eta_r}U\big|_2^2+ |\zeta_{a}\sqrt{\eta_r}U|_2|\zeta_{a}\sqrt{\eta_r}D_\eta U|_2\big)\,\mathrm{d}s\\
&\leq C(a,T)+\int_0^t|\zeta_{a}\sqrt{\eta_r}D_\eta U|_2^2\,\mathrm{d}s.
\end{aligned}
\end{equation*}

This completes the proof of Lemma \ref{lemma-u infty}.
\end{proof}

We now establish the global uniform upper bounds of $(\eta_r,\frac{\eta}{r})$ in $[0,T]\times [0,a]$ for $a\in (0,1)$.
\begin{Lemma}\label{lemma-upper jacobi near}
For any $a\in (0,1)$, there exists a constant $C(a,T)>0$ such that
\begin{equation*}
\chi_a\frac{\eta(t,r)}{r}+\chi_a\eta_r(t,r)\leq C(a,T) \qquad \text{for all $(t,r)\in [0,T]\times \bar I$},
\end{equation*}
where  $\chi_a$ denotes the characteristic function on $[0,a]$ 
{\rm  (}see {\rm \S \ref{othernotation})}.

\end{Lemma}
\begin{proof}
We divide the proof of this lemma into three steps.

\smallskip
\textbf{1.} First, multiplying \eqref{eq:v} by $2\chi_{a} \eta_r V$ ($a\in(0,1)$) and integrating the resulting equality over $I$, together with Lemmas \ref{lemma-bound depth}, \ref{lemma-v Linfty in}, and \ref{lemma-u infty}, and the H\"older inequality, give
\begin{equation*}
\begin{aligned}
&\,\frac{\mathrm{d}}{\dt}|\chi_{a}\sqrt{\eta_r}V|_2^2+\frac{A\gamma}{\mu}\big|\chi_{a}\sqrt{\eta_r}\varrho^\frac{\gamma-1}{2}V\big|_2^2=\int_0^a V^2U_r\,\mathrm{d}r+\frac{A\gamma}{\mu}\int_0^a \eta_r \varrho^{\gamma-1} UV\,\mathrm{d}r\\
&\leq |\zeta_{a} V|_\infty |\chi_{a}\sqrt{\eta_r}V|_2|\zeta_{a} \sqrt{\eta_r} D_\eta U|_2+\frac{A\gamma}{\mu}|\varrho|_\infty^{\gamma-1}\big|\chi_{a}\sqrt{\eta_r} V\big|_2|\zeta_{a}\sqrt{\eta_r}U|_2,
\end{aligned}
\end{equation*}
which, along with the Young inequality, leads to
\begin{equation*}
\frac{\mathrm{d}}{\dt}|\chi_{a}\sqrt{\eta_r}V|_2^2 \leq C(a,T)\big(|\chi_{a}\sqrt{\eta_r}V|_2^2+|\zeta_{a} \sqrt{\eta_r} D_\eta U|_2^2+1\big).
\end{equation*}
Then applying the Gr\"onwall inequality to the above and using Lemma \ref{lemma-u infty}, we obtain that, for any $t\in [0,T]$ and $a\in (0,1)$,
\begin{equation}\label{435}
|\chi_{a}\sqrt{\eta_r}V(t)|_2\leq C(a,T)(|\chi_{a}v_0|_2+1)\leq C(a,T).
\end{equation}
Here, the initial effective velocity is controlled by
\begin{equation*}
|\chi_a(\log\rho_0)_r|_2
\leq C_0|\chi_a \rho_0^{-\beta} (\rho_0^\beta)_r |_2
\leq C(a)|(\rho_0^\beta)_r|_\infty\leq C(a).
\end{equation*}

\smallskip
\textbf{2. Uniform boundedness of $\log\varrho$ near the origin.} First, it follows from \eqref{v-expression}, \eqref{435}, Lemmas \ref{lemma-lower bound jacobi}, \ref{lemma-u infty}, and \ref{sobolev-embedding}, and the H\"older and Young inequalities that
\begin{equation}\label{log h-pre}
\begin{aligned}
|\chi_{a}\log \varrho|_\infty^2&\leq C(a)|\chi_{a}\log \varrho|_2^2+ C(a)|\chi_{a}(\log \varrho)(\log \varrho)_r|_1 \\
&\leq C(a)|\chi_{a}\log \varrho|_2^2+C(a)|\chi_{a}\sqrt{\eta_r}\log \varrho|_2|\chi_{a}\sqrt{\eta_r}D_\eta\log \varrho|_2\\
&\leq C(a)(|\eta_r^{-1}|_\infty+1) |\chi_{a}\sqrt{\eta_r}\log \varrho|_2^2+\big|\sqrt{\eta_r}(\chi_{a}V,\zeta_{a}U)\big|_2^2\\
&\leq C(a,T)(|\chi_{a}\sqrt{\eta_r}\log \varrho|_2^2+1).
\end{aligned}
\end{equation}

Next, multiplying $\eqref{eq:VFBP-La}_1$ by $2\chi_{a}\eta_r\varrho^{-1}\log \varrho$ and integrating the resulting equality over $I$, we obtain from the above, and the H\"older and Young inequalities that
\begin{equation*}
\begin{aligned}
\frac{\mathrm{d}}{\dt}|\chi_{a}\sqrt{\eta_r}\log \varrho|_2^2&=\int_0^a  (\log \varrho)^2\eta_r D_\eta U\,\mathrm{d}r-2\int_0^a \eta_r(\log \varrho)\big(D_\eta U+\frac{mU}{\eta}\big)\,\mathrm{d}r\\
&\leq C(a,T)(|\chi_{a} \log \varrho|_\infty+1) |\chi_{a}\sqrt{\eta_r}\log \varrho|_2\Big|\zeta_{a}\sqrt{\eta_r}\big(D_\eta U,\frac{U}{\eta}\big)\Big|_2 \\
&\leq C(a,T)(|\chi_{a}\sqrt{\eta_r}\log \varrho|_2^2+1)\Big|\zeta_{a}\sqrt{\eta_r}\big(D_\eta U,\frac{U}{\eta}\big)\Big|_2,    
\end{aligned}
\end{equation*}
which, along with Lemma \ref{lemma-u infty}, the fact that $\rho_0^\beta\sim 1-r$, and the Gr\"onwall and H\"older inequalities, yields that, for all $t\in [0,T]$ and $a\in (0,1)$,
\begin{equation}\label{log h-L2}
|\chi_{a}\sqrt{\eta_r}\log \varrho(t)|_2\leq C(a,T)(|\chi_{a}\log \rho_0|_2+1)\leq C(a,T). 
\end{equation}
This, together with \eqref{log h-pre}, also yields that, for all $t\in [0,T]$ and $a\in (0,1)$,
\begin{equation}\label{log h-Linfty}
|\chi_{a}\log \varrho(t)|_\infty\leq C(a,T).
\end{equation}

\smallskip
\textbf{3. Upper bounds for $(\eta_r,\frac{\eta}{r})$ in $[0,T]\times [0,a]$.} Note that \eqref{log h-Linfty}, together with \eqref{eq:eta}, implies that, for all $(t,r)\in [0,T]\times [0,a]$ and $a\in (0,1)$,
\begin{equation}\label{439}
\frac{r^m\rho_0(r)}{(\eta^m\eta_r)(t,r)}\geq C(a,T)^{-1} \implies (\eta^m\eta_r)(t,r) \leq C(a,T)r^m.
\end{equation}
Therefore, it follows from \eqref{439} and Lemma \ref{lemma-lower bound jacobi} that, for all $(t,r)\in [0,T]\times [0,a]$,
\begin{equation*}
\frac{\eta(t,r)}{r} \leq \Big(\frac{C(a,T)}{\eta_r(t,r)}\Big)^\frac{1}{m} \leq C(a,T),\qquad \eta_r(t,r)\leq  \frac{C(a,T)r^m}{\eta^m(t,r)}\leq C(a,T).
\end{equation*}

This completes the proof of Lemma \ref{lemma-upper jacobi near}.
\end{proof}

Combining Lemma \ref{lemma-upper jacobi near} with Lemma \ref{lemma-upper jacobi-ex}, we obtain, for all $(t,r)\in[0,T]\times \bar I$,
\begin{equation*}
\frac{\eta(t,r)}{r}+\eta_r(t,r)=\chi_{a_T}\big(\frac{\eta(t,r)}{r}+ \eta_r(t,r)\big)+\chi_{a_T}^\sharp\big(\frac{\eta(t,r)}{r}+ \eta_r(t,r)\big)\leq C(T),
\end{equation*} 
which thus provides the global uniform upper bounds of $(\eta_r,\frac{\eta}{r})$ in $[0,T]\times \bar I$:
\begin{Lemma}\label{lemma-upper jacobi}
There exists a constant $C(T)>0$ such that
\begin{equation*}
\frac{\eta(t,r)}{r}+\eta_r(t,r)\leq C(T) \qquad \text{for all $(t,r)\in[0,T]\times \bar I$}.
\end{equation*}
\end{Lemma}

\smallskip
\section{Non-Formation of Vacuum States Inside the Fluids  in Finite Time}\label{Section-densitylower}

The purpose of this section is to show that there is no vacuum formation inside the fluids in finite time.

\begin{Lemma}\label{non-vac}
There exists a constant $C(T)>1$ such that
\begin{equation*}
C(T)^{-1}\leq \frac{\varrho(t,r)}{\rho_0(r)}\leq C(T) \qquad\,\, \text{for all $(t,r)\in[0,T]\times \bar I$}.
\end{equation*}
In particular, no vacuum state will occur in $[0,T]\times [0,a]$ for any $T>0$ and $a\in (0,1)$. 
\end{Lemma}
\begin{proof}
It follows from \eqref{eq:eta} and Lemmas \ref{lemma-lower bound jacobi} and \ref{lemma-upper jacobi} that, for all $(t,r)\in[0,T]\times \bar I$,
\begin{equation*}
\frac{\varrho(t,r)}{\rho_0(r)}=\frac{r^m}{(\eta^m\eta_r)(t,r)}\in \big[C(T)^{-1},C(T)\big].
\end{equation*}
At $r=1$, where both $\varrho$ and $\rho_0$ vanish, the quotient is
understood through the continuous expression
$\frac{r^m}{\eta^m\eta_r}$.
Moreover, since $\rho_0^\beta\sim 1-r$, the above statement implies 
\begin{equation*}
\varrho(t,r)>0 \ \ \text{in $[0,T]\times [0,a]$}\qquad \text{for any $T>0$ and $a\in (0,1).$}
\end{equation*}
This completes the proof.
\end{proof}

\section{Global-In-Time Uniform Weighted Energy Estimates on the Velocity}\label{Section-globalestimates}
The purpose of this section is to establish the global-in-time uniform weighted energy estimates on the velocity. Let $(U,\eta)(t,r)$ be the classical solution of {\rm\textbf{IBVP}} \eqref{eq:VFBP-La-eta} in $[0,T_*]\times \bar I $ for some $T_*>0$, as obtained in Theorem \ref{local-Theorem1.1}; let $\overline T_*$ denote the maximal existence time of such a classical solution, and let $0<T<\overline T_*$ be arbitrary. Then density $\varrho$ can be given by \eqref{eq:eta}, and $(\varrho,U,\eta)(t,r)$ still solves  problem  \eqref{eq:VFBP-La} in $[0,T]\times \bar I $. Clearly, this solution is in the class $D(T)$ as defined in Definition \ref{d2}, and hence the estimates obtained in \S \ref{Section-densityupper}--\S\ref{Section-densitylower} hold for $(U,\eta)$. Moreover, throughout this section, we always assume that $(\beta,\gamma)$ satisfy
\begin{equation*}
\beta\in\big(\frac{1}{3},\gamma-1\big],\qquad\, \gamma\in \big(\frac{4}{3},\infty\big) \ \ \text{if $n=2$}, \qquad\ \gamma\in \big(\frac{4}{3},3\big) \ \ \text{if $n=3$}.
\end{equation*}

Before establishing the uniform estimates for $U$, we first give the following lemma, which can be seen as a variant of the classical div-curl estimates for spherically symmetric functions. This lemma will be frequently used in the later analysis.
\begin{Lemma}\label{im-1}
Let $\boldsymbol{f}(\boldsymbol{y})=f(r) \frac{\boldsymbol{y}}{r}\in C_{\mathrm{c}}^\infty(B_1)$ with $B_1:=\{\boldsymbol{y}:|\boldsymbol{y}|<1\}$. For any $a\in (0,1)$, the following estimates hold{\rm :}
\begin{equation*}
\begin{aligned}
&\Big|\zeta_a r^\frac{m}{2}\big(D_\eta f,\frac{f}{\eta}\big)\Big|_2\sim\Big|\zeta_a r^\frac{m}{2}\big(D_\eta f+ \frac{mf}{\eta}\big)\Big|_2,\\
&\Big|\zeta_a r^\frac{m}{2}\Big(D_\eta^2 f, D_\eta\big(\frac{f}{\eta}\big)\Big)\Big|_2\sim\Big|\zeta_a r^\frac{m}{2}D_\eta\big(D_\eta f+ \frac{mf}{\eta}\big)\Big|_2,\\
&\Big|\zeta_a r^\frac{m}{2}\Big(D_\eta^3 f,D_\eta^2\big(\frac{f}{\eta}\big), \frac{1}{\eta}D_\eta\big(\frac{f}{\eta}\big)\Big)\Big|_2\sim \Big|\zeta_a r^\frac{m}{2}\Big(D_\eta^2\big(D_\eta f+ \frac{mf}{\eta}\big),\frac{1}{\eta}D_\eta\big(D_\eta f+ \frac{m f}{\eta}\big)\Big) \Big|_2,\\
&\Big|\zeta_a r^\frac{m}{2}\Big(D_\eta^4 f,D_\eta^3\big(\frac{f}{\eta}\big), D_\eta\big(\frac{1}{\eta}D_\eta(\frac{f}{\eta})\big)\Big)\Big|_2
\sim\Big|\zeta_a r^\frac{m}{2}\Big(D_\eta^3\big(D_\eta f+ \frac{mf}{\eta}\big),D_\eta\big(\frac{1}{\eta}D_\eta(D_\eta f+ \frac{m f}{\eta})\big)\Big) \Big|_2,
\end{aligned}
\end{equation*}
where $F_1\sim F_2$ denotes that there exists a constant $C(T)\geq 1$ depending only on $(C_0,T)$ such that $C(T)^{-1}F_1\leq F_2\leq C(T)F_1$.
\end{Lemma}
\begin{proof}
For simplicity, we only give the proof for $a=\frac{1}{2}$ and show that the quantities on the left-hand sides can be controlled by those on the right-hand sides. Besides, the following fact will be used frequently later:
\begin{equation*}
D_\eta\zeta\leq 0 \qquad \text{for all $(t,r)\in [0,T]\times\bar I$}. 
\end{equation*}
We divide the proof into four steps. 

\smallskip
\textbf{1.} It follows from integration by parts that
\begin{equation*}
\begin{aligned}
\Big|\zeta(\eta^m\eta_r)^\frac{1}{2}\big(D_\eta f+ \frac{mf}{\eta}\big)\Big|_2^2&=\int_0^1 \zeta^2\eta^m\eta_r\Big(|D_\eta f|^2+\frac{m^2f^2}{\eta^2}+\frac{2m}{\eta}fD_\eta f\Big)\,\mathrm{d}r\\
&=\int_0^1 \zeta^2\eta^m\eta_r\Big(|D_\eta f|^2+\frac{mf^2}{\eta^2}\Big)\,\mathrm{d}r-2m\int_0^1 \zeta D_\eta\zeta \eta^{m-1}\eta_r f^2\,\mathrm{d}r\\
&\geq \int_0^1 \zeta^2\eta^m\eta_r\Big(|D_\eta f|^2+\frac{mf^2}{\eta^2}\Big)\,\mathrm{d}r,
\end{aligned}
\end{equation*}
which, along with Lemmas \ref{lemma-lower bound jacobi} and \ref{lemma-upper jacobi}, leads to
\begin{equation*}
\Big|\zeta r^\frac{m}{2}\big(D_\eta f,\frac{f}{\eta}\big)\Big|_2\leq C(T)\Big|\zeta r^\frac{m}{2}\big(D_\eta f+ \frac{mf}{\eta}\big)\Big|_2.
\end{equation*}

\smallskip
\textbf{2.} First, a direct calculation yields 
\begin{equation}\label{di1}
\begin{aligned}
&\,\Big|\zeta (\eta^m\eta_r)^\frac{1}{2} D_\eta\big(D_\eta f+ \frac{mf}{\eta}\big)\Big|_2^2\\
&=\int_0^1 \zeta^2 \eta^m \eta_r \Big(|D_\eta^2 f|^2+m^2 \Big|D_\eta\big(\frac{f}{\eta}\big)\Big|^2\Big)\,\mathrm{d}r +\underline{2m \int_0^1\zeta^2 \eta^m \eta_r D_\eta^2 f D_\eta\big(\frac{f}{\eta}\big) \,\mathrm{d}r}_{:=\mathrm{L}_3}.
\end{aligned}
\end{equation}

Next, by chain rules, we have
\begin{equation}\label{hengdengshi}
\frac{D_\eta^2 f}{\eta}=D_\eta^2\big(\frac{f}{\eta}\big)+\frac{2}{\eta}D_\eta\big(\frac{f}{\eta}\big).
\end{equation}
Then $\mathrm{L}_3$ can be handled by using \eqref{hengdengshi} and integration by parts:
\begin{equation*}
\begin{aligned}
\mathrm{L}_3&=4m\int_0^1 \zeta^2 \eta^m \eta_r  \Big|D_\eta\big(\frac{f}{\eta}\big)\Big|^2\,\mathrm{d}r +2m\int_0^1 \zeta^2 \eta^{m+1} \eta_rD_\eta^2\big(\frac{f}{\eta}\big)D_\eta\big(\frac{f}{\eta}\big) \,\mathrm{d}r\\
&=(3m-m^2)\int_0^1 \zeta^2 \eta^m \eta_r \Big|D_\eta\big(\frac{f}{\eta}\big)\Big|^2 \,\mathrm{d}r-2m\int_0^1 \zeta D_\eta\zeta \eta^{m+1} \eta_r \Big|D_\eta\big(\frac{f}{\eta}\big)\Big|^2\,\mathrm{d}r\geq 0.
\end{aligned}
\end{equation*}

Hence, substituting the above into \eqref{di1} leads to
\begin{equation*}
\Big|\zeta (\eta^m\eta_r)^\frac{1}{2} D_\eta\big(D_\eta f+ \frac{mf}{\eta}\big)\Big|_2^2\geq \int_0^1 \zeta^2 \eta^m \eta_r \Big(|D_\eta^2 f|^2+m^2 \Big|D_\eta\big(\frac{f}{\eta}\big)\Big|^2\Big)\,\mathrm{d}r, 
\end{equation*}
which, along with Lemmas \ref{lemma-lower bound jacobi} and \ref{lemma-upper jacobi}, yields
\begin{equation*}
\Big|\zeta r^\frac{m}{2}\Big(D_\eta^2 f, D_\eta\big(\frac{f}{\eta}\big)\Big)\Big|_2\leq C(T)\Big|\zeta r^\frac{m}{2}D_\eta\big(D_\eta f+ \frac{mf}{\eta}\big)\Big|_2.
\end{equation*}

\smallskip
\textbf{3.} First, we have
\begin{equation}\label{L4}
\begin{aligned}
&\,\Big|\zeta (\eta^m\eta_r)^\frac{1}{2} D_\eta^2\big(D_\eta f+ \frac{mf}{\eta}\big)\Big|_2^2\\
&=\int_0^1 \zeta^2 \eta^m \eta_r \Big(|D_\eta^3 f|^2 +m^2 \Big|D_\eta^2\big(\frac{f}{\eta}\big)\Big|^2\Big) \,\mathrm{d}r+\underline{2m\int_0^1 \zeta^2\eta^m \eta_r D_\eta^3f D_\eta^2\big(\frac{f}{\eta}\big)\,\mathrm{d}r}_{:=\mathrm{L}_4},
\end{aligned}
\end{equation}
and we can obtain from \eqref{hengdengshi} that the following identity holds:
\begin{equation}\label{hengdengshi2}
D_\eta^3 f= \eta D_\eta^3\big(\frac{f}{\eta}\big)+ 3 D_\eta^2\big(\frac{f}{\eta}\big).
\end{equation}
Hence, for $\mathrm{L}_4$, it follows from the above and integration by parts that
\begin{equation*} 
\begin{aligned}
\mathrm{L}_4&=6m\int_0^1 \zeta^2 \eta^m \eta_r    \Big|D_\eta^2\big(\frac{f}{\eta}\big)\Big|^2 \,\mathrm{d}r  +2m\int_0^1 \zeta^2\eta^{m+1}\eta_r D_\eta^3\big(\frac{f}{\eta}\big)  D_\eta^2\big(\frac{f}{\eta}\big)\,\mathrm{d}r\\
&=(5m-m^2)\int_0^1 \zeta^2 \eta^m \eta_r  \Big|D_\eta^2\big(\frac{f}{\eta}\big)\Big|^2 \,\mathrm{d}r -2m\int_0^1 \zeta D_\eta \zeta \eta^{m+1} \eta_r  \Big|D_\eta^2 \big(\frac{f}{\eta}\big)\Big|^2\, \mathrm{d}r\geq 0,
\end{aligned}
\end{equation*}
which, combined with \eqref{L4}, leads to
\begin{equation}\label{L4'}
\Big|\zeta (\eta^m\eta_r)^\frac{1}{2} D_\eta^2\big(D_\eta f+ \frac{mf}{\eta}\big)\Big|_2^2\geq \int_0^1 \zeta^2 \eta^m \eta_r \Big(|D_\eta^3 f|^2 +m^2 \Big|D_\eta^2\big(\frac{f}{\eta}\big)\Big|^2\Big) \,\mathrm{d}r.
\end{equation}

Next, it follows from \eqref{hengdengshi} that
\begin{equation}\label{L5}
\begin{aligned}
&\,\Big|\zeta (\eta^m\eta_r)^\frac{1}{2} \frac{1}{\eta}D_\eta\Big(D_\eta f+ \frac{m f}{\eta}\Big)\Big|_2^2=\Big|\zeta(\eta^m\eta_r)^\frac{1}{2}\Big(D_\eta^2\big(\frac{f}{\eta}\big)+(m+2)\frac{1}{\eta}D_\eta\big(\frac{f}{\eta}\big)\Big)\Big|_2^2\\
&=\int_0^1 \zeta^2 \eta^{m} \eta_r \Big(\Big|D_\eta^2 \big(\frac{f}{\eta}\big)\Big|^2 +(m+2)^2 \Big|\frac{1}{\eta}D_\eta\big(\frac{f}{\eta}\big)\Big|^2\Big)\,\mathrm{d}r\\
&\quad +\underline{2(m+2)\int_0^1\zeta^2 \eta^{m-1} \eta_r  D_\eta^2\big(\frac{f}{\eta}\big)  D_\eta\big(\frac{f}{\eta}\big)\,\mathrm{d}r}_{:=\mathrm{L}_5},     
\end{aligned}
\end{equation}
where $\mathrm{L}_5$ can be handled by integration by parts:
\begin{equation*}
\begin{aligned}
\mathrm{L}_5&=(2-m-m^2)\int_0^1 \zeta^2 \eta^{m} \eta_r \Big|\frac{1}{\eta}D_\eta\big(\frac{f}{\eta}\big)\Big|^2\,\mathrm{d}r -2(m+2)\int_0^1 \zeta D_\eta\zeta \eta^{m-1} \eta_r \Big|D_\eta\big(\frac{f}{\eta}\big)\Big|^2 \,\mathrm{d}r \\
&\geq (2-m-m^2)\int_0^1 \zeta^2 \eta^{m} \eta_r \Big|\frac{1}{\eta}D_\eta\big(\frac{f}{\eta}\big)\Big|^2\,\mathrm{d}r. 
\end{aligned}    
\end{equation*}
Hence, \eqref{L5}, together with the above, implies that
\begin{equation*}
\Big|\zeta (\eta^m\eta_r)^\frac{1}{2} \frac{1}{\eta}D_\eta\Big(D_\eta f+ \frac{m f}{\eta}\Big)\Big|_2^2\geq \int_0^1 \zeta^2 \eta^{m} \eta_r \Big(\Big|D_\eta^2 \big(\frac{f}{\eta}\big)\Big|^2 +(3m+6)\Big|\frac{1}{\eta}D_\eta\big(\frac{f}{\eta}\big)\Big|^2\Big)\,\mathrm{d}r, 
\end{equation*}
which, along with \eqref{L4'} and Lemmas \ref{lemma-lower bound jacobi} and \ref{lemma-upper jacobi}, leads to
\begin{equation*}
\Big|\zeta r^\frac{m}{2}\Big(D_\eta^3 f,D_\eta^2\big(\frac{f}{\eta}\big), \frac{1}{\eta}D_\eta\big(\frac{f}{\eta}\big)\Big)\Big|_2 \leq C(T)\Big|\zeta r^\frac{m}{2}\Big(D_\eta^2\big(D_\eta f+ \frac{mf}{\eta}\big),\frac{1}{\eta}D_\eta\big(D_\eta f+ \frac{m f}{\eta}\big)\Big) \Big|_2.
\end{equation*}

\smallskip
\textbf{4.} First, a direct calculation gives
\begin{equation}\label{L6}
\begin{aligned}
&\,\Big|\zeta (\eta^m\eta_r)^\frac{1}{2} D_\eta^3\Big(D_\eta f+ \frac{mf}{\eta}\Big)\Big|_2^2\\
&=\int_0^1 \zeta^2 \eta^m \eta_r \Big(|D_\eta^4 f|^2 +m^2 \Big|D_\eta^3\big(\frac{f}{\eta}\big)\Big|^2\Big) \,\mathrm{d}r+\underline{2m\int_0^1 \zeta^2\eta^m \eta_r D_\eta^4f D_\eta^3\big(\frac{f}{\eta}\big)\,\mathrm{d}r}_{:=\mathrm{L}_6}.
\end{aligned}
\end{equation}
Then, for $\mathrm{L}_6$, note that the following identity holds due to \eqref{hengdengshi2}:
\begin{equation*}
D_\eta^4 f=\eta D_\eta^4\big(\frac{f}{\eta}\big)+ 4 D_\eta^3\big(\frac{f}{\eta}\big).
\end{equation*}
Hence, this, together with integration by parts, yields  
\begin{equation*}
\begin{aligned}
\mathrm{L}_6&=8m\int_0^1 \zeta^2 \eta^m \eta_r \Big|D_\eta^3\big(\frac{f}{\eta}\big)\Big|^2 \,\mathrm{d}r +2m\int_0^1 \zeta^2 \eta^{m+1} \eta_r D_\eta^4\big(\frac{f}{\eta}\big) D_\eta^3\big(\frac{f}{\eta}\big)\,\mathrm{d}r\\
&=(7m-m^2)\int_0^1 \zeta^2 \eta^m \eta_r \Big|D_\eta^3\big(\frac{f}{\eta}\big)\Big|^2 \,\mathrm{d}r -2m\int_0^1 \zeta D_\eta \zeta \eta^{m+1} \eta_r  \Big|D_\eta^3 \big(\frac{f}{\eta}\big)\Big|^2\, \mathrm{d}r\geq 0,    
\end{aligned}
\end{equation*}
which, along with \eqref{L6}, gives
\begin{equation}\label{L6'}
\Big|\zeta (\eta^m\eta_r)^\frac{1}{2} D_\eta^3\Big(D_\eta f+ \frac{mf}{\eta}\Big)\Big|_2^2\geq \int_0^1 \zeta^2 \eta^m \eta_r \Big(|D_\eta^4 f|^2 +m^2 \Big|D_\eta^3\big(\frac{f}{\eta}\big)\Big|^2\Big) \,\mathrm{d}r.
\end{equation}

Next, we note that, by \eqref{hengdengshi},
\begin{equation*}
D_\eta\big(\frac{D_\eta^2 f}{\eta}\big)=D_\eta^3\big(\frac{f}{\eta}\big)+2D_\eta\Big(\frac{1}{\eta}D_\eta\big(\frac{f}{\eta}\big)\Big)=\eta D_\eta^2\Big(\frac{1}{\eta}D_\eta\big(\frac{f}{\eta}\big)\Big)+4D_\eta\Big(\frac{1}{\eta}D_\eta\big(\frac{f}{\eta}\big)\Big).
\end{equation*}
Hence, this, together with a direct calculation, implies
\begin{equation}\label{9.6}
\begin{aligned}
&\,\Big|\zeta (\eta^m\eta_r)^\frac{1}{2} D_\eta\Big(\frac{1}{\eta}D_\eta\big(D_\eta f+ \frac{m f}{\eta}\big)\Big)\Big|_2^2\\
&=\Big|\zeta(\eta^m\eta_r)^\frac{1}{2}\Big(\eta D_\eta^2\Big(\frac{1}{\eta}D_\eta\big(\frac{f}{\eta}\big)\Big)+(m+4)D_\eta\Big(\frac{1}{\eta}D_\eta\big(\frac{f}{\eta}\big)\Big)\Big)\Big|_2^2\\
&=\int_0^1 \zeta^2 \eta^{m} \eta_r \Big(\Big|\eta D_\eta^2\Big(\frac{1}{\eta}D_\eta\big(\frac{f}{\eta}\big)\Big)\Big|^2 +(m+4)^2 \Big|D_\eta\Big(\frac{1}{\eta}D_\eta\big(\frac{f}{\eta}\big)\Big)\Big|^2\Big)\,\mathrm{d}r\\
&\quad +\underline{2(m+4)\int_0^1 \zeta^2\eta^{m+1}\eta_r D_\eta^2\Big(\frac{1}{\eta}D_\eta\big(\frac{f}{\eta}\big)\Big)D_\eta\Big(\frac{1}{\eta}D_\eta\big(\frac{f}{\eta}\big)\Big)\,\mathrm{d}r}_{:=\mathrm{L}_7},
\end{aligned}
\end{equation}
where $\mathrm{L}_7$ can be handled by integration by parts:
\begin{equation*}
\begin{aligned}
\mathrm{L}_7&=-(m^2+5m+4)\int_0^1 \zeta^2\eta^m\eta_r\Big|D_\eta\Big(\frac{1}{\eta}D_\eta\big(\frac{f}{\eta}\big)\Big)\Big|^2\,\mathrm{d}r\\
&\quad -2(m+4)\int_0^1 \zeta D_\eta\zeta \eta^{m+1} \eta_r \Big|D_\eta\Big(\frac{1}{\eta}D_\eta\big(\frac{f}{\eta}\big)\Big)\Big|^2 \,\mathrm{d}r\\
&\geq -(m^2+5m+4) \int_0^1 \zeta^2\eta^m\eta_r\Big|D_\eta\Big(\frac{1}{\eta}D_\eta\big(\frac{f}{\eta}\big)\Big)\Big|^2\,\mathrm{d}r.    
\end{aligned}
\end{equation*}
Therefore, \eqref{9.6}, combined with the above, implies 
\begin{equation*}
\Big|\zeta (\eta^m\eta_r)^\frac{1}{2} D_\eta\Big(\frac{1}{\eta}D_\eta\big(D_\eta f+ \frac{m f}{\eta}\big)\Big)\Big|_2^2\geq (3m+12)\int_0^1 \zeta^2 \eta^{m} \eta_r   \Big|D_\eta\Big(\frac{1}{\eta}D_\eta\big(\frac{f}{\eta}\big)\Big)\Big|^2 \,\mathrm{d}r,
\end{equation*}
which, together with \eqref{L6'}  and Lemmas \ref{lemma-lower bound jacobi} and \ref{lemma-upper jacobi}, leads to
\begin{equation*}
\begin{aligned}
&\,\Big|\zeta r^\frac{m}{2}\Big(D_\eta^4 f,D_\eta^3\big(\frac{f}{\eta}\big), D_\eta\big(\frac{1}{\eta}D_\eta(\frac{f}{\eta})\big)\Big)\Big|_2\\
&\leq C(T)\Big|\zeta r^\frac{m}{2}\Big(D_\eta^3\big(D_\eta f+ \frac{mf}{\eta}\big),D_\eta\big(\frac{1}{\eta}D_\eta(D_\eta f+ \frac{m f}{\eta})\big)\Big) \Big|_2.
\end{aligned}
\end{equation*}

This completes the proof.
\end{proof}

\subsection{Tangential estimates of the velocity}
This section is devoted to establishing the tangential estimates for $U$. We first give time-spatial estimates for the velocity, which will be frequently used in the rest of this section.
\begin{Lemma}\label{lemma-time-space}
For any $\iota$ satisfying
\begin{equation*}
\iota\in \big(-\frac{\beta}{2},1+\frac{\beta}{2}\big),
\end{equation*}
and any $a\in (0,1)$, there exists a constant $C(\iota,a,T)>0$ such that
\begin{equation*}
\big|\chi^\sharp_{a}\rho_0^{\iota-\frac{\beta}{2}}D_\eta U(t)\big|_\infty\leq C(\iota,a,T)\big(1+\big|\chi^\sharp_{a}\rho_0^\iota(D_\eta U,U_t)(t)\big|_2\big) \qquad \text{for all $t\in[0,T]$}.
\end{equation*}
Besides, for any such $\iota\in (-\frac{\beta}{2},1+\frac{\beta}{2})$, $a\in (0,1)$, and $\sigma>0$, there exists a constant $C(\sigma,\iota,a,T)>0$ such that 
\begin{equation*}
\big|\chi^\sharp_{a}\rho_0^{\iota-\beta+\sigma}D_\eta U(t)\big|_2\leq C(\sigma,\iota,a,T)\big(1+\big|\chi^\sharp_{a}\rho_0^\iota(D_\eta U,U_t)(t)\big|_2\big) \qquad \text{for all $t\in[0,T]$}.
\end{equation*}
\end{Lemma}
\begin{proof}
Integrating $\eqref{eq:VFBP-La-eta}_1\times\eta_r$ over $[r,1]$ ($r\in(0,1)$) with respect to the radial coordinate gives
\begin{equation}\label{zhongyao}
D_\eta U =\frac{A}{2\mu}\varrho^{\gamma-1} +\frac{m}{\varrho}\int_r^1 \frac{\tilde{r}^m\rho_0}{\eta^m}\big(\frac{D_\eta U}{\eta}-\frac{U}{\eta^2}\big)\mathrm{d}\tilde{r}-\frac{1}{2\mu \varrho}\int_r^1\frac{\tilde{r}^m}{\eta^m}\rho_0U_t\,\mathrm{d}\tilde{r}.
\end{equation}
Hence, due to $\rho_0^\beta\sim 1-r$ and Lemmas \ref{lemma-lower bound jacobi} and \ref{non-vac}, for all $r\in [a,1)$ and $\iota\in(-\frac{\beta}{2},1+\frac{\beta}{2})$,
\begin{equation}\label{floww}
\begin{aligned}
|D_\eta U|&\leq C(T)\rho_0^{\gamma-1} +\frac{C(a,T)}{\rho_0}\int_r^1 \rho_0\left(|U|+|D_\eta U|+|U_t|\right)\mathrm{d}\tilde{r}\\
&\leq C(T)\rho_0^{\gamma-1} +\frac{C(a,T)}{\rho_0}\Big(\int_r^1 \rho_0^{2-2\iota}\,\mathrm{d}\tilde{r}\Big)^\frac{1}{2}\Big(\int_a^1 \rho_0^{2\iota}\left(|U|^2+|D_\eta U|^2+|U_t|^2\right)\mathrm{d}\tilde{r}\Big)^\frac{1}{2}\\
&\leq C(T)\rho_0^{\gamma-1} +C(\iota,a,T)\rho_0^{\frac{\beta}{2}-\iota}\big|\chi^\sharp_{a}\rho_0^\iota(U,D_\eta U,U_t)\big|_2,
\end{aligned}
\end{equation}
which, along with Lemma \ref{lemma-refine u Lp}, yields  
\begin{equation*}
\big|\chi^\sharp_{a}\rho_0^{\iota-\frac{\beta}{2}}D_\eta U\big|_\infty\leq C(\iota,T)+C(\iota,a,T)\big(1+\big|\chi^\sharp_{a}\rho_0^\iota(D_\eta U,U_t)\big|_2\big).
\end{equation*}

Moreover, using the facts that 
\begin{equation*}
\gamma-1+\iota-\beta\geq -\frac{\beta}{2},\qquad \rho_0^\beta\sim 1-r,
\end{equation*}
we also see that, for all $\sigma>0$,
\begin{equation*}
\big|\chi^\sharp_{a}\rho_0^{\iota-\beta+\sigma}D_\eta U\big|_2\leq C(\sigma,\iota,T)+C(\sigma,\iota,a,T)\big(1+\big|\chi^\sharp_{a}\rho_0^\iota(D_\eta U,U_t)\big|_2\big).
\end{equation*}

This completes the proof.
\end{proof}

\subsubsection{Zeroth-order tangential estimates}
The zeroth-order tangential estimate for 
$U$ is exactly given by Lemma \ref{lemma-basic energy}. 
For ease of reference, we also present some lower-order estimates for $U$ here, which are due to Lemmas \ref{lemma-lower bound jacobi}, \ref{lemma-refine u Lp}, and \ref{lemma-u infty}.
\begin{Lemma}\label{lemma-u-0order}
There exists a constant $C(T)>0$ such that
\begin{equation*}
\big|(r^m\rho_0)^\frac{1}{2}U(t)\big|_2^2+ \int_0^t\Big|(r^m\rho_0)^\frac{1}{2}\big(D_\eta U,\frac{U}{\eta}\big)\Big|_2^2\,\mathrm{d}s\leq C_0 \qquad \text{for all $t\in [0,T]$}.
\end{equation*}
Moreover, for any $p\in (0,\infty)$, $\iota\in (-\frac{1}{p},\infty)$, and $a\in (0,1)$, there exist two positive constants $C(p,\iota,a,T)$ and $C(a,T)$ such that, for all $t\in [0,T]$,
\begin{equation*}
\big|\chi_{a}^\sharp \rho_0^{\iota\beta} U(t)\big|_p\leq C(p,\iota,a,T),\qquad
|\zeta_{a}U(t)|_2^2+\int_0^t\Big(\Big|\zeta_{a}\big(D_\eta U,\frac{U}{\eta}\big)\Big|_2^2+|\zeta_{a} U|_\infty^2\Big)\,\mathrm{d}s\leq C(a,T).
\end{equation*}
\end{Lemma}

\subsubsection{First-order tangential estimates}\label{subsub911}
\begin{Lemma}\label{lemma-u-D1}
There exists a constant $C(T)>0$ such that 
\begin{equation*}
\Big|(r^m\rho_0)^\frac{1}{2}\big(D_\eta U,\frac{U}{\eta}\big)(t)\Big|_2^2+ \int_0^t\big|(r^m\rho_0)^\frac{1}{2}U_t\big|_2^2\,\mathrm{d}s\leq C(T) \qquad \text{for all $t\in [0,T]$}.
\end{equation*}
\end{Lemma}
\begin{proof}
We divide the proof into two steps.

\smallskip
\textbf{1.} We give some auxiliary estimates. First, $\eqref{eq:VFBP-La-eta}_1$, combined with \eqref{v-expression}, yields 
\begin{equation}\label{ell-0}
D_\eta\big(D_\eta U+ \frac{mU}{\eta}\big)=\frac{1}{2\mu} U_t-\frac{1}{2\mu}(V-U)D_\eta U +\frac{A\gamma}{4\mu^2} \varrho^{\gamma-1}(V-U).
\end{equation}
Then it follows from \eqref{ell-0}, Lemmas \ref{lemma-bound depth},  \ref{lemma-v Linfty in}, \ref{im-1}, and \ref{lemma-u-0order}, and the H\"older inequality that 
\begin{equation}\label{in-ell-2}
\begin{aligned}
&\,\Big|\zeta_{a} r^\frac{m}{2}\Big(D_\eta^2 U,D_\eta\big(\frac{U}{\eta}\big)\Big)\Big|_2\leq C(T)\Big|\zeta_{a} r^\frac{m}{2} D_\eta\big(D_\eta U+ \frac{mU}{\eta}\big)\Big|_2\\
&\leq C(T)\big(|\zeta_{a} r^\frac{m}{2}U_t|_2+  |\zeta_{a}(V,U)|_\infty\big|\zeta_{\frac{1+3a}{4}} r^\frac{m}{2} D_\eta U\big|_2+|r^\frac{m}{2}|_\infty|\varrho|_\infty^{\gamma-1} |\zeta_{a} (V,U)|_2 \big)\\
&\leq C(a,T)\big(|\zeta_{a} r^\frac{m}{2}U_t|_2+ (1+|\zeta_{a} U|_\infty)\big|\zeta_{\frac{1+3a}{4}} r^\frac{m}{2}D_\eta U\big|_2 +1\big).
\end{aligned}
\end{equation}

Next, rewrite \eqref{ell-0} as
\begin{equation}\label{ell-0-ex}
D_\eta^2 U=\frac{1}{2\mu} U_t-\frac{m}{\eta}\big(D_\eta U-\frac{U}{\eta}\big)-\frac{1}{2\mu}(V-U)D_\eta U+\frac{A\gamma}{4\mu^2} \varrho^{\gamma-1}(V-U).
\end{equation}
Multiplying the above by $\chi^\sharp \rho_0^\frac{1+\varepsilon}{2}$ ($\varepsilon>0$), we obtain from Lemmas \ref{lemma-bound depth}--\ref{lemma-lower bound jacobi}, \ref{lemma-v Lp ex}, \ref{lemma-upper jacobi}, \ref{lemma-time-space} (with $\iota=\frac{1}{2}$), and \ref{lemma-u-0order}, and the H\"older inequality that
\begin{equation}\label{ex-ell-2}
\!\!\begin{aligned}
\big|\chi^\sharp \rho_0^\frac{1+\varepsilon}{2} D_\eta^2 U\big|_2&\leq C(T)\big(\big|\chi^\sharp \rho_0^\frac{1+\varepsilon}{2}(U,D_\eta U,U_t)\big|_2+  \big|\chi^\sharp \rho_0^\frac{\beta+\varepsilon}{2}(V,U)\big|_2(\big|\chi^\sharp \rho_0^\frac{1-\beta}{2} D_\eta U\big|_\infty+1)\big)\\
&\leq C(\varepsilon,T)\big(|\rho_0^\beta|_\infty^\frac{\varepsilon}{2\beta}\big|\chi^\sharp \rho_0^\frac{1}{2}(U,D_\eta U,U_t)\big|_2+ \big|\chi^\sharp \rho_0^\frac{1-\beta}{2} D_\eta U\big|_\infty +1\big)\\
&\leq C(\varepsilon,T)\big(\big|\chi^\sharp\rho_0^\frac{1}{2}(D_\eta U,U_t)\big|_2 +1\big).
\end{aligned}
\end{equation}

\smallskip
\textbf{2.} First, multiplying $\eqref{eq:VFBP-La-eta}_1$ by $\eta^m\eta_r$, along with \eqref{v-expression}, we have
\begin{equation*}
r^m\rho_0U_t +\frac{A\gamma}{2\mu} \frac{(r^m\rho_0)^\gamma}{(\eta^m\eta_r)^{\gamma-1}} (V-U)=2\mu \Big(r^m\rho_0\frac{D_\eta U}{\eta_r} \Big)_r-2\mu m r^m\rho_0\frac{U}{\eta^2}.   
\end{equation*}
Then multiplying the above by $U_t$ gives
\begin{equation}\label{eq-1order-pre}
\begin{aligned}
&\,\mu r^m\rho_0\Big(|D_\eta U|^2+ m \frac{U^2}{\eta^2}\Big)_t+r^m\rho_0U_t^2\\
&=2\mu \Big(r^m\rho_0\frac{D_\eta U U_t}{\eta_r}\Big)_r-\frac{A\gamma}{2\mu} \frac{(r^m\rho_0)^\gamma}{(\eta^m\eta_r)^{\gamma-1}} (V-U)U_t-2\mu r^m\rho_0\Big((D_\eta U)^3+m \frac{U^3}{\eta^3}\Big).
\end{aligned}
\end{equation}
Integrating \eqref{eq-1order-pre} over $I$, we obtain 
\begin{equation}\label{dt-j1-j4}
\begin{aligned}
&\, \mu \frac{\mathrm{d}}{\dt}\Big(\big|(r^m\rho_0)^\frac{1}{2}D_\eta U\big|_2^2+m\Big|(r^m\rho_0)^\frac{1}{2}\frac{U}{\eta}\Big|_2^2\Big) + \big|(r^m\rho_0)^\frac{1}{2}U_t\big|_2^2 \\
&=-\frac{A\gamma}{2\mu}\int_0^1 \frac{(r^m\rho_0)^\gamma}{(\eta^m\eta_r)^{\gamma-1}}  (V-U) U_t\,\mathrm{d}r -2\mu \int_0^1r^m\rho_0\Big((D_\eta U)^3+m \frac{U^3}{\eta^3}\Big)\,\mathrm{d}r := \sum_{i=1}^2 \mathrm{J}_i.
\end{aligned}    
\end{equation}

\smallskip
\textbf{2.1. Estimate for $\mathrm{J}_1$.} For $\mathrm{J}_1$, it follows from Lemmas \ref{lemma-v Lp ex}, \ref{lemma-v Linfty in}, and \ref{lemma-u-0order}, and the H\"older and Young inequalities that
\begin{equation}\label{rm-J1}
\begin{aligned}
\mathrm{J}_1&=-\frac{A\gamma}{2\mu}\int_0^1 (\zeta+\zeta^\sharp) \frac{(r^m\rho_0)^\gamma}{(\eta^m\eta_r)^{\gamma-1}}  (V-U) U_t\,\mathrm{d}r\\
&\leq C_0\big(|\zeta(V,U)|_2|\rho_0^\beta|_\infty^{\frac{2\gamma-1}{2\beta}}+ |\zeta^\sharp \rho_0^\beta (V,U)|_2|\rho_0^\beta|_\infty^{\frac{2\gamma-1-2\beta}{2\beta}}\big)\Big|\frac{r^{(\gamma-\frac{1}{2})m}}{(\eta^m\eta_r)^{\gamma-1}}\Big|_\infty\big|(r^m\rho_0)^\frac{1}{2}U_t\big|_2\\
&\leq C(T)(|\zeta V|_\infty+1)\big|(r^m\rho_0)^\frac{1}{2}U_t\big|_2 \leq C(T)+\frac{1}{8}\big|(r^m\rho_0)^\frac{1}{2}U_t\big|_2^2.
\end{aligned}    
\end{equation}

\smallskip
\textbf{2.2. Estimate for $\mathrm{J}_2$.} To estimate $\mathrm{J}_2$, we first divide it into three parts:
\begin{equation}\label{rm-J2}
\begin{aligned}
\mathrm{J}_2&=-2\mu \int_0^1 \zeta r^m\rho_0\Big((D_\eta U)^3+m \frac{U^3}{\eta^3}\Big)\,\mathrm{d}r-2\mu \int_0^1 \zeta^\sharp r^m\rho_0 (D_\eta U)^3 \,\mathrm{d}r\\
&\quad -2\mu m\int_0^1 \zeta^\sharp r^m\rho_0 \frac{U^3}{\eta^3} 
\mathrm{d}r:=\sum_{i=1}^3\mathrm{J}_{2,i}.
\end{aligned}
\end{equation}

For $\mathrm{J}_{2,1}$, note that, due to $\frac{2m+1}{4}\geq \frac{m}{2}$ and Lemmas \ref{lemma-upper jacobi} and \ref{GNinequality}, for any $a\in (0,1)$, 
\begin{equation}\label{DU-L4}
\Big|\chi_{a} r^\frac{m}{2}\big(D_\eta U,\frac{U}{\eta}\big)\Big|_4^2\leq C(a)\Big|\zeta_{a} r^\frac{m}{2}\big(D_\eta U,\frac{U}{\eta}\big)\Big|_2 \Big|\zeta_{a} r^\frac{m}{2}\big(D_\eta U,\frac{U}{\eta},D_\eta^2 U, D_\eta\big(\frac{U}{\eta}\big)\Big)\Big|_2.
\end{equation}
Hence, due to $\rho_0^\beta\sim 1-r$, \eqref{in-ell-2}, and the H\"older and Young inequalities, we have
\begin{equation}\label{jj1}
\begin{aligned}
\mathrm{J}_{2,1} &\leq C_0\Big|\zeta\big(D_\eta U,\frac{U}{\eta}\big)\Big|_2\Big|\chi_{\frac{5}{8}} r^\frac{m}{2}\big(D_\eta U,\frac{U}{\eta}\big)\Big|_4^2\\
&\leq C(T)\Big|\zeta\big(D_\eta U,\frac{U}{\eta}\big)\Big|_2\Big|\zeta_{\frac{5}{8}} r^\frac{m}{2}\big(D_\eta U,\frac{U}{\eta}\big)\Big|_2 \Big|\zeta_{\frac{5}{8}} r^\frac{m}{2}\big(D_\eta U,\frac{U}{\eta},U_t\big)\Big|_2 \\
&\quad + C(T)\Big|\zeta\big(D_\eta U,\frac{U}{\eta}\big)\Big|_2\Big|\zeta_{\frac{5}{8}} r^\frac{m}{2}\big(D_\eta U,\frac{U}{\eta}\big)\Big|_2  \big((1+|\zeta_{\frac{5}{8}} U|_\infty)\big|\zeta_{\frac{3}{4}} r^\frac{m}{2}D_\eta U\big|_2 +1\big)\\
&\leq C(T)\Big(1\!+\!|\zeta_{\frac{5}{8}} U|_\infty^2\!+\!\Big|\zeta\big(D_\eta U,\frac{U}{\eta}\big)\Big|_2^2\Big)\Big(1\!+\!\Big|(r^m\rho_0)^\frac{1}{2}\big(D_\eta U,\frac{U}{\eta}\big)\Big|_2^2\Big)\!+\!\frac{1}{8}\big|(r^m\rho_0)^\frac{1}{2}U_t\big|_2^2.    
\end{aligned}
\end{equation}

For $\mathrm{J}_{2,2}$, it follows from integration by parts, \eqref{eq:eta}, and \eqref{v-expression} that
\begin{equation*} 
\begin{aligned}
\mathrm{J}_{2,2}&=-2\mu \int_0^1 (\zeta^\sharp  \eta^m\varrho  |D_\eta U|^2)\eta_r D_\eta U \,\mathrm{d}r=2\mu \int_0^1 D_\eta(\zeta^\sharp  \eta^m\varrho|D_\eta U|^2)\eta_r U \,\mathrm{d}r \\
&= \int_0^1 r^m\rho_0\Big(2\mu \big(D_\eta \zeta^\sharp + \frac{m\zeta^\sharp}{\eta}\big) + \zeta^\sharp V\Big) U|D_\eta U|^2 \,\mathrm{d}r\\
&\quad\,\,\, \underline{-\int_0^1 \zeta^\sharp r^m\rho_0 U^2 |D_\eta U|^2 \,\mathrm{d}r}_{\,\leq 0}+4\mu \int_0^1 \zeta^\sharp  r^m\rho_0  U (D_\eta U) (D_\eta^2 U) \,\mathrm{d}r\leq \mathrm{J}_{2,21}+\mathrm{J}_{2,22},    
\end{aligned}
\end{equation*}
where $\mathrm{J}_{2,21}$--$\mathrm{J}_{2,22}$ can be handled by using \eqref{ex-ell-2}, Lemmas \ref{lemma-lower bound jacobi}, \ref{lemma-v Linfty ex}, \ref{lemma-time-space} (with $\iota=\frac{1}{2}$), \ref{lemma-u-0order}, and \ref{GNinequality}, and the H\"older and Young inequalities: for any fixed $\varepsilon\in (0,\frac{\beta}{2})$,  
\begin{align*}
&\begin{aligned}
\mathrm{J}_{2,21}&\leq C(T)(|\rho_0^\beta|_\infty+|\chi^\sharp \rho_0^{\beta} V|_\infty)\big|\chi^\sharp \rho_0^{-\frac{\varepsilon}{2}} U\big|_2\big|\chi^\sharp \rho_0^\frac{1-\beta}{2}D_\eta U\big|_\infty\big|\chi^\sharp \rho_0^\frac{1+\varepsilon-\beta}{2}D_\eta U\big|_2\\
&\leq C(T)\big(1+\big|\chi^\sharp \rho_0^\frac{1}{2}(D_\eta U,U_t)\big|_2\big)\big(\big|\chi^\sharp\rho_0^\frac{1+\varepsilon}{2}D_\eta U\big|_2+\big|\chi^\sharp\rho_0^\frac{1+\varepsilon}{2}D_\eta U\big|_2^\frac{1}{2}\big|\chi^\sharp\rho_0^\frac{1+\varepsilon}{2}D_\eta^2 U\big|_2^\frac{1}{2}\big) \\
&\leq C(T)\big(1+\big|(r^m\rho_0)^\frac{1}{2}D_\eta U\big|_2^2\big)+\frac{1}{16}\big|(r^m\rho_0)^\frac{1}{2}U_t\big|_2^2, 
\end{aligned}\\
&\begin{aligned}
\mathrm{J}_{2,22}&\leq C_0\big|\chi^\sharp\rho_0^{\frac{\beta}{4}-\varepsilon} U\big|_4\big|\chi^\sharp\rho_0^{\frac{1+\varepsilon}{2}-\frac{\beta}{4}}D_\eta U\big|_4\big|\chi^\sharp\rho_0^\frac{1+\varepsilon}{2}D_\eta^2 U\big|_2\\
&\leq C(T) \big(\big|\chi^\sharp\rho_0^\frac{1+\varepsilon}{2}D_\eta U\big|_2\big|\chi^\sharp\rho_0^\frac{1+\varepsilon}{2}D_\eta^2 U\big|_2+\big|\chi^\sharp\rho_0^\frac{1+\varepsilon}{2}D_\eta U\big|_2^\frac{1}{2}\big|\chi^\sharp\rho_0^\frac{1+\varepsilon}{2}D_\eta^2 U\big|_2^\frac{3}{2}\big)\\
&\leq C(T)\big(1+\big|(r^m\rho_0)^\frac{1}{2}D_\eta U\big|_2^2\big)+\frac{1}{16}\big|(r^m\rho_0)^\frac{1}{2}U_t\big|_2^2.    
\end{aligned}
\end{align*}
Hence, we arrive at
\begin{equation}\label{rm-J2,2}
\mathrm{J}_{2,2}\leq C(T)\big(1+\big|(r^m\rho_0)^\frac{1}{2}D_\eta U\big|_2^2\big)+\frac{1}{8}\big|(r^m\rho_0)^\frac{1}{2}U_t\big|_2^2.
\end{equation}

For $\mathrm{J}_{2,3}$, it follows from Lemmas \ref{lemma-lower bound jacobi} and \ref{lemma-u-0order} that
\begin{equation}\label{rm-J2,3}
\mathrm{J}_{2,3}\leq C_0\big|\chi^\sharp r^m\eta^{-3}\big|_\infty\big|\chi^\sharp \rho_0^\frac{1}{3}U\big|_3^3\leq C(T).
\end{equation}

Collecting \eqref{rm-J2}--\eqref{jj1} and \eqref{rm-J2,2}--\eqref{rm-J2,3} gives
\begin{equation}\label{rm-J2'}
\mathrm{J}_2\leq C(T)\Big(1+|\zeta_{\frac{5}{8}} U|_\infty^2+\Big|\zeta\big(D_\eta U,\frac{U}{\eta}\big)\Big|_2^2\Big)\Big(1+\Big|(r^m\rho_0)^\frac{1}{2}\big(D_\eta U,\frac{U}{\eta}\big)\Big|_2^2\Big)+\frac{1}{4}\big|(r^m\rho_0)^\frac{1}{2}U_t\big|_2^2.
\end{equation}
 
\smallskip
\textbf{2.3. Close the energy estimate.} Collecting \eqref{dt-j1-j4}--\eqref{rm-J1} and \eqref{rm-J2'}, and then applying the Gr\"onwall inequality, we obtain from Lemma \ref{lemma-u-0order} that, for $t\in[0,T]$,
\begin{equation*}
\Big|(r^m\rho_0)^\frac{1}{2}\big(D_\eta U,\frac{U}{\eta}\big)(t)\Big|_2^2+ \int_0^t\big|(r^m\rho_0)^\frac{1}{2}U_t\big|_2^2\,\mathrm{d}s \leq  C(T) (1+\cE(0,U)) \leq C(T) . 
\end{equation*}

This completes the proof.
\end{proof}

\subsubsection{Second-order tangential estimates}
\begin{Lemma}\label{lemma-u-D2}
There exists a constant $C(T)>0$ such that 
\begin{equation*}
\big|(r^m\rho_0)^\frac{1}{2}U_t(t)\big|_2^2+ \int_0^t\Big|(r^m\rho_0)^\frac{1}{2}\big(D_\eta U_t,\frac{U_t}{\eta}\big)\Big|_2^2\,\mathrm{d}s\leq C(T) \qquad \text{for all $t\in [0,T]$}.
\end{equation*}
Moreover, for all $t\in [0,T]$,
\begin{equation*}
\begin{aligned}
\big|\chi^\sharp \rho_0^\frac{1-\beta}{2}D_\eta U(t)\big|_\infty \leq C(T),\qquad \big|\chi^\sharp \rho_0^\frac{1+\varepsilon}{2}D_\eta^2 U(t)\big|_2&\leq C(\varepsilon,T) \qquad \text{for any $\varepsilon>0$},\\
\big|\zeta_{a}U(t)\big|_\infty+\Big|\zeta_{a} r^\frac{m}{2}\Big(D_\eta^2 U,D_\eta\big(\frac{U}{\eta}\big)\Big)(t)\Big|_2&\leq C(a,T) \qquad  \text{for any $a\in (0,1)$}.
\end{aligned}
\end{equation*}
\end{Lemma}
\begin{proof}
We divide the proof into two steps.

\smallskip
\textbf{1.} First, since $\frac{m}{2}\leq 1$ and $\rho_0^\beta\sim 1-r$, it follows from Lemmas \ref{lemma-lower bound jacobi}, \ref{lemma-upper jacobi}, \ref{lemma-u-0order}--\ref{lemma-u-D1}, and \ref{sobolev-embedding}--\ref{hardy-inequality} that, for all $a\in (0,1)$,
\begin{equation}\label{U-infty-new}
\begin{aligned}
|\zeta_{a}U|_\infty^2&\leq C_0|\zeta_{a}U|_2 |((\zeta_{a})_rU,\zeta_{a} U_r)|_2\leq C(a,T)(1+ |\zeta_{a}D_\eta U|_2)\\
&\leq C(a,T)\big(1 + \big|r(\zeta_{a}D_\eta U,(\zeta_{a})_rD_\eta U,\zeta_{a}D_\eta^2 U)\big|_2\big)\\
&\leq C(a,T)\big(1+\big|r^\frac{m}{2} (\rho_0^\frac{1}{2}D_\eta U,\zeta_{a} D_\eta^2 U)\big|_2\big) \leq C(a,T)\big(1 + \big|\zeta_{a}r^\frac{m}{2}D_\eta^2 U\big|_2\big),
\end{aligned}    
\end{equation}
where we have also used the fact that
\begin{equation*}
|(\zeta_a)_r|\leq C(a)\zeta_\frac{1+3a}{4} \qquad\text{for all $r\in \bar I$}.
\end{equation*}
Next, it follows from $\frac{m+3}{4}>\frac{m}{2}$ and Lemmas \ref{lemma-u-D1} and \ref{hardy-inequality} that
\begin{equation*}
\begin{aligned}
\Big|\chi_{a} r^\frac{m}{4}\big(D_\eta U, \frac{U}{\eta} \big)\Big|_4&\leq C(a)\Big|\chi_{a} r^\frac{m+3}{4}\Big(D_\eta U, \frac{U}{\eta} ,D_\eta^2 U,D_\eta \big(\frac{U}{\eta}\big)\Big)\Big|_2\\
&\leq C(a,T)\Big(\Big|\zeta_{a} r^\frac{m}{2}\Big(D_\eta^2 U,D_\eta \big(\frac{U}{\eta}\big)\Big)\Big|_2+1\Big).
\end{aligned}
\end{equation*}
Consequently, \eqref{in-ell-2}, together with the above, \eqref{U-infty-new}, and Lemma \ref{lemma-u-D1}, gives
\begin{equation*}
\begin{aligned}
&\,\Big|\chi_{a} r^\frac{m}{4}\big(D_\eta U, \frac{U}{\eta} \big)\Big|_4+\Big|\zeta_{a} r^\frac{m}{2}\Big(D_\eta^2 U,D_\eta \big(\frac{U}{\eta}\big)\Big)\Big|_2  \\
&\leq C(a,T)\big(\big|(r^m\rho_0)^\frac{1}{2}U_t\big|_2+ |\zeta_{a} U|_\infty +1\big) \leq C(a,T)\big(\big|(r^m\rho_0)^\frac{1}{2}U_t\big|_2+ \big|\zeta_{a}r^\frac{m}{2}D_\eta^2 U\big|_2^\frac{1}{2} +1\big),
\end{aligned}
\end{equation*}
which, along with the Young inequality and \eqref{U-infty-new}, implies that, for all $a\in (0,1)$,
\begin{equation}\label{in-ell-2-new}
|\zeta_{a}U|_\infty^2+\Big|\chi_{a} r^\frac{m}{4}\big(D_\eta U, \frac{U}{\eta} \big)\Big|_4+\Big|\zeta_{a} r^\frac{m}{2}\Big(D_\eta^2 U,D_\eta\big(\frac{U}{\eta}\big)\Big)\Big|_2\leq C(a,T)\big(\big|(r^m\rho_0)^\frac{1}{2}U_t\big|_2+1\big).  
\end{equation}

Besides, we obtain from  \eqref{ex-ell-2} and Lemmas \ref{lemma-upper jacobi}, \ref{lemma-time-space} (with $\iota=\frac{1}{2}$), and \ref{lemma-u-D1} that, for any $\varepsilon>0$,
\begin{equation}\label{ex-ell-2-new}
\big|\chi^\sharp \rho_0^\frac{1-\beta}{2}D_\eta U\big|_\infty+\big|\chi^\sharp \rho_0^\frac{1+\varepsilon}{2}D_\eta^2 U\big|_2\leq C(T)\big(\big|(r^m\rho_0)^\frac{1}{2}U_t\big|_2 +1\big).
\end{equation}

\smallskip
\textbf{2.}
First, multiplying $\eqref{eq:VFBP-La-eta}_1$ by $\eta^m\eta_r$, and then applying $\partial_t$ to the resulting equality, along with $\eqref{eq:VFBP-La}_1$ and \eqref{v-expression}, yield
\begin{equation}\label{eq-2order-pre}
\begin{aligned}
& \ r^m\rho_0U_{tt} -A\gamma \Big(\varrho^\gamma\big(D_\eta U+\frac{mU}{\eta}\big)\Big)_r\eta^m +\frac{mA\gamma}{2\mu}\eta^m\eta_r\varrho^\gamma(V-U)\frac{U}{\eta} \\
&=2\mu \Big(r^m\rho_0\frac{D_\eta U_t}{\eta_r}\Big)_r-2\mu m r^m\rho_0\frac{U_t}{\eta^2}-4\mu \Big(r^m\rho_0\frac{|D_\eta U|^2}{\eta_r} \Big)_r+4\mu m r^m\rho_0\frac{U^2}{\eta^3}.
\end{aligned}    
\end{equation}
Next, multiplying the above by $U_t$, we obtain 
\begin{equation}\label{eq-2order}
\begin{aligned}
&\,\frac{1}{2}(r^m\rho_0U_{t}^2)_t + 2\mu r^m\rho_0\Big(|D_\eta U_t|^2+m\frac{U_t^2}{\eta^2}\Big)\\
&= \Big(2\mu r^m\rho_0\frac{D_\eta U_tU_t}{\eta_r}-4\mu r^m\rho_0\frac{|D_\eta U|^2U_t}{\eta_r} +A\gamma  \eta^m \varrho^\gamma\big(D_\eta U+\frac{mU}{\eta}\big)U_t\Big)_r\\
&\quad-A\gamma \eta^m\eta_r \varrho^\gamma\big(D_\eta U+\frac{mU}{\eta}\big)\big(D_\eta U_t+\frac{mU_t}{\eta}\big)-\frac{mA\gamma}{2\mu}\eta^m\eta_r \varrho^\gamma (V-U)\frac{U}{\eta}U_t\\
&\quad+4\mu r^m\rho_0\Big(|D_\eta U|^2 D_\eta U_t+m\frac{U^2U_t}{\eta^3}\Big).
\end{aligned}    
\end{equation}
Hence, integrating \eqref{eq-2order} over $I$ leads to
\begin{equation}\label{dt-J3-J5}
\begin{aligned}
&\,\frac{1}{2}\frac{\mathrm{d}}{\dt}\big|(r^m\rho_0)^\frac{1}{2}U_t\big|_2^2+ 2\mu \big|(r^m\rho_0)^\frac{1}{2}D_\eta U_t\big|_2^2+ 2\mu m \Big|(r^m\rho_0)^\frac{1}{2}\frac{U_t}{\eta}\Big|_2^2\\
&= -A\gamma\int_0^1 \eta^{m}\eta_r\varrho^\gamma\big(D_\eta U+\frac{mU}{\eta}\big)\big(D_\eta U_t+\frac{mU_t}{\eta}\big)\,\mathrm{d}r\\
&\quad-\frac{mA\gamma}{2\mu}\int_0^1\eta^{m}\eta_r\varrho^\gamma (V-U)\frac{U}{\eta}U_t\,\mathrm{d}r\\
&\quad +4\mu \int_0^1 r^m\rho_0\Big(|D_\eta U|^2D_\eta U_t+m\frac{U^2U_t}{\eta^3}\Big)\mathrm{d}r := \sum_{i=3}^5 \mathrm{J}_i.
\end{aligned}
\end{equation}

\smallskip
\textbf{2.1. Estimate for $\mathrm{J}_3$.} For $\mathrm{J}_3$, it follows from \eqref{eq:eta}, Lemmas \ref{lemma-bound depth} and \ref{lemma-u-D1}, and the H\"older and Young inequalities that
\begin{equation}\label{rm-J3}
\begin{aligned}
\mathrm{J}_3&\leq C_0|\varrho|_\infty^{\gamma-1}\Big|(r^m\rho_0)^\frac{1}{2}\big(D_\eta U+\frac{mU}{\eta}\big)\Big|_2 \Big|(r^m\rho_0)^\frac{1}{2}\big(D_\eta U_t+\frac{mU_t}{\eta}\big)\Big|_2\\
&\leq C(T)+\frac{\mu}{8} \Big|(r^m\rho_0)^\frac{1}{2} \big(D_\eta U_t,\frac{U_t}{\eta}\big)\Big|_2^2.
\end{aligned}
\end{equation}

\smallskip
\textbf{2.2. Estimate for $\mathrm{J}_4$.} For $\mathrm{J}_4$, it follows from \eqref{in-ell-2-new}, the fact that $\beta\leq \gamma-1$, Lemmas \ref{lemma-bound depth}--\ref{lemma-lower bound jacobi}, \ref{lemma-v Linfty ex}, \ref{lemma-v Linfty in}, and \ref{lemma-u-0order}--\ref{lemma-u-D1}, and the H\"older and Young inequalities that
\begin{equation}\label{rm-J4}
\begin{aligned}
\mathrm{J}_4&=-\frac{mA\gamma}{2\mu}\int_0^1 (\zeta+\zeta^\sharp) \eta^{m}\eta_r\varrho^\gamma (V-U)\frac{U}{\eta}U_t\,\mathrm{d}r\\
&\leq  C_0|\varrho|_\infty^{\gamma-1} |\zeta(V,U)|_\infty\Big|(r^m\rho_0)^\frac{1}{2}\frac{U}{\eta}\Big|_2\big|(r^m\rho_0)^\frac{1}{2}U_t\big|_2\\
&\quad + C(T)\big(|\rho_0|_\infty^{\gamma-1-\beta} |\chi^\sharp \rho_0^\beta V|_\infty\big|(r^m\rho_0)^\frac{1}{2}U\big|_2+\big|\chi^\sharp\rho_0^\frac{2\gamma-1}{4}U\big|_4^2\big)\big|(r^m\rho_0)^\frac{1}{2}U_t\big|_2\\
&\leq C(T)\big(\big|(r^m\rho_0)^\frac{1}{2}U_t\big|_2^\frac{3}{2}+\big|(r^m\rho_0)^\frac{1}{2}U_t\big|_2\big)\leq C(T)\big(\big|(r^m\rho_0)^\frac{1}{2}U_t\big|_2^4+1\big).    
\end{aligned}
\end{equation}

\smallskip
\textbf{2.3. Estimate for $\mathrm{J}_5$.} To estimate $\mathrm{J}_5$, we first divide it into two parts:
\begin{equation}\label{rm-J5}
\begin{aligned}
\mathrm{J}_5&=4\mu \int_0^1 \chi r^m\rho_0\Big(|D_\eta U|^2D_\eta U_t+m\frac{U^2U_t}{\eta^3}\Big)\,\mathrm{d}r\\
&\quad +4\mu \int_0^1 \chi^\sharp r^m\rho_0\Big(|D_\eta U|^2D_\eta U_t+m\frac{U^2U_t}{\eta^3}\Big)\,\mathrm{d}r :=\sum_{i=1}^2 \mathrm{J}_{5,i}.
\end{aligned}
\end{equation}

Then, for $\mathrm{J}_{5,1}$--$\mathrm{J}_{5,2}$, we obtain from the fact that $\frac{3\beta}{2}>\frac{1}{2}$, \eqref{in-ell-2-new}--\eqref{ex-ell-2-new}, Lemmas \ref{lemma-lower bound jacobi}, \ref{lemma-u-0order}--\ref{lemma-u-D1}, and \ref{hardy-inequality}, and the H\"older and Young inequalities that
\begin{align}
&\begin{aligned}\label{rm-J5,1}
\mathrm{J}_{5,1}&\leq C_0|\rho_0^\beta|_\infty^\frac{1}{2\beta}\Big|\chi r^\frac{m}{4}\big(D_\eta U,\frac{U}{\eta}\big)\Big|_4^2\Big|(r^m\rho_0)^\frac{1}{2}\big(D_\eta U_t,\frac{U_t}{\eta}\big)\Big|_2\\
&\leq C(T)\big(\big|(r^m\rho_0)^\frac{1}{2}U_t\big|_2^4+1\big) +\frac{\mu}{8}\Big|(r^m\rho_0)^\frac{1}{2}\big(D_\eta U_t,\frac{U_t}{\eta}\big)\Big|_2^2,
\end{aligned}\\
&\begin{aligned}\label{rm-J5,2}
\mathrm{J}_{5,2}&\leq C(T)\big(\big|\chi^\sharp\rho_0^\frac{1-\beta}{2}D_\eta U\big|_\infty \big|\chi^\sharp\rho_0^\frac{\beta}{2}D_\eta U\big|_2+\big|\chi^\sharp\rho_0^\frac{1}{4}U\big|_4^2\big)\big|(r^m\rho_0)^\frac{1}{2}(U_t,D_\eta U_t)\big|_2\\
&\leq C(T)\big(\big|\chi^\sharp\rho_0^\frac{1-\beta}{2}D_\eta U\big|_\infty \big|\chi^\sharp\rho_0^\frac{3\beta}{2}(D_\eta U,D_\eta^2 U)\big|_2+1\big)\big|(r^m\rho_0)^\frac{1}{2}(U_t,D_\eta U_t)\big|_2\\
&\leq C(T)\big(\big|(r^m\rho_0)^\frac{1}{2}U_t\big|_2^2+1\big) \big|(r^m\rho_0)^\frac{1}{2}(U_t,D_\eta U_t)\big|_2\\
&\leq C(T)\big(\big|(r^m\rho_0)^\frac{1}{2}U_t\big|_2^4+1\big)+\frac{\mu}{8} \big|(r^m\rho_0)^\frac{1}{2}D_\eta U_t\big|_2^2.    
\end{aligned}
\end{align}

Collecting \eqref{rm-J5}--\eqref{rm-J5,2} thus leads to
\begin{equation}\label{rm-J5'}
\begin{aligned}
\mathrm{J}_{5}\leq C(T)\big(\big|(r^m\rho_0)^\frac{1}{2}U_t\big|_2^4+1\big)+\frac{\mu}{4}\Big|(r^m\rho_0)^\frac{1}{2}\big(D_\eta U_t,\frac{U_t}{\eta}\big)\Big|_2^2.
\end{aligned}    
\end{equation}

\smallskip
\textbf{2.4. Close the energy estimate.} Collecting \eqref{dt-J3-J5}--\eqref{rm-J4} and \eqref{rm-J5'}, and applying the Gr\"onwall inequality, we obtain from Lemma \ref{lemma-u-D1} that, for $t\in[0,T]$,
\begin{equation*}
\big|(r^m\rho_0)^\frac{1}{2}U_t(t)\big|_2^2+ \int_0^t\Big|(r^m\rho_0)^\frac{1}{2}\big(D_\eta U_t,\frac{U_t}{\eta}\big)\Big|_2^2\,\mathrm{d}s\leq C(T)(1+\cE(0,U))\leq  C(T).
\end{equation*}

Finally, this, together with \eqref{in-ell-2-new}--\eqref{ex-ell-2-new}, yields that, for all $t\in [0,T]$,
\begin{equation*}
\begin{aligned}
\big|\chi^\sharp \rho_0^\frac{1-\beta}{2}D_\eta U(t)\big|_\infty \leq C(T),\qquad \big|\chi^\sharp \rho_0^\frac{1+\varepsilon}{2}D_\eta^2 U(t)\big|_2&\leq C(\varepsilon,T) \qquad \text{for any $\varepsilon>0$},\\
\big|\zeta_{a}U(t)\big|_\infty+\Big|\zeta_{a} r^\frac{m}{2}\Big(D_\eta^2 U,D_\eta\big(\frac{U}{\eta}\big)\Big)(t)\Big|_2&\leq C(a,T) \qquad  \text{for any $a\in (0,1)$}.
\end{aligned}
\end{equation*}

This completes the proof of Lemma \ref{lemma-u-D2}.
\end{proof}

\subsubsection{Third-order tangential estimates}\label{subsub914}
We first establish the interior $L^2$-energy estimates for $(D_\eta V,\frac{V}{\eta})$.
\begin{Lemma}\label{lemma-Vr-L2}
For any $a\in(0,1)$, there exists a constant $C(a,T)>0$ such that
\begin{equation*}
\Big|\zeta_a r^\frac{m}{2}\big(D_\eta V,\frac{V}{\eta}\big)(t)\Big|_2 \leq C(a,T) \qquad \text{for all $t\in[0,T]$}.
\end{equation*}    
\end{Lemma}
\begin{proof}
First, since $\frac{m}{2}\leq 1$, it follows from Lemmas \ref{lemma-u-D1}--\ref{lemma-u-D2} and \ref{hardy-inequality} that, for all $a\in (0,1)$,
\begin{equation}\label{DU-L2}
\begin{aligned}
\Big|\zeta_{a}\big(D_\eta U,\frac{U}{\eta}\big)\Big|_2&\leq \Big|\chi_{\frac{1+3a}{4}}\big(D_\eta U,\frac{U}{\eta}\big)\Big|_2 \leq C(a)\Big|\chi_{\frac{1+3a}{4}}r\Big(D_\eta U,\frac{U}{\eta},D_\eta^2 U,D_\eta\big(\frac{U}{\eta}\big)\Big)\Big|_2\\
&\leq C(a)\Big|r^\frac{m}{2}\Big(\rho_0^\frac{1}{2}D_\eta U,\rho_0^\frac{1}{2}\frac{U}{\eta},\zeta_{\frac{1+3a}{4}} D_\eta^2 U,\zeta_{\frac{1+3a}{4}}D_\eta\big(\frac{U}{\eta}\big)\Big)\Big|_2\leq C(a,T).
\end{aligned}
\end{equation}

Next, multiplying \eqref{eq:v} by $\zeta_a^2r^{m-2}V$ and applying $\zeta_a^2r^m V_r\partial_r$ to \eqref{eq:v}, respectively, and then combining these two resulting equations yield
\begin{equation*}
\begin{aligned}
& \ \zeta_a^2r^m\big(V_r^2+\frac{V^2}{r^2}\big)_t+\frac{A\gamma}{\mu}\zeta_a^2r^m\varrho^{\gamma-1}\big(V_r^2+\frac{V^2}{r^2}\big)\\
&=\frac{A\gamma}{\mu}\zeta_a^2r^m\varrho^{\gamma-1}\big(V_rU_r+\frac{UV}{r^2}\big)-\frac{A\gamma(\gamma-1)}{2\mu^2}\zeta_a^2r^m \eta_r\varrho^{\gamma-1} V_r(V-U)^2.
\end{aligned}
\end{equation*}
Integrating the above over $I$, then we obtain from \eqref{DU-L2}, Lemmas \ref{lemma-bound depth}, \ref{lemma-v Linfty in}, \ref{lemma-upper jacobi}, and \ref{lemma-u-D1}--\ref{lemma-u-D2}, and the H\"older and Young inequalities that 
\begin{equation*}
\begin{aligned}
&\,\frac{\mathrm{d}}{\dt} \big(\big|\zeta_a r^\frac{m}{2}V_r\big|_2^2+\big|\zeta_a r^\frac{m-2}{2}V\big|_2^2\big) \\
&\leq C_0|\varrho|_\infty^{\gamma-1}\Big|\zeta_a r^\frac{m}{2}\big(V_r,\frac{V}{r}\big)\Big|_2\Big|\big(\eta_r,\frac{\eta}{r}\big)\Big|_\infty\Big|\zeta_a\big(D_\eta U,\frac{U}{\eta}\big)\Big|_2 \\
&\quad+C_0\big|\zeta_a r^\frac{m}{2}V_r\big|_2\big|r^\frac{m}{2}\eta_r\varrho^{\gamma-1}\big|_2\big|\zeta_{\frac{1+3a}{4}}(V,U)\big|_\infty^2\leq C(a,T)\big(1+\big|\zeta_a r^\frac{m}{2}V_r\big|_2^2+\big|\zeta_a r^\frac{m-2}{2}V\big|_2^2\big),
\end{aligned}
\end{equation*}
which, along with Lemma \ref{lemma-lower bound jacobi} and the Gr\"onwall inequality, leads to
\begin{equation*}
\Big|\zeta_a r^\frac{m}{2}\big(D_\eta V,\frac{V}{\eta}\big)\Big|_2 \leq C(T)\Big|\zeta_a r^\frac{m}{2}\big(V_r,\frac{V}{r}\big)\Big|_2 \leq C(a,T)\Big(1+\Big|\zeta_a r^\frac{m}{2}\big((v_0)_r,\frac{v_0}{r}\big)\Big|_2\Big)\leq C(a,T).
\end{equation*} 
Here, to control the initial value $v_0$, it suffices to see 
\begin{equation*}
\Big|\zeta_a r^\frac{m}{2}\big((\log\rho_0)_{rr},\frac{(\log\rho_0)_r}{r}\big)\Big|_2\leq C(a)\Big|\zeta_a r^\frac{m}{2}\big((\rho_0^\beta)_{rr},|(\rho_0^\beta)_r|^2,\frac{(\rho_0^\beta)_r}{r}\big)\Big|_2\leq C(a),
\end{equation*}
due to \eqref{distance-la} and Lemma \ref{lemma-initial}.
\end{proof}

\begin{Lemma}\label{lemma-u-D3}
There exists a constant $C(T)>0$ such that
\begin{equation*}
\Big|(r^m\rho_0)^\frac{1}{2}\big(D_\eta U_t,\frac{U_t}{\eta}\big)(t)\Big|_2^2+\int_0^t\big|(r^m\rho_0)^\frac{1}{2}U_{tt}\big|_2^2\,\mathrm{d}s\leq C(T) \qquad\text{for all $t\in[0,T]$}.
\end{equation*}
In addition, for any $t\in[0,T]$, $a\in (0,1)$, and $\varepsilon>0$ satisfying
\begin{equation*}
0<\varepsilon<\min\big\{\beta,\frac{3\beta-1}{2}\big\},
\end{equation*}
the following estimates also hold{\rm:}
\begin{equation*}
\begin{aligned}
&\big|\chi^\sharp \rho_0^\frac{1}{2}D_\eta^2 U(t)\big|_2+\Big|\big(U,D_\eta U,\frac{U}{\eta}\big)(t)\Big|_\infty\leq C(T),\qquad \big|\chi^\sharp \rho_0^{-\varepsilon} D_\eta U(t)\big|_\infty\leq C(\varepsilon,T),\\
&\Big|\zeta_{a}r^\frac{m}{2}\Big(D_\eta^3 U,D_\eta^2\big(\frac{U}{\eta}\big),\frac{1}{\eta}D_\eta\big(\frac{U}{\eta}\big)\Big)(t)\Big|_2\leq C(a,T).
\end{aligned}
\end{equation*}
\end{Lemma}
\begin{proof}
We divide the proof into two steps.

\smallskip
\textbf{1.} We first give some auxiliary estimates associated with the third-order derivatives of $U$.

\smallskip
\textbf{1.1.} First, it follows from \eqref{DU-L2} and Lemmas \ref{lemma-upper jacobi}, \ref{lemma-u-D2}, and \ref{sobolev-embedding}--\ref{hardy-inequality} that, for all $a\in (0,1)$,
\begin{equation}\label{in-Ur-infty}
\begin{aligned}
\Big|\zeta_{a}\big(D_\eta U,\frac{U}{\eta}\big)\Big|_\infty^2&\leq C_0\Big|\zeta_{a}\big(D_\eta U,\frac{U}{\eta}\big)\Big|_2\Big(\Big|(\zeta_{a})_r\big(D_\eta U,\frac{U}{\eta}\big)\Big|_2+\Big|\zeta_{a}\Big(D_\eta^2 U,D_\eta\big(\frac{U}{\eta}\big)\Big)\Big|_2\Big)\\
&\leq C(a,T)\big(1+\big|r\big(\zeta_{a}D_\eta^2 U,(\zeta_{a})_rD_\eta^2 U,\zeta_{a}D_\eta^3 U\big)\big|_2\big)\\
&\quad + C(a,T)\Big(1+\Big|r\Big(\zeta_{a}D_\eta\big(\frac{U}{\eta}\big),(\zeta_{a})_rD_\eta\big(\frac{U}{\eta}\big),\zeta_{a}D_\eta^2\big(\frac{U}{\eta}\big)\Big)\Big|_2\Big)\\
&\leq C(a,T)\Big(1+\Big|\zeta_{a}r^\frac{m}{2}\Big(D_\eta^3 U,D_\eta^2\big(\frac{U}{\eta}\big)\Big)\Big|_2\Big).
\end{aligned}
\end{equation}

Next, based on \eqref{ell-0}, we can obtain from Lemmas \ref{lemma-bound depth}, \ref{lemma-upper jacobi}, \ref{lemma-u-D1}, and \ref{lemma-Vr-L2} that, for all $a\in (0,1)$,
\begin{equation}\label{in-ell-3-pre1}
\begin{aligned}
\Big|\zeta_{a} r^\frac{m}{2}\frac{1}{\eta}D_\eta\big(D_\eta U+ \frac{mU}{\eta}\big)\Big|_2 &\leq C_0\Big(\Big|\zeta_{a} r^\frac{m}{2}\frac{U_t}{\eta}\Big|_2+  \Big|\zeta_{\frac{1+3a}{4}}r^\frac{m}{2}\big(\frac{V}{\eta},\frac{U}{\eta}\big)\Big|_2 |(\zeta_{a}D_\eta U,\varrho^{\gamma-1})|_\infty \Big)\\
&\leq C(a,T)\Big(\Big|\zeta_{a} r^\frac{m}{2}\frac{U_t}{\eta}\Big|_2+ |\zeta_{a}D_\eta U|_\infty +1\Big).
\end{aligned}
\end{equation}
On the other hand, applying $D_\eta$ to \eqref{ell-0}, along with \eqref{v-expression}, yields
\begin{equation}\label{ell-1}
\begin{aligned}
D_\eta^2\big(D_\eta U+ \frac{mU}{\eta}\big) &=\frac{1}{2\mu} D_\eta U_{t}-\frac{1}{2\mu}(V-U)D_\eta^2 U-\frac{1}{2\mu}(D_\eta V-D_\eta U)D_\eta U \\
&\quad +\frac{A\gamma}{4\mu^2} \varrho^{\gamma-1}(D_\eta V-D_\eta U)+\frac{A\gamma(\gamma-1)}{8\mu^3} \varrho^{\gamma-1}(V-U)^2,
\end{aligned}
\end{equation}
which, along with Lemmas \ref{lemma-bound depth}, \ref{lemma-v Linfty in}, \ref{lemma-upper jacobi}, and \ref{lemma-u-D1}--\ref{lemma-Vr-L2}, leads to
\begin{equation}\label{in-ell-3-pre2}
\begin{aligned}
\Big|\zeta_{a}r^\frac{m}{2}D_\eta^2\big(D_\eta U+ \frac{mU}{\eta}\big)\Big|_2&\leq C_0|\zeta_{a}r^\frac{m}{2}D_\eta U_t|_2+C_0\big|\zeta_{\frac{1+3a}{4}}(V,U)\big|_\infty |\zeta_{a}r^\frac{m}{2}D_\eta^2 U|_2 \\
&\quad +C(T) |(\varrho^{\gamma-1},\zeta_{a}D_\eta U)|_\infty \big|\zeta_{\frac{1+3a}{4}}r^\frac{m}{2}(D_\eta V,D_\eta U)\big|_2\\
&\quad +C(T)|\varrho|_\infty^{\gamma-1}\big|\zeta_{\frac{1+3a}{4}}(V,U)\big|_\infty|\zeta_{a}(V,U)\big|_2\\
&\leq C(a,T)\big(|\zeta_{a}r^\frac{m}{2}D_\eta U_t|_2+|\zeta_{a}D_\eta U|_\infty +1\big).
\end{aligned}
\end{equation}

Collecting \eqref{in-Ur-infty}--\eqref{in-ell-3-pre1}
and \eqref{in-ell-3-pre2}, together with the Young inequality and Lemma \ref{im-1}, we thus obtain that, for all $a\in (0,1)$,
\begin{equation}\label{in-ell-3}
\begin{aligned}
&\,\Big|\zeta_{a}\big(D_\eta U,\frac{U}{\eta}\big)\Big|_\infty^2+\Big|\zeta_{a}r^\frac{m}{2}\Big(D_\eta^3 U,D_\eta^2\big(\frac{U}{\eta}\big),\frac{1}{\eta}D_\eta\big(\frac{U}{\eta}\big)\Big)\Big|_2\\
&\leq C(a,T)\Big(\Big|\zeta_{a} r^\frac{m}{2}\big(D_\eta U_t,\frac{U_t}{\eta}\big)\Big|_2+1\Big)\leq C(a,T)\Big(\Big|(r^m\rho_0)^\frac{1}{2}\big(D_\eta U_t,\frac{U_t}{\eta}\big)\Big|_2+1\Big).
\end{aligned}    
\end{equation}

\textbf{1.2.} Since $\frac{3\beta}{2}>\frac{1}{2}$, it follows from Lemmas \ref{lemma-upper jacobi}, \ref{lemma-time-space} (with $\iota=\frac{\beta}{2}$), \ref{lemma-u-0order}--\ref{lemma-u-D2}, and \ref{sobolev-embedding}--\ref{hardy-inequality} and the H\"older inequality that
\begin{equation}\label{Ur-infty-high}
\begin{aligned}
\Big|\chi^\sharp\big(U,D_\eta U,\frac{U}{\eta}\big)\Big|_\infty&\leq C(T)|\chi^\sharp (D_\eta U,U)|_\infty\leq C(T)(|\chi^\sharp D_\eta U|_\infty+ |\chi^\sharp (D_\eta U,U)|_2)\\[-4pt]
&\leq C(T)(1+|\chi^\sharp D_\eta U|_\infty) \leq C(T)\big(1+\big|\chi^\sharp\rho_0^\frac{\beta}{2}(D_\eta U,U_t)\big|_2\big)\\
&\leq C(T)\big(1+\big|\chi^\sharp\rho_0^\frac{3\beta}{2}(D_\eta U,D_\eta^2 U,U_t,D_\eta U_t)\big|_2\big)\\
&\leq C(T)\big(1+\big|\chi^\sharp\rho_0^\frac{1}{2}D_\eta U_t\big|_2\big)\leq C(T)\big(1+\big|(r^m\rho_0)^\frac{1}{2}D_\eta U_t\big|_2\big).
\end{aligned}    
\end{equation}
Moreover, taking $\iota=\frac{\beta}{2}-\varepsilon$ in Lemma \ref{lemma-time-space} with
\begin{equation*}
0<\varepsilon<\min\big\{\beta,\frac{3\beta-1}{2}\big\},
\end{equation*}
and using an argument similar to \eqref{Ur-infty-high}, we can also obtain the following estimate:
\begin{equation}\label{Ur-infty-refine}
\begin{aligned}
|\chi^\sharp \rho_0^{-\varepsilon} D_\eta U|_\infty &\leq C(\varepsilon,T)\big(1+\big|\chi^\sharp\rho_0^\frac{3\beta-2\varepsilon}{2}(D_\eta U,D_\eta^2 U,U_t,D_\eta U_t)\big|_2\big)\\
&\leq C(\varepsilon,T)\big(1+\big|(r^m\rho_0)^\frac{1}{2}D_\eta U_t\big|_2\big).
\end{aligned}
\end{equation}

Finally, multiplying \eqref{ell-0-ex} by $\chi^\sharp \rho_0^\frac{1}{2}$ and choosing a fixed constant $\sigma>0$ such that
\begin{equation*}
\max\{0,1-\beta\}<\sigma<\min\{2\beta,1+\beta\},
\end{equation*}
we obtain from \eqref{eq:eta}, \eqref{Ur-infty-high}--\eqref{Ur-infty-refine}, and Lemmas \ref{lemma-bound depth}--\ref{lemma-lower bound jacobi}, \ref{lemma-v Lp ex}--\ref{lemma-v Linfty ex}, and \ref{lemma-u-0order}--\ref{lemma-u-D2} that
\begin{equation}\label{ex-ell-2-refine}
\begin{aligned}
\big|\chi^\sharp \rho_0^\frac{1}{2}D_\eta^2 U\big|_2&\leq C(T)\big(\big|\chi^\sharp \rho_0^\frac{1}{2}(U,D_\eta U,U_t)\big|_2+  \big|\chi^\sharp \rho_0^\frac{\beta+\sigma}{2}(V,U)\big|_2\big|\chi^\sharp \rho_0^{\frac{1-\beta-\sigma}{2}}D_\eta U\big|_\infty\big) \\
&\quad + C(T)\big(\big|\chi^\sharp\rho_0^{\frac{2\gamma-1-2\beta}{2}}\big|_2|\chi^\sharp \rho_0^\beta V|_\infty+ \big|\chi^\sharp\rho_0^{\frac{2\gamma-1}{2}}\big|_2|\chi^\sharp U|_\infty\big)\\
&\leq C(T)\big(1+\big|(r^m\rho_0)^\frac{1}{2}D_\eta U_t\big|_2\big). 
\end{aligned}
\end{equation}

\smallskip
\textbf{2.} Now, we come to establish the third-order tangential estimates for $U$. First, multiplying \eqref{eq-2order-pre} by $U_{tt}$, along with \eqref{eq:eta} and \eqref{v-expression}, gives
\begin{equation}\label{eq-3order}
\begin{aligned}
&\ \mu r^m\rho_0\Big(|D_\eta U_t|^2+m\frac{U_t^2}{\eta^2}\Big)_t+r^m\rho_0U_{tt}^2\\
&= \Big(2\mu r^m\rho_0\frac{(D_\eta U_t)U_{tt}}{\eta_r}\Big)_r+\frac{A\gamma}{2\mu} \eta^{m}\eta_r\varrho^\gamma (V-U)\big(\gamma D_\eta U+\frac{m(\gamma-1)U}{\eta}\big)U_{tt}\\
&\quad +A\gamma\eta^m\eta_r \varrho^\gamma D_\eta\big(D_\eta U+\frac{mU}{\eta}\big) U_{tt}-2\mu r^m\rho_0\Big(D_\eta U|D_\eta U_t|^2+m\frac{UU_t^2}{\eta^3}\Big)\\
&\quad -2r^m\rho_0\big(4\mu D_\eta UD_\eta^2 U+ (V-U) |D_\eta U|^2\big)U_{tt}+4\mu m\frac{r^m\rho_0}{\eta}\big(\frac{U^2}{\eta^2}-|D_\eta U|^2\big)U_{tt}.
\end{aligned}
\end{equation}

Hence, integrating \eqref{eq-3order} over $I$ leads to
\begin{equation}\label{dt-J6-J11}
\begin{aligned}
&\ \mu \frac{\mathrm{d}}{\dt}\Big(\big|(r^m\rho_0)^\frac{1}{2}D_\eta U_t\big|_2^2+ m \Big|(r^m\rho_0)^\frac{1}{2}\frac{U_t}{\eta}\Big|_2^2\Big)+\big|(r^m\rho_0)^\frac{1}{2}U_{tt}\big|_2^2\\
&= A\gamma \int_0^1\eta^{m}\eta_r\varrho^\gamma  \Big(\frac{1}{2\mu}(V-U)\big(\gamma D_\eta U+\frac{m(\gamma-1)U}{\eta}\big) +   D_\eta\big(D_\eta U+\frac{mU}{\eta}\big)\Big) U_{tt}\,\mathrm{d}r\\
&\quad-2 \int_0^1r^m\rho_0\Big(\mu D_\eta U|D_\eta U_t|^2+\mu m\frac{UU_t^2}{\eta^3}+\big( 4\mu D_\eta UD_\eta^2 U+ (V-U) |D_\eta U|^2\big)U_{tt}\Big)\,\mathrm{d}r\\
&\quad+4\mu m\int_0^1 \frac{r^m\rho_0}{\eta}\big(\frac{U^2}{\eta^2}-|D_\eta U|^2\big)U_{tt}\,\mathrm{d}r:= \sum_{i=6}^{8} \mathrm{J}_i.
\end{aligned}    
\end{equation}

\smallskip
\textbf{2.1. Estimate for $\mathrm{J}_6$.} For $\mathrm{J}_6$, we divide it into two parts:
\begin{equation}\label{rm-J60}
\begin{aligned}
\mathrm{J}_6& \leq C_0 \int_0^1 \zeta \eta^{m}\eta_r\varrho^\gamma  \Big(|V-U|\big(|D_\eta U|+\Big|\frac{U}{\eta}\Big|\big)+ |D_\eta^2 U|+\Big|D_\eta\big(\frac{U}{\eta}\big)\Big|\Big) |U_{tt}|\,\mathrm{d}r\\
&\quad +C_0 \int_0^1 \!\zeta^\sharp \eta^{m}\eta_r\varrho^\gamma  \Big(|V-U|\big(|D_\eta U|\!+\!\Big|\frac{U}{\eta}\Big|\big)\!+\! |D_\eta^2 U|\!+\!\Big|D_\eta\big(\frac{U}{\eta}\big)\Big|\Big) |U_{tt}|\,\mathrm{d}r:= \sum_{i=1}^{2} \mathrm{J}_{6,i}. 
\end{aligned}
\end{equation}
For $\mathrm{J}_{6,1}$, it follows from \eqref{eq:eta}, Lemmas \ref{lemma-bound depth}, \ref{lemma-v Linfty in}, and \ref{lemma-u-D1}--\ref{lemma-u-D2}, and  the H\"older and Young inequalities that
\begin{equation}\label{jj61}
\begin{aligned}
\mathrm{J}_{6,1}&\leq C_0|\varrho|_\infty^{\gamma-1} |\zeta(V,U)|_\infty \Big|(r^m\rho_0)^\frac{1}{2}\big(D_\eta U,\frac{U}{\eta}\big)\Big|_2\big|(r^m\rho_0)^\frac{1}{2}U_{tt}\big|_2\\
&\quad + C_0|\varrho|_\infty^{\gamma-1}|\rho_0|_\infty^\frac{1}{2}\Big|\zeta r^\frac{m}{2}\Big(D_\eta^2 U, D_\eta\big(\frac{U}{\eta}\big)\Big)\Big|_2\big|(r^m\rho_0)^\frac{1}{2}U_{tt}\big|_2\\
&\leq C(T)\big|(r^m\rho_0)^\frac{1}{2}U_{tt}\big|_2\leq C(T)+\frac{1}{16}\big|(r^m\rho_0)^\frac{1}{2}U_{tt}\big|_2^2.
\end{aligned}
\end{equation}
For $\mathrm{J}_{6,2}$, we can obtain from the fact that $\gamma-\beta-\frac{1}{2}\geq \frac{1}{2}$, \eqref{eq:eta}, \eqref{Ur-infty-high}, \eqref{ex-ell-2-refine}, and Lemmas  \ref{lemma-lower bound jacobi}, \ref{lemma-v Linfty ex}, and \ref{lemma-u-0order}--\ref{lemma-u-D1} that
\begin{equation}\label{jj62}
\begin{aligned}
\mathrm{J}_{6,2}&\leq C(T)(|\chi^\sharp \rho_0^\beta (V,U)|_\infty+|\rho_0^\beta|_\infty) \big|\chi^\sharp\rho_0^{\gamma-\beta-\frac{1}{2}}(U,D_\eta U,D_\eta^2 U)\big|_2 \big|(r^m\rho_0)^\frac{1}{2}U_{tt}\big|_2\\
&\leq C(T)\big(1+\big|(r^m\rho_0)^\frac{1}{2}D_\eta U_t\big|_2\big) \big|(r^m\rho_0)^\frac{1}{2}U_{tt}\big|_2\\
&\leq C(T)\big(1+\big|(r^m\rho_0)^\frac{1}{2}D_\eta U_t\big|_2^4\big)+\frac{1}{16}\big|(r^m\rho_0)^\frac{1}{2}U_{tt}\big|_2^2.    
\end{aligned}
\end{equation}

Hence, collecting \eqref{rm-J60}--\eqref{jj62}, we arrive at
\begin{equation}\label{rm-J6}
\mathrm{J}_6\leq C(T)\big(1+\big|(r^m\rho_0)^\frac{1}{2}D_\eta U_t\big|_2^4\big)+\frac{1}{8}\big|(r^m\rho_0)^\frac{1}{2}U_{tt}\big|_2^2.    
\end{equation}

\smallskip
\textbf{2.2. Estimate for $\mathrm{J}_7$.} For $\mathrm{J}_7$, we first obtain from the H\"older inequality that 
\begin{equation}\label{jj7}
\!\!\!\begin{aligned}
\mathrm{J}_{7}&\leq C_0\Big|\big(D_\eta U,\frac{U}{\eta}\big)\Big|_\infty\Big|(r^m\rho_0)^\frac{1}{2}\big(D_\eta U_t,\frac{U_{t}}{\eta}\big)\Big|_2^2+C_0|D_\eta U|_\infty\big|(r^m\rho_0)^\frac{1}{2}D_\eta^2 U\big|_2\big|(r^m\rho_0)^\frac{1}{2}U_{tt}\big|_2\\
&\quad +C_0\big|(r^m\rho_0)^\frac{1}{2}(V-U)|D_\eta U|^2\big|_2\big|(r^m\rho_0)^\frac{1}{2}U_{tt}\big|_2:=\sum_{i=1}^3 \mathrm{J}_{7,i}.
\end{aligned}
\end{equation}
Then, for $\mathrm{J}_{7,1}$--$\mathrm{J}_{7,2}$, it follows from \eqref{in-ell-3}--\eqref{Ur-infty-high}, \eqref{ex-ell-2-refine}, Lemma \ref{lemma-u-D2},  and the H\"older and Young inequalities that
\begin{equation}
\begin{aligned}
\mathrm{J}_{7,1}&\leq C_0\Big(\Big|\zeta\big(D_\eta U,\frac{U}{\eta}\big)\Big|_\infty+\Big|\chi^\sharp\big(D_\eta U,\frac{U}{\eta}\big)\Big|_\infty\Big)\Big|(r^m\rho_0)^\frac{1}{2}\big(D_\eta U_t,\frac{U_{t}}{\eta}\big)\Big|_2^2\\
&\leq C(T)\Big(1+\Big|(r^m\rho_0)^\frac{1}{2}\big(D_\eta U_t,\frac{U_{t}}{\eta}\big)\Big|_2^4\Big),\\
\mathrm{J}_{7,2}&\leq C_0(|\zeta D_\eta U|_\infty+|\chi^\sharp D_\eta U|_\infty)\big|(\zeta r^\frac{m}{2}D_\eta^2 U,\zeta^\sharp \rho_0^\frac{1}{2}D_\eta^2 U)\big|_2\big|(r^m\rho_0)^\frac{1}{2}U_{tt}\big|_2\\
&\leq C(T)\Big(1+\Big|(r^m\rho_0)^\frac{1}{2}\big(D_\eta U_t,\frac{U_{t}}{\eta}\big)\Big|_2^4\Big)+\frac{1}{16}\big|(r^m\rho_0)^\frac{1}{2}U_{tt}\big|_2^2.
\end{aligned}
\end{equation}
For $\mathrm{J}_{7,3}$, choosing a fixed constant $\sigma$ such that 
\begin{equation*}
\max\{0,1-\beta\}<\sigma<\min\{3\beta+1,5\beta-1\},
\end{equation*}
then we obtain from \eqref{in-ell-3}, \eqref{Ur-infty-refine}, and Lemmas \ref{lemma-v Lp ex}, \ref{lemma-v Linfty in}, and \ref{lemma-u-0order}--\ref{lemma-u-D2} that
\begin{equation}\label{rm-J80}
\begin{aligned}
\mathrm{J}_{7,3}&\leq C_0|\zeta_\frac{5}{8}(V,U)|_\infty |\zeta D_\eta U|_\infty \big|(r^m\rho_0)^\frac{1}{2}D_\eta U\big|_2\big|(r^m\rho_0)^\frac{1}{2}U_{tt}\big|_2\\
&\quad + C_0\big|\zeta^\sharp \rho_0^\frac{\beta+\sigma}{2}(V,U)\big|_2\big|\chi^\sharp \rho_0^\frac{1-\beta-\sigma}{4}D_\eta U\big|_\infty^2\big|(r^m\rho_0)^\frac{1}{2}U_{tt}\big|_2\\
&\leq C(T)\Big(1+\Big|(r^m\rho_0)^\frac{1}{2}\big(D_\eta U_t,\frac{U_{t}}{\eta}\big)\Big|_2^4\Big)+\frac{1}{16}\big|(r^m\rho_0)^\frac{1}{2}U_{tt}\big|_2^2.
\end{aligned}
\end{equation}

Hence, collecting \eqref{jj7}--\eqref{rm-J80} yields
\begin{equation}\label{rm-J8}
\mathrm{J}_{7}\leq C(T)\Big(1+\Big|(r^m\rho_0)^\frac{1}{2}\big(D_\eta U_t,\frac{U_{t}}{\eta}\big)\Big|_2^4\Big)+\frac{1}{8}\big|(r^m\rho_0)^\frac{1}{2}U_{tt}\big|_2^2.
\end{equation}

\smallskip
\textbf{2.3. Estimate for $\mathrm{J}_{8}$.} To estimate $\mathrm{J}_{8}$, we first obtain from integration by parts that
\begin{equation}\label{rm-J11}
\begin{aligned}
\mathrm{J}_{8}&=4\mu m \frac{\mathrm{d}}{\dt}\int_0^1  r^m\rho_0\big(\frac{U^2}{\eta^2}-|D_\eta U|^2\big)\frac{U_{t}}{\eta}\,\mathrm{d}r\\
&\quad +4\mu m \int_0^1 r^m\rho_0 \Big(\frac{3U^3}{\eta^3}-\frac{U|D_\eta U|^2}{\eta}-2(D_\eta U)^3+2D_\eta U D_\eta U_t-\frac{2UU_t}{\eta^2}\Big)\frac{U_{t}}{\eta}\,\mathrm{d}r\\
&:=\frac{\mathrm{d}}{\mathrm{d}t}\widetilde{\mathrm{J}}_{8}+\mathrm{J}_{8,1}.
\end{aligned}
\end{equation}
For $\mathrm{J}_{8,1}$, by \eqref{in-ell-3}--\eqref{Ur-infty-high}, Lemma \ref{lemma-u-D1}, and  the H\"older and Young inequalities, we have
\begin{equation*}
\begin{aligned}
\mathrm{J}_{8,1}&\leq C_0\Big|\big(D_\eta U,\frac{U}{\eta}\big)\Big|_\infty^2\Big|(r^m\rho_0)^\frac{1}{2}\big(D_\eta U,\frac{U}{\eta}\big)\Big|_2 \Big|(r^m\rho_0)^\frac{1}{2}\frac{U_t}{\eta}\Big|_2\\
&\quad +C_0\Big|\big(D_\eta U,\frac{U}{\eta}\big)\Big|_\infty\Big|(r^m\rho_0)^\frac{1}{2}\big(D_\eta U_t,\frac{U_t}{\eta}\big)\Big|_2\Big|(r^m\rho_0)^\frac{1}{2}\frac{U_t}{\eta}\Big|_2\\
&\leq C(T)\Big(1+\Big|(r^m\rho_0)^\frac{1}{2}\big(D_\eta U_t,\frac{U_{t}}{\eta}\big)\Big|_2^4\Big),
\end{aligned}
\end{equation*}
which, along with \eqref{rm-J11}, yields
\begin{equation}\label{rm-J11'}
\mathrm{J}_{8}\leq\frac{\mathrm{d}}{\dt}\widetilde{\mathrm{J}}_{8}+C(T)\Big(1+\Big|(r^m\rho_0)^\frac{1}{2}\big(D_\eta U_t,\frac{U_{t}}{\eta}\big)\Big|_2^4\Big).
\end{equation}

Besides, for $\widetilde{\mathrm{J}}_{8}$, since
\begin{equation*}
\frac{m+3}{4}>\frac{m}{2},\qquad \frac{1+3\beta}{4}>\frac{1}{2},
\end{equation*}
we can deduce from Lemmas \ref{lemma-lower bound jacobi}, \ref{lemma-upper jacobi}, \ref{lemma-u-0order}--\ref{lemma-u-D2}, and \ref{hardy-inequality} that
\begin{equation*}
\begin{aligned}
& \ \Big|(r^m\rho_0)^\frac{1}{4}\big(D_\eta U,\frac{U}{\eta}\big)\Big|_4\leq C_0\Big|\chi r^\frac{m}{4}\big(D_\eta U,\frac{U}{\eta}\big)\Big|_4+C_0\Big|\chi^\sharp \rho_0^\frac{1}{4}\big(D_\eta U,\frac{U}{\eta}\big)\Big|_4\\
&\leq C_0\Big|\chi r^\frac{m+3}{4}\Big(D_\eta U,\frac{U}{\eta},D_\eta^2 U,D_\eta\big(\frac{U}{\eta}\big)\Big)\Big|_2 +C_0\Big|\chi^\sharp \rho_0^\frac{1+3\beta}{4}\Big(D_\eta U,\frac{U}{\eta},D_\eta^2 U,D_\eta\big(\frac{U}{\eta}\big)\Big)\Big|_2\\
&\leq C(T) +C(T)\big|\chi^\sharp \rho_0^\frac{1+3\beta}{4}(U,D_\eta U,D_\eta^2 U)\big|_2\\
&\leq C(T)+C(T)\big|(r^m\rho_0)^\frac{1}{2}(U,D_\eta U)\big|_2+C(T)\big|\chi^\sharp \rho_0^\frac{1+3\beta}{4}D_\eta^2 U\big|_2\leq C(T),
\end{aligned}    
\end{equation*}
which implies that, for all $t\in [0,T]$,
\begin{equation}\label{l4-fuzhu}
\widetilde{\mathrm{J}}_{8}\leq C_0\Big|(r^m\rho_0)^\frac{1}{4}\big(D_\eta U,\frac{U}{\eta}\big)\Big|_4^2\Big|(r^m\rho_0)^\frac{1}{2}\frac{U_t}{\eta}\Big|_2\leq C(T)\Big|(r^m\rho_0)^\frac{1}{2}\frac{U_t}{\eta}\Big|_2.
\end{equation}

\smallskip
\textbf{2.4. Close the energy estimate.} First, collecting \eqref{dt-J6-J11}, \eqref{rm-J6}, \eqref{rm-J8}, and \eqref{rm-J11'} gives 
\begin{equation*}
\begin{aligned}
& \ \mu\frac{\mathrm{d}}{\dt}\Big(\big|(r^m\rho_0)^\frac{1}{2}D_\eta U_t\big|_2^2+ m \Big|(r^m\rho_0)^\frac{1}{2}\frac{U_t}{\eta}\Big|_2^2\Big)+\frac{1}{2}\big|(r^m\rho_0)^\frac{1}{2}U_{tt}\big|_2^2\\
&\leq \frac{\mathrm{d}}{\dt}\widetilde{\mathrm{J}}_{8}+C(T)\Big(1+\Big|(r^m\rho_0)^\frac{1}{2}\big(D_\eta U_t,\frac{U_{t}}{\eta}\big)\Big|_2^4\Big).
\end{aligned}
\end{equation*}
Thus, it follows from \eqref{l4-fuzhu}, Lemma \ref{lemma-u-D2}, and the Gr\"onwall inequality that 
\begin{equation*}
\begin{aligned}
&\sup_{s\in[0,t]}\Big|(r^m\rho_0)^\frac{1}{2}\big(D_\eta U_t,\frac{U_{t}}{\eta}\big)\Big|_2^2+\int_0^t\big|(r^m\rho_0)^\frac{1}{2}U_{tt}\big|_2^2\,\mathrm{d}s\\
&\leq C(T)\Big(\cE(0,U)+ \sup_{s\in[0,t]} \widetilde{\mathrm{J}}_8+1\Big)\leq C(T)\Big(1+  \sup_{s\in[0,t]} \Big|(r^m\rho_0)^\frac{1}{2}\frac{U_t}{\eta}\Big|_2\Big),
\end{aligned}    
\end{equation*}
which, along with the Young inequality, leads to
\begin{equation*} 
\Big|(r^m\rho_0)^\frac{1}{2}\big(D_\eta U_t,\frac{U_{t}}{\eta}\big)(t)\Big|_2^2+\int_0^t\big|(r^m\rho_0)^\frac{1}{2}U_{tt}\big|_2^2\,\mathrm{d}s\leq C(T) \qquad \text{for all $t\in [0,T]$}.
\end{equation*}

Finally, this, together with \eqref{in-ell-3}--\eqref{ex-ell-2-refine}, yields the rest of this lemma.
\end{proof}

\subsection{Estimates of the velocity near the symmetric center}\label{sub92}

\begin{Lemma}\label{ell-inner}
There exists a constant $C(T)>0$ such that
\begin{equation*}
\cE_{\mathrm{in}}(t,U)+\int_0^t \cD_{\mathrm{in}}(s,U)\,\mathrm{d}s \leq C(T) \qquad\text{for all $t\in [0,T]$}.
\end{equation*}
\end{Lemma}
\begin{proof}
We divide the proof into two steps.

\smallskip
\textbf{1. Elliptic estimates.} First, Lemmas \ref{lemma-u-0order}--\ref{lemma-u-D2}, and \ref{lemma-u-D3} lead to 
\begin{equation}\label{total-Ein}
\cE_{\mathrm{in}}(t,U)\leq C(T) \qquad\text{for all $t\in [0,T]$}.
\end{equation}

\smallskip
\textbf{2. Dissipation estimates.} Next, since we have already obtained 
\begin{equation}\label{D-tt}
\int_0^t |\zeta r^\frac{m}{2}U_{tt}|_2^2\,\mathrm{d}s\leq C(T)
\end{equation}
due to Lemma \ref{lemma-u-D3}, it remains to show that
\begin{equation}\label{D-txx/D-xxxx}
\begin{aligned}
\int_0^t\Big|\zeta r^\frac{m}{2}\Big(D_\eta^2 U_{t},D_\eta\big(\frac{U_{t}}{\eta}\big)\Big)\Big|_2^2\,\mathrm{d}s\leq C(T),\\
\int_0^t\Big|\zeta r^\frac{m}{2}\Big(D_\eta^4 U, D_\eta^3\big(\frac{U}{\eta}\big),D_\eta\big(\frac{1}{\eta}D_\eta(\frac{U}{\eta})\big) \Big)\Big|_2^2\,\mathrm{d}s\leq C(T).
\end{aligned}
\end{equation}

\smallskip
\textbf{2.1. Proof of $\eqref{D-txx/D-xxxx}_1$.} First, it follows from \eqref{eq:v}, Lemmas \ref{lemma-bound depth}, \ref{lemma-v Linfty in}, and \ref{lemma-u-D3} that
\begin{equation}\label{est-vt-0}
|\zeta r^\frac{m}{2} V_t|_2\leq |\zeta V_t|_\infty \leq C_0|\varrho|_\infty^{\gamma-1} |\zeta (V,U)|_\infty\leq C(T) \qquad\text{for all $t\in [0,T]$}.
\end{equation}

Next, applying $\partial_t$ to \eqref{ell-0}, along with $\eqref{eq:VFBP-La}_1$, we have
\begin{equation}\label{ell-0-t-0}
\begin{aligned}
D_\eta\big(D_\eta U_t+ \frac{mU_t}{\eta}\big)&=\underline{\frac{1}{2\mu}U_{tt} +3D_\eta U D_\eta^2 U+m\big(D_\eta U+\frac{2U}{\eta}\big)D_\eta\big(\frac{U}{\eta}\big)}_{:=\mathrm{J}_9} \\
&\quad \underline{-\frac{1}{2\mu}(V_t-U_t)D_\eta U-\frac{1}{2\mu}(V-U)(D_\eta U_t- |D_\eta U|^2)}_{:=\mathrm{J}_{10}}\\
&\quad +\underline{\frac{A\gamma}{4\mu^2} \varrho^{\gamma-1} \Big((V_t-U_t)-(\gamma-1)\big(D_\eta U+\frac{mU}{\eta}\big)(V-U)\Big)}_{:=\mathrm{J}_{11}}.
\end{aligned}
\end{equation}
Then we obtain from \eqref{total-Ein}, \eqref{est-vt-0}, and Lemmas \ref{lemma-bound depth}, \ref{lemma-v Linfty in}, and \ref{lemma-u-D3} that
\begin{equation}\label{proof0}
\begin{aligned}
|\zeta r^\frac{m}{2} \mathrm{J}_{9}|_2&\leq C_0|\zeta r^\frac{m}{2} U_{tt}|_2 +C_0\Big|\big(D_\eta U,\frac{U}{\eta}\big)\Big|_\infty\Big|\zeta r^\frac{m}{2}\Big(D_\eta^2 U,D_\eta\big(\frac{U}{\eta}\big)\Big)\Big|_2\\
&\leq C(T)(|\zeta r^\frac{m}{2} U_{tt}|_2+1),\\
|\zeta r^\frac{m}{2} \mathrm{J}_{10}|_2&\leq C_0|\chi_\frac{5}{8}(D_\eta U,V,U)|_\infty\big(|\zeta r^\frac{m}{2} (V_t,U_t,D_\eta U_t)|_2+|\zeta r^\frac{m}{2}|_2|D_\eta U|_\infty^2\big)\leq C(T),\\
|\zeta r^\frac{m}{2} \mathrm{J}_{11}|_2&\leq C_0|\varrho|_\infty^{\gamma-1}\Big(|\zeta r^\frac{m}{2} (V_t,U_t)|_2+|\chi_\frac{5}{8}(V,U)|_\infty\Big|\zeta r^\frac{m}{2}\big(D_\eta U,\frac{U}{\eta}\big)\Big|_2\Big)\leq C(T),
\end{aligned}
\end{equation}
which, along with \eqref{ell-0-t-0}, implies  
\begin{equation*}
\Big|\zeta r^\frac{m}{2}D_\eta\big(D_\eta U_t+ \frac{mU_t}{\eta}\big)\Big|_2\leq C(T)(|\zeta r^\frac{m}{2} U_{tt}|_2+1).
\end{equation*}
Finally, this, together with  \eqref{D-tt} and Lemma \ref{im-1}, leads to $\eqref{D-txx/D-xxxx}_1$.

\smallskip
\textbf{2.2. Proof of $\eqref{D-txx/D-xxxx}_2$.} We first show that, for any $t\in[0,T]$,
\begin{equation}\label{high-vr-vrr}
\Big|\zeta r^\frac{m}{4}\big(D_\eta V, \frac{V}{\eta}\big) (t)\Big|_4+\Big|\zeta r^\frac{m}{2}\Big(D_\eta^2 V,D_\eta\big(\frac{V}{\eta}\big)\Big)(t)\Big|_2\leq C(T).
\end{equation}
Indeed, due to the fact that $\frac{m+3}{4}>\frac{m}{2}$ and Lemmas \ref{lemma-upper jacobi}, \ref{lemma-Vr-L2}--\ref{lemma-u-D3}, and \ref{hardy-inequality}, we have
\begin{equation}\label{vr-urr-l4}
\begin{aligned}
\Big|\zeta r^\frac{m}{4}\big(D_\eta V, \frac{V}{\eta}\big) \Big|_4&\leq C(T) \Big|r^\frac{m+3}{4}\Big(\zeta D_\eta V,\zeta \frac{V}{\eta},\zeta_r D_\eta V,\zeta_r \frac{V}{\eta},\zeta D_\eta^2 V,\zeta D_\eta\big(\frac{V}{\eta}\big)\Big)\Big|_2 \\
&\leq C(T)\Big(\Big|\zeta r^\frac{m}{2}\Big(D_\eta^2 V,D_\eta\big(\frac{V}{\eta}\big)\Big)\Big|_2+1\Big),\\
\Big|\chi_\frac{5}{8} r^\frac{m}{4}\Big(D_\eta^2 U,D_\eta\big(\frac{U}{\eta}\big)\Big)\Big|_4&\leq C(T)
\Big|\chi_\frac{5}{8} r^\frac{m+3}{4}\Big(D_\eta^2 U,D_\eta\big(\frac{U}{\eta}\big),D_\eta^3 U,D_\eta^2\big(\frac{U}{\eta}\big)\Big)\Big|_2\leq C(T).
\end{aligned}    
\end{equation}

Then, by \eqref{v-expression}--\eqref{eq:v}, we have 
\begin{equation*}
\begin{aligned}
&\,\Big(D_\eta\big(D_\eta V+\frac{mV}{\eta}\big)\Big)_t+\frac{A\gamma}{2\mu}\varrho^{\gamma-1}D_\eta\big(D_\eta V+\frac{mV}{\eta}\big)\\
&=\underline{-2D_\eta U D_\eta^2 V -m\big(D_\eta U+\frac{U}{\eta}\big)D_\eta\big(\frac{V}{\eta}\big)-D_\eta^2 U D_\eta V -mD_\eta\big(\frac{U}{\eta}\big)\frac{V}{\eta}}_{:=\mathrm{\mathrm{J}}_{12}}\\
&\quad +\underline{\frac{A\gamma(\gamma-1)}{4\mu^2}\varrho^{\gamma-1}(V-U)\Big(\big(3D_\eta U+\frac{mU}{\eta}\big)-\big(3D_\eta V+\frac{mV}{\eta}\big)\Big)}_{:=\mathrm{\mathrm{J}}_{13}}\\
&\quad +\underline{\frac{A\gamma}{2\mu}\varrho^{\gamma-1} D_\eta\big(D_\eta U+\frac{mU}{\eta}\big) -\frac{A\gamma(\gamma-1)^2}{8\mu^3}\varrho^{\gamma-1}(V-U)^3}_{:=\mathrm{\mathrm{J}}_{14}}.
\end{aligned}
\end{equation*}
Hence, solving the above ODE yields
\begin{equation}\label{DDV}
\begin{aligned}
& \ D_\eta\big(D_\eta V+\frac{mV}{\eta}\big)= e^{Q(t)}\big((v_0)_r+\frac{mv_0}{r}\big)_r +\int_0^t e^{Q(t)-Q(\tau)}\sum_{i=12}^{14}\mathrm{\mathrm{J}}_i\,\mathrm{d}\tau\\
&\implies \Big|\zeta r^\frac{m}{2}\Big(D_\eta\big(D_\eta V+\frac{mV}{\eta}\big)\Big)\Big|_2\leq \Big|\zeta r^\frac{m}{2}\big((v_0)_r+\frac{mv_0}{r}\big)_r\Big|_2 +\int_0^t \sum_{i=12}^{14}|\zeta r^\frac{m}{2}\mathrm{\mathrm{J}}_i|_2\,\mathrm{d}\tau,
\end{aligned}
\end{equation}
where 
\begin{equation*}
Q(t):=-\int_0^t\frac{A\gamma}{2\mu}\varrho^{\gamma-1}\,\mathrm{d}s.
\end{equation*}

By \eqref{total-Ein}, \eqref{vr-urr-l4}, Lemmas \ref{lemma-bound depth}, \ref{lemma-v Linfty in}, and \ref{lemma-Vr-L2}--\ref{lemma-u-D3}, and the H\"older inequality, we have
\begin{align}
&\begin{aligned}[b]
|\zeta r^\frac{m}{2}\mathrm{\mathrm{J}}_{12}|_2&\leq C_0\Big|\big(D_\eta U,\frac{U}{\eta}\big)\Big|_\infty\Big|\zeta r^\frac{m}{2}\Big(D_\eta^2 V,D_\eta\big(\frac{V}{\eta}\big)\Big)\Big|_2\\
&\quad + C_0\Big|\zeta r^\frac{m}{4}\big(D_\eta V,\frac{V}{\eta}\big)\Big|_4\Big|\chi_\frac{5}{8} r^\frac{m}{4}\Big(D_\eta^2 U,D_\eta\big(\frac{U}{\eta}\big)\Big)\Big|_4\\
&\leq  C(T)\Big(\Big|\zeta r^\frac{m}{2}\Big(D_\eta^2 V,D_\eta\big(\frac{V}{\eta}\big)\Big)\Big|_2+1\Big),
\end{aligned}\label{meixiang}\\
&\begin{aligned}
|\zeta r^\frac{m}{2}\mathrm{\mathrm{J}}_{13}|_2&\leq C_0 |\varrho|_\infty^{\gamma-1} |\chi_\frac{5}{8}(V,U)|_\infty\Big|\zeta r^\frac{m}{2}\big(D_\eta U,\frac{U}{\eta},D_\eta V,\frac{V}{\eta}\big)\Big|_2\leq C(T),\notag\\
|\zeta r^\frac{m}{2}\mathrm{\mathrm{J}}_{14}|_2&\leq C_0|\varrho|_\infty^{\gamma-1}\Big(\Big|\zeta r^\frac{m}{2}\Big(D_\eta^2 U,D_\eta\big(\frac{U}{\eta}\big)\Big)\Big|_2+|\chi_\frac{5}{8}(V,U)|_\infty^3\Big)\leq C(T),
\end{aligned}
\end{align}
and, for the initial data, by $\rho_0^\beta\sim 1-r$ and Lemma \ref{lemma-initial}, we have
\begin{equation}\label{chuzhi}
\begin{aligned}
\Big|\zeta r^\frac{m}{2}\big((v_0)_r+\frac{mv_0}{r}\big)_r\Big|_2&\leq \cE_{\mathrm{in}}(0,U) +C_0\|\zeta\nabla_{\boldsymbol{y}}^3\log\rho_0\|_{L^2(\Omega)}\\
&\leq C_0+C_0\sum_{j=1}^3 \|\nabla_{\boldsymbol{y}}^j \rho_0^\beta\|_{L^2(\Omega)}\|\nabla_{\boldsymbol{y}}\rho_0^\beta\|_{L^\infty(\Omega)}^{3-j}\leq C_0.
\end{aligned}
\end{equation}

Thus, plugging \eqref{meixiang}--\eqref{chuzhi} into $\eqref{DDV}_2$, together with the Gr\"onwall inequality, Lemma \ref{im-1}, and \eqref{vr-urr-l4}, implies
claim \eqref{high-vr-vrr}.

Now, applying $D_\eta$ to \eqref{ell-1}, along with \eqref{v-expression}, we have
\begin{equation*} 
\begin{aligned}
D_\eta^3\big(D_\eta U+ \frac{mU}{\eta}\big)&=\underline{\frac{1}{2\mu}D_\eta^2 U_t-\frac{1}{2\mu}(D_\eta^2V-3D_\eta^2U)D_\eta U-\frac{1}{\mu}D_\eta VD_\eta^2 U}_{:=\mathrm{J}_{15}}\\
&\quad \underline{-\frac{1}{2\mu}(V-U)D_\eta^3U +\frac{A\gamma}{4\mu^2} \varrho^{\gamma-1}(D_\eta^2 V-D_\eta^2 U)}_{:=\mathrm{J}_{16}}\\
&\quad +\underline{\frac{A\gamma(\gamma-1)}{8\mu^3} \varrho^{\gamma-1}(V-U)\Big(3(D_\eta V-D_\eta U)+\frac{\gamma-1}{2\mu} (V-U)^2\Big)}_{:=\mathrm{J}_{17}}.
\end{aligned}
\end{equation*}
Then, based on the above, we obtain from \eqref{total-Ein}, \eqref{high-vr-vrr}, $\eqref{vr-urr-l4}_2$, Lemmas \ref{lemma-bound depth}, \ref{lemma-v Linfty in}, and \ref{lemma-u-D3}, and the H\"older inequality that
\begin{equation}\label{4jie-0}
\begin{aligned}
|\zeta r^\frac{m}{2}\mathrm{J}_{15}|_2&\leq C_0|\zeta r^\frac{m}{2} D_\eta^2 U_{t}|_2+C_0|D_\eta U|_\infty\big|\zeta r^\frac{m}{2} (D_\eta^2 V,D_\eta^2 U)\big|_2\\
&\quad +C_0|\zeta r^\frac{m}{4}D_\eta V|_4\big|\chi_\frac{5}{8} r^\frac{m}{4}D_\eta^2 U\big|_4 \leq C(T)(|\zeta r^\frac{m}{2} D_\eta^2 U_{t}|_2 +1),\\
|\zeta r^\frac{m}{2}\mathrm{J}_{16}|_2&\leq C_0 |\chi_\frac{5}{8}(V,U)|_\infty |\zeta r^\frac{m}{2}D_\eta^3 U|_2+C_0|\varrho|_\infty^{\gamma-1}\big|\zeta r^\frac{m}{2} (D_\eta^2 V,D_\eta^2 U)\big|_2\leq C(T),\\
|\zeta r^\frac{m}{2}\mathrm{J}_{17}|_2&\leq C_0|\varrho|_\infty^{\gamma-1} \big|\chi_\frac{5}{8}(V,U)\big|_\infty\big(\big|\zeta r^\frac{m}{2} (D_\eta V,D_\eta U)\big|_2+ \big|\chi_\frac{5}{8}(V,U)\big|_\infty^2\big)\leq C(T).
\end{aligned}
\end{equation}

Next, multiplying \eqref{ell-0} by $\frac{1}{\eta}$ and then applying $D_\eta$ to the resulting equality, together with \eqref{v-expression}, yield
\begin{equation*} 
\begin{aligned}
D_\eta\Big(\frac{1}{\eta}D_\eta\big(D_\eta U+ \frac{mU}{\eta}\big) \Big)&=\underline{\frac{1}{2\mu}D_\eta\big(\frac{U_{t}}{\eta}\big)-\frac{1}{2\mu}D_\eta\big(\frac{V-U}{\eta}\big)D_\eta U-\frac{1}{2\mu} \frac{V-U}{\eta} D_\eta^2 U}_{:=\mathrm{J}_{18}}\\
&\quad +\underline{\frac{A\gamma(\gamma-1)}{8\mu^3} \varrho^{\gamma-1}\frac{(V-U)^2}{\eta}+\frac{A\gamma}{4\mu^2} \varrho^{\gamma-1}D_\eta\big(\frac{V-U}{\eta}\big)}_{:=\mathrm{J}_{19}}.
\end{aligned}
\end{equation*}
By a calculation similar to that of \eqref{4jie-0}, we can derive from the above and Lemma \ref{lemma-Vr-L2} that
\begin{equation*}
\begin{aligned}
|\zeta r^\frac{m}{2}\mathrm{J}_{18}|_2&\leq C_0\Big|\zeta_a r^\frac{m}{2} D_\eta\big(\frac{U_{t}}{\eta}\big)\Big|_2+C_0\Big|\big(D_\eta U,\frac{U}{\eta}\big)\Big|_\infty\Big|\zeta r^\frac{m}{2} \Big(D_\eta\big(\frac{V}{\eta}\big),D_\eta\big(\frac{U}{\eta}\big),D_\eta^2 U\Big)\Big|_2\\
&\quad +C_0\Big|\zeta r^\frac{m}{4}\frac{V}{\eta}\Big|_4\big|\chi_\frac{5}{8} r^\frac{m}{4}D_\eta^2 U\big|_4\leq C(T)\Big(\Big|\zeta r^\frac{m}{2} D_\eta\big(\frac{U_{t}}{\eta}\big)\Big|_2 +1\Big),  \\
|\zeta r^\frac{m}{2}\mathrm{J}_{19}|_2&\leq C_0|\varrho|_\infty^{\gamma-1}\big(\big|\chi_\frac{5}{8}(V,U)\big|_\infty+1\big)\Big|\zeta r^\frac{m}{2}\Big(\frac{V}{\eta},\frac{U}{\eta},D_\eta\big(\frac{V}{\eta}\big),D_\eta\big(\frac{U}{\eta}\big)\Big)\Big|_2\leq C(T).
\end{aligned}    
\end{equation*}

Therefore, it follows from the above and \eqref{4jie-0} that, for all $t\in[0,T]$,
\begin{equation}\label{proof1}
\begin{aligned}
&\,\Big|\zeta r^\frac{m}{2}D_\eta^3\big(D_\eta U+ \frac{mU}{\eta}\big)\Big|_2+
\Big|\zeta r^\frac{m}{2}D_\eta\Big(\frac{1}{\eta}D_\eta\big(D_\eta U+ \frac{m U}{\eta}\big)\Big)\Big|_2\\
&\leq  C(T)\Big(\Big|\zeta r^\frac{m}{2} \Big(D_\eta^2 U_{t},D_\eta\big(\frac{U_{t}}{\eta}\big)\Big)\Big|_2+1\Big),
\end{aligned}
\end{equation}
which, along with Lemma \ref{im-1} and $\eqref{D-txx/D-xxxx}_1$, leads to $\eqref{D-txx/D-xxxx}_2$.
\end{proof}

\subsection{Estimates of the velocity away from the origin}\label{sub93}
This subsection is devoted to establishing the following estimates:
\begin{Lemma}\label{ell-ex}
There exists a constant $C(T)>0$ such that
\begin{equation*}
\cE_{\mathrm{ex}}(t,U)+\int_0^t \cD_{\mathrm{ex}}(s,U)\,\mathrm{d}s\leq C(T) \qquad\text{for all $t\in[0,T]$}.
\end{equation*}
\end{Lemma}
This proof will be fulfilled by \S \ref{931}--\S\ref{934}.

\subsubsection{Some preliminaries}\label{931}
In what follows, we choose parameter $\varepsilon_0$ 
satisfying \eqref{varepsilon0}, that is,
\begin{equation*}
\begin{cases}
\displaystyle 0<\varepsilon_0 < \min\big\{\frac{3}{2}-\frac{1}{2\beta},\frac{\gamma-1}{\beta}-1,\frac{1}{2}\big\}&\displaystyle\quad \text{for }\beta\in \big(\frac{1}{3},\gamma-1\big),\\[6pt]
\displaystyle 0<\varepsilon_0< \min\big\{\frac{3}{2}-\frac{1}{2\beta},\frac{1}{2}\big\}&\quad\text{for }\beta=\gamma-1.
\end{cases}
\end{equation*}

To further simplify our calculations, we define the following quantities:
\begin{equation}\label{jianhua1}
\Lambda:=D_\eta (\varrho^\beta) ,\qquad \psi_\ell:=D_\eta\log\big(\frac{\eta}{r}\big),\qquad \psi_h:= D_\eta \log\eta_r.
\end{equation}
Clearly, \eqref{jianhua1}, together with \eqref{eq:eta} and \eqref{v-expression}, also yields
\begin{equation}\label{v-express2}
\Lambda=\frac{\varrho^\beta (\rho_0^\beta)_r}{\rho_0^\beta\eta_r}-\beta\varrho^\beta(m\psi_\ell+\psi_h),\qquad\Lambda=\frac{\beta}{2\mu}\varrho^\beta(V-U).
\end{equation}

Then, by Lemmas \ref{lemma-bound depth}, \ref{lemma-v Linfty ex}, 
\ref{lemma-v Linfty in}, and \ref{lemma-u-D3}, we can derive the following 
lemma directly: 
\begin{Lemma}\label{lemma-bound-lambda}
There exists a constant $C(T)>0$ such that
\begin{equation*}
|\Lambda(t)|_\infty\leq C(T) \qquad\,\, \text{for all $t\in[0,T]$}.
\end{equation*}
\end{Lemma}

Finally, we summarize the following estimates for later use, which follow directly from Lemmas \ref{lemma-u-0order}--\ref{lemma-u-D2} and \ref{lemma-u-D3}. 
\begin{Lemma}\label{lemma-qiexiang}
There exists a constant $C(T)>0$ such that, for all $t\in [0,T]$,
\begin{equation*}
\big|\chi^\sharp\rho_0^\frac{1}{2}(U,D_\eta U, D_\eta^2 U ,U_t,D_\eta U_{t})(t)\big|_2+|(U,D_\eta U)(t)|_\infty\leq C(T). 
\end{equation*}
\end{Lemma}

\subsubsection{Elliptic estimates}

The first lemma concerns the second-order elliptic estimate. This estimate follows directly from the fact that $(\frac{3}{2}-\varepsilon_0)\beta>\frac{1}{2}$ and Lemma \ref{lemma-qiexiang}.
\begin{Lemma}\label{lemma-u-ell-D2}
There exists a constant $C(T)>0$ such that 
\begin{equation*}
\big|\chi^\sharp\rho_0^{(\frac{3}{2}-\varepsilon_0)\beta}D_\eta^2 U(t)\big|_2\leq C(T) \qquad\text{for all $t\in [0,T]$}.
\end{equation*}
\end{Lemma}

Next, we can derive some refined weighted estimates for $U$.
\begin{Lemma}\label{lemma-u-ell-D2-refine}
There exists a constant $C(T)>0$ such that, for all $t\in [0,T]$,
\begin{equation*}
\begin{aligned}
\big|\chi^\sharp\rho_0^{-(\frac{1}{2}+\varepsilon_0)\beta}D_\eta U(t)\big|_2+\big|\chi^\sharp\rho_0^{(\frac{1}{2}-\varepsilon_0)\beta}D_\eta^2 U(t)\big|_2+|\chi^\sharp\rho_0^\beta D_\eta^2 U(t)|_\infty \leq C(T).
\end{aligned}
\end{equation*}
\end{Lemma}
\begin{proof}
We divide the proof into three steps.

\smallskip
\textbf{1. $L^2$-estimates of $\chi^\sharp \rho_0^{-(\frac{1}{2}+\varepsilon_0)\beta}D_\eta U$.}  First, since 
\begin{equation*}
\varepsilon_0<\frac{1}{2}, \qquad \big(\frac{3}{2}-\varepsilon_0\big)\beta>\frac{1}{2},
\end{equation*}
we can choose fixed $(\iota,\sigma)$ in Lemma \ref{lemma-time-space} such that
\begin{equation*}
\iota+\sigma=\big(\frac{1}{2}-\varepsilon_0\big)\beta,\qquad \iota\in \big(-\frac{\beta}{2},1+\frac{\beta}{2}\big), \qquad 0<\sigma<\min\big\{(1-\varepsilon_0)\beta,\big(\frac{3}{2}-\varepsilon_0\big)\beta-\frac{1}{2}\big\}.
\end{equation*}
Then it follows from Lemmas \ref{lemma-upper jacobi}, \ref{lemma-time-space}, \ref{lemma-qiexiang}, and \ref{hardy-inequality} that
\begin{equation}\label{Ur-refine}
\begin{aligned}
\big|\chi^\sharp\rho_0^{-(\frac{1}{2}+\varepsilon_0)\beta}D_\eta U\big|_2&\leq C(T)\big(1+\big|\chi^\sharp\rho_0^{(\frac{1}{2}-\varepsilon_0)\beta-\sigma}(D_\eta U,U_t)\big|_2\big)\\
&\leq C(T)\big(1+\big|\chi^\sharp\rho_0^{(\frac{3}{2}-\varepsilon_0)\beta-\sigma}(D_\eta U,D_\eta^2 U,U_t,D_\eta U_{t})\big|_2\big)\\
&\leq C(T)\big(1+\big|\chi^\sharp\rho_0^{\frac{1}{2}}(D_\eta U,D_\eta^2 U,U_t,D_\eta U_{t})\big|_2\big)\leq C(T).
\end{aligned}    
\end{equation}

\smallskip
\textbf{2. $L^2$-estimate of $\chi^\sharp \rho_0^{(\frac{1}{2}-\varepsilon_0)\beta}D_\eta^2 U$.} First, in view of \eqref{v-express2}, \eqref{ell-0-ex} can be rewritten as
\begin{equation}\label{Urr-J12-J14}
D_\eta^2 U=\underline{\frac{1}{2\mu} U_t-\frac{m}{\eta}\big(D_\eta U-\frac{U}{\eta}\big)}_{:=\mathrm{J}_{20}} \ \underline{-\frac{1}{\beta}\frac{\Lambda}{\varrho^\beta} D_\eta U}_{:=\mathrm{J}_{21}} +\underline{\frac{A\gamma}{2\mu\beta}\varrho^{\gamma-1-\beta}\Lambda}_{:=\mathrm{J}_{22}}.
\end{equation}
Then it follows from the facts that
\begin{equation*}
\rho_0^\beta\sim 1-r,\qquad \big(\frac{1}{2}-\varepsilon_0\big)\beta>0,\qquad   \big(\frac{3}{2}-\varepsilon_0\big)\beta>\frac{1}{2},\qquad -\big(\frac{1}{2}+\varepsilon_0\big)\beta+\gamma-1>0,
\end{equation*}
\eqref{Ur-refine}, and Lemmas \ref{lemma-lower bound jacobi}, \ref{lemma-upper jacobi}, \ref{lemma-bound-lambda}--\ref{lemma-u-ell-D2}, and \ref{hardy-inequality} that, for all $t\in [0,T]$,
\begin{equation}\label{J12-J14-1}
\begin{aligned}
\big|\chi^\sharp \rho_0^{(\frac{1}{2}-\varepsilon_0)\beta}\mathrm{J}_{20}\big|_2&\leq C(T)\big|\chi^\sharp \rho_0^{(\frac{1}{2}-\varepsilon_0)\beta}(U,D_\eta U,U_t)\big|_2\\
&\leq C(T)\big|\chi^\sharp \rho_0^{(\frac{3}{2}-\varepsilon_0)\beta}(U,D_\eta U,D_\eta^2 U,U_t,D_\eta U_{t})\big|_2\leq C(T),\\
\big|\chi^\sharp \rho_0^{(\frac{1}{2}-\varepsilon_0)\beta}(\mathrm{J}_{21},\mathrm{J}_{22})\big|_2& \leq C(T)\big|\chi^\sharp \rho_0^{-(\frac{1}{2}+\varepsilon_0)\beta}(D_\eta U,\rho_0^{\gamma-1})\big|_2 |\Lambda|_\infty \leq C(T).
\end{aligned}    
\end{equation}

Therefore, combining \eqref{Urr-J12-J14}--\eqref{J12-J14-1} leads to
\begin{equation}\label{9984} 
\big|\chi^\sharp \rho_0^{(\frac{1}{2}-\varepsilon_0)\beta}D_\eta^2 U\big|_2\leq  \big|\chi^\sharp \rho_0^{(\frac{1}{2}-\varepsilon_0)\beta}(\mathrm{J}_{20},\mathrm{J}_{21},\mathrm{J}_{22})\big|_2 \leq C(T) \qquad\text{for all $t\in[0,T]$}.
\end{equation}

\smallskip
\textbf{3. $L^\infty$-estimate of $\chi^\sharp \rho_0^\beta D_\eta^2 U$.} Using the fact that
\begin{equation*}
\frac{3}{2}\beta>\frac{1}{2},
\end{equation*}
we obtain from Lemmas \ref{lemma-lower bound jacobi}, \ref{lemma-upper jacobi}, \ref{lemma-qiexiang}, and \ref{hardy-inequality} that, for all $t\in [0,T]$,
\begin{equation*}
\begin{aligned}
|\chi^\sharp \rho_0^\beta\mathrm{J}_{20}|_\infty&\leq C(T)|\chi^\sharp \rho_0^\beta(U,D_\eta U,U_t)|_\infty\\
&\leq C(T)\big(1+\big|\chi^\sharp \rho_0^{\frac{3\beta}{2}}(U_t, D_\eta U_{t})\big|_2\big) \leq C(T),\\
|\chi^\sharp \rho_0^\beta(\mathrm{J}_{21},\mathrm{J}_{22})|_\infty&\leq C(T)(|D_\eta U|_\infty+|\rho_0|_\infty^{\gamma-1}) |\chi^\sharp \Lambda|_\infty\leq C(T).
\end{aligned}    
\end{equation*}
Therefore, \eqref{Urr-J12-J14},  together with the above, leads to
\begin{equation*}
|\chi^\sharp \rho_0^{\beta} D_\eta^2 U|_\infty\leq  |\chi^\sharp \rho_0^{\beta}(\mathrm{J}_{20},\mathrm{J}_{21},\mathrm{J}_{22})|_\infty \leq  C(T) \qquad \text{for all $t\in[0,T]$}.
\end{equation*}

This completes the proof of Lemma \ref{lemma-u-ell-D2-refine}.
\end{proof}

Besides, we have the following estimate for $D_{\eta}\Lambda$.
\begin{Lemma}\label{lemma-jianhua}
There exists a constant $C(T)>0$ such that, for all $(t,r)\in [0,T]\times \bar I$,
\begin{equation*}
\chi^\sharp|D_\eta \Lambda| \leq  C(T)\chi^\sharp\Big(\rho_0^\beta\big(\int_0^t|D_\eta^2 U|\,\mathrm{d}s\big)^2+\int_0^t|(D_\eta^2 U,\rho_0^\beta D_\eta^3 U)|\,\mathrm{d}s+1\Big).
\end{equation*}
\end{Lemma}
\begin{proof}
A direct calculation, together with \eqref{eq:eta} and \eqref{v-express2}, gives
\begin{equation}\label{recal}
\begin{aligned}
D_\eta \Lambda&= \frac{\varrho^\beta(\rho_0^\beta)_{rr}}{\eta_r^2\rho_0^{\beta}}- \frac{\varrho^\beta(\rho_0^\beta)_r}{\eta_r\rho_0^\beta}\big(2\beta m \psi_\ell+ (2\beta+1)\psi_h\big) \\
&\quad + \beta\varrho^\beta\big(\beta ( m \psi_\ell+\psi_h)^2-(m D_\eta\psi_\ell+D_\eta\psi_h)\big).
\end{aligned}    
\end{equation}

Next, by $\eqref{eq:VFBP-La-eta}_2$, we have
\begin{equation*}
\begin{aligned}
((\log\eta_r)_r)_t=(D_\eta U)_r\implies (\log\eta_r)_r=\int_0^t (D_\eta U)_r\,\mathrm{d}s,\\
\big(\log (\frac{\eta}{r})\big)_{tr}= \big(\frac{U}{\eta}\big)_r\implies \big(\log (\frac{\eta}{r})\big)_r=\int_0^t\big(\frac{U}{\eta}\big)_r\,\mathrm{d}s,
\end{aligned} 
\end{equation*}
which, along with \eqref{jianhua1} and the chain rules, implies 
\begin{equation}\label{recal20}
\psi_h=\frac{1}{\eta_r}\int_0^t \eta_r D_\eta^2 U\,\mathrm{d}s,\qquad \psi_\ell=\frac{1}{\eta_r}\int_0^t \eta_r D_\eta\big(\frac{U}{\eta}\big)\,\mathrm{d}s,
\end{equation}
and
\begin{equation}\label{recal2}
\begin{aligned}
D_\eta\psi_h&=-\frac{\psi_h}{\eta_r}\int_0^t \eta_r D_\eta^2 U \,\mathrm{d}s+\frac{1}{\eta_r^2}\int_0^t (\eta_r^2\psi_h D_\eta^2 U +\eta_r^2 D_\eta^3 U) \,\mathrm{d}s\\
&=- \psi_h^2+\frac{1}{\eta_r^2}\int_0^t \big(\eta_r \psi_h (\eta_r \psi_h)_t +\eta_r^2 D_\eta^3 U\big) \,\mathrm{d}s =-\frac{1}{2}\psi_h^2+\frac{1}{\eta_r^2}\int_0^t \eta_r^2 D_\eta^3 U \,\mathrm{d}s,\\
D_\eta \psi_\ell&=-\frac{\psi_h}{\eta_r}\int_0^t \eta_r D_\eta\big(\frac{U}{\eta}\big)\,\mathrm{d}s+\frac{1}{\eta_r^2}\int_0^t \eta_r^2 \Big(\psi_h D_\eta\big(\frac{U}{\eta}\big)+ D_\eta^2\big(\frac{U}{\eta}\big)\Big)\,\mathrm{d}s\\
&=-\psi_h\psi_\ell+\frac{1}{\eta_r^2}\int_0^t\Big(\eta_r^2 \psi_h D_\eta\big(\frac{U}{\eta}\big)+\eta_r^2 D_\eta^2\big(\frac{U}{\eta}\big)\Big)\,\mathrm{d}s.
\end{aligned}
\end{equation}
Hence, it follows from Lemmas \ref{lemma-lower bound jacobi}, 
\ref{lemma-upper jacobi}, and \ref{lemma-u-D3} that
\begin{equation}\label{psi-lh}
\begin{aligned}
\chi^\sharp|\psi_h|&\leq C(T)\int_0^t |D_\eta^2 U|\,\mathrm{d}s,\quad \chi^\sharp|\psi_\ell|\leq C(T),\\
\chi^\sharp|D_\eta\psi_h|&\leq C(T)\Big(\int_0^t \chi^\sharp|D_\eta^2 U|\,\mathrm{d}s\Big)^2+C(T)\int_0^t \chi^\sharp|D_\eta^3 U|\,\mathrm{d}s,\\
\chi^\sharp|D_\eta\psi_\ell|&\leq C(T)\int_0^t \chi^\sharp|D_\eta^2 U|\,\mathrm{d}s+C(T),
\end{aligned}
\end{equation}
which, along with the fact that $\rho_0^\beta\in H^3(\frac{1}{2},1)$ and Lemma \ref{non-vac}, yields the desired estimates of this lemma.
\end{proof}

Now, we can establish the third-order elliptic estimate for $U$.
\begin{Lemma}\label{lemma-u-D3-ell}
There exists a constant $C(T)>0$ such that, for all $t\in [0,T]$,
\begin{equation*}
\big|\chi^\sharp\rho_0^{(\frac{3}{2}-\varepsilon_0)\beta}(\rho_0^{-\beta}D_\eta\Lambda,D_\eta^3 U)(t)\big|_2\leq C(T).
\end{equation*}
\end{Lemma}
\begin{proof}
We divide the proof into two steps.

\smallskip
\textbf{1.} According to Lemmas \ref{lemma-upper jacobi}, \ref{lemma-u-ell-D2-refine}--\ref{lemma-jianhua}, and \ref{hardy-inequality}, the H\"older inequality, and the Minkowski integral inequality such as
\begin{equation*}
\Big|\int_0^t \chi^\sharp\rho_0^{(\frac{1}{2}-\varepsilon_0)\beta}D_\eta^2 U\,\mathrm{d}s\Big|_2\leq \int_0^t \big|\chi^\sharp\rho_0^{(\frac{1}{2}-\varepsilon_0)\beta}D_\eta^2 U\big|_2 \,\mathrm{d}s,
\end{equation*}
we have
\begin{equation}\label{p1p}
\begin{aligned}
\big|\chi^\sharp\rho_0^{(\frac{1}{2}-\varepsilon_0)\beta} D_\eta\Lambda\big|_2&\leq C(T)\Big(\int_0^t \big|\chi^\sharp\rho_0^{\beta}D_\eta^2 U\big|_\infty\,\mathrm{d}s\Big)\Big(\int_0^t \big|\chi^\sharp\rho_0^{(\frac{1}{2}-\varepsilon_0)\beta} D_\eta^2 U\big|_2\,\mathrm{d}s\Big)\\
&\quad +C(T)\Big(\int_0^t \big|\chi^\sharp\rho_0^{(\frac{1}{2}-\varepsilon_0)\beta}(D_\eta^2 U,\rho_0^{\beta} D_\eta^3 U)\big|_2 \,\mathrm{d}s+1\Big)\\
&\leq C(T)\Big(\int_0^t \big|\chi^\sharp\rho_0^{(\frac{3}{2}-\varepsilon_0)\beta} D_\eta^3 U\big|_2\,\mathrm{d}s+1\Big).
\end{aligned}
\end{equation}

\smallskip
\textbf{2. $L^2$-estimate of $\rho_0^{(\frac{3}{2}-\varepsilon_0)\beta}D_\eta^3 U$.}  Applying $D_\eta$ to \eqref{Urr-J12-J14}, together with \eqref{jianhua1}, gives
\begin{equation}\label{eq:3'}
\!\!\begin{aligned}
D_\eta^3 U&=\underline{\frac{1}{2\mu} D_\eta U_t -\frac{m}{\eta} \big(D_\eta^2 U-\frac{2D_\eta U}{\eta}+\frac{2U}{\eta^2}\big)}_{:=\mathrm{J}_{23}}\\
&\quad \underline{-\frac{1}{\beta}\frac{\Lambda}{\varrho^\beta} D_\eta^2 U -\frac{1}{\beta} \big(\frac{D_\eta\Lambda}{\varrho^\beta}- \frac{\Lambda^2}{\varrho^{2\beta}}\big) D_\eta U}_{:=\mathrm{J}_{24}}\!\!\!+\underline{\frac{A\gamma}{2\mu\beta}\varrho^{\gamma-1}\big( \frac{D_\eta\Lambda}{\varrho^{\beta}}+\frac{\gamma-1-\beta}{\beta}\frac{\Lambda^2}{\varrho^{2\beta}}\big)}_{:=\mathrm{J}_{25}}\!\!.
\end{aligned}
\end{equation}

Then it follows from \eqref{p1p}, the facts that
\begin{equation*}
\rho_0^\beta\sim1-r,\qquad\big(\frac{3}{2}-\varepsilon_0\big)\beta>\frac{1}{2},\qquad -\big(\frac{1}{2}+\varepsilon_0\big)\beta+\gamma-1>0,
\end{equation*}
and Lemmas \ref{lemma-lower bound jacobi}, \ref{non-vac}, and \ref{lemma-bound-lambda}--\ref{lemma-u-ell-D2-refine} that, for all $t\in[0,T]$,
\begin{align*}
&\big|\chi^\sharp\rho_0^{(\frac{3}{2}-\varepsilon_0)\beta}\mathrm{J}_{23}\big|_2\leq C(T)\big|\chi^\sharp \rho_0^{(\frac{3}{2}-\varepsilon_0)\beta}(U,D_\eta U,D_\eta^2 U,D_\eta U_{t})\big|_2 \leq C(T),\\
&\begin{aligned}
\big|\chi^\sharp\rho_0^{(\frac{3}{2}-\varepsilon_0)\beta}\mathrm{J}_{24}\big|_2&\leq C(T)\big|\chi^\sharp \rho_0^{(\frac{1}{2}-\varepsilon_0)\beta}D_\eta^2 U\big|_2|\Lambda|_\infty\\
&\quad +C(T)\big(|D_\eta U|_\infty\big|\chi^\sharp \rho_0^{(\frac{1}{2}-\varepsilon_0)\beta} D_\eta\Lambda\big|_2+ \big|\chi^\sharp \rho_0^{-(\frac{1}{2}+\varepsilon_0)\beta}D_\eta U\big|_2|\Lambda|_\infty^2\big)\\
&\leq C(T)\Big(\int_0^t \big|\chi^\sharp\rho_0^{(\frac{3}{2}-\varepsilon_0)\beta} D_\eta^3 U\big|_2\,\mathrm{d}s+1\Big),
\end{aligned}\\
&\begin{aligned}
\big|\chi^\sharp\rho_0^{(\frac{3}{2}-\varepsilon_0)\beta}\mathrm{J}_{25}\big|_2
&\leq C(T)\big(|\rho_0|_\infty^{\gamma-1}\big|\chi^\sharp \rho_0^{(\frac{1}{2}-\varepsilon_0)\beta} D_\eta\Lambda\big|_2+ \big|\chi^\sharp\rho_0^{-(\frac{1}{2}+\varepsilon_0)\beta+\gamma-1}\big|_2|\Lambda|_\infty^2\big)\\
&\leq C(T)\Big(\int_0^t \big|\chi^\sharp\rho_0^{(\frac{3}{2}-\varepsilon_0)\beta} D_\eta^3 U\big|_2\,\mathrm{d}s+1\Big).    
\end{aligned}
\end{align*}

Therefore, the above estimates, together with \eqref{eq:3'}, give
\begin{equation*}
\big|\chi^\sharp \rho_0^{(\frac{3}{2}-\varepsilon_0)\beta}D_\eta^3 U\big|_2\leq  \big|\chi^\sharp \rho_0^{(\frac{3}{2}-\varepsilon_0)\beta}(\mathrm{J}_{23},\mathrm{J}_{24},\mathrm{J}_{25})\big|_2 \leq C(T)\Big(\int_0^t \big|\chi^\sharp\rho_0^{(\frac{3}{2}-\varepsilon_0)\beta} D_\eta^3 U\big|_2\,\mathrm{d}s+1\Big),
\end{equation*}
which, along with the Gr\"onwall inequality, implies 
\begin{equation*}
\big|\chi^\sharp \rho_0^{(\frac{3}{2}-\varepsilon_0)\beta}D_\eta^3 U(t)\big|_2\leq C(T) \qquad \text{for all $t\in[0,T]$}.
\end{equation*}

Finally, the estimate on $D_\eta\Lambda$ follows from the above and \eqref{p1p}.
\end{proof}

\subsubsection{Dissipation estimates}\label{934}

We first establish the $L^2([0,T];L^2)$-estimate for $\chi^\sharp \rho_0^{(\frac{1}{2}-\varepsilon_0)\beta}D_{\eta}U_t$.
\begin{Lemma}\label{lemma-Utx-refine}
There exists a constant $C(T)>0$ such that
\begin{equation*}
\int_0^t \big|\chi^\sharp \rho_0^{(\frac{1}{2}-\varepsilon_0)\beta}D_{\eta}U_t\big|_2^2\,\mathrm{d}s\leq C(T) \qquad\text{for all $t\in[0,T]$}.
\end{equation*}
\end{Lemma}
\begin{proof}
First, applying $\partial_t$ to \eqref{zhongyao}, together with $\eqref{eq:VFBP-La}_1$, gives
\begin{equation}\label{9995}
\!\!\begin{aligned}
D_\eta U_t&=-\frac{A(\gamma-1)}{2\mu}\varrho^{\gamma-1}\big(D_\eta U+\frac{mU}{\eta}\big)+\frac{m}{\varrho}\big(D_\eta U+\frac{mU}{\eta}\big)\int_r^1 \tilde{r}^m\rho_0 \big(\frac{D_\eta U}{\eta^{m+1}}-\frac{U}{\eta^{m+2}}\big)\mathrm{d}\tilde{r} \\
&\quad +\frac{m}{\varrho}\int_r^1 \tilde{r}^m\rho_0 \big(\frac{D_\eta U_t-|D_\eta U|^2}{\eta^{m+1}}-\frac{(m+1)UD_\eta U+U_t}{\eta^{m+2}}+\frac{(m+2)U^2}{\eta^{m+3}}\big)\mathrm{d}\tilde{r}\\
&\quad -\frac{1}{2\mu \varrho}\Big(\big(D_\eta U+\frac{mU}{\eta}\big)\int_r^1 \!\frac{\tilde{r}^m}{\eta^m}\rho_0U_t\,\mathrm{d}\tilde{r}+\!\int_r^1\!\tilde{r}^m \rho_0\big(\frac{U_{tt}}{\eta^m}-\frac{mUU_t}{\eta^{m+1}}\big)\,\mathrm{d}\tilde{r}\Big)+|D_\eta U|^2.
\end{aligned}
\end{equation}
Then, arguing as in the derivation of \eqref{floww} in Lemma \ref{lemma-time-space} and using $\rho_0^\beta\sim 1-r$, Lemmas \ref{lemma-lower bound jacobi} and \ref{non-vac}, and the H\"older inequality, we obtain, for all $r\in [\frac{1}{2},1]$, for all $r\in [\frac{1}{2},1]$,
\begin{equation}\label{9966}
\begin{aligned}
|D_\eta U_t|&\leq C(T)(|U|+|D_\eta U|)\Big(\rho_0^{\gamma-1}+\frac{1}{\rho_0}\int_r^1 \rho_0(|U|+|D_\eta U|+|U_t|)\mathrm{d}\tilde{r}\Big) \\
&\quad +\frac{C(T)}{\rho_0}\int_r^1  \rho_0 \big(|U_t|+|D_\eta U_t|+|U_{tt}|+|U||U_t|+|U|^2+|D_\eta U|^2)\mathrm{d}\tilde{r}+|D_\eta U|^2\\
&\leq C(T)\big(|\rho_0^{\gamma-1}(U,D_\eta U)|+|\rho_0^{\frac{\beta-1}{2}}(U,D_\eta U)|\big|\chi^\sharp\rho_0^\frac{1}{2}(U,D_\eta U,U_t)\big|_2\big) \\
&\quad +C(T)\rho_0^{\frac{\beta-1}{2}}\big|\chi^\sharp\rho_0^\frac{1}{2}(U_t,D_\eta U_t,U_{tt},UU_t,U^2,|D_\eta U|^2)\big|_2+|D_\eta U|^2.
\end{aligned}
\end{equation}
Using the facts that
\begin{equation*}
\rho_0^\beta\sim1-r,\qquad \big(\frac{1}{2}-\varepsilon_0\big)\beta+\gamma-1>(1-\varepsilon_0)\beta-\frac{1}{2}>-\frac{\beta}{2},
\end{equation*}
we can obtain from \eqref{9966} and Lemmas \ref{lemma-u-0order}, \ref{lemma-qiexiang}, and \ref{lemma-u-ell-D2-refine} that
\begin{equation}\label{9997}
\begin{aligned}
\big|\chi^\sharp\rho_0^{(\frac{1}{2}-\varepsilon_0)\beta}D_\eta U_t\big|_2&\leq C(T)\big|\chi^\sharp\rho_0^{(1-\varepsilon_0)\beta-\frac{1}{2}}(U,D_\eta U)\big|_2\big(1+\big|\chi^\sharp\rho_0^\frac{1}{2}(U,D_\eta U,U_t)\big|_2\big) \\
&\quad +C(T)\big( \big|\chi^\sharp\rho_0^\frac{1}{2}(U_t,D_\eta U_t,U_{tt})\big|_2+|U|_\infty\big|\chi^\sharp\rho_0^\frac{1}{2}U_t\big|_2\big)\\
&\quad +C(T)|(U,D_\eta U)|_\infty\big(\big|\chi^\sharp\rho_0^\frac{1}{2}(U,D_\eta U)\big|_2+\big|\chi^\sharp\rho_0^{(\frac{1}{2}-\varepsilon_0)\beta}D_\eta U\big|_2\big)\\
&\leq C(T)\big(1+\big|\chi^\sharp\rho_0^{(1-\varepsilon_0)\beta-\frac{1}{2}}D_\eta U\big|_2+\big|\chi^\sharp\rho_0^\frac{1}{2}U_{tt}\big|_2\big).
\end{aligned}
\end{equation}

Next, note that, by Lemmas \ref{lemma-upper jacobi}, \ref{lemma-qiexiang}, \ref{lemma-u-D3-ell}, and \ref{hardy-inequality}, we have
\begin{equation}\label{9998}
\begin{aligned}
\big|\chi^\sharp\rho_0^{(1-\varepsilon_0)\beta-\frac{1}{2}}D_\eta U\big|_2\leq C(T)\big|\chi^\sharp\rho_0^{\frac{3\beta}{2}}(D_\eta U,D_\eta^2 U,D_\eta^3 U)\big|_2\leq C(T).
\end{aligned}
\end{equation}
Therefore, combining \eqref{9997}--\eqref{9998}, along with Lemma \ref{lemma-u-D3}, leads to the desired estimate.
\end{proof}

Now, we establish the $L^2([0,T];L^2)$-estimate for $\chi^\sharp \rho_0^{(\frac{3}{2}-\varepsilon_0)\beta}D_{\eta}^2U_t$.
\begin{Lemma}\label{lemma-ell-D4-ex1}
There exists a constant $C(T)>0$ such that
\begin{equation*}
\int_0^t\big|\chi^\sharp\rho_0^{(\frac{3}{2}-\varepsilon_0)\beta}D_\eta^2 U_t \big|_2^2\,\mathrm{d}s \leq C(T) \qquad\text{for all $t\in [0,T]$}.
\end{equation*}
\end{Lemma}
\begin{proof}
First, it follows from \eqref{eq:v} and Lemmas \ref{lemma-bound depth}, \ref{lemma-v Linfty ex}, and \ref{lemma-qiexiang} that, for all $t\in [0,T]$,
\begin{equation*} 
\big|\chi^\sharp\rho_0^{(\frac{3}{2}-\varepsilon_0)\beta} V_t\big|_2 \leq C_0|\varrho|_\infty^{\gamma-1} \big|\chi^\sharp\rho_0^{(\frac{3}{2}-\varepsilon_0)\beta} (V,U)\big|_2\leq C(T)(|\chi^\sharp \rho_0^\beta V|_\infty+|\chi^\sharp\rho_0^{\frac{1}{2}}U|_2)\leq C(T).
\end{equation*}
Then we obtain from the above, \eqref{ell-0-t-0}, $\varepsilon_0<\frac{1}{2}$, Lemmas \ref{lemma-lower bound jacobi}, \ref{lemma-v Linfty ex}, and \ref{lemma-qiexiang}--\ref{lemma-u-ell-D2-refine} that
\begin{align}
& \ \begin{aligned}
\big|\chi^\sharp\rho_0^{(\frac{3}{2}-\varepsilon_0)\beta}\mathrm{J}_{9}\big|_2&\leq C(T) (|\chi^\sharp(U,D_\eta U)|_\infty+1)\big|\chi^\sharp\rho_0^{(\frac{3}{2}-\varepsilon_0)\beta}(U_{tt},U,D_\eta U,D_\eta^2 U)\big|_2 \\
&\leq C(T)\big(\big|(r^m\rho_0)^\frac{1}{2} U_{tt}\big|_2+1\big),\notag
\end{aligned}\\
&\begin{aligned}
\big|\chi^\sharp\rho_0^{(\frac{3}{2}-\varepsilon_0)\beta}\mathrm{J}_{10}\big|_2&\leq C_0 |D_\eta U|_\infty\big|\chi^\sharp\rho_0^{(\frac{3}{2}-\varepsilon_0)\beta}(V_t,U_t)\big|_2\\
&\quad +C_0 |\chi^\sharp \rho_0^\beta(V,U)|_\infty\big|\chi^\sharp\rho_0^{(\frac{1}{2}-\varepsilon_0)\beta}D_\eta U_t\big|_2 \\
&\quad +C_0|\chi^\sharp \rho_0^\beta(V,U)|_\infty|D_\eta U|_\infty\big|\chi^\sharp\rho_0^{(\frac{1}{2}-\varepsilon_0)\beta} D_\eta U\big|_2\\
&\leq C(T)\big(\big|\chi^\sharp\rho_0^{(\frac{1}{2}-\varepsilon_0)\beta}D_\eta U_t\big|_2+1\big),
\end{aligned}\label{proof3}\\
&\begin{aligned}
\big|\chi^\sharp\rho_0^{(\frac{3}{2}-\varepsilon_0)\beta}\mathrm{J}_{11}\big|_2&\leq C_0|\varrho|_\infty^{\gamma-1}\big|\chi^\sharp\rho_0^{(\frac{3}{2}-\varepsilon_0)\beta}(V_t,U_t)\big|_2\\
&\quad +C(T)|\varrho|_\infty^{\gamma-1} |\chi^\sharp\rho_0^\beta(V,U)|_\infty\big|\chi^\sharp\rho_0^{(\frac{1}{2}-\varepsilon_0)\beta}\big|_2|(U,D_\eta U)|_\infty\leq C(T),\notag
\end{aligned}
\end{align}
which thus yields 
\begin{equation*}
\begin{aligned}
\big|\chi^\sharp\rho_0^{(\frac{3}{2}-\varepsilon_0)\beta}D_\eta^2 U_t \big|_2&\leq \big|\chi^\sharp\rho_0^{(\frac{3}{2}-\varepsilon_0)\beta}(\mathrm{J}_{9},\mathrm{J}_{10},\mathrm{J}_{11})\big|_2+C(T)\big|\chi^\sharp\rho_0^{(\frac{3}{2}-\varepsilon_0)\beta}(U_t,D_\eta U_t)\big|_2\\
&\leq C(T)\big(\big|(r^m\rho_0)^\frac{1}{2} U_{tt}\big|_2+\big|\chi^\sharp\rho_0^{(\frac{1}{2}-\varepsilon_0)\beta}D_\eta U_t\big|_2+1\big).
\end{aligned}
\end{equation*}

Then this, together with Lemmas  \ref{lemma-u-D3}, \ref{lemma-qiexiang} and \ref{lemma-Utx-refine},  leads to the desired estimate.
\end{proof}

Finally, we establish the $L^2([0,T];L^2)$-estimate for $\chi^\sharp \rho_0^{(\frac{3}{2}-\varepsilon_0)\beta}D_{\eta}^4U$.
\begin{Lemma}\label{lemma-ell-D4-ex2}
There exists a constant $C(T)>0$ such that
\begin{equation*}
\int_0^t\big|\chi^\sharp\rho_0^{(\frac{3}{2}-\varepsilon_0)\beta}D_\eta^4 U \big|_2^2\,\mathrm{d}s \leq C(T) \qquad\text{for all $t\in [0,T]$}.
\end{equation*}
\end{Lemma}
\begin{proof}
We divide the proof into two steps.

\smallskip
\textbf{1.} First, we can show that
\begin{equation}\label{control1}
|\chi^\sharp(\psi_\ell,\psi_h,D_\eta \Lambda)|_\infty\leq C(T)\Big(\int_0^t\big|\chi^\sharp \rho_0^{(\frac{3}{2}-\varepsilon_0)\beta} D_\eta^4 U\big|_2\,\mathrm{d}s+1\Big).
\end{equation}
Indeed, recall from Lemma \ref{lemma-jianhua} that
\begin{equation*}
\chi^\sharp|D_\eta \Lambda| \leq  C(T)\chi^\sharp\Big(\rho_0^\beta\big(\int_0^t|D_\eta^2 U|\,\mathrm{d}s\big)^2+\int_0^t|(D_\eta^2 U,\rho_0^\beta D_\eta^3 U)|\,\mathrm{d}s+1\Big).
\end{equation*}
Then \eqref{control1} can be directly obtained from the above, Lemma \ref{lemma-u-ell-D2-refine}, and the following controls due to Lemmas \ref{lemma-upper jacobi}, \ref{lemma-u-ell-D2}, \ref{lemma-u-D3-ell}, and \ref{sobolev-embedding}--\ref{hardy-inequality}:
\begin{equation}\label{control2}
\begin{aligned}
|\chi^\sharp(D_\eta^2 U,\rho_0^\beta D_\eta^3 U)|_\infty&\leq C(T)\big(|\chi^\sharp (D_\eta^2 U,D_\eta^3 U)|_{1}+|\chi^\sharp \rho_0^\beta D_\eta^3 U|_\infty\big)\\
&\leq C(T)\sum_{j=2}^4 \big|\chi^\sharp \rho_0^{(\frac{3}{2}-\varepsilon_0)\beta} D_\eta^j U\big|_2\leq C(T) \big(\big|\chi^\sharp \rho_0^{(\frac{3}{2}-\varepsilon_0)\beta} D_\eta^4 U\big|_2+1\big),
\end{aligned}
\end{equation}
where we have used  Lemma \ref{hardy-inequality} by taking $\vartheta=-\varepsilon_0>-\frac{1}{2}$ and $\varepsilon=\varepsilon_0$.

Next, we claim:
\begin{equation}\label{d2-lambda}
\big|\chi^\sharp\rho_0^{(\frac{1}{2}-\varepsilon_0)\beta}D_\eta^2\Lambda\big|_2\leq C(T)\Big(\int_0^t \big|\chi^\sharp\rho_0^{(\frac{3}{2}-\varepsilon_0)\beta}D_\eta^4 U\big|_2\,\mathrm{d}s+1\Big).  
\end{equation}
Indeed, it follows from the calculations in \eqref{recal} that
\begin{equation*}
\begin{aligned}
D_\eta^2 \Lambda&= \frac{\varrho^\beta (\rho_0^\beta)_{rrr}}{\eta_r^3\rho_0^{\beta}}- \frac{\varrho^\beta(\rho_0^\beta)_{rr}}{\eta_r^2\rho_0^{\beta}}\big(3\beta m \psi_\ell+ (3\beta+3)\psi_h\big)\\
&\quad + \frac{\varrho^\beta(\rho_0^\beta)_r}{\eta_r\rho_0^\beta} \big(3\beta^2 m^2 \psi_\ell^2+ 3\beta m(2\beta+1) \psi_\ell\psi_h+ (3\beta^2+3\beta+1)\psi_h^2\big)\\
&\quad- \frac{\varrho^\beta(\rho_0^\beta)_r}{\eta_r\rho_0^\beta} \big(3\beta m D_\eta\psi_\ell+ (3\beta+1)D_\eta\psi_h\big)\\ 
&\quad -\beta\varrho^\beta \big(\beta^2(m \psi_\ell+ \psi_h)^3 -3\beta(m \psi_\ell+ \psi_h)(m D_\eta\psi_\ell+ D_\eta\psi_h)+(m D_\eta^2\psi_\ell+ D_\eta^2\psi_h)\big),
\end{aligned}
\end{equation*}
which, along with the fact that $\rho_0^\beta\in H^3(\frac{1}{2},1)$, Lemmas \ref{lemma-lower bound jacobi} and \ref{non-vac}, and the Young inequality, implies that, for all $(t,r)\in [0,T]\times \bar I$,
\begin{equation*}
\begin{aligned}
\chi^\sharp|D_\eta^2 \Lambda|&\leq  C(T)\chi^\sharp\big(1+|(\rho_0^\beta)_{rrr}|+|(\psi_\ell,\psi_h)|^2+|(D_\eta\psi_\ell,D_\eta\psi_h)|\big)\\
&\quad+ C(T)\chi^\sharp \rho_0^\beta \big(|(\psi_\ell,\psi_h)|^3+|(\psi_\ell,\psi_h)||(D_\eta\psi_\ell,D_\eta\psi_h)|+|(D_\eta^2\psi_\ell,D_\eta^2\psi_h)|\big).
\end{aligned}
\end{equation*}
Moreover, a direct calculation, together with \eqref{recal20}--\eqref{recal2}, gives
\begin{equation*}
\begin{aligned}
D_\eta^2\psi_h&=- \psi_h D_\eta\psi_h-\frac{2\psi_h}{\eta_r^2}\int_0^t \eta_r^2 D_\eta^3 U \,\mathrm{d}s+\frac{2}{\eta_r^3}\int_0^t \eta_r^3 \psi_h D_\eta^3 U \,\mathrm{d}s+\frac{1}{\eta_r^3}\int_0^t \eta_r^3 D_\eta^4 U \,\mathrm{d}s,\\
D_\eta^2 \psi_\ell&=-\psi_hD_\eta\psi_\ell-\psi_\ell D_\eta\psi_h-\frac{2\psi_h}{\eta_r^2}\int_0^t \eta_r^2 \psi_h D_\eta\big(\frac{U}{\eta}\big)\,\mathrm{d}s\\
&\quad +\frac{1}{\eta_r^3}\int_0^t \eta_r^3(2\psi_h^2+ D_\eta\psi_h) D_\eta\big(\frac{U}{\eta}\big)\,\mathrm{d}s+\frac{3}{\eta_r^3}\int_0^t \eta_r^3 \psi_h D_\eta^2\big(\frac{U}{\eta}\big)\,\mathrm{d}s\\
&\quad-\frac{2\psi_h}{\eta_r^2}\int_0^t \eta_r^2 D_\eta^2\big(\frac{U}{\eta}\big)\,\mathrm{d}s +\frac{1}{\eta_r^3}\int_0^t \eta_r^3 D_\eta^3\big(\frac{U}{\eta}\big)\,\mathrm{d}s,
\end{aligned}
\end{equation*}
which, along with \eqref{psi-lh}, Lemmas \ref{lemma-lower bound jacobi}, \ref{lemma-upper jacobi}, and \ref{lemma-qiexiang}, and the Young inequality, yields
\begin{equation*}
\begin{aligned}
\chi^\sharp|(D_\eta^2\psi_h,D_\eta^2\psi_\ell)|&\leq C(T)\Big(\big(\int_0^t \chi^\sharp|D_\eta^2 U|\,\mathrm{d}s\big)^3+ \big(\int_0^t \chi^\sharp|D_\eta^2 U|\,\mathrm{d}s\big)\big(\int_0^t \chi^\sharp|D_\eta^3 U|\,\mathrm{d}s\big)\Big)\\
&\quad +C(T)\Big(\int_0^t \chi^\sharp|(D_\eta^2 U,D_\eta^3 U,D_\eta^4 U)|\,\mathrm{d}s+1\Big).
\end{aligned}
\end{equation*}
Therefore, it follows from the facts that $\rho_0^\beta\in H^3(\frac{1}{2},1)$ and $\varepsilon_0<\frac{1}{2}$, \eqref{control2}, Lemmas \ref{lemma-u-ell-D2-refine} and \ref{lemma-u-D3-ell}, and the Minkowski integral inequality that 
\begin{equation*}
\begin{aligned}
\big|\chi^\sharp\rho_0^{(\frac{1}{2}-\varepsilon_0)\beta}D_\eta^2\Lambda\big|_2&\leq  C(T)\Big(|\rho_0^\beta|_\infty^{\frac{1}{2}-\varepsilon_0}|\chi^\sharp(\rho_0^\beta)_{rrr}|_2+\int_0^t \big|\chi^\sharp\rho_0^{(\frac{1}{2}-\varepsilon_0)\beta}D_\eta^3 U\big|_2\,\mathrm{d}s+1\Big)\\
&\quad +C(T)\Big(\int_0^t |\chi^\sharp \rho_0^\beta D_\eta^2 U|_\infty\,\mathrm{d}s\Big) \Big(\int_0^t \big|\chi^\sharp\rho_0^{(\frac{1}{2}-\varepsilon_0)\beta}D_\eta^2 U\big|_2\,\mathrm{d}s\Big)\int_0^t |\chi^\sharp D_\eta^2 U|_\infty\,\mathrm{d}s\\
&\quad+ C(T)\Big(1+\int_0^t \big|\chi^\sharp\rho_0^{(\frac{3}{2}-\varepsilon_0)\beta}D_\eta^3 U\big|_2\,\mathrm{d}s\Big) \int_0^t |\chi^\sharp D_\eta^2 U|_\infty\,\mathrm{d}s \\
&\quad + C(T)\int_0^t \big|\chi^\sharp\rho_0^{(\frac{3}{2}-\varepsilon_0)\beta}D_\eta^4 U\big|_2\,\mathrm{d}s\leq C(T)\Big(\int_0^t \big|\chi^\sharp\rho_0^{(\frac{3}{2}-\varepsilon_0)\beta}D_\eta^4 U\big|_2\,\mathrm{d}s+1\Big).
\end{aligned}
\end{equation*}
This completes the proof of \eqref{d2-lambda}.

\smallskip
\textbf{2.} Rewrite $\eqref{eq:VFBP-La-eta}_1$ in view of \eqref{jianhua1} as
\begin{equation*}
\varrho^\beta D_\eta^2 U+\frac{1}{\beta} D_\eta(\varrho^\beta) D_\eta U=-m\varrho^\beta D_\eta\big(\frac{U}{\eta}\big)+\frac{1}{2\mu}\varrho^\beta U_t +\frac{A\gamma}{2\mu\beta} \varrho^{\gamma-1} D_\eta(\varrho^\beta).
\end{equation*}
Applying $D_\eta^2$ to both sides of the above, then dividing the resulting equality by $\varrho^\beta$, together with \eqref{jianhua1}, implies 
\begin{equation}\label{979'}
\mathcal{T}_{\mathrm{cross}}:=(D_\eta^3 U)_r+\big(\frac{1}{\beta}+2\big) \frac{(\rho_0^\beta)_r}{\rho_0^\beta} D_\eta^3 U=\sum_{i=26}^{29}\mathrm{J}_i,
\end{equation}
where
\begin{equation*}
\begin{aligned}
\mathrm{J}_{26}&:=(1+2\beta) \eta_r(m \psi_\ell+ \psi_h) D_\eta^3 U-\big(1+\frac{2}{\beta} \big)\frac{\eta_r}{\varrho^{\beta}}D_\eta \Lambda D_\eta^2 U-\frac{1}{\beta}\frac{\eta_r}{\varrho^{\beta}} D_\eta^2\Lambda D_\eta U,\\
\mathrm{J}_{27}&:=-m\frac{\eta_r}{\varrho^{\beta}}\Big(D_\eta \Lambda D_\eta\big(\frac{U}{\eta}\big)+2 \Lambda D_\eta^2\big(\frac{U}{\eta}\big)+ \varrho^\beta D_\eta^3\big(\frac{U}{\eta}\big)\Big),\\
\mathrm{J}_{28}&:=\frac{1}{2\mu}\frac{\eta_r}{\varrho^{\beta}}\big(D_\eta \Lambda U_t+2\Lambda D_\eta U_t+ \varrho^\beta D_\eta^2 U_t\big),\\
\mathrm{J}_{29}&:=\frac{A\gamma}{2\mu\beta} \eta_r \varrho^{\gamma-1-3\beta}\Big( \varrho^{2\beta} D_\eta^2 \Lambda+\frac{3(\gamma-1)}{\beta}\varrho^\beta \Lambda D_\eta \Lambda +\frac{(\gamma-1)(\gamma-1-\beta)}{\beta^2}\Lambda^3\Big).
\end{aligned}
\end{equation*}

For $\mathrm{J}_{26}$--$\mathrm{J}_{28}$, it follows from the facts that
\begin{equation*}
\rho_0^\beta\sim 1-r,\quad \big(\frac{3}{2}-\varepsilon_0\big)\beta>\frac{1}{2}>\varepsilon_0,
\end{equation*}
\eqref{control1}, \eqref{d2-lambda}, and Lemmas \ref{lemma-lower bound jacobi}, \ref{lemma-upper jacobi}, \ref{non-vac}, \ref{lemma-bound-lambda}--\ref{lemma-u-ell-D2-refine}, \ref{lemma-u-D3-ell}, and \ref{hardy-inequality} that
\begin{align}
&\begin{aligned}
\big|\chi^\sharp \rho_0^{(\frac{3}{2}-\varepsilon_0)\beta}\mathrm{J}_{26}\big|_2&\leq C(T)\big|\chi^\sharp(\psi_\ell,\psi_h)\big|_\infty\big|\chi^\sharp \rho_0^{(\frac{3}{2}-\varepsilon_0)\beta}D_\eta^3U\big|_2\\[0.5mm]
&\quad +C(T)\big(|\chi^\sharp D_\eta\Lambda|_\infty\big|\chi^\sharp \rho_0^{(\frac{1}{2}-\varepsilon_0)\beta} D_\eta^2U\big|_2\!+\!\big|\chi^\sharp \rho_0^{(\frac{1}{2}-\varepsilon_0)\beta}D_\eta^2 \Lambda\big|_2|D_\eta U|_\infty\big)\notag\\[0.5mm]
&\leq C(T)\Big(\int_0^t\big|\chi^\sharp \rho_0^{(\frac{3}{2}-\varepsilon_0)\beta} D_\eta^4 U\big|_2\,\mathrm{d}s+1\Big),
\end{aligned}\\[0.5mm]
&\begin{aligned}
\big|\chi^\sharp \rho_0^{(\frac{3}{2}-\varepsilon_0)\beta}\mathrm{J}_{27}\big|_2&\leq C(T)\big|\chi^\sharp \rho_0^{(\frac{1}{2}-\varepsilon_0)\beta}D_\eta \Lambda\big|_2|(U,D_\eta U)|_\infty\\
&\quad + C(T)(1+|\Lambda|_\infty) \big|\chi^\sharp \rho_0^{(\frac{3}{2}-\varepsilon_0)\beta} (U,D_\eta U,D_\eta^2 U,D_\eta^3 U)\big|_2\leq C(T),\notag
\end{aligned}\\
&\begin{aligned}[t]
\big|\chi^\sharp \rho_0^{(\frac{3}{2}-\varepsilon_0)\beta}\mathrm{J}_{28}\big|_2&\leq C(T)\big(|\chi^\sharp D_\eta \Lambda|_\infty\big|\chi^\sharp \rho_0^{(\frac{1}{2}-\varepsilon_0)\beta} U_t\big|_2+|\Lambda|_\infty\big|\chi^\sharp \rho_0^{(\frac{1}{2}-\varepsilon_0)\beta} D_\eta U_t\big|_2\big)\\
&\quad +C(T)\big|\chi^\sharp \rho_0^{(\frac{3}{2}-\varepsilon_0)\beta}D_\eta^2 U_t\big|_2\\
&\leq C(T)|\chi^\sharp D_\eta \Lambda|_\infty\big|\chi^\sharp \rho_0^{(\frac{3}{2}-\varepsilon_0)\beta} (U_t,D_\eta U_t)\big|_2\\
&\quad +C(T)\big|\chi^\sharp \rho_0^{(\frac{3}{2}-\varepsilon_0)\beta}(D_\eta U_t,D_\eta^2 U_t)\big|_2\\
&\leq C(T)\Big(\int_0^t\big|\chi^\sharp \rho_0^{(\frac{3}{2}-\varepsilon_0)\beta} D_\eta^4 U\big|_2\,\mathrm{d}s+\big|\chi^\sharp \rho_0^{(\frac{3}{2}-\varepsilon_0)\beta} D_\eta^2 U_t\big|_2+1\Big).
\end{aligned}\label{j20}
\end{align}
For $\mathrm{J}_{29}$, when $\beta<\gamma-1$, since $\varepsilon_0<\frac{\gamma-1}{\beta}-1$, we have 
\begin{equation}\label{894}
\quad \big(\frac{3}{2}-\varepsilon_0\big)\beta+\gamma-1-3\beta>-\frac{\beta}{2}. 
\end{equation}
Hence, we can similarly obtain 
\begin{equation}\label{j29}
\begin{aligned}
\big|\chi^\sharp \rho_0^{(\frac{3}{2}-\varepsilon_0)\beta}\mathrm{J}_{29}\big|_2&\leq C(T) |\rho_0|_\infty^{\gamma-1}\big|\chi^\sharp \rho_0^{(\frac{1}{2}-\varepsilon_0)\beta}D_\eta^2\Lambda\big|_2 \\
&\quad +C(T)|\rho_0|_\infty^{\gamma-1-\beta} |\Lambda|_\infty\big|\chi^\sharp \rho_0^{(\frac{1}{2}-\varepsilon_0)\beta}D_\eta \Lambda\big|_2\\
&\quad +\underline{C(T)(\gamma-1-\beta)\big|\chi^\sharp \rho_0^{(\frac{3}{2}-\varepsilon_0)\beta+\gamma-1-3\beta}\big|_2|\Lambda|_\infty^3}_{:=\mathrm{J}_{29}^\star}\\
&\leq C(T)\Big(\int_0^t\big|\chi^\sharp \rho_0^{(\frac{3}{2}-\varepsilon_0)\beta} D_\eta^4 U\big|_2\,\mathrm{d}s+1\Big).
\end{aligned}
\end{equation}
When $\beta=\gamma-1$, $\mathrm{J}_{29}$ can be estimated in the same way without using \eqref{894} and considering the term $\mathrm{J}_{29}^\star$. Therefore, it follows from \eqref{979'}--\eqref{j20} and \eqref{j29} that
\begin{equation*}
\begin{aligned}
\big|\zeta^\sharp \rho_0^{(\frac{3}{2}-\varepsilon_0)\beta}\mathcal{T}_{\mathrm{cross}}\big|_2&\leq \big|\chi^\sharp \rho_0^{(\frac{3}{2}-\varepsilon_0)\beta}(\mathrm{J}_{26},\mathrm{J}_{27},\mathrm{J}_{28},\mathrm{J}_{29})\big|_2\\
&\leq C(T)\Big(\int_0^t\big|\chi^\sharp \rho_0^{(\frac{3}{2}-\varepsilon_0)\beta} D_\eta^4 U\big|_2\,\mathrm{d}s+\big|\chi^\sharp \rho_0^{(\frac{3}{2}-\varepsilon_0)\beta} D_\eta^2 U_t\big|_2+1\Big).
\end{aligned}
\end{equation*}
Since
\begin{equation*}
\chi^\sharp=\zeta-\chi+\zeta^\sharp,\qquad\rho_0^\beta\sim 1-r,
\end{equation*}
we can obtain from Lemma \ref{lemma-upper jacobi} and Proposition \ref{prop2.1} in Appendix \ref{subsection2.2} that
\begin{equation*}
\begin{aligned}
\big|\chi^\sharp \rho_0^{(\frac{3}{2}-\varepsilon_0)\beta} D_\eta^4 U\big|_2&\leq C_0\cD_{\mathrm{in}}(t,U)^\frac{1}{2}+\big|\zeta^\sharp \rho_0^{(\frac{3}{2}-\varepsilon_0)\beta} D_\eta^4 U\big|_2\\
&\leq C(T)\big(\cD_{\mathrm{in}}(t,U)^\frac{1}{2}+ \big|\chi^\sharp \rho_0^{(\frac{3}{2}-\varepsilon_0)\beta} D_\eta^3 U\big|_2 + \big|\zeta^\sharp \rho_0^{(\frac{3}{2}-\varepsilon_0)\beta}\mathcal{T}_{\mathrm{cross}}\big|_2\big)\\
&\leq C(T)\Big(\int_0^t\big|\chi^\sharp \rho_0^{(\frac{3}{2}-\varepsilon_0)\beta} D_\eta^4 U\big|_2\,\mathrm{d}s+\big|\chi^\sharp \rho_0^{(\frac{3}{2}-\varepsilon_0)\beta} D_\eta^2 U_t\big|_2+\cD_{\mathrm{in}}(t,U)^\frac{1}{2}+1\Big),
\end{aligned}
\end{equation*}
which, along with Lemmas \ref{ell-inner}, \ref{lemma-u-D3-ell} and \ref{lemma-ell-D4-ex1},    and the Gr\"onwall inequality, leads to 
\begin{equation}\label{proof4}
\big|\chi^\sharp \rho_0^{(\frac{3}{2}-\varepsilon_0)\beta} D_\eta^4 U\big|_2
\leq C(T)\big( \big|\chi^\sharp \rho_0^{(\frac{3}{2}-\varepsilon_0)\beta} D_\eta^2 U_t\big|_2+\cD_{\mathrm{in}}(t,U)^\frac{1}{2}+1\big).
\end{equation}

Finally, this, together with Lemma \ref{lemma-ell-D4-ex1}, implies the desired result of this lemma.
\end{proof}

\subsection{Time-weighted estimates of the velocity}\label{sub94}

\begin{Lemma}\label{lemma-qiexiang-time weight}
There exists a constant $C(T)>0$ such that
\begin{equation*}
t\big|(r^m\rho_0)^\frac{1}{2}U_{tt}(t)\big|_2^2+\int_0^t s\Big|(r^m\rho_0)^\frac{1}{2}\big(D_\eta U_{tt},\frac{U_{tt}}{\eta}\big)\Big|_2^2\,\ds\leq C(T) \qquad\text{for all $t\in [0,T]$}.
\end{equation*}
\end{Lemma}

\begin{proof}
First, by \eqref{v-expression}, we rewrite \eqref{eq-2order-pre} as
\begin{equation}\label{eq-4order-pre}
\begin{aligned}
& \ r^m\rho_0U_{tt} -A\gamma \Big(\varrho^\gamma\big(D_\eta U+\frac{mU}{\eta}\big)\Big)_r\eta^m +mA \eta^m\eta_rD_\eta(\varrho^{\gamma})\frac{U}{\eta}\\
&=2\mu \Big(r^m\rho_0\frac{D_\eta U_t}{\eta_r}\Big)_r-2\mu m r^m\rho_0\frac{U_t}{\eta^2}-4\mu \Big(r^m\rho_0\frac{|D_\eta U|^2}{\eta_r} \Big)_r+4\mu m r^m\rho_0\frac{U^2}{\eta^3}.
\end{aligned}    
\end{equation}
Note that the following equality holds in view of $\eqref{eq:VFBP-La}_1$:
\begin{equation*}
(\eta^m\eta_rD_\eta(\varrho^\gamma))_t=-\frac{\gamma}{\beta}\eta^m\eta_r\varrho^{\gamma-\beta}\Lambda\big(\gamma D_\eta U+\frac{m(\gamma-1)U}{\eta}\big)-\gamma\eta^m\eta_r\varrho^\gamma D_\eta\big(D_\eta U+\frac{mU}{\eta}\big).
\end{equation*}
Hence, based on the above, we can apply $\partial_t$ to \eqref{eq-4order-pre} and obtain from \eqref{jianhua1} that
\begin{equation*}
\begin{aligned}
&\,r^m\rho_0U_{ttt} +A\gamma  \Big(\gamma\varrho^{\gamma}\big(D_\eta U+\frac{mU}{\eta}\big)^2+\varrho^\gamma\big(|D_\eta U|^2+\frac{mU^2}{\eta^2}\big)-\varrho^\gamma\big(D_\eta U_t+\frac{mU_t}{\eta}\big)\Big)_r\eta^m\\
&-A\gamma m \Big(\varrho^\gamma\big(D_\eta U+\frac{mU}{\eta}\big)\Big)_r\eta^{m-1}U\\
&-mA\gamma\eta^m\eta_r\varrho^{\gamma-\beta}\Big(\big(\frac{\gamma}{\beta} D_\eta U\frac{U}{\eta}+\frac{m(\gamma-1)+1}{\beta}\frac{U^2}{\eta^2}-\frac{U_t}{\beta\eta}\big)\Lambda+ \varrho^\beta D_\eta\big(D_\eta U+\frac{mU}{\eta}\big) \frac{U}{\eta}\Big)\\
&=2\mu \Big(r^m\rho_0\frac{D_\eta U_{tt}}{\eta_r}\Big)_r-2\mu m r^m\rho_0\frac{U_{tt}}{\eta^2} \\
&\quad +12\mu \Big(\frac{r^m\rho_0}{\eta_r}\big((D_\eta U)^3-D_\eta U D_\eta U_t\big) \Big)_r+12\mu m r^m\rho_0\Big(\frac{U_tU}{\eta^3}- \frac{U^3}{\eta^4}\Big).
\end{aligned}
\end{equation*}

Multiplying the above by $U_{tt}$ and integrating the resulting equality over $I$, we arrive at
\begin{equation}\label{dt-j15-17}
\frac{1}{2}\frac{\mathrm{d}}{\dt}\big|(r^m\rho_0)^\frac{1}{2}U_{tt}\big|_2^2+2\mu \big|(r^m\rho_0)^\frac{1}{2}D_\eta U_{tt}\big|_2^2+2\mu m\Big|(r^m\rho_0)^\frac{1}{2}\frac{U_{tt}}{\eta}\Big|_2^2=\sum_{i=30}^{32}\mathrm{J}_i,
\end{equation}
where
\begin{align*}
&\begin{aligned}
\mathrm{J}_{30}&:=A\gamma\int_0^1 \eta^m\eta_r\varrho^{\gamma}\Big(\gamma\big(D_\eta U+\frac{mU}{\eta}\big)^2+\big(|D_\eta U|^2+\frac{mU^2}{\eta^2}\big)\Big)\big(D_\eta U_{tt}+\frac{mU_{tt}}{\eta}\big)\,\mathrm{d}r\\
&\quad\,\, - A\gamma\int_0^1 \eta^m\eta_r\varrho^{\gamma}\big(D_\eta U_t+\frac{mU_t}{\eta}\big) \big(D_\eta U_{tt}+\frac{mU_{tt}}{\eta}\big)\,\mathrm{d}r\notag \\
&\quad\,\, -A\gamma m\int_0^1 \eta^m\eta_r\varrho^{\gamma}\big(D_\eta U+\frac{mU}{\eta}\big)\Big(\frac{U}{\eta}\big(D_\eta U_{tt}+\frac{mU_{tt}}{\eta}\big)+D_\eta\big(\frac{U}{\eta}\big)U_{tt} \Big)\,\mathrm{d}r,
\end{aligned}\\
&\begin{aligned}
\mathrm{J}_{31}&:= mA\gamma  \int_0^1 \eta^m\eta_r\varrho^{\gamma-\beta}\big(\frac{\gamma}{\beta} D_\eta U\frac{U}{\eta} + \frac{m(\gamma-1)+1}{\beta}\frac{U^2}{\eta^2}-\frac{U_t}{\beta\eta}\big)\Lambda U_{tt}\,\mathrm{d}r\\
&\quad\,\, + mA\gamma  \int_0^1 \eta^m\eta_r\varrho^{\gamma}D_\eta\big(D_\eta U + \frac{mU}{\eta}\big) \frac{U}{\eta}U_{tt} \,\mathrm{d}r,\notag
\end{aligned}\\
&\begin{aligned}
\mathrm{J}_{32}&:=12\mu \int_0^1 r^m\rho_0\Big((D_\eta U_t-|D_\eta U|^2)D_\eta U D_\eta U_{tt}+m\big(\frac{U_{t}}{\eta}-\frac{U^2}{\eta^2}\big)\frac{UU_{tt}}{\eta^2}\Big)\,\mathrm{d}r.\notag
\end{aligned}
\end{align*}
Then it follows from \eqref{v-expression}, Lemmas \ref{lemma-bound depth}--\ref{lemma-lower bound jacobi}, \ref{lemma-v Linfty ex}, \ref{lemma-v Linfty in}, \ref{lemma-u-D1}, and \ref{lemma-u-D3}--\ref{lemma-qiexiang}, and the H\"older and Young inequalities that
\begin{align}
&\begin{aligned}
\mathrm{J}_{30}&\leq C_0|\varrho|_\infty^{\gamma-1} \Big|\big(1,D_\eta U,\frac{U}{\eta}\big)\Big|_\infty \Big|(r^m\rho_0)^\frac{1}{2}\big(D_\eta U,\frac{U}{\eta},D_\eta U_{t},\frac{U_{t}}{\eta}\big)\Big|_2\Big|(r^m\rho_0)^\frac{1}{2}\big(D_\eta U_{tt},\!\frac{U_{tt}}{\eta}\big)\Big|_2\\
&\quad +C(T)|\varrho|_\infty^{\gamma-1} \Big|\big(D_\eta U,\frac{U}{\eta}\big)\Big|_\infty\Big(\Big|\zeta r^\frac{m}{2}D_\eta\big(\frac{U}{\eta}\big)\Big|_2+\big|\chi^\sharp\rho_0^\frac{1}{2}(U,D_\eta U)\big|_2\Big)\big|(r^m\rho_0)^\frac{1}{2}U_{tt}\big|_2\notag\\
&\leq C(T)\big(1+\big|(r^m\rho_0)^\frac{1}{2}U_{tt}\big|_2^2\big)+\frac{\mu}{8}\Big|(r^m\rho_0)^\frac{1}{2}\big(D_\eta U_{tt},\frac{U_{tt}}{\eta}\big)\Big|_2^2,  
\end{aligned}\\
&\begin{aligned}
\mathrm{J}_{31}&\leq C_0|\varrho|_\infty^{\gamma-1-\beta}|\Lambda|_\infty\Big(\Big|\big(D_\eta U,\frac{U}{\eta}\big)\Big|_\infty \Big|(r^m\rho_0)^\frac{1}{2}\frac{U}{\eta}\Big|_2+\Big|(r^m\rho_0)^\frac{1}{2}\frac{U_t}{\eta}\Big|_2\Big)\big|(r^m\rho_0)^\frac{1}{2}U_{tt}\big|_2\\
&\quad +C_0|\varrho|_\infty^{\gamma-1}\Big|\frac{U}{\eta}\Big|_\infty \Big|(r^m\rho_0)^\frac{1}{2}\Big(D_\eta^2 U,D_\eta\big(\frac{U}{\eta}\big)\Big)\Big|_2\big|(r^m\rho_0)^\frac{1}{2}U_{tt}\big|_2\\
&\leq C(T)\cE(t,U)\big|(r^m\rho_0)^\frac{1}{2}U_{tt}\big|_2\leq C(T)\big(1+\big|(r^m\rho_0)^\frac{1}{2}U_{tt}\big|_2^2\big),
\end{aligned}\label{j15-17}\\
&\begin{aligned}
\mathrm{J}_{32}&\leq C(T)\Big|\big(1,D_\eta U,\frac{U}{\eta}\big)\Big|_\infty^2 \Big|(r^m\rho_0)^\frac{1}{2}\big(D_\eta U_{t},\frac{U_{t}}{\eta},D_\eta U,\frac{U}{\eta}\big)\Big|_2\Big|(r^m\rho_0)^\frac{1}{2}\big(D_\eta U_{tt},\frac{U_{tt}}{\eta}\big)\Big|_2\notag\\
&\leq C(T)+\frac{\mu}{8}\Big|(r^m\rho_0)^\frac{1}{2}\big(D_\eta U_{tt},\frac{U_{tt}}{\eta}\big)\Big|_2^2.
\end{aligned}
\end{align}

Therefore, \eqref{dt-j15-17}--\eqref{j15-17} lead to
\begin{equation*}
\frac{\mathrm{d}}{\dt}\big(t\big|(r^m\rho_0)^\frac{1}{2}U_{tt}\big|_2^2\big)+\mu t\Big|(r^m\rho_0)^\frac{1}{2}\big(D_\eta U_{tt},\frac{U_{tt}}{\eta}\big)\Big|_2^2\leq C(T)\big(1+\big|(r^m\rho_0)^\frac{1}{2}U_{tt}\big|_2^2\big).
\end{equation*}
Then we integrate the above over $[\tau,t]$ and use Lemma \ref{lemma-u-D3} to obtain
\begin{equation}\label{tauk}
t\big|(r^m\rho_0)^\frac{1}{2}U_{tt}(t)\big|_2^2+\mu\int_\tau^t s\Big|(r^m\rho_0)^\frac{1}{2}\big(D_\eta U_{tt},\frac{U_{tt}}{\eta}\big)\Big|_2^2\,\ds\leq \tau\big|(r^m\rho_0)^\frac{1}{2}U_{tt}(\tau)\big|_2^2+C(T).    
\end{equation}
Thanks to Lemmas \ref{lemma-u-D3} and \ref{bjr}, we can find a sequence $\{\tau_k\}_{k=1}^\infty$ such that
\begin{equation*}
\tau_k\to 0, \quad\ \tau_k\big|(r^m\rho_0)^\frac{1}{2}U_{tt}(\tau_k)\big|^2_2\to 0 
\qquad\,\, \text{as $k\to\infty$}.
\end{equation*}
Taking $\tau=\tau_k$ in \eqref{tauk} and then letting $k\to\infty$, we finally obtain
\begin{equation}\label{tauk2}
t\big|(r^m\rho_0)^\frac{1}{2}U_{tt}(t)\big|_2^2+\int_0^t s\Big|(r^m\rho_0)^\frac{1}{2}\big(D_\eta U_{tt},\frac{U_{tt}}{\eta}\big)\Big|_2^2\,\ds\leq  C(T) \qquad \text{for all $t\in[0,T]$}.    
\end{equation}

This completes the proof.
\end{proof}

\begin{Lemma}\label{lemma-ell-time weight-in1}
There exists a constant $C(T)>0$ such that
\begin{equation*}
t\cD_{\mathrm{in}}(t,U) \leq C(T) \qquad\,\, \text{for all $t\in [0,T]$}.
\end{equation*}
\end{Lemma}

\begin{proof}
Note that, by \eqref{proof0}, \eqref{proof1}, and Lemma \ref{lemma-qiexiang-time weight}, we have 
\begin{equation*} 
\begin{aligned}
&\,\sqrt{t}\Big|\zeta r^\frac{m}{2}D_\eta\big(D_\eta U_t+ \frac{mU_t}{\eta}\big)\Big|_2\leq C(T)\big(1+\sqrt{t}\big|(r^m\rho_0)^\frac{1}{2} U_{tt}\big|_2\big)\leq C(T),\\
&\,\sqrt{t} \Big|\zeta r^\frac{m}{2}D_\eta^3\big(D_\eta U+ \frac{mU}{\eta}\big)\Big|_2+
\sqrt{t} \Big|\zeta r^\frac{m}{2}D_\eta\Big(\frac{1}{\eta}D_\eta\big(D_\eta U+ \frac{m U}{\eta}\big)\Big)\Big|_2\\
&\quad\leq  C(T)\Big(\sqrt{t}\Big|\zeta r^\frac{m}{2} \Big(D_\eta^2 U_{t},D_\eta\big(\frac{U_{t}}{\eta}\big)\Big)\Big|_2+1\Big).
\end{aligned}
\end{equation*}
Hence, this, together with Lemma \ref{im-1}, leads to the desired result of this lemma.
\end{proof}

\begin{Lemma}\label{lemma-ell-time weight-ex2}
There exists a constant $C(T)>0$ such that
\begin{equation*}
t\cD_{\mathrm{ex}}(t,U)\leq C(T) \qquad\text{for all $t\in [0,T]$}.
\end{equation*}
\end{Lemma}
\begin{proof}
Using the same argument as in the proofs of Lemmas \ref{lemma-Utx-refine} and \ref{lemma-qiexiang-time weight}, we can obtain
\begin{equation*}
\sqrt{t}\big|\chi^\sharp \rho_0^{(\frac{1}{2}-\varepsilon_0)\beta}D_\eta U_t\big|_2\leq C(T).
\end{equation*}
Hence, this, together with \eqref{proof3}, \eqref{proof4}, and Lemmas \ref{lemma-qiexiang-time weight}--\ref{lemma-ell-time weight-in1}, leads to
\begin{equation*}
\begin{aligned}
\sqrt{t}\big|\chi^\sharp \rho_0^{(\frac{3}{2}-\varepsilon_0)\beta}(D_\eta^2 U_t, D_\eta^4 U)\big|_2
&\leq C(T)\big(\sqrt{t}\big|\chi^\sharp \rho_0^{(\frac{3}{2}-\varepsilon_0)\beta} D_\eta^2 U_t\big|_2+\sqrt{t}\cD_{\rm in}(t,U)^\frac{1}{2}+\sqrt{t}\big)\\
&\leq C(T)\big(\sqrt{t}\big|(r^m\rho_0)^{\frac{1}{2}}U_{tt}\big|_2+1\big)\leq C(T).
\end{aligned}
\end{equation*}

This completes the proof.
\end{proof}

\section{Global-In-Time  Existence  and Uniqueness of the Classical Solutions}\label{Section-global}

The purpose of this section is to prove Theorem \ref{Theorem1.1} by combining the local existence  result with the global \textit{a priori} estimates established above.

Let $\overline{T}_*>0$ be the maximal existence time of the unique local classical solution $(U,\eta)$ obtained in Theorem \ref{local-Theorem1.1}. We claim that $\overline{T}_*=\infty$.

Suppose, on the contrary, that $\overline{T}_*<\infty$. For every $T\in(0,\overline{T}_*)$, we apply Lemmas \ref{lemma-lower bound jacobi}, \ref{lemma-upper jacobi}, \ref{ell-inner}--\ref{ell-ex}, and \ref{lemma-ell-time weight-in1}--\ref{lemma-ell-time weight-ex2} on $[0,T]$. Since $T<\overline{T}_*$, the constants in these estimates can be bounded in terms of $\overline{T}_*$. Therefore,
\begin{equation}\label{811}
\begin{aligned}
&\sup_{0\leq t<\overline{T}_*}\big(\cE(t,U)+t\cD(t,U)\big)+\int_0^{\overline{T}_*}\cD(s,U)\,\mathrm{d}s\leq C(\overline{T}_*),\\
&(\eta_r,\frac{\eta}{r})(t,r)\in[C(\overline{T}_*)^{-1},C(\overline{T}_*)]
\qquad\text{for all $(t,r)\in[0,\overline{T}_*)\times\bar I$}.
\end{aligned}
\end{equation}
Here and below, $C(\overline{T}_*)\in(1,\infty)$ depends only on $(n,\mu,\gamma,A,\beta,\varepsilon_0,\rho_0,u_0,\cK_1,\cK_2,\overline{T}_*)$ and may be chosen nondecreasing in $\overline{T}_*$.

It follows from Lemma \ref{lemma-gaowei} in Appendix \ref{AppB} that \eqref{811} also implies
\begin{equation}\label{811'}
\sup_{0\leq t<\overline{T}_*}\big(\mathring\cE(t,U)+t\mathring\cD(t,U)\big)
+\int_0^{\overline{T}_*}\mathring\cD(s,U)\,\mathrm{d}s
\leq C(\overline{T}_*).
\end{equation}

Fix any $\tau\in[\frac{\overline{T}_*}{2},\overline{T}_*)$ and regard $(u_\star,\eta_\star)(r):=(U,\eta)(\tau,r)$ as the initial data for the problem restarted at $t=\tau$. We retain the original Lagrangian label $r$ and the same reference density $\rho_0$ in the restarted problem. The estimates in \eqref{811} show that
\begin{equation*}
\big((\eta_\star)_r,\frac{\eta_\star}{r}\big)\in[C(\overline{T}_*)^{-1},C(\overline{T}_*)]
\quad\text{for all $r\in\bar I$},\qquad  \cE(\tau,U) \leq C(\overline{T}_*),
\end{equation*}
with constants independent of $\tau$. We next verify the weighted regularity of $\eta_\star$. Since  
\begin{equation*}
\partial_r^k\eta_\star=\delta_{1k}+\int_0^\tau\partial_r^kU(s,r)\,\mathrm{d}s 
\quad\text{for } 1\leq k\leq4,\qquad \frac{\eta_\star}{r}=1+\int_0^\tau\frac{U(s,r)}{r}\,\mathrm{d}s,
\end{equation*}
we can deduce from \eqref{811}--\eqref{811'} and the H\"older inequality that
\begin{equation*}
\mathring\cE(\eta_\star)\leq C(\overline{T}_*)\Big(1+\mathring\cE(\tau,U)+\tau\int_0^\tau\mathring\cE(s,U)\,\mathrm{d}s+\int_0^\tau\mathring\cD(s,U)\,\mathrm{d}s\Big)
\leq C(\overline{T}_*).
\end{equation*}

The temporal derivatives entering $\cE(\tau,U)$ are the actual temporal derivatives of the existing classical solution. Since $(U,\eta)$ satisfy \eqref{eq:VFBP-La-eta} on $[0,\overline{T}_*)\times\bar I$, they also satisfy problem \eqref{general-initial-problem}, and the boundary conditions at $t=\tau$. Moreover, set $\delta_*=C(\overline{T}_*)^{-1}$, $\delta^*=C(\overline{T}_*)$, and $\mathfrak{M}_0=2C(\overline{T}_*)$. Then, by the local theory stated in Remark \ref{remk31}, there exists $T_0>0$, depending on $C(\overline{T}_*)$ and the fixed parameters but independent of $\tau$, such that the problem with initial data $(u_\star,\eta_\star)(r)$ admits a unique classical solution on $[\tau,\tau+T_0]\times\bar I$. 

Choose $\tau<\overline{T}_*$ sufficiently close to $\overline{T}_*$ so that $\overline{T}_*-\tau<T_0$. For every $S\in(\tau,\overline{T}_*)$, the restarted solution and the original solution solve the same problem on $[\tau,S]\times\bar I$ with the same data at $t=\tau$. The uniqueness assertion in Theorem \ref{local-Theorem1.1} and Remark \ref{remk31} implies that the two solutions coincide on this interval. The restarted solution therefore extends $(U,\eta)$ beyond $\overline{T}_*$, contradicting the maximality of $\overline{T}_*$. We conclude that $\overline{T}_*=\infty$.

Finally, for every finite $T>0$, the local regularity conclusions of Theorem \ref{local-Theorem1.1} and Remark \ref{remk31}, applied on overlapping time intervals, ensure that $(U,\eta)$ is a classical solution on $[0,T]\times\bar I$ satisfying \eqref{N111}. Uniqueness follows from the corresponding local uniqueness result. This completes the proof of Theorem \ref{Theorem1.1}.

\section{Local-In-Time Existence  and Uniqueness of the Classical Solutions}\label{Section-local}

In this section, we establish the local existence  and uniqueness of \textbf{IBVP} \eqref{eq:VFBP-La-eta} stated in Theorem \ref{local-Theorem1.1}, which  is highly non-trivial due to the strong degeneracy at the vacuum boundary.

In what follows, we denote by $\cH^1_{\mathrm{w}}(J)$ the space of all functions $f$ satisfying $(f,f_r,\frac{f}{r})\in L^2_\mathrm{w}(J)$ for some interval $J=(0,a)$ with $a\in (0,1)$:
\begin{equation*}
\cH^1_{\mathrm{w}}(J):=\big\{f:\,\|f\|_{\cH^1_{\mathrm{w}}(J)}<\infty\big\}\qquad\,\,
\text{with} \,\,\,\, \|f\|_{\cH^1_{\mathrm{w}}(J)}^2:=\int_J \mathrm{w}\Big(f^2+f_r^2+\frac{mf^2}{r^2}\Big)\,\mathrm{d}r,
\end{equation*}
where $\mathrm{w}=\mathrm{w}(r)\ge 0$ is a weight function on $J$, and we let $\cH^{-1}_{\mathrm{w}}(J):=(\cH^1_{\mathrm{w}}(J))^*$. In particular, if $J=I$, we simply write $\cH^1_{\mathrm{w}}=\cH^1_{\mathrm{w}}(I)$.

Before starting the proof, we first give an important property for space $\cH^1_{\mathrm{w}}$, which will be used in the subsequent analysis.
\begin{Proposition}\label{prop-bijin}
Let $\mathrm{w}=r^m\rho_0^c$ with $c\geq 0$. 
Then $\cH^1_{\mathrm{w}}$ is a reflexive separable Banach space. 
Moreover, for any $f\in \cH^1_{\mathrm{w}}$, there exists 
a sequence $\{f^\varepsilon\}_{\varepsilon>0}\subset C^\infty(\bar I)\cap\cH^1_{r^m}$ such that 
\begin{equation*}
\|f^\varepsilon-f\|_{\cH^1_\mathrm{w}}\to 0 \qquad\text{as $\varepsilon\to 0$}.
\end{equation*}
\end{Proposition}
\begin{proof}
The first assertion follows directly from Lemma \ref{W-space} in Appendix \ref{appendix A}. To derive the convergence, we first define $\boldsymbol{f}(\boldsymbol{y}):=f(r)\frac{\boldsymbol{y}}{r}$. Clearly, $\boldsymbol{f}$ is a spherically symmetric vector function. Moreover,  according to  Lemma \ref{lemma-initial} in Appendix \ref{appb} and the fact that $\rho_0^\beta\sim 1-r$, $\boldsymbol{f}\in H^1(B_a)$ for any $a\in (0,1)$, 
where $B_a:=\{\boldsymbol{y}: \,|\boldsymbol{y}|<a\}$. 

Now, we claim that there exists a sequence of spherically symmetric vector functions $\{\boldsymbol{f}_\flat^\varepsilon\}_{\varepsilon>0}\subset C^\infty(B_\frac{3}{4})$, with the form: $\boldsymbol{f}_\flat^\varepsilon(\boldsymbol{y})
=f^\varepsilon_\flat(r)\frac{\boldsymbol{y}}{r}$, such that
\begin{equation}\label{cliam121}
\boldsymbol{f}_\flat^\varepsilon\to \boldsymbol{f}\qquad \text{in $H^1(B_\frac{3}{4})$ as $\varepsilon\to 0$}.
\end{equation}
To obtain this, let $\{\omega_\epsilon(\boldsymbol{y})\}_{\varepsilon>0}$ be the standard spherically symmetric mollifier defined on $\mathbb{R}^n$. 
Then, due to the standard theory of regularization, $\boldsymbol{f}_\flat^\varepsilon(\boldsymbol{y}):=(\boldsymbol{f}*\omega_\epsilon)(\boldsymbol{y})$ (with $\varepsilon$ small) satisfies \eqref{cliam121}, 
and it suffices to check the spherical symmetry of $\boldsymbol{f}_\flat^\varepsilon(\boldsymbol{y})$. 
Thanks to Lemma \ref{duichen-dengjia}, this is equivalent to showing that $\boldsymbol{f}_\flat^\varepsilon(\mathcal{O}\boldsymbol{y})=(\mathcal{O}\boldsymbol{f}_\flat^\varepsilon)(\boldsymbol{y})$ for any matrix $\mathcal{O}\in \mathrm{O}(n)$. In fact, we have
\begin{equation*}
\boldsymbol{f}_\flat^\varepsilon(\mathcal{O}\boldsymbol{y})= \int_{\Omega} \boldsymbol{f}(\mathcal{O}\boldsymbol{y}-\boldsymbol{z})\omega_\epsilon(\boldsymbol{z})\,\mathrm{d}\boldsymbol{z} \qquad\text{for $0<\varepsilon<\frac{1}{100}$ and $\boldsymbol{y}\in B_\frac{3}{4}$}.
\end{equation*}
Changing the coordinate $\boldsymbol{z}\mapsto \mathcal{O}\boldsymbol{z}$, along with $|\mathcal{O}\boldsymbol{y}|=|\boldsymbol{y}|$ and $|\det \mathcal{O}|=1$, gives
\begin{equation*}
\begin{aligned}
\boldsymbol{f}_\flat^\varepsilon(\mathcal{O}\boldsymbol{y})&= \int_{\Omega} \boldsymbol{f}(\mathcal{O}\boldsymbol{y}-\mathcal{O}\boldsymbol{z})\omega_\epsilon(\mathcal{O}\boldsymbol{z})|\det \mathcal{O}|\,\mathrm{d}\boldsymbol{z} = \int_{\Omega} \boldsymbol{f}(\mathcal{O}(\boldsymbol{y}-\boldsymbol{z}))\omega_\epsilon(\boldsymbol{z}) \,\mathrm{d}\boldsymbol{z}\\
&= \int_{\Omega} (\mathcal{O} \boldsymbol{f})(\boldsymbol{y}-\boldsymbol{z})\omega_\epsilon(\boldsymbol{z}) \,\mathrm{d}\boldsymbol{z} =\Big(\mathcal{O}\big( \int_{\Omega} \boldsymbol{f}(\cdot-\boldsymbol{z})\omega_\epsilon(\boldsymbol{z}) \,\mathrm{d}\boldsymbol{z}\big)\Big)(\boldsymbol{y})=(\mathcal{O}\boldsymbol{f}_\flat^\varepsilon)(\boldsymbol{y}).
\end{aligned}
\end{equation*}
This completes the proof of the claim. 

Consequently, it follows from \eqref{cliam121} and Lemma \ref{lemma-initial} that 
\begin{equation}\label{jin}
\zeta f^\varepsilon_\flat\in \cH^1_{r^m},\qquad\,\, 
\|\zeta f^\varepsilon_\flat-\zeta f\|_{\cH^1_\mathrm{w}}\to 0 \qquad\text{as $\varepsilon\to 0$}.
\end{equation}

On the other hand, it follows from Lemma \ref{W-space} that there exists a smooth sequence $\{f_\sharp^\varepsilon\}_{\varepsilon>0}\subset C^\infty[\frac{1}{3},1]$ such that
\begin{equation}\label{yuan}
\|\zeta^\sharp f^\varepsilon_\sharp-\zeta^\sharp f\|_{\cH^1_\mathrm{w}}\leq C\|\zeta^\sharp f^\varepsilon_\sharp-\zeta^\sharp f\|_{1,\rho_0^c}\to 0 \qquad\text{as $\varepsilon\to 0$}.
\end{equation}
Therefore, defining $f^\varepsilon:=\zeta f_\flat^\varepsilon+\zeta^\sharp f_\sharp^\varepsilon$, we see that $f^\varepsilon\in C^\infty(\bar I)\cap \cH^1_{r^m}$. Then we can obtain from \eqref{jin}--\eqref{yuan} that
\begin{equation*}
\|f^\varepsilon-f\|_{\cH^1_\mathrm{w}}\leq \|\zeta f^\varepsilon_\flat-\zeta f\|_{\cH^1_\mathrm{w}}+\|\zeta^\sharp f^\varepsilon_\sharp-\zeta^\sharp f\|_{\cH^1_\mathrm{w}}\to 0 \qquad\text{as $\varepsilon\to 0$}.
\end{equation*}

This completes the proof.
\end{proof}

The rest of this section is organized as follows:
\begin{itemize}
\item[\S\ref{Subsection9.1}:] Establish the global existence  of the linearized problem via the Galerkin scheme;
\item[\S\ref{subsection9.2}:] Establish the uniform {\it a priori} estimates for the linearized problem;
\item[\S\ref{subsection9.3}:] Establish the local existence of the nonlinear problem via the Picard iteration and the uniqueness via the energy method;
\item[\S\ref{subsection9.4}:] Extend the identity initial flow map to a general one.
\end{itemize}

\subsection{Linearization and global existence  of the linearized problem}\label{Subsection9.1} 

In \S\ref{Subsection9.1},  $C\in (1,\infty)$ denotes a generic constant depending only on $(n,\mu,A,\gamma,\beta,\varepsilon_0,\rho_0,u_0,\cK_1,\cK_2)$, and $C(l_1,\cdots\!,l_k)\in (1,\infty)$ a generic  constant depending on $C$
and additional  parameters $(l_1,\cdots\!,l_k)$, which may be different at each occurrence.

We first linearize problem \eqref{eq:VFBP-La-eta}, and then establish global existence  of classical solutions of the linearized problem 
via a modified Galerkin scheme. 
Specifically, we begin by studying the following linearized problem in $[0, T] \times I$ (multiply $\eqref{eq:VFBP-La-eta}_1$ by $\eta^m\eta_r$ and then replace $\eta\mapsto \bar\eta$): 
\begin{equation}\label{lp}
\begin{cases}
\displaystyle r^m\rho_0U_t +A\Big(\frac{r^{\gamma m}\rho_0^\gamma }{\bar\eta_r(\bar\eta^m\bar\eta_r)^{\gamma-1}}\Big)_r-Am \frac{r^{\gamma m}\rho_0^\gamma }{\bar\eta(\bar\eta^m\bar\eta_r)^{\gamma-1}} =2\mu \big(r^m\rho_0\frac{U_r}{\bar\eta_r^2}  \big)_r -2\mu mr^m\rho_0\frac{U}{\bar\eta^2},\\[9pt]
U|_{t=0}= u_0 \qquad \text{in }I,
\end{cases}
\end{equation}
where $\bar \eta $ stands for the flow map corresponding to $\bar U$:
\begin{equation}\label{given-flow}
\bar \eta (t,r)=r+\int_0^t \bar U(s,r)\,\ds,\qquad \bar \eta(0,r)=r,
\end{equation}
and $\bar U$ is a given function satisfying $\bar U(0,r)=u_0(r)$ for $r\in I$ 
and, for any $T>0$,
\begin{equation}\label{given}
\begin{aligned}
&\bar\cE(t,\bar U)+t\bar\cD(t,\bar U)\in L^\infty(0,T),\qquad \bar\cD (t,\bar U)\in L^1(0,T),\\
&\big(\bar U,\bar U_r,\frac{\bar U}{r}\big)\in C([0,T];C(\bar I)),\qquad \big(\bar U_{rr},(\frac{\bar U}{r})_r,\bar U_t\big)\in C((0,T];C(\bar I)),\\[4pt]
&\bar U|_{r=0}=\bar U_r|_{r=1}=0\qquad \,\text{for $t\in(0,T]$},
\end{aligned} 
\end{equation}
where $(\bar\cE,\bar\cD)(t,f)$ are defined in the same way as $(\cE,\cD)(t,f)$ in \eqref{E-1} and \eqref{D-1}, except with $\eta$ replaced by $\bar\eta$. Besides, 
we define $(\bar\cE_{\mathrm{in}},\bar\cE_{\mathrm{ex}},\bar\cD_{\mathrm{in}},\bar\cD_{\mathrm{ex}})(t,f)$ in the similar manner to $(\bar\cE,\bar\cD)(t,f)$.

Clearly, \eqref{given-flow}, together with \eqref{given}, also implies the regularity of $\bar\eta$, that is,
\begin{equation}\label{given-bareta}
\begin{aligned}
&\zeta r^\frac{m}{2}\big(\bar \eta,\bar \eta_r,\frac{\bar \eta}{r},\bar \eta_{rr}, (\frac{\bar \eta}{r})_r,\bar \eta_{rrr}, (\frac{\bar \eta}{r})_{rr},\frac{1}{r}(\frac{\bar \eta}{r})_r,\bar \eta_{rrrr}, (\frac{\bar \eta}{r})_{rrr},(\frac{1}{r}(\frac{\bar \eta}{r})_r)_r \big)\in C([0,T];L^2),\\[4pt]
&\chi^\sharp \rho_0^\frac{1}{2}(\bar \eta,\bar \eta_r)\in C([0,T];L^2),\qquad \chi^\sharp\rho_0^{(\frac{3}{2}-\varepsilon_0)\beta}(\bar \eta_{rr},\bar \eta_{rrr},\bar \eta_{rrrr})\in C([0,T];L^2),\\[4pt]
&\big(\bar \eta, \bar\eta_r,\frac{\bar \eta}{r}\big)\in C^1([0,T];C(\bar I)),\qquad \big(\bar \eta_{rr},(\frac{\bar \eta}{r})_r\big)\in C^1((0,T];C(\bar I)).
\end{aligned} 
\end{equation}
Moreover, we assume here that 
\begin{equation}\label{jibenjiashe}
(\bar\eta_r,\frac{\bar\eta}{r})(t,r)\in \big[\frac{1}{2},\frac{3}{2}\big] \qquad \text{for all }(t,r)\in [0,T]\times \bar I.
\end{equation}
This requirement will be fulfilled in \S \ref{subsection9.2} for the corresponding linearization procedure. 

Now we define the classical solution of the linearized problem \eqref{lp}, which slightly differs from Definition \ref{definition-lag}.
\begin{Definition}\label{fed-cl}
We say that $U(t,r)$ is a classical solution of the linearized problem \eqref{lp} in $[0,T]\times \bar I$ if $U(t,r)$ satisfies equation $\eqref{lp}_1$ pointwise in $(0,T]\times \bar I$, takes the initial data $\eqref{lp}_2$ continuously, and
\begin{equation}\label{regu-class}
\big(U,U_r,\frac{U}{r}\big)\in C([0,T];C(\bar I)),\qquad \big(U_{rr},(\frac{U}{r})_r,U_t\big)\in C((0,T];C(\bar I)).
\end{equation}
\end{Definition}

Then the main conclusion in \S\ref{Subsection9.1} can be stated in the following lemma:
\begin{Lemma}\label{existence-linearize}
Let $n=2$ or $3$ and $\gamma\in (\frac{4}{3},\infty)$. Assume that $\rho_0(r)$ satisfies \eqref{distance-la} for some $\beta\in (\frac{1}{3},\gamma-1]$ and $u_0(r)$ satisfies
\begin{equation*}
\cE(0,U)<\infty.
\end{equation*}
Then, for any $T>0$, the linearized problem \eqref{lp} admits a unique classical solution $U$ in $[0,T]\times \bar I$ satisfying 
\begin{equation*}
\begin{aligned}
&\bar\cE(t,U)+t\,\bar\cD(t,U)\in L^\infty(0,T),\qquad \bar\cD (t,U)\in L^1(0,T),\\[4pt]
&U|_{r=0}=U_r|_{r=1}=0\qquad\qquad  \text{for $t\in(0,T]$}.
\end{aligned}
\end{equation*}
\end{Lemma}

\subsubsection{The modified Galerkin method{\rm:} weak and strong solutions of some general problems}
In order to establish the existence and uniqueness  of \eqref{lp}, we begin by studying a  general initial boundary value problem in $[0,T]\times I$:
\begin{equation}\label{galerkin-w}
\begin{cases}
\displaystyle  r^m\rho_0w_t-2\mu \big(r^m\rho_0\frac{w_r}{\bar\eta_r^2}  \big)_r+2\mu mr^m\rho_0\frac{w}{\bar\eta^2}=m(r^m\rho_0)^\frac{1}{2}\frac{\mathfrak{q}_1}{\bar\eta}-\big((r^m\rho_0)^\frac{1}{2}\frac{\mathfrak{q}_2}{\bar\eta_r} \big)_r,\\[8pt]
w|_{t=0}=w_0\qquad \text{on } I,
\end{cases}
\end{equation}
where $\mathfrak{q}_1,\mathfrak{q}_2\in L^2([0,T];L^2)$ are given functions, 
and $w_0\in L^2_{r^m\rho_0}$.

First, we study the weak solutions of problem \eqref{galerkin-w}. 
This existence result will be frequently used in \S \ref{subsection3.3} for proving Lemma \ref{existence-linearize}.

\begin{Definition}\label{def3.1}
We say that a function $w(t,r)$ is a weak solution in $[0,T]\times I$ of 
problem \eqref{galerkin-w} if the following three properties hold{\rm:}
\begin{enumerate}
\item[{\rm (i)}] $w\in C([0,T];L^2_{r^m\rho_0})\cap L^2([0,T];\cH^1_{r^m\rho_0})$ and $r^m\rho_0 w_t\in L^2([0,T];\cH^{-1}_{r^m\rho_0});$
\item[{\rm (ii)}]
for all $\varphi$ satisfying $\varphi \in \cH^1_{r^m\rho_0}$ and {\it a.e.} time $t\in(0, T)$,
\begin{equation}\label{weak.F.}
\begin{aligned}
&\left<r^m\rho_0 w_t, \varphi\right>_{\cH^{-1}_{r^m\rho_0}\times \cH^1_{r^m\rho_0}}+2\mu\langle r^m\rho_0 D_{\bar\eta} w,D_{\bar\eta}\varphi \rangle+2\mu m\big<\frac{r^m\rho_0}{\bar\eta^2}w,\varphi\big>\\
&=\big<(r^m\rho_0)^\frac{1}{2} \mathfrak{q}_1,\frac{m\varphi}{\bar\eta}\big>+\langle(r^m\rho_0)^\frac{1}{2} \mathfrak{q}_2, D_{\bar\eta}\varphi\rangle;
\end{aligned}
\end{equation}
\item[{\rm (iii)}] $w(0,r)=w_0(r)$ for {\it a.e.} $ r\in I$.
\end{enumerate}
\end{Definition}

Now we can establish the following existence of weak solutions and their related estimates. 
\begin{Proposition}\label{prop1}
For all $T>0$,  problem \eqref{galerkin-w} admits a unique weak solution $w$ in $[0,T]\times I$, which satisfies 
\begin{equation*}
\sup_{t\in[0,T]}|w|_{2,r^m\rho_0}^2 + \int_0^T \!\big(\|w\|_{\cH^{1}_{r^m\rho_0}}^2 \!\!\!+ \big\|r^m\rho_0 w_t\big\|_{\cH^{-1}_{r^m\rho_0}}^2\big)\,\dt\leq C(T)\Big(|w_0|_{2,r^m\rho_0}^2 \!\!+  \int_0^T \!|(\mathfrak{q}_1,\mathfrak{q}_2)|_2^2\,\dt\Big).
\end{equation*}
\end{Proposition}

\begin{proof}
We divide the proof into four steps.

\smallskip
\textbf{1. Introduction of the Galerkin scheme.}  First, according to Lemma \ref{W-space}, for given $w_0\in L^2_{r^m\rho_0}$, there exists a smooth sequence $\{w_0^\vartheta\}_{\vartheta>0}\subset C^\infty(\bar I)$ satisfying
\begin{equation}\label{wdelta-w}
\lim_{\vartheta\to0} |w^\vartheta_0- w_0|_{2,r^m\rho_0}=0.
\end{equation}

Next, we construct a sequence of Galerkin bases. In order to match the spherical symmetry structure of \eqref{galerkin-w} and fulfill the  Neumann boundary condition, it is reasonable to consider the following eigenvalue problem, which is the so-called Sturm-Liouville problem:
\begin{equation}\label{SL}
-(r^m\xi_r)_r+mr^{m-2} \xi=\lambda r^{m} \xi \quad \text{on $I$}, \qquad \,\,
\xi|_{r=0}=\xi_r|_{r=1}=0.
\end{equation}
Then we can expect to construct a Hilbert basis $\{\xi_j\}_{j\in \NN^*}$ of $\cH^1_{r^m}$, which is orthonormal in $L^2_{r^m}$ and orthogonal in $\cH^1_{r^m}$. Actually, we have the following well-known Sturm-Liouville theorem:

\begin{Lemma}[\cite{folland}]\label{hilbert}
Consider the Sturm-Liouville problem \eqref{SL}.
\begin{enumerate}
\item[{\rm(i)}] All eigenvalues $\lambda$ of problem \eqref{SL} are nonzero, 
real, and have multiplicity one. 
Moreover, there are countably infinitely many eigenvalues 
$\{\lambda_j\}_{j\in\NN^*}$, which form a strictly increasing sequence that is bounded below and satisfies $\lambda_j\to \infty$ as $j\to\infty$.
\item[{\rm(ii)}] There exists a sequence of eigenfunctions $\{\xi_j\}_{j\in\NN^*}\subset \cH^1_{r^m}$, corresponding to the eigenvalues $\{\lambda_j\}_{j\in\NN^*}$. Such a sequence of eigenfunctions $\{\xi_j\}_{j\in\NN^*}$ is orthonormal and  complete in $L^2_{r^m}$, namely, $\langle r^m  \xi_j ,  \xi_k \rangle=\delta_{kj}$ for $k,j\in\NN^*$, and
\begin{equation*}
\lim_{N\to\infty}|f_N-f|_{2,r^m}=0 \quad \text{for any $f\in L^2_{r^m}$}, \qquad\text{where $f_N:=\sum_{j=1}^N \langle r^m f, \xi_j\rangle\xi_j$}.
\end{equation*}

\item[{\rm(iii)}] Such a sequence of eigenfunctions $\{\xi_j\}_{j\in\NN^*}$ is also orthogonal and complete in $\cH^1_{r^m}$, namely, $\langle r^m (\xi_j)_r, (\xi_k)_r\rangle+m\big< r^{m-2} \xi_j, \xi_k\big>=\lambda_j\delta_{jk}$ for $k,j\in\NN^*$, and
\begin{equation*}
\lim_{N\to\infty}\|f_N-f\|_{\cH^1_{r^m}}=0 \qquad \text{for any $f\in \cH^1_{r^m}$}.
\end{equation*}
\end{enumerate}
\end{Lemma}
\begin{proof}
The substitution $g(r)=r^{\frac{m-1}{2}}\xi(r)$ transforms \eqref{SL} into the problem (5.33) in Folland's book; see \cite{folland}*{Section 5.4, page 144}, with 
\begin{equation*}
x=r, \quad f(x)=g(r),\quad b=1,\quad \nu=\frac{m+1}{2},\quad \beta=-\frac{m-1}{2},\quad \beta'=1,\quad \mu^2=\lambda,
\end{equation*}
where ($x,f(x),b,\nu,\beta,\beta',\mu$) are the notations used in  \cite{folland}. The corresponding eigenfunctions $\xi(r)=r^{\frac{1-m}{2}}g(r)$ satisfy $\xi(r)=O(r)$ and $\xi'(r)=O(1)$ as $r\to0$, by the asymptotic behavior of the regular Bessel solutions $g(r)$ at the origin, and hence belong to $\cH^1_{r^m}$. 
Since $\frac{\beta}{\beta'}=-\frac{m-1}{2}>-\frac{\nu}{b}$, assertion (i) follows from the spectral discussion in \cite{folland}*{page 145, including Lemma 5.3}. Moreover, the map $\xi\mapsto r^{\frac{m-1}{2}}\xi$ is an isometric isomorphism from $L^2_{r^m}$ onto $L^2_r$. Thus, \cite{folland}*{Theorem~5.3(b), page 147} yields (ii) after normalization.

 We now prove the completeness in (iii). Define the bilinear form $\mathscr{B}[\cdot,\cdot]$ by
\begin{equation*}
\mathscr{B}[f,g]:=\langle r^m f_r, g_r\rangle+m\langle r^{m-2} f, g\rangle.
\end{equation*}
First, a direct calculation, combined with Lemma \ref{hardy-inequality} and the H\"older inequality, implies
\begin{equation*}
\begin{aligned}
&\big|\mathscr{B}[f,g]\big|\leq |f_r|_{2,r^m}|g_r|_{2,r^m}+m|f|_{2,r^{m-2}} |g|_{2,r^{m-2}}\leq C\|f\|_{\cH^1_{r^m}}\|g\|_{\cH^1_{r^m}},\\
&\mathscr{B}[f,f]=|f_r|_{2,r^m}^2+ m|f|_{2,r^{m-2}}^2\geq C^{-1}\|f\|_{\cH^1_{r^m}}^2,
\end{aligned}   
\end{equation*}
which yields that $\mathscr{B}[f,g]$ is a bounded coercive bilinear form in $\cH^1_{r^m}$. Next, suppose that
\begin{equation}\label{comet}
\mathscr{B}[\xi_j,f]=0 \qquad \text{for all $j\in\NN^*$}.  
\end{equation}
To show that $\{\xi_j\}_{j\in\NN^*}$ is complete in $\cH^1_{r^m}$, it suffices to prove $f=0$.

Note that, since $\xi_j$ is the eigenfunction corresponding to the eigenvalue $\lambda_j$, we have
\begin{equation*}
\mathscr{B}[\xi_j,f]=\lambda_j\langle r^m \xi_j,f\rangle \qquad\,\, 
\text{for all $f\in \cH^1_{r^m}$}.
\end{equation*}
This, combined with \eqref{comet}, yields $\langle r^m \xi_j,f\rangle=0$ for all $j\in\NN^*$, which, along with the completeness of $\{\xi_j\}_{j\in\NN^*}$ in $L^2_{r^m}$, implies $f=0$. 
\end{proof}

Consequently, in view of the above Galerkin basis $\{\xi_j\}_{j\in \NN^*}$, we set 
\begin{equation}\label{U^n}
w^{N,\vartheta}(t,r):=\sum_{k=1}^N \mu_k^{N,\vartheta}(t) \xi_k(r) \qquad \text{for }\vartheta\in (0,1)\text{ and }N\in \NN^*.
\end{equation}
Here, $\mu_k^{N,\vartheta}(t)$ are selected by solving the following  ODE problem in $[0,T]$:
\begin{equation}\label{galerkin-n}
\begin{cases}
\displaystyle \big<r^m\rho_0 w^{N,\vartheta}_t, \xi_j\big>+2\mu\langle r^m\rho_0 D_{\bar\eta}w^{N,\vartheta},D_{\bar\eta}\xi_j\rangle+2\mu m\big<\frac{r^m\rho_0}{\bar\eta^2}w^{N,\vartheta},\xi_j\big>\\[8pt]
\qquad \displaystyle=\big<(r^m\rho_0)^\frac{1}{2} \mathfrak{q}_1, \frac{m\xi_j}{\bar\eta}\big>+\langle(r^m\rho_0)^\frac{1}{2} \mathfrak{q}_2, D_{\bar\eta}\xi_j\rangle,\\[8pt]
\mu_j^{N,\vartheta}(0)=\langle r^m w_0^\vartheta,\xi_j\rangle, \ \ j=1,2,\cdots\!,N,
\end{cases}
\end{equation}
which can also  be rewritten as 
\begin{equation}\label{mu^n}
\begin{cases}
\displaystyle \mathfrak{A}\cdot\frac{\mathrm{d}}{\dt}\mu^{N,\vartheta}(t)+\mathfrak{B}(t)\cdot\mu^{N,\vartheta}(t)=\mathfrak{c}(t) \qquad
\text{in $(0,T]$},\\[10pt]
\mu_j^{N,\vartheta}(0)=\langle r^mw_0^\vartheta,\xi_j\rangle \
\qquad \text{for $j=1,2,\cdots\!,N$},
\end{cases}
\end{equation}
where $\mu^{N,\vartheta}(t):=(\mu_1^{N,\vartheta},\cdots\!,\mu_N^{N,\vartheta})^\top(t)$ and 
\begin{equation}\label{ABC}
\begin{aligned}
\mathfrak{A}&\,=(\mathfrak{A}_{kj})_{1\leq k,j\leq N}, \qquad \mathfrak{B}(t)=(\mathfrak{B}_{kj}(t))_{1\leq k,j\leq N}, \qquad \mathfrak{c}(t)=(\mathfrak{c}_{1},\cdots\!,\mathfrak{c}_{N})^\top(t),\\
\mathfrak{A}_{kj}&:=\int_0^1 r^m \rho_0 \xi_k \xi_j\,\mathrm{d}r, \qquad\ \mathfrak{B}_{kj}(t):= 2\mu\int_0^1 r^m\rho_0\big(D_{\bar\eta}\xi_k D_{\bar\eta}\xi_j+m \frac{\xi_k \xi_j}{\bar\eta^2} \big)\,\mathrm{d}r,\\
\mathfrak{c}_j(t)&:= \int_0^1(r^m\rho_0)^\frac{1}{2}\big(\mathfrak{q}_1 \frac{m\xi_j}{\bar\eta} +\mathfrak{q}_2D_{\bar\eta}\xi_j\big)\,\mathrm{d}r.
\end{aligned}
\end{equation}

To solve  \eqref{mu^n}, we first note that, due to $\rho_0^\beta\sim 1-r$ and the linear independence of $\{\xi_j\}_{j\in \NN^*}$, $\{(r^m\rho_0)^\frac{1}{2} \xi_j\}_{j\in \NN^*}$ are linearly independent, and hence its Gram matrix $\mathfrak{A}$ is non-singular. Next, from \eqref{given-flow} and Lemma \ref{hilbert}, it is direct to check that  $\mathfrak{B}(t)\in C^1[0,T]$ and $\mathfrak{c}(t)\in L^2(0,T)$. Therefore, we can obtain the following existence result for \eqref{mu^n} from the classical  ODE theory.
\begin{Lemma}[\cite{cod}]
Problem \eqref{mu^n} admits a unique solution $\mu_j^{N,\vartheta}\in AC[0,T]$, for each $j=1,2,\cdots\!, N$, $N\in \NN^*$, and $\vartheta\in (0,1)$, where $AC$ denotes the space of absolutely continuous functions. As a consequence, $w^{N,\vartheta}(t,r)\in AC([0,T];\cH^1_{r^m})$ and $w^{N,\vartheta}$ is differentiable \textit{a.e.} in $t$, for each $N\in \NN^*$ and $\vartheta\in (0,1)$.
\end{Lemma}

\textbf{2. Uniform estimates of $w^{N,\vartheta}$.} First, multiplying \eqref{galerkin-n} by $\mu^{N,\vartheta}_j(t)$ and  summing the resulting equality with respect to $j$ from $1$ to $N$ imply 
\begin{equation*}
\begin{aligned}
&\,\frac{1}{2}\frac{\mathrm{d}}{\dt}\int_0^1 r^m\rho_0|w^{N,\vartheta}|^2\,\mathrm{d}r+2\mu\int_0^1 r^m\rho_0\big(|D_{\bar\eta}w^{N,\vartheta}|^2+m \frac{|w^{N,\vartheta}|^2}{\bar\eta^2} \big)\,\mathrm{d}r\\
&=\int_0^1(r^m\rho_0)^\frac{1}{2} \mathfrak{q}_1 \frac{mw^{N,\vartheta}}{\bar\eta} \,\mathrm{d}r+\int_0^1(r^m\rho_0)^\frac{1}{2} \mathfrak{q}_2D_{\bar\eta}w^{N,\vartheta}\,\mathrm{d}r,
\end{aligned}
\end{equation*}
which, along with the Young inequality, yields
\begin{equation}\label{I11I22}
\frac{\mathrm{d}}{\dt}\int_0^1 r^m\rho_0|w^{N,\vartheta}|^2\,\mathrm{d}r\!+\mu\int_0^1 r^m\rho_0\big(|D_{\bar\eta}w^{N,\vartheta}|^2+m \frac{|w^{N,\vartheta}|^2}{\bar\eta^2} \big)\,\mathrm{d}r\leq  C|(\mathfrak{q}_1,\mathfrak{q}_2)|_2^2.
\end{equation}

Integrating the above over $[0,t]$, together with \eqref{jibenjiashe}, leads to
\begin{equation}\label{33125}
|w^{N,\vartheta}(t)|_{2,r^m\rho_0}^2+\int_0^{t} \|w^{N,\vartheta}\|_{\cH^1_{r^m\rho_0}}^2\!\dt\leq C(T)\Big(|w^{N,\vartheta}(0)|_{2,r^m\rho_0}^2+ \int_0^{t} |(\mathfrak{q}_1,\mathfrak{q}_2)|_2^2\,\dt\Big),
\end{equation}
where, with $P_N$ denoting the $L^2_{r^m}$-orthogonal projection onto
$\operatorname{span}\{\xi_1,\ldots,\xi_N\}$,
\begin{equation*}
w^{N,\vartheta}(0,r)=P_Nw_0^\vartheta
:=\sum_{j=1}^N \langle r^m w_0^\vartheta,\xi_j\rangle\xi_j.
\end{equation*}

We next select the two approximation parameters by a diagonal argument. For each fixed
$\vartheta>0$, Lemma \ref{hilbert} and $\rho_0\in L^\infty(I)$ give
\begin{equation}\label{projection-initial}
\lim_{N\to\infty}|P_Nw_0^\vartheta-w_0^\vartheta|_{2,r^m\rho_0}=0.
\end{equation}
If $w_0=0$, we take $w_0^\vartheta=0$ and may choose any sequences
$\vartheta_j\downarrow0$ and $N_j\uparrow\infty$. If $w_0\neq0$, we first use
\eqref{wdelta-w} to choose $\vartheta_j\downarrow0$ and then, for each fixed
$\vartheta_j$, use \eqref{projection-initial} to choose $N_j\geq j$ such that
\begin{equation*}
|P_{N_j}w_0^{\vartheta_j}-w_0^{\vartheta_j}|_{2,r^m\rho_0}
+|w_0^{\vartheta_j}-w_0|_{2,r^m\rho_0}\leq \frac{1}{j}.
\end{equation*}
Consequently, $N_j,\vartheta_j^{-1}\to\infty$ and
\begin{equation}\label{initial-diagonal}
\begin{aligned}
|w^{N_j,\vartheta_j}(0)-w_0|_{2,r^m\rho_0}
&\leq |P_{N_j}w_0^{\vartheta_j}-w_0^{\vartheta_j}|_{2,r^m\rho_0}
+|w_0^{\vartheta_j}-w_0|_{2,r^m\rho_0}\to 0.
\end{aligned}
\end{equation}
In the rest of this proof, $(w^{N,\vartheta},\mu^{N,\vartheta})$ always denotes
this fixed diagonal sequence, and every limit as $N,\vartheta^{-1}\to\infty$ is
taken along it. After discarding finitely many terms when $w_0\neq0$, while the
zero case is immediate, we have
\begin{equation}\label{initial-converge}
|w^{N,\vartheta}(0)|_{2,r^m\rho_0}\leq 2|w_0|_{2,r^m\rho_0}.
\end{equation}

Finally, substituting  \eqref{initial-converge} into \eqref{33125} yields
\begin{equation}\label{3...18}
\sup_{t\in[0,T]} |w^{N,\vartheta}|_{2,r^m\rho_0}^2+\int_0^{T} \|w^{N,\vartheta}\|_{\cH^1_{r^m\rho_0}}^2\dt
\leq C(T)\Big(|w_0|_{2,r^m\rho_0}^2+\int_0^{T}|(\mathfrak{q}_1,\mathfrak{q}_2)|_2^2\,\dt\Big).
\end{equation}

\textbf{3. Taking the limit along the diagonal sequence as $N,\vartheta^{-1}\to \infty$.}
Using \eqref{3...18}, the fact that $L^2_{r^m\rho_0}$ is a separable reflexive Banach space (see Lemma \ref{W-space}), and weak convergence arguments, we can extract a subsequence of the fixed diagonal sequence (still denoted by $w^{N,\vartheta}$) and some limits $w$ and  $(\mathfrak{w}_1,\mathfrak{w}_2)$ such that
\begin{equation}\label{3..124}
\begin{aligned}
w^{N,\vartheta} \rightharpoonup w\qquad  &\text{weakly* in }L^\infty([0,T];L^2_{r^m\rho_0}),\\
\big(w^{N,\vartheta}_r,\frac{w^{N,\vartheta}}{r}\big) \rightharpoonup (\mathfrak{w}_1,\mathfrak{w}_2)\qquad &\text{weakly \, in }L^2([0,T];L^2_{r^m\rho_0}).
\end{aligned}
\end{equation}
Then the definition of weak derivatives and the uniqueness of limits imply that $(\mathfrak{w}_1,\mathfrak{w}_2)=(w_r,\frac{w}{r})$, and the lower semi-continuity of weak convergence implies that \eqref{3...18} holds for $w$.

Now we take the limit along the diagonal sequence in \eqref{galerkin-n}. Fix $M\in\NN^*$ and let $\mathfrak{g}^M(t,r)=\sum_{j=1}^M g_j(t)\xi_j(r)$ with $g_j \in C_{\mathrm{c}}^\infty(0,T)$ for $j=1,\cdots\!,M$. Since $N_j\geq j$, one has $N_j\geq M$ for all sufficiently large $j$.
Then, \eqref{galerkin-n} yields that, for any $N\geq M$,
\begin{equation}\label{equ3320}
\begin{aligned}
&\int_0^T \!\!\big<r^m\rho_0 w^{N,\vartheta}_t \!, \mathfrak{g}^M\big>\,\dt \! + \!2\mu \int_0^T \!\!\big< r^m\rho_0 D_{\bar\eta} w^{N,\vartheta}\! ,D_{\bar\eta} \mathfrak{g}^M \big>\,\dt \!+ \!2\mu m \int_0^T\!\! \big<\frac{r^m\rho_0}{\bar\eta^2}w^{N,\vartheta} \!,\mathfrak{g}^M\big>\,\dt\\
&=\int_0^T\big<(r^m\rho_0)^\frac{1}{2} \mathfrak{q}_1, \frac{m\mathfrak{g}^M}{\bar\eta}\big>\,\dt+\int_0^T\big<(r^m\rho_0)^\frac{1}{2} \mathfrak{q}_2, D_{\bar\eta} \mathfrak{g}^M\big>\,\dt.
\end{aligned}
\end{equation}
Since both $w^{N,\vartheta}$ and $\mathfrak{g}^M$ are differentiable with respect to $t$, $\partial_t$ in \eqref{equ3320} can be transferred from $w^{N,\vartheta}$ to $\mathfrak{g}^M$ via integration by parts. Then, based on \eqref{3..124}, we let $j\to\infty$ along the diagonal sequence in the resulting equality to obtain
\begin{equation}\label{3,,20}
\begin{aligned}
&-\int_0^T \langle r^m\rho_0 w, \mathfrak{g}^M_t\rangle\,\dt  + 2\mu \int_0^T \big< r^m\rho_0 D_{\bar\eta} w ,D_{\bar\eta} \mathfrak{g}^M \big>\,\dt + 2\mu m \int_0^T\big<\frac{r^m\rho_0}{\bar\eta^2}w,\mathfrak{g}^M\big>\,\dt\\
&=\int_0^T\big<(r^m\rho_0)^\frac{1}{2} \mathfrak{q}_1, \frac{m\mathfrak{g}^M}{\bar\eta}\big>\,\dt+\int_0^T\big<(r^m\rho_0)^\frac{1}{2} \mathfrak{q}_2, D_{\bar\eta} \mathfrak{g}^M\big>\,\dt.
\end{aligned}
\end{equation}

Next, since $r^m\rho_0 w_t\in H^{-1}([0,T];L^2)$ due to $w\in L^2([0,T];L^2_{r^m\rho_0})$, by \eqref{jibenjiashe}, \eqref{3,,20}, and the definition of  distributional derivatives, we obtain
\begin{equation*}
\begin{aligned}
&\,\big|\langle r^m\rho_0 w_t, \mathfrak{g}^M\rangle_{H_t^{-1}((H^1)^*)\times H^1_{0,t}(H^1)}\big|
=\absB{\int_0^T\langle r^m\rho_0 w, \mathfrak{g}^M_t\rangle\,\dt} \\
&\leq 2\mu\int_0^T\big|\big< r^m\rho_0 D_{\bar\eta} w ,D_{\bar\eta} \mathfrak{g}^M \big>\big|\,\dt+ 2\mu m\int_0^T\Big|\big<\frac{r^m\rho_0}{\bar\eta^2}w,\mathfrak{g}^M\big>\Big| \,\dt \\
&\quad+\int_0^T\Big|\big<(r^m\rho_0)^\frac{1}{2} \mathfrak{q}_1, \frac{m\mathfrak{g}^M}{\bar\eta}\big>\Big|\,\dt+\int_0^T\big|\big<(r^m\rho_0)^\frac{1}{2} \mathfrak{q}_2, D_{\bar\eta} \mathfrak{g}^M\big>\big|\,\dt \\
&\leq C \big(\|w\|_{L^2_t(\cH^1_{r^m\rho_0})}+ \|(\mathfrak{q}_1,\mathfrak{q}_2)\|_{L^2_t(L^2)}\big)\|\mathfrak{g}^M\|_{L^2_t(\cH^1_{r^m\rho_0})}. 
\end{aligned}    
\end{equation*}
We now claim 
\begin{equation}\label{dens}
\{\mathfrak{g}^M\}_{M\in\NN^*}\text{ is dense in }L^2([0,T];\cH^1_{r^m\rho_0}).
\end{equation}
Once this is proved, it follows from the Hahn-Banach theorem (see \cite{conway}) that $r^m\rho_0 w_t$, which is a functional initially defined on $H^1_0([0,T];L^2)$, can be uniquely extended to a functional defined on $L^2([0,T];\cH^1_{r^m\rho_0})$ and 
\begin{equation}\label{H-1}
\|r^m\rho_0 w_t\|_{L^2_t(\cH^{-1}_{r^m\rho_0})}\leq C\big(\|w\|_{L^2_t(\cH^1_{r^m\rho_0})}+ \|(\mathfrak{q}_1,\mathfrak{q}_2)\|_{L^2_t(L^2)}\big)\leq C.
\end{equation}
To derive \eqref{dens}, it suffices to prove the density of  $\mathrm{span}\{\xi_j\}_{j\in\NN^*}$ in $\cH^1_{r^m\rho_0}$. By Proposition \ref{prop-bijin}, for any given $f\in \cH^1_{r^m\rho_0}$, there exists a sequence $\{f^\varepsilon\}_{\varepsilon>0}\subset C^\infty(\bar I)\cap \cH^1_{r^m}$ such that
\begin{equation*}
f^\varepsilon\to f \ \ \text{in $\cH^1_{r^m\rho_0}$} \qquad \text{as $\varepsilon\to 0$}.
\end{equation*}
On the other hand, for any $f\in C^\infty(\bar I)\cap \cH^1_{r^m}\subset \cH^1_{r^m}$, we obtain from Lemma \ref{hilbert} and $\cH^1_{r^m}\subset \cH^1_{r^m\rho_0}$ that
\begin{equation*}
\sum_{j=1}^M \langle r^m f,\xi_j\rangle \xi_j\to f \, \,\,\, \text{in $\cH^1_{r^m\rho_0}$} \qquad\,\, \text{as $M\to\infty$}.
\end{equation*}
Consequently, set $f^{M,\varepsilon}:=\sum_{j=1}^M \langle r^m f^\varepsilon,\xi_j\rangle \xi_j$. First choose $\varepsilon_j\downarrow0$ and then, for each fixed $\varepsilon_j$, choose $M_j\geq j$ sufficiently large. The same diagonal argument as in \eqref{initial-diagonal} gives $f^{M_j,\varepsilon_j}\to f$ in $\cH^1_{r^m\rho_0}$, which proves claim \eqref{dens}. 

Finally, for any $\mathfrak{g}\in L^2([0,T];\cH^1_{r^m\rho_0})$,  
there exists a sequence $\{\mathfrak{g}^M\}_{M\in\NN^*}$ 
with each $\mathfrak{g}^M\in L^2([0,T];\cH^1_{r^m\rho_0})$ such that $\mathfrak{g}^M\to \mathfrak{g}$ in $L^2([0,T];\cH^1_{r^m\rho_0})$ as $M\to\infty$. Taking the limit as $M\to \infty$ in \eqref{3,,20}, along with  \eqref{H-1}, implies that, for all $\mathfrak{g}\in L^2([0,T];\cH^1_{r^m\rho_0})$,
\begin{equation}\label{3..133}
\begin{aligned}
&\int_0^T\!\! \langle r^m\rho_0 w_t, \mathfrak{g} \rangle_{\cH^{-1}_{r^m\rho_0} \!\times \cH^1_{r^m\rho_0}}\,\dt  + 2\mu \int_0^T\!\! \big< r^m\rho_0 D_{\bar\eta} w ,D_{\bar\eta} \mathfrak{g} \big>\,\dt + 2\mu m \int_0^T\!\!\big<\frac{r^m\rho_0}{\bar\eta^2}w,\mathfrak{g} \big>\,\dt \\
&=\int_0^T\big<(r^m\rho_0)^\frac{1}{2} \mathfrak{q}_1, \frac{m\mathfrak{g}}{\bar\eta}\big>\,\dt+\int_0^T\langle(r^m\rho_0)^\frac{1}{2} \mathfrak{q}_2, D_{\bar\eta} \mathfrak{g}\rangle\,\dt.
\end{aligned}
\end{equation}
Hence, the weak formulation \eqref{weak.F.} follows simply by setting $\mathfrak{g}(t,r)=\psi(t)\varphi(r)\in \cH^1_{r^m\rho_0}$ in \eqref{3..133} with $\psi(t)\in C_{\rm c}^\infty(0,T)$ and using the equation in the sense of time distributions.

\medskip
\textbf{4. Uniqueness and time continuity.}
First, it follows from $w\in L^2([0,T];\cH^1_{r^m\rho_0})$, \eqref{H-1}, and Lemma \ref{Aubin} (let $s=1$ in Lemma \ref{Aubin}) that $w\in C([0,T];L^2_{r^m\rho_0})$.

It still remains to show $w(0,r)=w_0(r)$ for \textit{a.e.} $r\in I$. On one hand, thanks to \eqref{weak.F.} and $w\in C([0,T];L^2_{r^m\rho_0})$, for any $\mathfrak{g}\in C_{\mathrm{c}}^1([0,T);\cH^1_{r^m\rho_0})$, 
\begin{equation}\label{3...125}
\begin{aligned}
&-\int_0^T \langle r^m\rho_0 w,\mathfrak{g}_t\rangle\,\dt+ 2\mu \int_0^T \big< r^m\rho_0 D_{\bar\eta} w ,D_{\bar\eta} \mathfrak{g} \big>\,\dt + 2\mu m \int_0^T \big<\frac{r^m\rho_0}{\bar\eta^2}w,\mathfrak{g} \big>\,\dt \\
&=\langle r^m\rho_0 w(0),\mathfrak{g}(0)\rangle+\int_0^T\big<(r^m\rho_0)^\frac{1}{2} \mathfrak{q}_1, \frac{m\mathfrak{g}}{\bar\eta}\big>\,\dt+\int_0^T\langle(r^m\rho_0)^\frac{1}{2} \mathfrak{q}_2, D_{\bar\eta} \mathfrak{g}\rangle\,\dt.
\end{aligned}
\end{equation}
On the other hand, choose $\mathfrak{g}^M(t,x)=\sum_{j=1}^M g_j(t) \xi_j$ in \eqref{equ3320} with $g_j(t)\in C^\infty_{\mathrm{c}}([0,T))$ ($j=1,\cdots\!,M$) such that $\mathfrak{g}^M\to \mathfrak{g}$ in $C^1_{\mathrm{c}}([0,T);\cH^1_{r^m\rho_0})$ as $M\to \infty$. Then \eqref{equ3320} becomes
\begin{equation*}
\begin{aligned}
&-\int_0^T\!\!\big<r^m\rho_0 w^{N,\vartheta}\!\!, \mathfrak{g}^M_t\big>\,\dt +2\mu\!\int_0^T\!\!\big< r^m\rho_0  D_{\bar\eta} w^{N,\vartheta}\!\!, D_{\bar\eta}\mathfrak{g}^M\big>\,\dt+2\mu m\!\int_0^T\!\!\big<\frac{r^m\rho_0}{\bar\eta^2}w^{N,\vartheta}\!\!,\mathfrak{g}^M\big>\,\dt \\
&=\big<r^m\rho_0 w^{N,\vartheta}(0), \mathfrak{g}^M(0)\big>+\!\int_0^T\!\!\big<(r^m\rho_0)^\frac{1}{2} \mathfrak{q}_1, \frac{m\mathfrak{g}^M}{\bar\eta}\big>\,\dt\!+\!\int_0^T\!\!\big<(r^m\rho_0)^\frac{1}{2} \mathfrak{q}_2, D_{\bar\eta}\mathfrak{g}^M \big>\,\dt
\end{aligned}
\end{equation*}
for all $M\leq N$. For fixed $M$, we first let $j\to\infty$ along the diagonal sequence, using \eqref{initial-diagonal}; we then let $M\to\infty$. This gives that, for all $\mathfrak{g}\in C^1_{\mathrm{c}}([0,T);\cH^1_{r^m\rho_0})$,
\begin{equation}\label{3...126}
\begin{aligned}
&-\int_0^T \langle r^m\rho_0 w,\mathfrak{g}_t\rangle\,\dt+ 2\mu \int_0^T \big< r^m\rho_0 D_{\bar\eta} w ,D_{\bar\eta} \mathfrak{g} \big>\,\dt + 2\mu m \int_0^T \big<\frac{r^m\rho_0}{\bar\eta^2}w,\mathfrak{g} \big>\,\dt \\
&=\langle r^m\rho_0 w_0,\mathfrak{g}(0)\rangle+\int_0^T\big<(r^m\rho_0)^\frac{1}{2} \mathfrak{q}_1, \frac{m\mathfrak{g}}{\bar\eta}\big>\,\dt+\int_0^T\langle(r^m\rho_0)^\frac{1}{2} \mathfrak{q}_2, D_{\bar\eta} \mathfrak{g}\rangle\,\dt.
\end{aligned}
\end{equation}
A comparison of \eqref{3...125}--\eqref{3...126} yields that $w(0,r)=w_0$ for \textit{a.e.} $r\in I$. Finally, setting $w_0=0$, $\varphi=w$, and $\mathfrak{q}_1=\mathfrak{q}_2=0$ in \eqref{weak.F.} yields $w=0$, which implies the uniqueness. 
\end{proof}

Next, we show that, if $w_0\in \cH^1_{r^m\rho_0}$, the regularity of weak solutions obtained in Proposition \ref{prop1} can be improved.

\begin{Proposition}\label{prop-strong}
Assume that $w$ is the weak solution of \eqref{galerkin-w} obtained in {\rm Proposition \ref{prop1}}. If $w_0\in \cH^1_{r^m\rho_0}$ and 
\begin{equation}\label{gal-ass1}
(\mathfrak{q_1},\mathfrak{q}_2)\in L^\infty([0,T];L^2),\qquad   ((\mathfrak{q_1})_t,(\mathfrak{q}_2)_t)\in L^2([0,T];L^2),
\end{equation}
then 
\begin{equation}\label{reg-H}
\begin{aligned}
&w\in L^\infty([0,T];\cH^1_{r^m\rho_0}),\quad w_t\in L^2([0,T];L^2_{r^m\rho_0}),\\
&\sqrt{t} w_t\in L^\infty([0,T];L^2_{r^m\rho_0})\cap L^2([0,T];\cH^1_{r^m\rho_0}),\quad \sqrt{t}(r^m\rho_0 w_{tt})\in L^2([0,T];\cH^{-1}_{r^m\rho_0}),
\end{aligned}
\end{equation}
and $w$ satisfies the following formulation{\rm :} for \textit{a.e.} $t>0$ and every $\varphi\in  \cH^1_{r^m\rho_0}$,
\begin{equation}\label{weak-F-H}
\begin{aligned}
&\, \big<r^m\rho_0 w_{tt}, \varphi\big>_{\cH^{-1}_{r^m\rho_0}\times \cH^1_{r^m\rho_0}}+2\mu\langle r^m\rho_0 D_{\bar\eta}w_t,D_{\bar\eta}\varphi\rangle+2\mu m\big<\frac{r^m\rho_0}{\bar\eta^2}w_t,\varphi\big>\\
&-4\mu\langle r^m\rho_0 D_{\bar\eta}\bar UD_{\bar\eta}w,D_{\bar\eta}\varphi\rangle-4\mu m\big<\frac{r^m\rho_0}{\bar\eta^3}\bar Uw,\varphi\big>\\ 
&=\big<(r^m\rho_0)^\frac{1}{2} \big((\mathfrak{q}_1)_t- \mathfrak{q}_1 \frac{\bar U}{\bar\eta}\big), \frac{m\varphi}{\bar\eta}\big>+\langle(r^m\rho_0)^\frac{1}{2} \big((\mathfrak{q}_2)_t- \mathfrak{q}_2D_{\bar\eta}\bar U\big), D_{\bar\eta}\varphi\rangle.
\end{aligned}
\end{equation}
In particular, the following energy equality holds for \textit{a.e.} $t>0${\rm:}
\begin{equation}\label{Energy-Id}
\begin{aligned}
&\,\frac{1}{2}\frac{\mathrm{d}}{\dt}\int_0^1 r^m\rho_0 |w_t|^2\,\mathrm{d}r+2\mu\int_0^1 r^m\rho_0\big(|D_{\bar\eta}w_t|^2+m \frac{|w_t|^2}{\bar\eta^2} \big)\,\mathrm{d}r\\
&=4\mu \int_0^1 r^m\rho_0 D_{\bar\eta}\bar UD_{\bar\eta}wD_{\bar\eta}w_t\,\mathrm{d}r+4\mu m\int_0^1 \frac{r^m\rho_0}{\bar\eta^3}\bar Uw w_t\,\mathrm{d}r\\
&\quad +m\int_0^1(r^m\rho_0)^\frac{1}{2} \big((\mathfrak{q}_1)_t- \mathfrak{q}_1 \frac{\bar U}{\bar\eta}\big) \frac{w_t}{\bar\eta} \,\mathrm{d}r+\int_0^1(r^m\rho_0)^\frac{1}{2} \big((\mathfrak{q}_2)_t- \mathfrak{q}_2D_{\bar\eta}\bar U\big) D_{\bar\eta}w_t\,\mathrm{d}r.
\end{aligned}
\end{equation}
\end{Proposition}
\begin{proof}
We divide the proof into four steps.

\smallskip
\textbf{1.} First, from Proposition \ref{prop-bijin}, for given $w_0\in \cH^1_{r^m\rho_0}$, it follows that there exists a smooth sequence $\{w_0^\vartheta\}_{\vartheta>0}\subset C^\infty(\bar I)\cap \cH^1_{r^m\rho_0}$ satisfying
\begin{equation}\label{wdelta-w2}
\lim_{\vartheta\to0} \|w^\vartheta_0- w_0\|_{\cH^1_{r^m\rho_0}}=0.
\end{equation}

\smallskip
\textbf{2.} Based on the mollified initial data $w_0^\vartheta$ and the Galerkin scheme shown in Step 1 of the proof of Proposition \ref{prop1}, we can construct a Galerkin approximate sequence $w^{N,\vartheta}\in AC([0,T];\cH^1_{r^m\rho_0})$ and an ODE problem, which are the same as those given in \eqref{U^n}--\eqref{mu^n}. 

Now, by multiplying $\eqref{galerkin-n}_1$ by $(\mu_j^{N,\vartheta})_t(t)$ and summing the
resulting equality with respect to $j$ from $1$ to $N$, we have
\begin{equation}\label{I0*}
\mu\frac{\mathrm{d}}{\mathrm{d}t}\int_0^1 r^m\rho_0 \big(|D_{\bar\eta} w^{N,\vartheta}|^2+m\frac{|w^{N,\vartheta}|^2}{\bar\eta^2}\big)\,\mathrm{d}r+\int_0^1 r^m\rho_0|w^{N,\vartheta}_t|^2\,\mathrm{d}r=\mathrm{I}_1 +\frac{\mathrm{d}}{\mathrm{d}t}\mathrm{I}_*(t).
\end{equation}
Here, $\mathrm{I}_1$ is defined  and estimated as follows:
\begin{align}\label{I1*}
&\begin{aligned}
\mathrm{I}_1&:=-2\mu\int_0^1 r^m\rho_0\big(D_{\bar\eta}\bar U|D_{\bar\eta} w^{N,\vartheta}|^2+m\frac{\bar U |w^{N,\vartheta}|^2}{\bar\eta^3}\big)\,\mathrm{d}r\\
&\quad +\int_0^1 (r^m\rho_0)^\frac{1}{2}\Big(\big(\mathfrak{q}_1\frac{\bar U}{\bar\eta}-(\mathfrak{q}_1)_t\big)\frac{m w^{N,\vartheta}}{\bar\eta}+(\mathfrak{q}_2D_{\bar\eta}\bar U-(\mathfrak{q}_2)_t) D_{\bar\eta} w^{N,\vartheta} \Big)\,\mathrm{d}r\notag
\end{aligned}\\
&\quad \begin{aligned}[t]
&\leq C\Big|\big(D_{\bar\eta}\bar U,\frac{\bar U}{\bar\eta}\big)\Big|_\infty\Big|\big(D_{\bar\eta} w^{N,\vartheta},\frac{w^{N,\vartheta}}{\bar\eta}\big)\Big|_{2,r^m\rho_0}^2\\
&\quad +C\Big(\Big|\big(D_{\bar\eta}\bar U,\frac{\bar U}{\bar\eta}\big)\Big|_\infty|(\mathfrak{q}_1,\mathfrak{q}_2)|_2+|((\mathfrak{q}_1)_t,(\mathfrak{q}_2)_t)|_2\Big)\Big|\big(D_{\bar\eta} w^{N,\vartheta},\frac{w^{N,\vartheta}}{\bar\eta}\big)\Big|_{2,r^m\rho_0}\\
&\leq C(T)\Big|\big(D_{\bar\eta} w^{N,\vartheta},\frac{w^{N,\vartheta}}{\bar\eta}\big)\Big|_{2,r^m\rho_0}^2 +C(T)|(\mathfrak{q}_1,\mathfrak{q}_2,(\mathfrak{q}_1)_t,(\mathfrak{q}_2)_t)|_2^2,
\end{aligned}
\end{align}
and $\mathrm{I}_*(t)$ is defined and controlled by, for all $\varepsilon\in(0,1)$,
\begin{equation}\label{I*}
\begin{aligned}
|\mathrm{I}_*(t)|&:=\Big|\int_0^1 (r^m\rho_0)^\frac{1}{2}\big(\mathfrak{q}_1\frac{m w^{N,\vartheta}}{\bar\eta}+\mathfrak{q}_2D_{\bar\eta} w^{N,\vartheta}\big)\,\mathrm{d}r\Big|\\
&\leq C\varepsilon^{-1}|(\mathfrak{q}_1,\mathfrak{q}_2)|_2^2+\varepsilon\Big|\big(D_{\bar\eta} w^{N,\vartheta},\frac{w^{N,\vartheta}}{\bar\eta}\big)\Big|_{2,r^m\rho_0}^2.
\end{aligned}
\end{equation}
Collecting \eqref{I0*}--\eqref{I*} with $\varepsilon$ sufficiently small, along with the Gr\"onwall inequality, gives
\begin{equation*}
\begin{aligned}
&\sup_{t\in[0,T]}\|w^{N,\vartheta}\|_{\cH^1_{r^m\rho_0}}^2+\int_0^T |w^{N,\vartheta}_t|_{2,r^m\rho_0}^2\,\mathrm{d}t\\
&\leq C(T)\Big(\|w^{N,\vartheta}(0)\|_{\cH^1_{r^m\rho_0}}^2+\sup_{t\in[0,T]}|(\mathfrak{q}_1,\mathfrak{q}_2)|_2^2+\int_0^T |((\mathfrak{q}_1)_t,(\mathfrak{q}_2)_t)|_2^2\,\mathrm{d}t\Big).
\end{aligned}
\end{equation*}
For the initial data, we again have to couple the smoothing and projection parameters.
For each fixed $\vartheta>0$, (iii) of Lemma \ref{hilbert}, together with
$\rho_0\in L^\infty(I)$, gives
\begin{equation*}
\lim_{N\to\infty}\|P_Nw_0^\vartheta-w_0^\vartheta\|_{\cH^1_{r^m\rho_0}}=0.
\end{equation*}
If $w_0=0$, we again take $w_0^\vartheta=0$. Otherwise, we first choose
$\vartheta_j\downarrow0$ by \eqref{wdelta-w2}, and then choose $N_j\geq j$ for
each fixed $\vartheta_j$ so that
\begin{equation}\label{initial-diagonal-H1}
\begin{aligned}
\|w^{N_j,\vartheta_j}(0)-w_0\|_{\cH^1_{r^m\rho_0}}
&\leq \|P_{N_j}w_0^{\vartheta_j}-w_0^{\vartheta_j}\|_{\cH^1_{r^m\rho_0}}
+\|w_0^{\vartheta_j}-w_0\|_{\cH^1_{r^m\rho_0}}\to 0.
\end{aligned}
\end{equation}
In the remainder of this proof, $w^{N,\vartheta}$ denotes this fixed diagonal
sequence and all limits are taken along it. Discarding finitely many terms in the
nonzero case yields
$\|w^{N,\vartheta}(0)\|_{\cH^1_{r^m\rho_0}}
\leq 2\|w_0\|_{\cH^1_{r^m\rho_0}}$. Therefore, we arrive at the uniform estimate:
\begin{equation}\label{uni-1}
\begin{aligned}
&\sup_{t\in[0,T]}\|w^{N,\vartheta}\|_{\cH^1_{r^m\rho_0}}^2+\int_0^T |w^{N,\vartheta}_t|_{2,r^m\rho_0}^2\,\mathrm{d}t\\
&\leq C(T)\Big(\|w_0\|_{\cH^1_{r^m\rho_0}}^2+\sup_{t\in[0,T]}|(\mathfrak{q}_1,\mathfrak{q}_2)|_2^2+\int_0^T |((\mathfrak{q}_1)_t,(\mathfrak{q}_2)_t)|_2^2\,\mathrm{d}t\Big).
\end{aligned}
\end{equation}

\smallskip
\textbf{3.} Note that, due to \eqref{gal-ass1}, we have $\mathfrak{B}(t)\in C^1[0,T],\qquad \mathfrak{c}(t)\in H^1(0,T)$, where $(\mathfrak{B}(t),\mathfrak{c}(t))$ are defined in \eqref{ABC}. Then the classical theory of ODEs (\textit{cf.} \cite{cod}) implies that $(\mu^{N,\vartheta})_t(t)\in AC([0,T])$, so that $w^{N,\vartheta}_{t}$ is differentiable \textit{a.e.} in $t$. As a consequence, 
we apply $\partial_t$ to $\eqref{galerkin-n}_1$ and obtain 
\begin{equation*}
\begin{aligned}
&\,\big<r^m\rho_0 w^{N,\vartheta}_{tt}, \xi_j\big>+2\mu\langle r^m\rho_0 D_{\bar\eta}w^{N,\vartheta}_t,D_{\bar\eta}\xi_j\rangle+2\mu m\big<\frac{r^m\rho_0}{\bar\eta^2}w^{N,\vartheta}_t,\xi_j\big>\\
&-4\mu\langle r^m\rho_0 D_{\bar\eta}\bar UD_{\bar\eta}w^{N,\vartheta},D_{\bar\eta}\xi_j\rangle-4\mu m\big<\frac{r^m\rho_0}{\bar\eta^3}\bar Uw^{N,\vartheta},\xi_j\big>\\ 
&=\big<(r^m\rho_0)^\frac{1}{2} \big((\mathfrak{q}_1)_t- \mathfrak{q}_1 \frac{\bar U}{\bar\eta}\big), \frac{m\xi_j}{\bar\eta}\big>+\langle(r^m\rho_0)^\frac{1}{2} \big((\mathfrak{q}_2)_t- \mathfrak{q}_2D_{\bar\eta}\bar U\big), D_{\bar\eta}\xi_j\rangle.
\end{aligned}
\end{equation*}

Next, multiplying the above by $(\mu_j^{N,\vartheta})_{t}(t)$ and summing the
resulting equality with respect to $j$ from $1$ to $N$,  it follows from the H\"older inequality that
\begin{align*}
&\begin{aligned}
&\,\frac{1}{2}\frac{\mathrm{d}}{\dt}\int_0^1 r^m\rho_0|w^{N,\vartheta}_t|^2\,\mathrm{d}r+2\mu\int_0^1 r^m\rho_0\big(|D_{\bar\eta}w^{N,\vartheta}_t|^2+m \frac{|w^{N,\vartheta}_t|^2}{\bar\eta^2} \big)\,\mathrm{d}r\\
&=4\mu \int_0^1 r^m\rho_0 D_{\bar\eta}\bar UD_{\bar\eta}w^{N,\vartheta} D_{\bar\eta}w^{N,\vartheta}_t\,\mathrm{d}r+4\mu m\int_0^1 \frac{r^m\rho_0}{\bar\eta^3}\bar Uw^{N,\vartheta} w^{N,\vartheta}_t\,\mathrm{d}r\\
&\quad +m\int_0^1(r^m\rho_0)^\frac{1}{2} \big((\mathfrak{q}_1)_t- \mathfrak{q}_1 \frac{\bar U}{\bar\eta}\big) \frac{w^{N,\vartheta}_t}{\bar\eta} \,\mathrm{d}r+\int_0^1(r^m\rho_0)^\frac{1}{2} \big((\mathfrak{q}_2)_t- \mathfrak{q}_2D_{\bar\eta}\bar U\big) D_{\bar\eta}w^{N,\vartheta}_t\,\mathrm{d}r
\end{aligned}\\
&\begin{aligned}
&\leq C\Big|\big(D_{\bar\eta}\bar U,\frac{\bar U}{\bar\eta}\big)\Big|_\infty\Big|\big(D_{\bar\eta} w^{N,\vartheta},\frac{w^{N,\vartheta}}{\bar\eta}\big)\Big|_{2,r^m\rho_0} \Big|\big(D_{\bar\eta} w^{N,\vartheta}_t,\frac{w^{N,\vartheta}_t}{\bar\eta}\big)\Big|_{2,r^m\rho_0}\\
&\quad +C\Big(\Big|\big(D_{\bar\eta}\bar U,\frac{\bar U}{\bar\eta}\big)\Big|_\infty|(\mathfrak{q}_1,\mathfrak{q}_2)|_2+|((\mathfrak{q}_1)_t,(\mathfrak{q}_2)_t)|_2\Big)\Big|\big(D_{\bar\eta} w_t^{N,\vartheta},\frac{w_t^{N,\vartheta}}{\bar\eta}\big)\Big|_{2,r^m\rho_0},
\end{aligned}
\end{align*}
which thus, along with the Young inequality, leads to
\begin{equation*}
\begin{aligned}
&\,\frac{\mathrm{d}}{\mathrm{d}t}|w^{N,\vartheta}_t|_{2,r^m\rho_0}^2+ \mu\|w^{N,\vartheta}_t\|_{\cH^1_{r^m\rho_0}}^2 \\
&\leq C(T)\Big(\Big|\big(D_{\bar\eta} w^{N,\vartheta},\frac{w^{N,\vartheta}}{\bar\eta}\big)\Big|_{2,r^m\rho_0}^2+|(\mathfrak{q}_1,\mathfrak{q}_2,(\mathfrak{q}_1)_t,(\mathfrak{q}_2)_t)|_2^2+|w^{N,\vartheta}_t|_{2,r^m\rho_0}^2\Big).
\end{aligned}
\end{equation*}
Multiplying the above by $t$ and integrating the resulting inequality over $[\tau,t]$, together with \eqref{uni-1}, we have
\begin{equation}\label{uni-2pre}
\begin{aligned}
&\,t|w^{N,\vartheta}_t(t)|_{2,r^m\rho_0}^2 +\int_\tau^ts\|w^{N,\vartheta}_t\|_{\cH^1_{r^m\rho_0}}^2\,\mathrm{d}s\\
&\leq\tau|w^{N,\vartheta}_t(\tau)|_{2,r^m\rho_0}^2+C(T)\Big(\|w_0\|_{\cH^1_{r^m\rho_0}}^2+\sup_{t\in[0,T]}|(\mathfrak{q}_1,\mathfrak{q}_2)|_2^2+\int_0^T |((\mathfrak{q}_1)_t,(\mathfrak{q}_2)_t)|_2^2\,\mathrm{d}t\Big).
\end{aligned}
\end{equation}
Finally, by Lemma \ref{bjr} and the fact that $(r^m\rho_0)^\frac{1}{2}w^{N,\vartheta}_t\in L^2([0,T];L^2)$, there exists a sequence $\{\tau_k\}_{k=1}^\infty\subset[0,T]$ such that
\begin{equation*}
\tau_k\to 0, \quad \tau_k|w^{N,\vartheta}_t(\tau_k)|_{2,r^m\rho_0}^2\to 0\qquad\,\,\,\text{as $k\to \infty$}.
\end{equation*}
Hence, setting $\tau=\tau_k$ in \eqref{uni-2pre} and then letting $\tau_k\to 0$, we arrive at
\begin{equation*}
\begin{aligned}
&\,t|w^{N,\vartheta}_t(t)|_{2,r^m\rho_0}^2 +\int_0^t s\|w^{N,\vartheta}_t\|_{\cH^1_{r^m\rho_0}}^2\,\mathrm{d}s\\
&\leq C(T)\Big(\|w_0\|_{\cH^1_{r^m\rho_0}}^2+\sup_{t\in[0,T]}|(\mathfrak{q}_1,\mathfrak{q}_2)|_2^2+\int_0^T |((\mathfrak{q}_1)_t,(\mathfrak{q}_2)_t)|_2^2\,\mathrm{d}t\Big).
\end{aligned}
\end{equation*}

\smallskip
\textbf{4.} Along the fixed diagonal sequence selected in \eqref{initial-diagonal-H1}, \eqref{reg-H} can be derived similarly from the weak convergence argument shown in Step 3 of the proof of Proposition \ref{prop1}. The resulting limit is a weak solution of \eqref{galerkin-w}; hence, by the uniqueness in Proposition \ref{prop1}, it coincides with the weak solution $w$ constructed there. For brevity, we omit the remaining details.

Next, we show that $\sqrt{t}r^m\rho_0w_{tt}\in L^2([0,T];\cH^{-1}_{r^m\rho_0})$.  By \eqref{weak.F.}, since $w_{t}\in L^2([0,T];L^2_{r^m\rho_0})$, we see that, 
for any $\varphi\in \cH^1_{r^m\rho_0}$,
\begin{equation*}
\begin{aligned}
&\left<r^m\rho_0 w_t, \varphi\right> +2\mu\langle r^m\rho_0 D_{\bar\eta} w,D_{\bar\eta}\varphi \rangle+2\mu m\big<\frac{r^m\rho_0}{\bar\eta^2}w,\varphi\big>\\
&=\big<(r^m\rho_0)^\frac{1}{2} \mathfrak{q}_1,\frac{m\varphi}{\bar\eta}\big>+\langle(r^m\rho_0)^\frac{1}{2} \mathfrak{q}_2, D_{\bar\eta}\varphi\rangle.
\end{aligned}
\end{equation*}
Differentiating the above with respect to $t$ gives that, for \textit{a.e.} $t>0$,
\begin{equation}\label{weak.F.tt}
\begin{aligned}
\frac{\mathrm{d}}{\mathrm{d}t}\left<r^m\rho_0 w_t, \varphi\right>&=\langle (r^m\rho_0)^\frac{1}{2} ((\mathfrak{q}_2)_t- \mathfrak{q}_2 D_{\bar\eta}\bar U) +2\mu r^m\rho_0 (2D_{\bar\eta} wD_{\bar\eta}\bar U- D_{\bar\eta} w_t),  D_{\bar\eta}\varphi \rangle\\
&\quad +\big< (r^m\rho_0)^\frac{1}{2} \big((\mathfrak{q}_1)_t-\mathfrak{q}_1\frac{\bar U}{\bar\eta}\big) + 2\mu r^m\rho_0 \big(\frac{2w\bar U}{\bar\eta^2}-  \frac{w_t}{\bar\eta}\big), \frac{m\varphi}{\bar\eta}\big>\\
&\leq C\Big(|((\mathfrak{q}_1)_t,(\mathfrak{q}_2)_t)|_2+\Big|\big(D_{\bar\eta}\bar U,\frac{\bar U}{\bar\eta}\big)\Big|_\infty|(\mathfrak{q}_1,\mathfrak{q}_2)|_2\Big)\|\varphi\|_{\cH^{1}_{r^m\rho_0}}\\
&\quad + C\Big(\|w_t\|_{\cH^{1}_{r^m\rho_0}}+\Big|\big(D_{\bar\eta}\bar U,\frac{\bar U}{\bar\eta}\big)\Big|_\infty\|w\|_{\cH^{1}_{r^m\rho_0}}\Big)\|\varphi\|_{\cH^{1}_{r^m\rho_0}},
\end{aligned}
\end{equation}
which leads to
\begin{equation*}
\Big|\frac{\mathrm{d}}{\mathrm{d}t}\big<(r^m\rho_0 w_t), \varphi\big>\Big|\leq F(t)\|\varphi\|_{\cH^{1}_{r^m\rho_0}} \qquad\text{for some $F(t)\in L^2(\tau,T)$ with $\tau \in (0,T)$}.
\end{equation*}
This, along with Lemma 1.1 on \cite{temam}*{page 250}, gives $(r^m\rho_0 w_{tt})\in L^2([\tau,T];\cH^{-1}_{r^m\rho_0})$. Hence, it follows from Lemma \ref{Aubin} that, for any $t\in (\tau,T)$ with $\tau \in (0,T)$,
\begin{equation}\label{dt-t}
\big<(r^m\rho_0 w_{tt}),w_t\big>_{\cH^{-1}_{r^m\rho_0}\times \cH^{1}_{r^m\rho_0}}=\frac{1}{2}\frac{\mathrm{d}}{\mathrm{d}t} \|w_t\|_{L^2_{r^m\rho_0}}^2 .
\end{equation}
Finally, \eqref{weak-F-H}--\eqref{Energy-Id} follow from \eqref{weak.F.tt}--\eqref{dt-t}. 

This completes the proof of Proposition \ref{prop-strong}.
\end{proof}

\subsubsection{Proof for Lemma \ref{existence-linearize}}\label{subsection3.3}
We now prove Lemma \ref{existence-linearize}. We divide the proof into eight steps. 

\smallskip
\textbf{1. Tangential estimate $U\in C([0,T];L^2_{r^m\rho_0})$.} First, we let 
\begin{equation*}
w^{(0)}|_{t=0}=w_0^{(0)}:=u_0,\qquad \mathfrak{q}_1^{(0)}=\mathfrak{q}_2^{(0)}:=A\frac{(r^m\rho_0)^{\gamma-\frac{1}{2}}}{(\bar\eta^{m}\bar\eta_r)^{\gamma-1}},
\end{equation*}
in \eqref{lp}. Then we can check that $(\mathfrak{q}_1^{(0)},\mathfrak{q}_2^{(0)})\in L^2([0,T];L^2)$ and $w_0^{(0)}\in L^2_{r^m\rho_0}$, thereby using Proposition \ref{prop1} to obtain a unique weak solution $w^{(0)}=U$ of \eqref{lp} such that
\begin{equation}\label{3..142}
U\in C([0,T];L^2_{r^m\rho_0})\cap L^2([0,T];\cH^1_{r^m\rho_0}),\qquad r^m\rho_0 U_t\in L^2([0,T];\cH^{-1}_{r^m\rho_0}).
\end{equation}

\smallskip
\textbf{2. Tangential estimates $U\in C([0,T];\cH^1_{r^m\rho_0})$ and $U_t\in C([0,T];L^2_{r^m\rho_0})$.}
First, if formally applying $\partial_t$ to both sides of the equation in \eqref{lp}, we would obtain
\begin{equation}\label{33151}
r^m\rho_0U_{tt}-2\mu \big(r^m\rho_0\frac{U_{tr}}{\bar\eta_r^2}  \big)_r +2\mu mr^m\rho_0\frac{U_t}{\bar\eta^2}=m\frac{(r^m\rho_0)^\frac{1}{2}}{\bar\eta}\mathfrak{q}_1^{(1)} - \big( \frac{(r^m\rho_0)^\frac{1}{2}}{\bar\eta_r}\mathfrak{q}_2^{(1)}\big)_{r},
\end{equation}
where
\begin{equation}\label{q1-q2-1}
\begin{aligned}
\mathfrak{q}_1^{(1)}&:=4\mu (r^m\rho_0)^\frac{1}{2}\frac{U \bar U}{\bar\eta^{2}} \ \underline{-A\frac{(r^{m}\rho_0)^{\gamma-\frac{1}{2}}}{(\bar\eta^{m}\bar\eta_r)^{\gamma-1}} \big(\frac{((\gamma-1) m+1)\bar U}{\bar\eta}+(\gamma-1)D_{\bar\eta}\bar U\big)}_{:=\mathfrak{q}_{1*}^{(1)}},\\
\mathfrak{q}_2^{(1)}&:=4\mu (r^m\rho_0)^\frac{1}{2}D_{\bar\eta}UD_{\bar\eta}\bar U \  \underline{- A\frac{(r^{m}\rho_0)^{\gamma-\frac{1}{2}}}{(\bar\eta^{m}\bar\eta_r)^{\gamma-1}}\big(\frac{(\gamma-1)m\bar U}{\bar\eta}+\gamma D_{\bar\eta}\bar U\big)}_{:=\mathfrak{q}_{2*}^{(1)}}.
\end{aligned}
\end{equation}
As a consequence, we regard \eqref{33151} as the equation of $w^{(1)}:=U_t$ and study the problem: 
\begin{equation}\label{33152}
\!\!\begin{cases}
\displaystyle r^m\rho_0w_t^{(1)}-2\mu \big(r^m\rho_0\frac{w_r^{(1)}}{\bar\eta_r^2}  \big)_r+2\mu mr^m\rho_0\frac{w^{(1)}}{\bar\eta^2}=m\frac{(r^m\rho_0)^\frac{1}{2}}{\bar\eta} \mathfrak{q}_1^{(1)}-\big(\frac{(r^m\rho_0)^\frac{1}{2}}{\bar\eta_r} \mathfrak{q}_2^{(1)}\big)_r,\\[10pt]
w^{(1)}|_{t=0}=w^{(1)}_0=U_t|_{t=0}\qquad\text{on } I.
\end{cases}
\end{equation}
Clearly, $w^{(1)}_0\in L^2_{r^m\rho_0}$ and $(\mathfrak{q}_1^{(1)},\mathfrak{q}_2^{(1)})\in L^2([0,T];L^2)$ due to \eqref{given}--\eqref{jibenjiashe} and \eqref{3..142}. Hence, by Proposition \ref{prop1},  \eqref{33152} admits a unique weak solution $w^{(1)}$ satisfying
\begin{equation}\label{3.142'}
w^{(1)}\in C([0,T];L^2_{r^m\rho_0})\cap L^2([0,T];\cH^1_{r^m\rho_0}),\qquad r^m\rho_0 w^{(1)}_t\in L^2([0,T];\cH^{-1}_{r^m\rho_0}).
\end{equation}
Now, we show that  $w^{(1)}=U_t$ for {\it a.e.} $(t,r)\in (0,T)\times I$. Define 
\begin{equation}\label{rela}
\widetilde U(t,r):=\int_0^t w^{(1)}(\tau,r)\,\mathrm{d}\tau+ u_0(r), \qquad \ Y:=\widetilde U-U.
\end{equation}
It suffices to show that $Y=0$ for {\it a.e.} $(t,r)\in (0,T)\times I$.

Since $U$ and $w^{(1)}$ are weak solutions of \eqref{lp} and \eqref{33152}, respectively, it follows from \eqref{weak.F.} that, for all $\varphi \in \cH^1_{r^m\rho_0}$ and {\it a.e.} $t\in(0,T)$,
\begin{equation}\label{equ1}
\begin{aligned}
&\,\left<r^m\rho_0 U_t, \varphi\right>_{\cH^{-1}_{r^m\rho_0}\times \cH^1_{r^m\rho_0}} +2\mu\big<r^m\rho_0 D_{\bar\eta}U,D_{\bar\eta}\varphi \big>+2\mu m\big<\frac{r^m\rho_0}{\bar\eta^2}U,\varphi\big>\\
&=\big<(r^m\rho_0)^\frac{1}{2} \mathfrak{q}_1^{(0)}, \frac{m\varphi}{\bar\eta}\big>+\langle(r^m\rho_0)^\frac{1}{2} \mathfrak{q}_2^{(0)}, D_{\bar\eta} \varphi\rangle,
\end{aligned}
\end{equation}
and
\begin{equation}\label{equ1'}
\begin{aligned}
&\,\big<r^m\rho_0 w^{(1)}_t, \varphi\big>_{\cH^{-1}_{r^m\rho_0}\times \cH^1_{r^m\rho_0}} +2\mu\big<r^m\rho_0 D_{\bar\eta}w^{(1)},D_{\bar\eta}\varphi \big>+2\mu m\big<\frac{r^m\rho_0}{\bar\eta^2}w^{(1)},\varphi\big>\\
&=\big<(r^m\rho_0)^\frac{1}{2} \mathfrak{q}_1^{(1)}, \frac{m\varphi}{\bar\eta}\big>+\langle(r^m\rho_0)^\frac{1}{2} \mathfrak{q}_2^{(1)}, D_{\bar\eta} \varphi\rangle.
\end{aligned}
\end{equation}
Next, due to \eqref{rela}, we replace $w^{(1)}$ in \eqref{equ1'} by $\widetilde U_t$ and integrate the resulting equality over $[0,t]$ for $t\in (0,T]$. Using the compatibility condition $\eqref{116}_1$ in Appendix \ref{AppB}:
\begin{equation*}
r^m\rho_0 U_t(0,r) =2\mu (r^m\rho_0(u_0)_r)_r-2\mu m r^m\rho_0 \frac{u_0}{r^2} +Am r^{m-1}\rho_0^\gamma-A (r^m\rho_0^\gamma)_r,
\end{equation*}
together with the relations:
\begin{equation*}
\big(\frac{\mathfrak{q}_1^{(0)}}{\bar\eta}\big)_t=\frac{\mathfrak{q}_{1*}^{(1)}}{\bar\eta},\qquad\big(\frac{\mathfrak{q}_2^{(0)}}{\bar\eta_r}\big)_t=\frac{\mathfrak{q}_{2*}^{(1)}}{\bar\eta_r},
\end{equation*}
we arrive at the following equality:
\begin{equation*}
\begin{aligned}
&\,\big<r^m\rho_0 \widetilde U_t,\varphi\big>_{\cH^{-1}_{r^m\rho_0}\times \cH^1_{r^m\rho_0}}+2\mu\big<r^m\rho_0 D_{\bar\eta}\widetilde U,D_{\bar\eta}\varphi \big>+2\mu m\big<\frac{r^m\rho_0}{\bar\eta^2}\widetilde U,\varphi\big>\\
&=\big<(r^m\rho_0)^\frac{1}{2} \mathfrak{q}_1^{(0)}, \frac{m\varphi}{\bar\eta}\big>+\langle(r^m\rho_0)^\frac{1}{2} \mathfrak{q}_2^{(0)}, D_{\bar\eta} \varphi\rangle\\
&\quad -4\mu\int_0^t \big< r^m\rho_0 D_{\bar\eta}\bar U D_{\bar\eta} Y, D_{\bar\eta}\varphi\big>\,\ds-4\mu m\int_0^t \big<r^m\rho_0\frac{\bar UY}{\bar\eta^3},\varphi\big>\,\mathrm{d}s.
\end{aligned}
\end{equation*}
Subtracting \eqref{equ1} from the above equality leads to
\begin{equation}\label{equ3}
\begin{aligned}
&\,\left<r^m\rho_0 Y_t,\varphi\right>_{\cH^{-1}_{r^m\rho_0}\times \cH^{1}_{r^m\rho_0}}+2\mu\big<r^m\rho_0 D_{\bar\eta}Y,D_{\bar\eta}\varphi \big>+2\mu m\big<\frac{r^m\rho_0}{\bar\eta^2}Y,\varphi\big>\\
&=-4\mu\int_0^t \big< r^m\rho_0 D_{\bar\eta}\bar U D_{\bar\eta} Y, D_{\bar\eta}\varphi\big>\,\ds-4\mu m\int_0^t \big<r^m\rho_0\frac{\bar UY}{\bar\eta^3},\varphi\big>\,\mathrm{d}s. 
\end{aligned}
\end{equation}
Thanks to $Y\in L^2([0,T];\cH^1_{r^m\rho_0})$, we set $\varphi=Y$ in \eqref{equ3} and obtain from Lemma \ref{hardy-inequality}, the Young inequality, and the similar calculations in \eqref{I11I22} that
\begin{equation*}
\begin{aligned}
\frac{\mathrm{d}}{\dt}|Y|_{2,r^m\rho_0}^2+\|Y\|_{\cH^1_{r^m\rho_0}}^2\leq C\Big(\Big|\big( D_{\bar\eta} \bar U,\frac{\bar U}{\bar\eta}\big)\Big|_{\infty}+1\Big)\int_0^t \|Y\|_{\cH^1_{r^m\rho_0}}^2\,\ds.
\end{aligned}
\end{equation*}
Integrating the above over $[0,t]$, together  with the strong continuity of $Y$ at $t=0$ and $Y|_{t=0}=0$, yields  that 
\begin{equation*}
\int_0^t \|Y\|_{\cH^1_{r^m\rho_0}}^2\,\ds\leq C\Big(\sup_{s\in [0,t]}\Big|\big( D_{\bar\eta} \bar U,\frac{\bar U}{\bar\eta}\big)\Big|_{\infty}+1\Big)\int_0^t\Big(\int_0^\tau\|Y\|_{\cH^1_{r^m\rho_0}}^2\,\ds\Big)\mathrm{d}\tau.
\end{equation*}
Then it follows from the Gr\"onwall inequality that $\|Y\|_{L^2_t(\cH^1_{r^m\rho_0})}= 0$, and hence $Y=0$.

Consequently, \eqref{3.142'} holds for $U_t$, \textit{i.e.},
\begin{equation}\label{33153}
U_t\in C([0,T];L^2_{r^m\rho_0})\cap L^2([0,T];\cH^1_{r^m\rho_0}),\qquad r^m\rho_0 U_{tt}\in L^2([0,T];\cH^{-1}_{r^m\rho_0}),
\end{equation}
which, together  with \eqref{3..142} and Lemma \ref{sobolev-embedding}, yields
\begin{equation}\label{equ33.38}
U\in C([0,T];\cH^1_{r^m\rho_0}).
\end{equation}

\smallskip
\textbf{3. Boundary condition of $U$.} Thanks to  \eqref{lp} and  \eqref{weak.F.}, we can show that $\eqref{lp}_1$ holds for \textit{a.e.} $(t,r)\in (0,T)\times (0,1)$, and $U$ satisfies the boundary condition $\rho_0 U_r|_{r=1}=0$. These are crucial to our further analysis.
\begin{Lemma}\label{Lemma-point}
For any $a\in (0,1)$ and \textit{a.e.} $t\in (0,T)$,
\begin{equation}\label{div-Uxx}
r^\frac{m}{2}\big(D_{\bar\eta}^2U,D_{\bar\eta}(\frac{U}{\bar\eta})\big)\in L^2(0,a),\qquad \big(r^m\rho_0\frac{U_r}{\bar\eta_r^2}\big)_r\in L^2(a,1).
\end{equation}
Furthermore, equation $\eqref{lp}_1$ holds for \textit{a.e.} $(t,r)\in (0,T)\times (0,1)$, and $U$ satisfies 
\begin{equation}\label{equ33.41}
U\in H^3_{\mathrm{loc}},\qquad \ \rho_0 U_r|_{r=1}=0 \quad \text{for {\it a.e.} $t\in (0,T)$}.
\end{equation}
\end{Lemma}
\begin{proof}
First, it follows from \eqref{weak.F.} in Definition \ref{def3.1}, \eqref{33153}, and the H\"older inequality that, for all $\varphi\in C_{\mathrm{c}}^\infty((0,1])$,
\begin{equation*}
\begin{aligned}
\Big|\big<r^m\rho_0\frac{U_r}{\bar\eta_r^2},\varphi_r \big>\Big|&\leq C\Big|\big<\frac{r^m\rho_0}{\bar\eta^2}U,\varphi\big>\Big|+C\big|\langle r^m\rho_0 U_t, \varphi\rangle\big|\\
&\quad + C\Big|\big<(r^m\rho_0)^\frac{1}{2} \frac{\mathfrak{q}_1^{(0)}}{\bar\eta}, \varphi\big>\Big|+C\big|\big<\big((r^m\rho_0)^\frac{1}{2} \frac{\mathfrak{q}_2^{(0)}}{\bar\eta_r}\big)_r, \varphi\big>\big|\leq C_\varphi |\varphi|_2,
\end{aligned}
\end{equation*}
where $C_\varphi>0$ is a constant that depends on the support of $\varphi$. This yields that $r^m\rho_0\frac{U_r}{\bar\eta_r^2}$ admits the weak derivative 
\begin{equation*}
\big(r^m\rho_0\frac{U_r}{\bar\eta_r^2}\big)_r\in L^2(a,1) \qquad\text{for $a\in(0,1)$ and \textit{a.e.} $t\in (0,T)$},
\end{equation*}
which, along with \eqref{33153} and the regularity of $\bar\eta$ in \eqref{given-bareta}, yields that each term in $\eqref{lp}_1$ belongs to $L^1_{\mathrm{loc}}$, and hence $\eqref{lp}_1$ holds for \textit{a.e.} $(t,r)\in (0,T)\times (0,1)$. Moreover, rewrite $\eqref{lp}_1$ as
\begin{equation*}
(\mathfrak{A}_1 U_r)_r+\mathfrak{A}_2 U=\mathfrak{A}_3,
\end{equation*}
where
\begin{equation*}
\mathfrak{A}_1=2\mu\frac{r^m\rho_0}{\bar\eta_r^2},\quad\mathfrak{A}_2=-2\mu m \frac{r^m\rho_0}{\bar\eta^2},\quad \mathfrak{A}_3=r^m\rho_0U_t - m(r^m\rho_0)^\frac{1}{2}\frac{\mathfrak{q}_1^{(0)}}{\bar\eta}+\big((r^m\rho_0)^\frac{1}{2}\frac{\mathfrak{q}_2^{(0)}}{\bar\eta_r}\big)_r.
\end{equation*}
As can be checked, $(\mathfrak{A}_1,\mathfrak{A}_2)\in H^3_{\mathrm{loc}}$ and $\mathfrak{A}_3\in H^1_{\mathrm{loc}}$  for \textit{a.e.} $t\in (0,T)$. Hence, it follows from the classical regularity theory of elliptic equations (\textit{cf.} \cites{evans}) that $U\in H^3_{\mathrm{loc}}$ for \textit{a.e.} $t\in (0,T)$.

Next, we show $\eqref{div-Uxx}_1$. Since $U\in H^3_{\mathrm{loc}}$ and $U_t\in H^1_{\mathrm{loc}}$ for \textit{a.e.} $t\in (0,T)$, and $\eqref{lp}_1$ holds for \textit{a.e.} $(t,r)\in (0,T)\times (0,1)$, we multiply $\eqref{lp}_1$ by $(r^m \rho_0)^{-1}$ to obtain
\begin{equation}\label{new-lp}
D_{\bar\eta}\big(D_{\bar\eta} U+ \frac{mU}{\bar\eta}\big) =\frac{1}{2\mu} U_t+\frac{A\gamma}{2\mu}\bar\varrho^{\gamma-2} D_{\bar\eta}\bar\varrho-\frac{D_{\bar\eta}\bar\varrho}{\bar\varrho} D_{\bar\eta} U,\qquad\bar\varrho:=\frac{r^m\rho_0}{\bar\eta^m\bar\eta_r}.
\end{equation}
Then we can directly check that, for any $a\in (0,1)$, 
\begin{equation*}
\Big|\zeta_a D_{\bar\eta}\big(D_{\bar\eta} U+ \frac{mU}{\bar\eta}\big)\Big|_{2,r^m}\leq C\big(|\zeta_aU_t|_{2,r^m}+ |\zeta_\frac{1+3a}{4}\bar\varrho^{\gamma-2}|_\infty|\zeta_a\bar\varrho_r|_{2,r^m}+|\zeta_a \bar\varrho_r D_{\bar\eta} U|_{2,r^m}\big)\leq C(a,T),
\end{equation*}
due to \eqref{33153}--\eqref{equ33.38}, the regularity of $\bar\eta$ in \eqref{given-bareta}, and the fact that $\varrho^\beta\sim 1-r$. Hence, by Lemma \ref{im-1} in \S \ref{Section-globalestimates}, we derive $\eqref{div-Uxx}_1$.

Finally, we give the proof for  \eqref{equ33.41}. Due to the weak formulation \eqref{weak.F.},  we obtain that, for all $\varphi\in C_{\mathrm{c}}^\infty((0,1])$,
\begin{equation}\label{WF'}
\begin{aligned}
&\left<r^m\rho_0 U_t, \varphi\right> +2\mu\langle r^m\rho_0 D_{\bar\eta} U,D_{\bar\eta}\varphi \rangle+2\mu m\big<\frac{r^m\rho_0}{\bar\eta^2}U,\varphi\big>\\
&=\big<(r^m\rho_0)^\frac{1}{2} \mathfrak{q}_1^{(0)},\frac{m\varphi}{\bar\eta}\big>+\langle(r^m\rho_0)^\frac{1}{2} \mathfrak{q}_2^{(0)}, D_{\bar\eta}\varphi\rangle.
\end{aligned}
\end{equation}
Moreover, since  Lemma \ref{sobolev-embedding} and \eqref{div-Uxx} lead 
to $\rho_0 U_r\in C([0,T]\times [a,1])$ for any $a\in (0,1)$, multiplying $\eqref{lp}_1$ by such $\varphi$ and integrating over $I$ yield 
\begin{equation}\label{WF''}
\begin{aligned}
&\left<r^m\rho_0 U_t, \varphi\right> +2\mu\langle r^m\rho_0 D_{\bar\eta} U,D_{\bar\eta}\varphi \rangle+2\mu m\big<\frac{r^m\rho_0}{\bar\eta^2}U,\varphi\big>\\
&=2\mu\frac{r^m\rho_0 U_r}{\bar\eta_r^2}\varphi\Big|_{r=1}+\big<(r^m\rho_0)^\frac{1}{2} \mathfrak{q}_1^{(0)},\frac{m\varphi}{\bar\eta}\big>+\langle(r^m\rho_0)^\frac{1}{2} \mathfrak{q}_2^{(0)}, D_{\bar\eta}\varphi\rangle.
\end{aligned}
\end{equation}
Comparing \eqref{WF'}--\eqref{WF''}, together with $(\bar\eta_r,\frac{\bar\eta}{r})\in [\frac{1}{2},\frac{3}{2}]$ for $(t,r)\in [0,T]\times \bar I$, we have
\begin{equation*}
r^m\rho_0 U_r\varphi\big|_{r=1}=0 \qquad\text{for all $\varphi\in C_{\mathrm{c}}^\infty((0,1])$},
\end{equation*}
which implies that  $\rho_0 U_r|_{r=1}=0$.  This completes the proof of Lemma \ref{Lemma-point}.
\end{proof}

\textbf{4. Tangential estimates $U_t\in L^\infty([0,T];\cH^1_{r^m\rho_0})$ and time-weighted estimates.}
We continue to improve the regularity of $U$. First, we claim that
\begin{equation}\label{L2Linfty}
\rho_0^\frac{1-\beta}{2}D_{\bar\eta}U\in L^\infty([0,T];L^\infty(a,1)) \qquad\text{for $a\in(0,1)$}.
\end{equation}
Indeed, since $\rho_0 D_{\bar\eta}U\in C[a,1]$ for $a\in (0,1)$ due to Lemmas \ref{Lemma-point} and \ref{sobolev-embedding}, by integrating \eqref{new-lp} over $[r,1]$ ($r\in (0,1)$), it follows from a similar argument in Lemma \ref{lemma-time-space}, the fact that
\begin{equation*}
1-r\sim \rho_0^\beta\in H^3(a,1),\qquad\beta\leq \gamma-1,
\end{equation*}
\eqref{33153}, \eqref{equ33.41}, Lemma \ref{hardy-inequality}, and the H\"older inequality that
\begin{equation}\label{3...136}
\begin{aligned}
& \ D_{\bar\eta} U(t,r) =\frac{A}{2\mu}\bar\varrho^{\gamma-1} +\frac{m}{\bar\varrho}\int_r^1 \bar\varrho\big(\frac{U_r}{\bar\eta}-\frac{U\bar\eta_r}{\bar\eta^2}\big)\mathrm{d}\tilde{r}-\frac{1}{2\mu \bar\varrho}\int_r^1\bar\varrho \bar\eta_r U_t\,\mathrm{d}\tilde{r}\\
&\implies \big|\chi^\sharp_a\rho_0^\frac{1- \beta}{2}D_{\bar\eta}U\big|_\infty\leq C(a)\big(1 + \big|\chi^\sharp_a\rho_0^\frac{1}{2}(U,D_{\bar\eta} U,U_t)\big|_2\big)\leq C(a,T).
\end{aligned}
\end{equation}
This implies claim \eqref{L2Linfty}. 

Next, by direct calculation, $w^{(1)}(0,r)\in \cH^1_{r^m\rho_0}$ and $(\mathfrak{q}_1^{(1)},\mathfrak{q}_2^{(1)})\in L^\infty([0,T];L^2)$ due to \eqref{given}--\eqref{jibenjiashe}, and \eqref{3..142}. It still remains to show that
\begin{equation}\label{q1t-q2t}
((\mathfrak{q}_1^{(1)})_t,(\mathfrak{q}_2^{(1)})_t)\in L^2([0,T];L^2).
\end{equation}
To obtain this, we first have
\begin{align}
&\begin{aligned}
(\mathfrak{q}_1^{(1)})_t&=4\mu (r^m\rho_0)^\frac{1}{2}\big(\frac{U \bar U_t}{\bar\eta^{2}}+\frac{U_t \bar U}{\bar\eta^{2}}-\frac{2U \bar U^2}{\bar\eta^{3}}\big)\\
&\quad -A\frac{(r^{m}\rho_0)^{\gamma-\frac{1}{2}}}{(\bar\eta^{m}\bar\eta_r)^{\gamma-1}} \Big(((\gamma-1) m+1) \big(\frac{\bar U_t}{\bar\eta}-\frac{\bar U^2}{\bar\eta^2}\big)+(\gamma-1)(D_{\bar\eta}\bar U_t-|D_{\bar\eta}\bar U|^2)\Big)\\
&\quad +A(\gamma-1)\frac{(r^{m}\rho_0)^{\gamma-\frac{1}{2}}}{(\bar\eta^{m}\bar\eta_r)^{\gamma-1}} \big(\frac{((\gamma-1) m+1)\bar U}{\bar\eta}+(\gamma-1)D_{\bar\eta}\bar U\big)\big(D_{\bar\eta}\bar U+\frac{m\bar U}{\bar\eta}\big),
\end{aligned}\label{q1-q2-1-t}\\
&\begin{aligned}
(\mathfrak{q}_2^{(1)})_t&=4\mu (r^m\rho_0)^\frac{1}{2}\big(D_{\bar\eta}UD_{\bar\eta}\bar U_t+D_{\bar\eta}U_tD_{\bar\eta}\bar U-2D_{\bar\eta}U|D_{\bar\eta}\bar U|^2\big) \\
&\quad - A\frac{(r^{m}\rho_0)^{\gamma-\frac{1}{2}}}{(\bar\eta^{m}\bar\eta_r)^{\gamma-1}}\Big((\gamma-1)m\big(\frac{\bar U_t}{\bar\eta}-\frac{\bar U^2}{\bar\eta^2}\big)+\gamma (D_{\bar\eta}\bar U_t-|D_{\bar\eta}\bar U|^2)\Big)\\
&\quad +A(\gamma-1)\frac{(r^{m}\rho_0)^{\gamma-\frac{1}{2}}}{(\bar\eta^{m}\bar\eta_r)^{\gamma-1}}\big(\frac{(\gamma-1)m\bar U}{\bar\eta}+\gamma D_{\bar\eta}\bar U\big)\big(D_{\bar\eta}\bar U+\frac{m\bar U}{\bar\eta}\big).
\end{aligned}\label{q1-q2-1-ttt}
\end{align}
Then, with the help of \eqref{L2Linfty}, we can obtain \eqref{q1t-q2t} by using the regularities of $(\bar\eta,\bar U)$, the facts that $\rho_0^\beta\sim 1-r$ and $\beta>\frac{1}{3}$, \eqref{33153}--\eqref{equ33.38}, and Lemmas \ref{Lemma-point} and \ref{hardy-inequality}, for example,
\begin{equation*}
\begin{aligned}
&\,\big|(r^m\rho_0)^\frac{1}{2} D_{\bar\eta}UD_{\bar\eta}\bar U_t\big|_2 \leq C\big(\big|\chi r^\frac{m}{2} D_{\bar\eta}UD_{\bar\eta}\bar U_t\big|_2+\big|\chi^\sharp \rho_0^\frac{1}{2} D_{\bar\eta}UD_{\bar\eta}\bar U_t\big|_2\big)\\
&\leq C\big|\chi r^\frac{m}{4} D_{\bar\eta}U|_4\big|\chi r^\frac{m}{4}D_{\bar\eta}\bar U_t\big|_4+C\big|\chi^\sharp\rho_0^\frac{1- \beta}{2}D_{\bar\eta}U\big|_\infty\big|\chi^\sharp\rho_0^\frac{\beta}{2}D_{\bar\eta}\bar U_t\big|_2\\
&\leq C(T)\big(\big|\chi r^\frac{m+3}{4}(D_{\bar\eta}U,D_{\bar\eta}^2U)\big|_2\big|\chi  r^\frac{m+3}{4}(D_{\bar\eta}\bar U_t,D_{\bar\eta}^2\bar U_t)\big|_2+ \big|\chi^\sharp \rho_0^\frac{3\beta}{2}(D_{\bar\eta}\bar U_t,D_{\bar\eta}^2\bar U_t)\big|_2\big)\\
&\leq C(T)(1+\bar\cD(t,\bar U)^\frac{1}{2}),
\end{aligned}
\end{equation*}
which implies that $(r^m\rho_0)^\frac{1}{2} D_{\bar\eta}UD_{\bar\eta}\bar U_t\in L^2([0,T];L^2)$ due to $\bar\cD(t,\bar U)\in L^1(0,T)$. The remaining calculation is straightforward.

Therefore, from Proposition \ref{prop-strong}, it follows that the weak solution $U_t$ of \eqref{33152} satisfies \eqref{weak-F-H}--\eqref{Energy-Id} with $(w,\mathfrak{q}_1,\mathfrak{q}_2)$ replaced by $(U_t,\mathfrak{q}_1^{(1)},\mathfrak{q}_2^{(1)})$ and
\begin{equation}\label{linear-D3-qiexiang}
\begin{aligned}
&U_t\in L^\infty([0,T];\cH^1_{r^m\rho_0}),\qquad U_{tt}\in L^2([0,T];L^2_{r^m\rho_0}),\\
&\sqrt{t}U_{tt}\in L^\infty([0,T];L^2_{r^m\rho_0})\cap L^2([0,T];\cH^{1}_{r^m\rho_0}),\quad \sqrt{t}(r^m\rho_0U_{ttt})\in L^2([0,T];\cH^{-1}_{r^m\rho_0}).
\end{aligned}
\end{equation}
Moreover, using a similar argument in Lemma \ref{Lemma-point}, we can deduce from \eqref{weak-F-H} that \eqref{33151} holds for \textit{a.e.} $(t,r)\in (0,T)\times (0,1)$, and for any $a\in (0,1)$ and \textit{a.e.} $t\in (0,T)$,
\begin{equation}\label{div-Uxx-H}
\begin{aligned}
&r^\frac{m}{2}\big(D_{\bar\eta}^2U_t,D_{\bar\eta}(\frac{U_t}{\bar\eta})\big)\in L^2(0,a),\qquad \big(r^m\rho_0\frac{U_{tr}}{\bar\eta_r^2}\big)_r\in L^2(a,1),\\
&U_t\in H^2_{\mathrm{loc}},\qquad U\in H^4_{\mathrm{loc}}, \qquad \rho_0 U_{tr}|_{r=1}=0.
\end{aligned}
\end{equation}

In summary, collecting \eqref{3..142}, \eqref{33153}--\eqref{equ33.38}, and \eqref{linear-D3-qiexiang}, we arrive at all the tangential estimates for $U$:
\begin{equation}\label{TETE}
\begin{aligned}
&U\in C([0,T];\cH^1_{r^m\rho_0}),\qquad U_t\in C([0,T];L^2_{r^m\rho_0})\cap L^\infty([0,T];\cH^1_{r^m\rho_0}),\\
&U_{tt}\in L^2([0,T];L^2_{r^m\rho_0}),\qquad
\sqrt{t}U_{tt}\in L^\infty([0,T];L^2_{r^m\rho_0})\cap L^2([0,T];\cH^{1}_{r^m\rho_0}),\\
&r^m\rho_0 U_{tt}\in L^2([0,T];\cH^{-1}_{r^m\rho_0}),\qquad \sqrt{t}(r^m\rho_0U_{ttt})\in L^2([0,T];\cH^{-1}_{r^m\rho_0}).
\end{aligned}
\end{equation}

\smallskip
\textbf{5. Total energy and dissipation estimates for $U$.} With the help of \eqref{TETE}, we now can  show that
\begin{equation}\label{bbbb}
\begin{aligned}
\bar\cE_{\mathrm{in}}(t,U)+t\bar\cD_{\mathrm{in}}(t,U)\in L^\infty(0,T),\qquad \bar\cD_{\mathrm{in}}(t,U)\in L^1(0,T),\\
\bar\cE_{\mathrm{ex}}(t,U)+t\bar\cD_{\mathrm{ex}}(t,U)\in L^\infty(0,T),\qquad \bar\cD_{\mathrm{ex}}(t,U)\in L^1(0,T).
\end{aligned}
\end{equation}
In fact, the regularity properties that $U\in H^4_{\mathrm{loc}}$ and $U_t\in H^2_{\mathrm{loc}}$, which follow from \eqref{div-Uxx-H}, allow us to freely apply $D_{\bar\eta}$, $D_{\bar\eta}^2$, or $\partial_t$ to \eqref{new-lp}. This process can then be used to rigorously obtain higher spatial regularity for $U$. However, the specific calculation above is rather tedious; we will provide the details in Lemmas \ref{c_1-c_2}--\ref{c_3} of \S \ref{subsub-11.2.2}, where the method can be applied here in a similar manner.

In fact, the major difference here, compared with the calculations in Lemmas \ref{c_1-c_2}--\ref{c_3}, lies in establishing 
\begin{equation}\label{eequ3.66}
\chi^\sharp\rho_0^{(\frac{3}{2}-\varepsilon_0)\beta}D_{\bar\eta}^4 U \in L^2([0,T];L^2),\qquad \sqrt{t}\chi^\sharp\rho_0^{(\frac{3}{2}-\varepsilon_0)\beta}D_{\bar\eta}^4 U \in L^\infty([0,T];L^2).
\end{equation}
Note that the derivations of $\eqref{eequ3.66}_1$ and $\eqref{eequ3.66}_2$ are identical. We therefore only sketch the proof of $\eqref{eequ3.66}_1$. Thus, assume that $\eqref{bbbb}_1$ holds, and the following estimates for $U$ in the 
exterior domain have been derived:
\begin{equation}\label{lowlow}
\bar\cE_{\mathrm{ex}}(t,U)\in L^\infty(0,T),\qquad (\rho_0^\frac{1}{2}U_{tt},\rho_0^{(\frac{3}{2}-\varepsilon_0)\beta}D_{\bar\eta}^2U_t)\in L^2([0,T];L^2).
\end{equation}
First, we can multiply  \eqref{new-lp} by $\varrho^{\beta}$ and then apply $\varrho^{-\beta}\bar\eta_rD_{\bar\eta}^2$ to the resulting equation to obtain
\begin{equation}\label{xxxx-l}
\begin{aligned}
&\bar\cT_{\mathrm{cross}}=\bar\cT_{\mathrm{cross}}(t,r):=(D_{\bar\eta}^3U)_r+\big(\frac{1}{\beta}+2\big)\frac{(\rho_0^\beta)_r}{\rho_0^\beta}D_{\bar\eta}^3U =\sum_{i=12}^{15}\mathrm{I}_{i},
\end{aligned}
\end{equation}
where $\mathrm{I}_i$ ($i=12,13,14,15$) is defined in \eqref{rr1-rr3} of \S\ref{subsub-11.2.2}. As can be checked, $\chi^\sharp\rho_0^{(\frac{3}{2}-\varepsilon_0)\beta}\mathrm{I}_i\in L^2([0,T];L^2)$ for $i=12,13,14,15$ (see Lemma \ref{c_3} for details), and hence
\begin{equation*}
\chi^\sharp\rho_0^{(\frac{3}{2}-\varepsilon_0)\beta}\bar\cT_{\mathrm{cross}} \in L^2([0,T];L^2).    
\end{equation*}
Consequently, once Proposition \ref{prop2.1} in Appendix \ref{subsection2.2} can be utilized, we obtain the desired conclusion immediately. 
However, Proposition \ref{prop2.1} is not applicable in this case, since it requires one \textit{a priori} assumption:
\begin{equation}\label{refine-3}
\chi^\sharp \rho_0^{\frac{1+\beta}{2}} D_{\bar\eta}^3 U \in L^2([0,T];L^2),    
\end{equation}
which can not be implied by \eqref{lowlow} directly, unless $\varepsilon_0>\frac{2\beta-1}{2\beta}$.

Therefore, to derive $\eqref{eequ3.66}_1$, we need to first obtain \eqref{refine-3}. 
To achieve this, we claim that
\begin{equation}\label{refine-Ux}
\chi^\sharp\rho_0^{\frac{1}{2}-\beta}D_{\bar\eta}U\in L^\infty([0,T];L^2).
\end{equation}
Indeed, based on $\eqref{3...136}_1$, using the same argument as in the proof of Lemma \ref{lemma-time-space}, we can obtain that, for all $\iota\in (-\frac{\beta}{2},1+\frac{\beta}{2})$ and $\sigma>0$,
\begin{equation}\label{cal-1}
\big|\chi^\sharp\rho_0^{\iota-\beta+\sigma}D_{\bar\eta} U\big|_2\leq C(\sigma,\iota,T)\big(1+\big|\chi^\sharp\rho_0^\iota(U,D_{\bar\eta} U,U_t)\big|_2\big) \qquad \text{for all $t\in[0,T]$}.
\end{equation}
Hence, choosing
\begin{equation*}
\iota=\frac{1}{2}-\sigma\qquad\text{with some fixed $\sigma<\min\big\{\frac{1+\beta}{2},\beta\big\}$},
\end{equation*}
then we derive from \eqref{cal-1} and Lemma \ref{hardy-inequality} that 
\begin{equation}\label{cal-2}
\begin{aligned}
\big|\chi^\sharp\rho_0^{\frac{1}{2}-\beta}D_{\bar\eta} U\big|_2&\leq C(T)\big(1+\big|\chi^\sharp\rho_0^{\frac{1}{2}-\sigma}(U,D_{\bar\eta} U,U_t)\big|_2\big)\\
&\leq C(T)\big(1+\big|\chi^\sharp\rho_0^{\frac{1}{2}-\sigma+2\beta}(U,D_{\bar\eta} U,D_{\bar\eta}^2 U,D_{\bar\eta}^3 U)\big|_2\big)\\
&\quad +C(T)\big|\chi^\sharp\rho_0^{\frac{1}{2}-\sigma+\beta}(U_t,D_{\bar\eta} U_t)\big|_2 \leq C(T)(1+\bar\cE_{\mathrm{ex}}(t,U)^\frac{1}{2}),
\end{aligned}
\end{equation}
which, along with \eqref{lowlow}, leads to \eqref{refine-Ux}.

We continue to show \eqref{refine-3}. Multiplying \eqref{new-lp} by $\chi^\sharp\rho_0^{\frac{1}{2}}$ and applying the $L^2$-norm to the resulting equality, we can deduce from $\rho_0^\beta\sim 1-r$, \eqref{lowlow}, \eqref{refine-Ux}, and the regularity of $\bar\eta$ in \eqref{given-bareta} that
\begin{equation*} 
\big|\chi^\sharp\rho_0^{\frac{1}{2}}D_{\bar\eta}^2 U\big|_2\leq C\big(\big|\chi^\sharp\rho_0^{\frac{1}{2}}(U_t,U,D_{\bar\eta}U)\big|_2+(\big|\chi^\sharp\rho_0^{\frac{1}{2}-\beta} D_{\bar\eta} U\big|_2+ |\bar\varrho|_\infty^{\gamma-1-\beta})|D_{\bar\eta}\bar\varrho^\beta|_\infty\big)\leq C(T),
\end{equation*}
that is, $\chi^\sharp\rho_0^{\frac{1}{2}}D_{\bar\eta}^2 U\in L^\infty([0,T];L^2)$.  Similarly, based on this and \eqref{refine-Ux}, applying $\chi^\sharp\rho_0^{\frac{1}{2}+\beta}D_{\bar\eta}$ to \eqref{new-lp}, we can further obtain
\begin{equation}\label{refine-2}
\chi^\sharp\rho_0^{\frac{1}{2}+\beta} D_{\bar\eta}^3 U\in L^\infty([0,T];L^2).    
\end{equation}
Next, multiplying \eqref{new-lp} by $\bar\varrho^{\beta}$ and then applying $\bar\varrho^{-\beta}\bar\eta_rD_{\bar\eta}$ to the resulting equality, we arrive at the following type of equation:
\begin{equation*} 
\widetilde{\cT}_{\mathrm{cross}}:=(D_{\bar\eta}^2 U)_r+ \big(\frac{1}{\beta}+1\big)\frac{(\rho_0^\beta)_r}{\rho_0^\beta} D_{\bar\eta}^2 U=\mathrm{I}_2+\mathrm{I}_3,
\end{equation*}
where
\begin{equation*}
\begin{aligned}
\mathrm{I}_2&:=\frac{1}{2\mu} \bar\eta_r D_{\bar\eta}U_t+\frac{1}{2\mu}\frac{(\bar\varrho^{\beta})_r}{\bar\varrho^\beta}U_t-m \bar\eta_r D_{\bar\eta}^2 \big(\frac{U}{\bar\eta}\big)-m\frac{(\bar\varrho^{\beta})_r}{\bar\varrho^\beta} D_{\bar\eta}\big(\frac{U}{\bar\eta}\big)\\
&\quad\,\, -\frac{1}{\beta} \frac{(D_{\bar\eta} \bar\varrho^\beta)_r}{\bar\varrho^\beta} D_{\bar\eta} U-(1+\beta)\frac{\bar\eta^m\bar\eta_r}{r^m}\big(\frac{r^m}{\bar\eta^m\bar\eta_r}\big)_rD_{\bar\eta}^2 U,\\
\mathrm{I}_3&:=\frac{A\gamma}{2\mu\beta} \bar\varrho^{\gamma-1-\beta}(D_{\bar\eta}\bar\varrho^\beta)_r+\frac{A\gamma(\gamma-1)}{2\mu\beta^2}   \bar\varrho^{\gamma-1-2\beta}(\bar\varrho^\beta)_rD_{\bar\eta}\bar\varrho^\beta.
\end{aligned}
\end{equation*}
After a direct calculation by using $\rho_0^\beta\sim 1-r$, $\beta\leq \gamma-1$, \eqref{lowlow}, Lemma \ref{hardy-inequality}, and the regularity of $\bar\eta$ in \eqref{given-bareta}, we can check that
\begin{equation*}
\chi^\sharp\rho_0^\frac{1+\beta}{2}(\mathrm{I}_2,\mathrm{I}_3)\in L^2([0,T];L^2)\implies \zeta^\sharp\rho_0^\frac{1+\beta}{2}\widetilde{\cT}_{\mathrm{cross}}\in L^2([0,T];L^2),
\end{equation*}
which, along with \eqref{refine-2} and the fact that $ \chi^\sharp\rho_0^{\frac{1+\beta}{2}}D_{\bar\eta}^2 U \in L^\infty([0,T];L^2)$ and Proposition \ref{prop2.1}, leads to $\zeta^\sharp\rho_0^\frac{1+\beta}{2}D_{\bar\eta}^3U\in L^2([0,T];L^2)$, so that
\begin{equation*}
\begin{aligned}
\big|\chi^\sharp\rho_0^\frac{1+\beta}{2}D_{\bar\eta}^3U\big|_2&\leq \big|(\zeta-\chi)\rho_0^\frac{1+\beta}{2}D_{\bar\eta}^3U\big|_2+\big|\zeta^\sharp\rho_0^\frac{1+\beta}{2}D_{\bar\eta}^3U\big|_2\\
&\leq \bar\cE_{\mathrm{in}}(t,U)^\frac{1}{2}+\bar\cD_{\mathrm{in}}(t,U)^\frac{1}{2}+\big|\zeta^\sharp\rho_0^\frac{1+\beta}{2}D_{\bar\eta}^3U\big|_2.
\end{aligned}
\end{equation*}
This, together with $\eqref{bbbb}_1$, yields the desired estimate \eqref{refine-3}, and thus yields $\eqref{eequ3.66}_1$.

\smallskip
\textbf{6. Regularity of $U$ given in \eqref{regu-class}.} Now, we can show that
\begin{equation}\label{classical-1}
\big(U,U_r,\frac{U}{r}\big)\in C([0,T];C(\bar I)),\qquad \big(U_{rr},(\frac{U}{r})_r,U_t\big)\in C((0,T];C(\bar I)).
\end{equation}

\smallskip
\textbf{6.1. Regularity of $U$ near the origin.} In this step, we aim to show that
\begin{equation}\label{classical-in}
\big(U,U_r,\frac{U}{r}\big)\in C([0,T];C(\bar I_\flat)),\qquad \big(U_{rr},(\frac{U}{r})_r\big)\in C((0,T];C(\bar I_\flat)),
\end{equation}
where $I_\flat:=[0,\frac{1}{2})$, and the regularity of $U_t$ can be derived similarly.

To obtain this, define the M-D representative of $U$ as
\begin{equation*}
\boldsymbol{U}(t,\boldsymbol{y})=U(t,r)\frac{\boldsymbol{y}}{r}.
\end{equation*}

First, using $\eqref{bbbb}_1$, Lemma \ref{lemma-initial}, and Lemma \ref{lemma-gaowei} in Appendix \ref{AppB}, we have
\begin{equation*}
\zeta\boldsymbol{U}\in L^\infty([0,T];H^3_0(\Omega))\cap L^2([0,T];H^4_0 (\Omega)),\qquad
\zeta\boldsymbol{U}_{t} \in L^2([0,T];H^2_0 (\Omega)).
\end{equation*}
Here, $\zeta$ is read as $\zeta(\boldsymbol{y})=\zeta(r)$ defined on $\Omega$. Then this, together with Lemma \ref{triple}, implies 
\begin{equation*}
\zeta\boldsymbol{U}\in C([0,T];H^3_0(\Omega)).
\end{equation*}
Using Lemma \ref{lemma-initial} again, we obtain from the above that
\begin{equation*}
r^\frac{m}{2}\big(U,U_r,\frac{U}{r},U_{rr},(\frac{U}{r})_r,U_{rrr},(\frac{U}{r})_{rr}\big)\in C([0,T];L^2(I_\flat)),
\end{equation*}
which, along with Lemmas \ref{sobolev-embedding}--\ref{hardy-inequality}, 
implies
\begin{equation*}
\big(U,U_r,\frac{U}{r},U_{rr},(\frac{U}{r})_r\big)\in C([0,T];L^2(I_\flat))\implies \big(U,U_r,\frac{U}{r}\big)\in C([0,T];C(\bar I_\flat)).
\end{equation*}

Next, similarly, it follows from $\eqref{bbbb}_1$ and Lemmas \ref{lemma-initial} and \ref{lemma-gaowei} that
\begin{equation*}
t\zeta\nabla_{\boldsymbol{y}}^2 \boldsymbol{U}\in L^\infty([0,T];H^2_0 (\Omega)),\qquad
(t\zeta\nabla_{\boldsymbol{y}}^2 \boldsymbol{U})_t\in L^2([0,T];L^2 (\Omega)),
\end{equation*}
which, along with Lemma \ref{triple}, leads to
\begin{equation*}
t\zeta\nabla_{\boldsymbol{y}}^2 \boldsymbol{U}\in C([0,T];W^{1,4} (\Omega)).
\end{equation*}
This, together with Lemma \ref{lemma-initial}, implies
\begin{equation}\label{444}
r^\frac{m}{4}\big(U_{rr},(\frac{U}{r})_r,U_{rrr},(\frac{U}{r})_{rr}\big)\in C((0,T];L^4(I_\flat)).
\end{equation}
On the other hand, since, for any function $f=f(r)$, 
\begin{equation*}
|f|_1\leq  |r^{-\frac{m}{4}}r^\frac{m}{4}f|_1\leq  |r^{-\frac{m}{4}}|_\frac{4}{3}|r^\frac{m}{4}f|_4\leq C|r^\frac{m}{4}f|_4,
\end{equation*}
we can derive from the above, \eqref{444}, and Lemma \ref{sobolev-embedding} that
\begin{equation*} 
\big(U_{rr},(\frac{U}{r})_r\big)\in C((0,T];W^{1,1}(I_\flat))\implies \big(U_{rr},(\frac{U}{r})_r\big)\in C((0,T];C(\bar I_\flat)),
\end{equation*}
which completes the proof of \eqref{classical-in}.

\smallskip
\textbf{6.2. Regularity of $U$ away from the origin.} Define $I_\sharp:=(\frac{1}{2},1)$. To obtain \eqref{classical-1}, it still remains to show
\begin{equation}\label{classical-ex}
U\in C([0,T];C^1(\bar I_\sharp)),\qquad (U_{rr},U_t)\in C((0,T];C(\bar I_\sharp)).
\end{equation}
 
First, it follows from \eqref{TETE} that
\begin{equation}\label{tete}
\rho_0^\frac{1}{2}(U,U_r,U_t)\in C([0,T];L^2(I_\sharp)).
\end{equation}

Next, note that, for \textit{a.e.} $t\in (0,T)$, due to the fact that 
\begin{equation*}
\big(\frac{1}{2}-\varepsilon_0)\beta>-\frac{\beta}{2},
\end{equation*}
we obtain from \eqref{bbbb} and Lemmas \ref{sobolev-embedding}--\ref{hardy-inequality} and \ref{lemma-gaowei} that
\begin{equation*}
\rho_0^{(\frac{3}{2}-\varepsilon_0)\beta}(U_{rr},U_{rrr})\in H^1(I_\sharp)\implies \rho_0^{(\frac{3}{2}-\varepsilon_0)\beta}(U_{rr},U_{rrr}) \in C(\bar I_\sharp),
\end{equation*}
which, by using \eqref{bbbb} again, leads to
\begin{equation*}
\begin{aligned}
&\zeta_\frac{1}{3}^\sharp\rho_0^\frac{(3-\varepsilon_0)\beta}{2} (U_{rr},U_{rrr}) \in L^\infty([0,T];L^2)\cap L^2([0,T];H^1_0),\\
&\zeta_\frac{1}{3}^\sharp\rho_0^\frac{(3-\varepsilon_0)\beta}{2}(U_{trr},U_{trrr}) \in L^2([0,T];H^{-1}).
\end{aligned}
\end{equation*}
Then it follows from the above and Lemma \ref{triple} that
\begin{equation}\label{tete2}
\rho_0^\frac{(3-\varepsilon_0)\beta}{2} (U_{rr},U_{rrr}) \in C([0,T];L^2(I_\sharp)).
\end{equation}
Hence, $\eqref{classical-ex}_1$ follows from \eqref{tete}--\eqref{tete2} and Lemmas \ref{sobolev-embedding}--\ref{hardy-inequality}:
\begin{equation*}
U \in C([0,T];W^{2,1}(I_\sharp))\implies U \in C([0,T];C^1(\bar I_\sharp)).
\end{equation*}

It remains to prove $\eqref{classical-ex}_2$. To this end,  we first obtain from \eqref{TETE} that
\begin{equation*}
\begin{aligned}
tU_{tt}\in L^\infty([0,T];L^2_{r^m\rho_0})\cap L^2([0,T];\cH_{r^m\rho_0}^{1}),\qquad  r^m\rho_0 (tU_{tt})_t  \in L^2([0,T];\cH_{r^m\rho_0}^{-1}),
\end{aligned}
\end{equation*}
which, along with \eqref{bbbb} and Lemmas \ref{sobolev-embedding} and \ref{Aubin}, gives
\begin{equation}\label{tete4}
t(U_{tr},U_{tt})\in C([0,T];L^2_{r^m\rho_0})\implies  \rho_0^\frac{1}{2}(U_{tr},U_{tt})\in C((0,T];L^2(I_\sharp)).
\end{equation}

Now, following an argument similar to the proof of $\sqrt{t}\rho_0^{(\frac{3}{2}-\varepsilon_0)\beta}D_{\bar\eta}^2U_t \in L^\infty([0,T];L^2(I_\sharp))$ 
(this proof is omitted in Step 5 above; see Steps 2.1 and 3 in the proof of Lemma \ref{c_3} below for details),
with the help of time-continuities \eqref{tete}--\eqref{tete4} and the regularities of $(\bar U,\bar\eta)$ in \eqref{given}--\eqref{given-bareta}, we can also arrive at
\begin{equation*} 
\rho_0^\frac{(3-\varepsilon_0)\beta}{2} D_{\bar\eta}^2U_t \in C((0,T];L^2(I_\sharp)),
\end{equation*}
and then additionally utilize the chain rules to obtain
\begin{equation}\label{tete5}
\rho_0^\frac{(3-\varepsilon_0)\beta}{2} U_{trr} \in C((0,T];L^2(I_\sharp)).
\end{equation}

Finally, recalling \eqref{xxxx-l}, we similarly define
\begin{equation*}
\bar\cT_{\mathrm{cross}}^*=\bar\cT_{\mathrm{cross}}^*(t,r):=U_{rrrr}+\big(\frac{1}{\beta}+2\big)\frac{(\rho_0^\beta)_r}{\rho_0^\beta}U_{rrr}.
\end{equation*}
Clearly, based on \eqref{tete2}--\eqref{tete5}, the chain rules and the regularity of $\bar\eta$ in \eqref{given-bareta}, we can follow a similar argument in Steps 2.2 and 3 of Lemma \ref{c_3} (this argument is actually contained in Step 5 above; however, we omit it)  and obtain the time-continuity for $\bar\cT^*_{\mathrm{cross}}$, \textit{i.e.},
\begin{equation*}
\rho_0^\frac{(3-\varepsilon_0)\beta}{2} \bar\cT^*_{\mathrm{cross}}\in C((0,T];L^2(I_\sharp)).
\end{equation*}
Then, by Proposition \ref{prop2.1}, for any $0<t,t_0\leq T$,
\begin{equation*}
\begin{aligned}
\big|\rho_0^\frac{(3-\varepsilon_0)\beta}{2} (U_{rrrr}(t)-U_{rrrr}(t_0))\big|_2&\leq C(T)\big|\rho_0^\frac{(3-\varepsilon_0)\beta}{2} (\bar\cT^*_{\mathrm{cross}}(t)-\bar\cT^*_{\mathrm{cross}}(t_0))\big|_2\\
&\quad +C(T)\big|\rho_0^\frac{(3-\varepsilon_0)\beta}{2} (U_{rrr}(t)-U_{rrr}(t_0))\big|_2.
\end{aligned}
\end{equation*}
Taking the limit as $t\to t_0$, together with \eqref{tete2}, yields
\begin{equation}\label{tete6}
\rho_0^\frac{(3-\varepsilon_0)\beta}{2} U_{rrrr} \in C((0,T];L^2(I_\sharp)).    
\end{equation}

Therefore, $\eqref{classical-ex}_2$ follows from \eqref{tete2}, \eqref{tete4}--\eqref{tete6}, and Lemmas \ref{sobolev-embedding}--\ref{hardy-inequality}:
\begin{equation*}
(U_{rr},U_t)\in C((0,T];W^{1,1}(I_\sharp))\implies (U_{rr},U_t)\in C((0,T];C(\bar I_\sharp)).
\end{equation*}

This completes the proof of \eqref{classical-ex}, and hence the proof of \eqref{classical-1}.

\smallskip
\textbf{7. Derivation of the boundary condition.} The Neumann boundary condition of $U$ can be proved by following essentially the same argument used in  Remark \ref{rmk1.4} of  \S \ref{Section-maintheorem}. First, since $\eqref{lp}_1$ holds pointwise in $(0,T]\times (0,1)$, we can divide $\eqref{lp}_1$ by $\bar\eta^m\bar\eta_r$ to obtain
\begin{equation}\label{lp-pre}
\bar\varrho U_t +A D_{\bar\eta}(\bar\varrho^{\gamma})=2\mu  \bar\varrho D_{\bar\eta}^2 U +2\mu (D_{\bar\eta}\bar\varrho)(D_{\bar\eta} U)+2\mu m \bar\varrho D_{\bar\eta}(\frac{U}{\bar\eta}),\qquad \bar\varrho=\frac{r^m\rho_0}{\bar\eta^m\bar\eta_r}.
\end{equation}
Then multiplying the above by $\bar\eta_r^2\rho_0^\beta \bar\varrho^{-1}$ gives
\begin{equation}\label{lp-pre;}
\begin{aligned}
\frac{2\mu}{\beta} (\rho_0^\beta)_rU_r&=\bar\eta_r^2\rho_0^{\beta}U_t+\frac{A\gamma}{\beta}\bar\eta_r \bar\varrho^{\gamma-1} \Big((\rho_0^\beta)_r-\beta \rho_0^\beta\big(m \frac{r}{\bar\eta}(\frac{\bar\eta}{r})_r + \frac{\bar\eta_{rr}}{\bar\eta_r} \big)\Big)\\
&\quad -2m\mu \bar\eta_r \rho_0^\beta \big((\frac{U}{r})_r\frac{r}{\bar\eta}-\frac{U}{r}\frac{r^2}{\bar\eta^2}(\frac{\bar\eta}{r})_r\big) -2\mu \bar\eta_r \rho_0^\beta \big(\frac{U_{rr}}{\bar\eta_r}-\frac{\bar\eta_{rr}U_r}{\bar\eta_r^2}\big) \\
&\quad +2\mu\rho_0^\beta\big(m \frac{r}{\bar\eta}(\frac{\bar\eta}{r})_r+ \frac{\bar\eta_{rr}}{\bar\eta_r} \big)U_r.
\end{aligned}
\end{equation}
Due to the facts that
\begin{equation}\label{the facts}
\begin{aligned}
&\big(U_r,\frac{U}{r},U_{rr},(\frac{U}{r})_r,\bar\eta_r,\frac{\bar\eta}{r},\bar\eta_{rr},(\frac{\bar\eta}{r})_r\big)\in C((0,T];C(\bar I)) ,\\
&\big(\bar\eta_r,\frac{\bar\eta}{r}\big)\in \big[\frac{1}{2},\frac{3}{2}\big],\qquad  \rho_0^\beta\sim 1-r,\qquad (\rho_0, (\rho_0^\beta)_r) \in C(\bar I),
\end{aligned}
\end{equation}
we see that the right-hand side of \eqref{lp-pre;} belongs to  $C((0,T]\times \bar I)$ and vanishes at $\{r=1\}$. As a consequence, taking the limit as $r\to 1$ in \eqref{lp-pre;}, we obtain from $(\rho_0^\beta)_r|_{r=1}\neq 0$ that
\begin{equation*}
(\rho_0^\beta)_rU_r|_{r=1}=\lim_{r\to 1}(\rho_0^\beta)_rU_r=0\implies U_r|_{r=1}=0.
\end{equation*}

Finally, $U|_{r=0}=0$ follows directly from $\frac{U}{r}\in C([0,T];C(\bar I))$.  

\smallskip
\textbf{8. Equation $\eqref{lp}_1$ holds pointwise in $(0, T]\times \bar I$.} By Definition \ref{fed-cl}, it remains to show that $U(t,r)$ satisfies equation $\eqref{lp}_1$ pointwise in $(0, T]\times \bar I$. However, to make the method in this step applicable to the existence and uniqueness of the nonlinear problem \eqref{eq:VFBP-La-eta}, we consider a more general case here. More precisely, we show that $\eqref{lp}_1\times (\bar\eta^m\bar\eta_r)^{-1}$, that is, \eqref{lp-pre}, holds pointwise in $(0, T]\times \bar I$. Furthermore, based on the structure of \eqref{lp-pre}, it suffices to show that the ``most singular'' term
\begin{equation*} 
\mathbb{S}=\mathbb{S}(t,r):=2\mu (D_{\bar\eta}\bar\varrho)(D_{\bar\eta} U)
\end{equation*}
holds pointwise in $(0,T]\times \bar I$.

A direct calculation gives 
\begin{equation}\label{most-singular1'}
\begin{aligned}
\SS&= - 2\mu \big(m\frac{r}{\bar\eta}(\frac{\bar\eta}{r})_r+\frac{\bar\eta_{rr}}{\bar\eta_r}\big)\frac{r^m\rho_0}{\bar\eta^m \bar\eta_r^3} U_r+\frac{2\mu }{\beta}\frac{r^m (\rho_0^\beta)_r}{\bar\eta^m \bar\eta_r^3} \underline{\rho_0^{1-\beta} U_r}_{:=\SS_*(t,r)}.
\end{aligned}
\end{equation}
From \eqref{the facts}, it follows that all the terms in \eqref{most-singular1'} belong to  $C((0,T]\times \bar I)$, hence holding pointwise in $(0,T]\times \bar I$, except for the term $\SS_*$. 

To prove that $\SS_*$ holds pointwise in $(0,T]\times \bar I$, we only need to check
\begin{equation}\label{claim-ss'}
\SS_*|_{r=1}<\infty \qquad\text{on  $[\sigma,T]$ for any $\sigma\in (0,T)$}. 
\end{equation}
Indeed, since $\rho_0^\beta \sim 1-r$ and $U_r|_{r=1}=0$, letting $r\to 1$, we obtain from \eqref{the facts} that, for all $(t,r)\in [\sigma,T]\times I$,
\begin{equation*}
|\SS_*(t,r)|\leq C(T)|(1-r)^{\frac{1}{\beta}-1}U_r|\leq C(T)(1-r)^\frac{1}{\beta} |U_{rr}|_\infty \leq C(\sigma,T)(1-r)^\frac{1}{\beta} \to 0.
\end{equation*}
Hence, we derive \eqref{claim-ss'}.

This completes the proof of Lemma \ref{existence-linearize}. 

\subsection{Uniform estimates for the linearized problem}\label{subsection9.2}
Based on  Lemma \ref{existence-linearize}, we can establish the uniform estimates 
for the classical solution $U$ of problem \eqref{lp}.

In \S\ref{subsection9.2}--\S\ref{subsection9.3}, $C\in (1,\infty)$ denotes a generic constant depending only on $(n,\mu,A,\gamma,\beta,\varepsilon_0)$, and $C(l_1,\cdots\!,l_k)\in (1,\infty)$ a generic  constant depending on $C$
and the additional  parameters $(l_1,\cdots\!,l_k)$, which may be different at each occurrence. Moreover, for simplicity, we denote by $\mathfrak{p}(\cdot)$ a generic polynomial function, taking the form: 
\begin{equation*}
\mathfrak{p}(s)=\sum_{j=1}^k s^j \qquad\text{with some $k\in \NN^*$}.
\end{equation*}

Now, let $(\rho_0,u_0)$ be given initial data satisfying the hypothesis of Lemma \ref{existence-linearize}.
Assume that 
there exists a constant $c_0>1$ such that 
\begin{equation}\label{38}
1+\cK_1+\cK_1^{-1}+\cK_2+\|\rho_0^\beta\|_{H^3(\Omega)}+\cE(0,U)\leq c_0.
\end{equation}
Then fix $T>0$, and assume that there exist some constants $T^*\in (0,T]$ and  $c_1$ such that $1<c_0\leq c_1$ and, for all $t\in [0,T^*]$, 
\begin{equation}\label{39}
\bar\cE(t,\bar U)+t\bar\cD(t,\bar U)+\int_0^{t}\bar\cD(s,\bar U)\,\mathrm{d}s\leq c_1,\qquad
|\bar\eta_r(t)-1|_\infty+\Big|\frac{\bar\eta(t)}{r}-1\Big|_\infty\leq \frac{1}{2}.
\end{equation}
Here, $(c_1,T^*)$ will be determined later, which depend only on $(c_0,\varepsilon_0,\beta,\mu,n,A,\gamma,T)$.

\subsubsection{Some basic estimates}\label{Remark-useful-bounds}

First, we have the following estimates associated with $\bar U$.
\begin{Lemma}\label{lemma-useful1}
For any $t\in[0,T^*]$,
\begin{equation*}
\Big|\big(\bar U,D_{\bar\eta} \bar U,\frac{\bar U}{\bar\eta}\big)\Big|_\infty
\leq C\mathfrak{p}(c_1),\qquad
\Big|\big(D_{\bar\eta}^2 \bar U,D_{\bar\eta}(\frac{\bar U}{\bar\eta}),\bar U_t\big)\Big|_\infty
\leq C(\mathfrak{p}(c_1)+\mathfrak{p}(c_0)\bar\cD(t,\bar U)^\frac{1}{2}),
\end{equation*}
and, for any $a\in (0,1)$,
\begin{equation*}
\begin{aligned}
&\Big|\chi_a r^\frac{m}{4}\big(D_{\bar\eta}\bar U,\frac{\bar U}{\bar\eta}\big)\Big|_{4}\leq  C(a)\mathfrak{p}(c_1),\qquad
\Big|\chi_a r^\frac{m}{4}\big(D_{\bar\eta}\bar U_{t},\frac{\bar U_t}{\bar\eta}\big)\Big|_{4} \leq  C(a)(\mathfrak{p}(c_1)+\bar\cD(t,\bar U)^\frac{1}{2}),\\
&\big|\chi^\sharp_a\rho_0^{\beta}(D_{\bar\eta}^2\bar U,\bar U_t)\big|_{\infty} \leq  C(a) \mathfrak{p}(c_1),\qquad
\big|\chi_a^\sharp\rho_0^\beta D_{\bar\eta}\bar U_{t}\big|_{\infty} \leq  C(a)(\mathfrak{p}(c_1)+\mathfrak{p}(c_0)\bar\cD(t,\bar U)^\frac{1}{2}).
\end{aligned}
\end{equation*}
\end{Lemma}
\begin{proof}
The proof can be directly derived by \eqref{39} and Lemmas \ref{sobolev-embedding}--\ref{hardy-inequality}. For example, it follows that, for all $t\in [0,T^*]$ and $a\in (0,1)$,
\begin{align*}
& \  \begin{aligned}
\Big|\big(D_{\bar\eta}^2 \bar U,D_{\bar\eta}(\frac{\bar U}{\bar\eta})\big)\Big|_\infty&\leq C\Big|\big(D_{\bar\eta}^2 \bar U,D_{\bar\eta}^3 \bar U,D_{\bar\eta}(\frac{\bar U}{\bar\eta}),D_{\bar\eta}^2(\frac{\bar U}{\bar\eta})\big)\Big|_1\\
&\leq  C \sum_{j=2}^4\Big(\big|\chi r^{\frac{m}{2}}D_{\bar\eta}^j\bar U\big|_{2} +\Big|\chi r^\frac{m}{2}D_{\bar\eta}^{j-1}(\frac{\bar U}{\bar\eta}) \Big|_2+\big|\chi^\sharp\rho_0^{(\frac{3}{2}-\varepsilon_0)\beta}D_{\bar\eta}^j\bar U\big|_{2}\Big) \\
&\leq C(\mathfrak{p}(c_1)+\mathfrak{p}(c_0)\bar\cD(t,\bar U)^\frac{1}{2}),
\end{aligned}\\
&\begin{aligned}
\Big|\chi_a r^\frac{m}{4}\big(D_{\bar\eta}\bar U_{t},\frac{\bar U_t}{\bar\eta}\big)\Big|_{4} &\leq C(a) \Big|\chi_a r^\frac{m+3}{4}\big(D_{\bar\eta}\bar U_{t},D_{\bar\eta}^2\bar U_{t},\frac{\bar U_t}{\bar\eta},D_{\bar\eta}(\frac{\bar U_t}{\bar\eta})\big)\Big|_{2}\\
&\leq C(a)(\mathfrak{p}(c_1)+\bar\cD(t,\bar U)^\frac{1}{2}),
\end{aligned}\\
&\qquad \ \ \big|\chi^\sharp_a\rho_0^\frac{\beta}{2} D_{\bar\eta}\bar U_{t}\big|_{2} \leq  C(a) \big|\chi^\sharp \rho_0^\frac{3\beta}{2}(D_{\bar\eta}\bar U_{t},D_{\bar\eta}^2\bar U_{t})\big|_{2} \leq C(a)(\mathfrak{p}(c_1)+\mathfrak{p}(c_0)\bar\cD(t,\bar U)^\frac{1}{2}).
\end{align*}

The rest of this lemma can be proved analogously; we omit the details here for brevity.
\end{proof}
 
Next, to further simplify the calculations, we define the quantities: 
\begin{equation}\label{bar-Lambda}
\bar\Lambda:=D_{\bar\eta} (\bar\varrho^\beta)=D_{\bar\eta} (\rho_0^\beta \bar{\mathscr{J}}^{-\beta}),\qquad \bar{\mathscr{J}}:=\frac{\bar\eta^m\bar\eta_r}{r^m}.
\end{equation}
Then we obtain some useful estimates for $(\bar\Lambda,\bar{\mathscr{J}})$.

\begin{Lemma}\label{lemma-Lambda-guji}
For any $0\leq t\leq T_1=\min\{T^*, \mathfrak{p}(c_1)^{-1}\}$, $a\in (0,1)$, and $\sigma>0$,
\begin{equation*}
\begin{aligned}
\Big|\zeta_a r^\frac{m}{4}\big(D_{\bar\eta}\bar\Lambda,\frac{\bar\Lambda}{\bar\eta}\big)\Big|_4+\Big|\zeta_a r^\frac{m}{2}\big(D_{\bar\eta}\bar\Lambda,\frac{\bar\Lambda}{\bar\eta},D_{\bar\eta}^2\bar\Lambda,D_{\bar\eta}(\frac{\bar\Lambda}{\bar\eta})\big)\Big|_2\leq C(a)\mathfrak{p}(c_0),\\
|\bar\Lambda|_\infty\leq C\mathfrak{p}(c_0),\qquad |\chi^\sharp_a D_{\bar\eta}\bar{\mathscr{J}}|_\infty
+|\chi^\sharp_a D_{\bar\eta}\bar\Lambda|_\infty+\big|\chi^\sharp_a \rho_0^{(\frac{1}{2}-\varepsilon_0)\beta} D_{\bar\eta}^2\bar\Lambda\big|_2\leq C(a)\mathfrak{p}(c_0),\\[5pt] 
\big|\chi^\sharp_a \rho_0^{-\frac{\beta}{2}+\sigma}\bar\Lambda\big|_2\leq C(a,\sigma)\mathfrak{p}(c_0),\qquad|\zeta_a r^\frac{m}{4}\bar\Lambda_t|_4+|\chi_a^\sharp\bar\Lambda_t|_\infty\leq C(a)\mathfrak{p}(c_1).
\end{aligned}
\end{equation*}
\end{Lemma}
\begin{proof}
By \eqref{given-flow} and \eqref{38}--\eqref{39}, we see that, for all $(t,r)\in [0,T^*]\times \bar I$,
\begin{equation}\label{jt}
\bar{\mathscr{J}}_t =\bar{\mathscr{J}}\big(D_{\bar\eta}\bar U+\frac{m\bar U}{\bar\eta}\big) ,\qquad C^{-1}\leq \bar{\mathscr{J}}(t,r)\leq C.
\end{equation}

Then 
\begin{align*}
&\begin{aligned}
(D_{\bar\eta}\bar{\mathscr{J}})_t&=D_{\bar\eta}\bar{\mathscr{J}}\frac{m\bar U}{\bar\eta}+\bar{\mathscr{J}}D_{\bar\eta}\big(D_{\bar\eta}\bar U+\frac{m\bar U}{\bar\eta}\big),\\
(D_{\bar\eta}^2\bar{\mathscr{J}})_t&=D_{\bar\eta}^2\bar{\mathscr{J}}\big(\frac{m\bar U}{\bar\eta}-D_{\bar\eta}\bar U\big)+D_{\bar\eta}\bar{\mathscr{J}}D_{\bar\eta}\big(D_{\bar\eta}\bar U+\frac{2m\bar U}{\bar\eta}\big)+\bar{\mathscr{J}}D_{\bar\eta}^2\big(D_{\bar\eta}\bar U+\frac{m\bar U}{\bar\eta}\big),
\end{aligned}\\
&\begin{aligned}
(D_{\bar\eta}^3\bar{\mathscr{J}})_t&=D_{\bar\eta}^3\bar{\mathscr{J}}\big(\frac{m\bar U}{\bar\eta}-2D_{\bar\eta}\bar U\big)+D_{\bar\eta}^2\bar{\mathscr{J}} D_{\bar\eta}\big(\frac{3m\bar U}{\bar\eta}\big)\\
&\quad +D_{\bar\eta}\bar{\mathscr{J}}D_{\bar\eta}^2\big(2D_{\bar\eta}\bar U+\frac{3m\bar U}{\bar\eta}\big)+\bar{\mathscr{J}}D_{\bar\eta}^3\big(D_{\bar\eta}\bar U+\frac{m\bar U}{\bar\eta}\big).
\end{aligned}
\end{align*}
Hence, these, together with \eqref{jt} and Lemma \ref{lemma-useful1}, yield
\begin{align}
&\begin{aligned}
|D_{\bar\eta}\bar{\mathscr{J}}|&\leq Ce^{Ct\mathfrak{p}(c_1)}\int_0^t\Big|\big(D_{\bar\eta}^2\bar U,D_{\bar\eta}(\frac{\bar U}{\bar\eta})\big)\Big|\,\mathrm{d}s,\notag\\
|D_{\bar\eta}^2\bar{\mathscr{J}}|&\leq Ce^{2Ct\mathfrak{p}(c_1)}\Big(\int_0^t \Big|\big(D_{\bar\eta}^2\bar U,D_{\bar\eta}(\frac{\bar U}{\bar\eta})\big)\Big|\,\mathrm{d}s\Big)^2+e^{Ct\mathfrak{p}(c_1)}\int_0^t\Big|\big(D_{\bar\eta}^3\bar U,D_{\bar\eta}^2(\frac{\bar U}{\bar\eta})\big)\Big|\,\mathrm{d}s,
\end{aligned}\\
&\begin{aligned}[t]\label{1189}
|D_{\bar\eta}^3\bar{\mathscr{J}}|&\leq Ce^{3Ct\mathfrak{p}(c_1)}\Big(\int_0^t \Big|\big(D_{\bar\eta}^2\bar U,D_{\bar\eta}(\frac{\bar U}{\bar\eta})\big)\Big|\,\mathrm{d}s\Big)^3\\
&\quad +Ce^{2Ct\mathfrak{p}(c_1)}\Big(\int_0^t \Big|\big(D_{\bar\eta}^2\bar U,D_{\bar\eta}(\frac{\bar U}{\bar\eta})\big)\Big|\,\mathrm{d}s\Big)\Big(\int_0^t\Big|\big(D_{\bar\eta}^3\bar U,D_{\bar\eta}^2(\frac{\bar U}{\bar\eta})\big)\Big|\,\mathrm{d}s\Big)\\
&\quad +Ce^{Ct\mathfrak{p}(c_1)}\int_0^t \Big|\big(D_{\bar\eta}^4\bar U,D_{\bar\eta}^3(\frac{\bar U}{\bar\eta})\big)\Big|\,\mathrm{d}s.
\end{aligned}
\end{align}

On the other hand, by the chain rules, we have
\begin{equation}\label{1190}
\begin{aligned}
D_{\bar\eta}^k\bar\Lambda &=\sum_{j=0}^{k+1} C_{k,j} (D_{\bar\eta}^{k+1-j} \rho_0^\beta) D_{\bar\eta}^{j}(\bar{\mathscr{J}}^{-\beta})\qquad \text{for $k=0,1,2$},\\
\bar\Lambda_t&=- \beta \bar\varrho^\beta D_{\bar\eta}\big(D_{\bar\eta}\bar U+\frac{m\bar U}{\bar\eta}\big)-\bar\Lambda\big((\beta+1)D_{\bar\eta}\bar U+\beta\frac{m\bar U}{\bar\eta}\big), 
\end{aligned}
\end{equation}
where $C_{k,j}$ are some constants depending only on $(k,j)$.

Therefore, for all $0\leq t\leq T_1:=\min\{T^*,\mathfrak{p}(c_1)^{-1}\}$, combining \eqref{1189}--\eqref{1190}, together with \eqref{38}--\eqref{39}, \eqref{jt}, Lemmas \ref{lemma-useful1} and \ref{sobolev-embedding}--\ref{hardy-inequality}, and the H\"older and Minkowski inequalities, we can recursively obtain the desired estimates of this lemma. The calculation is tedious, but rather direct.
\end{proof}

\subsubsection{Uniform estimates of $U$}\label{subsub-11.2.2}
The proof will be divided into the following several lemmas.

\begin{Lemma}\label{c_0-c_1}
There exists $T_2<T_1$, depending only on $(c_0,c_1,n,\mu,A,\gamma,\beta,\varepsilon_0)$, such that, for all $0\leq t\leq T_2$,
\begin{equation*}
\Big|(r^m\rho_0)^\frac{1}{2}\big(U,D_{\bar\eta} U,\frac{U}{\bar\eta},U_t\big)(t)\Big|_2 +\int_0^t\Big|(r^m\rho_0)^\frac{1}{2}\big(D_{\bar\eta} U_t,\frac{U_t}{\bar\eta}\big)\Big|_2^2\,\mathrm{d}s\leq C\mathfrak{p}(c_0).
\end{equation*}
\end{Lemma}
\begin{proof}
We divide the proof into three steps.

\smallskip 
\textbf{1.} We first prove that there exists $T_2<T_1$, depending only on $(c_0,c_1,n,\mu,A,\gamma,\beta,\varepsilon_0)$, such that, for any $t\in [0,T_2]$, 
\begin{equation}\label{10.99}
\Big|\big(D_{\bar\eta} \bar U,\frac{\bar U}{\bar\eta}\big)\Big|_\infty\leq C\mathfrak{p}(c_0).
\end{equation}
Indeed, by equation $\bar\eta_t= \bar U$ with $\bar\eta(0,r)=r$, and the following identities
\begin{equation*}
\bar U_{rr}=\bar\eta_r^2 D_{\bar\eta}^2\bar U+\bar\eta_{rr}D_{\bar\eta} \bar U,\qquad \bar U_{rrr}=\bar\eta_r^3 D_{\bar\eta}^3\bar U+3\bar\eta_{r}\bar\eta_{rr}D_{\bar\eta}^2 \bar U+\bar\eta_{rrr}D_{\bar\eta} \bar U,
\end{equation*}
we can obtain from \eqref{39} that, for any $t\in [0,T^*]$, 
\begin{equation*}
\big|\zeta r^\frac{m}{2}(\bar U_r,\bar U_{rr},\bar U_{rrr})\big|_2+\big|\chi^\sharp \rho_0^{(\frac{3}{2}-\varepsilon_0)\beta}(\bar U_r,\bar U_{rr},\bar U_{rrr})\big|_2\leq C\mathfrak{p}(c_1).
\end{equation*}
This, combined with Lemma \ref{hardy-inequality}, $\varepsilon_0<\frac{1}{2}$, and $\rho_0^\beta\sim 1-r$, implies that, for any $t\in [0,T^*]$,
\begin{equation*}
|\chi r^{1-\varepsilon_0}\bar U_{rr}|_\infty+|\chi^\sharp (1-r)^{1-\varepsilon_0}\bar U_{rr}|_\infty\leq C\mathfrak{p}(c_1).
\end{equation*}
Thus, it follows from the above and formula $\bar U_r(r_2)-\bar U_r(r_1)=\int_{r_1}^{r_2} \bar U_{rr}\mathrm{d}r$ that
\begin{equation}\label{holder-con}
\sup_{\substack{r_1,r_2\in \bar I\\r_1\neq r_2}}\frac{|\bar U_r(t,r_1)-\bar U_r(t,r_2)|+|(u_0)_r(r_1)-(u_0)_r(r_2)|}{|r_1-r_2|^{\varepsilon_0}}\leq C\mathfrak{p}(c_1).
\end{equation}
On the other hand, by \eqref{39}, we also have
\begin{equation}\label{pingjun}
\big|(r^m\rho_0)^\frac{1}{2}\big(\bar U_r(t)-(u_0)_r\big)\big|_2\leq \int_0^t \big|(r^m\rho_0)^\frac{1}{2}\bar U_{tr}\big|_2\,\mathrm{d}s\leq Ct\mathfrak{p}(c_1).
\end{equation}

Now, fix any $r\in \bar I$ and let $\sigma\in (0,\frac{1}{4})$. If $r\leq1-\sigma$, consider the interval $[r,r+\sigma]$; if $r>1-\sigma$, consider $[r-\sigma,r]$. Since the proof of the second case is the same as the first one, without loss of generality, we assume $r\leq1-\sigma$ with $\sigma\in (0,\frac{1}{4})$ and consider the interval $[r,r+\sigma]$ in the following. Thus, it follows from \eqref{holder-con} that, for all $\hat r\in[r,r+\sigma]$,
\begin{equation*}
\left|\big(\bar U_r(t,r)-(u_0)_r(r)\big)-\big(\bar U_r(t,\hat r)-(u_0)_r(\hat r)\big)\right|\leq C\mathfrak p(c_1)\sigma^{\varepsilon_0},
\end{equation*}
which implies that 
\begin{equation*}
\begin{aligned}
&\Big|\int_r^{r+\sigma} \hat r^m\rho_0(\hat r)\Big(\big(\bar U_r(t,r)-(u_0)_r(r)\big)-\big(\bar U_r(t,\hat r)-(u_0)_r(\hat r)\big)\Big)\mathrm d\hat r\Big|\\
&=\Big|\Big(\int_r^{r+\sigma} \hat r^m\rho_0(\hat r)\mathrm d\hat r\Big)\big(\bar U_r(t,r)-(u_0)_r(r)\big)-\int_r^{r+\sigma} \hat r^m\rho_0(\hat r)\big(\bar U_r(t,\hat r)-(u_0)_r(\hat r)\big)\mathrm d\hat r\Big|\\
&\leq C\mathfrak p(c_1)\sigma^{\varepsilon_0}\int_r^{r+\sigma}\hat r^m\rho_0(\hat r)\,\mathrm d\hat r.
\end{aligned}
\end{equation*}
Dividing the above by $\int_r^{r+\sigma}\hat r^m\rho_0(\hat r)\,\mathrm d\hat r$, we obtain from \eqref{pingjun} and $\rho_0^\beta\sim 1-r$ that
\begin{equation*}
\begin{aligned}
\big|\bar U_r(t,r)-(u_0)_r(r)\big|&\leq 
\Big|\frac{\int_r^{r+\sigma} \hat r^m\rho_0(\hat r) \big[\bar U_r(t,\hat r)-(u_0)_r(\hat r)\big]\,\mathrm d\hat r} {\int_r^{r+\sigma} \hat r^m\rho_0(\hat r)\,\mathrm d\hat r}\Big|+ C\mathfrak p(c_1)\sigma^{\varepsilon_0}\\
&\leq 
 \frac{Ct\mathfrak{p}(c_1)}{\big(\int_r^{r+\sigma} \hat r^m\rho_0(\hat r)\,\mathrm d\hat r\big)^\frac{1}{2}} + C\mathfrak p(c_1)\sigma^{\varepsilon_0}\leq C\mathfrak p(c_1)\big(t \sigma^{-\frac{1}{2}\max\{n,1+\frac{1}{\beta}\}}+ \sigma^{\varepsilon_0}\big).
\end{aligned}
\end{equation*}
Therefore, we can take  
\begin{equation*}
\sigma=\min\big\{t^{1/(\varepsilon_0+\frac12\max\{n,\,1+\frac{1}{\beta}\})},\frac{1}{8}\big\},
\end{equation*}
in the above inequality to derive
\begin{equation}\label{urur}
|\bar U_r(t,r)|\leq |(u_0)_r(r)|+ C\mathfrak p(c_1) t^{\varepsilon_0/(\varepsilon_0+\frac12\max\{n,\,1+\frac{1}{\beta}\})}\leq C\mathfrak{p}(c_0)
\end{equation}
for any $r\in \bar I$ and $t\in [0,T_2]$, where $T_2<T_1$ is chosen to satisfy
\begin{equation*}
\mathfrak p(c_1)  T_2^{\varepsilon_0/(\varepsilon_0+\frac12\max\{n,\,1+\frac{1}{\beta}\})}\leq \mathfrak{p}(c_0),
\end{equation*}
depending only on $(c_0,c_1,n,\mu,A,\gamma,\beta,\varepsilon_0)$. Since $\bar U(t,0)=0$, \eqref{urur} also controls $|\frac{\bar U}{r}|_\infty$. Finally, using the upper bound of $(\bar\eta_r,\frac{\bar\eta}{r})$ in \eqref{39}, we obtain the assertion \eqref{10.99}.

\smallskip
\textbf{2.} Multiplying  $\eqref{lp}_1$ by $U$ and integrating the resulting equation over $I$, 
then we obtain from \eqref{38}--\eqref{39} and the Young inequality that
\begin{equation*} 
\begin{aligned}
&\,\frac{1}{2}\frac{\mathrm{d}}{\dt}\big|(r^m\rho_0)^\frac{1}{2}U\big|_2^2+2\mu\big|(r^m\rho_0)^\frac{1}{2} D_{\bar\eta} U\big|_2^2+2\mu m\Big|(r^m\rho_0)^\frac{1}{2}\frac{U}{\bar\eta}\Big|_2^2\\
&= A \int_0^1 \bar\eta^m\bar\eta_r\bar\varrho^\gamma\big(D_{\bar\eta}U+\frac{mU}{\bar\eta}\big)\,\mathrm{d}r 
\leq  C\mathfrak{p}(c_0) +\frac{\mu}{8}\Big|(r^m\rho_0)^\frac{1}{2}\big(D_{\bar\eta} U,\frac{U}{\bar\eta}\big)\Big|_2^2.
\end{aligned}    
\end{equation*}
Integrating the above over $[0,t]$, along with \eqref{38}, implies that, for $t\in[0,T_2]$,
\begin{equation}\label{1}
\big|(r^m\rho_0)^\frac{1}{2}U(t)\big|_2^2+ \int_0^t\Big|(r^m\rho_0)^\frac{1}{2}\big(D_{\bar\eta} U,\frac{U}{\bar\eta}\big)\Big|_2^2\mathrm{d}s
\leq C(\bar\cE(0,U)+ \mathfrak{p}(c_0)t) \leq C\mathfrak{p}(c_0).
\end{equation}

\smallskip
\textbf{3.}
Multiplying  $\eqref{lp}_1$ by $U_t$ and integrating the resulting equation over $I$, we obtain from \eqref{38}--\eqref{39}, \eqref{10.99}, and the Young inequality that
\begin{equation}\label{DU-t}
\begin{aligned}
&\,\mu\frac{\mathrm{d}}{\dt}\Big(\big|(r^m\rho_0)^\frac{1}{2} D_{\bar\eta} U\big|_2^2+m\Big|(r^m\rho_0)^\frac{1}{2}\frac{U}{\bar\eta}\Big|_2^2\Big)+\big|(r^m\rho_0)^\frac{1}{2}U_t\big|_2^2\\
&=-2\mu\int_0^1 r^m\rho_0\big(D_{\bar\eta}\bar U|D_{\bar\eta} U|^2+m\frac{\bar U |U|^2}{\bar\eta^3}\big)\,\mathrm{d}r + A \int_0^1 \bar\eta^m\bar\eta_r\bar\varrho^\gamma\big(D_{\bar\eta}U_t+\frac{mU_t}{\bar\eta}\big)\,\mathrm{d}r\\
&\leq  C\Big|\big(D_{\bar\eta}\bar U,\frac{\bar U}{\bar\eta}\big)\Big|_\infty\Big|(r^m\rho_0)^\frac{1}{2}\big(D_{\bar\eta}U,\frac{U}{\bar\eta}\big)\Big|_2^2+C\mathfrak{p}(c_0) +\frac{\mu}{8}\Big|(r^m\rho_0)^\frac{1}{2}\big(D_{\bar\eta} U_t,\frac{U_t}{\bar\eta}\big)\Big|_2^2\\
&\leq  C\mathfrak{p}(c_0)\Big(\Big|(r^m\rho_0)^\frac{1}{2}\big(D_{\bar\eta} U,\frac{U}{\bar\eta}\big)\Big|_2^2+1\Big) +\frac{\mu}{8}\Big|(r^m\rho_0)^\frac{1}{2}\big(D_{\bar\eta} U_t,\frac{U_t}{\bar\eta}\big)\Big|_2^2.
\end{aligned}    
\end{equation}

Next, recalling $(\mathfrak{q}_1^{(1)},\mathfrak{q}_2^{(1)})$ in \eqref{q1-q2-1}, we derive from \eqref{38}--\eqref{39} and \eqref{10.99} that 
\begin{equation}\label{est-q1-q2-1}
\begin{aligned}
|(\mathfrak{q}_1^{(1)},\mathfrak{q}_2^{(1)})|_2&\leq C\Big|\big(D_{\bar\eta} \bar U,\frac{\bar U}{\bar\eta}\big)\Big|_\infty\Big(\big|(r^m\rho_0)^\frac{1}{2}\bar\varrho^{\gamma-1}\big|_2 +\Big|(r^m\rho_0)^\frac{1}{2}\big(D_{\bar\eta}U,\frac{U}{\bar\eta}\big)\Big|_2\Big)\\
&\leq C\mathfrak{p}(c_0)\Big(1+\Big|(r^m\rho_0)^\frac{1}{2}\big(D_{\bar\eta}U,\frac{U}{\bar\eta}\big)\Big|_2\Big).
\end{aligned}
\end{equation}
Then multiplying \eqref{33151} by $U_t$ and integrating the resulting equation over $I$, combined with \eqref{est-q1-q2-1} and the  Young inequality, imply 
\begin{equation}\label{O-2}
\begin{aligned}
&\,\frac{1}{2}\frac{\mathrm{d}}{\dt}\big|(r^m\rho_0)^\frac{1}{2}U_t\big|_2^2+2\mu\big|(r^m\rho_0)^\frac{1}{2} D_{\bar\eta} U_t\big|_2^2+2\mu m\Big|(r^m\rho_0)^\frac{1}{2}\frac{U_t}{\bar\eta}\Big|_2^2\\
&=\int_0^1 (r^m\rho_0)^\frac{1}{2}\Big(\mathfrak{q}_1^{(1)}\frac{m U_t}{\bar\eta}+ \mathfrak{q}_2^{(1)}D_{\bar\eta}U_t\Big)\,\mathrm{d}r \\
&\leq C\mathfrak{p}(c_0)\Big(1+\Big|(r^m\rho_0)^\frac{1}{2}\big(D_{\bar\eta}U,\frac{U}{\bar\eta}\big)\Big|_2^2\Big) +\frac{\mu}{8}\Big|(r^m\rho_0)^\frac{1}{2}\big(D_{\bar\eta}U_t,\frac{U_t}{\bar\eta}\big)\Big|_2^2,
\end{aligned}    
\end{equation}
which, along with \eqref{38}, \eqref{1}--\eqref{DU-t}, and the Gr\"onwall inequality, implies that, for all $0\leq t\leq  T_2$,
\begin{equation}\label{2}
\begin{aligned}
&\,\Big|(r^m\rho_0)^\frac{1}{2}\big(D_{\bar\eta} U,\frac{U}{\bar\eta}\big)(t)\Big|_2^2+\big|(r^m\rho_0)^\frac{1}{2}U_t(t)\big|_2^2\\
&+ \int_0^t\Big|(r^m\rho_0)^\frac{1}{2}\big(D_{\bar\eta} U_t,\frac{U_t}{\bar\eta}\big)\Big|_2^2\,\mathrm{d}s
\leq Ce^{C\mathfrak{p}(c_0)t}(\bar\cE(0,U)+ \mathfrak{p}(c_0)) \leq C\mathfrak{p}(c_0).
\end{aligned}
\end{equation}

This completes the proof.
\end{proof}

\begin{Lemma}\label{c_1}
For all $t\in[0,T_3]$ with $T_3:=\min\{T_2,\mathfrak{p}(c_1)^{-1}\}$,
\begin{equation*}
\Big|(r^m\rho_0)^\frac{1}{2}\big(D_{\bar\eta} U_t,\frac{U_t}{\bar\eta},\sqrt{t}U_{tt}\big)(t)\Big|_2 +\int_0^t \Big|(r^m\rho_0)^\frac{1}{2}\big(U_{tt},sD_{\bar\eta} U_{tt},s\frac{U_{tt}}{\bar\eta}\big)\Big|_2^2\,\mathrm{d}s\leq C\mathfrak{p}(c_0).
\end{equation*}
\end{Lemma}
\begin{proof}
We divide the proof into three steps.

\smallskip
\textbf{1.} First, it follows from \eqref{new-lp} and \eqref{bar-Lambda} that
\begin{equation}\label{new-lp*}
D_{\bar\eta}\big(D_{\bar\eta} U+ \frac{mU}{\bar\eta}\big)=\frac{1}{2\mu} U_t-\frac{1}{\beta}\frac{\bar\Lambda}{\bar\varrho^\beta} D_{\bar\eta} U+\frac{A\gamma}{2\mu \beta}\bar\varrho^{\gamma-1-\beta}\bar\Lambda.
\end{equation}
Then, due to the fact that $\rho_0^\beta\sim 1-r$, \eqref{38}, and Lemmas \ref{im-1} and \ref{lemma-Lambda-guji}--\ref{c_0-c_1}, we see that, for all $t\in [0,T_2]$, 
\begin{equation}\label{1197}
\begin{aligned}
&\,\Big|\zeta r^\frac{m}{2}\big(D_{\bar\eta}^2 U,D_{\bar\eta}(\frac{U}{\bar\eta})\big)\Big|_2\leq C\Big|\zeta r^\frac{m}{2} D_{\bar\eta}\big(D_{\bar\eta} U+ \frac{mU}{\bar\eta}\big)\Big|_2\\
&\leq C\big(|\zeta r^\frac{m}{2}U_t|_2+(|\zeta r^\frac{m}{2}D_{\bar\eta} U|_2+ |\zeta r^\frac{m}{2}\bar\varrho^{\gamma-1-\beta}|_2)|\bar\Lambda|_\infty\big)\leq C\mathfrak{p}(c_0),   
\end{aligned}
\end{equation}
which, along with Lemma \ref{hardy-inequality}, also leads to
\begin{equation}\label{11123}
\Big|\zeta r^\frac{m}{4}\big(D_{\bar\eta} U,\frac{U}{\bar\eta}\big)\Big|_4\leq\Big|r^\frac{m+3}{4}\big(\zeta D_{\bar\eta} U,\zeta \frac{U}{\bar\eta},\zeta_r D_{\bar\eta} U,\zeta_r \frac{U}{\bar\eta},\zeta D_{\bar\eta}^2 U,\zeta D_{\bar\eta}(\frac{U}{\bar\eta})\big)\Big|_2 \leq C\mathfrak{p}(c_0).   
\end{equation}

Next, using \eqref{1}, \eqref{2}, and the same argument as in the proof of Lemma \ref{lemma-time-space}, we obtain that, for all $t\in [0,T_2]$, 
\begin{equation}\label{xingxing}
\big|\chi^\sharp\rho_0^\frac{1-\beta}{2}D_{\bar\eta} U\big|_\infty \leq C\mathfrak{p}(c_0)\big(1+  \big|\chi^\sharp\rho_0^{\frac{1}{2} }(U,D_{\bar\eta} U,U_{t})\big|_2\big)\leq C\mathfrak{p}(c_0).
\end{equation}
Of course, based on \eqref{2}, we also derive from \eqref{est-q1-q2-1} that, for all $t\in [0,T_2]$, 
\begin{equation}\label{qq12}
|(\mathfrak{q}_1^{(1)},\mathfrak{q}_2^{(1)})|_2\leq C\mathfrak{p}(c_0).    
\end{equation}

Finally, recall $((\mathfrak{q}_1^{(1)})_t,(\mathfrak{q}_2^{(1)})_t)$ in \eqref{q1-q2-1-t}--\eqref{q1-q2-1-ttt}. It then follows from \eqref{38}--\eqref{39}, \eqref{11123}, Lemmas \ref{lemma-useful1} and \ref{c_0-c_1}, and the H\"older inequality that 
\begin{align}
& \ \begin{aligned}\label{qq12-t-1}
\big|\chi((\mathfrak{q}_1^{(1)})_t,(\mathfrak{q}_2^{(1)})_t)\big|_2&\leq \mathfrak{p}(c_0) \Big|\chi r^\frac{m}{4}\big(D_{\bar\eta}U,\frac{U}{\bar\eta}\big)\Big|_4\Big|\chi r^\frac{m}{4}\big(D_{\bar\eta}\bar U_t,\frac{\bar U_t}{\bar\eta}\big)\Big|_4\\
&\quad +\mathfrak{p}(c_0)\Big|\chi r^\frac{m}{2}\big(D_{\bar\eta}U_t,\frac{U_t}{\bar\eta}\big)\Big|_2\Big|\big(D_{\bar\eta}\bar U,\frac{\bar U}{\bar\eta}\big)\Big|_\infty\\
&\quad +\mathfrak{p}(c_0)\Big|\chi r^\frac{m}{2}\big(D_{\bar\eta}U,\frac{U}{\bar\eta}\big)\Big|_2\Big|\big(D_{\bar\eta}\bar U,\frac{\bar U}{\bar\eta}\big)\Big|_\infty^2 \\
&\leq  C\mathfrak{p}(c_0)\Big(\mathfrak{p}(c_1)+\bar\cD(t,\bar U)^\frac{1}{2}+\Big|(r^m\rho_0)^\frac{1}{2}\big(D_{\bar\eta}U_t,\frac{U_t}{\bar\eta}\big)\Big|_2\Big),
\end{aligned}\\
&\begin{aligned}\label{qq12-t-2}
\big|\chi^\sharp((\mathfrak{q}_1^{(1)})_t,(\mathfrak{q}_2^{(1)})_t)\big|_2&\leq C\big(\mathfrak{p}(c_0)+ |\chi^\sharp\rho_0^\frac{1}{2} U|_2(|\bar U_t|_\infty+|\bar U|_\infty^2)+\big|\chi^\sharp\rho_0^\frac{1}{2}D_{\bar\eta} U\big|_2|D_{\bar\eta} \bar U|_\infty^2\big)\\
&\quad+ C\big|\chi^\sharp\rho_0^\frac{1}{2}(U_t,D_{\bar\eta} U_t)\big|_2|(\bar U,D_{\bar\eta} \bar U)|_\infty\\
&\quad + C\big|\chi^\sharp\rho_0^\frac{\beta}{2}D_{\bar\eta}\bar U_t\big|_2\big|\chi^\sharp\rho_0^\frac{1-\beta}{2}D_{\bar\eta}U\big|_\infty\\
&\leq  C\mathfrak{p}(c_0)\big(\mathfrak{p}(c_1)+\bar\cD(t,\bar U)^\frac{1}{2}+\big|(r^m\rho_0)^\frac{1}{2}D_{\bar\eta} U_t\big|_2\big).
\end{aligned}
\end{align}

\smallskip
\textbf{2.} Now, multiplying \eqref{33151} by $U_{tt}$ and integrating over $I$, we have
\begin{equation}\label{I0*3}
\begin{aligned}
&\ \mu\frac{\mathrm{d}}{\mathrm{d}t}\Big(\underline{\big|(r^m\rho_0)^\frac{1}{2} D_{\bar\eta} U_t\big|_2^2+ m\Big|(r^m\rho_0)^\frac{1}{2}\frac{U_t}{\bar\eta}\Big|_2^2}_{:=\cF(t)}\Big)+\big|(r^m\rho_0)^\frac{1}{2}U_{tt}\big|_2^2\\
&=\underline{-2\mu\int_0^1 r^m\rho_0\big(D_{\bar\eta}\bar U|D_{\bar\eta} U_t|^2+m\frac{\bar U |U_t|^2}{\bar\eta^3}\big)\,\mathrm{d}r}\\
&\quad\,\, \underline{+\int_0^1 (r^m\rho_0)^\frac{1}{2}\Big(\big(\mathfrak{q}_1^{(1)}\frac{\bar U}{\bar\eta}-(\mathfrak{q}_1^{(1)})_t\big)\frac{m U_t}{\bar\eta}+\big(\mathfrak{q}_2^{(1)}D_{\bar\eta}\bar U-(\mathfrak{q}_2^{(1)})_t\big) D_{\bar\eta} U_t \Big)\,\mathrm{d}r}_{:=\mathrm{I}_4}\\
&\quad\,\, +\frac{\mathrm{d}}{\mathrm{d}t}\underline{\int_0^1 (r^m\rho_0)^\frac{1}{2}\big(\mathfrak{q}_1^{(1)}\frac{m U_t}{\bar\eta}+\mathfrak{q}_2^{(1)}D_{\bar\eta} U_t\big)\,\mathrm{d}r}_{:=\mathrm{I}_{**}(t)}.
\end{aligned}
\end{equation}
For $\mathrm{I}_4$, it follows from \eqref{qq12}--\eqref{qq12-t-2}, \eqref{10.99}, and the H\"older and Young inequalities that
\begin{equation}\label{I1*3}
\begin{aligned}
\mathrm{I}_4&\leq C\Big(\Big|\big(D_{\bar\eta}\bar U,\frac{\bar U}{\bar\eta}\big)\Big|_\infty\cF(t)^\frac{1}{2}\! +\!\Big|\big(D_{\bar\eta}\bar U,\frac{\bar U}{\bar\eta}\big)\Big|_\infty|(\mathfrak{q}_1^{(1)},\mathfrak{q}_2^{(1)})|_2\!+|((\mathfrak{q}_1^{(1)})_t,(\mathfrak{q}_2^{(1)})_t)|_2\Big)\cF(t)^\frac{1}{2}\\
&\leq C\mathfrak{p}(c_0)(\mathfrak{p}(c_1)+\cF(t))+C\mathfrak{p}(c_0)\bar\cD(t,\bar U)^\frac{1}{2}\cF(t)^\frac{1}{2},
\end{aligned}
\end{equation}
and $\mathrm{I}_{**}(t)$ can be handled by
\begin{equation}\label{I*3}
|\mathrm{I}_{**}(t)|\leq C\varepsilon^{-1}\mathfrak{p}(c_0)+\varepsilon\cF(t)\qquad\text{for all $\varepsilon\in(0,1)$}.
\end{equation}
Thus, substituting \eqref{I1*3}--\eqref{I*3} into \eqref{I0*3} with $\varepsilon$ sufficiently small, we can deduce from the Gr\"onwall inequality that, for all $t\in[0,T_3]$ with $T_3:=\min\{T_2,\mathfrak{p}(c_1)^{-1}\}$,
\begin{equation*}
\cF(t)+\int_0^t \big|(r^m\rho_0)^\frac{1}{2}U_{tt}\big|_2^2\,\mathrm{d}s\leq C\mathfrak{p}(c_0)\Big(1+\int_0^t \bar\cD(s,\bar U)^\frac{1}{2}\cF(s)^\frac{1}{2}\,\mathrm{d}s\Big).
\end{equation*}

To further simplify the above inequality, define
\begin{equation*}
\cY(t)=C\mathfrak{p}(c_0)\Big(1+\int_0^t \bar\cD(s,\bar U)^\frac{1}{2}\cF(s)^\frac{1}{2}\,\mathrm{d}s\Big).
\end{equation*}
Then $\cF(t)\leq \cY(t)$ and 
\begin{equation*}
\cY'(t)=C\mathfrak{p}(c_0) \bar\cD(t,\bar U)^\frac{1}{2}\cF(t)^\frac{1}{2} \leq C\mathfrak{p}(c_0) \bar\cD(t,\bar U)^\frac{1}{2}\cY(t)^\frac{1}{2}.
\end{equation*}
Clearly, this, together with \eqref{39}, implies that, for all $t\in[0,T_3]$, 
\begin{equation*}
\cY^\frac{1}{2}(t)\leq \cY^\frac{1}{2}(0)+C\mathfrak{p}(c_0)\int_0^t\bar\cD(s,\bar U)^\frac{1}{2}\,\mathrm{d}s\leq C\mathfrak{p}(c_0)^\frac{1}{2}+C\mathfrak{p}(c_0)(c_1t)^\frac{1}{2}\leq C\mathfrak{p}(c_0),
\end{equation*}
which yields that, for all $t\in[0,T_3]$,
\begin{equation}\label{11130}
\Big|(r^m\rho_0)^\frac{1}{2}\big(D_{\bar\eta} U_{t},\frac{U_{t}}{\bar\eta}\big)(t)\Big|_2+\int_0^t \big|(r^m\rho_0)^\frac{1}{2}U_{tt}\big|_2^2\,\mathrm{d}s\leq C\mathfrak{p}(c_0) .
\end{equation}

\smallskip
\textbf{3.} Since we have shown that $U$ satisfies \eqref{Energy-Id} with $(w,\mathfrak{q}_1,\mathfrak{q}_2)$ replaced by $(U_t,\mathfrak{q}_1^{(1)},\mathfrak{q}_2^{(1)})$ in Step 4 of \S\ref{subsection3.3}, we have
\begin{equation}\label{O-4}
\begin{aligned}
&\,\frac{1}{2}\frac{\mathrm{d}}{\dt}\big|(r^m\rho_0)^\frac{1}{2}U_{tt}\big|_2^2+2\mu\big|(r^m\rho_0)^\frac{1}{2} D_{\bar\eta} U_{tt}\big|_2^2+2\mu m\Big|(r^m\rho_0)^\frac{1}{2}\frac{U_{tt}}{\bar\eta}\Big|_2^2\\
&=\underline{4\mu \int_0^1 r^m\rho_0 \Big(D_{\bar\eta}\bar UD_{\bar\eta} U_{t}D_{\bar\eta} U_{tt}+m \frac{\bar U U_{t}  U_{tt}}{\bar\eta^3}\Big)\,\mathrm{d}r}_{:=\mathrm{I}_5}\\
&\quad +\underline{\int_0^1(r^m\rho_0)^\frac{1}{2} \Big(\big((\mathfrak{q}_2^{(1)})_t- \mathfrak{q}_2^{(1)}D_{\bar\eta}\bar U\big) D_{\bar\eta}U_{tt} +m \big((\mathfrak{q}_1^{(1)})_t- \mathfrak{q}_1^{(1)} \frac{\bar U}{\bar\eta}\big) \frac{U_{tt}}{\bar\eta}\Big) \,\mathrm{d}r}_{:=\mathrm{I}_6},
\end{aligned}    
\end{equation}
where, by using \eqref{qq12}--\eqref{qq12-t-2}, \eqref{11130}, Lemma \ref{lemma-useful1}, the H\"older and Young inequalities, 
\begin{align}
&\begin{aligned}\label{i5}
\mathrm{I}_5&\leq C\Big|\big(D_{\bar\eta}\bar U,\frac{\bar U}{\bar\eta}\big)\Big|_\infty\Big|(r^m\rho_0)^\frac{1}{2}\big(D_{\bar\eta} U_{t},\frac{U_{t}}{\bar\eta}\big)\Big|_2\Big|(r^m\rho_0)^\frac{1}{2}\big(D_{\bar\eta} U_{tt},\frac{U_{tt}}{\bar\eta}\big)\Big|_2\\
&\leq C\mathfrak{p}(c_1)+\frac{\mu}{100}\Big|(r^m\rho_0)^\frac{1}{2}\big(D_{\bar\eta} U_{tt},\frac{U_{tt}}{\bar\eta}\big)\Big|_2^2,
\end{aligned}\\
&\begin{aligned}\label{i6}
\mathrm{I}_6&\leq C\Big(\Big|\big(D_{\bar\eta}\bar U,\frac{\bar U}{\bar\eta}\big)\Big|_\infty|(\mathfrak{q}_1^{(1)},\mathfrak{q}_2^{(1)})|_2\!+|((\mathfrak{q}_1^{(1)})_t,(\mathfrak{q}_2^{(1)})_t)|_2\Big)\Big|(r^m\rho_0)^\frac{1}{2}\big(D_{\bar\eta} U_{tt},\frac{U_{tt}}{\bar\eta}\big)\Big|_2\\
&\leq C(\mathfrak{p}(c_1)+\bar\cD(t,\bar U))+\frac{\mu}{100}\Big|(r^m\rho_0)^\frac{1}{2}\big(D_{\bar\eta} U_{tt},\frac{U_{tt}}{\bar\eta}\big)\Big|_2^2.
\end{aligned}
\end{align}

Thus, plugging \eqref{i5}--\eqref{i6} into \eqref{O-4}, we obtain from \eqref{39} that
\begin{equation*}
\begin{aligned}
&\,\frac{\mathrm{d}}{\dt}\big(t\big|(r^m\rho_0)^\frac{1}{2}U_{tt}\big|_2^2\big)+\mu t\Big|(r^m\rho_0)^\frac{1}{2}\big(D_{\bar\eta} U_{tt},\frac{U_{tt}}{\bar\eta}\big)\Big|_2^2\\
&\leq \big|(r^m\rho_0)^\frac{1}{2}U_{tt}\big|_2^2+Ct(\mathfrak{p}(c_1)+\bar\cD(t,\bar U))\leq \big|(r^m\rho_0)^\frac{1}{2}U_{tt}\big|_2^2+C\mathfrak{p}(c_1),
\end{aligned}
\end{equation*}
which, along with the same argument as for \eqref{tauk}--\eqref{tauk2} 
in \S\ref{Section-globalestimates}, yields the desired results.
\end{proof}

\begin{Lemma}\label{c_1-c_2}
For any $t\in [0,T_3]$,
\begin{equation}
\bar\cE(t,U)+\Big|(U,D_{\bar\eta}U,\frac{U}{\bar\eta}\big)(t)\Big|_\infty\leq C\mathfrak{p}(c_0).
\end{equation}
\end{Lemma}
\begin{proof}
We divide the proof into three steps.

\smallskip
\textbf{1. Boundedness of $\bar\cE_{\mathrm{in}}(t,U)$.} It only remains to establish the third-order elliptic estimate for $U$ near the origin. First, it follows from \eqref{new-lp} that
\begin{equation}\label{new-lp*3}
\begin{aligned}
D_{\bar\eta}^2\big(D_{\bar\eta} U+ \frac{mU}{\bar\eta}\big)&=\frac{1}{2\mu}D_{\bar\eta} U_t-\frac{1}{\beta}\frac{D_{\bar\eta}\bar\Lambda}{\bar\varrho^\beta} D_{\bar\eta} U+\frac{1}{\beta}\frac{\bar\Lambda^2}{\bar\varrho^{2\beta}} D_{\bar\eta} U-\frac{1}{\beta}\frac{\bar\Lambda}{\bar\varrho^\beta} D_{\bar\eta}^2 U\\
&\quad +\frac{A\gamma(\gamma-1-\beta)}{2\mu \beta^2}\bar\varrho^{\gamma-1-2\beta}\bar\Lambda^2+\frac{A\gamma}{2\mu \beta}\bar\varrho^{\gamma-1-\beta}D_{\bar\eta}\bar\Lambda. 
\end{aligned}
\end{equation}
Then it follows from the fact that $\rho_0^\beta\sim 1-r$, \eqref{38}, and Lemmas \ref{im-1} and \ref{lemma-Lambda-guji}--\ref{c_0-c_1} that, for all $t\in [0,T_3]$, 
\begin{equation}\label{11112}
\begin{aligned}
&\,\Big|\zeta r^\frac{m}{2}\big(D_{\bar\eta}^3 U,D_{\bar\eta}^2(\frac{U}{\bar\eta}),\frac{1}{\bar\eta}D_{\bar\eta}(\frac{U}{\bar\eta})\big)\Big|_2\\
&\leq C\Big|\zeta r^\frac{m}{2} D_{\bar\eta}^2\big(D_{\bar\eta} U+ \frac{mU}{\bar\eta}\big)\Big|_2+\Big|\zeta r^\frac{m}{2} \frac{1}{\bar\eta}D_{\bar\eta}\big(D_{\bar\eta} U+ \frac{mU}{\bar\eta}\big) \Big|_2\\
&\leq C\Big(\Big|\zeta r^\frac{m}{2}\big(D_{\bar\eta}U_t,\frac{ U_t}{\bar\eta}\big)\Big|_2+\Big|\zeta_\frac{5}{8} r^\frac{m}{4}\big(D_{\bar\eta}\bar\Lambda,\frac{\bar\Lambda}{\bar\eta}\big)\Big|_4|\zeta r^\frac{m}{4} D_{\bar\eta} U|_4+|\bar\Lambda|_\infty\big|\zeta r^\frac{m}{2}D_{\bar\eta} U\big|_2\Big)\\
&\quad +C\Big(|\bar\Lambda|_\infty^2\big|\zeta r^\frac{m}{2}(D_{\bar\eta}^2 U,\bar\varrho^{\gamma-1-2\beta})\big|_2+\Big|\zeta r^\frac{m}{2}\big(D_{\bar\eta}\bar\Lambda,\frac{\bar\Lambda}{\bar\eta}\big)\Big|_2|\bar\varrho|_\infty^{\gamma-1-\beta}\Big)\leq C\mathfrak{p}(c_0).   
\end{aligned}
\end{equation}

\smallskip
\textbf{2. Boundedness of $\bar\cE_{\mathrm{ex}}(t,U)$.} It only remains to establish the second- and third-order elliptic estimates for $U$ away from the origin.

\smallskip
\textbf{2.1.} First, rewrite \eqref{new-lp*} as
\begin{equation}\label{new-lp**}
\begin{aligned}
D_{\bar\eta}^2 U=\frac{1}{2\mu} U_t-mD_{\bar\eta}
\big(\frac{U}{\bar\eta}\big)-\frac{1}{\beta}\frac{\bar\Lambda}{\bar\varrho^\beta} D_{\bar\eta} U+\frac{A\gamma}{2\mu \beta}\bar\varrho^{\gamma-1-\beta}\bar\Lambda.
\end{aligned}
\end{equation}
Then it follows from the above, \eqref{38}, and Lemmas \ref{lemma-Lambda-guji}--\ref{c_0-c_1} that, for all $\varepsilon>0$,
\begin{equation}
\begin{split}\label{10126}
\big|\chi^\sharp \rho_0^{\frac{1}{2}+\varepsilon}D_{\bar\eta}^2 U\big|_2&\leq C\big|\chi^\sharp \rho_0^{\frac{1}{2}+\varepsilon}(U_t,U,D_{\bar\eta}U)\big|_2+C\big|\chi^\sharp \rho_0^{-\frac{\beta}{2}+\varepsilon}\bar\Lambda\big|_2\big|\chi^\sharp \rho_0^\frac{1-\beta}{2}D_{\bar\eta} U\big|_\infty\\
&\quad +C\big|\chi^\sharp\rho_0^{\gamma-\frac{1}{2}-\beta+\varepsilon}\big|_2|\bar\Lambda|_\infty\leq C(\varepsilon)\mathfrak{p}(c_0).
\end{split}
\end{equation}
Clearly, since $(\frac{3}{2}-\varepsilon_0)\beta>\frac{1}{2}$, the above also leads to
\begin{equation}\label{11139}
\big|\chi^\sharp \rho_0^{(\frac{3}{2}-\varepsilon_0)\beta}D_{\bar\eta}^2 U\big|_2 \leq C\mathfrak{p}(c_0).
\end{equation}

\smallskip
\textbf{2.2.} 
Next, we can obtain similarly from \eqref{cal-1} that, for all $\iota\in (-\frac{\beta}{2},1+\frac{\beta}{2})$ and $\sigma>0$,
\begin{equation}\label{cal-1'}
\big|\chi^\sharp\rho_0^{\iota-\beta+\sigma}D_{\bar\eta} U\big|_2\leq C(\sigma,\iota)\mathfrak{p}(c_0)\big(1+\big|\chi^\sharp\rho_0^\iota(U,D_{\bar\eta} U,U_t)\big|_2\big) \qquad \text{for all $t\in[0,T_3]$}.
\end{equation}
Then, due to the facts that
\begin{equation*}
\varepsilon_0<\frac{1}{2}, \qquad \big(\frac{3}{2}-\varepsilon_0\big)\beta>\frac{1}{2},
\end{equation*}
we can choose fixed $(\iota,\sigma)$ in \eqref{cal-1} such that
\begin{equation*}
\iota+\sigma=\big(\frac{1}{2}-\varepsilon_0\big)\beta,\qquad \iota\in \big(-\frac{\beta}{2},1+\frac{\beta}{2}\big), \qquad 0<\sigma<\min\big\{(1-\varepsilon_0)\beta,\big(\frac{3}{2}-\varepsilon_0\big)\beta-\frac{1}{2}\big\}.
\end{equation*}
Hence, it follows from $(\frac{3}{2}-\varepsilon_0)\beta-\sigma>\frac{1}{2}$, \eqref{10126}--\eqref{11139}, and Lemmas \ref{c_0-c_1}--\ref{c_1} and \ref{hardy-inequality} that
\begin{equation}\label{4015}
\begin{aligned}
&\,\big|\chi^\sharp\rho_0^{-(\frac{1}{2}+\varepsilon_0)\beta}D_{\bar\eta} U\big|_2\leq C\mathfrak{p}(c_0)\big(1+\big|\chi^\sharp\rho_0^{(\frac{1}{2}-\varepsilon_0)\beta-\sigma}(U,D_{\bar\eta} U,U_t)\big|_2\big)\\
&\leq C\mathfrak{p}(c_0)\big(1+\big|\chi^\sharp\rho_0^{(\frac{3}{2}-\varepsilon_0)\beta-\sigma}(U,D_{\bar\eta} U,D_{\bar\eta}^2 U,U_t,D_{\bar\eta} U_{t})\big|_2\big)\leq C\mathfrak{p}(c_0).
\end{aligned}    
\end{equation}

Finally, based on \eqref{new-lp**}, using an argument similar to \eqref{J12-J14-1}--\eqref{9984} in Lemma \ref{lemma-u-ell-D2-refine} with $\eta$ replaced by $\bar\eta$, together with \eqref{4015} and Lemmas \ref{lemma-Lambda-guji}--\ref{c_1}, yields that, for all $0\leq t\leq T_3$,
\begin{equation}\label{4016}
\big|\chi^\sharp\rho_0^{(\frac{1}{2}-\varepsilon_0)\beta}D_{\bar\eta}^2 U\big|_2\leq C\mathfrak{p}(c_0).
\end{equation}

\smallskip
\textbf{2.3.} Based on \eqref{new-lp*3}, using a similar argument in Step 2 of the proof of Lemma \ref{lemma-u-D3-ell} with $\eta$ replaced by $\bar\eta$, we obtain from \eqref{11139}--\eqref{4016} and Lemmas \ref{lemma-Lambda-guji}--\ref{c_1} that, for all $t\in[0,T_3]$,
\begin{equation}\label{3369}
\big|\chi^\sharp\rho_0^{(\frac{3}{2}-\varepsilon_0)\beta}D_{\bar\eta}^3 U\big|_2\leq C\mathfrak{p}(c_0)\Big(1+\int_0^t \big|\chi^\sharp\rho_0^{(\frac{3}{2}-\varepsilon_0)\beta}D_{\bar\eta}^3 U\big|_2\,\mathrm{d}s\Big),
\end{equation}
which, along with the Gr\"onwall inequality, leads to the desired result of this lemma.

\smallskip
\textbf{3.} Finally,  it follows from \eqref{1197}, \eqref{11112}, \eqref{11139}, \eqref{3369}, and Lemmas \ref{c_0-c_1}--\ref{c_1} and  \ref{sobolev-embedding}--\ref{hardy-inequality} that, for all $t\in[0,T_3]$,
\begin{equation*}
\begin{aligned}
&\,\Big|(U,D_{\bar\eta}U,\frac{U}{\bar\eta}\big)\Big|_\infty \leq C\Big|(U,D_{\bar\eta}U,D_{\bar\eta}^2U,\frac{U}{\bar\eta},D_{\bar\eta}(\frac{U}{\bar\eta})\big)\Big|_1\\
&\leq C\sum_{j=0}^3 \big(\big|\chi r^\frac{m}{2}D_{\bar\eta}^j U\big|_2+\big|\chi^\sharp\rho_0^{(\frac{3}{2}-\varepsilon_0)\beta}D_{\bar\eta}^j U\big|_2\big) +C\sum_{j=0}^2\Big|\zeta r^\frac{m}{2}D_{\bar\eta}^j(\frac{U}{\bar\eta})\Big|_2\leq C\mathfrak{p}(c_0).
\end{aligned}
\end{equation*}

This completes the proof.
\end{proof}

\begin{Lemma}\label{c_3}
For all $t\in[0,T_3]$,
\begin{equation*}
t\bar\cD(t,U)+\int_0^t\bar\cD(s,U)\,\mathrm{d}s \leq C\mathfrak{p}(c_0).
\end{equation*}
\end{Lemma}
\begin{proof}
We divide the proof into three steps.

\smallskip
\textbf{1. $L^1(0,T)$-boundedness of $\bar\cD_{\mathrm{in}}(t,U)$.} 

\smallskip
\textbf{1.1.} 
First, it follows from \eqref{new-lp*} by applying $\partial_t$ that 
\begin{equation}\label{11140*}
D_{\bar\eta}\big(D_{\bar\eta} U_t+ \frac{mU_t}{\bar\eta}\big)=\sum_{i=7}^9\mathrm{I}_i,
\end{equation}
where
\begin{equation}\label{I789}
\begin{aligned}
\mathrm{I}_7&:=D_{\bar\eta}^2 \bar U D_{\bar\eta} U+D_{\bar\eta}(\frac{\bar U}{\bar\eta})\frac{m U}{\bar\eta}+2D_{\bar\eta} \bar U D_{\bar\eta}^2 U+m\big(D_{\bar\eta} \bar U+\frac{ \bar U}{\bar\eta}\big)D_{\bar\eta}(\frac{U}{\bar\eta}),\\
\mathrm{I}_8&:=\frac{1}{2\mu} U_{tt}-\frac{1}{\beta}\frac{\bar\Lambda_t}{\bar\varrho^\beta} D_{\bar\eta} U-\frac{\bar\Lambda}{\bar\varrho^{\beta}} \Big(\big(\frac{\beta-1}{\beta}D_{\bar\eta}\bar U+\frac{m\bar U}{\bar\eta}\big)D_{\bar\eta} U+\frac{1}{\beta} D_{\bar\eta} U_t\Big),\\
\mathrm{I}_9&:=-\frac{A\gamma(\gamma-1-\beta)}{2\mu \beta}\bar\varrho^{\gamma-1-\beta}\big(D_{\bar\eta}\bar U+\frac{m\bar U}{\bar\eta}\big)\bar\Lambda+\frac{A\gamma}{2\mu \beta}\bar\varrho^{\gamma-1-\beta}\bar\Lambda_t. 
\end{aligned}
\end{equation}

Then, by the fact that $\rho_0^\beta\sim 1-r$, \eqref{38},  Lemmas \ref{lemma-useful1}--\ref{lemma-Lambda-guji}, and the H\"older inequality, we have
\begin{align}
&\begin{aligned}\label{11142}
|\zeta r^\frac{m}{2}\mathrm{I}_7|_2&\leq C\Big|\zeta_\frac{5}{8} r^\frac{m}{4}\big(D_{\bar\eta}^2 \bar U,D_{\bar\eta}(\frac{\bar U}{\bar\eta})\big)\Big|_4\Big|\zeta r^\frac{m}{4}\big(D_{\bar\eta}U,\frac{U}{\bar\eta}\big)\Big|_4\\
&\quad +C\Big|\big(D_{\bar\eta} \bar U,\frac{\bar U}{\bar\eta}\big)\Big|_\infty\Big|\zeta r^\frac{m}{2}\big(D_{\bar\eta}^2 U,D_{\bar\eta}(\frac{ U}{\bar\eta})\big)\Big|_2\leq C\mathfrak{p}(c_0),   
\end{aligned}\\
&\begin{aligned}
|\zeta r^\frac{m}{2}\mathrm{I}_8|_2&\leq C|\zeta r^\frac{m}{2}U_{tt}|_2+C|\zeta r^\frac{m}{4}\bar\Lambda_t|_4|\zeta r^\frac{m}{4}D_{\bar\eta} U|_4\\
&\quad +C|\bar\Lambda|_\infty\Big(\Big|\big(D_{\bar\eta} \bar U,\frac{\bar U}{\bar\eta}\big)\Big|_\infty+|\zeta r^\frac{m}{2}D_{\bar\eta} U_t|_2\Big)\leq C(|\zeta r^\frac{m}{2}U_{tt}|_2+\mathfrak{p}(c_1)),
\end{aligned}\\
&\begin{aligned}\label{11144}
|\zeta r^\frac{m}{2}\mathrm{I}_9|_2\leq C|\bar\varrho|_\infty^{\gamma-1-\beta}\Big(|\bar\Lambda|_\infty \Big|\big(D_{\bar\eta} \bar U,\frac{\bar U}{\bar\eta}\big)\Big|_\infty + |\zeta r^\frac{m}{4}\bar\Lambda_t|_4\Big)\leq C\mathfrak{p}(c_1).
\end{aligned}
\end{align}

Hence, collecting \eqref{11140*} and \eqref{11142}--\eqref{11144}, together with Lemmas \ref{im-1} and \ref{c_1}, gives that, for all $t\in [0,T_3]$,
\begin{equation}\label{11145}
\int_0^t\Big|\zeta r^\frac{m}{2} \big(D_{\bar\eta}^2 U_t,D_{\bar\eta}(\frac{U_t}{\bar\eta})\big)\Big|_2^2\,\mathrm{d}s\leq C\Big(\int_0^t|\zeta r^\frac{m}{2}U_{tt}|_2^2\,\mathrm{d}s+\mathfrak{p}(c_1)t\Big)\leq C\mathfrak{p}(c_0).
\end{equation}

\smallskip
\textbf{1.2.} It follows from \eqref{new-lp*3} by applying $D_{\bar\eta}$ that
\begin{equation}\label{11146}
D_{\bar\eta}^3\big(D_{\bar\eta} U+ \frac{mU}{\bar\eta}\big)=\sum_{i=10}^{11}\mathrm{I}_i,
\end{equation}
where
\begin{equation*}
\begin{aligned}
\mathrm{I}_{10}&:=\frac{1}{2\mu} D_{\bar\eta}^2 U_t-\frac{1}{\beta}\frac{D_{\bar\eta}^2\bar\Lambda}{\bar\varrho^\beta} D_{\bar\eta} U+\frac{3}{\beta}\frac{\bar\Lambda D_{\bar\eta}\bar\Lambda}{\bar\varrho^{2\beta}} D_{\bar\eta} U-\frac{2}{\beta}\frac{\bar\Lambda^3}{\bar\varrho^{3\beta}} D_{\bar\eta} U\\
&\quad\,\,-\frac{2}{\beta}\frac{D_{\bar\eta}\bar\Lambda}{\bar\varrho^\beta} D_{\bar\eta}^2 U +\frac{2}{\beta}\frac{\bar\Lambda^2}{\bar\varrho^{2\beta}} D_{\bar\eta}^2 U-\frac{1}{\beta}\frac{\bar\Lambda}{\bar\varrho^\beta} D_{\bar\eta}^3 U,\\
\mathrm{I}_{11}&:=\frac{A\gamma(\gamma-1-\beta)(\gamma-1-2\beta)}{2\mu \beta^3}\bar\varrho^{\gamma-1-3\beta}\bar\Lambda^3\\
&\quad\,\,+\frac{3A\gamma(\gamma-1-\beta)}{2\mu \beta^2}\bar\varrho^{\gamma-1-2\beta}\bar\Lambda D_{\bar\eta}\bar\Lambda+\frac{A\gamma}{2\mu \beta}\bar\varrho^{\gamma-1-\beta}D_{\bar\eta}^2\bar\Lambda.
\end{aligned}
\end{equation*}
Then, due to the fact that $\rho_0^\beta\sim 1-r$, \eqref{38}, \eqref{11123}, Lemmas \ref{lemma-Lambda-guji} and \ref{c_1-c_2}, and the H\"older inequality, we obtain that, for all $t\in [0,T_3]$, 
\begin{align}
&\begin{aligned}
|\zeta r^\frac{m}{2}\mathrm{I}_{10}|_2&\leq C|\zeta r^\frac{m}{2}D_{\bar\eta}^2 U_t|_2+C\big(|\zeta r^\frac{m}{2}D_{\bar\eta}^2 \bar\Lambda|_2+|\bar\Lambda|_\infty|\zeta r^\frac{m}{2}D_{\bar\eta} \bar\Lambda|_2+|\bar\Lambda|_\infty^3\big)|D_{\bar\eta} U|_\infty\\
&\quad\, +C\big(\big|\zeta_\frac{5}{8} r^\frac{m}{4}D_{\bar\eta} \bar\Lambda\big|_4+|\bar\Lambda|_\infty^2\big)|\zeta r^\frac{m}{4} D_{\bar\eta}^2 U|_4+C |\bar\Lambda|_\infty |\zeta r^\frac{m}{2} D_{\bar\eta}^3 U|_2\\
&\leq C(|\zeta r^\frac{m}{2}D_{\bar\eta}^2 U_t|_2+\mathfrak{p}(c_0)),
\end{aligned}\\
&\begin{aligned}\label{11149}
|\zeta r^\frac{m}{2}\mathrm{I}_{11}|_2& \leq C\mathfrak{p}(c_0)\big(|\bar\Lambda|_\infty^3+|\bar\Lambda|_\infty|\zeta r^\frac{m}{2}D_{\bar\eta}\bar \Lambda|_2+|\zeta r^\frac{m}{2}D_{\bar\eta}^2\bar \Lambda|_2\big)\leq C\mathfrak{p}(c_0).
\end{aligned}
\end{align}
Hence, collecting \eqref{11146}--\eqref{11149} gives that, for all $t\in[0,T_3]$,
\begin{equation}\label{11150}
\Big|\zeta r^\frac{m}{2}D_{\bar\eta}^3\big(D_{\bar\eta} U+ \frac{mU}{\bar\eta}\big)\Big|_2\leq C(|\zeta r^\frac{m}{2}D_{\bar\eta}^2 U_t|_2+\mathfrak{p}(c_0)).
\end{equation}

Next, we multiply \eqref{new-lp*} by $\frac{1}{\bar\eta}$ and apply $D_{\bar\eta}$ to the resulting equality to obtain 
\begin{equation*}
\begin{aligned}
D_{\bar\eta}\Big(\frac{1}{\bar\eta}D_{\bar\eta}\big(D_{\bar\eta} U+ \frac{mU}{\bar\eta}\big)\Big)&=\frac{1}{2\mu} D_{\bar\eta}(\frac{U_t}{\bar\eta})-\frac{1}{\beta}\Big(D_{\bar\eta}(\frac{\bar\Lambda}{\bar\eta})\frac{D_{\bar\eta} U}{\bar\varrho^\beta}- \frac{\bar\Lambda^2}{\bar\eta\bar\varrho^{2\beta}} D_{\bar\eta} U+ \frac{\bar\Lambda}{\bar\eta\bar\varrho^\beta} D_{\bar\eta}^2 U\Big)\\
&\quad +\frac{A\gamma(\gamma-1-\beta)}{2\mu \beta^2}\bar\varrho^{\gamma-1-2\beta}\frac{\bar\Lambda^2}{\bar\eta}+\frac{A\gamma}{2\mu \beta}\bar\varrho^{\gamma-1-\beta}D_{\bar\eta}(\frac{\bar\Lambda}{\bar\eta}).
\end{aligned}
\end{equation*}
Then, following the calculations \eqref{11146}--\eqref{11150}, we can similarly obtain 
\begin{equation}\label{11150-}
\Big|\zeta r^\frac{m}{2}D_{\bar\eta}\Big(\frac{1}{\bar\eta}D_{\bar\eta}\big(D_{\bar\eta} U+ \frac{mU}{\bar\eta}\big)\Big)\Big|_2\leq C\Big(\Big|\zeta r^\frac{m}{2} D_{\bar\eta}(\frac{U_t}{\bar\eta})\Big|_2+\mathfrak{p}(c_0)\Big).
\end{equation}

Finally, combining \eqref{11150}--\eqref{11150-}, along with Lemma \ref{im-1} and \eqref{11145}, implies that, for all $t\in[0,T_3]$,
\begin{equation}
\begin{aligned}
&\,\int_0^t\Big|\zeta r^\frac{m}{2} \Big(D_{\bar\eta}^4 U,D_{\bar\eta}^3(\frac{U}{\bar\eta}),D_{\bar\eta}\big(\frac{1}{\bar\eta}D_{\bar\eta}(\frac{U}{\bar\eta})\big)\Big)\Big|_2^2\,\mathrm{d}s \\
&\leq C\int_0^t\Big|\zeta r^\frac{m}{2} \big(D_{\bar\eta}^2 U_t,D_{\bar\eta}(\frac{U_t}{\bar\eta})\big)\Big|_2^2\,\mathrm{d}s+Ct \mathfrak{p}(c_0)\leq  C\mathfrak{p}(c_0).
\end{aligned}
\end{equation}

\smallskip
\textbf{2. $L^1(0,T)$-boundedness of $\bar\cD_{\mathrm{ex}}(t,U)$.} 

\smallskip
\textbf{2.1.} First, applying $\partial_t$ to $\eqref{3...136}_1$ yields
\begin{equation*}
\begin{aligned}
D_{\bar\eta} U_t&=-\frac{A(\gamma-1)}{2\mu}\bar\varrho^{\gamma-1}\big(D_{\bar\eta} \bar U+\frac{m\bar U}{\bar\eta}\big)+\frac{m}{\bar\varrho}\big(D_{\bar\eta} \bar U+\frac{m\bar U}{\bar \eta}\big)\int_r^1 \tilde{r}^m\rho_0 \big(\frac{D_{\bar\eta} U}{\bar\eta^{m+1}}-\frac{U}{\bar\eta^{m+2}}\big)\mathrm{d}\tilde{r} \\
&\quad\,\, +\frac{m}{\bar\varrho}\int_r^1 \tilde{r}^m\rho_0 \big(\frac{D_{\bar\eta} U_t-D_{\bar\eta}\bar U D_{\bar\eta} U}{\bar\eta^{m+1}}-\frac{(m+1)\bar UD_{\bar\eta} U+U_t}{\bar\eta^{m+2}}+\frac{(m+2)\bar U U}{\bar \eta^{m+3}}\big)\mathrm{d}\tilde{r}\\
&\quad\,\, -\frac{1}{2\mu \bar\varrho}\Big(\big(D_{\bar\eta} \bar U+\frac{m\bar U}{\bar\eta}\big)\int_r^1 \frac{\tilde{r}^m}{\bar\eta^m}\rho_0U_t\,\mathrm{d}\tilde{r}+\int_r^1\tilde{r}^m \rho_0\big(\frac{U_{tt}}{\bar\eta^m}-\frac{m \bar UU_t}{\bar\eta^{m+1}}\big)\,\mathrm{d}\tilde{r}\Big)+ D_{\bar\eta}\bar UD_{\bar\eta} U.
\end{aligned}
\end{equation*}
Note that the above equality enjoys a structure similar to that of \eqref{9995}, and hence we can follow analogous calculations \eqref{9966}--\eqref{9998} to derive
\begin{equation*} 
\big|\chi^\sharp\rho_0^{(\frac{1}{2}-\varepsilon_0)\beta}D_{\bar\eta} U_t\big|_2\leq C\mathfrak{p}(c_1)\big(1+\big|\chi^\sharp\rho_0^{(1-\varepsilon_0)\beta-\frac{1}{2}}D_{\bar\eta} U\big|_2\big)+C\mathfrak{p}(c_0)\big|\chi^\sharp\rho_0^\frac{1}{2}U_{tt}\big|_2,
\end{equation*}
and, by Lemma \ref{hardy-inequality}, 
\begin{equation*} 
\big|\chi^\sharp\rho_0^{(1-\varepsilon_0)\beta-\frac{1}{2}}D_{\bar\eta} U\big|_2\leq C(T)\big|\chi^\sharp\rho_0^{\frac{3\beta}{2}}(D_{\bar\eta} U,D_{\bar\eta}^2 U,D_{\bar\eta}^3 U)\big|_2\leq C\mathfrak{p}(c_1).
\end{equation*}
Therefore, combining the above two inequalities gives that, for all $t\in[0,T_3]$,
\begin{equation}\label{9997**}
\big|\chi^\sharp\rho_0^{(\frac{1}{2}-\varepsilon_0)\beta}D_{\bar\eta} U_t\big|_2\leq C\big(\mathfrak{p}(c_1)+\mathfrak{p}(c_0)\big|\chi^\sharp\rho_0^\frac{1}{2}U_{tt}\big|_2\big).
\end{equation}

As a consequence, recalling \eqref{11140*}--\eqref{I789}, we obtain from \eqref{varepsilon0}, \eqref{38}, \eqref{4015}--\eqref{4016}, \eqref{9997**}, and Lemmas \ref{lemma-useful1}--\ref{lemma-Lambda-guji} and \ref{hardy-inequality} that, for all $t\in [0, T_3]$,
\begin{align*}
&\begin{aligned} 
\big|\chi^\sharp \rho_0^{(\frac{3}{2}-\varepsilon_0)\beta}\mathrm{I}_7\big|_2&\leq C\big|\chi^\sharp \rho_0^\beta D_{\bar\eta}^2 \bar U\big|_\infty \big|\chi^\sharp \rho_0^{(\frac{1}{2}-\varepsilon_0)\beta} D_{\bar\eta} U\big|_2\\
&\quad +C|(\bar U,D_{\bar\eta}\bar U)|_\infty\big|\chi^\sharp\rho_0^{(\frac{1}{2}-\varepsilon_0)\beta}(U,D_{\bar\eta} U,D_{\bar\eta}^2 U)\big|_2\leq C\mathfrak{p}(c_1),
\end{aligned}\\
&\begin{aligned}
\big|\chi^\sharp \rho_0^{(\frac{3}{2}-\varepsilon_0)\beta}\mathrm{I}_8\big|_2&\leq C\big(\big|\chi^\sharp \rho_0^{(\frac{3}{2}-\varepsilon_0)\beta}U_{tt}\big|_2+|\chi^\sharp \bar\Lambda_t|_\infty\big|\chi^\sharp \rho_0^{(\frac{1}{2}-\varepsilon_0)\beta}D_{\bar\eta}U\big|_2\big)\\
&\quad +C|\bar\Lambda|_\infty\big(|(\bar U,D_{\bar\eta}\bar U)|_\infty\big|\chi^\sharp \rho_0^{(\frac{1}{2}-\varepsilon_0)\beta}D_{\bar\eta}U\big|_2+ \big|\chi^\sharp \rho_0^{(\frac{1}{2}-\varepsilon_0)\beta}D_{\bar\eta}U_t\big|_2\big)\\
&\leq C\big(\mathfrak{p}(c_1)+\mathfrak{p}(c_0)\big|\chi^\sharp\rho_0^\frac{1}{2}U_{tt}\big|_2\big),
\end{aligned}\\
&\begin{aligned}
\big|\chi^\sharp \rho_0^{(\frac{3}{2}-\varepsilon_0)\beta}\mathrm{I}_9\big|_2&\leq C\big|\chi^\sharp \rho_0^{\gamma-1+(\frac{1}{2}-\varepsilon_0)\beta}\big|_2 |(1,\bar U,D_{\bar\eta}\bar U)|_\infty|\chi^\sharp(\bar\Lambda,\bar\Lambda_t)|_\infty\leq C\mathfrak{p}(c_1), 
\end{aligned}
\end{align*}
which, along with \eqref{11140*} and Lemmas \ref{c_0-c_1}--\ref{c_1}, leads to
\begin{equation}\label{9997***}
\begin{aligned}
\big|\chi^\sharp\rho_0^{(\frac{3}{2}-\varepsilon_0)\beta}D_{\bar\eta}^2 U_t\big|_2&\leq C\big|\chi^\sharp\rho_0^{(\frac{3}{2}-\varepsilon_0)\beta}(U_t,D_{\bar\eta} U_t)\big|_2+C\big(\mathfrak{p}(c_1)+\mathfrak{p}(c_0)\big|\chi^\sharp\rho_0^\frac{1}{2}U_{tt}\big|_2\big)\\
&\leq C\big(\mathfrak{p}(c_1)+\mathfrak{p}(c_0)\big|\chi^\sharp\rho_0^\frac{1}{2}U_{tt}\big|_2\big).
\end{aligned}
\end{equation}
Finally, integrating the above over $[0,t]$, together with Lemma \ref{c_1}, implies that, for all $t\in[0,T_3]$,
\begin{equation}\label{9997+}
\int_0^t\big|\chi^\sharp\rho_0^{(\frac{3}{2}-\varepsilon_0)\beta}D_{\bar\eta}^2 U_t\big|_2^2\,\mathrm{d}s\leq C\mathfrak{p}(c_1)t+C\mathfrak{p}(c_0)\int_0^t\big|\chi^\sharp\rho_0^\frac{1}{2}U_{tt}\big|_2^2\,\mathrm{d}s\leq C\mathfrak{p}(c_0).
\end{equation}

\smallskip
\textbf{2.2.}
First, recall the derivation of \eqref{xxxx-l}:
\begin{equation}\label{xxxx-l'}
\bar\cT_{\mathrm{cross}}:=(D_{\bar\eta}^3U)_r+\big(\frac{1}{\beta}+2\big)\frac{(\rho_0^\beta)_r}{\rho_0^\beta}D_{\bar\eta}^3U =\sum_{i=12}^{15}\mathrm{I}_i,
\end{equation}
where
\begin{equation}\label{rr1-rr3}
\begin{aligned}
\mathrm{I}_{12}&:=(1+2\beta) \bar\eta_r\frac{D_{\bar\eta}\bar{\mathscr{J}}}{\bar{\mathscr{J}}} D_{\bar\eta}^3 U-\big(1+\frac{2}{\beta} \big)\frac{\bar\eta_r}{\bar\varrho^{\beta}}D_{\bar\eta} \bar\Lambda D_{\bar\eta}^2 U-\frac{1}{\beta}\frac{\bar\eta_r}{\bar\varrho^{\beta}} D_{\bar\eta}^2\bar\Lambda D_{\bar\eta} U,\\
\mathrm{I}_{13}&:=-m\frac{\bar\eta_r}{\bar\varrho^{\beta}}\Big(D_{\bar\eta} \bar\Lambda D_{\bar\eta}\big(\frac{U}{\bar\eta}\big)+2 \bar\Lambda D_{\bar\eta}^2\big(\frac{U}{\bar\eta}\big)+ \bar\varrho^\beta D_{\bar\eta}^3\big(\frac{U}{\bar\eta}\big)\Big),\\
\mathrm{I}_{14}&:=\frac{1}{2\mu}\frac{\bar\eta_r}{\bar\varrho^{\beta}}\big(D_{\bar\eta} \bar\Lambda U_t+2\bar\Lambda D_{\bar\eta} U_t+ \bar\varrho^\beta D_{\bar\eta}^2 U_t\big),\\
\mathrm{I}_{15}&:=\frac{A\gamma}{2\mu\beta} \bar\eta_r \bar\varrho^{\gamma-1-3\beta}\Big( \bar\varrho^{2\beta} D_{\bar\eta}^2 \bar\Lambda+\frac{3(\gamma-1)}{\beta}\bar\varrho^\beta \bar\Lambda D_{\bar\eta} \bar\Lambda +\frac{(\gamma-1)(\gamma-1-\beta)}{\beta^2}\bar\Lambda^3\Big).
\end{aligned}
\end{equation}

For the right-hand side of the above, it follows from the facts that
\begin{equation*}
\rho_0^\beta\sim 1-r,\qquad \big(\frac{1}{2}-\varepsilon_0\big)\beta+\gamma-1>0,
\end{equation*} 
\eqref{varepsilon0}, \eqref{38}, \eqref{4016}, and Lemmas \ref{lemma-useful1}--\ref{c_1-c_2} and \ref{hardy-inequality} that, for all $t\in [0,T_3]$,
\begin{align}
&\begin{aligned}
\big|\chi^\sharp \rho_0^{(\frac{3}{2}-\varepsilon_0)\beta}\mathrm{I}_{12}\big|_2&\leq C\big(|\chi^\sharp D_{\bar\eta}\bar{\mathscr{J}}|_\infty\big|\chi^\sharp \rho_0^{(\frac{3}{2}-\varepsilon_0)\beta}D_{\bar\eta}^3U\big|_2+|\chi^\sharp D_{\bar\eta}\bar\Lambda|_\infty\big|\chi^\sharp \rho_0^{(\frac{1}{2}-\varepsilon_0)\beta} D_{\bar\eta}^2U\big|_2\big)\\
&\quad +C\big|\chi^\sharp \rho_0^{(\frac{1}{2}-\varepsilon_0)\beta}D_{\bar\eta}^2 \bar\Lambda\big|_2|\chi^\sharp D_{\bar\eta} U|_\infty \leq C\mathfrak{p}(c_0),\notag
\end{aligned}\\
&\begin{aligned}
\big|\chi^\sharp \rho_0^{(\frac{3}{2}-\varepsilon_0)\beta}\mathrm{I}_{13}\big|_2&\leq C\big|\chi^\sharp \rho_0^{(\frac{1}{2}-\varepsilon_0)\beta}D_{\bar\eta} \bar\Lambda\big|_2|\chi^\sharp(U,D_{\bar\eta} U)|_\infty\\
&\quad + C(1+|\bar\Lambda|_\infty) \big|\chi^\sharp \rho_0^{(\frac{3}{2}-\varepsilon_0)\beta} (U,D_{\bar\eta} U,D_{\bar\eta}^2 U,D_{\bar\eta}^3 U)\big|_2\leq C\mathfrak{p}(c_0),\notag
\end{aligned}\\
&\begin{aligned}[t]\label{QQQ}
\big|\chi^\sharp \rho_0^{(\frac{3}{2}-\varepsilon_0)\beta}\mathrm{I}_{14}\big|_2&\leq C|\chi^\sharp D_{\bar\eta} \bar\Lambda|_\infty\big|\chi^\sharp \rho_0^{(\frac{3}{2}-\varepsilon_0)\beta} (U_t,D_{\bar\eta} U_t)\big|_2\\
&\quad +C(1+|\bar\Lambda|_\infty)\big|\chi^\sharp \rho_0^{(\frac{3}{2}-\varepsilon_0)\beta} (D_{\bar\eta} U_t,D_{\bar\eta}^2 U_t)\big|_2\\
&\leq C\mathfrak{p}(c_0)\big(1+\big|\chi^\sharp \rho_0^{(\frac{3}{2}-\varepsilon_0)\beta} D_{\bar\eta}^2 U_t\big|_2\big).
\end{aligned}
\end{align}
For  $\mathrm{I}_{15}$, if $\beta<\gamma-1$, we can additionally use $(\frac{3}{2}-\varepsilon_0)\beta+\gamma-1-3\beta>-\frac{\beta}{2}$ to obtain
\begin{equation}\label{QQQ'}
\begin{aligned}
\big|\chi^\sharp \rho_0^{(\frac{3}{2}-\varepsilon_0)\beta}\mathrm{I}_{15}\big|_2&\leq C\big(|\rho_0|_\infty^{\gamma-1}\big|\chi^\sharp \rho_0^{(\frac{1}{2}-\varepsilon_0)\beta}D_{\bar\eta}^2\bar\Lambda\big|_2 + |\bar\Lambda|_\infty\big|\chi^\sharp \rho_0^{(\frac{1}{2}-\varepsilon_0)\beta}D_{\bar\eta} \bar\Lambda\big|_2\big)\\
&\quad +C\underline{(\gamma-1-\beta)\big|\chi^\sharp \rho_0^{(\frac{3}{2}-\varepsilon_0)\beta+\gamma-1-3\beta}\big|_2|\bar\Lambda|_\infty^3}_{\,:=\mathrm{I}_{15}^\star}\leq C\mathfrak{p}(c_0);
\end{aligned}
\end{equation} 
if $\beta=\gamma-1$, $\mathrm{I}_{15}$ can be estimated in the same way without considering
the term $\mathrm{I}_{15}^\star$. 

Therefore, substituting \eqref{QQQ}--\eqref{QQQ'} into \eqref{xxxx-l'} gives
\begin{equation*}
\big|\zeta^\sharp\rho_0^{(\frac{3}{2}-\varepsilon_0)\beta}\bar\cT_{\mathrm{cross}}\big|_2\leq C\mathfrak{p}(c_0)\big(1+\big|\chi^\sharp \rho_0^{(\frac{3}{2}-\varepsilon_0)\beta} D_{\bar\eta}^2 U_t\big|_2\big),
\end{equation*}
which, along with  Proposition \ref{prop2.1} and Lemma \ref{c_1-c_2}, yields
\begin{equation*}
\begin{aligned}
\big|\chi^\sharp\rho_0^{\left(\frac{3}{2}-\varepsilon_0\right)\beta}D_{\bar\eta}^4U\big|_2&\leq \big|(\zeta-\chi)\rho_0^{\left(\frac{3}{2}-\varepsilon_0\right)\beta}D_{\bar\eta}^4U\big|_2+\big|\zeta^\sharp\rho_0^{\left(\frac{3}{2}-\varepsilon_0\right)\beta}D_{\bar\eta}^4U\big|_2\\
&\leq C\bar\cD_{\mathrm{in}}(t,U)^\frac{1}{2}+C\big|\zeta^\sharp\rho_0^{(\frac{3}{2}-\varepsilon_0)\beta}\bar\cT_{\mathrm{cross}}\big|_2+C\mathfrak{p}(c_0) \big|\rho_0^{\left(\frac{3}{2}-\varepsilon_0\right)\beta}D_{\bar\eta}^3U\big|_2\\
&\leq  C\mathfrak{p}(c_0)\big(1+\bar\cD_{\mathrm{in}}(t,U)^\frac{1}{2}+\big|\chi^\sharp \rho_0^{(\frac{3}{2}-\varepsilon_0)\beta} D_{\bar\eta}^2 U_t\big|_2\big).
\end{aligned}
\end{equation*}

Finally, it follows from the $L^1(0,T)$-estimate of $\bar\cD(t,U)$ and \eqref{9997+} that, for all $t\in[0,T_3]$,
\begin{equation*}
\begin{aligned}
\int_0^t\big|\chi^\sharp\rho_0^{\left(\frac{3}{2}-\varepsilon_0\right)\beta}D_{\bar\eta}^4U\big|_2^2\,\mathrm{d}s\leq  C\mathfrak{p}(c_0)\Big(t+\int_0^t\big(\bar\cD_{\mathrm{in}}(s,U)+\big|\chi^\sharp \rho_0^{(\frac{3}{2}-\varepsilon_0)\beta} D_{\bar\eta}^2 U_t\big|_2^2\big)\,\mathrm{d}s\Big)\leq C\mathfrak{p}(c_0).
\end{aligned}
\end{equation*}

\smallskip
\textbf{3. $L^\infty(0,T)$-boundedness for $t\bar\cD(t,U)$.} Following a similar argument in Step 1--Step 2, we can obtain
\begin{equation*}
\begin{aligned}
&\sqrt{t}\Big|\zeta r^\frac{m}{2} \big(D_{\bar\eta}^2 U_t,D_{\bar\eta}(\frac{U_t}{\bar\eta})\big)\Big|_2\leq C\sqrt{t}(|\zeta r^\frac{m}{2}U_{tt}|_2+\mathfrak{p}(c_1)),\\
&\sqrt{t}\Big|\zeta r^\frac{m}{2} \Big(D_{\bar\eta}^4 U,D_{\bar\eta}^3(\frac{U}{\bar\eta}),D_{\bar\eta}\big(\frac{1}{\bar\eta}D_{\bar\eta}(\frac{U}{\bar\eta})\big)\Big)\Big|_2\leq C\sqrt{t}\Big(\Big|\zeta r^\frac{m}{2} \big(D_{\bar\eta}^2 U_t,D_{\bar\eta}(\frac{U_t}{\bar\eta})\big)\Big|_2+\mathfrak{p}(c_0)\Big),\\
&\sqrt{t}\big|\chi^\sharp\rho_0^{(\frac{3}{2}-\varepsilon_0)\beta}D_{\bar\eta}^2 U_t\big|_2\leq C\sqrt{t}\big(\mathfrak{p}(c_0)\big|\chi^\sharp\rho_0^\frac{1}{2}U_{tt}\big|_2+\mathfrak{p}(c_1)\big),\\
&\sqrt{t}\big|\chi^\sharp\rho_0^{\left(\frac{3}{2}-\varepsilon_0\right)\beta}D_{\bar\eta}^4U\big|_2\leq C\sqrt{t}\mathfrak{p}(c_0)\big(\big|\chi^\sharp \rho_0^{(\frac{3}{2}-\varepsilon_0)\beta} D_{\bar\eta}^2 U_t\big|_2+\bar\cD_{\mathrm{in}}(t,U)^\frac{1}{2}+1\big).
\end{aligned}
\end{equation*}

Therefore, this, combined with Lemma \ref{c_1}, leads to the desired result of this lemma.
\end{proof}

\begin{Lemma}\label{lemma-12.11}
For any $0\leq t\leq T_4=\min\{T_3,(2C\mathfrak{p}(c_0))^{-1}\}$,
\begin{equation}\label{39''-0}
\cE(t,U)+t\cD(t,U)+\int_0^{t} \cD(s,U)\,\mathrm{d}s\leq C\mathfrak{p}(c_0),\qquad|\eta_r(t)-1|_\infty+\Big|\frac{\eta(t)}{r}-1\Big|_\infty\leq \frac{1}{2}.
\end{equation}
\end{Lemma}
\begin{proof}
Collecting the  estimates established  in Lemmas \ref{c_0-c_1}--\ref{c_3}, we see that, for all $t\in[0,T_3]$,
\begin{equation}\label{39''}
\bar\cE(t,U)+t\bar\cD(t,U)+\int_0^{t} \bar\cD(s,U)\,\mathrm{d}s\leq C\mathfrak{p}(c_0).
\end{equation}

First, \eqref{given-flow}, together  with \eqref{39} and Lemma \ref{c_1-c_2}, yields that, for any $0\leq t\leq T_4:=\min\{T_3,(2C\mathfrak{p}(c_0))^{-1}\}$, 
\begin{equation}\label{eta-etar}
\begin{aligned}
|\eta_r(t)-1|_\infty+\Big|\frac{\eta(t)}{r}-1\Big|_\infty&\leq \int_0^t \Big|\big(U_r,\frac{U}{r}\big)\Big|_\infty \,\mathrm{d}s\\
&\leq C\int_0^t \Big|\big(D_{\bar\eta}U,\frac{U}{\bar\eta}\big)\Big|_\infty \,\mathrm{d}s\leq C\mathfrak{p}(c_0)T_4\leq \frac{1}{2}.
\end{aligned}
\end{equation}

Next, let $(\mathring{\cE},\mathring{\cD})(t,f)$ be defined in the same way as $(\cE,\cD)(t,f)$ in \eqref{E-1} and \eqref{D-1}, respectively, except for letting $\eta(r)=r$. Then,  based on \eqref{39''}--\eqref{eta-etar}, we obtain from Lemma \ref{lemma-gaowei} that, 
for all $t\in[0,T_4]$, 
\begin{equation*}
\begin{aligned}
&\sup_{s\in[0,t]}\big(\cE(s,U)+s\cD(s,U)\big)+\int_0^{t} \cD(s,U)\,\mathrm{d}s\\
&\qquad\leq C\mathfrak{p}(c_0)\Big(\sup_{s\in[0,t]}\big(\mathring\cE(s,U)+s\mathring\cD(s,U)\big)+\int_0^{t} \mathring\cD(s,U)\,\mathrm{d}s\Big)\\
&\qquad\leq C\mathfrak{p}(c_0)\Big(\sup_{s\in[0,t]}\big(\bar\cE(s,U)+s\bar\cD(s,U)\big)+\int_0^{t} \bar\cD(s,U)\,\mathrm{d}s \Big),
\end{aligned}
\end{equation*}
which thus leads to $\eqref{39''-0}_1$.
\end{proof}

Hence, defining constants $(c_1,T^*)$ as 
\begin{equation}\label{3.84}
c_1:= C\mathfrak{p}(c_0),\qquad
T^*:=T_4=\min\{T_3,(2C\mathfrak{p}(c_0))^{-1}\},
\end{equation}
then we obtain from Lemma \ref{lemma-12.11} that, for all $t\in[0,T^*]$,
\begin{equation}\label{uniform bounds}
\cE(t,U)+t\cD(t,U)+\int_0^{t} \cD(s,U)\,\mathrm{d}s\leq c_1,\qquad |\eta_r(t)-1|_\infty+\Big|\frac{\eta(t)}{r}-1\Big|_\infty \leq \frac{1}{2}.
\end{equation}

\subsection{Local-in-time existence and uniqueness  of the nonlinear problem}\label{subsection9.3}

In this section, we prove the local well-posedness of classical solutions of \eqref{eq:VFBP-La-eta} stated in Theorem \ref{local-Theorem1.1}. 

For convenience, in the rest of \S \ref{subsection9.3}, we let $(\mathring{\cE},\mathring{\cD})(t,f)$ and $(\cE_\mathrm{k}, \cD_\mathrm{k})(t,f)$ be defined in the same way as $(\cE,\cD)(t,f)$ in \eqref{E-1} and \eqref{D-1}, except with $\eta(r)$ replaced by $r$ and $\eta^\mathrm{k}$, respectively, where $\eta^\mathrm{k}$ denotes the $\mathrm{k}$-th generation of the iterative sequence that will be given later. 

\begin{proof}
We divide the proof into the following four steps.

\smallskip
\textbf{1. Construction of the iterative sequence.} Let $c_0$ be given as in \eqref{38} and let $(U^0,\eta^0)$ solve the following linear parabolic system in $(0,T]\times I$:
\begin{equation}\label{eqeq1}
\begin{cases}
\displaystyle r^m\rho_0 U^0_t-2\mu(r^m\rho_0 U^0_r)_r+2\mu m r^{m-2}\rho_0 U^0=-A(r^m\rho_0^\gamma)_r+Am r^{m-1}\rho_0^\gamma, \\
\eta^0_t=U^0 ,\\
(U^0, \eta^0)(0,r)=(u_0(r),r) \quad  \text{for $r\in \bar I$}.
\end{cases}
\end{equation}
Notice that the linear equation $\eqref{eqeq1}_1$ is nothing but $\eqref{lp}_1$ with $(\bar\eta,\bar\eta_r)$ replaced by $(r,1)$. Moreover, $(U^0_t,U^0_{tr})|_{t=0}$ satisfy the same compatibility condition as \eqref{116} in Appendix \ref{AppB}. Thus, following the same argument as used in \S \ref{Subsection9.1}--\S\ref{subsection9.2}, the linear problem \eqref{eqeq1} admits a classical solution $(U^0,\eta^0)$, satisfying the boundary condition $U_r^0|_{r=1}=0$; and we can find a small positive time $T^\prime\leq T^*$, with $T^*$ defined in \eqref{3.84}, such that,  for all $t\in[0,T']$,
\begin{equation}\label{395}
\cE_0(t,U^0)+t\cD_0(t,U^0)+\int_0^{t} \cD_0(s,U^0)\,\mathrm{d}s\leq c_1,\quad |\eta_r^0(t)-1|_\infty+\Big|\frac{\eta^0(t)}{r}-1\Big|_\infty \leq \frac{1}{2}.
\end{equation}

Next, set $\bar \eta=\eta^0$ in \eqref{lp}. By Lemma \ref{existence-linearize} and \eqref{given-flow}, problem \eqref{lp} admits a unique classical solution $(U^1,\eta^1)$ in $[0,T']\times \bar I$. Certainly, it follows from \eqref{uniform bounds} that $(U^1,\eta^1)$ also satisfies the following uniform estimates on $[0,T']$:
\begin{equation}\label{395'''}
\cE_1(t,U^1)+t\cD_1(t,U^1)+\int_0^{t} \cD_1(s,U^1)\,\mathrm{d}s\leq c_1,\quad |\eta_r^1(t)-1|_\infty+\Big|\frac{\eta^1(t)}{r}-1\Big|_\infty \leq \frac{1}{2}.
\end{equation}

Hence, the approximate sequence $(U^{\mathrm{k}+1},\eta^{\mathrm{k}+1})$ ($\mathrm{k}\in \NN^*$) can be constructed iteratively as follows: Given $(U^\mathrm{k},\eta^\mathrm{k})$, define $(U^{\mathrm{k}+1},  \eta^{\mathrm{k}+1})$ by solving the following problem in $(0,T']\times I$: 
\begin{equation}\label{396}
\begin{cases}
\displaystyle r^m\rho_0U^{\mathrm{k}+1}_t-2\mu \big(r^m\rho_0\frac{U^{\mathrm{k}+1}_r}{(\eta^\mathrm{k}_r)^2}  \big)_r +2\mu mr^m\rho_0\frac{U^{\mathrm{k}+1}}{(\eta^\mathrm{k})^2}\\[8pt]
\qquad\hspace{1mm} \displaystyle =-A ((\eta^\mathrm{k})^m (\varrho^\mathrm{k})^\gamma)_r+Am (\eta^\mathrm{k})^{m-1}\eta^\mathrm{k}_r (\varrho^\mathrm{k})^\gamma,\\[4pt]
\displaystyle \eta^{\mathrm{k}+1}_t=U^{\mathrm{k}+1},\\[4pt]
(U^{\mathrm{k}+1},\eta^{\mathrm{k}+1})(0,r)=(u_0(r),r)\qquad \text{for $r\in I$},
\end{cases}
\end{equation}
where $\varrho^\mathrm{k}$ is given by 
\begin{equation*}
\varrho^\mathrm{k}=\frac{r^m\rho_0}{(\eta^\mathrm{k})^m\eta^\mathrm{k}_r}.
\end{equation*}
By \eqref{uniform bounds}, we obtain an iterative solution sequence $(U^\mathrm{k},\eta^\mathrm{k})$ satisfying \eqref{395}, that is, for all $t\in[0,T']$ and $\mathrm{k}\in \NN$, 
\begin{equation}\label{395k}
\cE_\mathrm{k}(t,U^\mathrm{k})+t\cD_\mathrm{k}(t,U^\mathrm{k})+\int_0^{t} \cD_\mathrm{k}(s,U^\mathrm{k})\,\mathrm{d}s\leq c_1,\quad |\eta_r^\mathrm{k}(t)-1|_\infty+\Big|\frac{\eta^\mathrm{k}(t)}{r}-1\Big|_\infty \leq \frac{1}{2}.
\end{equation}
Of course, by Lemma \ref{lemma-gaowei}, \eqref{395k} also implies
\begin{equation}\label{395k-r}
\mathring\cE(t,U^\mathrm{k})+t\mathring\cD(t,U^\mathrm{k})+\int_0^{t} \mathring\cD(s,U^\mathrm{k})\,\mathrm{d}s\leq C\mathfrak{p}(c_1),\quad |\eta_r^\mathrm{k}(t)-1|_\infty+\Big|\frac{\eta^\mathrm{k}(t)}{r}-1\Big|_\infty \leq \frac{1}{2}.
\end{equation}

\smallskip
\textbf{2. Convergence of $(U^\mathrm{k},\eta^\mathrm{k})$.}
Define
\begin{equation*}
\begin{aligned}
\wU^{\mathrm{k}+1}&:=U^{\mathrm{k}+1}-U^\mathrm{k},\qquad &&\weta^{\mathrm{k}+1}:=\eta^{\mathrm{k}+1}-\eta^\mathrm{k}=\int_0^t \wU^{\mathrm{k}+1}(s,r)\,\ds,\\
\Phi^\mathrm{k}&:=(\varrho^\mathrm{k})^{\gamma-1},\qquad &&\widehat \Phi^\mathrm{k+1}\!:=\Phi^\mathrm{k+1}-\Phi^\mathrm{k},
\end{aligned}
\end{equation*}
and introduce the following energy function:
\begin{equation*}
\begin{aligned}
\widehat E_\mathrm{k}(t)&:=\sup_{s\in[0,t]}\big|(r^m\rho_0)^\frac{1}{2} \wU^{\mathrm{k}}\big|_2^2+\int_0^t\Big|(r^m\rho_0)^\frac{1}{2}\big(\wU_r^{\mathrm{k}},\frac{\wU^{\mathrm{k}}}{r}\big)\Big|_2^2\,\ds.   
\end{aligned}
\end{equation*}

Then, based on \eqref{396}, the problem of $(\wU^{\mathrm{k}+1},\weta^{\mathrm{k}+1})$ can be written as 
\begin{equation}\label{398}
\!\!\begin{cases}
\displaystyle r^m\rho_0\wU^{\mathrm{k}+1}_t-2\mu \big(r^m\rho_0\frac{\wU^{\mathrm{k}+1}_r}{(\eta^\mathrm{k}_r)^2}  \big)_r+2\mu mr^m\rho_0\frac{\wU^{\mathrm{k}+1}}{(\eta^\mathrm{k})^2}=\big((r^m\rho_0)^\frac{1}{2}\frac{\mathfrak{q}_1^\mathrm{k}}{\eta^\mathrm{k}_r}\big)_r+(r^m\rho_0)^\frac{1}{2}\frac{\mathfrak{q}_2^\mathrm{k}}{\eta^\mathrm{k}},\\[8pt]
\weta^{\mathrm{k}+1}_t=\wU^{\mathrm{k}+1},\\[4pt]
(\wU^{\mathrm{k}+1},\weta^{\mathrm{k}+1})(0,r)=(0,0) \qquad\text{for }r\in I,
\end{cases}
\end{equation}
where
\begin{equation*}
\begin{aligned}
\mathfrak{q}_1^\mathrm{k}&:=(r^m\rho_0)^\frac{1}{2} \Big(2\mu \frac{U_{r}^\mathrm{k}}{\eta^\mathrm{k}_r}\big(1-\frac{(\eta^\mathrm{k}_r)^2}{(\eta^{\mathrm{k}-1}_r)^{2}}\big) -A  \big(\widehat \Phi^\mathrm{k}+\Phi^\mathrm{k-1} (1 - \frac{\eta^\mathrm{k}_r}{\eta^\mathrm{k-1}_r})\big)\Big),\\
\mathfrak{q}_2^\mathrm{k}&:=(r^m\rho_0)^\frac{1}{2}\Big(-2\mu m \frac{U^{\mathrm{k}}}{\eta^\mathrm{k}}\big(1-\frac{(\eta^\mathrm{k})^2}{(\eta^{\mathrm{k}-1})^{2}}\big)+Am \big( \widehat \Phi^\mathrm{k}+ \Phi^\mathrm{k-1}(1-\frac{\eta^\mathrm{k}}{\eta^\mathrm{k-1}})\big)\Big).
\end{aligned}
\end{equation*}

\smallskip
\textbf{2.1. Estimates for $(\mathfrak{q}_1^\mathrm{k},\mathfrak{q}_2^\mathrm{k})$.} Since
\begin{align*}
&\begin{aligned}
&\ \big(\frac{\eta^\mathrm{k}}{\eta^\mathrm{k-1}}-1\big)_t=\frac{\wU^\mathrm{k}}{\eta^\mathrm{k-1}}-\frac{U^\mathrm{k-1}}{\eta^\mathrm{k-1}}\big(\frac{\eta^\mathrm{k}}{\eta^\mathrm{k-1}}-1\big)\\
&\implies\frac{\eta^\mathrm{k}}{\eta^\mathrm{k-1}}=1+\int_0^{t} \exp\Big(-\int_s^t\frac{U^\mathrm{k-1}}{\eta^\mathrm{k-1}}\,\mathrm{d}\tau\Big)\frac{\wU^\mathrm{k}}{\eta^\mathrm{k-1}}\,\mathrm{d}s,
\end{aligned}
\end{align*}
it follows from \eqref{395k} that, for all $t\in[0,T']$,
\begin{equation}\label{395kk}
\Big|\frac{\eta^\mathrm{k}}{\eta^\mathrm{k-1}}-1\Big|+\Big|\big(\frac{\eta^\mathrm{k}}{\eta^\mathrm{k-1}}\big)^2-1\Big|\leq Ce^{C\mathfrak{p}(c_1)t}\int_0^{t} \Big|\frac{\wU^\mathrm{k}}{r}\Big|\,\mathrm{d}s.
\end{equation}
Similarly, we can also obtain
\begin{equation}\label{395kk'}
\Big|\frac{\eta_r^\mathrm{k}}{\eta_r^\mathrm{k-1}}-1\Big|+\Big|\big(\frac{\eta_r^\mathrm{k}}{\eta_r^\mathrm{k-1}}\big)^2-1\Big| \leq Ce^{C\mathfrak{p}(c_1)t}\int_0^{t} |\wU^\mathrm{k}_r|\,\mathrm{d}s.
\end{equation}
Moreover, since
\begin{equation*}
\begin{aligned}
\widehat \Phi^\mathrm{k}_t&=-(\gamma-1)\widehat \Phi^\mathrm{k}\big(\frac{U_r^\mathrm{k}}{\eta_r^\mathrm{k}}+\frac{mU^\mathrm{k}}{\eta^\mathrm{k}}\big)-(\gamma-1) \Phi^\mathrm{k-1}\Big(\frac{\wU_r^\mathrm{k}}{\eta_r^\mathrm{k}}+ \frac{U_r^\mathrm{k-1}}{\eta_r^\mathrm{k}}\big(1-\frac{\eta_r^\mathrm{k}}{\eta_r^\mathrm{k-1}}\big)\Big)\\
&\quad -m(\gamma-1) \Phi^\mathrm{k-1}\Big(\frac{\wU^\mathrm{k}}{\eta^\mathrm{k}}+ \frac{U^\mathrm{k-1}}{\eta^\mathrm{k}}\big(1-\frac{\eta^\mathrm{k}}{\eta^\mathrm{k-1}}\big)\Big),
\end{aligned}
\end{equation*}
it follows from \eqref{395k} and \eqref{395kk}--\eqref{395kk'} that, for all $t\in [0,T']$,
\begin{equation}\label{395kkk}
|\widehat \Phi^\mathrm{k}| \leq Ce^{C\mathfrak{p}(c_1)t}\int_0^{t} \Big(|\wU^\mathrm{k}_r|+\Big|\frac{\wU^\mathrm{k}}{r}\Big|\Big)\,\mathrm{d}s.
\end{equation}

Therefore, it follows from \eqref{395k}, \eqref{395kk}--\eqref{395kkk}, and the H\"older and Minkowski inequalities that, for all $t\in [0,T']$ and $\mathrm{k}\in \NN^*$, 
\begin{equation}\label{basic}
\begin{aligned}
|(\mathfrak{q}_1^\mathrm{k},\mathfrak{q}_2^\mathrm{k})(t)|_2^2
\leq C\mathfrak{p}(c_1)te^{C\mathfrak{p}(c_1)t} \int_0^t\Big|(r^m\rho_0)^\frac{1}{2}\big(\wU_r^{\mathrm{k}},\frac{\wU^{\mathrm{k}}}{r}\big)\Big|_2^2\,\ds.
\end{aligned}
\end{equation}

\textbf{2.2.}  Now, multiplying $\eqref{398}_1$ by $\wU^{\mathrm{k}+1}$ and integrating the resulting equality over $I$, we obtain from \eqref{395k}, \eqref{basic}, and the Young inequality that
\begin{equation*}
\begin{aligned}
&\,\frac{\mathrm{d}}{\dt}\big|(r^m\rho_0)^\frac{1}{2} \wU^{\mathrm{k}+1}\big|_2^2+\mu\Big|(r^m\rho_0)^\frac{1}{2}\big(\wU_r^{\mathrm{k+1}},\frac{\wU^{\mathrm{k+1}}}{r}\big)\Big|_2^2\leq C|(\mathfrak{q}_1^\mathrm{k},\mathfrak{q}_2^\mathrm{k})|_2^2\\
&\leq C\mathfrak{p}(c_1)te^{C\mathfrak{p}(c_1)t} \int_0^t\Big|(r^m\rho_0)^\frac{1}{2}\big(\wU_r^{\mathrm{k}},\frac{\wU^{\mathrm{k}}}{r}\big)\Big|_2^2\,\ds.
\end{aligned}
\end{equation*}
Integrating the above over $[0,t]$ leads to
\begin{equation*}
\widehat E_{\mathrm{k}+1}^\frac{1}{2}(t)\leq \sqrt{C\mathfrak{p}(c_1)}te^{\frac{1}{2}C\mathfrak{p}(c_1)t} \widehat E_\mathrm{k}^\frac{1}{2}(t) \qquad \text{for all $t\in [0,T^\prime]$ and $\mathrm{k}\in \NN^*$}.
\end{equation*}

Choosing $t=T_*$ in the above inequality such that
\begin{equation*}
\sqrt{C\mathfrak{p}(c_1)}T_*e^{\frac{1}{2}C\mathfrak{p}(c_1)T_*}\leq \frac{1}{2},\qquad T_*\leq T',
\end{equation*}
leads to 
\begin{equation}\label{total-bound}
\sum_{\mathrm{k}=1}^{\infty} \widehat E_\mathrm{k}^\frac{1}{2}(T_*) \leq \Big(\sum_{\mathrm{k}=0}^{\infty}2^{-\mathrm{k}}\Big) \widehat E_1^\frac{1}{2}(T_*)\leq \sqrt{C\mathfrak{p}(c_1)}<\infty.
\end{equation}

Then \eqref{total-bound} implies that $\{U^{\mathrm{k}}\}_{\mathrm{k}\in \NN}$ is a Cauchy sequence that converges to some limit $U$ as $\mathrm{k}\to\infty$ in the following sense:
\begin{equation*} 
\begin{aligned}
(r^m\rho_0)^\frac{1}{2} U^{\mathrm{k}}\to (r^m\rho_0)^\frac{1}{2} U &\qquad  \text{in } C([0,T_*];L^2),\\
(r^m\rho_0)^\frac{1}{2}\big(U^\mathrm{k}_r,\frac{U^\mathrm{k}}{r}\big)\to (r^m\rho_0)^\frac{1}{2}\big(U_r,\frac{U}{r}\big)&\qquad  \text{in } L^2([0,T_*];L^2),
\end{aligned}
\end{equation*}
which also leads to 
\begin{equation}\label{5.10}
U^{\mathrm{k}}\to U \qquad  \text{in } L^2([0,T_*];H^1_{\mathrm{loc}}) \qquad \text{as $\mathrm{k}\to \infty$}.
\end{equation}
On the other hand, by \eqref{395k-r} and $\rho_0^\beta\sim 1-r$,  for any $t\in [0,T_*]$ and $a\in (0,\frac{1}{2})$, 
\begin{equation}\label{5.11}
\begin{aligned}
\int_0^{T_*}\|U^\mathrm{k}\|_{H^4(a,1-a)}^2\,\mathrm{d}t\leq C(a)\int_0^{T_*} (\mathring\cE(t,U^{\mathrm{k}})+\mathring\cD(t,U^{\mathrm{k}}))\,\mathrm{d}t&\leq C(a)(1+T_*)\mathfrak{p}(c_1),\\
t\|U^\mathrm{k}\|_{H^4(a,1-a)}^2\leq C(a) t(\mathring\cE(t,U^{\mathrm{k}})+ \mathring\cD(t,U^{\mathrm{k}}))&\leq C(a)(1+T_*)\mathfrak{p}(c_1).
\end{aligned}
\end{equation}
Hence, it follows from \eqref{5.10}--\eqref{5.11} and Lemma \ref{GNinequality'} that 
\begin{equation}\label{H333}
U^\mathrm{k}\to U\qquad \text{in }L^2([0,T_*];H^3_\mathrm{loc}) \ \ \text{as $\mathrm{k}\to \infty$}.
\end{equation}

Now, based on \eqref{H333} and $\eqref{396}_2$, we have
\begin{equation}\label{H333'}
\eta^\mathrm{k}=r+\int_0^tU^\mathrm{k}\,\mathrm{d}s\to r+\int_0^tU \,\mathrm{d}s \qquad \text{in } C([0,T_*];H^3_\mathrm{loc}). 
\end{equation}
This also implies that $\{\eta^\mathrm{k}\}_{k\in\NN}$ is a Cauchy sequence in $C([0,T_*];H^3_\mathrm{loc})$ which converges to some limit $\eta$ so that
\begin{equation*}
\eta=r+\int_0^tU \,\mathrm{d}s \qquad \text{for \textit{a.e.} $(t,r)\in (0,T_*)\times (0,1)$}.
\end{equation*}

To recover equation $\eqref{eq:VFBP-La-eta}_1$, we first divide $\eqref{396}_1$ by $(\eta^\mathrm{k})^m\eta_r^\mathrm{k}$ to derive
\begin{equation*}
U^\mathrm{k+1}_t= \frac{1}{\varrho^\mathrm{k}\eta_r^\mathrm{k}} \Big(2\mu \varrho^\mathrm{k} \frac{U_r^\mathrm{k+1}}{\eta_r^\mathrm{k}}+ 2\mu m \varrho^\mathrm{k}\frac{U^\mathrm{k+1}}{\eta^\mathrm{k}} -A (\varrho^\mathrm{k})^{\gamma}\Big)_r- 2\mu m \frac{(\varrho^\mathrm{k})_r U^\mathrm{k+1}}{\varrho^\mathrm{k}\eta^\mathrm{k}\eta_r^\mathrm{k}}.
\end{equation*}
Then, due to Lemma \ref{sobolev-embedding}, $\rho_0^\beta\sim 1-r$, $\rho^\beta_0 \in H^3_{\mathrm{loc}}$, and \eqref{H333}--\eqref{H333'}, we take the limit as $k\to\infty$ in the above and obtain
\begin{equation*} 
U^\mathrm{k}_t\to \frac{1}{\varrho} D_\eta\Big(2\mu \Big(\varrho\big(D_\eta U+ \frac{mU}{\eta}\big)\Big) -A \varrho^{\gamma}\Big)- 2\mu m \frac{ D_\eta\varrho U}{\varrho\eta}\qquad \text{in } L^2([0,T_*];H^1_\mathrm{loc}),
\end{equation*}
which, along with the uniqueness of limits, implies that $\eqref{eq:VFBP-La-eta}_1$ holds for \textit{a.e.} $(t,r)\in (0,T_*)\times (0,1)$.  Moreover, by the lower semi-continuity of weak convergence and \eqref{395k-r}, we have
\begin{equation*}
\mathring\cE(t,U)+t\mathring\cD(t,U)+\int_0^{t} \mathring\cD(s,U)\,\mathrm{d}s\leq C\mathfrak{p}(c_1),\quad |\eta_r(t)-1|_\infty+\Big|\frac{\eta(t)}{r}-1\Big|_\infty \leq \frac{1}{2}.
\end{equation*}

This completes the proof of the existence.

\smallskip
\textbf{3. Uniqueness.} Let $(U^\mathsf{a},\eta^\mathsf{a})$ and $(U^\mathsf{b},\eta^\mathsf{b})$ be two solutions of \eqref{eq:VFBP-La-eta} in $(0,T_*)\times (0,1)$. Define 
\begin{equation*}
(\weta,\wU):=(\eta^\mathsf{b}-\eta^\mathsf{a},U^\mathsf{b}-U^\mathsf{a}),\quad \,\,
\widehat E(t):=\sup_{s\in[0,t]}\big|(r^m\rho_0)^\frac{1}{2} \wU\big|_2^2+\int_0^t\Big|(r^m\rho_0)^\frac{1}{2}\big(\wU_r,\frac{\wU}{r}\big)\Big|_2^2\,\ds.
\end{equation*}

It follows from  \eqref{eq:VFBP-La-eta} that  $(\wU,\weta)$ solves the following problem in $(0,T_*]\times I$:
\begin{equation*}
\begin{cases}
\displaystyle r^m\rho_0\wU_t-2\mu \big(r^m\rho_0\frac{\wU_r}{(\eta^\mathsf{b}_r)^2}  \big)_r+2\mu mr^m\rho_0\frac{\wU}{(\eta^\mathsf{b})^2}=\big((r^m\rho_0)^\frac{1}{2}\frac{\tilde{\mathfrak{q}}_1}{\eta^\mathsf{b}_r}\big)_r+(r^m\rho_0)^\frac{1}{2}\frac{\tilde{\mathfrak{q}}_2}{\eta^\mathsf{b}},\\[8pt]
\weta_t=\wU,\\[4pt]
(\wU,\weta)(0,r)=(0,0) \qquad\text{for }r\in I,
\end{cases}
\end{equation*}
where
\begin{equation*}
\begin{aligned}
\tilde{\mathfrak{q}}_1 &:=(r^m\rho_0)^\frac{1}{2} \Big(2\mu \frac{U_{r}^\mathsf{a}}{\eta^\mathsf{b}_r}\big(1-\frac{(\eta^\mathsf{b}_r)^2}{(\eta^{\mathsf{a}}_r)^{2}}\big) -A  \big(\widehat \Phi +\Phi^\mathsf{a} (1 - \frac{\eta^\mathsf{b}_r}{\eta^\mathsf{a}_r})\big)\Big),\\
\tilde{\mathfrak{q}}_2 &:=(r^m\rho_0)^\frac{1}{2}\Big(-2\mu m \frac{U^{\mathsf{a}}}{\eta^\mathsf{b}}\big(1-\frac{(\eta^\mathsf{b})^2}{(\eta^{\mathsf{a}})^{2}}\big)+Am \big(\widehat \Phi +\Phi^\mathsf{a} (1 - \frac{\eta^\mathsf{b}}{\eta^\mathsf{a}})\big)\Big),\\
\Phi^\mathsf{a}&:=\big(\frac{r^m\rho_0}{(\eta^\mathsf{a})^m\eta^\mathsf{a}_r}\big)^{\gamma-1},\quad \Phi^\mathsf{b}  :=\big(\frac{r^m\rho_0}{(\eta^\mathsf{b})^m\eta^\mathsf{b}_r}\big)^{\gamma-1},\quad \widehat \Phi :=\Phi^\mathsf{b}-\Phi^\mathsf{a}. 
\end{aligned}
\end{equation*}

Similarly, we can show that $(\tilde{\mathfrak{q}}_1,\tilde{\mathfrak{q}}_2)$ satisfy \eqref{basic} with $\wU^{\mathrm{k}}$ replaced by $\wU$. Hence, by the same arguments as in Step 2.2, we have 
\begin{equation*}
\frac{\mathrm{d}}{\dt}\big|(r^m\rho_0)^\frac{1}{2} \wU\big|_2^2+\mu\Big|(r^m\rho_0)^\frac{1}{2}\big(\wU_r,\frac{\wU}{r}\big)\Big|_2^2\leq C\mathfrak{p}(c_1)te^{C\mathfrak{p}(c_1)t} \int_0^t\Big|(r^m\rho_0)^\frac{1}{2}\big(\wU_r,\frac{\wU}{r}\big)\Big|_2^2\,\ds,
\end{equation*}
which, together  with the Gr\"onwall inequality, yields that  $\widehat E(t) \equiv 0$, \textit{i.e.}, $U^\mathsf{a}\equiv U^\mathsf{b}$.

\smallskip
\textbf{4. $(U,\eta)$ is classical satisfying \eqref{b1-lo}.} First, the regularity of $U$ and \eqref{b1-lo} can be proved by the same argument as in Steps 6--7 in \S \ref{subsection3.3}. Then the regularity of $\eta$ follows directly from the formula: $\eta_t=U$. Finally, following a similar argument in Step 8 in \S \ref{subsection3.3}, we can show that $\eqref{eq:VFBP-La-eta}_1$ holds pointwise in $(0,T_*]\times \bar I$. 

This completes the proof of Theorem \ref{local-Theorem1.1}.
\end{proof}

\subsection{Local-In-Time  existence  and uniqueness with a general initial flow map}\label{subsection9.4}

We prove the extension stated in Remark \ref{remk31}. Fix $\tau>0$ and let $(u_\star,\eta_\star)$ satisfy \eqref{general-initial-assumptions}. $(U_t,D_{\eta_\star} U_t)|_{t=\tau}$ are determined by $\eqref{general-initial-problem}_1$ and its spatial and temporal derivatives at $(U,\eta)=(u_\star,\eta_\star)$, with $\eta_t=u_\star$.

Moreover, the compatibility argument in the proof of Lemma \ref{mewman}, applied to \eqref{reform} at $t=\tau$, gives $(u_\star)_r(1)=0$, with the additional coefficient terms controlled by $\mathring\cE(\eta_\star)$; also, $u_\star(0)=0$ follows from the continuity of $\frac{u_\star}{r}$.

\begin{proof}[Proof of Remark \ref{remk31}]
We adapt the construction in \S\ref{Subsection9.1}--\S\ref{subsection9.3}. 

\smallskip
\textbf{1. Linearization.} Likewise in \S \ref{Subsection9.1}, we consider the linearized problem \eqref{lp} on $(\tau,\tau+T]\times I$ with $U|_{t=\tau}=u_\star$, where the corresponding flow map $\bar\eta$ is defined by
\begin{equation}\label{general-given-flow}
\bar\eta(t,r)=\eta_\star(r)+\int_\tau^t\bar U(s,r)\,\mathrm{d}s,
\end{equation}
and $\bar U$ is a given function satisfying $\bar U|_{t=\tau}=u_\star$ for $r\in I$ and \eqref{given} with the time interval from $0$ to $T$ replaced by $\tau$ to $\tau+T$ and $t\bar \cD(t,\bar U)\in L^\infty(0,T)$ replaced by $(t-\tau)\bar \cD(t,\bar U)\in L^\infty(\tau,\tau+T)$. Similarly, \eqref{given}, combined with \eqref{general-given-flow} and \eqref{general-initial-assumptions}, yields that $\bar\eta$ satisfies \eqref{given-bareta} with the time interval from $\tau$ to $\tau+T$. Moreover, we assume that 
\begin{equation}
\frac{\delta_*}{2}\leq\big(\bar\eta_r,\frac{\bar\eta}{r}\big)\leq \frac{3\delta^*}{2}\quad\text{on }[\tau,\tau+T]\times\bar I,
\end{equation}
which is the counterpart of \eqref{jibenjiashe}. 

Since the change in the initial value of $\bar\eta$ does not affect the modified Galerkin scheme used in \S \ref{Subsection9.1}, we can therefore follow the rest of \S \ref{Subsection9.1} and obtain the global well-posedness result for the linearized problem:
\begin{Lemma}\label{existence-linearize-tau}
Let $n=2$ or $3$ and $\gamma\in (\frac{4}{3},\infty)$. Assume that $\rho_0(r)$ satisfies \eqref{distance-la} for some $\beta\in (\frac{1}{3},\gamma-1]$ and $u_\star(r)$ satisfies $\cE(\tau,U)<\infty$. Then, for any $T>0$, the linearized problem \eqref{lp} with $U|_{t=\tau}=u_\star$ admits a unique classical solution $U$ in $[\tau,\tau+T]\times \bar I$ satisfying 
\begin{equation*}
\begin{aligned}
&\bar\cE(t,U)+(t-\tau) \,\bar\cD(t,U)\in L^\infty(\tau,\tau+T),\qquad \bar\cD (t,U)\in L^1(\tau,\tau+T),\\[4pt]
&U|_{r=0}=U_r|_{r=1}=0\qquad\qquad  \text{for $t\in[\tau,\tau+T]$}.
\end{aligned}
\end{equation*}
\end{Lemma}
It should be noted here that we use Lemma \ref{lemma-gaowei} in the original proof of Step 6 in \S \ref{subsection3.3}. Fortunately, since all the initial values of the flow map used in the proof of Lemma \ref{lemma-gaowei} are bounded by $(\delta_*,\delta^*,\mathfrak{M}_0)$, we can follow a similar argument to show that Lemma \ref{lemma-gaowei} still holds for $t\in [\tau,\tau+T]$.

\smallskip
\textbf{2. Uniform estimates for the linearized problem.}
We adapt the proof in \S\ref{subsection9.2} in this step. In what follows, $C$ denotes a generic positive constant, depending only on $(\delta_*,\delta^*,n,\mu,A,\gamma,\beta,\varepsilon_0)$ and independent of $\tau$, and $\mathfrak p(s)=\sum_{j=1}^k s^j$ with some $k\in \NN^*$ denotes a generic polynomial. We choose $c_0>1$, independently of $\tau$, so that
\begin{equation}\label{ccc0}
1+\cK_1+\cK_1^{-1}+\cK_2+\|\rho_0^\beta\|_{H^3(\Omega)}+\mathfrak M_0\leq c_0.
\end{equation}
Then fix $T>0$, and assume that there exist some constants $T^*\in (0,T]$ and  $c_1$ such that $1<c_0\leq c_1$ and, for all $t\in [\tau,\tau+T^*]$, 
\begin{equation}\label{39-tau}
\bar\cE(t,\bar U)+(t-\tau)\bar\cD(t,\bar U)+\int_\tau^{t}\bar\cD(s,\bar U)\,\mathrm{d}s\leq c_1,\quad
\big|\bar\eta_r(t)-(\eta_\star)_r\big|_\infty+\Big|\frac{\bar\eta(t)}{r}-\frac{\eta_\star}{r}\Big|_\infty\leq \frac{\delta_*}{2}.
\end{equation}
Here, $(c_1,T^*)$ will be determined later, which depend only on $(\delta_*,\delta^*,c_0,n,\mu,A,\gamma,\beta,\varepsilon_0,T)$ and are independent of $\tau$.

First, Lemma \ref{lemma-useful1} still holds for $t\in [\tau,\tau+T^*]$. Moreover, in \eqref{1189}, we retain the initial derivatives of $\bar{\mathscr J}|_{t=\tau}=(\frac{\eta_\star}{r})^m(\eta_\star)_r$. These are bounded by $C\mathfrak p(c_0)$. Thus, applying the same argument as in the proof of Lemma \ref{lemma-Lambda-guji} with the time from $\tau$ to $t$, we can show that Lemma \ref{lemma-Lambda-guji} holds for $t\in [\tau,\tau+\widetilde T_1]$ with some $\widetilde T_1:=\min\{T^*,\mathfrak{p}(c_1)^{-1}\}$.

Next,  we can obtain the counterparts of Lemmas \ref{c_0-c_1}--\ref{c_3} without major modifications, that is, for all $t\in[\tau,\tau+\widetilde T_3]$ with $\widetilde T_3$ defined in a way similar to $T_3$,
\begin{equation}\label{39''-tau}
\bar\cE(t,U)+(t-\tau)\bar\cD(t,U)+\int_\tau^{t} \bar\cD(s,U)\,\mathrm{d}s\leq C\mathfrak{p}(c_0).
\end{equation}

Consequently, based on \eqref{39''-tau}, we can follow the proof of Lemma \ref{lemma-gaowei} and Lemma \ref{lemma-12.11} to obtain that, for any $\tau\leq t\leq \tau+\widetilde T_4$ with $\widetilde T_4:=\min\{\widetilde T_3,(2C\mathfrak{p}(c_0))^{-1}\delta_*\}$,
\begin{equation}\label{eta-etar-tau}
\begin{aligned}
&\big|\eta_r(t)-(\eta_\star)_r\big|_\infty+\Big|\frac{\eta(t)}{r}-\frac{\eta_\star}{r}\Big|_\infty\leq C\int_\tau^t \Big|\big(D_{\bar\eta}U,\frac{U}{\bar\eta}\big)\Big|_\infty \,\mathrm{d}s\leq C\mathfrak{p}(c_0)\widetilde T_4\leq \frac{\delta_*}{2},\\
&\cE(t,U)+(t-\tau)\cD(t,U)+\int_\tau^{t}\cD(s,U)\,\mathrm{d}s\leq C\mathfrak{p}(c_0).
\end{aligned}
\end{equation}
Hence, defining constants $(c_1,T^*)$ as 
\begin{equation}\label{3.84-tau}
c_1:= C\mathfrak{p}(c_0),\qquad
T^*:=\widetilde T_4,
\end{equation}
then we obtain from \eqref{eta-etar-tau} that, for all $t\in[\tau,\tau+T^*]$,
\begin{equation}\label{uniform bounds-tau}
\cE(t,U)+(t-\tau)\cD(t,U)+\int_\tau^{t} \cD(s,U)\,\mathrm{d}s\leq c_1,\quad \big|\eta_r(t)-(\eta_\star)_r\big|_\infty+\Big|\frac{\eta(t)}{r}-\frac{\eta_\star}{r}\Big|_\infty \leq \frac{\delta_*}{2}.
\end{equation}

\smallskip
\textbf{3. Local-in-time existence and uniqueness  of the nonlinear problem.} Let $U^0$ solve the following linear  problem in $(\tau,\tau+T]\times I$:
\begin{equation}\label{eqeq1-tau}
\begin{cases}
\displaystyle
r^m\rho_0U_t^0
-2\mu\Big(
\frac{r^m\rho_0U_r^0}{(\eta_\star)_r^2}
\Big)_r
+\frac{2\mu m r^m\rho_0U^0}{\eta_\star^2} =-A(\eta_\star^m\rho_\star^\gamma)_r
+Am\eta_\star^{m-1}(\eta_\star)_r\rho_\star^\gamma,
\\[5pt]
U^0(\tau,r)=u_\star(r),
\end{cases}
\end{equation}
where $\rho_\star =\frac{r^m\rho_0}{\eta_\star^m(\eta_\star)_r}$, and define
\begin{equation}
\eta^0(t,r)=\eta_\star(r)+\int_\tau^tU^0(s,r)\,\mathrm{d}s.
\end{equation}
Following the same argument as used in Steps 1--2, the linear problem \eqref{eqeq1-tau} admits a classical solution $(U^0,\eta^0)$, satisfying the boundary condition $U_r^0|_{r=1}=0$; then, for given $c_0$ as in \eqref{ccc0} and $c_1$ defined in \eqref{3.84-tau}, we can find a small positive time $T^\prime\leq T^*$, with $T^*$ defined in \eqref{3.84-tau}, such that,  for all $t\in [\tau,\tau+T']$,
\begin{equation}\label{uniform bounds-tau3}
\cE_0(t,U^0)+(t-\tau)\cD_0(t,U^0)+\int_\tau^{t} \cD_0(s,U^0)\,\mathrm{d}s\leq c_1, \ \ \big|\eta^0_r-(\eta_\star)_r\big|_\infty+\Big|\frac{\eta^0}{r}-\frac{\eta_\star}{r}\Big|_\infty \leq \frac{\delta_*}{2},
\end{equation}
where we adapt the notations $(\cE_\mathrm{k},\cD_\mathrm{k})(t,f)$ as in \S \ref{subsection9.3}.  

Consequently, for $\mathrm{k}\geq0$, once $(U^{\mathrm{k}},\eta^{\mathrm{k}})$ is known, we define $(U^{\mathrm{k}+1},\eta^{\mathrm{k}+1})$ by
\begin{equation}\label{general-flow-iteration}
\begin{cases}
\displaystyle r^m\rho_0U_t^{\mathrm{k}+1}-2\mu\Big(\frac{r^m\rho_0U_r^{\mathrm{k}+1}}{(\eta_r^{\mathrm{k}})^2}\Big)_r+\frac{2\mu m r^m\rho_0U^{\mathrm{k}+1}}{(\eta^{\mathrm{k}})^2}\\
\qquad =-A\big((\eta^{\mathrm{k}})^m(\varrho^{\mathrm{k}})^\gamma\big)_r+Am(\eta^{\mathrm{k}})^{m-1}\eta_r^{\mathrm{k}}(\varrho^{\mathrm{k}})^\gamma,\\
\displaystyle \eta^{\mathrm{k}+1}(t,r)=\eta_\star(r)+\int_\tau^tU^{\mathrm{k}+1}(s,r)\,\mathrm{d}s,\qquad \varrho^{\mathrm{k}}=\frac{r^m\rho_0}{(\eta^{\mathrm{k}})^m\eta_r^{\mathrm{k}}},\\
(U^{\mathrm{k}+1},\eta^{\mathrm{k}+1})(\tau,r)=(u_\star(r),\eta_\star(r)).
\end{cases}
\end{equation}
The well-posedness of the $(\mathrm{k}+1)$-th generation $(U^{\mathrm{k+1}},\eta^{\mathrm{k+1}})$ is ensured by Lemma \ref{existence-linearize-tau}. Then, following the same argument as in Step 1 in \S \ref{subsection9.3}, \eqref{uniform bounds-tau}, and the modified Lemma \ref{lemma-gaowei}, we obtain a sequence $(U^\mathrm{k},\eta^\mathrm{k})$ satisfying, for all $t\in[\tau,\tau+T']$ and $\mathrm{k}\in \NN$, 
\begin{equation}\label{395k-tau}
\begin{aligned}
&\cE_\mathrm{k}(t,U^\mathrm{k})+(t-\tau)\cD_\mathrm{k}(t,U^\mathrm{k})+\int_\tau^{t} \cD_\mathrm{k}(s,U^\mathrm{k})\,\mathrm{d}s\leq c_1,\\
&\big|\eta^{\mathrm{k}}_r-(\eta_\star)_r\big|_\infty +\Big|\frac{\eta^{\mathrm{k}}}{r}-\frac{\eta_\star}{r}\Big|_\infty \leq \frac{\delta_*}{2},\\
&\mathring\cE(t,U^\mathrm{k})+(t-\tau)\mathring\cD(t,U^\mathrm{k})+\int_\tau^{t} \mathring\cD(s,U^\mathrm{k})\,\mathrm{d}s\leq C\mathfrak{p}(c_1). 
\end{aligned}
\end{equation}
Moreover, by $\eqref{general-flow-iteration}_2$ and the Gr\"onwall inequality, we can also obtain 
\begin{equation}\label{general-flow-uniform-map}
\mathring\cE(\eta^{\mathrm{k}}(t))\leq C\mathfrak p(c_0)\quad \text{for all $t\in[\tau,\tau+T']$ and $\mathrm{k}\in \NN$}.
\end{equation}

Next, we derive the convergence of the approximate solutions. Set
\begin{equation*}
\wU^{\mathrm{k}+1}=U^{\mathrm{k}+1}-U^{\mathrm{k}},\qquad \weta^{\mathrm{k}+1}=\eta^{\mathrm{k}+1}-\eta^{\mathrm{k}}=\int_\tau^t\wU^{\mathrm{k}+1}(s,r)\,\mathrm{d}s,
\end{equation*}
and use the difference energy from Step 2 of \S\ref{subsection9.3}, now on $[\tau,t]$:
\begin{equation*}
\widehat E_{\mathrm{k}}(t)=\sup_{\tau\leq s\leq t}\big|(r^m\rho_0)^{\frac{1}{2}}\wU^{\mathrm{k}}(s)\big|_2^2+\int_\tau^t\Big|(r^m\rho_0)^{\frac{1}{2}}\big(\wU_r^{\mathrm{k}},\frac{\wU^{\mathrm{k}}}{r}\big)\Big|_2^2\,\mathrm{d}s.
\end{equation*}
For $\mathrm{k}\geq1$, subtracting two consecutive equations in \eqref{general-flow-iteration} gives $\eqref{398}_1$, with zero initial data at $t=\tau$ and the same expressions for $(\mathfrak q_1^{\mathrm{k}},\mathfrak q_2^{\mathrm{k}})$. All iterates have the same initial map. Therefore, the initial values of the quotients in \eqref{395kk}--\eqref{395kkk} are still $1$, and $\widehat\Phi^{\mathrm{k}}|_{t=\tau}=0$.

Notice that
\begin{equation*}
\Big|(r^m\rho_0)^{\frac{1}{2}}\big(\weta_r^{\mathrm{k}},\frac{\weta^{\mathrm{k}}}{r}\big)(t)\Big|_2^2\leq(t-\tau)\int_\tau^t\Big|(r^m\rho_0)^{\frac{1}{2}}\big(\wU_r^{\mathrm{k}},\frac{\wU^{\mathrm{k}}}{r}\big)\Big|_2^2\,\mathrm{d}s.
\end{equation*}
Combining this with \eqref{395k-tau}--\eqref{general-flow-uniform-map}, the quotient calculations leading to \eqref{basic} give
\begin{equation}\label{general-flow-remainder-estimate}
\big|(\mathfrak q_1^{\mathrm{k}},\mathfrak q_2^{\mathrm{k}})(t)\big|_2^2\leq C\mathfrak p(c_0)(t-\tau)e^{C\mathfrak p(c_0)(t-\tau)}\int_\tau^t\Big|(r^m\rho_0)^{\frac{1}{2}}\big(\wU_r^{\mathrm{k}},\frac{\wU^{\mathrm{k}}}{r}\big)\Big|_2^2\,\mathrm{d}s.
\end{equation}

Multiplying the difference equation by $\wU^{\mathrm{k}+1}$, integrating over $I$, and using \eqref{general-flow-remainder-estimate} and the Young inequality, we obtain
\begin{equation*}
\frac{\mathrm{d}}{\mathrm{d}t}\big|(r^m\rho_0)^{\frac{1}{2}}\wU^{\mathrm{k}+1}\big|_2^2+\frac{1}{C}\Big|(r^m\rho_0)^{\frac{1}{2}}\big(\wU_r^{\mathrm{k}+1},\frac{\wU^{\mathrm{k}+1}}{r}\big)\Big|_2^2\leq C\big|(\mathfrak q_1^{\mathrm{k}},\mathfrak q_2^{\mathrm{k}})\big|_2^2.
\end{equation*}
It follows by integration over $[\tau,t]$ that
\begin{equation}\label{general-flow-contraction}
\widehat E_{\mathrm{k}+1}^\frac{1}{2}(t)\leq \sqrt{C\mathfrak p(c_0)}(t-\tau) e^{\frac{1}{2}C\mathfrak p(c_0)(t-\tau)}\widehat E_{\mathrm{k}}^\frac{1}{2}(t).
\end{equation}
We decrease $T\leq T'$ once more so that
\begin{equation*}
\sqrt{C\mathfrak p(c_0)}T e^{\frac{1}{2}C\mathfrak p(c_0)T}\leq\frac{1}{2}.
\end{equation*}
Then \eqref{general-flow-contraction} yields
\begin{equation*}
\widehat E_{\mathrm{k}+1}^\frac{1}{2}(\tau+T)\leq\frac{1}{2}\widehat E_{\mathrm{k}}^\frac{1}{2}(\tau+T),\qquad \sum_{\mathrm{k}=1}^{\infty}\widehat E^{\frac{1}{2}}_{\mathrm{k}}(\tau+T)<\infty.
\end{equation*}
Consequently, $(r^m\rho_0)^{\frac{1}{2}}U^{\mathrm{k}}$ converges in $C([\tau,\tau+T];L^2)$, and $(r^m\rho_0)^{\frac{1}{2}}(U_r^{\mathrm{k}},\frac{U^{\mathrm{k}}}{r})$ converges in $L^2([\tau,\tau+T];L^2)$. By \eqref{395k-tau}, the chain-rule estimates, and the interpolation argument in Step 2 of \S\ref{subsection9.3}, we also obtain convergence in $L^2([\tau,\tau+T];H^3_{\mathrm{loc}})$ and convergence of the flow maps in $C([\tau,\tau+T];H^3_{\mathrm{loc}})$. Hence, we can pass to the limit in \eqref{general-flow-iteration} by the limit passage following \eqref{H333'}. The limit satisfies
\begin{equation*}
\eta(t,r)=\eta_\star(r)+\int_\tau^tU(s,r)\,\mathrm{d}s,\qquad \varrho(t,r)=\frac{r^m\rho_0}{(\eta^m\eta_r)(t,r)},
\end{equation*}
and solves \eqref{general-initial-problem}. Lower semicontinuity preserves the full $c_1$ bound in \eqref{395k-tau}, the integrated dissipation, and the positive lower and upper bounds of $(\eta_r,\frac{\eta}{r})$.

Finally, the uniqueness, time-continuity, and the boundary condition \eqref{b1-lo3} follow by arguments similar to Steps 3--4 of \S\ref{subsection9.3} with the time interval from $\tau$ to $\tau+T$. Taking $T_*=T$, we obtain all the conclusions of Remark \ref{remk31}. In particular, $T_*$ depends only on $(\delta_*,\delta^*,\mathfrak M_0,n,\mu,A,\gamma,\beta,\varepsilon_0,\rho_0,\cK_1,\cK_2)$, independently of $\tau$.
\end{proof}

\appendix
\section{Some Basic Lemmas}\label{appendix A}

For the convenience of readers, we list some basic facts that have been used frequently in this paper. Throughout the following Appendices \ref{appendix A}--\ref{subsection2.2}, let $\rho_0(\boldsymbol{y})=\rho_0(r)$, with $\rho_0(r)$ satisfying \eqref{distance-la}, that is, for some constants $\cK_2>\cK_1>0$ and $\beta\in (\frac{1}{3},\gamma-1]$,
\begin{equation*}
\begin{aligned}
&r^\frac{m}{2}\big(\rho_0^\beta,(\rho_0^\beta)_r,(\rho_0^\beta)_{rr},\frac{(\rho_0^\beta)_r}{r},(\rho_0^\beta)_{rrr},(\frac{(\rho_0^\beta)_r}{r})_r\big)\in L^2(I),\\
&\cK_1(1-r)^\frac{1}{\beta}\leq \rho_0(r)\leq \cK_2(1-r)^\frac{1}{\beta} \qquad\,\, \text{for all $r\in I$}.
\end{aligned}
\end{equation*}

The first lemma concerns the separability and density of the weighted Sobolev spaces. 
\begin{Lemma}[\cite{kufner}]\label{W-space}
Let $\vartheta_1,\vartheta_2\in (-1,\infty)$, and let $d=d(r)=r^{\vartheta_1}(1-r)^{\vartheta_2}$ be a function defined on $I$. Then, for $k\in \ZZ$, $H^k_{d}$ is a reflexive separable Banach space. Moreover, if $k\in \mathbb{N}$, $C^\infty(\bar I)$ is dense in $H^k_{d}$ with respect to the norm $\norm{\cdot}_{k,d}$. 
\end{Lemma}

The second lemma concerns the well-known interpolation theory for the $H^k$ spaces.
\begin{Lemma}[\cite{leoni}]\label{GNinequality'} 
Let $J\subset \RR$ be some open interval and $F\in H^p(J)\cap H^q(J)$ $(p,q\geq 0)$. Then $F\in H^l(J)$ for $l=p\vartheta+q(1-\vartheta)$ and  $0\leq \vartheta\leq 1$, and the following inequality holds{\rm:}
\begin{equation*}
\norm{F}_{l}\leq C \norm{F}_{p}^\vartheta\norm{F}_{q}^{1-\vartheta},
\end{equation*}
where $C>0$ is a constant depending only on $(p, q,\vartheta)$.
\end{Lemma}

The third lemma is on the classical Sobolev embedding theorem.
\begin{Lemma}[\cite{leoni}]\label{sobolev-embedding}
Let $J\subset \RR$ be some open interval and $f=f(r)$ be some function on $J$. Then there exist two positive constants $(s_0,C)$, depending only on the Lebesgue measure of $J$, such that
\begin{equation*}
\begin{aligned}
&\|f\|_{L^\infty(J)} \leq s_0\|f\|_{L^1(J)}+C\|f_r\|_{L^1(J)} &&\quad \text{for all }f\in W^{1,1}(J),\\[4pt]
&\|f\|_{L^\infty(J)} \leq s_0\|f\|_{L^2(J)}+C\|f_r\|_{L^2(J)} &&\quad \text{for all }f\in H^1(J),\\
&\|f\|_{L^\infty(J)} \leq s_0\|f\|_{L^2(J)}+ C\|f\|^\frac{1}{2}_{L^2(J)}\|f_r\|^\frac{1}{2}_{L^2(J)} &&\quad \text{for all }f\in H^1(J).
\end{aligned}
\end{equation*}
In particular, $W^{1,1}(J)\into C(\bar J)$ and $H^1(J)\into C(\bar J)$ continuously{\rm ;} moreover, 
if $f(r_0)=0$ for some $r_0\in \bar J$, then $s_0=0$ can be chosen in the above.
\end{Lemma}

The next two lemmas are on the Hardy inequality and some weighted interpolation inequality. In Lemmas \ref{hardy-inequality}--\ref{GNinequality}, we let $0\leq a<b<\infty$ and $J=(a,b)$, and let $d=d(r)$ be some function on $J$, taking one of the following two forms{\rm:}
\begin{equation*}
d(r)=r-a,\qquad \text{or} \ \ d(r)=b-r.
\end{equation*}
\begin{Lemma}[\cites{opic1,opic2}]\label{hardy-inequality}
Let $p\in [1,\infty]$ and $\vartheta>-\frac{1}{p}$ $(\vartheta>0\text{ if }p=\infty)$. Then, for any $f$ such that $d^{\vartheta+\frac{1}{2}+\frac{1}{p}} (f,f_r)\in L^2(J)$,
\begin{enumerate}
\item[{\rm (i)}] If $p\in [1,2)$, for any $\varepsilon>0$, there exists a constant $C_1>0$, depending only on $(\varepsilon, p,a,b,\vartheta)$, such that
\begin{equation*}
\|d^{\vartheta+\varepsilon} f\|_{L^p(J)} \leq C_1 \big\|d^{\vartheta+\frac{1}{2}+\frac{1}{p}} (f,f_r)\big\|_{L^2(J)};
\end{equation*}
\item[{\rm (ii)}] If $p\in [2,\infty]$, there exists a constant $C_2>0$, which depends only on $(p,a,b,\vartheta)$ if $p\ne \infty$ and depends only on $(a,b,\vartheta)$ if $p=\infty$, such that
\begin{equation*}
\|d^\vartheta f\|_{L^p(J)} \leq C_2 \big\|d^{\vartheta+\frac{1}{2}+\frac{1}{p}} (f,f_r)\big\|_{L^2(J)}. 
\end{equation*}
Moreover, if $p=\infty$, $d^\vartheta f\in C(\bar J)$.
\end{enumerate}
\end{Lemma}
 
\begin{Lemma}[\cite{ZJW}]\label{GNinequality}
Let $p\in [2,\infty]$ and $\vartheta>-\frac{1}{p}$ $(\vartheta>0\text{ if }p=\infty)$. Then, for any $f$ satisfying $d^{\vartheta+\frac{1}{p}} (f,f_r)\in L^2(J)$, there exists a constant $C>0$, which depends only on $(p,a,b,\vartheta)$ if $p\ne \infty$ and depends only on $(a,b,\vartheta)$ if $p=\infty$, such that
\begin{equation}\label{G-N1}
\|d^\vartheta f\|_{L^p(J)} \leq C\big(\big\|d^{\vartheta+\frac{1}{p}} f\big\|_{L^2(J)}+\big\|d^{\vartheta+\frac{1}{p}} f\big\|_{L^2(J)}^\frac{1}{2}\big\|d^{\vartheta+\frac{1}{p}} f_r\big\|_{L^2(J)}^\frac{1}{2}\big).
\end{equation}
\end{Lemma}
\begin{proof}
It suffices to prove \eqref{G-N1} with $J=I$ and $d(r)=r$. The proof for \eqref{G-N1} 
with the other cases can be obtained similarly. 

Let  $f\in C^\infty(\bar I)$. First, if $p=2$ and $\vartheta>-\frac{1}{2}$, a direct calculation, combined with integration by parts, yields
\begin{equation}\label{A..5}
\begin{aligned}
|r^\vartheta f|_2^2=& \frac{1}{2\vartheta+1}f^2(1)- \frac{2}{2\vartheta+1} \int_0^1 r^{2\vartheta+1} ff_r\,\mathrm{d}r,\\
|r^{\vartheta+\frac{1}{2}} f|_2^2=& \frac{1}{2\vartheta+2}f^2(1)- \frac{2}{2\vartheta+2} \int_0^1 r^{2\vartheta+2} f f_r\,\mathrm{d}r,
\end{aligned}
\end{equation}
which, along with the H\"older inequality, implies
\begin{equation}\label{AA4}
\begin{aligned}
|r^\vartheta f|_2^2
&= \frac{2\vartheta+2}{2\vartheta+1} \big|r^{\vartheta+\frac{1}{2}} f\big|_2^2+ \frac{2}{2\vartheta+1} \int_0^1 (r-1)r^{2\vartheta+1} ff_r\,\mathrm{d}r\\
&\leq C\big(\big|r^{\vartheta+\frac{1}{2}} f\big|_2^2+ \big|r^{\vartheta+\frac{1}{2}} f\big|_2\big|r^{\vartheta+\frac{1}{2}} f_r\big|_2\big).
\end{aligned}
\end{equation}

Next, if $p=\infty$ and $\vartheta>0$, it follows from the above, Lemma \ref{sobolev-embedding}, and the H\"older inequality that
\begin{equation}\label{A7}
\begin{aligned}
|r^\vartheta f|_\infty^2&\leq  C\int_0^1 |(r^{2\vartheta} f^2)_r|\,\mathrm{d}r=C\int_0^1 r^{2\vartheta-1}f^2\,\mathrm{d}r+C\int_0^1 r^{2\vartheta}|f||f_r|\,\mathrm{d}r\\
&\leq C\big|r^{\vartheta-\frac{1}{2}} f\big|_2^2+C|r^\vartheta f|_2|r^\vartheta f_r|_2\leq C\big(\big|r^\vartheta f\big|_2^2+|r^\vartheta f|_2|r^\vartheta f_r|_2\big).
\end{aligned}
\end{equation}

Finally, if $p\in (2,\infty)$ and $\vartheta>-\frac{1}{p}$, repeating the similar calculations in \eqref{A..5}--\eqref{AA4}, combined with \eqref{A7}, leads to
\begin{equation} 
\begin{aligned}
|r^\vartheta f|_p^p
&= \frac{p\vartheta+2}{p\vartheta+1} \big|r^{\vartheta+\frac{1}{p}} f\big|_p^p+ \frac{p}{p\vartheta+1} \int_0^1 (r-1)r^{p\vartheta+1} |f|^{p-2}ff_r\,\mathrm{d}r\\
&\leq C\big(\big|r^{\vartheta+\frac{1}{p}} f\big|_p^p+ \big|r^{(\vartheta+\frac{1}{p})(p-1)} |f|^{p-1}\big|_2\big|r^{\vartheta+\frac{1}{p}} f_r\big|_2\big)\\
&\leq C\big|r^{\vartheta+\frac{1}{p}} f\big|_\infty^{p-2}\big(\big|r^{\vartheta+\frac{1}{p}} f\big|_2^2+ \big|r^{\vartheta+\frac{1}{p}} f\big|_2\big|r^{\vartheta+\frac{1}{p}} f_r\big|_2\big)\\
&\leq C \big(\big|r^{\vartheta+\frac{1}{p}} f\big|_2^p+ \big|r^{\vartheta+\frac{1}{p}} f\big|_2^\frac{p}{2}\big|r^{\vartheta+\frac{1}{p}} f_r\big|_2^\frac{p}{2}\big).
\end{aligned}
\end{equation}
Therefore, we complete the proof for \eqref{G-N1} when $f\in C^\infty(\bar I)$.

For $f\in H^1_{d^{2\vartheta+\frac{2}{p}}}$, there exists a sequence $\{f^\varepsilon\}_{\varepsilon>0}\subset C^\infty(\bar I)$ due to Lemma \ref{W-space} 
such that
\begin{equation*} 
\big|r^{\vartheta+\frac{1}{p}} (f^\varepsilon- f)\big|_2+\big|r^{\vartheta+\frac{1}{p}} (f^\varepsilon_r- f_r)\big|_2 \to 0 \qquad \text{as }\varepsilon\to 0.
\end{equation*}
Then we can show that \eqref{G-N1} holds for all $f\in H^1_{d^{2\vartheta+\frac{2}{p}}}$ via the density argument. 

This completes the proof of Lemma \ref{GNinequality}.
\end{proof}

The sixth lemma is used to obtain the time-weighted estimates of the velocity.
\begin{Lemma}[\cite{bjr}]\label{bjr}
Let $f\in L^2([0,T]; L^2)$. Then there exists a sequence $\{t_k\}_{k=1}^\infty$ such that
\begin{equation*}
t_k\rightarrow 0, \quad\,\, t_k |f(t_k)|^2_{2}\rightarrow 0 \qquad\,\, \text{as $k\rightarrow\infty$}.
\end{equation*}
\end{Lemma}
\begin{proof}
Let $F(t)=|f(t)|^2_{2}$. Clearly, $0\leq F(t)\in L^1(0,T)$. Then it suffices to show that, for any $k\ge 1$, there exists $t_k\in (0,\frac{T}{1+k})$ such that 
\begin{equation*}
t_kF(t_k)<\frac{1}{k}\to 0\qquad \text{as $k\to \infty$}.
\end{equation*}
Assume by contradiction that there exists some $k_0\geq 1$ such that
\begin{equation*}
tF(t)\ge \frac{1}{k_0} \qquad \text{for any $t\in\big(0,\frac{T}{1+k_0}\big)$}.
\end{equation*}
Then
\begin{equation*}
\int_0^T F(s)\,\mathrm{d}s\ge \frac{1}{k_0}\int_0^\frac{T}{1+k_0} \frac{1}{s }\,\mathrm{d}s=\infty.    
\end{equation*}
This contradicts the fact that $F(t)\in L^1(0,T)$. Thus, the claim holds. 
\end{proof}

The seventh lemma is an equivalent statement for spherically symmetric vector functions.
\begin{Lemma}\label{duichen-dengjia}
Let $\boldsymbol{f}=\boldsymbol{f}(\boldsymbol{y})$ be a spherically symmetric  continuous vector function on $B_R=\{\boldsymbol{y}: \,|\boldsymbol{y}|<R\}$ for some $R>0$. Then $\boldsymbol{f}$ takes the form{\rm :} $\boldsymbol{f}(\boldsymbol{y})=f(|\boldsymbol{y}|)\frac{\boldsymbol{y}}{|\boldsymbol{y}|}$ if and only if 
\begin{equation}\label{dengjia}
\mathcal{O}\boldsymbol{f}(\boldsymbol{y})=\boldsymbol{f}(\mathcal{O}\boldsymbol{y}) \qquad \text{for all $\boldsymbol{y}\in B_R$ and $\mathcal{O}\in \mathrm{O}(n)$}.
\end{equation}
In particular, any spherically symmetric  vector  function $\boldsymbol{f}$ satisfies $\boldsymbol{f}(\boldsymbol{0})=\boldsymbol{0}$.
\end{Lemma}
\begin{proof}
Indeed, since $\boldsymbol{f}(\boldsymbol{y})=f(|\boldsymbol{y}|)\frac{\boldsymbol{y}}{|\boldsymbol{y}|}$ is spherically symmetric, it follows that
$\boldsymbol{f}(\mathcal{O}\boldsymbol{y})=\mathcal{O}\boldsymbol{f}(\boldsymbol{y})$ for all $\boldsymbol{y}\in B_R$ and  
$\mathcal{O}\in \mathrm{O}(n)$. 

Conversely, if \eqref{dengjia} holds, we take the 3-D case as an example. Let $\boldsymbol{y}_0\in B_R$ be any fixed displacement vector and $\boldsymbol{e}_1=\frac{\boldsymbol{y}_0}{|\boldsymbol{y}_0|}$. Assume that $\mathcal{O}_1\in \mathrm{O}(n)$ is a rotation by 180 degrees about an axis parallel to $\boldsymbol{y}_0$, that is, $\mathcal{O}_1\boldsymbol{y}_0=\boldsymbol{y}_0$. Then we can obtain from \eqref{dengjia} that
\begin{equation}\label{cls'}
\mathcal{O}_1\boldsymbol{f}(\boldsymbol{y}_0)=\boldsymbol{f}(\boldsymbol{y}_0).
\end{equation}
Next, let $\{\boldsymbol{e}_2,\boldsymbol{e}_3\}$ be two unit vectors such that $\{\boldsymbol{e}_1,\boldsymbol{e}_2,\boldsymbol{e}_3\}$ becomes an orthonormal basis in $\mathbb{R}^3$. 
Then there exist constants $\alpha_i=\alpha_i(\boldsymbol{y}_0)\in \mathbb{R}$ $(i=1,2,3)$, depending only on $\boldsymbol{y}_0$, such that 
\begin{equation*}
\boldsymbol{f}(\boldsymbol{y}_0)=\alpha_1\boldsymbol{e}_1+\alpha_2\boldsymbol{e}_2+\alpha_3\boldsymbol{e}_3.
\end{equation*}
This, together with \eqref{cls'}, gives 
\begin{equation*}
\alpha_1\boldsymbol{e}_1+\alpha_2\boldsymbol{e}_2+\alpha_3\boldsymbol{e}_3=\alpha_1\boldsymbol{e}_1-\alpha_2\boldsymbol{e}_2-\alpha_3\boldsymbol{e}_3\implies \alpha_2\boldsymbol{e}_2+\alpha_3\boldsymbol{e}_3=\boldsymbol{0},
\end{equation*}
which, along with the linear independence of $\{\boldsymbol{e}_2,\boldsymbol{e}_3\}$, leads to $\alpha_2=\alpha_3=0$. Hence, for any fixed $\boldsymbol{y}_0\in B_R$,
\begin{equation*}
\boldsymbol{f}(\boldsymbol{y}_0)=\alpha_1\boldsymbol{e}_1=\alpha_1(\boldsymbol{y}_0)\frac{\boldsymbol{y}_0}{|\boldsymbol{y}_0|}.
\end{equation*}
By \eqref{dengjia}, we see that $\alpha_1(\cO\boldsymbol{y}_0)=\alpha_1(\boldsymbol{y}_0)$ for all $\cO\in \mathrm{O}(n)$, that is, $\alpha_1(\boldsymbol{y}_0)=\alpha_1(|\boldsymbol{y}_0|)$. 
Then defining  $f(r):=\alpha_1(r)$ implies that $\boldsymbol{f}$ takes the form: $\boldsymbol{f}(\boldsymbol{y})=f(|\boldsymbol{y}|)\frac{\boldsymbol{y}}{|\boldsymbol{y}|}$.

Finally, from \eqref{dengjia}, we can take $\boldsymbol{y}=\boldsymbol{0}$ to obtain
\begin{equation}\label{s0}
\boldsymbol{f}(\boldsymbol{0})=\mathcal{O}\boldsymbol{f}(\boldsymbol{0}) \qquad \text{for all $\mathcal{O}\in \mathrm{O}(n)$}.
\end{equation}
Then, choosing $\mathcal{O}=\mathcal{O}_2$ to be a rotation by 180 degrees about an axis perpendicular to $\boldsymbol{f}(\boldsymbol{0})$, that is, $\mathcal{O}_2\boldsymbol{f}(\boldsymbol{0})=-\boldsymbol{f}(\boldsymbol{0})$, we obtain from \eqref{s0} that $\boldsymbol{f}(\boldsymbol{0})=-\boldsymbol{f}(\boldsymbol{0})$, so that $\boldsymbol{f}(\boldsymbol{0})=\boldsymbol{0}$. 
\end{proof}

Finally, to give the time continuity of the velocity in our proof, the following two types of evolution triple embeddings are required. 
\begin{Lemma}[\cite{evans}]\label{triple}
Let $T>0$ and $J\subset \RR^k$ $(k=1,2,3)$ be some open subset. Assume that $f\in L^2([0,T];H^1_0(J))$ and  $f_t\in L^2([0,T];H^{-1}(J))$. 
Then $f\in C([0,T];L^2(J))$, and the map{\rm :} $t\mapsto \|f(t)\|_{L^2(J)}^2$ is absolutely continuous 
with
\begin{equation*}
\frac{\mathrm{d}}{\mathrm{d}t} \|f(t)\|_{L^2(J)}^2=2\left<f_t, f\right>_{H^{-1}(J)\times H_0^1(J)} \qquad \text{for {\it a.e.} $t\in (0,T)$}.
\end{equation*}
Moreover, if additionally $f\in L^\infty([0,T];H_0^1(J))$, then $f\in C([0,T];L^4(J))$.
\end{Lemma}
\begin{Lemma}[\cite{ZJW}]\label{Aubin}
Let $T>0$, $s\geq 0$, $\cH\into L^2_{r^m\rho_0^s} \into \cH^*$, where $\cH$ is a separable Hilbert space. Assume that $f\in L^2([0,T];\cH)$ and  $r^m\rho_0^sf_t\in L^2([0,T];\cH^*)$. Then $f\in C([0,T];L^2_{r^m\rho_0^s})$, and the map{\rm :} $t\mapsto |f(t)|_{2,r^m\rho_0^s}^2$ is absolutely continuous with 
\begin{equation}\label{timedirivative}
\frac{\mathrm{d}}{\dt} |f(t)|_{2,r^m\rho_0^s}^2=2\left<r^m\rho_0^sf_t, f\right>_{\cH^*\times \cH}.
\end{equation}
Here, the functional pairing $\left<\cdot, \cdot\right>_{\cH^*\times \cH}$ should be read as
\begin{equation*}
\left<r^m\rho_0^sf, g\right>_{\cH^*\times \cH}=\int_0^1 r^m\rho_0^sfg \,\mathrm{d}r\qquad\text{for $g\in \cH$}.
\end{equation*}
\end{Lemma}

\begin{proof}
We first establish the weighted time continuity by regularization in time. Here, $r^m\rho_0^sf$ is viewed as an $\cH^*$-valued function by integration against elements of $\cH$. By the definition of the dual norm and the H\"older inequality, we have
\begin{equation*}
\|r^m\rho_0^sf\|_{\cH^*}\leq \big|(r^m\rho_0^s)^\frac{1}{2}f\big|_2\leq \|f\|_{\cH}.
\end{equation*}
Since $r^m\rho_0^s$ is independent of $t$, the assumptions imply that
\begin{equation*}
r^m\rho_0^sf\in H^1([0,T];\cH^*)\subset C([0,T];\cH^*).
\end{equation*}
Thus, the endpoint traces of $r^m\rho_0^sf$ are well defined in $\cH^*$. 

Next, we extend $f$ to $(-T,2T)$ by reflection, setting $f(t,x)=f(-t,x)$ for $-T<t<0$ and $f(t,x)=f(2T-t,x)$ for $T<t<2T$, and retain the same notation. In particular, the extension satisfies
\begin{equation*}
f\in L^2([-T,2T];\cH),\qquad r^m\rho_0^sf_t\in L^2([-T,2T];\cH^*).
\end{equation*}

Let $\omega_\varepsilon$ be a standard time mollifier supported in $(-\varepsilon,\varepsilon)$, and set $f^\varepsilon=f*\omega_\varepsilon$ on $[0,T]$ for $0<\varepsilon<\frac{T}{2}$. Then $f^\varepsilon\in C^\infty([0,T];\cH)$, and $r^m\rho_0^sf_t^\varepsilon=(r^m\rho_0^sf_t)*\omega_\varepsilon$. It follows from the approximation properties of convolution that
\begin{equation*}
\normf{f^\varepsilon-f}_{L^2_t(L^2_{r^m\rho_0^s})}+\normf{f^\varepsilon-f}_{L^2_t(\cH)}+\normf{r^m\rho_0^sf_t^\varepsilon-r^m\rho_0^sf_t}_{L^2_t(\cH^*)}\to0\quad\text{as }\varepsilon\to0,
\end{equation*}
where the time norms are taken over $(0,T)$.

Thus,  for any $0<\varepsilon,\sigma<\frac{T}{2}$, we have
\begin{equation*}
\frac{\mathrm{d}}{\dt} \|f^\varepsilon(t)- f^\sigma(t)\|_{L^2_{r^m\rho_0^s}}^2 
=2\big<r^m\rho_0^s(f_t^\varepsilon- f_t^\sigma), f^\varepsilon- f^\sigma\big>_{\cH^*\times \cH}.
\end{equation*}
Integrating the above from any $\tau\in[0,T]$ to $t\in[0,T]$ and estimating the time integral over $(0,T)$, we obtain
\begin{equation*}
\|(f^\varepsilon - f^\sigma)(t)\|_{L^2_{r^m\rho_0^s}}^2\leq\|(f^\varepsilon - f^\sigma)(\tau)\|_{L^2_{r^m\rho_0^s}}^2+2\int_0^T\normf{r^m\rho_0^s(f_t^\varepsilon- f_t^\sigma)}_{\cH^*}\normf{f^\varepsilon-f^\sigma}_{\cH}\,\dt.
\end{equation*}
Integrating the above inequality with respect to $\tau$ over $(0,T)$, dividing by $T$, and applying the H\"older inequality yield
\begin{equation*}
\begin{aligned}
&\sup_{t\in[0,T]}\|(f^\varepsilon - f^\sigma)(t)\|_{L^2_{r^m\rho_0^s}}^2\\
&\leq\frac1T\normf{(f^\varepsilon - f^\sigma)}_{L^2_t(L^2_{r^m\rho_0^s})}^2+2\normf{r^m\rho_0^s(f_t^\varepsilon- f_t^\sigma)}_{L^2_t(\cH^*)}\normf{f^\varepsilon-f^\sigma}_{L^2_t(\cH)}\to0
\end{aligned}
\end{equation*}
as $(\varepsilon,\sigma)\to(0,0)$. Thus, $\{(r^m\rho_0^s)^\frac{1}{2}f^\varepsilon\}$ is a Cauchy family in $C([0,T];L^2)$. Since it also converges to $(r^m\rho_0^s)^\frac{1}{2}f$ in $L^2([0,T];L^2)$, its uniform limit agrees with $(r^m\rho_0^s)^\frac{1}{2}f$ for a.e. $t\in (0,T)$. Hence,
\begin{equation*}
(r^m\rho_0^s)^\frac{1}{2}f\in C([0,T];L^2),\qquad (r^m\rho_0^s)^\frac{1}{2}f^\varepsilon\to (r^m\rho_0^s)^\frac{1}{2}f\quad\text{in }C([0,T];L^2).
\end{equation*}

Similarly, we have
\begin{equation*}
\|f^\varepsilon(t)\|_{L^2_{r^m\rho_0^s}}^2=\|f^\varepsilon(\tau)\|_{L^2_{r^m\rho_0^s}}^2+2\int_\tau^t\big<r^m\rho_0^sf_t^\varepsilon, f^\varepsilon\big>_{\cH^*\times \cH}\,\mathrm{d}t'
\qquad\,\,\text{for all $\tau,t\in [0,T]$}.
\end{equation*}
 Taking the limit as $\varepsilon\to 0$ implies
\begin{equation}\label{A10}
\|f(t)\|_{L^2_{r^m\rho_0^s}}^2=\|f(\tau)\|_{L^2_{r^m\rho_0^s}}^2+2\int_\tau^t\big<r^m\rho_0^sf_t, f\big>_{\cH^*\times \cH}\,\mathrm{d}t',
\end{equation}
which implies the absolute  continuity of $\|f(t)\|_{L^2_{r^m\rho_0^s}}^2$.  Applying $\partial_t$ to \eqref{A10} yields \eqref{timedirivative}.
\end{proof}

\section{Coordinate Transformations}\label{appb}

This appendix is devoted to showing the conversion of some Sobolev spaces between the M-D coordinates $\boldsymbol{y}$ and the spherical coordinate $r=|\boldsymbol{y}|$ for spherically symmetric functions. Let $n$ be the number of spatial dimensions and $m=n-1$.

\subsection{Transformation between multidimensional coordinates and spherical ones}

Let $0\leq a<b$, $\cJ:=\{\boldsymbol{y}\in \mathbb{R}^n: \, a\leq |\boldsymbol{y}|< b\}$, and $r\in J:=[a,b)$ with $r=|\boldsymbol{y}|$. Consider a coordinate transformation $\boldsymbol{\xi}=\boldsymbol{\xi}(\boldsymbol{y})\in C^\infty(\bar\cJ)$ such that
\begin{equation*}
\boldsymbol{\xi}(\boldsymbol{y})=\xi(r)\frac{\boldsymbol{y}}{r}\quad \text{with $\xi(r)\geq0$},\qquad\boldsymbol{\xi}:\cJ\to \cG:=\boldsymbol{\xi}(\cJ),\qquad \boldsymbol{y}\mapsto \boldsymbol{x}=\boldsymbol{\xi}(\boldsymbol{y}). 
\end{equation*} 
Assume that $\nabla_{\boldsymbol{y}}\boldsymbol{\xi}$ is a non-singular matrix and $(\xi_r,\frac{\xi}{r})$ are positive on $J$, and define
\begin{equation*}
\begin{gathered}
D_\xi f=\frac{f_r}{\xi_r},\qquad\cB=(\cB_{ij})_{1\leq i,j\leq n} \qquad \text{with $\cB_{ij}:=((\nabla_{\boldsymbol{y}}\boldsymbol{\xi})^{-1})_{ij}$},\\[-2pt]
\nabla_\cB f=((\nabla_\cB f)_1,\cdots\!,(\nabla_\cB f)_n)^\top \qquad \text{with $(\nabla_\cB f)_i=\sum_{k=1}^n\cB_{ki}\partial_{y_k}f$},\\[-3pt]
\boldsymbol{f}=(f_1,\cdots\!, f_n)^\top,\qquad\nabla_\cB \boldsymbol{f}=((\nabla_\cB \boldsymbol{f})_{ij})_{1\leq i,j\leq n} \qquad \text{with $(\nabla_\cB \boldsymbol{f})_{ij}=\sum_{k=1}^n \cB_{kj}\partial_{y_k}f_i$},    
\end{gathered}
\end{equation*}
where $f(\boldsymbol{y})=f(r)$ and $\boldsymbol{f}(\boldsymbol{y})=f(r)\frac{\boldsymbol{y}}{r}$ are sufficiently smooth functions. 

Then we have the following coordinate transformations.
\begin{Lemma}\label{lemma-initial}
Assume that $(g,\boldsymbol{g})(\boldsymbol{x})$ are spherically symmetric functions defined on $\cG$ and $(f,\boldsymbol{f})(\boldsymbol{y})$ satisfy
\begin{equation*}
f(\boldsymbol{y})=g(\boldsymbol{\xi}(\boldsymbol{y}))=g(\boldsymbol{x}),\qquad\boldsymbol{f}(\boldsymbol{y})=\boldsymbol{g}(\boldsymbol{\xi}(\boldsymbol{y}))=\boldsymbol{g}(\boldsymbol{x}).
\end{equation*}
Then, for any $q\in [1,\infty]$, the following statements hold{\rm :}
\begin{enumerate}
\item[{\rm(i)}] Transformations for $(g,f)${\rm:} for $j=0,1$ and $k=2,3$, 
\begin{equation*}
\begin{aligned}
\|\nabla^j g\|_{L^q(\cG)}&\sim\|(\det \nabla_{\boldsymbol{y}}\boldsymbol{\xi})^\frac{1}{q}\nabla_\cB^j f\|_{L^q(\cJ)}\sim 
\|(\xi^m\xi_r)^\frac{1}{q}D_\xi^j f\|_{L^q(J)},\\[2pt]
\|\nabla^kg\|_{L^q(\cG)}&\sim\|(\det \nabla_{\boldsymbol{y}}\boldsymbol{\xi})^\frac{1}{q}\nabla_\cB^k f\|_{L^q(\cJ)}\sim  \Big\|(\xi^m\xi_r)^\frac{1}{q}\Big(D_\xi^k f,D_\xi^{k-2}\big(\frac{D_\xi f}{\xi}\big)\Big)\Big\|_{L^q(J)};
\end{aligned}
\end{equation*}
\item[{\rm(ii)}] Transformations for $(\boldsymbol{g},\boldsymbol{f})${\rm:} for $j=1,2$ and $k=3,4$, 
\begin{equation*}
\begin{aligned}
\|\boldsymbol{g}\|_{L^q(\cG)}&\sim\|(\det \nabla_{\boldsymbol{y}}\boldsymbol{\xi})^\frac{1}{q}\boldsymbol{f}\|_{L^q(\cJ)}\sim \|(\xi^m\xi_r)^\frac{1}{q}f\|_{L^q(J)},\\[2pt]
\|\nabla^{j}\boldsymbol{g}\|_{L^q(\cG)}&\sim\|(\det \nabla_{\boldsymbol{y}}\boldsymbol{\xi})^\frac{1}{q}\nabla_\cB^{j} \boldsymbol{f}\|_{L^q(\cJ)}\sim \Big\|(\xi^m\xi_r)^\frac{1}{q}\Big(D_\xi^{j} f,D_\xi^{j-1}\big(\frac{f}{\xi}\big)\Big)\Big\|_{L^q(J)},\\
\|\nabla^{k}\boldsymbol{g}\|_{L^q(\cG)}&\sim\|(\det \nabla_{\boldsymbol{y}}\boldsymbol{\xi})^\frac{1}{q}\nabla_\cB^{k} \boldsymbol{f}\|_{L^q(\cJ)}\sim \Big\|(\xi^m\xi_r)^\frac{1}{q}\Big(D_\xi^{k} f,D_\xi^{k-1}\big(\frac{f}{\xi}\big),D_\xi^{k-3}\big(\frac{1}{\xi}D_\xi(\frac{f}{\xi})\big)\Big)\Big\|_{L^q(J)}.
\end{aligned}
\end{equation*}
\end{enumerate}
Here, $E\sim F$ denotes $C^{-1}E\leq F\leq CE$ for some constant $C\geq 1$ depending only on $n$, and we emphasize that, for any function space $X$ and functions $(\varphi ,h_1,\cdots\!,h_k)$,
\begin{equation*}
\|\varphi (h_1,\cdots\!,h_k)\|_{X}:=\sum_{i=1}^k\|\varphi h_i\|_X.
\end{equation*}
\end{Lemma}

\begin{proof}
It suffices to prove the transformations for $(\boldsymbol{g},\boldsymbol{f})$ since $\nabla_{\boldsymbol{y}} f=f_r\frac{\boldsymbol{y}}{r}$ can be regarded as a vector  function $\boldsymbol{h}=h\frac{\boldsymbol{y}}{r}$ with $h=f_r$.

We divide the proof into two steps.

\smallskip
\textbf{1.} We first prove the case when $\xi(r)=r$. In this case, $\boldsymbol{\xi}(\boldsymbol{y})=\boldsymbol{y}$ and $\cB$ is the $n\times n$ identity matrix. It follows from direct calculations that
\begin{align*}
& \quad \ \,\begin{aligned}
(f_k)_{y_i}&=\frac{y_i y_k}{r^2}f_r+  \frac{\delta_{ik}r^2-y_i y_k}{r^3}f,\\
(f_k)_{y_iy_j}&=\frac{y_i y_j y_k}{r^3}f_{rr}+ \Big(\frac{\delta_{ij}y_k+\delta_{ik}y_j+\delta_{jk}y_i}{r}-\frac{3y_i y_j y_k}{r^3}\Big)\big(\frac{f}{r}\big)_r,
\end{aligned}\\
& \ \ \,\begin{aligned}
(f_k)_{y_iy_jy_\ell}
&=\frac{y_i y_j y_ky_\ell}{r^4}f_{rrr}\\
&\quad +\Big(\frac{\delta_{i\ell}y_jy_k+\delta_{j\ell}y_iy_k+\delta_{k\ell}y_iy_j}{r^2}\\
&\quad\quad \ \ +\frac{\delta_{ij}y_ky_\ell+\delta_{ik}y_jy_\ell+\delta_{jk}y_iy_\ell}{r^2}-\frac{6y_i y_j y_ky_\ell}{r^4}\Big)\big(\frac{f}{r}\big)_{rr}\\
&\quad + \Big(\delta_{ij}\delta_{k\ell}+\delta_{ik}\delta_{j\ell}+\delta_{jk}\delta_{i\ell}-\frac{\delta_{i\ell}y_jy_k+\delta_{j\ell}y_iy_k+\delta_{k\ell}y_iy_j}{r^2}\\
&\quad\quad \ \ -\frac{\delta_{ij}y_ky_\ell+\delta_{ik}y_jy_\ell+\delta_{jk}y_iy_\ell}{r^2}+\frac{3y_i y_j y_ky_\ell}{r^4}\Big)\Big(\frac{1}{r}\big(\frac{f}{r}\big)_r\Big),    
\end{aligned}\\
&\begin{aligned}
(f_k)_{y_iy_jy_\ell y_p}
&=\frac{y_i y_j y_ky_\ell y_p}{r^5}f_{rrrr}\\
&\quad +\Big(\frac{\delta_{ip} y_j y_ky_\ell+\delta_{jp} y_i y_ky_\ell+\delta_{kp} y_i y_jy_\ell+\delta_{\ell p} y_i y_jy_k}{r^3}\\
&\quad\quad \ \ + \frac{\delta_{i\ell}y_jy_ky_p+\delta_{j\ell}y_iy_ky_p+\delta_{k\ell}y_iy_jy_p}{r^3}\\
&\quad\quad \ \ +\frac{\delta_{ij}y_ky_\ell y_p+\delta_{ik}y_jy_\ell y_p+\delta_{jk}y_iy_\ell y_p}{r^3}-\frac{10y_i y_j y_ky_\ell y_p}{r^5}\Big)\big(\frac{f}{r}\big)_{rrr}
\end{aligned}\\
&\qquad\quad\quad \ \ \begin{aligned}
&\quad +\Big(\frac{\delta_{i\ell}\delta_{jp}y_k+\delta_{i\ell}\delta_{kp}y_j+\delta_{j\ell}\delta_{ip}y_k+\delta_{j\ell}\delta_{kp}y_i+\delta_{k\ell}\delta_{ip}y_j+\delta_{k\ell}\delta_{jp}y_i}{r}\\
&\quad\quad \ \ +\frac{\delta_{ij}\delta_{kp}y_\ell+\delta_{ij}\delta_{\ell p}y_k+\delta_{ik}\delta_{jp}y_\ell+\delta_{ik}\delta_{\ell p} y_j+\delta_{jk}\delta_{ip}y_\ell+\delta_{jk}\delta_{\ell p}y_i}{r}\\
&\quad\quad \ \ +\frac{\delta_{ij}\delta_{k\ell}y_p+\delta_{ik}\delta_{j\ell}y_p+\delta_{jk}\delta_{i\ell}y_p}{r}\\
&\quad\quad \ \ -\frac{3(\delta_{ip} y_j y_ky_\ell+\delta_{jp} y_i y_ky_\ell+\delta_{kp} y_i y_jy_\ell+\delta_{\ell p}y_i y_j y_k)}{r^3}\\
&\quad\quad \ \ -\frac{3(\delta_{i\ell}y_jy_ky_p+\delta_{j\ell}y_iy_ky_p+\delta_{k\ell}y_iy_jy_p)}{r^3}\\
&\quad\quad \ \ -\frac{3(\delta_{ij}y_ky_\ell y_p+\delta_{ik}y_jy_\ell y_p+\delta_{jk}y_iy_\ell y_p)}{r^3}+\frac{15y_i y_j y_ky_\ell y_p}{r^5}\Big)\Big(\frac{1}{r}\big(\frac{f}{r}\big)_{r}\Big)_r.    
\end{aligned}
\end{align*}
Then the above expressions yield
\begin{align*}
|\boldsymbol{f}|^2&=\sum_{k=1}^n |f_k|^2= |f|^2,\quad |\nabla_{\boldsymbol{y}} \boldsymbol{f}|^2=\sum_{i,k=1}^n |(f_k)_{y_i}|^2= |f_r|^2+ m\Big|\frac{f}{r}\Big|^2,\notag\\
|\nabla_{\boldsymbol{y}}^2 \boldsymbol{f}|^2&=\sum_{i,j,k=1}^n |(f_k)_{y_iy_j}|^2=|f_{rr}|^2+3m\Big|\big(\frac{f}{r}\big)_r\Big|^2,\notag\\
|\nabla_{\boldsymbol{y}}^3 \boldsymbol{f}|^2&=\sum_{i,j,k,\ell=1}^n \big|(f_k)_{y_iy_jy_\ell}\big|^2=|f_{rrr}|^2+6m\Big|\big(\frac{f}{r}\big)_{rr}\Big|^2+ (3m^2+6m)  \Big|\frac{1}{r}\big(\frac{f}{r}\big)_r\Big|^2, \\
|\nabla_{\boldsymbol{y}}^4 \boldsymbol{f}|^2&=\sum_{i,j,k,\ell,p=1}^n \big|(f_k)_{y_iy_jy_\ell y_p}\big|^2=|f_{rrrr}|^2\!+10m\Big|\big(\frac{f}{r}\big)_{rrr}\Big|^2\!+ (15m^2\!+\!30m) \Big|\Big(\frac{1}{r}\big(\frac{f}{r}\big)_r\Big)_r\Big|^2,\notag
\end{align*}
which implies 
\begin{equation}\label{BB}
\begin{gathered}
|\boldsymbol{f}|\sim |f|,\quad |\nabla_{\boldsymbol{y}} \boldsymbol{f}|\sim |f_r|+\Big|\frac{f}{r}\Big|,\quad |\nabla^2_{\boldsymbol{y}} \boldsymbol{f}|\sim |f_{rr}|+\Big|\big(\frac{f}{r}\big)_r\Big|,\\
|\nabla^3_{\boldsymbol{y}} \boldsymbol{f}|\sim |f_{rrr}|+\Big|\big(\frac{f}{r}\big)_{rr}\Big|+\Big|\frac{1}{r}\big(\frac{f}{r}\big)_{r}\Big|,\quad
|\nabla^4_{\boldsymbol{y}} \boldsymbol{f}|\sim |f_{rrrr}|+\Big|\big(\frac{f}{r}\big)_{rrr}\Big|+\Big|\Big(\frac{1}{r}\big(\frac{f}{r}\big)_{r}\Big)_r\Big|.
\end{gathered}
\end{equation}

Finally, thanks to the integral identity:  
\begin{equation*}
\int_\cJ f(\boldsymbol{y})\,\mathrm{d}\boldsymbol{y}= \omega_n\int_J f(r)r^m\,\mathrm{d}r,
\end{equation*}
where $\omega_n$ denotes the surface area of the unit sphere in $\mathbb {R}^n$, we thus obtain the desired conclusions of this lemma when $\boldsymbol{\xi}(\boldsymbol{y})=\boldsymbol{y}$ from \eqref{BB}.

\smallskip
\textbf{2.} For general $\boldsymbol{\xi}(\boldsymbol{y})$, we can first repeat the calculations in Step 1 with the coordinate $\boldsymbol{x}=\boldsymbol{\xi}(\boldsymbol{y})$ and the function $\boldsymbol{g}(\boldsymbol{x})$. Specifically, if we let $x:=|\boldsymbol{x}|$, then
\begin{equation*}
x=\xi(r)\in G,\qquad G:=[\xi(a),\xi(b)),
\end{equation*}
and we can obtain from \eqref{BB} that
\begin{equation}\label{BB=g}
\begin{aligned}
\|\boldsymbol{g}\|_{L^q(\cG)}&\sim \|x^\frac{m}{q}g\|_{L^q(G)},\qquad
\|\nabla \boldsymbol{g}\|_{L^q(\cG)} \sim \Big\|x^\frac{m}{q}\big(g_x,\frac{g}{x}\big)\Big\|_{L^q(G)},\\
\|\nabla^2 \boldsymbol{g}\|_{L^q(\cG)}&\sim \Big\|x^\frac{m}{q}\Big(g_{xx},\big(\frac{g}{x}\big)_x\Big)\Big\|_{L^q(G)},\\
\|\nabla^3 \boldsymbol{g}\|_{L^q(\cG)}&\sim \Big\|x^\frac{m}{q}\Big(g_{xxx},\big(\frac{g}{x}\big)_{xx},\frac{1}{x}\big(\frac{g}{x}\big)_x\Big)\Big\|_{L^q(G)},\\
\|\nabla^4 \boldsymbol{g}\|_{L^q(\cG)}&\sim \Big\|x^\frac{m}{q}\Big(g_{xxxx},\big(\frac{g}{x}\big)_{xxx},\big(\frac{1}{x}(\frac{g}{x})_x\big)_x\Big)\Big\|_{L^q(G)}.
\end{aligned}
\end{equation}

Next, using the coordinate transformations $\boldsymbol{x}=\boldsymbol{\xi}(\boldsymbol{y})$ and $x=\xi(r)$, we have
\begin{equation}\label{bbg}
\nabla^k\boldsymbol{g}=\nabla_{\cB}^k\boldsymbol{f},\quad\,\,\partial_x^k g=D_\xi^k f,\qquad\,\, k=0,1,2,3,4.
\end{equation}
Therefore, \eqref{BB=g}--\eqref{bbg}, together with the following integral identities
\begin{equation*}
\int_\cG g(\boldsymbol{x})\,\mathrm{d}\boldsymbol{x}=\int_\cJ f(\boldsymbol{y})(\det \nabla_{\boldsymbol{y}}\boldsymbol{\xi}) \,\mathrm{d}\boldsymbol{y},\qquad \int_G g(x)\,\mathrm{d}x=  \int_J f(r) \xi_r\,\mathrm{d}r,
\end{equation*}
lead to the desired results of this lemma.
\end{proof}

\subsection{Transformation from the M-D Lagrangian coordinates to the spherically symmetric coordinates}\label{secB2}
Generally, it is desirable to consider the Lagrangian formulation, so that we can pull back  \eqref{eq:1.1-vfbp} on the moving domain $\Omega(t)$ to a problem on a fixed domain $\Omega$. To this end, denote by $\boldsymbol{x}=\boldsymbol{\eta}(t,\boldsymbol{y})$ the position of the fluid particle $\boldsymbol{x} \in \Omega(t)$ at time $t\geq 0$ so that
\begin{equation}\label{flow-map-md}
\boldsymbol{\eta}_t(t,\boldsymbol{y})=\boldsymbol{u}(t,\boldsymbol{\eta}(t,\boldsymbol{y})) \ \ \text{for $t>0$},\qquad \text{with $\boldsymbol{\eta}(0,\boldsymbol{y})=\boldsymbol{y}$},
\end{equation}
and $(t,\boldsymbol{y})$ are the M-D Lagrangian coordinates. Then, by introducing the Lagrangian density and velocity: 
\begin{equation}\label{L-dens-velo}
\varrho(t,\boldsymbol{y})=\rho(t,\boldsymbol{\eta}(t,\boldsymbol{y})),\qquad \boldsymbol{U}(t,\boldsymbol{y})=\boldsymbol{u}(t,\boldsymbol{\eta}(t,\boldsymbol{y})),
\end{equation}
we can rewrite \eqref{eq:1.1-vfbp} as 
\begin{equation}\label{eq:1.1-la-vfbp}
\begin{cases}
\displaystyle \varrho_t+ \varrho \dive_\cA \boldsymbol{U} =0&\text{in }(0,T]\times\Omega,\\[6pt]
\displaystyle \varrho \boldsymbol{U}_t + \nabla_\cA P= \mu\dive_{\cA}\big(\varrho (\nabla_\cA\boldsymbol{U}+(\nabla_\cA\boldsymbol{U})^\top)\big)&\text{in }(0,T]\times\Omega,\\[6pt]
\boldsymbol{\eta}_t=\boldsymbol{U}&\text{in }(0,T]\times\Omega,\\[6pt]
\varrho>0&\text{in }(0,T]\times\Omega,\\[6pt]
\varrho|_{\partial\Omega}=0&\text{on }(0,T],\\[6pt]
(\varrho,\boldsymbol{U},\boldsymbol{\eta})(0,\boldsymbol{y})=(\rho_0(\boldsymbol{y}),\boldsymbol{u}_0(\boldsymbol{y}),\boldsymbol{y}) &\text{for $\boldsymbol{y}\in\Omega$}.
\end{cases}
\end{equation}
Here $\boldsymbol{y}=(y_1,\cdots\!,y_n)^\top$, $\boldsymbol{\eta}=(\eta_1,\cdots\!,\eta_n)^\top$, $\boldsymbol{U}=(U_1,\cdots\!,U_n)^\top$, and
\begin{equation}\label{sym}
\begin{aligned}
&\cA=(\nabla_{\boldsymbol{y}}\boldsymbol{\eta})^{-1},\qquad \nabla_{\boldsymbol{y}}\boldsymbol{\eta} \quad \text{with $(\nabla_{\boldsymbol{y}}\boldsymbol{\eta})_{ij}=\frac{\partial \eta_i}{\partial y_j}$},\\[-2pt]
&\nabla_\cA P=\cA^\top (\nabla P) \quad \text{with $(\nabla_\cA P)_{i}=\sum_{k=1}^n\cA_{ki}\frac{\partial P}{\partial y_k}$},\\[-2pt]
&\nabla_\cA \boldsymbol{U}=(\nabla \boldsymbol{U})\cdot \cA \quad \text{with $(\nabla_\cA\boldsymbol{U})_{ij}=\sum_{k=1}^n\cA_{kj}\frac{\partial U_i}{\partial y_k}$},\qquad \dive_\cA \boldsymbol{U}=\sum_{i=1}^n(\nabla_\cA\boldsymbol{U})_{ii}.
\end{aligned}
\end{equation}

The  initial density $\rho_0$ under consideration satisfies the following condition:
\begin{equation}\label{distance}
\rho_0^\beta(\boldsymbol{y})\in H^3(\Omega), \quad\,\, \cK_1(1-|\boldsymbol{y}|)^\frac{1}{\beta}\leq \rho_0(\boldsymbol{y})\leq \cK_2(1-|\boldsymbol{y}|)^\frac{1}{\beta} \qquad\,\, \text{for all $\boldsymbol{y}\in \overline\Omega$},
\end{equation}
for some  constants $\cK_2>\cK_1>0$ and $\beta>0$.

Now, we show that problem \eqref{eq:1.1-la-vfbp} also implies \eqref{eq:VFBP-La} in \S\ref{section1}. 
\begin{Lemma}\label{lemmad2}
Under the spherical symmetry assumption of $(\rho,\boldsymbol{u})(t,\boldsymbol{x})$ in \eqref{ss-ass}, if $(\rho,\boldsymbol{u})(t,\boldsymbol{x})$ is the classical solution of {\rm\bf VFBP} \eqref{eq:1.1-vfbp}, then
\begin{equation}\label{from3}
(\varrho,\boldsymbol{U},\boldsymbol{\eta})(t,\boldsymbol{y})=(\varrho(t,r),U(t,r)\frac{\boldsymbol{y}}{r},\eta(t,r) \frac{\boldsymbol{y}}{r}),
\end{equation}
where $r=|\boldsymbol{y}|$, and $\eta(t,r)$ and $(\varrho,U)(t,r)$ are given by \eqref{flowmap-r-la} and \eqref{varrho-U}, respectively. Moreover, $(\varrho,U,\eta)(t,r)$ satisfies {\rm\bf IBVP} \eqref{eq:VFBP-La}.
\end{Lemma}
\begin{proof}
First, we show that, under \eqref{ss-ass}, $\boldsymbol{\eta}(t,\boldsymbol{y})$ is spherically symmetric, taking the form:
\begin{equation}\label{eta-qiuduichen}
\boldsymbol{\eta}(t,\boldsymbol{y})=\eta(t,r)\frac{\boldsymbol{y}}{r},
\end{equation}
where $r=|\boldsymbol{y}|$, and $\eta(t,r)=x=|\boldsymbol{x}|$ is defined in \eqref{flowmap-r-la}. By Lemma \ref{duichen-dengjia} in Appendix \ref{appendix A}, it suffices to prove 
\begin{equation}\label{ttt}
(\cO\boldsymbol{\eta})(t,\boldsymbol{y})= \boldsymbol{\eta}(t,\cO\boldsymbol{y})  \qquad\text{for any $\boldsymbol{y}\in \Omega$ and $\cO\in \mathrm{O}(n)$}.
\end{equation}
Indeed, define 
\begin{equation}
\tilde{\boldsymbol{\eta}}(t,\boldsymbol{y})=(\cO^{-1} \boldsymbol{\eta})(t,\cO\boldsymbol{y}).
\end{equation}
Then it follows from \eqref{flow-map-md} and the spherical symmetry of $\boldsymbol{u}(t,\boldsymbol{x})$ \eqref{ss-ass} that 
\begin{equation*}
\tilde{\boldsymbol{\eta}}_t(t,\boldsymbol{y})=(\cO^{-1} \boldsymbol{u})(t, \boldsymbol{\eta} (t,\cO\boldsymbol{y}))=(\cO^{-1} \boldsymbol{u})(t, (\cO\tilde{\boldsymbol{\eta}}) (t,\boldsymbol{y}))=\boldsymbol{u} (t,  \tilde{\boldsymbol{\eta}} (t,\boldsymbol{y})).
\end{equation*}
Note that $\boldsymbol{u}(t,\boldsymbol{x})$ is a classical solution of problem \eqref{eq:1.1-vfbp}, \textit{i.e.}, $\boldsymbol{u}(t,\boldsymbol{x})\in C^1(\overline\Omega(t))$ for each $t>0$. Comparing the above with \eqref{flow-map-md}, since $\tilde{\boldsymbol{\eta}}(0,\boldsymbol{y})=\boldsymbol{\eta}(0,\boldsymbol{y})=\boldsymbol{y}$, we can derive from the uniqueness of ODEs \eqref{flow-map-md} that $\tilde{\boldsymbol{\eta}}\equiv \boldsymbol{\eta}$, which leads to \eqref{ttt}.

Next, from the radial coordinate $r=|\boldsymbol{y}|$ and the definition of $(\varrho, U)(t,r)$ in \eqref{varrho-U}, it follows that \eqref{L-dens-velo} becomes
\begin{equation} 
\begin{aligned}
\varrho(t,\boldsymbol{y})&=\rho(t,|\boldsymbol{\eta}(t,\boldsymbol{y})|)=\rho(t, \eta(t,r))=\varrho(t,r),\\
\boldsymbol{U}(t,\boldsymbol{y})&=u(t,|\boldsymbol{\eta}(t,\boldsymbol{y})|)\frac{\boldsymbol{\eta}(t,\boldsymbol{y})}{|\boldsymbol{\eta}(t,\boldsymbol{y})|}=u(t,\eta(t,r))\frac{\boldsymbol{\eta}(t,\boldsymbol{y})}{\eta(t,r)}=U(t,r)\frac{\boldsymbol{y}}{r}.
\end{aligned}
\end{equation}
This implies that $(\varrho,U)(t,r)$ in \eqref{varrho-U} are radial projections of $(\varrho,\boldsymbol{U})(t,\boldsymbol{y})$.

Finally, from the proof of Lemma \ref{lemma-initial}, it follows that
\begin{equation}\label{A!}
\begin{aligned}
(\nabla_{\boldsymbol{y}}\boldsymbol{\eta})_{ij}&=\frac{\eta}{r}\delta_{ij}+\big(\eta_r-\frac{\eta}{r}\big)\frac{y_i y_j}{r^2},\\
(\nabla_{\boldsymbol{y}}\boldsymbol{U})_{ij}&=\frac{U}{r}\delta_{ij}+\big(U_r-\frac{U}{r}\big)\frac{y_i y_j}{r^2},\qquad (\nabla_{\boldsymbol{y}}P)_i=\frac{y_i}{r}P_r,
\end{aligned}
\end{equation}
which, along with a direct calculation, yields
\begin{equation}\label{A!!}
\begin{aligned}
\cA_{ij}&=((\nabla_{\boldsymbol{y}}\boldsymbol{\eta})^{-1})_{ij}=  \frac{r}{\eta} \delta_{ij} 
+ \big(\frac{1}{\eta_r} - \frac{r}{\eta}\big) \frac{y_i y_j}{r^2},\\
(\nabla_\cA\boldsymbol{U})_{ij}&=\sum_{k=1}^n\cA_{kj}\frac{\partial U_i}{\partial y_k}=\frac{U}{\eta}\delta_{ij}+\big(\frac{U_r}{\eta_r}-\frac{U}{\eta}\big)\frac{y_iy_j}{r^2},\\
\dive_\cA \boldsymbol{U} &=\frac{U_r}{\eta_r}+\frac{mU}{\eta},\qquad (\nabla_\cA P)_{i}=\sum_{k=1}^n\cA_{ki}\frac{\partial P}{\partial y_k}=\frac{y_i}{r}\frac{P_r}{\eta_r}.
\end{aligned}
\end{equation}
Based on the above, we can derive the equations of $(\varrho,U,\eta)(t,r)$ in $\eqref{eq:VFBP-La}_1$--$\eqref{eq:VFBP-La}_3$.
\end{proof}

\subsection{Transformation between the Eulerian and Lagrangian coordinates} 
We give the transformation relations between $(\rho(t,\boldsymbol{x}),\boldsymbol{u}(t,\boldsymbol{x}),\partial\Omega(t))$ and $(\varrho(t,r),U(t,r))$. 

First, define the moving domain $\Omega(t)$ as
\begin{equation}
\Omega(t):=\big\{\boldsymbol{x}=\boldsymbol{\eta}(t,\boldsymbol{y}): \,\boldsymbol{y}\in \Omega\big\},
\end{equation}
and, for every $t\in [0,T]$ and $\boldsymbol{x}\in \Omega(t)$, define the inverse flow map $\boldsymbol{\eta}_*$ by
\begin{equation*}
\boldsymbol{y}=\boldsymbol{\eta}_*(t, \boldsymbol{x}):\quad \Omega(t)\to \Omega,\quad (t, \boldsymbol{x} )\mapsto (t,\boldsymbol{y}),
\end{equation*}
which satisfies
\begin{equation}\label{danwei}
\boldsymbol{\eta}(t,\boldsymbol{\eta}_*(t,\boldsymbol{x}))=\boldsymbol{x},\qquad \boldsymbol{\eta}_*(t,\boldsymbol{\eta}(t,\boldsymbol{y}))=\boldsymbol{y}.
\end{equation}

Then we have the following result:
\begin{Lemma}\label{lemma-B3}
Let $(U,\eta)$ be the classical solution of {\rm\bf IBVP} \eqref{eq:VFBP-La-eta} obtained in {\rm Theorem \ref{Theorem1.1}}. Then the inverse flow map $\boldsymbol{\eta}_*=\boldsymbol{\eta}_*(t, \boldsymbol{x})$ is well-defined on $\Omega(t)$ for each $t\in [0,T]$, and $\boldsymbol{\eta}_*$ is spherically symmetric, which takes the form{\rm:}
\begin{equation}\label{forms}
\boldsymbol{\eta}_*(t, \boldsymbol{x})=\eta_*(t, x)\frac{\boldsymbol{x}}{x},\qquad x=|\boldsymbol{x}|.
\end{equation}
Moreover, if we set 
\begin{equation}\label{equ1.28}
\rho(t,\boldsymbol{x}):=\varrho(t, \boldsymbol{\eta}_*(t,\boldsymbol{x})),\qquad \boldsymbol{u}(t,\boldsymbol{x}):=\boldsymbol{U}(t, \boldsymbol{\eta}_*(t,\boldsymbol{x})),
\end{equation}
where $(\varrho,\boldsymbol{U})(t,\boldsymbol{y})$ takes the form \eqref{from3} in {\rm Lemma \ref{lemmad2}}, then $(\rho,\boldsymbol{u})(t,\boldsymbol{x})$ is spherically symmetric, taking the form \eqref{ss-ass}, and $(\rho(t,\boldsymbol{x}),\boldsymbol{u}(t,\boldsymbol{x}),\partial\Omega(t))$ is a solution of {\rm\bf VFBP} \eqref{eq:1.1-vfbp} described by {\rm(i)--(ii)} of {\rm Theorem \ref{theorem1.3}}.
\end{Lemma}
\begin{proof}
We divide the proof into five steps.

\smallskip
\textbf{1.} First, by \eqref{A!}--\eqref{A!!},
\begin{equation*}
\det(\nabla_{\boldsymbol{y}}\boldsymbol{\eta})=\frac{\eta^m\eta_r}{r^m},\qquad  |\nabla_{\boldsymbol{y}}\boldsymbol{\eta}|^2=\eta_r^2+\frac{m\eta^2}{r^2},\qquad |\cA|^2=\frac{1}{\eta_r^2}+\frac{mr^2}{\eta^2},
\end{equation*}
it thus follows from \eqref{b1} that the flow map $\boldsymbol{\eta}(t,\cdot)\!:\Omega\to\Omega(t)$ is injective.  
Moreover, $(\eta,\eta_r,\frac{\eta}{r})\in C^1([0,T];C(\bar I))$ ensures that 
$\boldsymbol{\eta}\in C^1([0,T]\times \overline\Omega)$. By the inverse function theorem, for each fixed $t\in[0,T]$, $\boldsymbol{\eta}(t,\cdot)$ is a diffeomorphism from $\Omega$ onto its image $\Omega(t)$. Thus, the inverse map $\boldsymbol{\eta}_*:\Omega(t)\to\Omega$ is well‑defined and belongs to $C^1(\overline{\EE(T)})$, where $\EE(T)$ is defined by
\begin{equation*}
\mathbb{E}(T)=\big\{(t,\boldsymbol{x}) : \,  t\in (0,T], \ \boldsymbol{x}\in \overline\Omega(t)\big\}.
\end{equation*}
In particular, $\boldsymbol{\eta}_*$ satisfies the following relations:
\begin{equation}\label{A!!!}
\nabla\boldsymbol{\eta}_*(t,\boldsymbol{x})=\cA(t,\boldsymbol{\eta}_*(t,\boldsymbol{x})),\qquad (\boldsymbol{\eta}_*)_t(t,\boldsymbol{x})=-(\cA \boldsymbol{U})(t,\boldsymbol{\eta}_*(t,\boldsymbol{x})).
\end{equation}

\smallskip
\textbf{2.} We show the spherical symmetry of $\boldsymbol{\eta}_*$. For any $\boldsymbol{x}\in \Omega(t)$, there exists $\boldsymbol{y}\in \Omega$ such that $\boldsymbol{\eta}(t,\boldsymbol{y})=\boldsymbol{x}$. It then follows from \eqref{danwei} that, for any $\cO\in \mathrm{O}(n)$, 
\begin{equation*}
\begin{aligned}
(\cO\boldsymbol{\eta}_*)(t,\boldsymbol{x})&=(\cO\boldsymbol{\eta}_*)(t,\boldsymbol{\eta}(t,\boldsymbol{y}))=\cO\boldsymbol{y},\\
\boldsymbol{\eta}_*(t,\cO\boldsymbol{x})&= \boldsymbol{\eta}_* (t,(\cO\boldsymbol{\eta})(t,\boldsymbol{y}))= \boldsymbol{\eta}_* (t, \boldsymbol{\eta} (t,\cO\boldsymbol{y}))=\cO\boldsymbol{y}.
\end{aligned}
\end{equation*}
Hence, for any $\boldsymbol{x}\in \Omega(t)$ and $\cO\in \mathrm{O}(n)$, we have
\begin{equation*}
(\cO\boldsymbol{\eta}_*)(t,\boldsymbol{x})=\boldsymbol{\eta}_*(t,\cO\boldsymbol{x}),   
\end{equation*}
which, along with Lemma \ref{duichen-dengjia}, yields \eqref{forms}.

\smallskip
\textbf{3.} Now, set $(\rho,\boldsymbol{u})(t,\boldsymbol{x})$ as in \eqref{equ1.28}. Clearly, by \eqref{forms}, we can directly derive 
\begin{equation} 
\begin{aligned}
\varrho(t, \boldsymbol{\eta}_*(t,\boldsymbol{x}))&=\varrho(t, |\boldsymbol{\eta}_*(t,\boldsymbol{x})|)=\varrho(t,\eta _*(t,x)),\\
\boldsymbol{U}(t, \boldsymbol{\eta}_*(t,\boldsymbol{x}))&=U(t,|\boldsymbol{\eta}_*(t,\boldsymbol{x})|)\frac{\boldsymbol{\eta}_*(t,\boldsymbol{x})}{|\boldsymbol{\eta}_*(t,\boldsymbol{x})|}=U(t,\eta_*(t,x))\frac{\boldsymbol{x}}{x},
\end{aligned}
\end{equation}
and hence $(\rho,\boldsymbol{u})(t,\boldsymbol{x})$ satisfies \eqref{ss-ass}. 

Next, we show that $(\rho(t,\boldsymbol{x}),\boldsymbol{u} (t,\boldsymbol{x}),\partial\Omega(t))$ solves problem \eqref{eq:1.1-vfbp}. Indeed, from \eqref{A!}--\eqref{A!!} and \eqref{A!!!}, we have
\begin{equation}\label{QQ30}
\begin{aligned}
\rho_t(t,\boldsymbol{x})&=(\varrho_t- \boldsymbol{U} \cdot\nabla_\cA \varrho)(t, \boldsymbol{\eta}_*(t,\boldsymbol{x}))=\big(\varrho_t-U\frac{\varrho_r}{\eta_r}\big)(t,\eta_*(t,x)),\\
\dive(\rho \boldsymbol{u})(t,\boldsymbol{x})&=\dive_\cA(\varrho \boldsymbol{U})(t, \boldsymbol{\eta}_*(t,\boldsymbol{x}))=\big(\frac{(\varrho U)_r}{\eta_r}+\frac{m\varrho U}{\eta}\big)(t,\eta_*(t,x)),
\end{aligned}
\end{equation}
which, along with $\eqref{eq:VFBP-La}_1$, leads to
\begin{equation}\label{masss}
\rho_t+\dive(\rho \boldsymbol{u})=0\qquad \text{in }\Omega(t).
\end{equation}

Similarly, we can also obtain
\begin{equation*}
\begin{aligned}
\rho\boldsymbol{u}_t(t,\boldsymbol{x})&=(\varrho \boldsymbol{U}_t- \varrho\boldsymbol{U} \cdot\nabla_\cA\boldsymbol{U})(t, \boldsymbol{\eta}_*(t,\boldsymbol{x}))=\big(\varrho U_t- \varrho\frac{UU_r}{\eta_r}\big)(t,\eta_*(t,x))\frac{\boldsymbol{x}}{|\boldsymbol{x}|},\\
(\rho\boldsymbol{u}\cdot\nabla \boldsymbol{u})(t,\boldsymbol{x})&=(\varrho\boldsymbol{U} \cdot\nabla_\cA\boldsymbol{U})(t, \boldsymbol{\eta}_*(t,\boldsymbol{x}))=\big(\varrho\frac{UU_r}{\eta_r}\big)(t,\eta_*(t,x))\frac{\boldsymbol{x}}{|\boldsymbol{x}|},\\
\dive(\rho D(\boldsymbol{u}))(t,\boldsymbol{x})&=\frac{1}{2}\dive_{\cA}\big(\varrho (\nabla_\cA\boldsymbol{U}+(\nabla_\cA\boldsymbol{U})^\top)\big) (t, \boldsymbol{\eta}_*(t,\boldsymbol{x})) \\
&=\Big(\frac{1}{\eta_r}\big(\varrho (\frac{U_r}{\eta_r}+ \frac{mU}{\eta})\big)_r - m\frac{\varrho_r U}{\eta\eta_r}\Big)(t,\eta_*(t,x))\frac{\boldsymbol{x}}{|\boldsymbol{x}|},\\
\nabla \rho^\gamma(t,\boldsymbol{x})&=(\nabla_\cA \varrho^\gamma)(t, \boldsymbol{\eta}_*(t,\boldsymbol{x}))=\big(\frac{(\varrho^\gamma)_r}{\eta_r}\big)(t,\eta_*(t,x))\frac{\boldsymbol{x}}{|\boldsymbol{x}|},
\end{aligned}
\end{equation*}
which, along with $\eqref{eq:VFBP-La}_2$ and \eqref{masss}, leads to
\begin{equation}
(\rho \boldsymbol{u})_t+\dive(\rho \boldsymbol{u}\otimes \boldsymbol{u})+A\nabla\rho^\gamma=2\mu \dive(\rho D(\boldsymbol{u}))\qquad \text{in }\Omega(t).
\end{equation}

Finally, since $\partial\Omega(t)=\{\boldsymbol{x}: \,\boldsymbol{x}=\boldsymbol{\eta}(t,\boldsymbol{y}),\,\boldsymbol{y}\in\partial\Omega\}$, for any $\boldsymbol{y}_0\in \partial\Omega$, there exists a unique point $\boldsymbol{x}_0\in \partial\Omega(t)$. Thus, we have
\begin{equation}
\mathcal{V}(\partial\Omega(t))|_{\boldsymbol{x}=\boldsymbol{x}_0}=\boldsymbol{\eta}_t(t,\cdot)\cdot\boldsymbol{N}|_{\boldsymbol{y}=\boldsymbol{y}_0}=\boldsymbol{U}(t,\cdot)\cdot\boldsymbol{N}|_{\boldsymbol{y}=\boldsymbol{y}_0}=(\boldsymbol{u}\cdot\boldsymbol{n})(t,\cdot)|_{\boldsymbol{x}=\boldsymbol{x}_0},
\end{equation}
where $(\boldsymbol{N},\boldsymbol{n})$ denote the exterior unit normal vectors to $(\partial\Omega,\partial\Omega(t))$, respectively, that is,
\begin{equation}
\boldsymbol{N}=\boldsymbol{N}(\boldsymbol{y})=\frac{\boldsymbol{y}}{r},\qquad \boldsymbol{n}=\boldsymbol{N}(\boldsymbol{\eta}_*(t,\boldsymbol{x}))=\frac{\boldsymbol{\eta}_*(t,\boldsymbol{x})}{|\boldsymbol{\eta}_*(t,\boldsymbol{x})|}=\frac{\boldsymbol{x}}{x}.
\end{equation}

\smallskip
\textbf{4.} We show that $(\rho(t,\boldsymbol{x}),\boldsymbol{u}(t,\boldsymbol{x}),\partial\Omega(t))$ is a solution described by (ii) of Theorem \ref{theorem1.3}. First, we can derive from Lemma \ref{lemma-initial}, \eqref{eq:eta}, 
and the regularity of $(U,\eta)$ that
\begin{equation}\label{D111}
(\varrho, \varrho_t, \boldsymbol{U},\nabla_{\boldsymbol{y}}\boldsymbol{U},\nabla_{\boldsymbol{y}}^2 \boldsymbol{U},\boldsymbol{U}_t,\boldsymbol{\eta},\nabla_{\boldsymbol{y}}\boldsymbol{\eta},\nabla_{\boldsymbol{y}}^2\boldsymbol{\eta}) \in C((0,T];C(\overline\Omega^\sharp)),
\end{equation}
where $\Omega^\sharp:=\{\boldsymbol{y}: \,\frac{1}{2}\leq |\boldsymbol{y}|<1\}$.

Next, to derive the time continuity of $(\varrho,\boldsymbol{U},\boldsymbol{\eta})$ in $(0,T]\times\overline\Omega^\flat$ ($\Omega^\flat:=\Omega \backslash \Omega^\sharp$), we can follow the same argument as in Step 6.1 of \S \ref{subsection3.3} to derive
\begin{equation*}
\boldsymbol{U}\in C((0,T];H^3(\Omega^\flat)),\qquad (\nabla_{\boldsymbol{y}}^2\boldsymbol{U},\boldsymbol{U}_t)\in C((0,T];W^{1,4}(\Omega^\flat)).
\end{equation*}
Then, by the classical Sobolev embeddings: 
\begin{equation*}
H^3(\Omega^\flat)\into C^1(\overline\Omega^\flat),\qquad W^{1,4}(\Omega^\flat)\into C(\overline\Omega^\flat),
\end{equation*}
it follows from \eqref{eq:eta} and the regularity of $(U,\eta)$ that
\begin{equation}\label{D112}
(\varrho, \varrho_t, \boldsymbol{U},\nabla_{\boldsymbol{y}}\boldsymbol{U},\nabla_{\boldsymbol{y}}^2 \boldsymbol{U},\boldsymbol{U}_t,\boldsymbol{\eta},\nabla_{\boldsymbol{y}}\boldsymbol{\eta},\nabla_{\boldsymbol{y}}^2\boldsymbol{\eta}) \in C((0,T];C(\overline\Omega^\flat)).
\end{equation}
Therefore, \eqref{D111}--\eqref{D112} lead to
\begin{equation}\label{D113}
(\varrho, \varrho_t, \boldsymbol{U},\nabla_{\boldsymbol{y}}\boldsymbol{U},\nabla_{\boldsymbol{y}}^2 \boldsymbol{U},\boldsymbol{U}_t,\boldsymbol{\eta},\nabla_{\boldsymbol{y}}\boldsymbol{\eta},\nabla_{\boldsymbol{y}}^2\boldsymbol{\eta}) \in C((0,T];C(\overline\Omega)).
\end{equation}

Finally, \eqref{D113}, together with the identity  
\begin{equation*}
\frac{\partial\cA_{ij}}{\partial y_k}=-\sum_{p,\ell =1}^n\cA_{ip}\frac{\partial^2 \eta_p}{\partial y_k\partial y_\ell} \cA_{\ell j} \qquad\text{for $k=1,\cdots\!,n$},
\end{equation*}
ensures the following regularity of $(\boldsymbol{\eta}_*,\cA)$:
\begin{equation}\label{xinxin}
\boldsymbol{\eta}_*\in C(\EE(T)),\qquad(\cA,\nabla\cA) \in C((0,T];C(\overline\Omega)).
\end{equation}
Consequently, it follows from the following relations:
\begin{equation*}
\begin{aligned}
\dot\rho(t,\boldsymbol{x})&\!:=(\rho_t+\boldsymbol{u}\cdot\nabla\rho)(t,\boldsymbol{x})= \varrho_t (t, \boldsymbol{\eta}_*(t,\boldsymbol{x})),\\[6pt]
\boldsymbol{u}_t(t,\boldsymbol{x})&=(\boldsymbol{U}_t-\boldsymbol{U} \cdot\nabla_\cA\boldsymbol{U})(t, \boldsymbol{\eta}_*(t,\boldsymbol{x})),\\[6pt]
(\nabla \boldsymbol{u}) (t,\boldsymbol{x})&=(\nabla_\cA \boldsymbol{U})(t, \boldsymbol{\eta}_*(t,\boldsymbol{x})),\\
(\partial_{j}\partial_k u_i)(t,\boldsymbol{x})&=\sum_{p,\ell =1}^n\big(\cA_{pj}\frac{\partial\cA_{\ell k}}{\partial y_p}\frac{\partial U_i}{\partial y_\ell}+\cA_{pj}\frac{\partial^2 U_i}{\partial y_p\partial y_\ell}\cA_{\ell k}\big)(t, \boldsymbol{\eta}_*(t,\boldsymbol{x})),
\end{aligned}
\end{equation*}
and \eqref{D113}--\eqref{xinxin} that
\begin{equation}\label{BB30}
(\rho, \dot\rho, \boldsymbol{u},\nabla\boldsymbol{u},\nabla^2 \boldsymbol{u},\boldsymbol{u}_t) \in C(\EE(T)),\qquad \partial\Omega(t)\in C^2((0,T]).
\end{equation}

\smallskip
\textbf{5.} We show that $(\rho(t,\boldsymbol{x}),\boldsymbol{u}(t,\boldsymbol{x}),\partial\Omega(t))$ is a solution described by (i) of Theorem \ref{theorem1.3}. Indeed, if moreover $\beta\leq 1$, then
\begin{equation}\label{BB*}
\nabla\rho(t,\boldsymbol{x})=\frac{1}{\beta}(\varrho^{1-\beta}\nabla_{\cA}(\varrho^\beta))(t, \boldsymbol{\eta}_*(t,\boldsymbol{x}))=\frac{1}{\beta}\big(\varrho^{1-\beta}\frac{(\varrho^\beta)_r}{\eta_r}\big)(t, \eta_*(t,x))\frac{\boldsymbol{x}}{x}.
\end{equation}
Note that, since $(\eta_r,\frac{\eta}{r})\in C((0,T];C^1(\bar I))$ and
\begin{equation}\label{BB**}
(\varrho^\beta)_r=\frac{r^{m\beta}(\rho_0^\beta)_r}{\eta^{m\beta}\eta_r^\beta}+\rho_0^\beta\big(\frac{r^{m\beta}}{\eta^{m\beta}\eta_r^\beta}\big)_r,
\end{equation}
it is direct to show that $(\varrho^\beta)_r\in C((0,T];C(\bar I))$. Moreover, from \eqref{distance-la} and Lemmas \ref{sobolev-embedding} and \ref{lemma-initial}, it follows that $\rho_0^\beta(\boldsymbol{x})\in H^3(\Omega)\into C^1(\overline\Omega)$. This, along with the spherically symmetric assumption of $\nabla\rho_0^\beta(\boldsymbol{x})$, gives $\nabla\rho_0^\beta(\boldsymbol{0})=\boldsymbol{0}$, that is,  $(\rho_0^\beta)_r|_{r=0}=0$. Thus, we can obtain from 
\begin{equation*}
\eta_{rr}|_{r=0}=(\frac{\eta}{r})_r\Big|_{r=0}=0
\end{equation*}
and \eqref{BB**} that
\begin{equation*} 
(\varrho^\beta)_r\in C((0,T];C(\bar I)),\qquad (\varrho^\beta)_r|_{r=0}=0,
\end{equation*}
which, along with \eqref{BB*}, gives
\begin{equation}
|\nabla\rho|(t,\boldsymbol{0})=0,\qquad \ \nabla\rho(t,\boldsymbol{x})\in C(\EE(T)).
\end{equation}

Finally, it follows from the above, \eqref{QQ30}, and \eqref{BB30} that $\rho_t(t,\boldsymbol{x})\in C(\EE(T))$. Hence, $(\rho(t,\boldsymbol{x}),\boldsymbol{u}(t,\boldsymbol{x}),\partial\Omega(t))$ is a classical solution of \textbf{VFBP} \eqref{eq:1.1-vfbp} in $\overline{\EE(T)}$.
\end{proof}

\section{Remarks on the Energy Functionals and the Initial Condition}\label{AppB}
This appendix is devoted to giving some equivalent forms of the energy functionals, and the initial condition  \eqref{a2} in terms of $(\rho_0,u_0)$ themselves and their spatial derivatives. In what follows, we always let $(\rho_0,u_0)$ be the initial data of {\rm\bf IBVP} \eqref{eq:VFBP-La-eta},  $\rho_0$ satisfy \eqref{distance-la} for some $\beta\in (\frac{1}{3},\gamma-1]$, and $u_0$ satisfy \eqref{a2}.

\subsection{Some equivalent forms of the energy functionals}

Define 
\begin{equation}\label{E-1a}
\begin{aligned}
\mathring\cE (t,f)&=\mathring\cE_{\mathrm{in}}(t,f)+\mathring\cE_{\mathrm{ex}}(t,f),\\
\mathring\cE_{\mathrm{in}}(t,f)&:=\Big|\zeta r^\frac{m}{2}\big(f,f_r,\frac{f}{r},f_t,f_{tr},\frac{f_t}{r}\big)(t)\Big|_{2}^2+\Big|\zeta r^\frac{m}{2}\Big(f_{rr}, \big(\frac{f}{r}\big)_r,f_{rrr}, \big(\frac{f}{r}\big)_{rr},\frac{1}{r}\big(\frac{f}{r}\big)_r \Big)(t)\Big|_{2}^2,\\
\mathring\cE_{\mathrm{ex}}(t,f)&:=\big|\chi^\sharp\rho_0^\frac{1}{2}(f,f_r,f_t,f_{tr})(t)\big|_{2}^2+\big|\chi^\sharp\rho_0^{(\frac{3}{2}-\varepsilon_0)\beta}(f_{rr},f_{rrr})(t)\big|_{2}^2,
\end{aligned}
\end{equation}
and
\begin{equation}\label{D-1a}
\begin{aligned}
\mathring\cD(t,U)&=\mathring\cD_{\mathrm{in}}(t,U)+\mathring\cD_{\mathrm{ex}}(t,U),\\
\mathring\cD_{\mathrm{in}}(t,f)&:=\Big|\zeta r^\frac{m}{2}\Big(f_{tt}, f_{trr},\big(\frac{f_{t}}{r}\big)_r,f_{rrrr},\big(\frac{f}{r}\big)_{rrr},\big(\frac{1}{r}(\frac{f}{r})_r\big)_r \Big)(t)\Big|_{2}^2,\\
\mathring\cD_{\mathrm{ex}}(t,f)&:=\big|\chi^\sharp\rho_0^\frac{1}{2}f_{tt}(t)\big|_{2}^2+\big|\chi^\sharp\rho_0^{(\frac{3}{2}-\varepsilon_0)\beta}(f_{trr},f_{rrrr})(t)\big|_{2}^2.
\end{aligned}
\end{equation}

Clearly, we have
\begin{equation*}
\mathring\cX (0,f)=\cX(0,f) \qquad\text{for $\cX=\cE,\cE_{\mathrm{in}},\cE_{\mathrm{ex}},\cD,\cD_{\mathrm{in}},\cD_{\mathrm{ex}}$}.
\end{equation*}
Also, recall the definition of $\mathring{\mathcal{E}}(f)$ in \eqref{chongqi-def}.

\begin{Lemma}\label{lemma-gaowei}
Let $T>0$, and let $\eta$ be defined by $\eqref{eq:VFBP-La-eta}_2$ with $\eta|_{t=0}=\eta_0$ and $(\eta_0,\eta,U)$ satisfy
\begin{equation}\label{C333}
\begin{gathered}
(\eta_r,\frac{\eta}{r})\in [\delta_*,\delta^*],\qquad\text{and \ \ } \eta_0(r)=r\text{ \ \ or \ \ }\mathring{\mathcal{E}}(\eta_0)=\mathrm{E}_0<\infty ,\\
\sup_{t\in[0,T]}\cE (t,U)=\mathrm{E}_1<\infty,\quad \sup_{t\in[0,T]}(\cE (t,U)+t\cD (t,U))+\int_0^T \cD(s,U)\,\ds=\mathrm{E}_2<\infty 
\end{gathered}
\end{equation}
for some given positive constants $(\delta_*,\delta^*,\mathrm{E}_0,\mathrm{E}_1,\mathrm{E}_2)$. Assume that $f=f(t,r)$ is a function defined on $[0,T]\times I$ the satisfies any one of the following conditions{\rm:}
\begin{align}
\mathrm{(C_1)}:\quad \sup_{t\in[0,T]}(\cE (t,f)+t\cD (t,f))+\int_0^T \cD(s,f)\,\ds<\infty;\label{c444}\\
\mathrm{(C_2)}:\quad \sup_{t\in[0,T]}(\mathring\cE (t,f)+t\mathring\cD (t,f))+\int_0^T\mathring\cD(s,f)\,\ds<\infty.\label{c444'}
\end{align}
Then, for any $t\in[0,T]$,
\begin{equation}\label{c5555}
\begin{aligned}
&\qquad\cE_{\star} (t,f)\sim \mathring\cE_{\star} (t,f),\\
&\sup_{s\in[0,t]}\big(\cE_{\star} (s,f)+s\cD_{\star} (s,f)\big)+\int_0^t \cD_{\star}(s,f)\,\ds\\
&\qquad \sim \sup_{s\in[0,t]}\big(\mathring\cE_{\star} (s,f)+s\mathring\cD_{\star} (s,f)\big)+\int_0^t \mathring\cD_{\star}(s,f)\,\ds, 
\end{aligned}
\end{equation}
where $\star\in \{\mathrm{in},\mathrm{ex}\}$ and $F_1\sim F_2$ denotes
\begin{equation*}
C(T)^{-1}F_1\leq F_2\leq C(T)F_1
\end{equation*}
for some finite constant $C(T)>1$, which depends only on $(\delta_*,\delta^*,\mathrm{E}_0,\mathrm{E}_1,\mathrm{E}_2,n,\rho_0,\beta,\varepsilon_0,T)$ {\rm(}if $\eta_0(r)=r$, the dependence on $\mathrm{E}_0$ should be omitted{\rm)}, such that $C(T)\to\infty$ as $T\to \infty$.

\end{Lemma}
\begin{proof}
Let $C>1$ be a generic constant, depending only on $(\delta_*,\delta^*,\mathrm{E}_1,\mathrm{E}_2,n,\rho_0,\beta,\varepsilon_0)$, which may be different at each occurrence. For simplicity, we only show that all the right-hand sides of \eqref{c5555} can be controlled by the left-hand sides, when $f$ satisfies condition \eqref{c444}. The reverse energy comparisons follow from the inverse chain rules.  Moreover, since
\begin{equation}\label{C88}
f_r=\eta_rD_\eta f,\qquad f_{tr}=\eta_rD_\eta f_t,
\end{equation}
we directly obtain 
\begin{equation}\label{C888}
\Big|\zeta r^\frac{m}{2}\big(f,f_r,\frac{f}{r},f_t,f_{tr},\frac{f_t}{r}\big)\Big|_{2}^2\leq C \cE_{\mathrm{in}} (t,f),\qquad
\big|\chi^\sharp\rho_0^\frac{1}{2}(f,f_r,f_t,f_{tr})\big|_{2}^2\leq C \cE_{\mathrm{ex}} (t,f),    
\end{equation}
and thus we only need to focus on the estimates for the higher spatial derivatives of $f$. 

We divide the proof into three steps.

\smallskip
\textbf{1. Estimates for $\eta$.} First, according to $\eqref{eq:VFBP-La-eta}_2$, we have
\begin{equation}\label{bbb1}
\begin{aligned}
\eta_{trr}&=\eta_{rr} D_{\eta}U+\eta_r^2 D_{\eta}^2U,\qquad \eta_{trrr}=\eta_{rrr} D_{\eta}U+3\eta_r\eta_{rr} D_{\eta}^2U+\eta_r^3 D_{\eta}^3U,\\
\big(\frac{\eta}{r}\big)_{tr}&=\big(\frac{\eta}{r}\big)_r\frac{U}{\eta}+\eta_r\frac{\eta}{r}D_{\eta}\big(\frac{U}{\eta}\big),\qquad \frac{1}{r}\big(\frac{\eta}{r}\big)_{tr} =\frac{1}{r}\big(\frac{\eta}{r}\big)_r\frac{U}{\eta}+\eta_r\big(\frac{\eta}{r}\big)^2\frac{1}{\eta}    D_{\eta}\big(\frac{U}{\eta}\big),\\
\big(\frac{\eta}{r}\big)_{trr}&=\big(\frac{\eta}{r}\big)_{rr}\frac{U}{\eta}+2\eta_r\big(\frac{\eta}{r}\big)_rD_\eta\big(\frac{U}{\eta}\big)+\eta_{rr}\frac{\eta}{r}D_{\eta}\big(\frac{U}{\eta}\big)+\eta_r^2\frac{\eta}{r} D_{\eta}^2\big(\frac{U}{\eta}\big).
\end{aligned}
\end{equation}
On the other hand, it follows from \eqref{C333} and Lemmas \ref{sobolev-embedding}--\ref{hardy-inequality} that
\begin{align}
&\begin{aligned}\label{bbb2-0}
&\,\Big|\big(D_\eta U,\frac{U}{\eta}\big)\Big|_\infty \leq C\Big|\Big(D_\eta U,D_\eta^2 U,\frac{U}{\eta},D_\eta\big(\frac{U}{\eta}\big)\Big)\Big|_1\\
&\leq C\sum_{j=1}^3\Big(\Big|\chi r^\frac{m}{2}\Big(D_\eta^j U,D_\eta^{j-1}\big(\frac{U}{\eta}\big)\Big)\Big|_2+\big|\chi^\sharp \rho_0^{(\frac{3}{2}-\varepsilon_0)\beta}D_\eta^j U\big|_2\Big)\leq C,
\end{aligned}\\
&\begin{aligned}\label{bbb2}
&\,\Big|\big(D_\eta^2 U,D_\eta(\frac{U}{\eta})\big)\Big|_\infty \leq C\Big|\Big(D_\eta^2 U,D_\eta^3 U,D_\eta\big(\frac{U}{\eta}\big),D_\eta^2\big(\frac{U}{\eta}\big)\Big)\Big|_1\\
&\leq C\sum_{j=2}^4\Big(\Big|\chi r^\frac{m}{2}\Big(D_\eta^j U,D_\eta^{j-1}\big(\frac{U}{\eta}\big)\Big)\Big|_2+\big|\chi^\sharp \rho_0^{(\frac{3}{2}-\varepsilon_0)\beta}D_\eta^j U\big|_2\Big)\leq C+C\cD(t,U)^\frac{1}{2}.
\end{aligned}
\end{align}
As a consequence, \eqref{bbb1}, together with \eqref{C333}, \eqref{bbb2}, and the H\"older inequality, implies that
\begin{equation}\label{bbb3}
\Big|\Big(\eta_{rr},\big(\frac{\eta}{r}\big)_{r}\Big)\Big|_\infty\leq C\mathrm{E}_0^\frac{1}{2}+ Ce^{Ct}\int_0^t \Big|\Big(D_\eta^2 U,D_{\eta}\big(\frac{U}{\eta}\big)\Big)\Big|_\infty\,\mathrm{d}s\leq C(T),
\end{equation}
and
\begin{equation*}
\begin{aligned}
&\Big|\Big(\eta_{rrr},\big(\frac{\eta}{r}\big)_{rr},\frac{1}{r}\big(\frac{\eta}{r}\big)_{r}\Big)\Big|\\
&\leq C e^{Ct}\Big|\Big((\eta_0)_{rrr},\big(\frac{\eta_0}{r}\big)_{rr},\frac{1}{r}\big(\frac{\eta_0}{r}\big)_{r}\Big)\Big|+Ce^{Ct}\Big|\Big(\eta_{rr},\big(\frac{\eta}{r}\big)_{r}\Big)\Big|_\infty\int_0^t \Big|\Big(D_\eta^2 U,D_{\eta}\big(\frac{U}{\eta}\big)\Big)\Big|\,\mathrm{d}s \\
&\quad +Ce^{Ct} \int_0^t \Big|\Big(D_\eta^3 U,D_{\eta}^2\big(\frac{U}{\eta}\big),\frac{1}{\eta}D_{\eta}\big(\frac{U}{\eta}\big)\Big)\Big|\,\mathrm{d}s\\
&\leq C(T)\Big|\Big((\eta_0)_{rrr},\big(\frac{\eta_0}{r}\big)_{rr},\frac{1}{r}\big(\frac{\eta_0}{r}\big)_{r}\Big)\Big|+C(T)\int_0^t \Big|\Big(D_\eta^3 U,D_{\eta}^2\big(\frac{U}{\eta}\big),\frac{1}{\eta}D_{\eta}\big(\frac{U}{\eta}\big)\Big)\Big|\,\mathrm{d}s,
\end{aligned}
\end{equation*}
which, along with the Minkowski inequality, leads to
\begin{equation}\label{bbb4}
\Big|\zeta r^\frac{m}{2}\Big(\eta_{rrr},\big(\frac{\eta}{r}\big)_{rr},\frac{1}{r}\big(\frac{\eta}{r}\big)_{r}\Big)\Big|_2 \leq C(T),\qquad 
\big|\chi^\sharp \rho_0^{(\frac{3}{2}-\varepsilon_0)\beta} \eta_{rrr}\big|_2\leq C(T) .
\end{equation}

\smallskip
\textbf{2. Estimates for $f$.} Now, a direct calculation gives 
\begin{equation}\label{C99}
\begin{aligned}
f_{rr}&=\eta_{rr}D_\eta f+\eta_r^2 D_\eta^2 f,\qquad f_{rrr} =\eta_{rrr}D_\eta f+3\eta_r\eta_{rr}D_\eta^2 f+\eta_r^3 D_\eta^3 f,\\
\big(\frac{f}{r}\big)_r&=\eta_r\frac{\eta}{r} D_{\eta}\big(\frac{f}{\eta}\big) +\frac{f}{\eta}\big(\frac{\eta}{r}\big)_r,\qquad \frac{1}{r}\big(\frac{f}{r}\big)_r=\eta_r\big(\frac{\eta}{r}\big)^2 \frac{1}{\eta}D_{\eta}\big(\frac{f}{\eta}\big) +\frac{f}{\eta}\frac{1}{r}\big(\frac{\eta}{r}\big)_r,\\
\big(\frac{f}{r}\big)_{rr}&=\eta_{rr}\frac{\eta}{r} D_{\eta}\big(\frac{f}{\eta}\big)+2\eta_r\big(\frac{\eta}{r}\big)_r D_{\eta}\big(\frac{f}{\eta}\big)+\eta_r^2\frac{\eta}{r} D_{\eta}^2\big(\frac{f}{\eta}\big)+\frac{f}{\eta}\big(\frac{\eta}{r}\big)_{rr},
\end{aligned}
\end{equation}
which, combined with \eqref{bbb3}, yields
\begin{equation}\label{bbb5}
\begin{aligned}
\Big|\Big(f_{rr},\big(\frac{f}{r}\big)_r\Big)\Big|&\leq C(T)\Big|\Big(D_\eta f,\frac{f}{\eta},D_\eta^2 f,D_{\eta}\big(\frac{f}{\eta}\big)\Big)\Big|,\\
\Big|\Big(f_{rrr},\big(\frac{f}{r}\big)_{rr},\frac{1}{r}\big(\frac{f}{r}\big)_r\Big)\Big|&\leq  C(T)\Big|\Big(D_\eta^2 f,D_{\eta}\big(\frac{f}{\eta}\big),D_\eta^3 f,D_{\eta}^2\big(\frac{f}{\eta}\big),\frac{1}{\eta}D_{\eta}\big(\frac{f}{\eta}\big)\Big)\Big|\\
&\quad +\Big|\Big(\eta_{rrr},\big(\frac{\eta}{r}\big)_{rr},\frac{1}{r}\big(\frac{\eta}{r}\big)_{r}\Big)\Big|\Big|\big(D_\eta f,\frac{f}{\eta}\big)\Big|.
\end{aligned}
\end{equation}

Repeating the same calculation \eqref{bbb2-0}, we have
\begin{equation}\label{fwuqiong}
\Big|\chi^\sharp\big(D_\eta f,\frac{f}{\eta}\big)\Big|_\infty\leq C\cE_{\mathrm{ex}}(t,f)^{\frac{1}{2}}.
\end{equation}
Hence, \eqref{bbb5}, together with the above and \eqref{bbb4}, implies
\begin{equation}\label{C-17}
\big|\chi^\sharp \rho_0^{(\frac{3}{2}-\varepsilon_0)\beta}(f_{rr},f_{rrr})\big|_2^2\leq C(T)\cE_{\mathrm{ex}} (t,f).
\end{equation}

Finally, we prove the interior energy comparison. Since $\zeta=1$ on $[0,\frac12]$ and $\rho_0$ is bounded away
from zero on $[\frac12,\frac58]$,
\eqref{bbb3}--\eqref{bbb4} imply
\begin{equation}\label{prec1}
\Big|\chi_{\frac{5}{8}}r^\frac{m}{2}
\Big(\eta_{rrr},\big(\frac{\eta}{r}\big)_{rr},
\frac1r\big(\frac{\eta}{r}\big)_r\Big)\Big|_2\leq C(T).
\end{equation}
Moreover, for any $a\in[\frac12,\frac58]$, it follows from \eqref{C333} and Lemmas \ref{sobolev-embedding}--\ref{hardy-inequality} that
\begin{equation}\label{prec2}
\Big|\chi_a\Big(D_\eta f,\frac f\eta\Big)\Big|_\infty\leq C(a)\Big|\chi_a r^\frac{m}{2} \Big(D_\eta f,\frac{f}{\eta},D_\eta^2 f,D_{\eta}\big(\frac{f}{\eta}\big),D_\eta^3 f,D_{\eta}^2\big(\frac{f}{\eta}\big)\Big)\Big|_2.
\end{equation}
Hence, multiplying \eqref{bbb5} by $\chi_a r^{\frac{m}{2}}$ and applying \eqref{prec1}--\eqref{prec2}, we obtain
\begin{equation*}
\begin{aligned}
&\Big|\chi_a r^\frac{m}{2}
\Big(f_{rr},\big(\frac fr\big)_r,f_{rrr},
\big(\frac fr\big)_{rr},\frac1r\big(\frac fr\big)_r\Big)\Big|_2^2\\
&\leq C(a,T)\Big|\chi_a r^\frac{m}{2} \Big(D_\eta f,D_\eta^2 f,\frac{f}{\eta},D_{\eta}\big(\frac{f}{\eta}\big),D_\eta^3 f,D_{\eta}^2\big(\frac{f}{\eta}\big),\frac{1}{\eta}D_{\eta}\big(\frac{f}{\eta}\big)\Big)\Big|_2^2.
\end{aligned}
\end{equation*}

Next, the definition of $\chi_a$ and the monotonicity of $\zeta$ give
\begin{equation}\label{tihuan1}
\zeta(r)^2=\int_{\frac{1}{2}}^{\frac{5}{8}}-(\zeta^2)'(a)\chi_a(r)\,\mathrm{d}a,
\qquad - (\zeta^2)'(a)\geq0.
\end{equation}
Thus, by changing the order of integration, for example,
\begin{equation}\label{tihuan2}
\int_{\frac{1}{2}}^{\frac{5}{8}}-(\zeta^2)'(a)
\big|\chi_a r^{\frac{m}{2}}f_{rrr}\big|_2^2\,\mathrm{d}a
=\big|\zeta r^{\frac{m}{2}}f_{rrr}\big|_2^2.
\end{equation}
We first replace the squared tuple norms in the estimate on $[0,a]$
by the equivalent sums of squared component norms. Multiplying that
estimate by $-(\zeta^2)'(a)$, integrating over $[\frac12,\frac58]$,
and applying the preceding identity to each component, we obtain
\begin{equation}\label{C-19}
\Big|\zeta r^\frac{m}{2}
\Big(f_{rr},\big(\frac fr\big)_r,f_{rrr},
\big(\frac fr\big)_{rr},\frac1r\big(\frac fr\big)_r\Big)\Big|_2^2
\leq C(T)\cE_{\mathrm{in}}(t,f).
\end{equation}
 
Collecting \eqref{C888}, \eqref{C-17}, and \eqref{C-19}, we obtain 
\begin{equation*}
\mathring\cE_{\star} (t,f) \leq C(T) \cE_{\star} (t,f) \qquad\text{with $\star\in \{\mathrm{in},\mathrm{ex}\}$}.
\end{equation*}

\smallskip
\textbf{3. Estimates for the dissipation.}
First, it follows from \eqref{C88} and \eqref{C99} that
\begin{equation*}
\begin{aligned}
f_{trr}&=\eta_{rr}D_\eta f_t+\eta_{r}^2D_\eta^2 f_t,\\
\big(\frac{f_t}{r}\big)_{r}&=\eta_r\frac{\eta}{r} D_{\eta}\big(\frac{f_t}{\eta}\big) +\frac{f_t}{\eta}\big(\frac{\eta}{r}\big)_r,
\end{aligned}
\end{equation*}
which, combined with \eqref{bbb3}, yields
\begin{equation*} 
\begin{aligned}
\Big|\Big(f_{trr},\big(\frac{f_t}{r}\big)_r\Big)\Big|&\leq C(T)\Big|\Big(D_\eta f_t,\frac{f_t}{\eta},D_\eta^2 f_t,D_{\eta}\big(\frac{f_t}{\eta}\big)\Big)\Big|.
\end{aligned}
\end{equation*}
Therefore, we can obtain 
\begin{equation}\label{C-23}
\begin{aligned}
\Big|\zeta r^\frac{m}{2}\Big(f_{trr},\big(\frac{f_t}{r}\big)_r\Big)\Big|_2^2&\leq C\big(\cE_{\mathrm{in}}(t,f)+\cD_{\mathrm{in}}(t,f)\big),\\
\big|\chi^\sharp \rho_0^{(\frac{3}{2}-\varepsilon_0)\beta}f_{trr}\big|_2^2&\leq C\big(\cE_{\mathrm{ex}}(t,f)+\cD_{\mathrm{ex}}(t,f)\big).
\end{aligned}
\end{equation}

Next, it follows from  \eqref{C99} that
\begin{equation}\label{eq:C-supp-fourth-chain}
\begin{aligned}
f_{rrrr}&=\eta_r^4D_\eta^4f+6\eta_r^2\eta_{rr}D_\eta^3f
+(3\eta_{rr}^2+4\eta_r\eta_{rrr})D_\eta^2f+\eta_{rrrr}D_\eta f,\\
\big(\frac{f}{r}\big)_{rrr}
&=\frac{\eta}{r}\eta_r^3D_\eta^3\big(\frac{f}{\eta}\big)
+3\Big(\frac{\eta}{r}\eta_r\eta_{rr}
+\big(\frac{\eta}{r}\big)_r\eta_r^2\Big)D_\eta^2\big(\frac{f}{\eta}\big)\\
&\quad+\Big(\frac{\eta}{r}\eta_{rrr}
+3\big(\frac{\eta}{r}\big)_r\eta_{rr}
+3\big(\frac{\eta}{r}\big)_{rr}\eta_r\Big)D_\eta\big(\frac{f}{\eta}\big)
+\big(\frac{\eta}{r}\big)_{rrr}\frac{f}{\eta},\\
\Big(\frac{1}{r}\big(\frac{f}{r}\big)_r\Big)_r\!\!
&=\eta_r^2\big(\frac{\eta}{r}\big)^2
D_\eta\Big(\frac{1}{\eta}D_\eta\big(\frac{f}{\eta}\big)\Big)
+\Big(\eta_{rr}\big(\frac{\eta}{r}\big)^2
+3\eta_r\frac{\eta}{r}\big(\frac{\eta}{r}\big)_r\Big)
\frac{1}{\eta}D_\eta\big(\frac{f}{\eta}\big)
+\Big(\frac{1}{r}\big(\frac{\eta}{r}\big)_r\Big)_r\frac{f}{\eta}.
\end{aligned}
\end{equation}
Replacing $f$ by $U=\eta_t$ in these identities, we obtain equations
for the fourth-order derivatives of $\eta$. The terms containing
the fourth-order derivatives of $\eta$ are linear, with coefficients bounded by
$|(D_\eta U,\frac{U}{\eta})|_\infty$. Moreover, by \eqref{bbb2-0}--\eqref{bbb4},
the remaining terms are bounded in the corresponding weighted
$L^2$ spaces by $C(T)(1+\cD(t,U)^{\frac{1}{2}})$. Thus, we obtain from \eqref{C333} and the Gr\"onwall and H\"older inequalities that
\begin{equation}\label{eq:C-supp-flow}
\Big|\zeta r^\frac{m}{2}
\Big(\eta_{rrrr},\big(\frac{\eta}{r}\big)_{rrr},
\big(\frac{1}{r}(\frac{\eta}{r})_r\big)_r\Big)\Big|_2
\leq C(T),\qquad
\big|\chi^\sharp\rho_0^{(\frac32-\varepsilon_0)\beta}
\eta_{rrrr}\big|_2\leq C(T).
\end{equation}

Now, by \eqref{C333} and \eqref{bbb3}, we obtain from \eqref{eq:C-supp-fourth-chain} that
\begin{equation}\label{wai4}
|f_{rrrr}|\leq C(T)\big|\big(D_\eta^4f,D_\eta^3f,D_\eta^2f\big)\big|+ C(T) |\eta_{rrr}||D_\eta^2f|+C(T)|\eta_{rrrr}||D_\eta f|,
\end{equation}
\begin{align}
\Big|\Big(f_{rrrr},\big(\frac{f}{r}\big)_{rrr},\big(\frac{1}{r}(\frac{f}{r})_r\big)_r\Big)\Big| &\leq C(T) \Big(\sum_{j=2}^4\Big|\Big(D_\eta^j f,D_\eta^{j-1}\big(\frac{f}{\eta}\big)\Big)\Big|+ \sum_{k=0}^1\Big|D_\eta^k\Big(\frac{1}{\eta}D_{\eta}\big(\frac{f}{\eta}\big)\Big) \Big|\Big)\notag\\
&\quad +C(T) \Big|\Big(\eta_{rrr},\big(\frac{\eta}{r}\big)_{rr}\Big)\Big| \Big|\Big(D_\eta^2f,D_\eta\big(\frac{f}{\eta}\big)\Big)\Big|\label{nei4}\\
&\quad + C(T)\Big|\Big(\eta_{rrrr},\big(\frac{\eta}{r}\big)_{rrr},\big(\frac{1}{r}(\frac{\eta}{r})_r\big)_r\Big)\Big|\Big|\big(D_\eta f,\frac{f}{\eta}\big)\Big|.\notag
\end{align}

For the exterior estimates, repeating the same calculation \eqref{bbb2-0}, we have
\begin{equation}\label{fwuqiong2}
\Big|\chi^\sharp\Big(D_\eta^2 f,D_\eta\big(\frac{f}{\eta}\big)\Big)\Big|_\infty\leq C\big(\cE_{\mathrm{ex}}(t,f)^{\frac{1}{2}}+\cD_{\mathrm{ex}}(t,f)^{\frac{1}{2}}\big).
\end{equation}
Hence, \eqref{wai4}, together with \eqref{bbb4}, \eqref{fwuqiong}, \eqref{eq:C-supp-flow}, and \eqref{fwuqiong2}, implies
\begin{equation}\label{C-29}
\big|\chi^\sharp \rho_0^{(\frac{3}{2}-\varepsilon_0)\beta} f_{rrrr}\big|_2^2\leq C(T)\big(\cD_{\mathrm{ex}} (t,f)+\cE_{\mathrm{ex}} (t,f)\big).
\end{equation}

For the interior estimates, for $\frac12\leq a\leq\frac58$, it follows from Lemmas \ref{sobolev-embedding}--\ref{hardy-inequality} that
\begin{equation*}
\Big|\chi_a\Big(D_\eta^2f,
D_\eta\big(\frac{f}{\eta}\big)\Big)\Big|_\infty \leq C(a)\sum_{j=2}^4\Big|\chi_a r^\frac{m}{2}\Big(D_\eta^j f,D_\eta^{j-1}\big(\frac{f}{\eta}\big)\Big)\Big|_2 .
\end{equation*}
Thus, \eqref{nei4}, together with \eqref{bbb3}--\eqref{bbb4}, \eqref{fwuqiong}, and
\eqref{eq:C-supp-flow}, gives
\begin{equation*}
\begin{aligned}
&\Big|\chi_a r^\frac{m}{2}\Big(f_{rrrr},
\big(\frac{f}{r}\big)_{rrr},
\Big(\frac{1}{r}\big(\frac{f}{r}\big)_r\Big)_r\Big)\Big|_2\\
&\leq C(a,T)\sum_{j=1}^4\Big|\chi_a r^\frac{m}{2}\Big(D_\eta^j f,D_\eta^{j-1}\big(\frac{f}{\eta}\big)\big)\Big|_2+C(a,T)\sum_{k=0}^1\Big|\chi_a r^\frac{m}{2} D_\eta^k\Big(\frac{1}{\eta}D_{\eta}\big(\frac{f}{\eta}\big)\Big) \Big|_2.
\end{aligned}
\end{equation*}
Repeating the same argument \eqref{tihuan1}--\eqref{C-19} in Step 2, we obtain
\begin{equation}\label{C-30}
\Big|\zeta r^\frac{m}{2}
\Big(f_{rrrr},
\big(\frac{f}{r}\big)_{rrr},
\Big(\frac{1}{r}\big(\frac{f}{r}\big)_r\Big)_r\Big)\Big|_2^2
\leq C(T)\big(\cE_{\mathrm{in}}(t,f)+\cD_{\mathrm{in}}(t,f)\big).
\end{equation}

Collecting \eqref{C888}, \eqref{C-17}, \eqref{C-19}, \eqref{C-23}, and \eqref{C-29}--\eqref{C-30}, we obtain 
\begin{equation*}
\begin{aligned}
&\sup_{s\in[0,t]}\big(\mathring\cE_{\star} (s,f)+s\mathring\cD_{\star} (s,f)\big)+\int_0^t \mathring\cD_{\star}(s,f)\,\ds\\
&\leq C(T)\Big(\sup_{s\in[0,t]}\big(\cE_{\star} (s,f)+s\cD_{\star} (s,f)\big)+\int_0^t \cD_{\star}(s,f)\,\ds\Big),
\end{aligned}
\end{equation*} 
with $\star\in \{\mathrm{in},\mathrm{ex}\}$. This completes the proof.
\end{proof}

\subsection{Some equivalent forms of the initial condition}
First, according to the time evolution equation of $U$ in \eqref{eq:VFBP-La-eta},  all the desired initial values of time derivatives of $U$ in \eqref{a2} can be completely expressed by those of $(\rho_0,u_0)$ themselves and their spatial derivatives:
\begin{equation}\label{116}
\begin{aligned}
U_t(0,r)&=2\mu \frac{1}{\rho_0}(\rho_0 (u_0)_r)_r+2\mu m (\frac{u_0}{r})_r-\frac{A\gamma}{\gamma-1}(\rho_0^{\gamma-1})_r,\\
U_{tr}(0,r)&=2\mu \big(\frac{1}{\rho_0}(\rho_0 (u_0)_r)_r+ m (\frac{u_0}{r})_r\big)_r-\frac{A\gamma}{\gamma-1}(\rho_0^{\gamma-1})_{rr}. 
\end{aligned}
\end{equation}

First, we show that the initial condition \eqref{a2} implicitly contains 
the Neumann boundary condition of $u_0$.
\begin{Lemma}\label{mewman}
If $\cE_{\mathrm{ex}}(0,U)<\infty$, $u_0$ satisfies the Neumann boundary condition $(u_0)_r|_{r=1}=0$.
\end{Lemma}
\begin{proof}
If $\cE_{\mathrm{ex}}(0,U)<\infty$, then $U_t(0,r)\in H^1_{\rho_0}(\frac{1}{2},1)$. Due to the facts that
\begin{equation*}
\rho_0^\beta\sim 1-r,\qquad \varepsilon_0<\frac{1}{2},\qquad \big(\frac{3}{2}-\varepsilon_0\big)\beta>\frac{1}{2},
\end{equation*}
and Lemma \ref{hardy-inequality}, we have
\begin{equation}\label{b15}
\rho_0^{(1-\varepsilon_0)\beta} U_t(0)\in C\big[\frac{1}{2},1\big]\implies \rho_0^\beta U_t(0)|_{r=1}=0. 
\end{equation}
Similarly, since 
\begin{equation*}
\rho_0^{(\frac{3}{2}-\varepsilon_0)\beta}\big((u_0)_{rr},(\frac{u_0}{r})_r,(u_0)_{rrr},(\frac{u_0}{r})_{rr}\big)\in L^2\big(\frac{1}{2},1\big),
\end{equation*}
we also deduce
\begin{equation}\label{b16}
\rho_0^{(1-\varepsilon_0)\beta} \big((u_0)_{rr},(\frac{u_0}{r})_r\big)\in C\big[\frac{1}{2},1\big]\implies \rho_0^\beta (u_0)_{rr}|_{r=1}=0,\quad \rho_0^\beta(\frac{u_0}{r})_r|_{r=1}=0. 
\end{equation}

Now, multiplying $\eqref{116}_1$ by $\rho_0^\beta$, we have 
\begin{equation*}
\frac{2\mu}{\beta}(\rho_0^\beta)_r(u_0)_r=\rho_0^\beta U_t(0,r)-2\mu \rho_0^\beta \big((u_0)_r+\frac{mu_0}{r}\big)_r+\frac{A\gamma}{\beta}\rho_0^{\gamma-1}(\rho_0^\beta)_r.
\end{equation*}
Then, taking the limit as $r\to 1$, we find that, by \eqref{b15}--\eqref{b16}, the right-hand side of the above vanishes at $\{r=1\}$, which, along with the fact that $(\rho_0^\beta)_r|_{r=1}\neq 0$, yields
\begin{equation*}
(\rho_0^\beta)_r(u_0)_r|_{r=1}=0\implies (u_0)_r|_{r=1}=0.
\end{equation*}
\end{proof}

\begin{Lemma}\label{reduce-ex}
$\cE_\mathrm{ex}(0,U)<\infty$ if and only if
\begin{equation}\label{con-b3}
\begin{aligned}
&(u_0)_r|_{r=1}=0,\qquad (u_0)_r\in L^2_{\rho_0^{1-2\beta}}\big(\frac{1}{2},1\big),\qquad u_0\in H^2_{\rho_0}\big(\frac{1}{2},1\big),\\
&(u_0)_{rrr}\in L^2_{\rho_0^{(3-2\varepsilon_0)\beta}}\big(\frac{1}{2},1\big),\qquad 2\mu \big(\frac{1}{\rho_0}(\rho_0 (u_0)_r)_r\big)_r-\frac{A\gamma}{\gamma-1}(\rho_0^{\gamma-1})_{rr} \in L^2_{\rho_0}\big(\frac{1}{2},1\big).  
\end{aligned}
\end{equation}
Additionally, if $\beta=\gamma-1$ or $\beta<\frac{2\gamma-1}{5}$, then the above can be reduced to 
\begin{equation}\label{con-b4}
\begin{aligned}
&(u_0)_r|_{r=1}=0,\qquad (u_0)_r\in L^2_{\rho_0^{1-2\beta}}\big(\frac{1}{2},1\big),\qquad u_0\in H^2_{\rho_0}\big(\frac{1}{2},1\big),\\
&(u_0)_{rrr}\in L^2_{\rho_0^{(3-2\varepsilon_0)\beta}}\big(\frac{1}{2},1\big),\qquad  \big(\frac{1}{\rho_0}(\rho_0 (u_0)_r)_r\big)_r \in L^2_{\rho_0}\big(\frac{1}{2},1\big).  
\end{aligned}
\end{equation}
\end{Lemma}
\begin{proof}
The proof of \eqref{con-b4} is straightforward, since 
\begin{equation*}
(\rho_0^{\gamma-1})_{rr}=\frac{\gamma-1}{\beta}\frac{\gamma-1-\beta}{\beta}\rho_0^{\gamma-1-2\beta}|(\rho_0^\beta)_r|^2+\frac{\gamma-1}{\beta}\rho_0^{\gamma-1-\beta}(\rho_0^\beta)_{rr},
\end{equation*}
and $(\rho_0^{\gamma-1})_{rr}$ belongs to $L^2_{\rho_0}(\frac{1}{2},1)$ whenever $\beta=\gamma-1$ or $\beta<\frac{2\gamma-1}{5}$.

Therefore, it suffices to prove \eqref{con-b3}. We divide the proof into two steps.

\smallskip
\textbf{1.} We first prove the necessity. Assume that $\cE_\mathrm{ex}(0,U)<\infty$. Then we have 
\begin{equation}\label{b18}
(u_0,U_t(0))\in H^1_{\rho_0}\big(\frac{1}{2},1\big), \qquad (u_0)_{rr}\in H^1_{\rho_0^{(3-2\varepsilon_0)\beta}}, 
\end{equation}
and $(u_0)_r|_{r=1}=0$ follows from Lemma \ref{mewman}. 

Next, we show that $\rho_0^{\frac{1}{2}-\beta}(u_0)_r\in L^2 (\frac{1}{2},1)$. Thanks to  \eqref{b18} and 
\begin{equation*}
(u_0)_r|_{r=1}=0,\qquad \rho_0^\beta\sim 1-r,\qquad \rho_0^\beta\in H^3(\frac{1}{2},1),\qquad \beta>\frac{1}{3},
\end{equation*}
multiplying $\eqref{116}_1$ by $\rho_0$ and integrating  over $[r,1]$ with $r\in [\frac{1}{2},1]$, we obtain from Lemma \ref{hardy-inequality} and the H\"older inequality that 
\begin{equation*}
\begin{aligned}
2\mu \big|\rho_0^{\frac{1}{2}-\beta}(u_0)_r\big|&= \Big|-A\rho_0^{\gamma-\frac{1}{2}-\beta}+\rho_0^{-\frac{1}{2}-\beta}\int_r^1 \rho_0\big(U_t(0,\tilde r)-2\mu m\frac{(u_0)_r}{\tilde{r}}+2\mu m\frac{u_0}{\tilde{r}^2}\big)\,\mathrm{d}\tilde{r}\Big|\\
&\leq C+ C\rho_0^{\frac{1}{4}-\frac{\beta}{2}}\big|\chi^\sharp\rho_0^\frac{1}{4}(u_0,(u_0)_r,U_t(0))\big|_2 \\
&\leq C+ C\rho_0^{\frac{1}{4}-\frac{\beta}{2}}\Big(\sum_{j=0}^3 \big|\chi^\sharp\rho_0^{\frac{1}{4}+2\beta}\partial_r^ju_0 \big|_2 +  \big|\chi^\sharp\rho_0^{\frac{1}{4}+\beta}(U_t,U_{tr})(0)\big|_2\Big)\\
&\leq C+ C\rho_0^{\frac{1}{4}-\frac{\beta}{2}}\cE_{\mathrm{ex}}(0,U)^\frac{1}{2}.
\end{aligned}
\end{equation*}
Since $\rho_0^{\frac{1}{4}-\frac{\beta}{2}}\in L^2(\frac{1}{2},1)$, we derive from the above that $\rho_0^{\frac{1}{2}-\beta}(u_0)_r\in L^2(\frac{1}{2},1)$. 

Finally, based on this, we can obtain from $\eqref{116}_1$ that 
\begin{equation*} 
U_t(0)\in L^2_{\rho_0}\big(\frac{1}{2},1\big)\implies \rho_0^{-\frac{1}{2}}(\rho_0 (u_0)_r)_r\in L^2\big(\frac{1}{2},1\big)\implies \rho_0^\frac{1}{2}(u_0)_{rr}\in L^2\big(\frac{1}{2},1\big),
\end{equation*}
and thus obtain from $\eqref{116}_2$ that 
\begin{equation}\label{b19}
\begin{aligned}
U_{tr}(0)\in L^2_{\rho_0}\big(\frac{1}{2},1\big)&\implies 2\mu \big(\frac{1}{\rho_0}(\rho_0 (u_0)_r)_r+ m (\frac{u_0}{r})_r\big)_r-\frac{A\gamma}{\gamma-1}(\rho_0^{\gamma-1})_{rr}\in L^2_{\rho_0}\big(\frac{1}{2},1\big),\\
&\implies 2\mu \big(\frac{1}{\rho_0}(\rho_0 (u_0)_r)_r\big)_r -\frac{A\gamma}{\gamma-1}(\rho_0^{\gamma-1})_{rr}\in L^2_{\rho_0}\big(\frac{1}{2},1\big).
\end{aligned}
\end{equation}

This completes the proof of necessity.

\smallskip
\textbf{2.} We prove the sufficiency. Assume that \eqref{con-b3} holds. It suffices to show that
\begin{equation}
U_t(0)\in H^1_{\rho_0}\big(\frac{1}{2},1\big).
\end{equation}
$U_t(0)\in L^2_{\rho_0}\big(\frac{1}{2},1\big)$ is obvious under \eqref{con-b3}.
Moreover, we obtain from $\eqref{116}_2$ that
\begin{equation}
U_{tr}(0)\in L^2_{\rho_0}\big(\frac{1}{2},1\big)\iff (\frac{u_0}{r})_{rr}\in L^2_{\rho_0}\big(\frac{1}{2},1\big)\impliedby u_0\in H^2_{\rho_0}\big(\frac{1}{2},1\big).
\end{equation}

This completes the proof of sufficiency.
\end{proof}

\begin{Lemma}\label{reduce-in}
$\cE_\mathrm{in}(0,U)<\infty$ if and only if 
\begin{equation}\label{con25}
\zeta r^\frac{m}{2}\big(u_0,(u_0)_r,\frac{u_0}{r},(u_0)_{rr},(\frac{u_0}{r})_r,(u_0)_{rrr},(\frac{u_0}{r})_{rr},\frac{1}{r}(\frac{u_0}{r})_{r}\big)\in L^2.
\end{equation}
\end{Lemma}
\begin{proof}
In fact, we only need to show
\begin{equation}\label{B8}
\zeta r^\frac{m}{2}\big(U_t,U_{tr},\frac{U_t}{r}\big)(0)\in L^2 \qquad\text{when \eqref{con25} holds}.
\end{equation}
Indeed, it follows from  \eqref{116} and Lemma \ref{lemma-initial} that 
\begin{equation*}
\begin{aligned}
|\zeta r^\frac{m}{2}U_t(0)|_2&\leq C \big(\big|\zeta r^\frac{m}{2}\big((u_0)_{rr}, (\frac{u_0}{r})_r\big)\big|_2+|r^\frac{m}{2} (\rho_0^\beta)_{r}|_2 |\zeta_\frac{5}{8} \rho_0^{-\beta}|_\infty|\zeta (u_0)_r|_\infty \big)\\
& \quad + C|r^\frac{m}{2} (\rho_0^\beta)_{r}|_2|\zeta \rho_0^{\gamma-1-\beta}|_\infty \leq C,\\
|\zeta r^\frac{m}{2}U_{tr}(0)|_2&\leq C\big(\big|\zeta r^\frac{m}{2}\big((u_0)_{rrr},(\frac{u_0}{r})_{rr}\big)\big|_2+ |(\rho_0^\beta)_{r}|_\infty|\zeta_\frac{5}{8} \rho_0^{-\beta}|_\infty|\zeta r^\frac{m}{2}(u_0)_{rr}|_2\big)\\
&\quad +C |r^\frac{m}{2}(\rho_0^\beta)_{rr}|_2\big(|\zeta_\frac{5}{8} \rho_0^{-\beta}|_\infty|\zeta (u_0)_{r}|_\infty+|\zeta \rho_0^{\gamma-1-\beta}|_\infty\big) \\
&\quad + C|r^\frac{m}{2}(\rho_0^\beta)_{r}|_2|(\rho_0^\beta)_{r}|_\infty\big(|\zeta_\frac{5}{8} \rho_0^{-2\beta}|_\infty|\zeta(u_0)_r|_\infty+|\zeta \rho_0^{\gamma-1-2\beta}|_\infty\big)\leq C, \\
|\zeta r^\frac{m-2}{2} U_t(0)|_2&\leq C \big(\big|\zeta r^\frac{m-2}{2}\big((u_0)_{rr},(\frac{u_0}{r})_{r}\big)\big|_2+|r^\frac{m-2}{2}(\rho_0^\beta)_r|_2|\zeta_\frac{5}{8} \rho_0^{-\beta}|_\infty|\zeta (u_0)_{r}|_\infty\big) \\
&\quad+C|r^\frac{m-2}{2} (\rho_0^\beta)_{r}|_2|\zeta \rho_0^{\gamma-1-\beta}|_\infty\leq C, 
\end{aligned}
\end{equation*}
where we have also used the fact that $|\zeta_a\rho_0^{-l}|_\infty\leq C(a,l)$ for $l>0$.  
\end{proof}

Finally, based on Lemmas  \ref{reduce-ex}--\ref{reduce-in},  we give one remark  to show that the examples of the initial data given by \eqref{initial-ex1}--\eqref{initial-ex2} in Remark \ref{initialexample} satisfy the initial assumptions \eqref{distance-la} and  \eqref{a2} in Theorem \ref{Theorem1.1}.
\begin{Remark}\label{initalexample3}
First, due to {\rm Lemma \ref{lemma-initial}}, the initial assumption \eqref{distance-la} is nothing but \eqref{distance} in {\rm Appendix \ref{appb}}, that is,
\begin{equation*} 
\rho_0^\beta(\boldsymbol{y})\in H^3(\Omega), \quad\,\, \cK_1(1-|\boldsymbol{y}|)^\frac{1}{\beta}\leq \rho_0(\boldsymbol{y})\leq \cK_2(1-|\boldsymbol{y}|)^\frac{1}{\beta} \qquad\,\, \text{for all $\boldsymbol{y}\in \overline\Omega$},
\end{equation*}
for some  constants $\cK_2>\cK_1>0$ and $\beta\in (\frac{1}{3},\gamma-1]$. Hence, $\rho_0(\boldsymbol{y})$ given in \eqref{initial-ex1} takes the form{\rm :}
\begin{equation*}
\rho_0^\beta(\boldsymbol{y})=1-|\boldsymbol{y}|^{2k}\qquad {k\in \NN^*},
\end{equation*}
and thus satisfies the condition above.

Next we check that $\boldsymbol{u}_0$ given in \eqref{initial-ex2} satisfies \eqref{a2}. Note that, in the interior domain $\Omega^\flat:=\{\boldsymbol{y}:\,|\boldsymbol{y}|<\frac{1}{2}\}$, from {\rm Lemmas \ref{lemma-initial}} and {\rm\ref{reduce-in}}, we only need to ensure that $\boldsymbol{u}_0\in H^3(\Omega^\flat)$, which is, of course, satisfied by $\boldsymbol{u}_0$ given in \eqref{initial-ex2}.

In the exterior domain $\Omega^\sharp:=\Omega \backslash \Omega^\flat$, when $\beta\in (\frac{1}{3},\frac{2\gamma-1}{5})$ or $\beta=\gamma-1$, we can easily show that $\boldsymbol{u}_0$ given in $\eqref{initial-ex2}_1$ satisfies \eqref{con-b4} in {\rm Lemma \ref{reduce-ex}}. Hence, it still remains to show that $\boldsymbol{u}_0$ given in $\eqref{initial-ex2}_2$ satisfies \eqref{con-b3}, or equivalently, show that
\begin{equation}\label{C28}
u_0(r)= - \frac{A}{2\mu}\int_r^1 \rho_0^{\gamma-1}\,\mathrm{d}\tilde{r}+\tilde{u}_0(r)\qquad \text{satisfies \eqref{con-b3} when $\beta \in \big[\frac{2\gamma-1}{5},\gamma-1\big)$.}
\end{equation}

Indeed, a direct calculation gives that
\begin{equation*}
\begin{aligned}
(u_0)_r&=\frac{A}{2\mu}\rho_0^{\gamma-1} +(\tilde{u}_0)_r,\\
(u_0)_{rr}&= \frac{A}{2\mu}\frac{\gamma-1}{\beta} \rho_0^{\gamma-1-\beta}(\rho_0^\beta)_r+(\tilde{u}_0)_{rr},\\
(u_0)_{rrr}&= \frac{A}{2\mu}\frac{\gamma-1}{\beta}\big(\frac{\gamma-1-\beta}{\beta}\rho_0^{\gamma-1-2\beta}|(\rho_0^\beta)_r|^2+ \rho_0^{\gamma-1-\beta}(\rho_0^\beta)_{rr}\big)+(\tilde{u}_0)_{rrr},\\
\big(\frac{1}{\rho_0}(\rho_0 (u_0)_r)_r\big)_r&=\frac{A\gamma}{2\mu(\gamma-1)}(\rho_0^{\gamma-1})_{rr}+\big(\frac{1}{\rho_0}(\rho_0 (\tilde u_0)_r)_r\big)_r.
\end{aligned}
\end{equation*}
Then, based on  the facts that 
\begin{equation*}
\begin{aligned}
&\tilde u_0\in C_\mathrm{c}^\infty\big[\frac{1}{2},1\big),\quad\rho_0^\beta\sim 1-r,\quad \big(\frac{3}{2}-\varepsilon_0\big)\beta+\gamma-1-2\beta>0,\quad \gamma-\frac{1}{2}-\beta>0,\\
&\big(\frac{1}{\rho_0}(\rho_0 (\tilde u_0)_r)_r\big)_r=\big(\frac{(\rho_0)_r}{\rho_0}\big)_r(\tilde u_0)_r+ \frac{(\rho_0)_r}{\rho_0}(\tilde u_0)_{rr}+  (\tilde u_0)_{rrr} \text{ is compactly supported in $\big[\frac{1}{2},1\big)$},
\end{aligned}
\end{equation*}
we obtain that \eqref{C28} holds.
\end{Remark}

\section{Cross-Derivatives Embedding}\label{subsection2.2}

The following embedding theorem is used to obtain the higher-order elliptic estimates.
\begin{Proposition}\label{prop2.1}
Assume that $\varphi=\varphi(r)$ is a function defined on $I$ and  satisfies
\begin{equation}\label{con2.9-pre}
\varphi\in C^1(\bar I)\cap C^2((0,1]),\qquad \frac{1}{K}(1-r)\leq \varphi \leq K(1-r) \ \ \text{for some $K>1$}.
\end{equation}
Let $(b,c)$ be two parameters such that
\begin{equation}\label{con2.9}
\frac{1}{2}<b\leq \frac{c+1}{2}, 
\end{equation}
and let $f=f(r)\in L^1_{\mathrm{loc}}$ satisfy both $f\in H^1_{\varphi^{2q}}(\frac{1}{2},1)$ for some $q\in [b, \frac{c+1}{2}]$ and 
\begin{equation}\label{con2.10}
\big|\zeta^\sharp(\varphi^{b}f_r+ c\varphi^{b-1}\varphi_rf)\big|_2+|\chi^\sharp\varphi^{b} f|_2<\infty.
\end{equation}
Then, for any $\varphi$ satisfying \eqref{con2.9-pre} and $(b,c)$ satisfying \eqref{con2.9},  there exists a constant $C>0$, which depends only on $(\varphi,b,c)$, such that, for all $f$ satisfying the above assumptions,
\begin{equation}\label{con2.11}
|\zeta^\sharp\varphi^{b}f_r|_2\leq C\big(\big|\zeta^\sharp(\varphi^{b}f_r+c\varphi^{b-1}\varphi_rf)\big|_2+|\chi^\sharp\varphi^{b} f|_2\big).
\end{equation}
\end{Proposition}

\begin{proof}
We divide the proof into three steps.

\smallskip
\textbf{1. Case $f\in C^\infty([\frac{1}{2},1])$.}
It follows from integration by parts, \eqref{con2.9-pre}--\eqref{con2.10}, Lemma \ref{GNinequality}, and the Young inequality that
\begin{equation}\label{2004}
\!\!\!\!\begin{aligned}
|\zeta^\sharp\varphi^b f_r|_2^2&=\big|\zeta^\sharp(\varphi^b f_r+c\varphi^{b-1}\varphi_r f)\big|_2^2-c^2|\zeta^\sharp\varphi^{b-1}\varphi_rf|_2^2 - 2c\int_0^1 (\zeta^\sharp)^2\varphi^{2b-1}\varphi_r  f f_r\,\mathrm{d}r\\
&=\big|\zeta^\sharp(\varphi^b f_r+c\varphi^{b-1}\varphi_r f)\big|_2^2+ c\int_0^1 \big(2\zeta^\sharp(\zeta^\sharp)_r\varphi_{r}+(\zeta^\sharp)^2\varphi_{rr}\big) \varphi^{2b-1} f^2\,\mathrm{d}r\\
& \quad +\underline{ (2b-1-c)c\int_0^1 (\zeta^\sharp)^2\varphi^{2b-2} (\varphi_r)^2 f^2\,\mathrm{d}r}_{\,\leq 0}- \underline{c(\zeta^\sharp)^2\varphi^{2b-1}\varphi_r f^2\Big|_{r=0}^{r=1}}_{\, =0}\\
&\leq \big|\zeta^\sharp(\varphi^b f_r+c\varphi^{b-1}\varphi_r f)\big|_2^2+C\big(|(\zeta^\sharp)_r\varphi^{-1} \varphi_{r}|_\infty|\chi^\sharp\varphi^{b}f|_2^2+ |\chi^\sharp\varphi_{rr}|_\infty\big|\zeta^\sharp\varphi^{b-\frac{1}{2}}f\big|_2^2\big)\\
&\leq  \big|\zeta^\sharp(\varphi^b f_r+c\varphi^{b-1}\varphi_r f)\big|_2^2+C|\chi^\sharp\varphi^{b}f|_2^2+\frac{1}{2}|\zeta^\sharp\varphi^b f_r|_2^2,
\end{aligned}
\end{equation}
which yields \eqref{con2.11}.

\smallskip
\textbf{2. Case  $f\in H^1_{\varphi^{2b}}(\frac{1}{2},1)$.} In this case, we can repeat the calculation in \eqref{2004} to derive \eqref{con2.11}, except for justifying the following integral equality: 
\begin{equation}\label{fenbujifen}
\begin{aligned}
-2\int_0^1 (\zeta^\sharp)^2\varphi^{2b-1}\varphi_r  f f_r\,\mathrm{d}r
&=  \int_0^1 \big(2\zeta^\sharp(\zeta^\sharp)_r\varphi_{r}+(\zeta^\sharp)^2\varphi_{rr}\big) \varphi^{2b-1} f^2\,\mathrm{d}r\\
&\quad\,\, +(2b-1)\int_0^1 (\zeta^\sharp)^2\varphi^{2b-2}(\varphi_r)^2 f^2\,\mathrm{d}r.
\end{aligned}
\end{equation}
Indeed, thanks to Lemma \ref{W-space}, there exists a sequence $\{f^\varepsilon\}_{\varepsilon>0}\subset C^\infty([\frac{1}{2},1])$ such that
\begin{equation}\label{b6b}
|\chi^\sharp\varphi^b (f^\varepsilon- f)|_2+|\chi^\sharp\varphi^b(f^\varepsilon_r- f_r)|_2\to 0 \qquad \text{as $\varepsilon\to 0$},
\end{equation}
which, along with Lemma \ref{hardy-inequality}, yields
\begin{equation}\label{b7b}
|\chi^\sharp\varphi^{b-1}(f^\varepsilon-f)|_2+\big|\chi^\sharp\varphi^{b-\frac{1}{2}}(f^\varepsilon-f)\big|_\infty \to 0 \qquad \text{as $\varepsilon\to 0$}.
\end{equation}
Hence, according to \eqref{b6b}--\eqref{b7b}  and integration by parts for $f^\varepsilon$, we have
\begin{equation*}
\begin{aligned}
-2\int_0^1 (\zeta^\sharp)^2\varphi^{2b-1}\varphi_r f^\varepsilon f^\varepsilon_r\,\mathrm{d}r
&=  \int_0^1 \big(2\zeta^\sharp(\zeta^\sharp)_r\varphi_{r}+(\zeta^\sharp)^2\varphi_{rr}\big) \varphi^{2b-1} (f^\varepsilon)^2\,\mathrm{d}r\\
&\quad +(2b-1)\int_0^1 (\zeta^\sharp)^2\varphi^{2b-2}(\varphi_r)^2 (f^\varepsilon)^2\,\mathrm{d}r.
\end{aligned}
\end{equation*}
Letting $\varepsilon\to 0$ implies that \eqref{fenbujifen} holds for all $f\in H^1_{\varphi^{2b}}(\frac{1}{2},1)$. 

This completes the proof of \eqref{con2.11} when $q=b$.

\smallskip
\textbf{3. General case.}
It suffices to establish \eqref{con2.11} when \eqref{con2.10} holds and
\begin{equation*}
b<\frac{c+1}{2}, \qquad  q=\frac{c+1}{2},\qquad \ f\in H^1_{\varphi^{c+1}}\big(\frac{1}{2},1\big),
\end{equation*}
due to the fact that $H^1_{\varphi^{2q}}(\frac{1}{2},1)\subset H^1_{\varphi^{c+1}}(\frac{1}{2},1)$ if $q\leq \frac{c+1}{2}$. Note that, in this case, integration by parts in \eqref{fenbujifen} fails due  to the fact that we do not assume   $(\zeta^\sharp)^2\varphi^{2b-1}\varphi_r ff_r \in L^1$. 

To overcome this difficulty, set
\begin{equation*}
\vartheta:= \frac{c+1}{2}-b,\qquad \varphi_j:=\varphi+\frac{1}{j} \quad \text{for $j\in \NN^*$}.
\end{equation*}
We first show a variant of \eqref{G-N1} in Lemma \ref{GNinequality}, that is, for all $f\in H^1_{\varphi^{c+1}}(\frac{1}{2},1)$,
\begin{equation}\label{2..8}
\big|\zeta^\sharp\varphi^\frac{c}{2}\varphi_j^{-\vartheta}f\big|_2^2\leq C\big(\big|\chi^\sharp\varphi^\frac{c+1}{2}\varphi_j^{-\vartheta}f\big|_2^2+\big|\zeta^\sharp\varphi^\frac{c+1}{2}\varphi_j^{-\vartheta}f\big|_2\big|\zeta^\sharp\varphi^\frac{c+1}{2}\varphi_j^{-\vartheta}f_r\big|_2\big).
\end{equation}

Based on Lemma \ref{W-space} and the proof in Lemma \ref{GNinequality}, it suffices to show that \eqref{2..8} holds for $f\in C^\infty(\frac{1}{2},1)$. Clearly, we can further let $(\varphi,\varphi_j)=(d,d_j)$ with $d=d(r):=1-r$ and $d_j:=d+\frac{1}{j}$, due to
\begin{equation*}
\frac{d}{K}\leq \varphi \leq K d,\qquad \frac{d_j}{K}\leq \varphi_j \leq K d_j.
\end{equation*} 
It follows from the above reductions, integration by parts, and $(d^{c+1})_r =-(c+1)d^c$ that
\begin{equation*}
\begin{aligned}
\int_0^1 (\zeta^\sharp)^2 d^c d_j^{-2\vartheta} f^2\,\mathrm{d}r&= \frac{2}{c+1} \Big( \int_0^1 \zeta^\sharp(\zeta^\sharp)_r d^{c+1} d_j^{-2\vartheta} f^2\,\mathrm{d}r+\int_0^1 (\zeta^\sharp)^2 d^{c+1} d_j^{-2\vartheta} f f_r\,\mathrm{d}r\Big)\\
&\quad +\frac{2\vartheta}{c+1} \int_0^1 (\zeta^\sharp)^2 d^{c+1} d_j^{-2\vartheta-1} f^2 \,\mathrm{d}r\\
&\leq  \frac{2}{c+1} \Big( \int_0^1 \zeta^\sharp(\zeta^\sharp)_r d^{c+1} d_j^{-2\vartheta} f^2\,\mathrm{d}r+\int_0^1 (\zeta^\sharp)^2 d^{c+1} d_j^{-2\vartheta} f f_r\,\mathrm{d}r\Big)\\
&\quad +\frac{2\vartheta}{c+1} \int_0^1 (\zeta^\sharp)^2 d^{c} d_j^{-2\vartheta} f^2 \,\mathrm{d}r,
\end{aligned}
\end{equation*}
which implies
\begin{equation}\label{2..9}
\begin{aligned}
\int_0^1 (\zeta^\sharp)^2 d^c d_j^{-2\vartheta} f^2\,\mathrm{d}r& \leq  \frac{1}{b} \Big( \int_0^1 \zeta^\sharp(\zeta^\sharp)_r d^{c+1} d_j^{-2\vartheta} f^2\,\mathrm{d}r+\int_0^1 (\zeta^\sharp)^2 d^{c+1} d_j^{-2\vartheta} f f_r\,\mathrm{d}r\Big)\\
&\leq C\big(\big|\chi^\sharp d^\frac{c+1}{2} d_j^{-\vartheta}f\big|_2^2+\big|\zeta^\sharp d^\frac{c+1}{2} d_j^{-\vartheta}f\big|_2\big|\zeta^\sharp d^\frac{c+1}{2} d_j^{-\vartheta}f_r\big|_2\big).
\end{aligned}
\end{equation}

This completes the proof of \eqref{2..8}.

\smallskip
Now, we continue to prove \eqref{con2.11}. It follows from \eqref{con2.10} and Lemma \ref{hardy-inequality} that
\begin{equation}\label{con2.15}
Q_j:=\zeta^\sharp\varphi_j^{-\vartheta}(\varphi^\frac{c+1}{2}f_r+c\varphi^\frac{c-1}{2}\varphi_{r}f)\in L^2 \qquad \text{for any $j\in \NN^*$}.
\end{equation}
Using a density argument similar to that in Step 2, we can show that the following integral equality still holds, {\it i.e.}, for all $f\in H^1_{\varphi^{c+1}}(\frac{1}{2},1)$ and $j\in \NN^*$,
\begin{equation*}
\begin{aligned}
-2\int_0^1 (\zeta^\sharp)^2\varphi^{c}\varphi_j^{-2\vartheta}\varphi_r f f_r\,\mathrm{d}r&=  \int_0^1 \big(2\zeta^\sharp(\zeta^\sharp)_r\varphi_{r}+(\zeta^\sharp)^2\varphi_{rr}\big) \varphi^{c}\varphi_j^{-2\vartheta} f^2\,\mathrm{d}r\\
&\quad +\int_0^1 (\zeta^\sharp)^2\big(c\varphi^{c-1}\varphi_j^{-2\vartheta} -2\vartheta \varphi^{c}\varphi_j^{-2\vartheta-1}\big)(\varphi_r)^2 f^2\,\mathrm{d}r.
\end{aligned}
\end{equation*}
Hence, based on the above equality and the calculation similar to \eqref{2004}, we deduce from \eqref{con2.15} and the Young inequality that
\begin{equation*}
\begin{aligned}
\big|\zeta^\sharp\varphi^\frac{c+1}{2}\varphi_j^{-\vartheta} f_r\big|_2^2&=|Q_j|_2^2-c^2\big|\zeta^\sharp\varphi^{\frac{c-1}{2}}\varphi_j^{-\vartheta}\varphi_rf\big|_2^2 - 2c\int_0^1 (\zeta^\sharp)^2\varphi^{c}\varphi_j^{-2\vartheta}\varphi_r  f f_r\,\mathrm{d}r\\
&=|Q_j|_2^2+ c\int_0^1 \big(2\zeta^\sharp(\zeta^\sharp)_r\varphi_{r}+(\zeta^\sharp)^2\varphi_{rr}\big) \varphi^{c}\varphi_j^{-2\vartheta} f^2\,\mathrm{d}r\notag\\
&\quad \underline{-2\vartheta c\int_0^1 (\zeta^\sharp)^2 \varphi^{c}\varphi_j^{-2\vartheta-1} (\varphi_r)^2 f^2\,\mathrm{d}r}_{\,\leq 0}\\
&\leq |Q_j|_2^2+C\big(|(\zeta^\sharp)_r\varphi^{-1} \varphi_{r}|_\infty\big|\chi^\sharp\varphi^{\frac{c+1}{2}}\varphi_j^{-\vartheta}f\big|_2^2+ |\chi^\sharp\varphi_{rr}|_\infty\big|\zeta^\sharp\varphi^{\frac{c}{2}}\varphi_j^{-\vartheta}f\big|_2^2\big),
\end{aligned} 
\end{equation*}
which, along with \eqref{con2.9-pre} and \eqref{2..8}, gives
\begin{equation*}
\begin{aligned}
\big|\zeta^\sharp\varphi^{\frac{c+1}{2}}\varphi_j^{-\vartheta}f_r\big|_2^2&\leq C\big(|Q_j|_2^2+\big|\chi^\sharp\varphi^{\frac{c+1}{2}}\varphi_j^{-\vartheta}f\big|_2^2\big)\\
&\leq C\big(\big|\zeta^\sharp(\varphi^{b}f_r+ c\varphi^{b-1}\varphi_rf)\big|_2^2+|\chi^\sharp\varphi^{b} f|_2^2\big).
\end{aligned}
\end{equation*}
Since $C$ is independent of $j$, we can extract a subsequence (still denoted by $j$)  such that
\begin{equation}\label{equ213}
\zeta^\sharp\varphi^{\frac{c+1}{2}}\varphi_j^{-\vartheta}f_r\to g\qquad \text{weakly in }L^2 \quad \text{as }j\to \infty,
\end{equation}
for some limit function $g\in L^2$, and 
\begin{equation*}
|g|_2\leq \liminf_{j\to \infty}\big|\zeta^\sharp\varphi^{\frac{c+1}{2}}\varphi_j^{-\vartheta}f_r\big|_2\leq C\big(\big|\zeta^\sharp(\varphi^{b}f_r+ c\varphi^{b-1}\varphi_rf)\big|_2+|\chi^\sharp\varphi^{b} f|_2\big). 
\end{equation*}

Note that $f_r\in L^1_{\mathrm{loc}}$. The Lebesgue dominated convergence theorem also gives
\begin{equation}\label{equ214}
\zeta^\sharp\varphi^{\frac{c+1}{2}}\varphi_j^{-\vartheta}f_r\to \zeta^\sharp\varphi^b f_r\qquad \text{in }L^1_{\mathrm{loc}} \ \ \text{as }j\to\infty.
\end{equation}
Hence, by \eqref{equ213}--\eqref{equ214} and the uniqueness of the limits, we have $g=\zeta^\sharp\varphi^b f_r$. 

This completes the proof of Proposition  \ref{prop2.1}.
\end{proof}

\bigskip
\noindent{\bf Acknowledgments:}  The authors are grateful to Prof. Didier Bresch for helpful discussions on the effective velocity and BD entropy of the degenerate compressible Navier-Stokes equations. This research is partially supported by National Key R$\&$D Program of China (No. 2022YFA1007300). The research of Gui-Qiang G. Chen was also supported in part by the UK Engineering and Physical Sciences Research Council Award EP/L015811/1, EP/V008854, and EP/V051121/1. The research of Shengguo Zhu was also supported in part by the National Natural Science Foundation of China under the Grant  12471212, and the Royal Society (UK)-Newton International Fellowships NF170015.

\bigskip
\noindent{\bf Author Contributions:}  Gui-Qiang G. Chen, Jiawen Zhang, and Shengguo Zhu have contributed equally to this work.

\bigskip
\noindent{\bf Conflict of Interest:} The authors declare that they have no conflict of interest.

\bigskip
\noindent{\bf Data availability:} Data sharing is not applicable to this article as no datasets were generated or analyzed during the current study.

\end{document}